%% file: paper.tex
\documentclass[a4paper]{report}
\input{preamble}
\input{preamble-macros}

\title{Manifold Diagrams for Higher Categories}
\author{Lukas Heidemann}
\college{Worcester College}
\degreedate{Hilary 2024}
\degree{Doctor of Philosophy in Computer Science}

\begin{document}

\maketitle

\begin{abstract}
  We develop a graphical calculus of manifold diagrams which generalises
  string and surface diagrams to arbitrary dimensions.
  Manifold diagrams are pasting diagrams for $(\infty, n)$-categories
  that admit a semi-strict composition operation for which associativity and unitality is strict.
  The weak interchange law satisfied by composition of manifold diagrams
  is determined geometrically through isotopies of diagrams.
  By building upon framed combinatorial topology, we can classify critical
  points in isotopies at which the arrangement of cells changes.
  This allows us to represent manifold diagrams combinatorially
  and use them as shapes with which to probe $(\infty, n)$-categories, presented as
  $n$-fold Segal spaces.
  Moreover, for any system of labels for the singularities in a manifold diagram,
  we show how to generate a free $(\infty, n)$-category.
\end{abstract}

\tableofcontents

\clearpage

\input{chapter-introduction}
\input{chapter-background}
\input{chapter-fct}
\input{chapter-diagrams}

\bibliography{paper}{}
\bibliographystyle{alphaurl}

\end{document}

%% file: preamble.tex
\usepackage[utf8]{inputenc}
\usepackage[includehead,hmargin={3.5cm, 3.5cm}, vmargin={2.5cm,2.7cm}, headsep=.8cm,footskip=1.2cm]{geometry}
\usepackage[dvipsnames]{xcolor}
\usepackage{graphicx}
\usepackage[center, small]{titlesec}
\usepackage{microtype}
\usepackage{nicefrac}
\usepackage{amsmath}
\usepackage{amsthm}
\usepackage{thmtools, thm-restate}
\usepackage[font=small,labelfont=bf]{caption}
\usepackage{subcaption}
\usepackage{multicol,multirow,array}
\usepackage{parskip}
\usepackage{enumitem}
\usepackage{url}
\usepackage{url}
\usepackage[colorlinks=true,linkcolor=Black,anchorcolor=Black,citecolor=Black,filecolor=Black,menucolor=Black,runcolor=Black,urlcolor=Black]{hyperref}
\usepackage{appendix}
\usepackage[LGR,TS1,T1]{fontenc}
\usepackage[greek,latin,english]{babel}
\usepackage{csquotes}
\include{preamble-tikz}
\usepackage{xargs}
\usepackage{contour}
\DeclareMathAlphabet{\mathbbe}{U}{bbold}{m}{n}
\usepackage{mathtools}
\usepackage{scalerel}

\usepackage{amssymb}
\usepackage{lmodern}
\usepackage{stmaryrd}
\usepackage{eucal}
\usepackage[bb=libus]{mathalpha}
\usepackage{fancyhdr}
\fancypagestyle{plain}{\fancyhf{}\fancyfoot[C]{\emph{\thepage}}}

\makeatletter
\def\degreedate#1{\gdef\@degreedate{#1}}
\def\degree#1{\gdef\@degree{#1}}
\def\college#1{\gdef\@college{#1}}
\def\crest{{\includegraphics{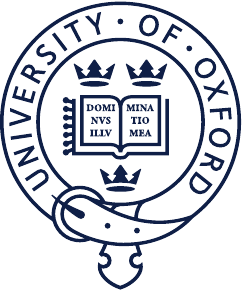}}}
\newcommand{\submittedtext}{{A thesis submitted for the degree of}}

\renewcommand{\maketitle}{%
	\renewcommand{\footnotesize}{\small}
	\renewcommand{\footnoterule}{\relax}
	\thispagestyle{empty}
	\begin{center}
		{ \LARGE {\bfseries {\@title}} \par}
		{\large \vspace*{30mm} {\crest \par} \vspace*{25mm}}
		{{\Large \@author} \par}
		{\large \vspace*{1ex}
			{{\@college} \par}
			\vspace*{1ex}
			{University of Oxford \par}
			\vspace*{20mm}
			{{\submittedtext} \par}
			\vspace*{1ex}
			{\it {\@degree} \par}
			\vspace*{2ex}
			{\@degreedate}}
	\end{center}
	\vfill\null}
\makeatother

\newtheoremstyle{hypothesis}%
{}{}%
{}{}%
{}{.}%
{ }%
{\textbf{\thmnote{#3}}}

\newtheoremstyle{note}%
{}{}%
{\color{BrickRed}}{}%
{\bfseries}{.}%
{ }%
{}

\newtheoremstyle{theorem} 
{6pt} 
{3pt} 
{\itshape} 
{} 
{\bfseries} 
{. } 
{0pt} 
{\thmnumber{#2 }{\thmname{#1}}\thmnote{ (#3)}} 
\newtheoremstyle{definition} 
{8pt} 
{3pt} 
{} 
{} 
{\bfseries} 
{. } 
{0pt} 
{\thmnumber{#2 }{\thmname{#1}}\thmnote{ (#3)}} 
\newtheoremstyle{para} 
{8pt} 
{3pt} 
{} 
{} 
{\bfseries} 
{. } 
{0pt} 
{\thmnumber{#2}\thmnote{ (#3)}} 

\theoremstyle{theorem}
\newtheorem{thm}{Theorem}[subsection]
\newtheorem{proposition}[thm]{Proposition}
\newtheorem{lemma}[thm]{Lemma}
\newtheorem{cor}[thm]{Corollary}

\theoremstyle{hypothesis}

\theoremstyle{definition}
\newtheorem{definition}[thm]{Definition}
\newtheorem{construction}[thm]{Construction}
\newtheorem{observation}[thm]{Observation}

\newtheorem{example}[thm]{Example}
\newtheorem{remark}[thm]{Remark}

\theoremstyle{para}
\newtheorem{para}[thm]{}
\theoremstyle{note}



%% file: preamble-tikz.tex
\usepackage{tikz}
\usepackage{tikz-3dplot}
\usetikzlibrary{calc, backgrounds, arrows, arrows.meta, babel}
\usetikzlibrary{decorations.pathmorphing}
\usetikzlibrary{decorations.markings}
\usepackage{pgfplots} 
\pgfplotsset{compat=newest} 
\usepackage{tikz-cd}

\definecolor{diagram-background-color}{RGB}{246, 246, 246}
\definecolor{mesh-background-color}{RGB}{246, 246, 246}

\tikzstyle{diagram-background}=[diagram-background-color]
\tikzstyle{diagram-box}=[fill = lightgray, thick]
\tikzstyle{diagram-wire}=[thick]
\tikzstyle{diagram-wire-over}=[diagram-background-color, line width = 3]
\tikzstyle{diagram-vertex}=[circle, inner sep=2pt, fill = black]

\tikzstyle{mesh-background}=[mesh-background-color]
\tikzstyle{mesh-open-boundary}=[lightgray, densely dotted]
\tikzstyle{mesh-closed-boundary}=[lightgray]
\tikzstyle{mesh-stratum-over}=[mesh-background-color, line width = 3]
\tikzstyle{mesh-stratum}=[thick, black]
\tikzstyle{mesh-stratum-orange}=[thick, BurntOrange]
\tikzstyle{mesh-stratum-blue}=[thick, NavyBlue]
\tikzstyle{mesh-stratum-dual}=[thick, BrickRed]
\tikzstyle{mesh-stratum-3d}=[fill=gray, fill opacity = 0.3]
\tikzstyle{mesh-vertex}=[circle, inner sep = 1.5pt, fill=black]
\tikzstyle{mesh-vertex-orange}=[circle, inner sep = 1.5pt, fill=BurntOrange]
\tikzstyle{mesh-vertex-blue}=[circle, inner sep = 1.5pt, fill=NavyBlue]
\tikzstyle{mesh-vertex-over}=[circle, inner sep = 2.5pt, fill=mesh-background-color]
\tikzstyle{mesh-vertex-dual}=[circle, inner sep = 1.5pt, fill=BrickRed]

\tikzstyle{mesh3-stratum}=[thick, black]
\tikzstyle{mesh3-stratum-orange}=[thick, BurntOrange]
\tikzstyle{mesh3-stratum-blue}=[thick, NavyBlue]
\tikzstyle{mesh3-stratum-dual}=[thick, BrickRed]
\tikzstyle{mesh3-surface}=[fill=gray, fill opacity = 0.3, gray]
\tikzstyle{mesh3-vertex}=[circle, inner sep = 1pt, fill=black]
\tikzstyle{mesh3-vertex-dual}=[circle, inner sep = 1.5pt, fill=BrickRed]
\tikzstyle{mesh3-bounding-box}=[thin, lightgray]

\tikzstyle{truss-vertex}=[]

\newcommand{\meshBoundingBox}[3]{
  \draw[mesh3-bounding-box]
    (0, 0, 0) --
    (#1, 0, 0) -- 
    (#1, 0, #3) -- 
    (0, 0, #3) -- cycle;
  \draw[mesh3-bounding-box]
    (0, #2, 0) --
    (#1, #2, 0) -- 
    (#1, #2, #3) -- 
    (0, #2, #3) -- cycle;
  \draw[mesh3-bounding-box] (0, 0, 0) -- (0, #2, 0);
  \draw[mesh3-bounding-box] (#1, 0, 0) -- (#1, #2, 0); 
  \draw[mesh3-bounding-box] (#1, 0, #3) -- (#1, #2, #3); 
  \draw[mesh3-bounding-box] (0, 0, #3) -- (0, #2, #3);
}

\tikzset{mid vert/.style={/utils/exec=\tikzset{every node/.append style={outer sep=0.5ex}},
postaction=decorate,decoration={markings,
mark=at position 0.5 with {\draw[-] (0,#1) -- (0,-#1);}}},
mid vert/.default=0.5ex}

\tikzset{curve/.style={settings={#1},to path={(\tikztostart)
    .. controls ($(\tikztostart)!\pv{pos}!(\tikztotarget)!\pv{height}!270:(\tikztotarget)$)
    and ($(\tikztostart)!1-\pv{pos}!(\tikztotarget)!\pv{height}!270:(\tikztotarget)$)
    .. (\tikztotarget)\tikztonodes}},
    settings/.code={\tikzset{quiver/.cd,#1}
        \def\pv##1{\pgfkeysvalueof{/tikz/quiver/##1}}},
    quiver/.cd,pos/.initial=0.35,height/.initial=0}
\tikzset{tail reversed/.code={\pgfsetarrowsstart{tikzcd to}}}
\tikzset{2tail/.code={\pgfsetarrowsstart{Implies[reversed]}}}
\tikzset{2tail reversed/.code={\pgfsetarrowsstart{Implies}}}
\tikzset{no body/.style={/tikz/dash pattern=on 0 off 1mm}}

%% file: preamble-macros.tex
\DeclareFontFamily{U}{min}{}
\DeclareFontShape{U}{min}{m}{n}{<-> udmj30}{}

\newcommand{\catname}[1]{\textup{#1}}

\newcommand{\sgn}{\operatorname{sgn}}

\newcommand{\R}{\mathbb{R}}

\newcommand{\ord}[1]{\lbrack #1 \rbrack}
\newcommand{\op}{\textup{op}}
\newcommand{\cat}[1]{\mathcal{#1}}

\newcommand{\strat}[1]{{#1}}
\newcommand{\stratMap}[1]{\cramped{\mathbb{s}}({#1})}
\newcommand{\stratPos}[1]{\cramped{\mathbb{P}}({#1})}
\newcommand{\stratForget}[1]{#1}
\newcommand{\unstrat}{\operatorname{U}}

\newcommand{\framing}{\operatorname{fr}}

\newcommand{\core}{\operatorname{core}}

\newcommand{\CatN}[1]{\catname{Cat}_{(\infty, #1)}}
\newcommand{\CatInfty}{\catname{Cat}_\infty}

\newcommand{\Space}{\catname{S}}

\newcommand{\Tw}{\operatorname{Tw}}

\newcommand{\Cart}[1]{\catname{Cart}(#1)}
\newcommand{\coCart}[1]{\catname{coCart}(#1)}
\newcommand{\LFib}[1]{\catname{LFib}(#1)}
\newcommand{\RFib}[1]{\catname{RFib}(#1)}
\newcommand{\Pos}{\catname{Pos}}
\newcommand{\Set}{\catname{Set}}
\newcommand{\FinSet}{\catname{FinSet}}
\newcommand{\FinOrd}{\Delta}

\newcommand{\sSet}{\catname{sSet}}

\newcommand{\pack}{\operatorname{pack}}

\newcommand{\tr}{\operatorname{tr}}

\newcommand{\sSetM}{\catname{sSet}^+}

\newcommand{\PSh}{\catname{P}}

\newcommand{\poly}[2]{#1[#2]}
\newcommand{\epoly}[2]{#1_*[#2]}

\newcommand{\join}{*}

\newcommand{\Weq}{\catname{W}}
\newcommand{\Cof}{\catname{Cof}}
\newcommand{\Fib}{\catname{Fib}}
\newcommand{\MSKan}{\textup{Kan}}
\newcommand{\MSJoyal}{\textup{Joyal}}

\newcommand{\MSReedy}{\textup{Reedy}}
\newcommand{\MSCov}{\textup{Cov}}
\newcommand{\MSContr}{\textup{Contr}}
\newcommand{\MSCart}{\textup{Cart}}

\newcommand{\psh}[1]{\mathcal{#1}}

\newcommand{\TrussOpen}[1]{\cramped{\dot{\mathfrak{T}}}^{#1}}
\newcommand{\TrussClosed}[1]{\cramped{\bar{\mathfrak{T}}}^{#1}}

\newcommand{\BTruss}[1]{\catname{T}^{#1}}
\newcommand{\BTrussOpen}[1]{\cramped{\dot{\catname{T}}}^{#1}}
\newcommand{\BTrussClosed}[1]{\cramped{\bar{\catname{T}}}^{#1}}

\newcommand{\BTrussOpenL}[2]{\BTrussOpen{#1}(#2)}
\newcommand{\BTrussClosedL}[2]{\BTrussClosed{#1}(#2)}

\newcommand{\ETrussOpen}[1]{\BTrussOpen{#1}_*}
\newcommand{\ETrussClosed}[1]{\BTrussClosed{#1}_*}

\newcommand{\ETrussOpenL}[2]{\BTrussOpen{#1}_*(#2)}

\newcommand{\trussOR}[1]{\dot{\textup{r}}_{#1}}
\newcommand{\trussOS}[1]{\dot{\textup{s}}_{#1}}
\newcommand{\trussCR}[1]{\bar{\textup{r}}_{#1}}
\newcommand{\trussCS}[1]{\bar{\textup{s}}_{#1}}

\newcommand{\StratIntLR}{\mathbb{I}}

\newcommand{\CylMeshClosed}{\textup{C}}

\newcommand{\lbl}[1]{\ell(#1)}

\newcommand{\LLP}{\boxslash}
\newcommand{\llp}{\operatorname{llp}}
\newcommand{\rlp}{\operatorname{rlp}}

\newcommand{\terminal}{*}

\newcommand{\Mesh}[1]{\mathfrak{M}^{#1}}
\newcommand{\MeshOpen}[1]{\cramped{\dot{\mathfrak{M}}}^{#1}}
\newcommand{\MeshClosed}[1]{\cramped{\bar{\mathfrak{M}}}^{#1}}

\newcommand{\BMesh}[1]{\catname{M}^{#1}}
\newcommand{\BMeshOpen}[1]{\cramped{\dot{\catname{M}}}^{#1}}
\newcommand{\BMeshClosed}[1]{\cramped{\bar{\catname{M}}}^{#1}}

\newcommand{\BMeshClosedCB}[1]{\cramped{\bar{\catname{M}}}^{#1}_{\textup{cb}}}

\newcommand{\BMeshL}[2]{\BMesh{#1}(#2)}
\newcommand{\BMeshOpenL}[2]{\BMeshOpen{#1}(#2)}
\newcommand{\BMeshClosedL}[2]{\BMeshClosed{#1}(#2)}

\newcommand{\EMesh}[1]{\BMesh{#1}_*}
\newcommand{\EMeshOpen}[1]{\BMeshOpen{#1}_*}
\newcommand{\EMeshClosed}[1]{\BMeshClosed{#1}_*}

\newcommand{\EMeshClosedL}[2]{\BMeshClosed{#1}_*(#2)}

\newcommand{\TrussOpenNerve}{\textup{N}_{\dot{\textup{T}}}}
\newcommand{\TrussOpenReal}[1]{\| #1 \|_{\dot{\textup{T}}}}
\newcommand{\TrussClosedNerve}{\textup{N}_{\bar{\textup{T}}}}
\newcommand{\TrussClosedReal}[1]{\| #1 \|_{\bar{\textup{T}}}}

\newcommand{\Conf}[1]{\catname{Conf}^{#1}}
\newcommand{\BConf}[1]{\Conf{#1}}
\newcommand{\BConfL}[2]{\Conf{#1}(#2)}

\newcommand{\BConfOrtho}[1]{\Conf{#1}_{\perp}}

\newcommand{\LConfMap}[1]{\mathfrak{C}\textup{onf}^{#1}_{\textup{L}}}
\newcommand{\LConf}[1]{\textup{Conf}^{#1}_{\textup{L}}}
\newcommand{\LConfL}[2]{\textup{Conf}^{#1}_{\textup{L}}(#2)}
\newcommand{\RConfMap}[1]{\mathfrak{C}\textup{onf}^{#1}_{\textup{R}}}
\newcommand{\RConf}[1]{\textup{Conf}^{#1}_{\textup{R}}}
\newcommand{\RConfL}[2]{\textup{Conf}^{#1}_{\textup{R}}(#2)}
\newcommand{\EConfMap}[1]{\mathfrak{C}\textup{onf}^{#1}_{\textup{LR}}}
\newcommand{\EConf}[1]{\textup{Conf}^{#1}_{\textup{LR}}}
\newcommand{\EConfL}[2]{\textup{Conf}^{#1}_{\textup{LR}}(#2)}

\newcommand{\regular}{\textup{r}}
\newcommand{\glob}{\textup{g}}

\newcommand{\inert}{\textup{int}}

\newcommand{\CylO}{\cramped{\textup{Cyl}}_\textup{o}}

\newcommand{\Seg}{\operatorname{Seg}}

\newcommand{\grid}{\cramped{\boxplus}_{\dot{\textup{T}}}}

\newcommand{\gridTrussOpen}{\cramped{\boxplus}_{\scriptscriptstyle\dot{\textup{T}}}}
\newcommand{\gridMeshOpen}{\cramped{\boxplus}_{\scriptscriptstyle\dot{\textup{M}}}}
\newcommand{\gridTrussClosed}{\cramped{\boxplus}_{\scriptscriptstyle\bar{\textup{T}}}}
\newcommand{\gridMeshClosed}{\cramped{\boxplus}_{\scriptscriptstyle\bar{\textup{M}}}}

\newcommand{\trussFC}{\mathfrak{t}}
\newcommand{\meshFC}{\mathfrak{m}}

\newcommand{\frtimes}{\rtimes}

\newcommand{\frleq}{\leq_{\textup{fr}}}

\newcommand{\Cptd}[1]{\cramped{\catname{Cptd}}_{(\infty, #1)}}

\newcommand{\DiagCatFN}[1]{\cramped{\catname{fdCat}}_{(\infty, n)}}

\newcommand{\DiagFree}{\mathbb{F}}
\newcommand{\DiagForget}{\mathbb{U}}

\newcommand{\Arr}{\textup{Arr}}

\newcommand{\BMeshDiag}[1]{\BMesh{#1}_{\textup{d}}}
\newcommand{\BTrussDiag}[1]{\BTruss{#1}_{\textup{d}}}
\newcommand{\Act}{\textup{Act}}

\newcommand{\MfldDiag}{\cramped{\mathbb{D}}}
\newcommand{\MfldDiagTame}{\cramped{\mathbb{D}}_{\textup{tame}}}
\newcommand{\MfldDiagExt}{\cramped{\mathbb{E}}}

\newcommand{\MeshDiag}{\cramped{\mathbb{D}}_{\textup{M}}}

\newcommand{\MeshDiagExt}{\cramped{\mathbb{E}}_{\textup{M}}}
\newcommand{\TrussDiag}{\cramped{\mathbb{D}}_{\textup{T}}}

\newcommand{\TrussDiagExt}{\cramped{\mathbb{E}}_{\textup{T}}}

\newcommand{\degMeshClosed}{\cramped{\textup{deg}}_{\bar{\textup{M}}}}
\newcommand{\degMeshOpen}{\cramped{\textup{deg}}_{\dot{\textup{M}}}}

\newcommand{\atom}[1]{\textup{atom}_{#1}}
\newcommand{\cell}[1]{\textup{cell}_{#1}}

\newcommand{\Grid}{\catname{Grid}}

\newcommand{\Cat}{\textup{Cat}}

\newcommand{\Nat}{\mathbb{N}}
\newcommand{\Z}{\mathbb{Z}}

\newcommand{\sdepth}{\textup{sdepth}}

\newcommand{\stype}{\textup{stype}}
\newcommand{\fproj}{\operatorname{proj}_{\textup{fr}}}

\newcommand{\sd}{\textup{sd}}
\newcommand{\ExitMark}{\textup{Exit}^+}
\newcommand{\StratRealMark}[1]{\|#1\|^+}

\newcommand{\StratReal}[1]{\|#1\|}
\newcommand{\TopReal}[1]{|#1|}
\newcommand{\Sing}{\operatorname{Sing}}

\newcommand{\Strat}{\catname{Strat}}

\newcommand{\StrictInt}{\nabla}

\newcommand{\Fun}{\catname{Fun}}

\newcommand{\TopSp}{\catname{Top}}

\newcommand{\DeltaStrat}[1]{\StratReal{\Delta\ord{#1}}}
\newcommand{\HornStrat}[2]{\StratReal{\Lambda^{#1}\ord{#2}}}

\newcommand{\DeltaStratR}[1]{\StratReal{\Delta\ord{#1}^\op}}

\newcommand{\DeltaTop}[1]{{|\Delta\ord{#1}|}}
\newcommand{\HornTop}[2]{{|\Lambda^{#1}\ord{#2}|}}

\newcommand{\intOC}[1]{{(#1]}}
\newcommand{\intCC}[1]{{[#1]}}
\newcommand{\intCO}[1]{{[#1)}}
\newcommand{\intOO}[1]{{(#1)}}

\newcommand{\FrStrat}{\catname{FrStrat}}
\newcommand{\FrBun}{\catname{FrBun}}

\newcommand{\sing}{\textup{sing}}

\newcommand{\reg}{\textup{reg}}

\newcommand{\Ho}{\textup{h}}
\newcommand{\colim}{\operatorname*{colim}}

\newcommand{\dom}{\textup{dom}}

\newcommand{\cl}{\textup{cl}}
\newcommand{\im}{\textup{im}}
\newcommand{\const}{\textup{const}}
\newcommand{\id}{\textup{id}}

\renewcommand{\dim}{\textup{dim}}

\newcommand{\cod}{\textup{cod}}
\newcommand{\Ex}{\textup{Ex}}
\newcommand{\Sd}{\textup{Sd}}

\newcommand{\Exit}{\textup{Exit}}
\newcommand{\Entr}{\textup{Entr}}

\newcommand{\cone}{\textup{c}}

\newcommand{\shape}[1]{\square{#1}}
\newcommand{\shapeBord}[2]{\square_{#2}{#1}}

\newcommand{\CatToSpace}{\textup{k}}

\newcommand{\nf}{\textup{nf}}

\newcommand{\Nerve}{\textup{N}}

\newcommand{\pullbackcorner}{\arrow[dr, phantom, "\lrcorner", very near start, color=black]}
\newcommand{\pullbackdl}{\ar["\llcorner", phantom, very near start, dl]}
\newcommand{\pushoutcorner}{\ar["\lrcorner", phantom, very near end, dr]}

\newcommand{\defn}[1]{\textit{#1}}
\newcommand{\eps}{\varepsilon}

\newcommand{\classify}[1]{{\llbracket #1 \rrbracket}}
\newcommand{\classifyMark}[1]{{\llbracket #1 \rrbracket^+}}

\newcommand{\Alex}{\textup{Alex}}

\newcommand\pto{\mathrel{\ooalign{\hfil$\mapstochar\mkern5mu$\hfil\cr$\to$\cr}}}

%% file: chapter-introduction.tex
\chapter{Introduction}

String diagrams, Gray surface diagrams and extended topological quantum field theories
are bridges between geometry and the algebra of higher categories.
In this thesis we generalise string and surface diagrams to arbitrary
dimensions, to obtain an $n$-dimensional graphical calculus of manifold diagrams
for $(\infty, n)$-categories.
Whereas surface diagrams have vertices, edges, faces and volumes, an $n$-manifold
diagram consists of strata that are $k$-manifolds for all $0 \leq k \leq n$,
each representing a $k$-codimensional cell of an $(\infty, n)$-category.

Higher category theory has promising applications throughout mathematics,
physics and computer science by
providing a common language to formally capture homotopy coherent algebraic structures that are otherwise difficult to describe or manipulate.
The expressivity of higher categories comes at the cost of complexity,
which has limited its applications in practice beyond homotopy theory.
The coherence structure alone that guarantees weak associativity, weak unitality
and the weak interchange law for the composition operations is highly non-trivial.
A graphical calculus like string diagrams for higher categories in arbitrary dimensions would enable more concrete calculations and simplify coherence of the composition operation.

While every $2$-category is equivalent to one in which
the coherence structure is strict, strictification is no longer possible 
in general for $n$-categories when $n > 2$ without losing essential structure
~\cite{simpsonHomotopyTypesStrict1998}.
For example the fundamental $3$-groupoid $\Pi_3(S^2)$ of the $2$-sphere $S^2$
cannot be equivalent to a strict $3$-groupoid since then the Eckmann-Hilton
argument would trivialise the $3$-cell that corresponds to the generator of
$\pi_3(S^2) \cong \Z$.
In principle, it is possible to explicitly write down the coherence laws for an $n$-category
for any $n \geq 0$, but the ensuing explosion of coherence axioms makes this approach unworkable.

We can therefore not fully avoid the complexity of the coherence laws without
losing the expressiveness of $n$-categories.
This does not preclude us in principle from strictifying at least part of the coherence structure, arriving at a semi-strict model of $n$-categories that is simpler than but still equivalent to the general theory.
The surface diagram calculus of~\cite{gray-category-diagrams} is built upon Gray categories,
a semi-strict model of $3$-categories in which associativity and unitality of composition holds strictly and the interchange law remains weak.
This style of semi-strictification is a natural choice for a geometric model since it allows translating between composition of terms and gluing of subdiagrams.

So far there have been no known semi-strict models for $4$-categories (or higher)
in the style of Gray categories, with strict associativity and unitality but weak
interchange.
Teisi~\cite{crans-teisi} and the quasi-strict $4$-categories of~\cite{quasistrict-datastructures}
are candidates, but have not yet been proven equivalent to the general theory of $4$-categories.
Instead of finding a graphical calculus for a known semi-strict model
as in~\cite{gray-category-diagrams}, we can go in the opposite direction
and characterise a semi-strict model
for $n$-categories by starting with the geometry of its graphical calculus
in the form of manifold diagrams.
The composition operation of the resulting theory simply juxtaposes
diagrams of varying sizes and therefore, generalising Moore paths, is strictly
associative and unital.

The most widely used models of higher categories to date are based on (abstract) homotopy theory.
Via the homotopy hypothesis, we can already understand a topological space as an $\infty$-groupoid:
the points are the objects, paths are $1$-morphisms, homotopies between paths are $2$-morphisms, etc.
The coherences are fully weak, but the complexity remains manageable via the wide range of tools of homotopy theory. Spaces or $\infty$-groupoids are the first in a sequence of types of $\omega$-categories with a decreasing amount of invertibility:
An $(\infty, n)$-category is an $\omega$-category
with the requirement that a $k$-morphism is invertible whenever $k \geq n$.
The case of $n = 1$ is particularly well understood, with most of $1$-category theory transferring to $(\infty, 1)$-categories.

At first, it appears that moving from $n$-categories to $(\infty, n)$-categories only makes things harder for us.
But coherences are invertible, and so we can offload much of their complexity onto the homotopy theory of spaces.
This trick is already present in the string diagram calculus for plain, braided and
symmetric monoidal categories, where isotopies of the diagrams correspond to
the coherence laws in their algebraic formulation.
We can therefore rely on the geometric nature of manifold diagrams and describe
the weak interchange coherences directly as a space of isotopies rather than characterising them algebraically.

Homotopy theoretic models for higher categories are glued together from a category of basic shapes,
such as the simplices $\FinOrd$, multi-simplices $\FinOrd^n$, globular pasting diagrams $\Theta^n$
and various versions of a category of cubes.
These shapes are typically optimised to be as minimal or basic as possible,
which is a good trade-off in situations where we want to prove general theorems.
There is however value in finding bigger and more structured shapes, as every object that is modelled from them automatically inherits their structure. This is preferable whenever we want to perform calculations on particular examples of higher categories.
\begin{quote}
I don’t like being tied to one particular shape of cell, particularly simplicial ones. When working with low-dimensional higher categories, like 2-categories and 3-categories, we use pasting diagrams with cells of many different shapes. I would like a notion of higher category that has an underlying sort of “higher computad” where we can talk about arbitrary pasting diagrams and “higher surface diagrams.”~\textit{Mike Shulman} - \cite{confessions-higher-cats}
\end{quote}

We identify the higher surface diagrams in this quote as our manifold diagrams.
Manifold diagrams are a particularly useful shape to do higher category theory.
Because manifold diagrams are composed by juxtaposition,
we can describe complex composites in a single shape.
While this composition operation does not satisfy a strict interchange law,
the coherences for weak interchange in $n$-manifold diagrams are isotopies
and therefore themselves $(n + 1)$-diagrams.
Manifold diagrams are also well suited to describe freely generated higher categories.
The generating data for a free $n$-category is called an $n$-computad
and is typically defined inductively by specifying the generating $n$-morphisms whose domain
and codomain are $(n - 1)$-morphisms in a free $(n - 1)$-category.
This interleaved process is not ideal as the free $k$-categories are very hard to describe explicitly.
Because manifold diagrams are closed under composition and isotopy, we can sidestep the intermediate steps
and specify the generating data for all dimensions at once.

String diagrams were formalised as mathematical objects in~\cite{geometry-tensor-calculus}
via topological graphs. The Gray surface diagram calculus of~\cite{gray-category-diagrams}
presents diagrams as a stratification of the
closed $3$-cube $\intCC{0, 1}^3$ which satisfies a series of conditions.
Trimble outlined the beginnings of an approach to $n$-manifold diagrams for arbitrary $n \geq 0$ in an online note~\cite{trimble-surface}, which unfortunately has remained incomplete.
Dorn and Douglas described manifold diagrams~\cite{manifold-diagrams-tame-tangles} in arbitrary dimension for the purpose
of studying the combinatorialisability of smooth structures. However, they do not
consider how such diagrams would map to higher categorical structures.

Trimble's note identified five essential ingredients that a theory of manifold diagrams would be built upon: \textit{tameness}, \textit{coincidences}, \textit{progressivity}, \textit{globularity} and \textit{deformations}. We address each of these ingredients as follows.

Tameness seeks to avoid the wild phenomena that topology is famously~\cite{grothendieck-sketch} filled with, such as space filling curves or Cantor spaces, which are far removed from any intuitive concept of a geometric shape.
Smooth geometry corrects some of them but also comes with various surprising features like infinite oscillations and spirals.
While there is a wealth of tools that are developed to deal with smooth manifolds,
the toolset for smooth stratified spaces is significantly sparser. 
Smoothness is also non-trivial to maintain under many constructions.
Algebraic or analytic geometry is too rigid, since we need bump functions to glue diagrams.
Trimble's note on surface diagrams therefore recommends o-minimal geometry, and specifically the o-minimal structure $\R_{\exp}$ which extends algebraic sets with the exponential function.
The surface diagrams for Gray categories~\cite{gray-category-diagrams} achieve tameness with piecewise linear geometry.
We follow their lead and base our notion of tameness on finite triangulations.

The string-diagram calculus for braided monoidal categories detects edge crossings and translates them into the braiding isomorphism.
Similarly, isotopies of Gray surface diagrams are translated into applications of the various coherence isomorphisms in the structure of a Gray $3$-category.
Lurie's proof sketch for the cobordism hypothesis uses Morse theory to decompose manifolds into a finite sequence of segments, connected via Morse singularities that correspond to the higher cells generated by dualisability.
For a manifold diagram calculus we need a similar way to detect and classify critical points at which the relative positioning of strata changes in a way that is algebraically significant.

From the example of an edge crossing in a string diagram we can intuit that these critical points
are connected to projection:
when we project away the coordinate direction in which one edge passes over the other,
the projected edges intersect exactly at the height where the crossing occurred.
In~\cite{trimble-surface} these critical points are described in terms of \textit{coincidence}.
It is common to define geometric structure by specifying the transformations that preserve it.
Framed combinatorial topology~\cite{framed-combinatorial-topology, framed-combinatorial-topology-infty-cats} is a theory of tame geometry that is based around the concept of coincidence.
In~\S\ref{sec:fct-framing} we define a sequence of equivalence relations $\sim_i$ on $\R^n$
so that $x \sim_i y$ whenever the last $i$ coordinates of $x$ and $y$ agree.
We then equip stratified spaces with an embedding into $\R^n$, called an $n$-framing, and only allow
those transformations that are continuous, respect the stratification and preserve all coincidence relations $\sim_i$.

Tame $n$-framed stratified spaces such as $n$-manifold diagrams can be refined
by an $n$-mesh, making the structure that is encoded in the $n$-framing visible
in the stratification (see for example Figure~\ref{fig:i-braid-mesh}).
Meshes are defined inductively in~\S\ref{sec:fct-mesh} as a sequence of
$1$-dimensional mesh bundles, each adding one dimension at a time.
We concentrate on two types of meshes: open meshes that model the geometry of manifold diagrams,
and closed meshes that closely resemble commutative or pasting diagrams in higher dimensions.
Open and closed mesh bundles are classified by quasicategories $\BMeshOpen{n}$ and $\BMeshClosed{n}$.
Meshes admit purely combinatorial and finite descriptions in terms of trusses,
which we explore in~\S\ref{sec:fct-truss}.
Trusses are classified by $1$-categories $\BTrussOpen{n} \simeq \BMeshOpen{n}$ and $\BTrussClosed{n} \simeq \BMeshClosed{n}$ (see Figure~\ref{fig:i-mesh-truss}).

\begin{figure}[ht]
  \[
    \begin{tikzpicture}[scale = 0.5, baseline=(current bounding box.center)]
      \begin{scope}[shift={(-2, -2)}]
        \draw[->] (0,0,0) -- (1,0,0) node[anchor=west]{\tiny{2}};
        \draw[->] (0,0,0) -- (0,1,0) node[anchor=south]{\tiny{3}};
        \draw[->] (0,0,0) -- (0,0,1) node[anchor=north east]{\tiny{1}};
      \end{scope}
    
      \begin{scope}[canvas is xz plane at y = 0]
        \coordinate (s0-a) at (3, 1);
        \coordinate (s0-b) at (1, 2);
      \end{scope}
      \begin{scope}[canvas is xz plane at y = 6]
        \coordinate (s2-a) at (1, 1);
        \coordinate (s2-b) at (3, 2);
      \end{scope}

      \draw[mesh3-bounding-box]
        (0, 0, 0) --
        (4, 0, 0) -- 
        (4, 0, 3) -- 
        (0, 0, 3) -- cycle;
      \draw[mesh3-bounding-box]
        (0, 6, 0) --
        (4, 6, 0) -- 
        (4, 6, 3) -- 
        (0, 6, 3) -- cycle;
      \draw[mesh3-bounding-box] (0, 0, 0) -- (0, 6, 0);
      \draw[mesh3-bounding-box] (4, 0, 0) -- (4, 6, 0); 
      \draw[mesh3-bounding-box] (4, 0, 3) -- (4, 6, 3); 
      \draw[mesh3-bounding-box] (0, 0, 3) -- (0, 6, 3);

      \draw[mesh3-stratum-orange] (s0-a) -- (s2-a);
      \draw[mesh3-stratum-blue] (s0-b) -- (s2-b);
    \end{tikzpicture}
    \qquad
    \begin{tikzpicture}[scale = 0.5, baseline=(current bounding box.center)]
      \begin{scope}[shift={(-2, -2)}]
        \draw[->] (0,0,0) -- (1,0,0) node[anchor=west]{\tiny{2}};
        \draw[->] (0,0,0) -- (0,1,0) node[anchor=south]{\tiny{3}};
        \draw[->] (0,0,0) -- (0,0,1) node[anchor=north east]{\tiny{1}};
      \end{scope}
    
      \begin{scope}[canvas is xz plane at y = 0]
        \coordinate (s0-a) at (3, 1);
        \coordinate (s0-b) at (1, 2);

        \coordinate (s0-a-front) at (3, 3);
        \coordinate (s0-a-back) at (3, 0);

        \coordinate (s0-b-front) at (1, 3);
        \coordinate (s0-b-back) at (1, 0);
      \end{scope}
      \begin{scope}[canvas is xz plane at y = 6]
        \coordinate (s2-a) at (1, 1);
        \coordinate (s2-b) at (3, 2);

        \coordinate (s2-a-front) at (1, 3);
        \coordinate (s2-a-back) at (1, 0);

        \coordinate (s2-b-front) at (3, 3);
        \coordinate (s2-b-back) at (3, 0);
      \end{scope}

      \draw[mesh3-bounding-box]
        (0, 0, 0) --
        (4, 0, 0) -- 
        (4, 0, 3) -- 
        (0, 0, 3) -- cycle;
      \draw[mesh3-bounding-box]
        (0, 6, 0) --
        (4, 6, 0) -- 
        (4, 6, 3) -- 
        (0, 6, 3) -- cycle;
      \draw[mesh3-bounding-box] (0, 0, 0) -- (0, 6, 0);
      \draw[mesh3-bounding-box] (4, 0, 0) -- (4, 6, 0); 
      \draw[mesh3-bounding-box] (4, 0, 3) -- (4, 6, 3); 
      \draw[mesh3-bounding-box] (0, 0, 3) -- (0, 6, 3);

      \fill[mesh3-surface] (s0-a-front) -- (s0-a-back) -- (s2-a-back) -- (s2-a-front) -- cycle;
      \fill[mesh3-surface] (s0-b-front) -- (s0-b-back) -- (s2-b-back) -- (s2-b-front) -- cycle;
  
      \begin{scope}[canvas is xz plane at y = 3]
        \fill[mesh3-surface] (0, 0) -- (4, 0) -- (4, 3) -- (0, 3) -- cycle;
      \end{scope}

      \draw[mesh3-stratum-orange] (s0-a) -- (s2-a);
      \draw[mesh3-stratum-blue] (s0-b) -- (s2-b);

      \begin{scope}[canvas is xz plane at y = 3]
        \draw[mesh3-stratum] (2, 0) -- (2, 3);
        \node[mesh3-vertex] at (2, 1) {};
        \node[mesh3-vertex] at (2, 2) {};
      \end{scope}
    \end{tikzpicture}
  \]
  \caption{\label{fig:i-braid-mesh}
    The $3$-diagram which braids a blue and an orange string,
    and its corresponding open $3$-mesh which detects the critical moment
    at which the two strata pass over each other.
  }
\end{figure}
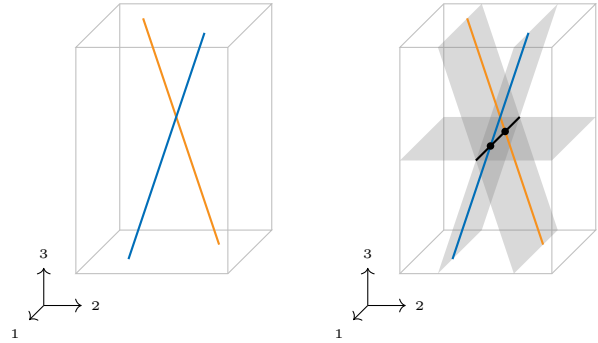

\begin{figure}[ht]
    \[
      \begin{tikzpicture}[xscale = 0.45, yscale=0.45, baseline=(current bounding box.center)]
        \fill[mesh-background] (0, -1) rectangle (9, 7);
        \draw[mesh-stratum] (1, -1) -- (1, 1) .. controls +(0, 0.2) and +(-1, 0) .. (2, 2);
        \draw[mesh-stratum] (3, -1) -- (3, 1) .. controls +(0, 0.2) and +(1, 0) .. (2, 2);
        \draw[mesh-stratum] (2, 2) -- (2, 5) .. controls +(0, 0.2) and +(-1, 0) .. (3, 6);
        \draw[mesh-stratum] (5, 0) -- (5, 4);
        \draw[mesh-stratum] (5, 4) .. controls +(-1, 0) and +(0, -0.2) .. (4, 5) .. controls +(0, 0.2) and +(1, 0) .. (3, 6);
        \draw[mesh-stratum] (5, 4) .. controls +(1, 0) and +(0, -0.2) .. (6, 5) .. controls +(0, 0.2) and +(-1, 0) .. (7, 6);
        \draw[mesh-stratum] (8, -1) -- (8, 5) .. controls +(0, 0.2) and +(1, 0) .. (7, 6);
        \draw[mesh-stratum] (7, 6) -- (7, 7);
        \node[mesh-vertex] at (5, 0) {};
        \node[mesh-vertex] at (5, 4) {};
        \node[mesh-vertex] at (2, 2) {};
        \node[mesh-vertex] at (7, 6) {};
        \node[mesh-vertex] at (3, 6) {};
      \end{tikzpicture}
      \quad
      \begin{tikzpicture}[xscale = 0.45, yscale=0.45, baseline=(current bounding box.center)]
        \fill[mesh-background] (0, -1) rectangle (9, 7);
        \draw[mesh-stratum] (1, -1) -- (1, 1) .. controls +(0, 0.2) and +(-1, 0) .. (2, 2);
        \draw[mesh-stratum] (3, -1) -- (3, 1) .. controls +(0, 0.2) and +(1, 0) .. (2, 2);
        \draw[mesh-stratum] (2, 2) -- (2, 5) .. controls +(0, 0.2) and +(-1, 0) .. (3, 6);
        \draw[mesh-stratum] (5, 0) -- (5, 4);
        \draw[mesh-stratum] (5, 4) .. controls +(-1, 0) and +(0, -0.2) .. (4, 5) .. controls +(0, 0.2) and +(1, 0) .. (3, 6);
        \draw[mesh-stratum] (5, 4) .. controls +(1, 0) and +(0, -0.2) .. (6, 5) .. controls +(0, 0.2) and +(-1, 0) .. (7, 6);
        \draw[mesh-stratum] (8, -1) -- (8, 5) .. controls +(0, 0.2) and +(1, 0) .. (7, 6);
        \draw[mesh-stratum] (7, 6) -- (7, 7);
        \node[mesh-vertex] at (5, 0) {};
        \node[mesh-vertex] at (5, 4) {};
        \node[mesh-vertex] at (2, 4) {};
        \node[mesh-vertex] at (8, 4) {};
        \node[mesh-vertex] at (2, 2) {};
        \node[mesh-vertex] at (5, 2) {};
        \node[mesh-vertex] at (8, 2) {};
        \node[mesh-vertex] at (7, 6) {};
        \node[mesh-vertex] at (3, 6) {};
        \node[mesh-vertex] at (1, 0) {};
        \node[mesh-vertex] at (3, 0) {};
        \node[mesh-vertex] at (8, 0) {};

        \draw[mesh-stratum] (0, 0) -- (9, 0);
        \draw[mesh-stratum] (0, 2) -- (9, 2);
        \draw[mesh-stratum] (0, 4) -- (9, 4);
        \draw[mesh-stratum] (0, 6) -- (9, 6);
      \end{tikzpicture}
      \quad
    \begin{tikzpicture}[xscale = 0.45, yscale = 0.45, baseline=(current bounding box.center)]
      \node [font=\tiny] (r-0-0) at (0, -1) {2};
      \node [font=\tiny] (r-0-1) at (1, -1) {1};
      \node [font=\tiny] (r-0-2) at (2, -1) {2};
      \node [font=\tiny] (r-0-3) at (3, -1) {1};
      \node [font=\tiny] (r-0-4) at (5, -1) {2};
      \node [font=\tiny] (r-0-5) at (8, -1) {1};
      \node [font=\tiny] (r-0-6) at (9, -1) {2};
      \draw[->] (r-0-1) -- (r-0-0);
      \draw[->] (r-0-1) -- (r-0-2);
      \draw[->] (r-0-3) -- (r-0-2);
      \draw[->] (r-0-3) -- (r-0-4);
      \draw[->] (r-0-5) -- (r-0-4);
      \draw[->] (r-0-5) -- (r-0-6);

      \node [font=\tiny] (s-0-0) at (0, 0) {2};
      \node [font=\tiny] (s-0-1) at (1, 0) {1};
      \node [font=\tiny] (s-0-2) at (2, 0) {2};
      \node [font=\tiny] (s-0-3) at (3, 0) {1};
      \node [font=\tiny] (s-0-4) at (4, 0) {2};
      \node [font=\tiny] (s-0-5) at (5, 0) {0};
      \node [font=\tiny] (s-0-6) at (6, 0) {2};
      \node [font=\tiny] (s-0-7) at (8, 0) {1};
      \node [font=\tiny] (s-0-8) at (9, 0) {2};
      \draw[->] (s-0-1) -- (s-0-0);
      \draw[->] (s-0-1) -- (s-0-2);
      \draw[->] (s-0-3) -- (s-0-2);
      \draw[->] (s-0-3) -- (s-0-4);
      \draw[->] (s-0-5) -- (s-0-4);
      \draw[->] (s-0-5) -- (s-0-6);
      \draw[->] (s-0-7) -- (s-0-6);
      \draw[->] (s-0-7) -- (s-0-8);

      \draw[->] (s-0-0) -- (r-0-0);
      \draw[->] (s-0-1) -- (r-0-1);
      \draw[->] (s-0-2) -- (r-0-2);
      \draw[->] (s-0-2) -- (r-0-2);
      \draw[->] (s-0-3) -- (r-0-3);
      \draw[->] (s-0-4) -- (r-0-4);
      \draw[->] (s-0-6) -- (r-0-4);
      \draw[->] (s-0-7) -- (r-0-5);
      \draw[->] (s-0-8) -- (r-0-6);
      
      \node [font=\tiny] (r-1-0) at (0, 1) {2};
      \node [font=\tiny] (r-1-1) at (1, 1) {1};
      \node [font=\tiny] (r-1-2) at (2, 1) {2};
      \node [font=\tiny] (r-1-3) at (3, 1) {1};
      \node [font=\tiny] (r-1-4) at (4, 1) {2};
      \node [font=\tiny] (r-1-5) at (5, 1) {1};
      \node [font=\tiny] (r-1-6) at (6, 1) {2};
      \node [font=\tiny] (r-1-7) at (8, 1) {1};
      \node [font=\tiny] (r-1-8) at (9, 1) {2};
      \draw[->] (r-1-1) -- (r-1-0);
      \draw[->] (r-1-1) -- (r-1-2);
      \draw[->] (r-1-3) -- (r-1-2);
      \draw[->] (r-1-3) -- (r-1-4);
      \draw[->] (r-1-5) -- (r-1-4);
      \draw[->] (r-1-5) -- (r-1-6);
      \draw[->] (r-1-7) -- (r-1-6);
      \draw[->] (r-1-7) -- (r-1-8);

      \draw[->] (s-0-0) -- (r-1-0);
      \draw[->] (s-0-1) -- (r-1-1);
      \draw[->] (s-0-2) -- (r-1-2);
      \draw[->] (s-0-3) -- (r-1-3);
      \draw[->] (s-0-4) -- (r-1-4);
      \draw[->] (s-0-5) -- (r-1-5);
      \draw[->] (s-0-6) -- (r-1-6);
      \draw[->] (s-0-7) -- (r-1-7);
      \draw[->] (s-0-8) -- (r-1-8);

      \node [font=\tiny] (s-1-0) at (0, 2) {2};
      \node [font=\tiny] (s-1-1) at (2, 2) {0};
      \node [font=\tiny] (s-1-2) at (4, 2) {2};
      \node [font=\tiny] (s-1-3) at (5, 2) {1};
      \node [font=\tiny] (s-1-4) at (6, 2) {2};
      \node [font=\tiny] (s-1-5) at (8, 2) {1};
      \node [font=\tiny] (s-1-6) at (9, 2) {2};
      \draw[->] (s-1-1) -- (s-1-0);
      \draw[->] (s-1-1) -- (s-1-2);
      \draw[->] (s-1-3) -- (s-1-2);
      \draw[->] (s-1-3) -- (s-1-4);
      \draw[->] (s-1-5) -- (s-1-4);
      \draw[->] (s-1-5) -- (s-1-6);

      \draw[->] (s-1-0) -- (r-1-0);
      \draw[->] (s-1-1) -- (r-1-1);
      \draw[->] (s-1-1) -- (r-1-3);
      \draw[->] (s-1-2) -- (r-1-4);
      \draw[->] (s-1-3) -- (r-1-5);
      \draw[->] (s-1-4) -- (r-1-6);
      \draw[->] (s-1-5) -- (r-1-7);
      \draw[->] (s-1-6) -- (r-1-8);

      \node [font=\tiny] (r-2-0) at (0, 3) {2};
      \node [font=\tiny] (r-2-1) at (2, 3) {1};
      \node [font=\tiny] (r-2-2) at (4, 3) {2};
      \node [font=\tiny] (r-2-3) at (5, 3) {1};
      \node [font=\tiny] (r-2-4) at (6, 3) {2};
      \node [font=\tiny] (r-2-5) at (8, 3) {1};
      \node [font=\tiny] (r-2-6) at (9, 3) {2};
      \draw[->] (r-2-1) -- (r-2-0);
      \draw[->] (r-2-1) -- (r-2-2);
      \draw[->] (r-2-3) -- (r-2-2);
      \draw[->] (r-2-3) -- (r-2-4);
      \draw[->] (r-2-5) -- (r-2-4);
      \draw[->] (r-2-5) -- (r-2-6);

      \draw[->] (s-1-0) -- (r-2-0);
      \draw[->] (s-1-1) -- (r-2-1);
      \draw[->] (s-1-2) -- (r-2-2);
      \draw[->] (s-1-3) -- (r-2-3);
      \draw[->] (s-1-4) -- (r-2-4);
      \draw[->] (s-1-5) -- (r-2-5);
      \draw[->] (s-1-6) -- (r-2-6);

      \node [font=\tiny] (s-2-0) at (0, 4) {2};
      \node [font=\tiny] (s-2-1) at (2, 4) {1};
      \node [font=\tiny] (s-2-2) at (4, 4) {2};
      \node [font=\tiny] (s-2-3) at (5, 4) {0};
      \node [font=\tiny] (s-2-4) at (6, 4) {2};
      \node [font=\tiny] (s-2-5) at (8, 4) {1};
      \node [font=\tiny] (s-2-6) at (9, 4) {2};
      \draw[->] (s-2-1) -- (s-2-0);
      \draw[->] (s-2-1) -- (s-2-2);
      \draw[->] (s-2-3) -- (s-2-2);
      \draw[->] (s-2-3) -- (s-2-4);
      \draw[->] (s-2-5) -- (s-2-4);
      \draw[->] (s-2-5) -- (s-2-6);

      \draw[->] (s-2-0) -- (r-2-0);
      \draw[->] (s-2-1) -- (r-2-1);
      \draw[->] (s-2-2) -- (r-2-2);
      \draw[->] (s-2-3) -- (r-2-3);
      \draw[->] (s-2-4) -- (r-2-4);
      \draw[->] (s-2-5) -- (r-2-5);
      \draw[->] (s-2-6) -- (r-2-6);

      \node [font=\tiny] (r-3-0) at (0, 5) {2};
      \node [font=\tiny] (r-3-1) at (2, 5) {1};
      \node [font=\tiny] (r-3-2) at (3, 5) {2};
      \node [font=\tiny] (r-3-3) at (4, 5) {1};
      \node [font=\tiny] (r-3-4) at (5, 5) {2};
      \node [font=\tiny] (r-3-5) at (6, 5) {1};
      \node [font=\tiny] (r-3-6) at (7, 5) {2};
      \node [font=\tiny] (r-3-7) at (8, 5) {1};
      \node [font=\tiny] (r-3-8) at (9, 5) {2};
      \draw[->] (r-3-1) -- (r-3-0);
      \draw[->] (r-3-1) -- (r-3-2);
      \draw[->] (r-3-3) -- (r-3-2);
      \draw[->] (r-3-3) -- (r-3-4);
      \draw[->] (r-3-5) -- (r-3-4);
      \draw[->] (r-3-5) -- (r-3-6);
      \draw[->] (r-3-7) -- (r-3-6);
      \draw[->] (r-3-7) -- (r-3-8);

      \draw[->] (s-2-0) -- (r-3-0);
      \draw[->] (s-2-1) -- (r-3-1);
      \draw[->] (s-2-2) -- (r-3-2);
      \draw[->] (s-2-3) -- (r-3-3);
      \draw[->] (s-2-3) -- (r-3-5);
      \draw[->] (s-2-4) -- (r-3-6);
      \draw[->] (s-2-5) -- (r-3-7);
      \draw[->] (s-2-6) -- (r-3-8);

      \node [font=\tiny] (s-3-0) at (0, 6) {2};
      \node [font=\tiny] (s-3-1) at (3, 6) {0};
      \node [font=\tiny] (s-3-2) at (5, 6) {2};
      \node [font=\tiny] (s-3-3) at (7, 6) {0};
      \node [font=\tiny] (s-3-4) at (9, 6) {2};
      \draw[->] (s-3-1) -- (s-3-0);
      \draw[->] (s-3-1) -- (s-3-2);
      \draw[->] (s-3-3) -- (s-3-2);
      \draw[->] (s-3-3) -- (s-3-4);

      \draw[->] (s-3-0) -- (r-3-0);
      \draw[->] (s-3-1) -- (r-3-1);
      \draw[->] (s-3-1) -- (r-3-3);
      \draw[->] (s-3-2) -- (r-3-4);
      \draw[->] (s-3-3) -- (r-3-5);
      \draw[->] (s-3-3) -- (r-3-7);
      \draw[->] (s-3-4) -- (r-3-8);

      \node [font=\tiny] (r-4-0) at (3, 7) {2};
      \node [font=\tiny] (r-4-1) at (7, 7) {1};
      \node [font=\tiny] (r-4-2) at (9, 7) {2};
      \draw[->] (r-4-1) -- (r-4-0);
      \draw[->] (r-4-1) -- (r-4-2);

      \draw[->] (s-3-0) -- (r-4-0);
      \draw[->] (s-3-2) -- (r-4-0);
      \draw[->] (s-3-3) -- (r-4-1);
      \draw[->] (s-3-4) -- (r-4-2);
    \end{tikzpicture}
  \]
    
  \caption{\label{fig:i-mesh-truss}
    Starting with a string diagram on the left,
    we can find the coarsest open $2$-mesh that refines it.
    The open $2$-mesh can then be described as a combinatorial object
    via the open $2$-truss on the right.
  }
\end{figure}
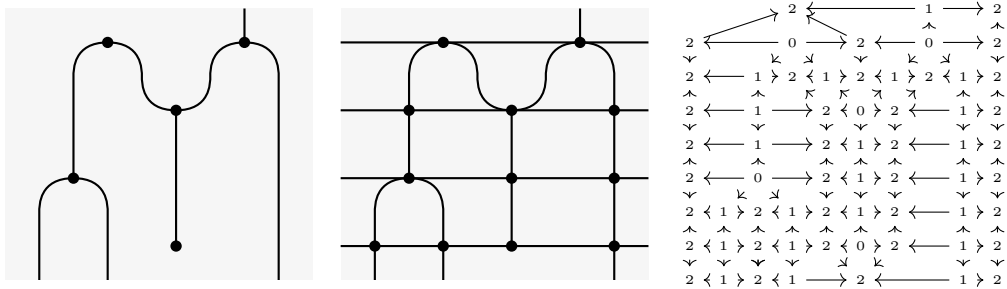

The edges in a string diagram have to progress in the direction from the source
to the target boundary of the diagram, without turning around or stalling.
This is the progressivity condition. 
The manifold diagrams in~\cite{manifold-diagrams-tame-tangles}
guarantee progressivity in higher dimensions by the framed conicality condition which requires that
each point has an open neighbourhood of the form $\cone \strat{L} \times \R^k$
where $\cone\strat{L}$ is a stratified cone.
Framed stratified isomorphisms can not exchange the coordinate order,
and so this framed conicality condition guarantees that on each $k$-dimensional
stratum $\strat{S}$ of an $n$-manifold diagram the projection $\R^n \to \R^k$ onto the last $k$ coordinates restricts to an immersion $\strat{S} \to \R^k$.
We modify the framed conicality requirement of~ 
\cite{manifold-diagrams-tame-tangles}
in~\S\ref{sec:d-mfld-local}
and require that the open neighbourhoods in an $n$-manifold diagram
are refined by mesh atoms, which we identify in~\S\ref{sec:d-diag-space}
as the appropriate geometric shapes for cells in $(\infty, n)$-categories.

The globularity condition of a string diagram controls which strata can intersect
which boundary of the diagram. In particular it restricts vertices to lie strictly
within the interior of the diagram, while edges can intersect only the top and bottom
boundary. In~\S\ref{sec:d-mfld-global} we describe a generalisation of the globularity
condition to arbitrary dimensions, which allows $n$-manifold diagrams to be modelled
on top of any arbitrary closed $n$-mesh. This includes $n$-cubes of varying size,
to ensure semi-strict composition, but also more complex shapes that describe pasting
diagrams of $n$-manifold diagrams. 

The last ingredient for $n$-manifold diagrams are their deformations or isotopies.
Isotopies of string diagrams for symmetric monoidal categories
are required to preserve the order of the incoming and outgoing edges of
every vertex, at least within some neighbourhood around the vertex. Without that condition, any map $A \otimes A \to A$ expressed
in a string diagram would have to automatically be commutative.
The ordering condition for isotopies of string diagrams does not trivially generalise to isotopies of diagrams in higher dimensions since a cell can have adjacent cells of various dimensions in relative arrangements that can not be reduced to a linear order. 
The notes~\cite{trimble-surface} identify this problem as well.
Our solution in~\S\ref{sec:d-mfld-isotopy} requires that the $n$-framed
stratified bundles which represent isotopies of $n$-manifold diagrams
are $n$-framed stratified submersions, i.e.\ trivial up to $n$-framed
stratified isomorphism within a neighbourhood of any stratum.
This notion leaves enough flexibility for the interchange laws to arise as isotopies
while being strict enough to not impose additional relations on the cells
in the modelled $(\infty, n)$-categories.

Free plain, braided or symmetric monoidal categories can be generated via their
string diagram calculi by attaching labels to the strata, often presented via colours.
We demonstrate how to attach labels to manifold diagrams in~\S\ref{sec:d-label},
characterising systems of admissible labels as diagrammatic computads.
Via the equivalence between meshes and trusses, such diagrammatic $n$-computads
can be seen alternatively but equivalently as combinatorial or geometric objects.
Using the theory of transverse configurations in FCT, developed in~\S\ref{sec:fct-conf},
we can show that labelled manifold diagrams create free $(\infty, n)$-categories
from $n$-diagrammatic computads via an adjunction of $(\infty, 1)$-categories
$\DiagFree : \Cptd{n} \rightleftarrows \CatN{n} : \DiagForget$.

As a two-dimensional graphical calculus, string diagrams can be represented on paper, blackboards and screens.
Three-dimensional surface diagrams can still be rendered on a flat surface with projection, using artistic tricks to make them intelligible as needed.
The useability of a graphical notation then quickly declines when the dimension is increased beyond three.
Multiple decades had passed between appearances of string diagrams in the literature~\cite{neumann-string, penrose-diagrams} and their formalisation in~\cite{geometry-tensor-calculus}.
The authors of the latter paper partly attribute this delay to technological limitations
in printing at the time, which prevented graphical notations to escape from
the blackboard onto the printed page at scale.
The \textsc{homotopy.io} project posits that the use of manifold diagrams in higher dimensions similarly can be helped through a technological improvement.

\textsc{homotopy.io} is a graphical proof assistant that enables a user
to inspect and manipulate low-dimensional slices and projections through an (in principle) arbitrary dimensional diagram in an $\omega$-category freely generated by a signature.
The conceptual model behind the freely generated $\omega$-category used for
\textsc{homotopy.io} originates from~\cite{dorn-thesis}, but has not previously been shown equivalent to established models.
Via our combinatorialisability results for manifold diagrams labelled in a diagrammatic computads, discussed in~\S\ref{sec:d-label}, we can describe the
free $\omega$-category generated by a \textsc{homotopy.io} signature
as an $n$-fold Segal space where $n$ is the dimension of the highest dimensional generator in the (finite) signature.
In this way we demonstrate that the proof assistant has sound semantics in higher categories as claimed.

\section{Related Work}

The projects surrounding the cobordism hypothesis~\cite{lurie-tqft, cobordism-hypothesis-note, geometric-cobordism-hypothesis} and factorisation homology~\cite{ayala-factorization-homology} also establish a connection between geometry and category theory in higher dimensions, building upon decades of deep insights into the geometry of smooth manifolds.
These theories aim to study existing geometries that are of prior interest via the language and toolset of higher-dimensional categories.
Our aim in this thesis is to do the opposite:
we start with higher categories and want to elucidate their structure using geometry.
We therefore are afforded the opportunity to pick framed combinatorial topology as a theory of geometry that is tailored specifically to our purpose.

The graphical proof assistant Globular~\cite{globular-proof-assistant}
implemented a combinatorial encoding of the geometry of higher-dimensional
diagrams, based on the data structure described in~\cite{quasistrict-datastructures}.
This data structure, and therefore also the proof assistant, relied
on an explicit classification of generators for the weak interchange law,
such as braids, braid naturality, etc. So far this explicit classification
was not extended beyond dimension four, which limits the expressive power of the system.
Dorn's thesis~\cite{dorn-thesis} is the conceptual origin of the ideas of
framed combinatorial topology. Instead of attempting to classify algebraic
generators for the weak interchange structure, the singular interval bundles
of~\cite{dorn-thesis} mirror the geometry of higher-dimensional diagrams
directly.
This theory admits an alternative mathematical formulation in terms
of the zigzag construction, an implementation as a data structure,
and algorithms to construct and type check isotopies of diagrams step by step~\cite{zigzag-contraction, zigzag-normalisation}.
The proof assistant \textsc{homotopy.io}~\cite{homotopy-io} was then developed based on these results.

The framed combinatorial topology (FCT) project~\cite{framed-combinatorial-topology}
aimed to apply the combinatorial theory of geometry beyond higher categorical
diagrams, with a focus on computability in combinatorial topology.
FCT introduced trusses, derived from the singular interval bundles of~\cite{dorn-thesis},
and meshes, which serve as a bridge
between the geometry of stratified spaces in $\R^n$ and the combinatorics of trusses.
This was applied in~\cite{manifold-diagrams-tame-tangles} to explore combinatorialisability
of smooth structures.
Trusses and meshes were discovered independently in the author's attempts to
give geometric meaning to the zigzag construction of~\cite{zigzag-contraction, zigzag-normalisation}, and were realised with techniques of higher category theory in~\cite{framed-combinatorial-topology-infty-cats}. To avoid fragmentation, we have decided to
adopt the terminology of framed combinatorial topology.
Our formulation of meshes in terms of quasicategories is more technically
convenient for the purposes of this thesis. Moreover, our meshes are not required
to have continuous boundaries, allowing us to express embeddings between meshes
as the inert bordisms in $\BMeshOpen{n}$ and $\BMeshClosed{n}$,
described in~\S\ref{sec:fct-embed}. This is essential for the theory of
diagrammatic spaces which we develop in~\S\ref{sec:d-diag-space}.

The idea of strictifying higher categorical structures by probing them with
shapes that admit strict composition operations themselves originated from~\cite{kv-strictify},
which attempted to strictify associativity, unitality and the interchange law for $\infty$-groupoids.
This was shown to be impossible by Simpson in~\cite{simpsonHomotopyTypesStrict1998},
who suggested compromising via a semi-strict model in which unitality is weak.
The diagrammatic sets of~\cite{amar-diagrammatic-sets} build upon this idea and are conjectured to be a semi-strict model for $\omega$-categories, particularly well-suited for defining the Gray tensor product~\cite{amar-smash-product}
and admitting a computer implementation~\cite{amar-data-structures, amar-matching}.
While this model does admit a geometric realisation similar to manifold diagrams,
each cell must have at least one stratum in its domain and codomain, respectively.
These additional strata prevent any non-trivial braiding to be performed directly by isotopy,
instead relying on the coherence structure of the weak units.

The theory of $(\infty, n)$-categories is determined essentially uniquely,
up to a choice of direction for each dimension~\cite{unicity-infty-n}.
We use complete $n$-fold Segal spaces as the model for $(\infty, n)$-categories,
originally defined in~\cite{barwick-n-fold-segal} and popularised by~\cite{lurie-tqft}.
Our construction of the $n$-diagrammatic spaces $\MfldDiag^n(-)$ and $\MfldDiagExt^n(-)$
with $n$-manifold diagrams that are orthogonal to a closed $n$-mesh is derived
from the technique used in~\cite{lurie-tqft} to construct the $n$-fold Segal space
of $n$-corbordisms.
Our universal bundles
$\MeshClosed{n} : \EMeshClosed{n} \to \BMeshClosed{n}$
and 
$\MeshOpen{n} : \EMeshOpen{n} \to \BMeshOpen{n}$
that classify closed and open $n$-meshes
are inspired by the bundle $\mathcal{E}\textup{xit} \to \mathcal{B}\textup{un}$
of~\cite{stratified-homotopy-hypothesis}, although their concrete implementation as quasicategories differs substantially.
Our $n$-diagrammatic computads were loosely inspired by the geometric structures of~\cite{geometric-cobordism-hypothesis}.

\section{Acknowledgements}

I would like to thank my advisor Jamie Vicary for his exceptional support
and the freedom to pursue my ideas in the directions I found most promising.
I fondly look back
at the time spent with Alex, Calin, Chiara, Cole, Ioannis, Nathan, Nick, Nick, Sivert and Vincent in
the Oxford and Cambridge offices and at numerous conferences.
My thanks also go to Sam Staton who readily picked up the responsibility of
advisor in Oxford after Jamie's move to Cambridge;
to Bob Coecke, Urs Schreiber, Benjamin Kästner, Claus Kirchner and Janis Voigtländer who supported my application to the department;
to Denis-Charles Cisinski, Catherine Meusburger, Thomas Nikolaus and Bernhard von Stengel
for their support with my application for funding;
and to Andr\'e Henriques, Aleks Kissinger and Sam Staton for volunteering their time
for the transfer and confirmation examinations.
My studies have been made possible via the funding of the Studienstiftung des Deutschen Volkes,
and later through the support by Quantinuum enabled by Ross Duncan.
I want to thank Chris Dorn, Chris Douglas and Lukas Waas for many insightful discussions
about framed stratified topology, stratified spaces and beyond.
I am grateful to my parents and my friends for their unwavering support and encouragement,
and of course to Percyval for helping in his own way.

%% file: chapter-background.tex
\chapter{Background}\label{sec:b}

In this chapter we broadly review the notions from the theories of $(\infty, 1)$-categories,
polynomial functors and stratified spaces that form the basis of our constructions.
We assume knowledge of ordinary category theory. We do not claim originality
for any concepts in this chapter beyond their presentation.

\section{\texorpdfstring{$(\infty, 1)$-Categories}{Infinity Categories}}

An $(\infty, 1)$-category is an $\omega$-category in which all $k$-morphisms for $k > 1$ are invertible.
Equivalently, we can see an $(\infty, 1)$-category $\cat{C}$ as a collection of objects together with a space
of $1$-morphisms $\cat{C}(X, Y)$ for any pair of objects, equipped with a composition operation that is associative and unital up to homotopy.
$(\infty, 1)$-categories are the natural place for homotopy coherent mathematics,
and we will make constant use of them throughout this thesis.
The theory of $(\infty, 1)$-categories closely matches the theory of ordinary $1$-categories.
We will commonly refer to $(\infty, 1)$-categories simply as $\infty$-categories.
For a proper introduction to $\infty$-category theory we refer to~\cite{cisinski-infty-cats, rezk-intro-quasicats, htt}.
Most of the definitions and results in this section are common,
and can be found in the cited introductory texts;
we will only cite specific items that are not as well known.

\subsection{Quasicategories and Kan Complexes}

Similar to picking a basis in linear algebra, we pick a model for $\infty$-categories.
By now there are multiple choices of such a model which have been shown to be equivalent in the appropriate way.
For the purposes of this thesis we mostly work with quasicategories,
which are the most well developed at this point.

\begin{definition}
  We denote by $\FinOrd$ the category of finite non-empty ordinals,
  whose objects are the linearly ordered sets
  $
    \ord{n} := \{ 0 \leq 1 \leq \cdots \leq n \}
  $
  for all natural numbers $n \geq 0$.
  The maps in $\FinOrd$ are order-preserving functions.
  We denote by for some $n, m \geq 0$ and any monotone sequence
  $0 \leq i_0 \leq \cdots \leq i_n \leq m$ we denote by
  $\langle i_0, \ldots, i_k \rangle$
  the map $\ord{n} \to \ord{m}$ which sends $j \in \ord{n}$ to $i_j \in \ord{m}$.
  A \defn{simplicial set} is a functor $\FinOrd^\op \to \Set$,
  and a map of simplicial sets is a natural transformation of functors.
  We denote by $\sSet$ the $1$-category of simplicial sets.
  For $k \geq 0$ we write $\Delta\ord{k}$ for the simplicial set
  given by the representable functor $\sSet(-, \ord{k})$.
  When $X$ is a simplicial set, we call an element $X(\ord{k})$
  a \defn{$k$-simplex} of $X$; by the Yoneda lemma there is a natural bijection
  between $X(\ord{k})$ and the set of maps $\Delta\ord{k} \to X$.
  The \defn{boundary} of the $k$-simplex $\Delta\ord{k}$ is the union of all subsimplices $\Delta\ord{k - 1} \hookrightarrow \Delta\ord{k}$.
  For $0 \leq i \leq k$ the $i$th horn $\Lambda^i\ord{k}$ is the union of all
  subsimplices $\Delta\ord{k - 1} \hookrightarrow \Delta\ord{k}$
  except the one which corresponds to the inclusion
  $\ord{k - 1} \hookrightarrow \ord{k}$
  that misses $i \in \ord{k}$.
  We write $X \join Y$ for the \defn{join} of two simplicial sets $X$ and $Y$.
\end{definition}

\begin{para}
  Let $\cat{C}$ be a $1$-category and $f : A \to B$, $g : X \to Y$ maps in $\cat{C}$. 
  We write $f \LLP g$ when for every (solid) diagram in $\cat{C}$ of the form
  \[
    \begin{tikzcd}
      {A} \ar[r] \ar[d, "f"'] &
      {X} \ar[d, "g"] \\
      {B} \ar[r] \ar[ur, dashed] &
      {Y}
    \end{tikzcd}
  \]
  there exists a dashed map $B \to Y$ that makes the diagram commute.  
  In this case we say that $f$ has the \defn{left lifting property} with respect to
  $g$, or equivalently that $g$ has the \defn{right lifting property} with respect
  to $f$.  
  For a set $S$ of maps in $\cat{C}$ we write $\llp(S)$ for the subcategory
  of $\cat{C}$ consisting of maps $g$ such that $f \LLP g$ for all $f \in S$.
  The set $\llp(S)$ is closed under composition, pullbacks, products and retracts
  where they exist in $\cat{C}$.
  Dually we write $\rlp(S)$ for
  the subcategory of $\cat{C}$ consisting of maps $f$ such that $f \LLP g$.
  Then $\rlp(S)$ is closed under composition, pushouts, coproducts and retracts
  where they exist in $\cat{C}$.
\end{para}

\begin{para}
  A \defn{quasicategory} is a simplicial set $\cat{C}$ such that the unique
  map $!_{\cat{C}} : \cat{C} \to \Delta\ord{0}$ has the right lifting property with respect to
  all inner horn inclusions $\Lambda^i\ord{k} \hookrightarrow \Delta\ord{k}$ for $0 < i < k$.
  When $\cat{C}$ is a quasicategory, the \defn{objects} of $\cat{C}$ are the $0$-simplices
  and the \defn{maps} of $\cat{C}$ are the $1$-simplices.  
  An inner horn $\Lambda^1 \ord{2} \to \cat{C}$ is a composable sequence of two maps in $\cat{C}$,
  and the existence of a filler $\Delta\ord{2} \to \cat{C}$ guarantees that the sequence has a composite.
  The fillers of inner horns $\Lambda^i \ord{k} \to \cat{C}$ when $k > 2$ ensure that the composite
  is unique up to equivalence, and that composition is weakly associative and unital.
\end{para}

\begin{para}
  Every $1$-category $\cat{C}$ induces a quasicategory $\Nerve \cat{C}$ by taking its nerve.
  The nerve of $\cat{C}$ is the simplicial set whose $k$-simplices are functors
  $\ord{k} \to \cat{C}$, i.e.\ composable sequences in $\cat{C}$ of length $k$
  together with a choice of composites. Because composites in a $1$-category are
  uniquely determined, the nerve $\Nerve \cat{C}$ has unique fillers for all inner
  horns. Conversely, every simplicial set with unique fillers for all inner horns
  is isomorphic to the nerve of a $1$-category. Moreover, the nerve construction
  extends to a fully faithful functor
  $\Nerve : \Cat \hookrightarrow \sSet$.
  We will therefore conflate $1$-categories with their nerve.
\end{para}

\begin{para}
  The nerve functor $\Nerve : \Cat \hookrightarrow \sSet$ has a left adjoint
  $\Ho : \sSet \to \Cat$ which sends a simplicial set $X$ to its \defn{homotopy category}
  $\Ho(X)$. When $\cat{C}$ is a quasicategory, the homotopy category $\Ho(\cat{C})$
  admits the following description: The objects of $\Ho(\cat{C})$ are the $0$-simplices
  of $\cat{C}$ and the maps from $x$ to $y$ in $\Ho(\cat{C})$ are equivalence classes
  of $1$-simplices from $x$ to $y$ in $\cat{C}$, where two $1$-simplices are identified
  when there is a $2$-simplex
  \[
    \begin{tikzcd}
      {} &
      {y} \ar[dr, "\id"] \\
      {x} \ar[rr, "g"'] \ar[ur, "f"]&
      {} &
      {z}
    \end{tikzcd}
  \]
  Then the composition operation is uniquely determined, well-defined,
  unital and associative due to the inner horn fillers for $\cat{C}$.
  We say that a map in $\cat{C}$ is an \defn{equivalence} if it is sent to an
  isomorphism in $\Ho(\cat{C})$ by the adjunction unit 
  $\cat{C} \to \Nerve(\Ho(\cat{C}))$.
\end{para}

\begin{para}
  A \defn{Kan complex} is a simplicial set $X$ such that the unique map
  $!_X : X \to \Delta\ord{0}$ has the right lifting property with respect
  to all horn inclusions $\Lambda^i\ord{k} \hookrightarrow \Delta\ord{k}$.
  A Kan complex is therefore in particular a quasicategory. The fillers for the outer horns
  $\Lambda^0\ord{k}$ and $\Lambda^k\ord{k}$ guarantee that every map of a Kan complex
  is invertible. Kan complexes therefore serve as a model of $\infty$-groupoids or spaces.
\end{para}

\begin{para}
  For any topological space $X$ we can define a simplicial set $\Sing(X)$
  whose $k$-simplices are continuous maps from the topological standard
  simplex $\DeltaTop{k} \to X$. An object in $\Sing(X)$ is a point in $X$.
  A map is a path $\gamma : \DeltaTop{1} \to X$. Famously composition of
  paths with constant length can not be made associative or unital, but
  it is associative and unital up to homotopy.  
  Every topological horn inclusion $\HornTop{i}{k} \hookrightarrow \DeltaTop{k}$
  admits a retraction, with which we can construct horn fillers for $\Sing(X)$.
  Therefore, $\Sing(X)$ is a Kan complex.
  The functor $\Sing : \TopSp \to \sSet$ has a left adjoint,
  the geometric realisation functor $\TopReal{-} : \sSet \to \TopSp$.
\end{para}

\begin{para}
  When $\cat{C}$ is a quasicategory, the \defn{core} of $\cat{C}$ is the largest
  simplicial subset $\core(\cat{C}) \hookrightarrow \cat{C}$ whose $1$-simplices
  must be equivalences in $\cat{C}$. In other words, the core of $\cat{C}$ is
  the largest Kan complex contained within $\cat{C}$. 
\end{para}

\begin{para}
  When $\cat{C}$ and $\cat{D}$ are quasicategories, the product of simplicial
  sets $\cat{C} \times \cat{D}$ is a quasicategory as well. Similarly, the
  coproduct $\cat{C} \sqcup \cat{D}$ is again a quasicategory.
  For any simplicial set $X$ and a quasicategory $\cat{C}$ we write
  $\Fun(X, \cat{C})$ for the simplicial set whose $k$-simplices are maps
  $\Delta\ord{k} \times X \to \cat{C}$. Then $\Fun(X, \cat{C})$ is a quasicategory
  itself, whose objects are \defn{functors} from $X$ to $\cat{C}$ and whose maps
  are \defn{natural transformations}. We write $\CatInfty(X, \cat{C})$ for the
  core of $\Fun(X, \cat{C})$.
  The quasicategory of functors $\Fun(\Delta\ord{1}, \cat{C})$
  is the \defn{arrow category} of $\cat{C}$, equipped with functors
  $\dom, \cod : \Fun(\Delta\ord{1}, \cat{C}) \to \cat{C}$ that sends a map to
  its domain and codomain, respectively.
  For any pair of objects $x$, $y$ in $\cat{C}$ we can take the pullbacks
  of simplicial sets 
  \[
    \begin{tikzcd}
      {\cat{C}(x, y)} \ar[r] \ar[d] \pullbackcorner &
      {\Fun(\Delta\ord{1}, \cat{C})} \ar[d, "{(\dom, \cod)}"] \\
      {\Delta\ord{0}} \ar[r, "{(x, y)}", swap] &
      {\cat{C} \times \cat{C}}
    \end{tikzcd}
    \qquad
    \begin{tikzcd}
      {\cat{C}_{x/}} \ar[r] \ar[d] \pullbackcorner &
      {\Fun(\Delta\ord{1}, \cat{C})} \ar[d, "\dom"] \\
      {\Delta\ord{0}} \ar[r, "x", swap] &
      {\cat{C}}
    \end{tikzcd}
    \qquad
    \begin{tikzcd}
      {\cat{C}_{/y}} \ar[r] \ar[d] \pullbackcorner &
      {\Fun(\Delta\ord{1}, \cat{C})} \ar[d, "\cod"] \\
      {\Delta\ord{0}} \ar[r, "y", swap] &
      {\cat{C}}
    \end{tikzcd}
  \]
  The simplicial set $\cat{C}(x, y)$ is a Kan complex that represents the
  \defn{mapping space} from $x$ to $y$ in $\cat{C}$. The simplicial sets $\cat{C}_{x/ }$
  and $\cat{C}_{/y}$ are the \defn{slice categories} of $\cat{C}$ under $x$ and over $y$.
  More generally for functors $f : \cat{A} \to \cat{C}$ and $g : \cat{B} \to \cat{C}$ of quasicategories, we define the \defn{comma categories} via the pullbacks of simplicial sets
  \[
    \begin{tikzcd}
      {(f / g)} \ar[r] \ar[d] \pullbackcorner &
      {\Fun(\Delta\ord{1}, \cat{C})} \ar[d, "{(\dom, \cod)}"] \\
      {\cat{A} \times \cat{B}} \ar[r, "f \times g"'] &
      {\cat{C} \times \cat{C}}
    \end{tikzcd}
    \qquad
    \begin{tikzcd}
      {\cat{C}_{f /}} \ar[r] \ar[d] \pullbackcorner &
      {\Fun(\Delta\ord{1}, \cat{C})} \ar[d, "\dom"] \\
      {\cat{A}} \ar[r, "f"'] &
      {\cat{C}}
    \end{tikzcd}
    \qquad
    \begin{tikzcd}
      {\cat{C}_{/ g}} \ar[r] \ar[d] \pullbackcorner &
      {\Fun(\Delta\ord{1}, \cat{C})} \ar[d, "\cod"] \\
      {\cat{B}} \ar[r, "g"'] &
      {\cat{C}}
    \end{tikzcd}
  \]
\end{para}

\begin{para}
  A map of simplicial sets $f : X \to Y$ is a \defn{Kan equivalence} when
  its geometric realisation $|f| : |X| \to |Y|$ is a weak homotopy equivalence
  of topological spaces. The map $f : X \to Y$ is a \defn{categorical equivalence}
  when for every quasicategory $\cat{C}$ the map $\CatInfty(Y, \cat{C}) \to \CatInfty(X, \cat{C})$,
  induced by precomposition with $f$, is a Kan equivalence.
\end{para}

\begin{para}
  Let $f : \cat{C} \to \cat{D}$ be a functor between quasicategories.
  Then $f$ is \defn{fully faithful} if for each pair of objects $x$, $y$ in $\cat{C}$
  the induced map $\cat{C}(x, y) \to \cat{D}(f(x), f(y))$
  is a Kan equivalence.
  An object $d$ in $\cat{D}$ is in the \defn{essential image} of
  $f$ when there exists an object $c$ in $\cat{C}$ and an equivalence
  $f(c) \to d$ in $\cat{D}$.
  The functor $f$ is \defn{essentially surjective} when every object
  of $\cat{D}$ is in the essential image of $f$.
  A functor between quasicategories is a categorical equivalence if and only
  if it is fully faithful and essentially surjective.
\end{para}

\begin{para}
  A functor $f : \cat{C} \to \cat{D}$ between quasicategories is a categorical equivalence if and only
  if there exists a functor $g : \cat{D} \to \cat{C}$ and natural equivalences
  $f \circ g \simeq \id_\cat{D}$ as well as $g \circ f \simeq \id_\cat{C}$.
\end{para}

\begin{para}
  Let $f : X \to Y$ be a map of simplicial sets. Suppose that for every
  $k \geq 0$ and every (solid) diagram of simplicial sets
  \[
    \begin{tikzcd}
      {\partial\Delta\ord{k}} \ar[r] \ar[d, hook] &
      {X} \ar[d, "f"] \\
      {\Delta\ord{k}} \ar[r, "v"'] \ar[ur, dashed, "w"{description}] &
      {Y}
    \end{tikzcd}
  \]
  there exists a map $w : \Delta\ord{k} \to Y$ which makes the upper left
  square commute strictly together with a natural equivalence $f \circ w \simeq u$.  
  Then $f$ is a Kan equivalence by~\cite[Lemma~4.3]{quasicat-of-frames-cofibration}.
  Moreover if $Y$ is a Kan complex, then so is $X$.
  The analogous claim holds for categorical equivalences and quasicategories.
\end{para}

\begin{para}
  Let $\cat{C}$ be a quasicategory and $X \subseteq \cat{C}$ a simplicial
  set contained in $\cat{C}$. Then $X$ is a \defn{subcategory} of $\cat{C}$
  if it is closed under composition, i.e. when the inclusion map
  $X \hookrightarrow \cat{C}$ has the right lifting
  property with respect to the horn inclusion $\Lambda^1\ord{2} \to \Delta\ord{2}$.
  In this case $X \hookrightarrow \cat{C}$ has the right lifting property
  with respect to all $\Lambda^i\ord{k} \hookrightarrow \Delta\ord{k}$
  for $0 < i < k$ and so $X$ is itself a quasicategory.
  A subcategory $\cat{D} \subseteq \cat{C}$ is \defn{full} if the inclusion map
  $\cat{D} \hookrightarrow \cat{C}$ is fully faithful. A subcategory
  $\cat{D} \subseteq \cat{C}$ is \defn{wide} if $\cat{D}$ contains all 
  objects of $\cat{C}$.
\end{para}

\begin{para}
  A pair of functors $L : \cat{C} \to \cat{D}$ and $R : \cat{D} \to \cat{C}$
  between quasicategories form an adjunction when there exist natural
  transformations $\eta : \id_{\cat{D}} \to R \circ L$ and $\eps : L \circ R \to \id_{\cat{C}}$,
  called the \defn{unit} and \defn{counit},
  together with $2$-simplices
  \[
    \begin{tikzcd}
      {L} \ar[r, "\id_L \circ \eta"] \ar[dr, "\id_R"'] &
      {L \circ R \circ L} \ar[d, "\eps \circ \id_L"] \\
      {} &
      {L}
    \end{tikzcd}
    \qquad
    \begin{tikzcd}
      {R} \ar[r, "\eta \circ \id_R"] \ar[dr, "\id_R"'] &
      {R \circ L \circ R} \ar[d, "\id_R \circ \eps"] \\
      {} &
      {R}
    \end{tikzcd}
  \]
  in $\Fun(\cat{C}, \cat{D})$ and $\Fun(\cat{D}, \cat{C})$.
  In this case we write $L \dashv R$ and $L : \cat{C} \rightleftarrows \cat{D} : R$.
  The data of an adjunction is determined uniquely up to equivalence by just $L$ or
  $R$ individually, as long as it exists.
  The adjunction is an \defn{adjoint equivalence} whenever both the unit and
  counit are natural equivalences.
\end{para}

\begin{para}
  Let $I$ be a simplicial set.
  Then the left and right cones $I^\triangleleft$ and $I^\triangleright$
  are defined by the pushouts of simplicial sets
  \[
    \begin{tikzcd}[column sep = large]
      {\Delta\ord{0} \times I} \ar[hook, r, "{\langle 0 \rangle \times \id}"] \ar[d] \pushoutcorner &
      {\Delta\ord{1} \times I} \ar[d] \\
      {\Delta\ord{0}} \ar[r] &
      {I^\triangleleft}
    \end{tikzcd}
    \qquad
    \begin{tikzcd}[column sep = large]
      {\Delta\ord{0} \times I} \ar[hook, r, "{\langle 1 \rangle \times \id}"] \ar[d] \pushoutcorner &
      {\Delta\ord{1} \times I} \ar[d] \\
      {\Delta\ord{0}} \ar[r] &
      {I^\triangleright}
    \end{tikzcd}
  \]
  The cones let us define limits and colimits as follows:
  Suppose that $\cat{C}$ is a quasicategory and $\hat{f} : I^\triangleleft \to \cat{C}$ a map of
  simplicial sets which restricts to the map $f : I \to \cat{C}$.
  Then $\hat{f}$ is a \defn{limit diagram} whenever the restriction map
  $\cat{C}_{/ \hat{f}} \to \cat{C}_{/ f}$
  is a categorical equivalence.
  \defn{Colimit diagrams} are defined dually.
  Let $\const_I : \cat{C} \to \Fun(I, \cat{C})$ be the constant diagram functor.
  A right adjoint of $\const_I$, when it exists, computes limits of all $I$-shaped
  diagrams; similarly a left adjoint of $\const_I$ computes colimits of $I$-shaped
  diagrams.
\end{para}

\begin{para}  
  Suppose that $\cat{C}$ is a quasicategory and
  \[
    \begin{tikzcd}
      {X_0} \ar[r] \ar[d] &
      {X_1} \ar[r] \ar[d] \pullbackcorner &
      {X_2} \ar[d] \\
      {Y_0} \ar[r] &
      {Y_1} \ar[r] &
      {Y_2}
    \end{tikzcd}
  \]
  a commutative diagram in $\cat{C}$ in which the square on the right
  is a pullback square. Then the square on the left is a pullback square
  if and only if the overall diagram composes to a pullback square.
\end{para}

\begin{para}
  Suppose that $\cat{C}$ is a quasicategory and 
  \[
    \begin{tikzcd}[column sep = {2.5em, between origins}, row sep = {2.5em, between origins}]
    	& B && Y \\
    	A && X \\
    	& D && W \\
    	C && Z
    	\arrow[from=2-1, to=1-2]
    	\arrow[from=2-1, to=4-1]
    	\arrow[from=4-1, to=4-3]
    	\arrow[from=2-3, to=1-4]
    	\arrow[from=1-2, to=3-2]
    	\arrow[from=4-1, to=3-2]
    	\arrow[from=4-3, to=3-4]
    	\arrow[from=3-2, to=3-4]
    	\arrow[from=1-4, to=3-4]
    	\arrow[from=1-2, to=1-4]
    	\arrow[crossing over, from=2-3, to=4-3]
    	\arrow[crossing over, from=2-1, to=2-3]
    \end{tikzcd}
  \]
  a commutative diagram in $\cat{C}$ in which the left, back and right square
  \[
    \begin{tikzcd}
      {A} \ar[r] \ar[d] \pullbackcorner &
      {B} \ar[d] \\
      {C} \ar[r] &
      {D}
    \end{tikzcd}
    \qquad
    \begin{tikzcd}
      {B} \ar[r] \ar[d] \pullbackcorner &
      {Y} \ar[d] \\
      {D} \ar[r] &
      {W}
    \end{tikzcd}
    \qquad
    \begin{tikzcd}
      {X} \ar[r] \ar[d] \pullbackcorner &
      {Y} \ar[d] \\
      {Z} \ar[r] &
      {W}
    \end{tikzcd}
  \]
  are all pullback squares. Then also the front square
  \[
    \begin{tikzcd}
      {A} \ar[r] \ar[d] &
      {X} \ar[d] \\
      {C} \ar[r] &
      {Z}
    \end{tikzcd}
  \]
  is a pullback square.
\end{para}

\subsection{Localisation and Model Categories}

\begin{para}
  Let $\cat{C}$ be a quasicategory and $W$ a subset of the set of maps in $\cat{C}$.
  The \defn{localisation} of $\cat{C}$ at $W$ is a quasicategory $\cat{C}[W^{-1}]$
  together with a functor $\gamma : \cat{C} \to \cat{C}[W^{-1}]$ so that for
  every quasicategory $\cat{D}$ the functor $\CatInfty(\cat{C}[W^{-1}], \cat{D}) \to \CatInfty(\cat{C}, \cat{D})$
  induced by $\gamma$ is fully faithful, and its essential image consists of those functors
  $f : \cat{C} \to \cat{D}$ which sends maps in $W$ to equivalences in $\cat{D}$.
  The localisation always exists and is determined uniquely up to equivalence.
  Even when $\cat{C}$ is a $1$-category, the localisation of $\cat{C}$ at some set
  may not be equivalent to a $1$-category. Every quasicategory arises as the
  localisation of a $1$-category.
\end{para}

\begin{para}
  Let $\cat{C}$ be a quasicategory and $W$ a subset of the set of maps in $\cat{C}$.
  Then $W$ satisfies the \defn{$2$-out-of-$3$} property
  if whenever $f : x \to y$ and $g : y \to z$ are composable maps in $\cat{C}$ and two of
  $f$, $g$, $g \circ f$ are in $W$, then so is the third.
  In particular if $W$ is the set of equivalences in $\cat{C}$, then it satisfies
  the $2$-out-of-$3$ property.
\end{para}

\begin{para}
  For any Grothendieck universe $U$ we have a $1$-category $\sSet_U$ of
  $U$-small simplicial sets. Then the quasicategory of $U$-small $\defn{spaces}$
  $\Space_U$ is the localisation of $\sSet_U$ at the Kan equivalences.
  The quasicategory of $U$-small $\infty$-categories $(\CatInfty)_{U}$ is the
  localisation of $\sSet_U$ at the categorical equivalences.
  We largely ignore size issues, and therefore drop the Grothendieck
  universe from the notation to write $\Space$ and $\CatInfty$.
\end{para}

Localisations are difficult to work with in general.
However, there are tools that use additional structure on a relative category
$(\cat{C}, \cat{W})$ to enable us to perform computations in the localisation $\cat{C}[\cat{W}^{-1}]$
via constructions purely in $\cat{C}$.
Model categories are one such tool.
We will quickly review the parts of the theory of model categories that are relevant to this thesis,
but refer to the literature for proofs or more in depth discussion.

\begin{para}
  Let $\cat{C}$ be a $1$-category that is complete and cocomplete.
  A \defn{model structure} $S$ on $\cat{C}$ consists of three wide subcategories  
  $\Weq_S$, $\Fib_S$ and $\Cof_S$ of $\cat{C}$ such that
  \begin{enumerate}
    \item $\Weq_S$, $\Fib_S$ and $\Cof_S$ are closed under retract.
    \item $(\Weq_S \cap \Cof_S) \LLP \Fib_S$ and $\Cof_S \LLP (\Weq_S \cap \Fib_S)$.
    \item Each map in $\cat{C}$ factors as a map in $\Weq_S \cap \Cof_S$ followed by a map in $\Fib_S$.
    \item Each map in $\cat{C}$ factors as a map in $\Cof_S$ followed by a map in $\Weq_S \cap \Fib_S$.
  \end{enumerate}
  The $1$-category $\cat{C}$ together with the model structure $S$
  form a \defn{model category} $(\cat{C}, S)$.
  The maps in $\Weq_S$, $\Fib_S$ and $\Cof_S$ are called the
  \defn{weak equivalences}, \defn{fibrations} and \defn{cofibrations}
  of $S$. A fibration or cofibration is \defn{trivial} or \defn{acyclic}
  when it is also in $\Weq_S$. An object $X$ in $\cat{C}$ is
  \defn{fibrant} when the unique map $X \to \terminal$ into the terminal
  object of $\cat{C}$ is a fibration. An object $X$ in $\cat{C}$
  is \defn{cofibrant} when the unique map $\emptyset \to X$
  from the initial object is a cofibration.
  The model structure is \defn{cofibrantly generated} when there exist
  sets of maps $I$, $J$ in $\cat{C}$ such that $\Fib_S = \rlp(I)$
  and $\Fib_S \cap \Weq_S = \rlp(J)$.
  The model category is \defn{combinatorial} when $\cat{C}$ is a
  locally presentable category and the model structure is cofibrantly generated.
\end{para}

\begin{para}
  The \defn{Kan-Quillen model structure} $\MSKan$ on the category $\sSet$
  of simplicial sets is the unique model structure
  on $\sSet$ whose cofibrations are the monomorphisms and the fibrant
  objects are Kan complexes.
  The weak equivalences are the Kan equivalences as defined above.
  It is cofibrantly generated by the sets
  \begin{align*}
    I &:= \{ \Lambda^i\ord{k} \hookrightarrow \Delta\ord{k} \mid 0 \leq i \leq k, k \geq 1 \} \\
    J &:= \{ \partial \Delta\ord{k} \hookrightarrow \Delta\ord{k} \mid k \geq 0 \}
  \end{align*}
  We call the fibrations in $\Fib_{\MSKan}$ the \defn{Kan fibrations}.
  Since $\sSet$ is locally presentable, the model category $(\sSet, \MSKan)$
  is combinatorial.
\end{para}

\begin{para}
  The \defn{Joyal model structure} $\MSJoyal$ is the unique model structure on
  $\sSet$ whose cofibrations are the monomorphisms and the fibrant
  objects are quasicategories.
  Its weak equivalences are the categorical equivalences.
  It is cofibrantly generated by the sets 
  \begin{align*}
    I &:= \{ \Lambda^i\ord{k} \hookrightarrow \Delta\ord{k} \mid 0 < i < k \} \cup \{ \Delta\ord{0} \hookrightarrow E\ord{1} \} \\
    J &:= \{ \partial \Delta\ord{k} \hookrightarrow \Delta\ord{k} \mid k \geq 0 \}
  \end{align*}
  A fibration in $\Fib_{\MSJoyal}$ is a \defn{categorical fibration}.
  The trivial fibrations in $\Fib_{\MSJoyal} \cap \Weq_{\MSJoyal}$ agree with
  the trivial fibrations of the Kan-Quillen model structure.
  Since $\sSet$ is locally presentable, the model category $(\sSet, \MSJoyal)$
  is combinatorial.
  We say that a map in $\sSet$ is an \defn{inner fibration} whenever it has the right lifting property with respect to the set
  \[ \{ \Lambda^i\ord{k} \hookrightarrow \Delta\ord{k} \mid 0 < i < k \} \]
\end{para}

\begin{lemma}\label{lem:b-inner-fib-1-cat}
  Suppose that $f : \cat{C} \to \cat{D}$ is a map of simplicial sets such that
  $\cat{C}$ is a quasicategory and $\cat{D}$ is the nerve of a $1$-category.
  Then $f$ is an inner fibration.  
\end{lemma}
\begin{proof}
  Suppose that $0 < i < k$ and we have a lifting problem
  \begin{equation}\label{eq:b-inner-fib-1-cat:problem}
    \begin{tikzcd}
      {\Lambda^i\ord{k}} \ar[r, "\tau'"] \ar[d, hook] &
      {\cat{C}} \ar[d, "f"] \\
      {\Delta\ord{k}} \ar[r, "\sigma"'] \ar[ur, dashed] &
      {\cat{D}}
    \end{tikzcd}
  \end{equation}
  Because $\cat{C}$ is a quasicategory, we have an inner horn filler
  \[
    \begin{tikzcd}
      {\Lambda^i\ord{k}} \ar[r, "\tau'"] \ar[d, hook] &
      {\cat{C}} \\
      {\Delta\ord{k}} \ar[ur, dashed, "\tau"'] &
      {}
    \end{tikzcd}
  \]
  Then $f \circ \tau$ and $\sigma$ are both solutions to the lifting problem
  \[
    \begin{tikzcd}
      {\Lambda^i\ord{k}} \ar[r, "f \circ \tau'"] \ar[d, hook] &
      {\cat{D}} \\
      {\Delta\ord{k}} \ar[ur, dashed] &
      {}
    \end{tikzcd}
  \]
  By assumption $\cat{D}$ is the nerve of a $1$-category and so $f \circ \tau = \sigma$.
  Therefore, $\tau$ is a solution to the lifting problem~(\ref{eq:b-inner-fib-1-cat:problem}).
\end{proof}

\begin{para}
  Suppose that $(\cat{C}, S)$ is a model category and $\cat{I}$ is a small $1$-category.
  Let $\cat{W}$ be the wide subcategory of $\Fun(\cat{I}, \cat{C})$ consisting of
  the pointwise weak equivalences, i.e.\ those natural transformations whose
  components are weak equivalences in $\Weq_S$.
  Then the universal property of the localisation induces a functor of $\infty$-categories
  \[
    \Fun(\cat{I}, \cat{C})[\cat{W}^{-1}]
    \longrightarrow 
    \Fun(\cat{I}, \cat{C}[\Weq_S^{-1}]) 
  \]  
  By~\cite[Theorem 7.9.8.]{cisinski-infty-cats} this functor is an equivalence.
  Therefore, every diagram in the localisation $X : \cat{I} \to \cat{C}[\Weq_S^{-1}]$
  can be \defn{rectified} to a diagram $\hat{X} : \cat{I} \to \cat{C}$
  in the model category:
  \[
    \begin{tikzcd}
      {\cat{I}} \ar[r, dashed, "\hat{X}"] \ar[dr, "X"'] &
      {\cat{C}} \ar[d] \\
      {} &
      {\cat{C}[\Weq_S^{-1}]}
    \end{tikzcd}
  \]  
  In particular every diagram of $\infty$-categories in $\CatInfty$
  or spaces in $\Space$
  rectifies to a diagram in the $1$-category of simplicial sets.
\end{para}

\begin{para}
  A \defn{Reedy category} is a $1$-category $\cat{I}$ together with an
  ordinal $\alpha$ and a degree functor $\deg : \cat{I} \to \alpha$ and
  two wide subcategories $\cat{I}_+$ and $\cat{I}_-$ which together satisfy:
  \begin{enumerate}
    \item Each non-identity map in $\cat{I}_+$ strictly raises the degree.
    \item Each non-identity map in $\cat{I}_-$ strictly lowers the degree.
    \item Each map of $\cat{I}$ factors uniquely as a map in $\cat{I}_-$ followed by a map in $\cat{I}_+$.
  \end{enumerate}
  Whenever $\cat{I}$ is a Reedy category, then its opposite $\cat{I}^\op$ can
  also be equipped with the structure of a Reedy category.
  The category $\FinOrd$ is a Reedy category with the degree functor
  $\deg : \FinOrd \to \omega$ which sends $\ord{n}$ to $n \in \omega$,
  where $\FinOrd_+$ and $\FinOrd_-$ are the injective and surjective maps.
  When $(\cat{C}, S)$ is a model category and $\cat{I}$ a Reedy category,
  we can equip $\Fun(\cat{I}, \cat{C})$ with the \defn{Reedy model structure}
  $\MSReedy(\cat{I}, S)$ that presents the quasicategory of functors
  $\Fun(\cat{I}, \cat{C}[\Weq_S^{-1}])$.
\end{para}

\begin{para}
  Let $(\cat{C}, S)$ be a model category and $X \in \cat{C}$ an object.
  Let $U : \cat{C}_{/ X} \to \cat{C}$ be the forgetful functor from the slice
  category. Then $\cat{C}_{/X}$ admits a \defn{slice model structure}
  $S_{/X}$ with
  \begin{align*}
    \Weq_{S_{/X}} &:= \{ f \in \Arr(\cat{C}_{/ X}) \mid U(f) \in \Weq_S \} \\
    \Fib_{S_{/X}} &:= \{ f \in \Arr(\cat{C}_{/ X}) \mid U(f) \in \Fib_S \} \\
    \Cof_{S_{/X}} &:= \{ f \in \Arr(\cat{C}_{/ X}) \mid U(f) \in \Cof_S \}
  \end{align*}
  The slice model structure is cofibrantly generated and combinatorial whenever
  $(\cat{C}, S)$ is.
\end{para}

\begin{para}
  Let $(\cat{C}, S)$ and $(\cat{D}, T)$ be model categories.
  An adjunction $L : \cat{C} \rightleftarrows \cat{D} : R$ is a
  \defn{Quillen adjunction}
  whenever the following equivalent conditions are satisfied:
  \begin{enumerate}
    \item $L$ preserves cofibrations and trivial cofibrations.
    \item $R$ preserves fibrations and trivial fibrations.
    \item $L$ preserves cofibrations and $R$ preserves fibrations.
    \item $L$ preserves trivial cofibrations and $R$ preserves fibrations.
  \end{enumerate}
  In this case we denote the Quillen adjunction by
  \[ L : (\cat{C}, S) \rightleftarrows (\cat{D}, T) : R \]
  and say that $L$, $R$ are left and right Quillen functors, respectively.
  By Ken Brown's Lemma left Quillen functors preserve weak equivalences
  between cofibrant objects and right Quillen functors preserve weak equivalences
  between fibrant objects.
  By~\cite[Theorem 1.1.]{mazel-gee-quillen-adjunctions} the Quillen
  adjunction induces a derived adjunction
  \[ \mathbb{L} : \cat{C}[\Weq_S^{-1}] \rightleftarrows \cat{D}[\Weq_T^{-1}] : \mathbb{R} \]
  between the quasicategories obtained by localisation.
\end{para}

  

\begin{para}
  The identity functors form a Quillen adjunction
  \[ \id : (\sSet, \MSJoyal) \rightleftarrows (\sSet, \MSKan) : \id \]
  between the Joyal and Kan model structures on simplicial sets,
  which induces an adjunction $\CatInfty \rightleftarrows \Space$.
  The right adjoint is the inclusion $\Space \hookrightarrow \CatInfty$
  of spaces into $\infty$-categories, while the left adjoint
  $\CatToSpace : \CatInfty \to \Space$ produces a space $\CatToSpace(\cat{C})$ from an $\infty$-category $\cat{C}$
  by freely inverting all maps.
  We also call $\CatToSpace(\cat{C})$ the \defn{classifying space} of $\cat{C}$.
\end{para}

\begin{para}
  Let $(\cat{C}, S)$ be a model structure and $\cat{I}$ a Reedy category.
  Then the adjunction between the constant diagram functor and the limit functor
  \[
    \const : (\cat{C}, S) \rightleftarrows (\Fun(\cat{I}, \cat{C}), \MSReedy(\cat{I}, S)) : \lim
  \]
  is a Quillen adjunction so that the induced adjunction
  \[
    \const : \cat{C}[\Weq_S^{-1}]
    \rightleftarrows
    \Fun(\cat{I}, \cat{C}[\Weq_S^{-1}]) : \lim
  \]
  determines $\cat{I}$-shaped limits in the $\infty$-category $\cat{C}[\Weq_S^{-1}]$.
  In particular when $X: \cat{I} \to \cat{C}$ is a diagram in the $1$-category $\cat{C}$ that is fibrant in the Reedy model structure, 
  then the localisation functor $\cat{C} \to \cat{C}[\Weq_S^{-1}]$ preserves the limit of $X$.
  Analogously, the adjunction 
  \[
    \colim : (\Fun(\cat{I}, \cat{C}), \MSReedy(\cat{I}, S)) \rightleftarrows (\cat{C}, S) : \const
  \]
  is a Quillen adjunction   
  and so the localisation $\cat{C} \to \cat{C}[\Weq_S^{-1}]$ preserves colimits of Reedy cofibrant diagrams.
\end{para}

\begin{para}
  Suppose that $(\cat{C}, S)$ is a model category.
  When we let
  \[ \cat{I} := \{ 0 \rightarrow 1 \leftarrow 2 \}, \]
  then an $\cat{I}$-shaped limit is a pullback.
  The category $\cat{I}$ is a Reedy category so that a diagram
  \[
    \begin{tikzcd}
      X(0) \ar[r] & X(1) & X(2) \ar[l]
    \end{tikzcd}
  \]
  is fibrant in $(\Fun(\cat{I}, \cat{C}), \MSReedy(S))$ when the objects
  $X(0)$, $X(1)$ and $X(2)$ are all fibrant in $(\cat{C}, S)$
  and $X(2) \to X(1)$ is a fibration.
  This allows us to compute pullbacks in $\cat{C}[\Weq^{-1}]$
  via certain pullbacks in $\cat{C}$.
  The pullbacks of quasicategories above all have this form and
  therefore represent pullbacks of $\infty$-categories.
\end{para}

\begin{para}
  Let $(\cat{C}, S)$ be a model category and $f : X \to Y$ a map in $\cat{C}$.
  Then the adjunction $\Sigma_f \dashv f^*$ between the slice categories is
  a Quillen adjunction for the slice model structures
  \[
    \Sigma_f : (\cat{C}_{/ X}, S_{/ X}) \rightleftarrows (\cat{C}_{/ Y}, S_{/ Y}) : f^*
  \]
\end{para}

\begin{para}
  Let $\cat{C}$ be an $\infty$-category.
  A \defn{reflective localisation} of $\cat{C}$ is a full subcategory
  $i : \cat{D} \hookrightarrow \cat{C}$ such that the inclusion functor
  has a left adjoint $L : \cat{C} \to \cat{D}$.
  We say that a map $f : X \to Y$ in $\cat{C}$ is $\cat{D}$-local when for every
  object $D \in \cat{D}$, precomposition with $f$ induces an equivalence
  $\cat{C}(Y, D) \to \cat{C}(X, D).$
  Then $L : \cat{C} \to \cat{D}$ is the localisation functor which freely
  inverts the $\cat{D}$-local maps in $\cat{C}$.
\end{para}

\begin{para}
  A diagram $X : \cat{I}^\triangleleft \to \cat{D}$ is a limit diagram
  if and only if $i \circ X : \cat{I}^\triangleleft \to \cat{C}$ is a limit
  diagram. 
  A colimit in $\cat{D}$ is computed by first taking the colimit in $\cat{C}$ and then 
  applying the functor $L$ to the result.
\end{para}

\begin{para}
  Let $(\cat{C}, S)$ be a model category and $W$ a collection of maps in $\cat{C}$.
  The \defn{left Bousfield localisation} of $(\cat{C}, S)$ at $W$ is,
  if it exists, the unique model structure $S[W^{-1}]$ on $\cat{C}$
  with the same cofibrations $\Cof_{S[W^{-1}]} = \Cof_S$ 
  so that the weak equivalences $\Weq_{S[W^{-1}]}$ are those maps of
  $\cat{C}$ that become equivalences in the localisation $\cat{C}[(\Weq_S \cup W)^{-1}]$.
  Then the identity functor $\id : \cat{C} \to \cat{C}$ induces a Quillen
  adjunction $\id : (\cat{C}, S[W^{-1}]) \rightleftarrows (\cat{C}, S) : \id$.
  The derived adjunction $\cat{C}[\Weq_{S[W^{-1}]}] \rightleftarrows \cat{C}[\Weq_{S}]$
  has a fully faithful right adjoint so that $\cat{C}[\Weq_{S[W^{-1}]}]$
  is a reflective subcategory of $\cat{C}[\Weq_{S}]$.
\end{para}

\subsection{Fibrations}

\begin{para}
  A map of simplicial sets $p : X \to Y$ is an isofibration if it is
  an inner fibration and for every object $x_0$ of $X$ and every
  equivalence $f : y_0 \to y_1$ in $Y$ with $p(x_0) = y_0$ there exists
  an equivalence $\hat{f} : x_0 \to x_1$ in $X$ with $p(\hat{f}) = f$.
  Whenever $\cat{E}$ and $\cat{K}$ are quasicategories, a map
  $p : \cat{E} \to \cat{K}$ is an isofibration if and only if it
  is a categorical fibration.
\end{para}

\begin{para}
  A map of simplicial sets $p : E \to K$ is a \defn{left fibration}
  if it has the right lifting property against all horn inclusions
  $\Lambda^i\ord{k} \to \Delta\ord{k}$ with $0 \leq i < k$.
  Dually, a map of simplicial sets $p : E \to K$ is a \defn{right fibration}
  when it has the right lifting property against all horn inclusions  
  $\Lambda^i\ord{k} \to \Delta\ord{k}$ with $0 < i \leq k$.
  Both left and right fibrations are closed under pullback and composition.
\end{para}

\begin{para}
  Let $K$ be a simplicial set. The \defn{covariant model structure} over $K$
  is the unique model structure $\MSCov(K)$ on the slice category $\sSet_{/ K}$ whose fibrant
  objects are the left fibrations and whose cofibrations are the monomorphisms.
  For any map $f : K \to L$ of simplicial sets, the induced base change adjunction between
  the slice categories $\Sigma_f : \sSet_{/K} \rightleftarrows \sSet_{/L} : f^*$
  is a Quillen adjunction for the covariant model structures.
  Moreover, when $f$ is also a categorical equivalence, then the adjunction is a 
  Quillen equivalence.
  When $\cat{K}$ is a quasicategory, the covariant model structure represents a full $\infty$-subcategory
  $\LFib{\cat{K}}$ of $(\CatInfty)_{/ \cat{K}}$.
  There is an equivalence of $\infty$-categories
  $
    \LFib{\cat{K}} \simeq \Fun(\cat{K}, \Space)
  $
  that is natural in $\cat{K}$.
  The functor $\LFib{\cat{K}} \to \Fun(\cat{K}, \Space)$ is called
  \defn{straightening} and its inverse \defn{unstraightening}.
  Given a left fibration $p : \cat{E} \to \cat{K}$ over a quasicategory
  $\cat{K}$ which straightens to a functor $F : \cat{K} \to \Space$,
  we can compute the Kan complex which represents the space $F(x)$
  for an object $x \in \cat{K}$ by taking the pullback of
  simplicial sets
  \[
    \begin{tikzcd}
      {F(x)} \ar[r] \ar[d] \pullbackcorner &
      {\cat{E}} \ar[d, "p"] \\
      {\Delta\ord{0}} \ar[r, "x"'] &
      {\cat{K}}
    \end{tikzcd}
  \]
\end{para}

\begin{para}
  Dually, the right fibrations over a simplicial set $K$ are the fibrant
  objects of the \defn{contravariant model structure} $\MSContr(K)$,
  which is compatible with base change. When $\cat{K}$ is a quasicategory,
  the contravariant model structure represents a full $\infty$-subcategory
  $\RFib{\cat{K}}$ of $(\CatInfty)_{/ \cat{K}}$. 
  There is an equivalence of $\infty$-categories
  $
    \RFib{\cat{K}} \simeq \Fun(\cat{K}^\op, \Space)
  $
  that is natural in $\cat{K}$.
  The functor $\RFib{\cat{K}} \to \Fun(\cat{K}^\op, \Space)$ is called
  \defn{straightening} and its inverse \defn{unstraightening}.
\end{para}

\begin{para}
  Let $\cat{C}$ be a quasicategory. A \defn{presheaf} on $\cat{C}$ is a functor
  $\cat{C}^\op \to \Space$ and a map of presheaves is a natural transformation
  of such functors. Presheaves on $\cat{C}$ organise into the quasicategory
  $\PSh(\cat{C}) := \Fun(\cat{C}^\op, \Space)$.
  Since $\Space$ is complete and cocomplete, and limits and colimits in a functor
  category are computed pointwise, it follows that $\PSh(\cat{C})$ is also complete and cocomplete.
  The $\infty$-category of presheaves $\PSh(\cat{C})$ can be presented by the contravariant model structure over $\cat{C}$ or the covariant model structure over $\cat{C}^\op$.
\end{para}

\begin{para}
  Let $\cat{C}$ be a quasicategory and $x$ an object of $\cat{C}$.
  Then the projection from the slice category $\cat{C}_{x/} \to \cat{C}$
  under $x$ is a left fibration, which represents the covariant functor
  $\cat{C}(x, -) : \cat{C} \to \Space$.
  Similarly, the projection from the slice category $\cat{C}_{/x} \to \cat{C}$
  over $x$ is a right fibration, representing the contravariant functor
  $\cat{C}(-, x) : \cat{C}^\op \to \Space$.
  The functor $\cat{C}(-, x)$ is called the \defn{representable presheaf} for $x$.
\end{para}

\begin{para}
  Let $X$ be a simplicial set.
  The \defn{twisted arrow construction} $\Tw(X)$ of $X$ is the simplicial set
  whose $k$-simplices are the maps $\ord{k}^\op \join \ord{k} \to X$.
  This defines a right adjoint functor $\Tw(-) : \sSet \to \sSet$
  together with a natural projection map $\Tw(X) \to X^\op \times X$ 
  that is induced by the inclusion $\ord{k}^\op \sqcup \ord{k} \hookrightarrow \ord{k}^\op \join \ord{k}$.
  When $\cat{C}$ is a quasicategory, the projection map $\Tw(\cat{C}) \to \cat{C}^\op \times \cat{C}$ is a left fibration which, via the equivalence $\LFib{\cat{C}^\op \times \cat{C}} \simeq \Fun(\cat{C}^\op \times \cat{C}, \Space)$,
  represents the mapping space functor
  $\cat{C}(-, -) : \cat{C}^\op \times \cat{C} \to \Space$.  
  Expanding the definitions, we see that a $k$-simplex of $\Tw(\cat{C})$
  corresponds to a diagram in $\cat{C}$ of the form
  \[
    \begin{tikzcd}
      {s_0} \ar[d] &
      {s_1} \ar[d] \ar[l] &
      {\cdots} \ar[l] &
      {s_{k - 1}} \ar[d] \ar[l]&
      {s_k} \ar[l] \ar[d] \\
      {t_0} \ar[r] &
      {t_1} \ar[r] &
      {\cdots} \ar[r] &
      {t_{k - 1}} \ar[r] &
      {t_k}
    \end{tikzcd}
  \]
\end{para}

\begin{para}
  The \defn{Yoneda embedding} is the fully faithful functor
  $\cat{C} \hookrightarrow \PSh(\cat{C})$ which sends an object $x$ of $\cat{C}$
  to the representable presheaf $\cat{C}(-, x) : \cat{C}^\op \to \Space$.
  The Yoneda embedding can be constructed concretely by starting with the left
  fibration $\Tw(\cat{C}) \to \cat{C}^\op \times \cat{C}$, obtaining the
  functor $\cat{C}(-, -) : \cat{C}^\op \times \cat{C} \to \Space$ and
  then applying the adjunction $(\cat{C}^\op \times -) \dashv \Fun(\cat{C}^\op, -)$.
\end{para}

\begin{para}
  Let $f : \cat{C} \to \cat{D}$ be a functor between quasicategories.
  Then precomposition with $f$ induces a functor $f^* : \PSh(\cat{D}) \to \PSh(\cat{C})$.
  The functor $f^*$ has a left adjoint $f_! : \PSh(\cat{C}) \to \PSh(\cat{D})$
  and a right adjoint $f_* : \PSh(\cat{C}) \to \PSh(\cat{D})$.
  When the quasicategories of presheaves are presented using
  the covariant model structures, the adjunction $f_! : \PSh(\cat{C}) \rightleftarrows \PSh(\cat{D}) : f^*$ is the derived adjunction of the Quillen adjunction
  $\Sigma_f : (\sSet_{/\cat{C}}, \MSCov(\cat{C})) \rightleftarrows (\sSet_{/ \cat{D}}, \MSCov(\cat{D})) : f^*$.
  When the functor $f$ is fully faithful, then so are the left and right adjoints
  $f_!, f_* : \PSh(\cat{C}) \to \PSh(\cat{D})$.
\end{para}

\begin{para}
  Let $\cat{B}$ be a quasicategory and $\cat{W} \hookrightarrow \cat{B}$
  a wide subcategory. Then the slice category $\sSet_{/ \cat{B}}$
  admits a model structure $\MSCov(\cat{B}, \cat{W})$, called the  
  \defn{$\cat{W}$-local covariant model structure}, that presents the
  quasicategory $\Fun(\cat{B}[\cat{W}^{-1}], \Space)$.
  The cofibrations are the monomorphisms and the fibrant objects
  are the left fibrations $p : \cat{E} \to \cat{B}$ that restrict to
  a Kan fibration $p'$ over the wide subcategory $\cat{W}$ of $\cat{B}$:  
  \[
    \begin{tikzcd}
      {\cat{E} \times_{\strat{B}} \cat{W}} \ar[r] \ar[d, "p'"'] \pullbackcorner &
      {\cat{E}} \ar[d, "p"] \\
      {\cat{W}} \ar[r] &
      {\cat{B}}
    \end{tikzcd}
  \]
\end{para}

\begin{para}
  A left fibration $p : E \to B$ is a \defn{left covering map} when the lifts
  against the horn inclusions 
  $\Lambda^i \ord{k} \hookrightarrow \Delta\ord{k}$
  with $0 \leq i < k$ are unique.
  Ignoring size issues, for every right covering map $p : E \to B$ there is a unique 
  pullback square of simplicial sets
  \[
    \begin{tikzcd}
      {E} \ar[r] \ar[d] \pullbackcorner &
      {\Set_{*}} \ar[d] \\
      {B} \ar[r] &
      {\Set}
    \end{tikzcd}
  \]
  where $\Set_* \to \Set$ is the functor from pointed sets to sets
  which forgets the point.
  Analogously, right fibrations with unique lifts are the \defn{right covering maps},
  and classify functors into $\Set^\op$.
\end{para}

\begin{para}
  Let $p : E \to B$ be an inner fibration of simplicial sets.
  A map $f : \Delta\ord{1} \to B$ in $E$ is \defn{$p$-cartesian}
  if for every $k \geq 2$ and every diagram of the form
  \[
    \begin{tikzcd}
      {\Delta\ord{1}} \ar[d, hook, "{\langle k - 1, k\rangle}"'] \ar[dr, "f"] \\
      {\Lambda^k\ord{k}} \ar[r] \ar[d, hook] &
      {E} \ar[d, "p"] \\
      {\Delta\ord{k}} \ar[r] \ar[ur, dashed] &
      {B}
    \end{tikzcd}
  \]
  there exists a lift $\Delta\ord{k} \to E$ that makes the diagram commute.
  Dually a map $f$ in $E$ is \defn{$p$-cocartesian} if it is $p^\op$-cartesian
  as a map of $E^\op$. Both $p$-cartesian and $p$-cocartesian maps are closed
  under composition in $E$.
\end{para}

\begin{para}
  Let $p : E \to B$ be an inner fibration of simplicial sets.
  Then $p$ is a \defn{cartesian fibration} if for every object $e_1$ of $E$
  and every map $f : b_0 \to b_1$ in $B$ with $p(e_1) = b_1$ there exists
  a $p$-cartesian map $\hat{f} : e_0 \to e_1$ in $E$ with $p(\hat{f}) = f$.
  Dually $p$ is a \defn{cocartesian fibration} when $p^\op : E^\op \to B^\op$ is a cartesian fibration.
  Both cartesian and cocartesian fibrations are closed under pullback and composition.
\end{para}

\begin{para}
  A \defn{marked simplicial set} is a simplicial set $X$ together with a subset
  $T \subseteq X(\ord{1})$ which contains all degenerate $1$-simplices;
  we say that a $1$-simplex is \defn{marked} when it is contained in $T$.
  A map of marked simplicial sets from $(X, T)$ to $(Y, S)$ is a map of
  simplical sets $f : X \to Y$ such that the induced map $X(\ord{1}) \to Y(\ord{1})$
  sends elements of $T$ to elements of $S$. Marked simplicial sets and
  their maps form a $1$-category $\sSetM$, which is a quasi-topos and therefore
  complete, cocomplete and locally cartesian closed.
  There is a functor $(-)_\flat : \sSetM \to \sSet$ which forgets the marking.
  The forgetful functor has a left adjoint $(-)^\flat : \sSet \to \sSetM$ which
  sends a simplicial set $X$ to the marked simplicial set $X^\flat$ in which
  only the degenerate $1$-simplices of $X$ are marked.
  The forgetful functor also has a right adjoint $(-)^\sharp : \sSet \to \sSetM$
  so that every $1$-simplex of $X^\sharp$ is marked.
\end{para}

\begin{para}
  Let $p : E \to B$ be a cartesian fibration. We then write
  $E^\natural$ for the marked simplicial set whose underlying simplicial
  set is $E$ and in which a map is marked when it is $p$-cartesian.
  This is called the \defn{natural marking} of $p$.
  The \defn{cartesian model structure} over $B$ is the unique model structure
  $\MSCart(B)$ on the slice category $\sSetM_{/ B^\sharp}$ whose cofibrations
  are those maps whose underlying map of simplicial sets is a monomorphism and whose fibrant objects are the cartesian fibrations
  with the natural marking.
  When $\cat{B}$ is a quasicategory, the cartesian model structure over $\cat{B}$
  represents a subcategory $\Cart{\cat{B}}$ of the $\infty$-category $(\CatInfty)_{/ \cat{B}}$.
  We then have a straightening/unstraightening equivalence
  \[
    \Cart{\cat{B}} \simeq \Fun(\cat{B}^\op, \CatInfty).
  \]
  When $p : \cat{E} \to \cat{B}$ is a cartesian fibration which straightens
  to a functor $F : \cat{B}^\op \to \CatInfty$, we can compute the $\infty$-category
  $F(b)$ for any $b \in \cat{B}$ as a quasicategory by taking the pullback of
  simplicial sets
  \[
    \begin{tikzcd}
      {F(b)} \ar[r] \ar[d] &
      {\cat{E}} \ar[d, "p"] \\
      {\Delta\ord{0}} \ar[r, "b"'] &
      {\cat{B}}
    \end{tikzcd}
  \]
\end{para}

\begin{para}
  When $\cat{B} = \Delta\ord{0}$ we have an isomorphism of $1$-categories
  $\cramped{\sSetM}_{/ \Delta\ord{0}^\sharp} \cong \sSetM$. The cartesian model structure
  over $\Delta\ord{0}$ therefore induces a model structure on the category $\sSetM$ of
  marked simplicial sets. The fibrant objects in this model structure
  are the quasicategories $\cat{C}$ equipped with the \defn{natural marking}
  $\cat{C}^\natural$ in which the marked maps are the equivalences of $\cat{C}$.
  This model structure presents the $\infty$-category $\CatInfty$, and a fibrant
  replacement is computed by localising at the marked maps.
\end{para}

\begin{para}
  Let $\cat{C}$ be a quasicategory.
  Then $\cod : \Fun(\Delta\ord{1}, \cat{C}) \to \cat{C}$
  is a cocartesian fibration which sends an object $x \in \cat{C}$ to the slice
  category $\cat{C}_{/ x}$ over $x$ and a map $f : x \to y$ to the functor
  $\Sigma_f : \cat{C}_{/ x} \to \cat{C}_{/y}$ given by postcomposition.
  When $\cat{C}$ has all pullbacks, the functor $\cod : \Fun(\Delta\ord{1}, \cat{C}) \to \cat{C}$ is also a cartesian fibration which sends a map $f : x \to y$ in $\cat{C}$
  to the pullback functor $f^* : \cat{C}_{/y} \to \cat{C}_{/ x}$.
  In this case $\Sigma_f$ is the left adjoint of $f^*$.
\end{para}

\begin{para}
  Let $\cat{C}$ be a quasicategory. We then have projection maps
  $\dom : \Tw(\cat{C}) \to \cat{C}^\op$ and $\cod : \Tw(\cat{C}) \to \cat{C}$
  from the twisted arrow construction of $\cat{C}$. These maps are cocartesian
  fibrations which straighten to the functors $\cat{C} \to \CatInfty$
  which send $x \in \cat{C}$ to $\cat{C}_{x /}$ and $\cat{C}_{/x}^\op$,
  respectively.
\end{para}

\begin{lemma}\label{lem:b-cat-cocartesian-top-row}
  Suppose that we have a commutative diagram of quasicategories
  \[
    \begin{tikzcd}
      {\cat{E} \times_{\cat{D}} \cat{C}} \ar[r, "q^* r"] \ar[d, "r^* q"'] \pullbackcorner &
      {\cat{E}} \ar[d, "q"] \ar[r, "p"] &
      {\cat{B}} \\
      {\cat{C}} \ar[r, "r"'] &
      {\cat{D}}
    \end{tikzcd}
  \]
  in which the square is a pullback square, the map $r : \cat{C} \to \cat{D}$
  is a categorical fibration and $p$ is a cocartesian fibration.
  Suppose further that a map in $\cat{E}$ is $p$-cocartesian if and only
  if it is sent to an equivalence by $q : \cat{E} \to \cat{D}$.
  Then the top row of the diagram composes to a cocartesian fibration
  $p \circ q^* r$. Moreover, a map in $\cat{E} \times_{\cat{D}} \cat{C}$ is
  $(p \circ q^* r)$-cocartesian if and only if it is sent to an equivalence
  by $r^* q : \cat{E} \times_{\cat{D}} \cat{C} \to \cat{C}$.
\end{lemma}
\begin{proof}
  The map $q^* r : \cat{E} \times_{\cat{D}} \cat{C}$ is a categorical fibration
  since it is the pullback of the categorical fibration $r$.
  Therefore $q^* r$
  is in particular an inner fibration. The cocartesian fibration $p : \cat{E} \to \cat{B}$
  is also an inner fibration, and so the composite map $p \circ q^* r$ is an inner fibration.

  Suppose that we have a map $(g, h) : (e_0, c_0) \to (e_1, c_1)$ in $\cat{E} \times_{\cat{D}} \cat{C}$
  such that $h$ is an equivalence, and suppose further that we have
  a lifting problem of the form
  \begin{equation}\label{lem:b-cat-cocartesian-top-row:problem}
    \begin{tikzcd}
      {\Delta\ord{1}} \ar[d, hook, "{\langle 0, 1\rangle}"'] \ar[dr, "{(g, h)}"] \\
      {\Lambda^0\ord{0}} \ar[r, "{(\sigma', \tau')}"] \ar[d, hook] &
      {\cat{E} \times_{\cat{D}} \cat{C}} \ar[d, "p \circ q^* r"] \\
      {\Delta\ord{k}} \ar[r] \ar[ur, dashed] &
      {\cat{B}}
    \end{tikzcd}
  \end{equation}
  Because $h$ is an equivalence, so is the map $r(h) = q(g)$ in $\cat{D}$.
  By assumption we then have that $g$ is $p$-cocartesian, and so there
  exists a solution to the induced lifting problem
  \[
    \begin{tikzcd}
      {\Delta\ord{1}} \ar[d, hook, "{\langle 0, 1\rangle}"'] \ar[dr, "{g}"] \\
      {\Lambda^0\ord{0}} \ar[r, "\sigma'"] \ar[d, hook] &
      {\cat{E}} \ar[d, "p"] \\
      {\Delta\ord{k}} \ar[r] \ar[ur, dashed, "\sigma"{description}] &
      {\cat{B}}
    \end{tikzcd}
  \]
  The map $r$ is a categorical fibration and $h$ is an equivalence in $\cat{C}$.
  We can therefore find a solution to the lifting problem
  \[
    \begin{tikzcd}
      {\Delta\ord{1}} \ar[d, hook, "{\langle 0, 1\rangle}"'] \ar[dr, "{h}"] \\
      {\Lambda^0\ord{0}} \ar[r, "\tau'"] \ar[d, hook] &
      {\cat{C}} \ar[d, "r"] \\
      {\Delta\ord{k}} \ar[r, "q \circ \sigma"'] \ar[ur, dashed, "\tau"{description}] &
      {\cat{D}}
    \end{tikzcd}
  \]
  Then the lifts $\sigma : \Delta\ord{k} \to \cat{E}$ and $\tau : \Delta\ord{k} \to \cat{C}$
  together define a solution $(\sigma, \tau)$ for the lifting problem (\ref{lem:b-cat-cocartesian-top-row:problem}). It follows that $(g, h)$ is $(p \circ q^* r)$-cocartesian.
  
  It remains to show that $p \circ q^* r$ has enough cocartesian lifts.
  Whenever $(e_0, c_0)$ is an object of $\cat{E} \times_{\cat{D}} \cat{C}$
  and $f : b_0 \to b_1$ is a map in $\cat{B}$ with $p(e_0) = b_0$,
  there exists a $p$-cocartesian lift $g : b_0 \to b_1$ with
  $p(g) = f$.
  Then the map $q(g) : q(b_0) \to q(b_1)$ is an equivalence in $\cat{D}$.
  The categorical fibration $r$ is in particular an isofibration, and so
  there exists a lift $h : c_0 \to c_1$ in $\cat{C}$ with $r(h) = q(g)$ so that
  $h$ is an equivalence.
  The maps $g$ and $h$ together determine a map $(e_0, c_0) \to (e_1, c_1)$
  in the pullback $\cat{E} \times_{\cat{D}} \cat{C}$ which is a $(p \circ q^* r)$-cocartesian
  lift of $f$ since $h$ is an equivalence.
\end{proof}

\begin{para}
  Suppose that $\cat{K}$ is an $\infty$-category,
  then the inclusion functor $\coCart{\cat{K}} \hookrightarrow (\CatInfty)_{/ \cat{K}}$
  has a left adjoint which sends a functor $f : \cat{E} \to \cat{K}$ to
  the \defn{free cocartesian fibration} $F$ over $\cat{K}$.
  Following~\cite[Theorem~4.5]{lax-colimits-free-fibrations},
  we can construct $F$ explicitly as the composite of the top row in the diagram
  \[
    \begin{tikzcd}
      {\Fun(\Delta\ord{1}, \cat{K}) \times_{\cat{K}} \cat{E}}
      \ar[r] \ar[d] \pullbackcorner  &
      {\Fun(\Delta\ord{1}, \cat{K})} \ar[d, "\cod"] \ar[r, "\dom"'] &
      {\cat{K}} \\
      {\cat{E}} \ar[r, "f"'] &
      {\cat{K}}
    \end{tikzcd}
  \]
  Similarly the free cartesian fibration $G$ is the composite of the top row in the diagram
  \[
    \begin{tikzcd}
      {\Fun(\Delta\ord{1}, \cat{K}) \times_{\cat{K}} \cat{E}}
      \ar[r] \ar[d] \pullbackcorner &
      {\Fun(\Delta\ord{1}, \cat{K})} \ar[d, "\dom"] \ar[r, "\cod"'] &
      {\cat{K}} \\
      {\cat{E}} \ar[r, "f"'] &
      {\cat{K}}
    \end{tikzcd}
  \]
\end{para}

\begin{para}
  Suppose that $f : \cat{E} \to \cat{K}$ is a functor of $\infty$-categories.
  We can obtain a cocartesian fibration $R$ via the twisted arrow construction
  by composing the top row in the diagram
  \[
    \begin{tikzcd}[column sep = large]
      {\Tw(\cat{K}) \times_{\cat{K}} \cat{E}}
      \ar[r] \ar[d] \pullbackcorner &
      {\Tw(\cat{K})} \ar[d, "\cod"] \ar[r, "\dom"'] &
      {\cat{K}^\op} \\
      {\cat{E}} \ar[r, "f"'] &
      {\cat{K}}
    \end{tikzcd}
  \]
  The cocartesian fibration $R$ represents the functor
  $\cat{K}^\op \to \CatInfty$
  which sends an object $k \in \cat{K}$ to the slice category
  $\cat{K}_{k / } \times_{\cat{K}} \cat{E}$.
  This is the same functor as represented by the free cartesian fibration of $f$.
\end{para}

\begin{para}
  Let $\cat{K}$ be an $\infty$-category.
  The inclusion $\Space \hookrightarrow \CatInfty$ together with its left
  adjoint induce an adjunction on arrows categories
  $\Fun(\cat{K}, \CatInfty) \rightleftarrows \Fun(\cat{K}, \Space)$.
  Then using the straightening/unstraightening equivalences for cocartesian
  and left fibrations over $\cat{K}$, we see how to compute the \defn{free left
  fibration} from a cocartesian fibration:
  \[
    \begin{tikzcd}
    	{\Fun(\cat{K}, \CatInfty)} & {\Fun(\cat{K}, \Space)} \\
    	{\coCart{\cat{K}}} & {\LFib{\cat{K}}}
    	\arrow["\simeq"', from=1-1, to=2-1]
    	\arrow["\simeq", from=1-2, to=2-2]
    	\arrow[""{name=0, anchor=center, inner sep=0}, shift left=2, from=1-1, to=1-2]
    	\arrow[""{name=1, anchor=center, inner sep=0}, shift left=2, hook, from=1-2, to=1-1]
    	\arrow[""{name=2, anchor=center, inner sep=0}, shift left=2, hook, from=2-2, to=2-1]
    	\arrow[""{name=3, anchor=center, inner sep=0}, shift left=2, dashed, from=2-1, to=2-2]
    	\arrow["\dashv"{anchor=center, rotate=-90}, draw=none, from=3, to=2]
    	\arrow["\dashv"{anchor=center, rotate=-90}, draw=none, from=0, to=1]
    \end{tikzcd}
  \]
  We then see that the free left fibration induced by some functor $p : \cat{E} \to \cat{K}$
  is obtained by first replacing it with the free cocartesian fibration:
  \[
    \begin{tikzcd}
    	{(\CatInfty)_{/\cat{K}}} & {\coCart{\cat{K}}} & {\LFib{\cat{K}}}
    	\arrow[""{name=0, anchor=center, inner sep=0}, shift left=2, from=1-1, to=1-2]
    	\arrow[""{name=1, anchor=center, inner sep=0}, shift left=2, from=1-2, to=1-3]
    	\arrow[""{name=2, anchor=center, inner sep=0}, shift left=2, hook, from=1-2, to=1-1]
    	\arrow[""{name=3, anchor=center, inner sep=0}, shift left=2, hook, from=1-3, to=1-2]
    	\arrow["\dashv"{anchor=center, rotate=-90}, draw=none, from=0, to=2]
    	\arrow["\dashv"{anchor=center, rotate=-90}, draw=none, from=1, to=3]
    \end{tikzcd}
  \]
  The dual applies for cartesian fibrations and right fibrations:
  \[
    \begin{tikzcd}
    	{(\CatInfty)_{/\cat{K}}} & {\Cart{\cat{K}}} & {\RFib{\cat{K}}}
    	\arrow[""{name=0, anchor=center, inner sep=0}, shift left=2, from=1-1, to=1-2]
    	\arrow[""{name=1, anchor=center, inner sep=0}, shift left=2, from=1-2, to=1-3]
    	\arrow[""{name=2, anchor=center, inner sep=0}, shift left=2, hook, from=1-2, to=1-1]
    	\arrow[""{name=3, anchor=center, inner sep=0}, shift left=2, hook, from=1-3, to=1-2]
    	\arrow["\dashv"{anchor=center, rotate=-90}, draw=none, from=0, to=2]
    	\arrow["\dashv"{anchor=center, rotate=-90}, draw=none, from=1, to=3]
    \end{tikzcd}
  \]
\end{para}

\begin{para}
  Suppose we have a left fibration $p : \cat{E} \to \cat{C}$ that represents a presheaf
  $\psh{F} : \cat{C}^\op \to \Space$
  and a map of quasicategories $f : \cat{C} \to \cat{D}$.
  Because the postcomposition functor $\Sigma_f : (\sSet_{/\cat{C}}, \MSCov(\cat{C})) \to (\sSet_{/ \cat{D}}, \MSCov(\cat{D}))$
  is a left Quillen functor and every object in the covariant model structure over $\cat{C}$ is
  cofibrant, the composite map $\Sigma_f p = f \circ p : \cat{E} \to \cat{D}$ 
  represents the presheaf $f_! \psh{F}$.
  The map $f \circ p$ itself is not a left fibration in general, but by taking
  the free left fibration of $f \circ p$ we can find a fibrant replacement that
  represents the presheaf $f_!\psh{F}$.  
\end{para}

\subsection{Segal Spaces}

Complete Segal spaces are an alternative model for $\infty$-categories~\cite{rezk-complete-segal},
which generalises to a model for $(\infty, n)$-categories~\cite{barwick-n-fold-segal, lurie-tqft}.

\begin{para}
  An order-preserving map of finite non-empty ordinals $\alpha: \ord{t} \to \ord{k}$ is \defn{inert}
  if $\alpha(i) = \alpha(0) + i$ for all $i \in \ord{t}$.
  We write $\FinOrd^\inert$ for the wide subcategory of $\FinOrd$ consisting of the
  inert maps, and $\FinOrd_{\leq 1}$ for the full subcategory of $\FinOrd$
  that contains $\ord{0}$ and $\ord{1}$.
  A simplicial space $\psh{F} : \FinOrd^\op \to \Space$ is a \defn{Segal space}  
  if for all $k \geq 0$ the natural map
  \[
    \psh{F}(\ord{k}) \longrightarrow
    \lim(
      \FinOrd_{\leq 1} \times_{\FinOrd} \FinOrd^\inert_{/ \ord{k}}
      \longrightarrow
      \FinOrd
      \xrightarrow{\;\psh{F}\;}
      \Space
    )
  \]
  is an equivalence of spaces. In other words, for every $k \geq 1$ we
  have an equivalence
  \[
    \psh{F}(\ord{k}) \simeq
    \underbrace{
    \psh{F}(\ord{1}) \times_{\psh{F}(\ord{0})}
    \cdots 
    \times_{\psh{F}(\ord{0})}
    \psh{F}(\ord{1})
    }_{\text{$k$ times}}
  \]
  We interpret this property as follows: for every composable sequence
  of maps in $\psh{F}$, the space of compatible choices of composites
  is contractible.  
  The Segal spaces form a reflective subcategory $\Seg(\FinOrd)$ of $\PSh(\FinOrd)$.
\end{para}




\begin{para}
  A Segal space $\psh{F} \in \Seg(\FinOrd)$ is \defn{complete} if it
  is local for the unique map $E\ord{1} \to \Delta\ord{0}$
  from the free contractible groupoid with two objects $E\ord{1}$.
  Complete Segal spaces are a model of $\infty$-categories: 
  For every $\infty$-category $\cat{C}$ we have a presheaf on $\FinOrd$
  which sends $\ord{k}$ to the space of functors $\ord{k} \to \cat{C}$.
  This defines a fully faithful functor
  $\CatInfty \hookrightarrow \PSh(\FinOrd)$,
  the $\infty$-categorical version of the nerve
  $\Nerve : \Cat \hookrightarrow \sSet$.
  Then a presheaf is in the essential image of the inclusion
  $\CatInfty \hookrightarrow \PSh(\FinOrd)$
  precisely if it is a complete Segal space.
\end{para}

An $(\infty, n)$-category is an $\omega$-category in which
a $k$-cell must be invertible if $k > n$.
There are multiple equivalent models of $(\infty, n)$-categories.
In particular quasicategories and complete Segal spaces are models of
$(\infty, n)$-categories for $n = 1$.
For general $n \geq 0$ we will present $(\infty, n)$-categories as complete $n$-fold
Segal spaces.

\begin{para}
  For any $n \geq 0$ we write $\FinOrd^n$ for the $n$-fold product of the
  category $\FinOrd$ with itself: 
  \[
    \FinOrd^n := \underbrace{\FinOrd \times \cdots \times \FinOrd}_{\text{$n$ times}}
  \]
  We will sometimes use the following abbreviated notation for objects of $\FinOrd^n$:
  \[
    \ord{k_1, \ldots, k_n} := (\ord{k_1}, \ldots, \ord{k_n})
  \]
  A presheaf $\psh{F} : \FinOrd^{n, \op} \to \Space$ is a \defn{$n$-uple Segal space}
  if for all $k_1, \ldots, k_n \geq 0$ the natural map  
  \[
    \psh{F}(\ord{k_1, \ldots, k_n}) \longrightarrow
    \lim(
      \FinOrd^n_{\leq 1} \times_{\FinOrd^n} \FinOrd^{\inert, n}_{/ \ord{k_1, \ldots, k_n}}
      \longrightarrow
      \FinOrd
      \xrightarrow{\;\psh{F}\;}
      \Space
    )
  \]
  is an equivalence of spaces.  
  The $n$-uple Segal spaces form a reflective subcategory $\Seg(\FinOrd^n)$
  of $\PSh(\FinOrd^n)$.
\end{para}



\begin{para}
  An $n$-uple Segal space $\psh{F}$ is \defn{globular} when
  $\psh{F}(\ord{0}) : \FinOrd^{n - 1, \op} \to \Space$ is a constant diagram
  and $\psh{F}(\ord{i})$ is a globular $(n - 1)$-uple Segal space for all $i \geq 0$.
  Globular $n$-uple Segal spaces are also called \defn{$n$-fold Segal spaces}
  and form a reflective subcategory $\Seg^\glob(\FinOrd^n)$ of $\PSh(\FinOrd^n)$.
\end{para}

\begin{para}
  An $n$-fold Segal space $\psh{F}$ is \defn{complete} when the Segal space
  $\psh{F}(\ord{-, 0, \ldots, 0}) : \FinOrd^\op \to \Space$ is complete and
  the $(n - 1)$-fold Segal space $\psh{F}(\ord{1})$ is complete.
  The complete $n$-fold Segal spaces form a reflective subcategory of
  $\PSh(\FinOrd^n)$ which we denote by $\CatN{n}$.
\end{para}



\section{Polynomial Functors}

\begin{para}
  When $\cat{C}$ is an $\infty$-category with pullbacks,
  every map $f : X \to Y$ in $\cat{C}$ induces a pullback functor
  $f^* : \cat{C}_{/ Y} \to \cat{C}_{/ X}$ that is right adjoint to the
  postcomposition functor $\Sigma_f : \cat{C}_{/ X} \to \cat{C}_{/ Y}$.  
  We say that $f$ is \defn{exponentiable} when the pullback functor
  $f^* : \cat{C}_{/ Y} \to \cat{C}_{/ X}$ also has a right adjoint $\Pi_f : \cat{C}_{/ X} \to \cat{C}_{/ Y}$.
  An $\infty$-category $\cat{C}$ is \defn{locally cartesian closed} when it has pullbacks
  and every map in $\cat{C}$ is exponentiable.
\end{para}

\begin{example}
  For every $1$-category $\cat{C}$, the $1$-category $\Fun(\cat{C}^\op, \Set)$
  of discrete presheaves on $\cat{C}$ is locally cartesian closed.
  This includes the $1$-category of sets $\Set \cong \Fun(\terminal, \Set)$ and the $1$-category of simplicial sets $\sSet \cong \Fun(\FinOrd, \Set)$.
  Analogously, the $\infty$-category of presheaves $\PSh(\cat{C}) = \Fun(\cat{C}^\op, \Space)$ on any $\infty$-category $\cat{C}$ is locally cartesian closed, including the $\infty$-category of spaces $\Space$.
\end{example}

\begin{example}
  The $1$-category of categories $\Cat$ has all pullbacks but is \textbf{not} locally cartesian closed.
  An exponentiable map in $\Cat$ is known as a Conduch\'e fibration.
  The $\infty$-category $\CatInfty$ also is \textbf{not} locally cartesian closed.  
  The exponentiable maps in $\CatInfty$ are called the \defn{exponentiable fibrations}~\cite{ayala-factorization-homology}.
\end{example}

\begin{definition}
  A \defn{polynomial} in an $\infty$-category $\cat{C}$ is a diagram in $\cat{C}$ of the form
  \[
    \begin{tikzcd}
      X &
      Y \ar[l, "f"'] \ar[r, "g"] &
      Z \ar[r, "h"] &
      W
    \end{tikzcd}
  \]
  The \defn{polynomial functor} induced by the polynomial is the composite functor
  \[
    \begin{tikzcd}
      \cat{C}_{/ X} \ar[r, "f^*"] &
      \cat{C}_{/ Y} \ar[r, "\Pi_g"] &
      \cat{C}_{/ Z} \ar[r, "\Sigma_h"] &
      \cat{C}_{/ W}
    \end{tikzcd}
  \]
  provided that $f^*$ and $\Pi_g$ exist.
\end{definition}

\begin{para}
  Suppose that $\cat{C}$ is an $\infty$-category with a terminal object $\terminal$
  and products.
  Then the slice category of $\cat{C}$ over the terminal object $\cat{C}_{/ *}$ 
  is equivalent to $\cat{C}$ itself.
  We will be mostly interested in polynomial functors induced by a polynomial of the form
  \[
    \begin{tikzcd}
      \terminal &
      X \ar[r, "f"] \ar[l, "!_X"'] &
      Y \ar[r, "!_Y"] &
      \terminal 
    \end{tikzcd}
  \]
  where $!_X$ and $!_Y$ are the (essentially) unique maps into the terminal object
  and $f$ is exponentiable.
  In this case we abbreviate $(!_X)^* = (X \times -)$ and $\Sigma_Y := \Sigma_{!_X}$
  and write $\poly{f}{-} : \cat{C} \to \cat{C}$ for the induced polynomial functor  
  \[
    \begin{tikzcd}
      \cat{C} \ar[r, "X \times -"] &
      \cat{C}_{/ X} \ar[r, "\Pi_f"] &
      \cat{C}_{/ Y} \ar[r, "\Sigma_Y"] &
      \cat{C}
    \end{tikzcd}
  \]
\end{para}

The theory of polynomial functors for $1$-categories is well-developed~\cite{kock-polynomial, spivak-poly}.
For $\infty$-categories the literature is much sparser; we refer to~\cite{haugseng-analytic-monad}
for the special case of polynomial functors on the $\infty$-category of spaces $\Space$.
For our purposes we need polynomial functors on $\CatInfty$, which we approach via quasicategories
by constructing polynomial functors on the $1$-category $\sSet$ of simplicial sets.




\begin{lemma}\label{lem:b-poly-characterise}
  Let $\cat{C}$ be a locally cartesian closed $1$-category with a terminal object
  and $f : E \to B$ a map in $\cat{C}$.
  Whenever $A$, $X$ are objects in $\cat{C}$,
  there is a bijection between maps of simplicial sets
  $\varphi : A \to \poly{f}{X}$
  and diagrams of the form
  \[
    \begin{tikzcd}
      {X} &
      {E \times_B A} \ar[r] \ar[d] \ar[l] \pullbackcorner &
      {E} \ar[d, "f"] \\
      {} &
      {A} \ar[r, swap] &
      {B}
    \end{tikzcd}
  \]
  Moreover this bijection is natural in $A$ and $X$.
\end{lemma}
\begin{proof}
  The polynomial functor $\poly{f}{-} : \cat{C} \to \cat{C}$ is the composite
  $\poly{f}{-} = \Sigma_B \circ \Pi_f \circ (E \times -)$.
  The functor $\Sigma_B : {\cat{C}}_{/B} \to \cat{C}$ sends an object
  $X \to B$ 
  of the slice category over $B$
  to the domain $X$.
  Therefore, any map $\varphi : A \to \poly{f}{X}$ induces a map
  $\alpha : A \to B$ so that the diagram
  \[
    \begin{tikzcd}[column sep = small]
      {A} \ar[rr, "\varphi"] \ar[dr, "\alpha", swap] &
      {} &
      {\Sigma_B \Pi_f (E \times X)} \ar[dl, "{\Pi_f (E \times X)}"] \\
      {} &
      {B} &
      {}
    \end{tikzcd}
  \]
  represents a map in the slice category $\cat{C}_{/B}$.
  The dependent product functor $\Pi_f$ is the right adjoint of the pullback functor
  $f^*$.
  Then the bijection
  \[
    \sSet_{/B}((A, \alpha), \Pi_f (E \times X)) \cong
    \sSet_{/E}(f^* (A, \alpha), E \times X)
  \]
  sends $\varphi$ to a map $\varphi_1$ in the slice
  category $\cat{C}_{/E}$ over $E$ which fits into the diagram
  \[
    \begin{tikzcd}
      {} &
      {E \times X} \ar[d]  \\
      {E \times_B A} \ar[r, "f^* \alpha"] \ar[d] \ar[ur, dashed, "\varphi_1"] \pullbackcorner &
      {E} \ar[d, "f"] \\
      {A} \ar[r, "\alpha", swap] &
      {B}
    \end{tikzcd}
  \]
  We then let $\varphi_2$ be the composite of $\varphi_1$ with the product projection
  \[
    \begin{tikzcd}
      {E \times_B A} \ar[r, "\varphi_1"] &
      {E \times X} \ar[r] &
      {X}
    \end{tikzcd}
  \]
  Each step in the construction of $\varphi_2$ from $\varphi$ is invertible.
\end{proof}

\begin{observation}\label{obs:b-poly-base-1}
  Suppose that $\cat{C}$ is a locally cartesian closed $1$-category with a terminal object $\terminal$
  and $f : E \to K$ a map in $\cat{C}$.
  By Lemma~\ref{lem:b-poly-characterise} there is a bijection
  natural between maps $A \to \poly{f}{\terminal}$ and diagrams
  \begin{equation}\label{eq:b-poly-base-1}
    \begin{tikzcd}
      {\terminal} &
      {E \times_B A} \ar[r] \ar[d] \ar[l] \pullbackcorner &
      {E} \ar[d, "f"] \\
      {} &
      {A} \ar[r, swap] &
      {K}
    \end{tikzcd}
  \end{equation}
  Since $\terminal$ is the terminal object of $\cat{C}$, the map
  $E \times_B A \to \terminal$ is uniquely determined.
  Therefore the diagram (\ref{eq:b-poly-base-1}) is uniquely determined
  by the map $A \to K$.
  It follows that the map $\poly{f}{\terminal} \to K$ is an isomorphism.
\end{observation}

\subsection{Maps of Polynomials}

\begin{construction}\label{eq:b-poly-functorial-base}
  Let $\cat{C}$ be a locally cartesian closed $1$-category with a terminal object
  and suppose that we have a diagram in $\cat{C}$ of the form
  \begin{equation}\label{eq:b-poly-functorial-base:map}
    \begin{tikzcd}[column sep = small]
      {E_0} \ar[rr] \ar[dr, "f", swap] &
      {} &
      {E_1} \ar[dl, "g"] \\
      {} &
      {K}
    \end{tikzcd}
  \end{equation}
  Then~(\ref{eq:b-poly-functorial-base:map}) induces a natural transformation
  $\poly{g}{-} \to \poly{f}{-}$.
  In components, we can construct this natural transformation as follows.
  When $A$ and $X$ are objects of $\cat{C}$,
  then by Lemma~\ref{lem:b-poly-characterise} there is a bijection between
  maps $A \to \poly{g}{X}$ and diagrams of the form
  \begin{equation}\label{eq:b-poly-functorial-base:right}
    \begin{tikzcd}
      {X} &
      {E_1 \times_{K_1} A} \ar[r] \ar[d] \ar[l] \pullbackcorner &
      {E_1} \ar[d, "g"] \\
      {} &
      {A} \ar[r, swap] &
      {K}
    \end{tikzcd}
  \end{equation}
  Then (\ref{eq:b-poly-functorial-base:right}) fits into the right side of the diagram
  \begin{equation}\label{eq:b-poly-functorial-base:full}
    \begin{tikzcd}[column sep={3em,between origins}, row sep={3em,between origins}]
    	&& X && X \\
    	& {E_0 \times_{K} A} && {E_1 \times_{K} A} \\
    	{E_0} && {E_1} \\
    	& A && A \\
    	{K} && {K}
    	\arrow[from=2-2, to=3-1]
    	\arrow[from=2-4, to=3-3]
    	\arrow[from=2-2, to=2-4]
    	\arrow["f"'{pos=0.3}, from=3-1, to=5-1]
    	\arrow[from=2-4, to=4-4]
    	\arrow[from=2-2, to=4-2]
    	\arrow[equals, from=4-2, to=4-4]
    	\arrow[from=4-4, to=5-3]
    	\arrow[from=4-2, to=5-1]
    	\arrow[from=5-1, to=5-3, equal]
    	\arrow[from=2-2, to=1-3, dashed]
    	\arrow[from=2-4, to=1-5]
    	\arrow[equals, from=1-3, to=1-5]
    	\arrow[crossing over, from=3-1, to=3-3]
    	\arrow[crossing over, "g"{pos=0.3}, from=3-3, to=5-3]
    \end{tikzcd}
  \end{equation}
  in which the two diagonal squares are pullback squares.
  The left hand square then determines a map $A \to \poly{f}{X}$.
\end{construction}

\begin{construction}\label{con:b-poly-functorial}
  Let $\cat{C}$ be a locally cartesian closed $1$-category with a terminal object
  and suppose that we have a pullback square in $\cat{C}$ of the form
  \begin{equation}\label{eq:b-poly-functorial:map}
    \begin{tikzcd}
      {E_0} \ar[r] \ar[d, "f", swap] \pullbackcorner &
      {E_1} \ar[d, "g"] \\
      {K_0} \ar[r] &
      {K_1}
    \end{tikzcd}
  \end{equation}
  Then (\ref{eq:b-poly-functorial:map}) induces a natural transformation
  $\poly{f}{-} \to \poly{g}{-}$
  between the polynomial functors so that the component on any object $X \in \cat{C}$
  makes the following diagram commute and into a pullback square:
  \[
    \begin{tikzcd}
      {\poly{f}{X}} \ar[r] \ar[d] &
      {\poly{g}{X}} \ar[d] \\
      {K_0} \ar[r] &
      {K_1}
    \end{tikzcd}
  \]
  
  In components, this natural transformation is constructed as follows.
  When $A$ and $X$ are objects of $\cat{C}$,
  then by Lemma~\ref{lem:b-poly-characterise} there is a bijection between
  maps $A \to \poly{f}{X}$ and diagrams of the form
  \begin{equation}\label{eq:b-poly-functorial:left}
    \begin{tikzcd}
      {X} &
      {E_0 \times_{K_0} A} \ar[r] \ar[d] \ar[l] \pullbackcorner &
      {E_0} \ar[d, "f"] \\
      {} &
      {A} \ar[r, swap] &
      {K_0}
    \end{tikzcd}
  \end{equation}
  Then (\ref{eq:b-poly-functorial:left}) fits into the left side of the diagram
  \begin{equation}\label{eq:b-poly-functorial:full}
    \begin{tikzcd}[column sep={3em,between origins}, row sep={3em,between origins}]
    	&& X && X \\
    	& {E_0 \times_{K_0} A} && {E_1 \times_{K_1} A} \\
    	{E_0} && {E_1} \\
    	& A && A \\
    	{K_0} && {K_1}
    	\arrow[from=2-2, to=3-1]
    	\arrow[from=2-4, to=3-3]
    	\arrow[from=2-2, to=2-4]
    	\arrow["f"'{pos=0.3}, from=3-1, to=5-1]
    	\arrow[from=2-4, to=4-4]
    	\arrow[from=2-2, to=4-2]
    	\arrow[equals, from=4-2, to=4-4]
    	\arrow[from=4-4, to=5-3]
    	\arrow[from=4-2, to=5-1]
    	\arrow[from=5-1, to=5-3]
    	\arrow[from=2-2, to=1-3]
    	\arrow[from=2-4, to=1-5, dashed]
    	\arrow[equals, from=1-3, to=1-5]
    	\arrow[crossing over, from=3-1, to=3-3]
    	\arrow[crossing over, "g"{pos=0.3}, from=3-3, to=5-3]
    \end{tikzcd}
  \end{equation}
  in which the two diagonal squares are pullback squares and the front
  square of the cube is the pullback square (\ref{eq:b-poly-functorial:map}).
  Then the back square of the cube is a pullback as well,
  and so the induced map between the pullbacks
  $
    E_0 \times_{K_0} A \to
    E_1 \times_{K_1} A
  $
  must be an isomorphism.
  Therefore, there exists a unique map $E_1 \times_{K_1} A \to X$ which makes the
  overall diagram (\ref{eq:b-poly-functorial:full}) commute.
  Then the right side of (\ref{eq:b-poly-functorial:full}) uniquely corresponds to a map
  $A \to \poly{g}{C}$.
\end{construction}

\begin{construction}\label{con:b-poly-retract}
  Let $\cat{C}$ be a locally cartesian closed $1$-category with a terminal object.
  Suppose that we have a commutative diagram
  \[
    \begin{tikzcd}
      {E_0} \ar[r, hook, "i'"] \ar[d, "f_0"'] &
      {E_1} \ar[d, "f_1"{description}] \ar[r, "r'"] &
      {E_0} \ar[d, "f_0"] \\
      {K_0} \ar[r, hook, "i"'] &
      {K_1} \ar[r, "r", swap] &
      {K_0}
    \end{tikzcd}
  \]
  in $\cat{C}$ in which the rows compose to the identity.
  Then we have natural transformations
  \[
    \begin{tikzcd}
      {\poly{f_0}{-}} \ar[r, "I"] &
      {\poly{f_1}{-}} \ar[r, "R"] &
      {\poly{f_0}{-}}
    \end{tikzcd}
  \]
  that compose to the identity so that for every object $X \in \cat{C}$
  the following diagram commutes
  \[
    \begin{tikzcd}
      {\poly{f_0}{X}} \ar[r, "I_X"] \ar[d] &
      {\poly{f_1}{X}} \ar[r, "R_X"] \ar[d] &
      {\poly{f_0}{X}} \ar[d] \\
      {K_0} \ar[r, "i"', hook] &
      {K_1} \ar[r, "r"'] &
      {K_0}
    \end{tikzcd}
  \]

  For any $X \in \cat{C}$ we only describe the component $I_X$ explicitly;
  the component $R_X$ can be constructed analogously.
  By Lemma~\ref{lem:b-poly-characterise} there is a bijection between maps
  $A \to \poly{f_0}{X}$ and diagrams of the form
  \begin{equation}\label{eq:b-poly-retract:left}
    \begin{tikzcd}
      {X} &
      {E_0 \times_{K_0} A} \ar[r] \ar[d] \ar[l, "\varphi"'] \pullbackcorner &
      {E_0} \ar[d, "f_0"] \\
      {} &
      {A} \ar[r, swap] &
      {K_0}
    \end{tikzcd}
  \end{equation}
  Then (\ref{eq:b-poly-retract:left}) fits twice into the induced diagram
  \begin{equation}\label{eq:b-poly-retract:full}
    \begin{tikzcd}[column sep={3em,between origins}, row sep={3em,between origins}]
    	&& X && X && X \\
    	& {E_0 \times_{K_0} A} && {E_1 \times_{K_1} A} && {E_0 \times_{K_0} A} \\
    	{E_0} && {E_1} && {E_0} \\
    	& A && A && A \\
    	{K_0} & {} & {K_1} && {K_0}
    	\arrow[from=2-2, to=3-1]
    	\arrow[from=2-4, to=3-3]
    	\arrow[from=2-6, to=3-5]
    	\arrow[from=2-4, to=2-6]
    	\arrow["{r'}"{pos=0.8}, from=3-3, to=3-5]
    	\arrow[from=3-1, to=5-1]
    	\arrow[from=2-6, to=4-6]
    	\arrow[from=2-4, to=4-4]
    	\arrow[from=2-2, to=4-2]
    	\arrow[equals, from=4-4, to=4-6]
    	\arrow[from=4-6, to=5-5]
    	\arrow[from=4-4, to=5-3]
    	\arrow[from=4-2, to=5-1]
    	\arrow["r"', from=5-3, to=5-5]
    	\arrow["{\varphi_0}", from=2-2, to=1-3]
    	\arrow["{\varphi_1}", dashed, from=2-4, to=1-5]
    	\arrow["{\varphi_0}", from=2-6, to=1-7]
    	\arrow[equals, from=1-5, to=1-7]
    	\arrow["i"', hook, from=5-1, to=5-3]
    	\arrow[equals, from=4-2, to=4-4]
    	\arrow["{i'}"{pos=0.8}, hook, from=3-1, to=3-3]
    	\arrow[hook, from=2-2, to=2-4]
    	\arrow[equals, from=1-3, to=1-5]
    	\arrow[crossing over, from=3-3, to=5-3]
    	\arrow[crossing over, from=3-5, to=5-5]
    \end{tikzcd}
  \end{equation}
  Because $r$ is a retraction of $i$ it follows that the induced map
  $E_1 \times_{K_1} A \to E_0 \to E_0 \times_{K_0} A$
  is a retraction as well.
  Therefore, there is a unique map
  $\varphi_1$ which completes the diagram (\ref{eq:b-poly-retract:full}).
  By applying Lemma~\ref{lem:b-poly-characterise} to the middle slice of
  (\ref{eq:b-poly-retract:full}) we therefore have a map $A \to \poly{f_1}{X}$.
\end{construction}

\subsection{Composition of Polynomials}

\begin{definition}\label{def:b-poly-total}
  Let $\cat{C}$ be a locally cartesian closed $1$-category
  with a terminal object
  and $f : E \to K$ a map in $\cat{C}$.
  Then for any object $X \in \cat{C}$ we define
  $\epoly{f}{X}$
  via the pullback
  \[
    \begin{tikzcd}
      {\epoly{f}{X}} \ar[r] \ar[d] \pullbackcorner &
      {E} \ar[d, "f"] \\
      {\poly{f}{X}} \ar[r] &
      {K}
    \end{tikzcd}
  \]
  By letting $X$ vary 
  this defines a functor $\epoly{f}{-} : \cat{C} \to \cat{C}$ together
  with a natural transformation
  \[ \epoly{f}{-} \longrightarrow \poly{f}{-}.\]
\end{definition}

\begin{observation}
  Suppose that $f : E \to K$ is a map in a locally cartesian closed $1$-category $\cat{C}$ 
  with a terminal object.
  When we evaluate the functor $\epoly{f}{-} : \cat{C} \to \cat{C}$ on
  the terminal object $1$, the result $\epoly{f}{1}$ is defined via the pullback
  \[
    \begin{tikzcd}
      {\epoly{f}{1}} \ar[r] \ar[d] \pullbackcorner &
      {E} \ar[d, "f"] \\
      {\poly{f}{1}} \ar[r] &
      {K}
    \end{tikzcd}
  \]
  Since the bottom map $\poly{f}{1} \to K$ is an isomorphism it therefore follows
  that also $\epoly{f}{1} \to E$ is an isomorphism.
  In particular the pullback square is an isomorphism
  in the arrow category of $\cat{C}$
  between $f$ and
  the map $\epoly{f}{1} \to \poly{f}{1}$.
\end{observation}

\begin{observation}
  Suppose that $f : E \to K$ is a map in a locally cartesian closed $1$-category $\cat{C}$ 
  with a terminal object.
  When $X$ is an object of $\cat{C}$ then the projection map
  $\epoly{f}{X} \to \poly{f}{X}$
  is uniquely characterised via Lemma~\ref{lem:b-poly-characterise}
  by a diagram of the form
  \[
    \begin{tikzcd}
      {X} &
      {E \times_{K} \poly{f}{X}} \ar[r] \ar[d] \ar[l] \pullbackcorner &
      {E} \ar[d, "f"] \\
      {} &
      {\poly{f}{X}} \ar[r] &
      {K}
    \end{tikzcd}
  \]
  But by definition we have $\epoly{f}{X} = E \times_{K} \poly{f}{X}$.
  Therefore, we have a natural map
  \[
    \epoly{f}{X} \longrightarrow X
  \]
  for every object $X$ of $\cat{C}$.
\end{observation}

\begin{lemma}\label{lem:b-poly-composite}
  Let $\cat{C}$ be a locally cartesian closed $1$-category
  with a terminal object $\terminal$
  and let $f : E_0 \to K_0$, $g: E_1 \to K_1$ a pair of maps in $\cat{C}$.
  Then the composite functor $X \mapsto \poly{g}{\poly{f}{X}}$ is again a polynomial
  functor which is induced by the map
  \[
    \epoly{f}{\poly{g}{\terminal}} \times_{\poly{g}{\terminal}} \epoly{g}{\terminal} 
    \to
    \poly{f}{\poly{g}{\terminal}}
  \]
\end{lemma}
\begin{proof}
  We observe that there is a diagram
  \[
    \begin{tikzcd}[column sep={7.5em,between origins}, row sep={3em,between origins}]
    	&& {\epoly{f}{\poly{g}{\terminal}} \times_{\poly{g}{\terminal}} \epoly{g}{\terminal}} &
      {\epoly{f}{\poly{g}{\terminal}}} & {\poly{f}{\poly{g}{\terminal}}]} \\
    	& {\epoly{g}{\terminal} \times \epoly{f}{\terminal}} & {\poly{g}{\terminal} \times \epoly{f}{\terminal}} \\
    	{\epoly{g}{\terminal}} & {\poly{g}{\terminal}} && {\epoly{f}{\terminal}} & {\poly{f}{\terminal}} \\
    	\terminal && \terminal && \terminal
    	\arrow["g"', from=3-1, to=3-2]
    	\arrow[from=3-2, to=4-3]
    	\arrow[from=3-4, to=4-3]
    	\arrow["f"', from=3-4, to=3-5]
    	\arrow[from=2-3, to=3-2]
    	\arrow[from=2-3, to=3-4]
    	\arrow[from=2-2, to=3-1]
    	\arrow[from=2-2, to=2-3]
    	\arrow[from=1-3, to=2-2]
    	\arrow[from=1-4, to=2-3]
    	\arrow[from=1-4, to=1-5]
    	\arrow[from=1-5, to=3-5]
    	\arrow[from=1-4, to=3-4]
    	\arrow[from=1-3, to=1-4]
    	\arrow[from=3-1, to=4-1]
    	\arrow[from=3-5, to=4-5]
    \end{tikzcd}
  \]
  in which every square is a pullback square.
  We then refer to~\cite[Proposition 1.12]{kock-polynomial}.
\end{proof}

\subsection{Polynomial Functors of Quasicategories}

In general, polynomial functors on $\sSet$ do not induce polynomial functors
on $\CatInfty$. Flat categorical fibrations~\cite[B.3]{higher-algebra} are
the implementation of exponential fibrations in terms of quasicategories.



\begin{para}
  An inner fibration $p : \strat{X} \to \Delta\ord{2}$ is \defn{flat} if
  the induced map $\strat{X} \times_{\Delta\ord{2}} \Lambda^1\ord{2} \hookrightarrow \strat{X}$
  is a categorical equivalence. An inner fibration $p : E \to K$ is \defn{flat} if
  for every $2$-simplex $\sigma : \Delta\ord{2} \to K$ the pullback
  $\sigma^* p : E \times_K \Delta\ord{2} \to \Delta\ord{2}$ is flat.
  Flat inner fibrations are closed under composition, taking opposites,
  products, equivalence in the arrow category and pullback along any map of
  simplicial sets.
\end{para}

\begin{lemma}\label{lem:b-poly-flat-cat-pullback-equiv}
  Suppose we have a diagram of simplicial sets
  \[
    \begin{tikzcd}
      {\cat{X} \times_{\cat{Y}} A} \ar[r] \ar[d] \pullbackcorner &
      {\cat{X} \times_{\cat{Y}} B} \ar[r] \ar[d] \pullbackcorner &
      {\cat{X}} \ar[d] \\
      {A} \ar[r] &
      {B} \ar[r] &
      {\cat{Y}}
    \end{tikzcd}
  \]
  in which both squares are pullback squares,
  $\cat{X} \to \cat{Y}$ is a flat categorical fibration between
  quasicategories
  and $A \to B$ is a categorical equivalence.  
  Then the map
  \[
    \cat{X} \times_{\cat{Y}} A
    \longrightarrow
    \cat{X} \times_{\cat{Y}} B
  \]
  is a categorical equivalence as well.
\end{lemma}
\begin{proof}
  This is~\cite[Corollary~B.3.15]{higher-algebra}.
\end{proof}

\begin{lemma}
  Let $f : E \to K$ be an inner fibration.
  Then $f$ is flat if and only if for every $2$-simplex $\sigma : \Delta\ord{2} \to K$
  the map
  \[
    \begin{tikzcd}
      \Fun_{/ K}((\Delta\ord{2}, \sigma), (E, f))
      \ar[r, "{\langle 0, 2 \rangle}^*"] &
      \Fun_{/ K}((\Delta\ord{1}, \sigma \circ \langle 0, 2 \rangle), (E, f))
    \end{tikzcd}
  \]
  has weakly contractible fibres.
\end{lemma}
\begin{proof}
  Via~\cite[Remark~B.3.9]{higher-algebra}.
\end{proof}

\begin{proposition}\label{prop:b-poly-flat-cat}
  Let $f : \cat{E} \to \cat{B}$ be a flat categorical fibration between quasicategories.
  Then the polynomial functor $\poly{f}{-} : \sSet \to \sSet$ represented by $f$ satisfies:  
  \begin{enumerate}
    \item $\poly{f}{-}$ preserves categorical fibrations.
    \item $\poly{f}{\cat{C}} \to \cat{B}$ is a categorical fibration for any quasicategory $\cat{C}$.
    \item $\poly{f}{-}$ preserves quasicategories.
    \item $\poly{f}{-}$ preserves categorical equivalences between quasicategories.
  \end{enumerate}
\end{proposition}
\begin{proof}
  The polynomial functor $\poly{f}{-}$ is the composite
  \[
    \begin{tikzcd}
      {\sSet} \ar[r, "\cat{E}^*"] &
      {\sSet_{/\cat{E}}} \ar[r, "\Pi_f"] &
      {\sSet_{/\cat{B}}} \ar[r, "\Sigma_{\cat{B}}"] &
      {\sSet}
    \end{tikzcd}
  \]
  We equip $\sSet$ with the Joyal model structure and the slice categories
  $\sSet_{/ \cat{E}}$ and $\sSet_{/ \cat{B}}$ with the slice model structures
  over the Joyal model structure.
  The product functor $(\cat{E} \times -)$ is a right Quillen functor because $\cat{E}$ is a quasicategory and therefore fibrant.
  Because $f$ is a flat categorical fibration, the dependent product functor
  $\Pi_f$
  is a right Quillen functor by~\cite[Proposition~B.4.5]{higher-algebra}.
  Therefore, the composite $\Pi_f \circ \cat{E}^*$ is also a right Quillen functor.

  Let $p : X \to Y$ be a categorical fibration.
  Right Quillen functors preserve fibrations, and so
  \[
    \begin{tikzcd}[column sep = small]
      {\poly{f}{X}} \ar[rr] \ar[dr] &
      {} &
      {\poly{f}{Y}} \ar[dl] \\
      {} &
      {\cat{B}}
    \end{tikzcd}
  \]
  is a fibration in the slice model structure $\sSet_{/ \cat{B}}$,
  and so by definition of the slice model structure $\poly{f}{X} \to \poly{f}{Y}$ 
  is a categorical fibration.
  Let $\cat{C}$ be a quasicategory then $\cat{C} \to \terminal$ is a categorical fibration,
  and so $\poly{f}{\cat{C}} \to \poly{f}{\terminal} \cong \cat{B}$ is a categorical fibration as well.
  Moreover since $\cat{B}$ is assumed to be a quasicategory, it follows that $\poly{f}{\cat{C}}$ is also a quasicategory.

  Let $\varphi : \cat{C} \to \cat{D}$ be a categorical equivalence between quasicategories.
  By Ken Brown's Lemma the right Quillen functor $\Pi_f \circ \cat{E}^*$ preserves weak
  equivalences between fibrant objects, and so $\varphi$ induces a weak equivalence
  $\poly{f}{\cat{C}} \to \poly{f}{\cat{D}}$ over $\cat{B}$.
\end{proof}

\begin{lemma}\label{lem:b-poly-functorial-equiv}
  Suppose that we have a pullback square of quasicategories
  \[
    \begin{tikzcd}
      {\cat{E}_0} \ar[r] \ar[d, "f_0"'] \pullbackcorner &
      {\cat{E}_1} \ar[d, "f_1"] \\
      {\cat{B}_0} \ar[r, "\varphi", swap] &
      {\cat{B}_1}
    \end{tikzcd}
  \]
  We suppose further that both $f_0$ and $f_1$ are flat categorical fibrations
  and that $\varphi$ is a categorical equivalence.
  Then for every quasicategory $\cat{C}$ the induced map
  \[
    \poly{f_0}{\cat{C}}
    \longrightarrow
    \poly{f_1}{\cat{C}}
  \]
  from Construction~\ref{con:b-poly-functorial} is a categorical equivalence.
\end{lemma}
\begin{proof}
  The map fits into a pullback square
  \[
    \begin{tikzcd}
      {\poly{f_0}{\cat{C}}} \ar[r] \ar[d] \pullbackcorner&
      {\poly{f_1}{\cat{C}}} \ar[d] \\
      {\cat{B}_0} \ar[r, "\varphi", swap] &
      {\cat{B}_1}
    \end{tikzcd}
  \]
  The vertical maps are categorical fibrations by Proposition~\ref{prop:b-poly-flat-cat}
  and therefore the square is a homotopy pullback.
  Because $\varphi$ is a categorical equivalence it then follows that also
  $\poly{f_0}{\cat{C}} \to \poly{f_1}{\cat{C}}$ is a categorical equivalence.
\end{proof}

\begin{lemma}\label{lem:b-poly-retract-equiv}
  Suppose that we have a commutative square of quasicategories
  \[
    \begin{tikzcd}
      {\cat{E}_0} \ar[r, hook, "i'"] \ar[d, "f_0"'] &
      {\cat{E}_1} \ar[d, "f_1"{description}] \ar[r, "r'"] &
      {\cat{E}_0} \ar[d, "f_0"] \\
      {\cat{B}_0} \ar[r, hook, "i"'] &
      {\cat{B}_1} \ar[r, "r", swap] &
      {\cat{B}_0}
    \end{tikzcd}
  \]
  such that the rows compose to the identity.
  We suppose further that both $f_0$ and $f_1$ are flat categorical fibrations
  and that $i$, $i'$ are categorical equivalences.
  Then for every quasicategory $\cat{C}$ the components
  \[
    \begin{tikzcd}
      {\poly{f_0}{\cat{C}}} \ar[r, "I_{\cat{C}}"] &
      {\poly{f_1}{\cat{C}}} \ar[r, "R_{\cat{C}}"] &
      {\poly{f_0}{\cat{C}}}
    \end{tikzcd}
  \]
  of the natural transformations from Construction~\ref{con:b-poly-retract} are categorical equivalences.
\end{lemma}



\section{Geometry}\label{sec:b-geo}

\subsection{Topology and Simplices}

\begin{para}
  By \defn{topological space} we will generally mean a $\Delta$-generated topological
  space~\cite{dugger-delta-generated, delta-generated-presentable}. Topological
  spaces and continuous maps then organise into a $1$-category
  $\TopSp$ that is cartesian closed and presentable.
\end{para}

\begin{para}
  The \defn{topological $n$-simplex} $\DeltaTop{n}$ is the subspace
  \[
    \DeltaTop{n} := \{ (x_0, \ldots, x_n) \mid x_0 + \cdots + x_n = 1 \} \subseteq \R^{n + 1}.
  \]
  This convention is referred to as \defn{barycentric coordinates}.
  The topological simplices organise into a cosimplicial object
  $\DeltaTop{\bullet} : \FinOrd \to \TopSp$. Since $\TopSp$ is cocomplete,
  we then have an induced nerve/realisation adjunction
  \[
    \TopReal{-} : \sSet \rightleftarrows \TopSp : \Sing.
  \]
\end{para}

\begin{para}
  A \defn{triangulation} of a topological space $X$ consists of a poset
  $P$ together with an isomorphism of topological spaces $\TopReal{P} \cong X$.
  A topological space $X$ is \defn{triangulable} if there exists a triangulation
  of $X$. A triangulation $\TopReal{P} \cong X$ is \defn{finite} when $P$ is a
  finite poset, and \defn{locally finite} as long as $P_{/p}$ is finite for all $p \in P$.
  A triangulable space is compact if and only if it admits a finite triangulation;
  it is locally compact if and only if it has a locally finite triangulation.
\end{para}

\begin{remark}
  Triangulations are usually defined in terms of simplicial complexes.
  A simplicial complex is a pair $(V, S)$ consisting of a set $V$ and a
  set $S$ of non-empty subsets of $V$ which together satisfy the following
  conditions:
  \begin{enumerate}
    \item When $\sigma \in S$ and $\tau \subseteq \sigma$ then $\tau \in S$.
    \item When $v \in V$ then $\{ v \} \in S$.
  \end{enumerate}
  Elements of $V$ and $S$ are referred to as \defn{vertices} and \defn{simplices}.
  Then in particular $S$ is a poset when ordered by set inclusion.
  The geometric realisation of the simplicial complex $(V, S)$ as commonly
  defined is isomorphic to $\TopReal{S}$.
  Conversely, for every poset $P$ the \defn{flag complex} of $P$ is the simplicial
  complex whose vertices are the elements of $P$ and whose simplices are
  the images of any injective monotone map $\ord{n} \to P$ for any $n \geq 0$.
  Then the geometric realisation of the flag complex of $P$ is isomorphic to $\TopReal{P}$.
\end{remark}

We will use finiteness of triangulations to ensure tameness of topological spaces.
In order for the notion of tameness to also allow for some non-compact topological spaces,
we will use finite relative triangulations.

\begin{definition}
  A \defn{relative triangulation} of a topological space $X$ consists of a poset
  $P$ and a downwards closed subposet $P_0 \subseteq P$ such that $X$ is isomorphic
  to $\TopReal{P} \setminus \TopReal{P_0}$. The relative triangulation is \defn{finite}
  when $P$ is finite.
\end{definition}

\subsection{Stratified Spaces}\label{sec:b-geo-strat}

Stratified spaces are topological spaces that are composed out of distinguished parts,
which are called the strata.
For example, a simplicial complex or CW complex is a stratified space whose strata are the interiors of the simplices or cells, respectively.
Stratified spaces are also commonly used for spaces with singularities, such as algebraic varieties or pseudomanifolds.
We use stratified spaces to define manifold diagrams.
There is substantial literature on stratified spaces, especially in the smooth case.
However, the field is sufficiently fragmented that there is no standard notion
which we can assume to be known or agreed upon.
We therefore give a quick overview over the type of stratified space that we will
use throughout this thesis.
The ideas in this section are due to~\cite{nand-lal-stratified, douteau-waas-links, higher-algebra, waas-stratified}.

\begin{para}
	The \defn{Alexandroff topology} on a poset is the topology
	the underlying set of $P$ in which a subset $U \subseteq P$ is open
	exactly if it is upwards closed, i.e.\ if for every $x \in U$ and $x \leq y$
	also $y \in P$.
  By sending a poset to the topological space with the Alexandroff topology,
  we obtain a fully faithful functor $\Alex : \Pos \hookrightarrow \TopSp$.
\end{para}

\begin{para}
	A \defn{pre-stratified space} $\strat{X}$ is a topological space $X$ together with a poset
	$\stratPos{\strat{X}}$ and a continuous map $\stratMap{\strat{X}}: X \to \stratPos{\strat{X}}$.
	For any $p \in \stratPos{\strat{X}}$ we denote by $\strat{X}_p := \stratMap{\strat{X}}^{-1}(p)$
	the \defn{stratum} of $\strat{X}$ over $p$.
	The category of pre-stratified spaces is the comma category $\TopSp / \Alex$
	for the inclusion functor $\Alex : \Pos \hookrightarrow \TopSp$ which equips a poset with its Alexandroff topology.
\end{para}

\begin{para}
	The \defn{stratified $n$-simplex} $\DeltaStrat{n}$ is the topological $n$-simplex
	\[
		\DeltaTop{n} = \{ (x_0, \ldots, x_n) \mid x_0 + \cdots + x_n = 1 \} \subseteq \R^{n + 1}
	\]
  presented in barycentric coordinates such that $\stratPos{\DeltaStrat{n}} := \ord{n}$ and
  \[ \stratMap{\DeltaStrat{n}}(x) := \max \{ i \in \ord{n} \mid x_i > 0 \}. \]
  A map $\gamma : \DeltaStrat{1} \to \strat{X}$ of pre-stratified spaces is called
  an \defn{exit path} in $\strat{X}$.
\end{para}


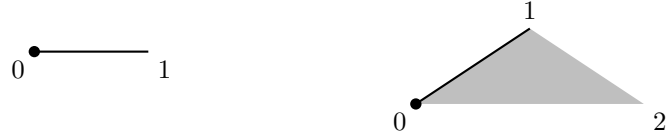
\begin{figure}
  \centering
  \begin{subfigure}[t]{0.3\textwidth}
    \begin{tikzpicture}[scale = 0.5, baseline=(current bounding box.center)]
      \draw[black, thick] (0, 0) -- (3, 0);
      \node[circle, inner sep = 1.5pt, fill = black] at (0, 0) {};
      \node[anchor = north east] at (0, 0) {$0$};
      \node[anchor = north west] at (3, 0) {$1$};
    \end{tikzpicture}
  \end{subfigure}
  \qquad
  \begin{subfigure}[t]{0.3\textwidth}
    \begin{tikzpicture}[scale = 0.5, baseline=(current bounding box.center)]
      \fill[lightgray] (0, 0) -- (3, 2) -- (6, 0);
      \draw[black, thick] (0, 0) -- (3, 2);
      \node[circle, inner sep = 1.5pt, fill = black] at (0, 0) {};
      \node[anchor = north east] at (0, 0) {$0$};
      \node[anchor = south] at (3, 2) {$1$};
      \node[anchor = north west] at (6, 0) {$2$};
    \end{tikzpicture}
  \end{subfigure}
  \caption{The stratified standard simplices $\DeltaStrat{1}$ and $\DeltaStrat{2}$.}
\end{figure}

\begin{para}\label{defn:stratified-space}
  A pre-stratified space $\strat{X}$ \defn{carries the $\Delta_s$-topology}
  when its underlying topological space $\unstrat(\strat{X})$ has the final
  topology with respect to all stratified maps $\DeltaStrat{k} \to \strat{X}$
  for all $k \geq 0$.
  A pre-stratified space $\strat{X}$ is \defn{surjectively stratified} when
  the map $\stratMap{\strat{X}} : X \to \stratPos{\strat{X}}$ surjective.
  In other words, $\strat{X}$ is surjectively stratified when its strata are non-empty.
  A \defn{stratified space} is a pre-stratified space $\strat{X}$ that carries
  the $\Delta_s$-topology, is surjectively stratified and satisfies the following
  condition:
	Whenever $x_0, x_1 \in \strat{X}$ satisfy     
  $\stratMap{\strat{X}}(x_0) \leq \stratMap{\strat{X}}(x_1)$ then there exists a sequence
  of exit paths $\gamma_0, \ldots, \gamma_k : \DeltaStrat{1} \to \strat{X}$ such that
  $\gamma_0(0) = x_0$, $\gamma_k(1) = x_1$ and $\gamma_{i - 1}(1) = \gamma_i(0)$ for all
  $1 \leq i \leq k$.
\end{para}

\begin{para}
	The category of stratified spaces $\Strat$ is the full subcategory of $\TopSp / \Alex$.
  We write $\unstrat(-) : \Strat \to \TopSp$ for the functor which forgets the stratification and sends a stratified space to its \defn{underlying topological space}.
	The category $\Strat$ is a coreflective subcategory of the comma category
  $\TopSp / \Alex$, i.e. the inclusion functor
  $\Strat \hookrightarrow \TopSp \downarrow \Alex$
  has a right adjoint.
	In particular $\Strat$ is presentable, complete and cocomplete.
\end{para}

\begin{para}
  Let $f : \strat{X} \to \strat{Y}$ be a map of stratified spaces.
  \begin{enumerate}
    \item The map $f$ is a \defn{refinement map} when its underlying map
    of topological spaces $\unstrat(f) : \unstrat(\strat{X}) \to \unstrat(\strat{Y})$
    is an isomorphism and $\stratPos{f} : \stratPos{\strat{X}} \to \stratPos{\strat{Y}}$
    is surjective. In this case we say that $\strat{X}$ is a \defn{refinement}
    of $\strat{Y}$ or, equivalently, that $\strat{Y}$ is a \defn{coarsening} of $\strat{X}$.
    \item The map $f$ is \defn{open} when $\unstrat(f) : \unstrat(\strat{X}) \to \unstrat(\strat{Y})$ is an open embedding of topological spaces and $f$ restricts to a refinement map onto its image.
    \item The map $f$ is a \defn{constructible embedding} when its map on posets
    $\stratPos{f} : \stratPos{\strat{X}} \to \stratPos{\strat{Y}}$ is injective and there is a pullback
    square of topological spaces
  	\[
  		\begin{tikzcd}
  			{X} \ar[r, "f", hook] \ar[d, "\stratMap{\strat{X}}", swap] \pullbackcorner &
  			{Y} \ar[d, "\stratMap{\strat{Y}}"] \\
  			{\stratPos{\strat{X}}} \ar[r, hook] &
  			{\stratPos{\strat{Y}}}
  		\end{tikzcd}
  	\]
  	In this situation we say that $\strat{X}$ is a \defn{constructible subspace}
  	of $\strat{Y}$.
  \end{enumerate}
\end{para}

\begin{para}
  The stratified $n$-simplices define a cosimplicial object
  \[
    \DeltaStrat{-} : \FinOrd \to \Strat.
  \]
  Since $\Strat$ is cocomplete, we have an induced nerve/realisation adjunction
  \[
    \StratReal{-} : \sSet \rightleftarrows \Strat : \Exit
  \]
  The functor $\StratReal{-} : \Strat \to \sSet$ preserves finite products.
\end{para}

\begin{para}
  For every poset $P$, interpreted as a simplicial set via its nerve, we have
  a stratified space $\StratReal{P}$. The underlying unstratified space of
  $\StratReal{P}$ is the geometric realisation $\TopReal{P}$. The stratified
  space $\StratReal{P}$ is glued together via a colimit in $\Strat$ from stratified simplices $\DeltaStrat{k}$
  for each flag of $P$ given by an injective map $\Delta\ord{k} \hookrightarrow P$.
  A \defn{triangulation} of a stratified space $\strat{X}$ is a poset $S$
  together with a refinement map $\StratReal{S} \to \strat{X}$.
  A stratified space $\strat{X}$ is \defn{triangulable} when it admits
  a triangulation.
\end{para}

\begin{para}
  Concretely, the functor $\Exit$ sends a stratified space $\strat{X}$ to the
  simplicial set $\Exit(\strat{X})$ whose $k$-simplices are the maps of stratified
  spaces $\DeltaStrat{k} \to \strat{X}$.
  We say that the stratified space $\strat{X}$ is \defn{fibrant} if the simplicial set
  $\Exit(\strat{X})$ is a quasicategory; in this case $\Exit(\strat{X})$ represents
  the \defn{$\infty$-category of exit paths} in $\strat{X}$.
  A map $f : \strat{X} \to \strat{Y}$ of stratified spaces is a \defn{stratified weak equivalence}
  when the induced map $\Exit(f) : \Exit(\strat{X}) \to \Exit(\strat{Y})$ is a categorical equivalence.
  The unit $\eta_X : X \to \Exit(\StratReal{X})$ of the adjunction $\StratReal{-} \dashv \Exit$ is a monomorphism and a categorical
  equivalence. In particular $\eta_X$ is an acyclic cofibration in the Joyal
  model structure.
\end{para}

\begin{remark}
  The terminology appears to imply that we could equip $\Strat$ with a model structure.
  Unfortunately this is not possible~\cite{nand-lal-stratified}.
  More recently it has been shown that $\Strat$ admits a left semi-model structure~\cite{waas-stratified},
  but we will not need this result here.
  The localisation of $\Strat$ at the stratified weak equivalences is a reflective
  $\infty$-subcategory of $\CatInfty$ consisting of layered $\infty$-categories:
  An $\infty$-category $\cat{C}$ is \defn{layered} when every endomorphism is an automorphism.
  Equivalently $\cat{C}$ is layered when there exists a poset $P$ and a conservative functor
  $\cat{C} \to P$.
\end{remark}

\begin{para}
  A stratified space $\strat{X}$ satisfies the \defn{frontier condition} when for
  each $p \leq q$ in $\stratPos{\strat{X}}$ the stratum $\strat{X}_p$ is contained
  within the closure of $\strat{X}_q$ within $\strat{X}$.
  Equivalently, a stratified space $\strat{X}$ satisfies the frontier condition
  when for every pair of points $x_0, x_1 \in \strat{X}$ with $\stratPos{\strat{X}}(x_0) \leq \stratPos{\strat{X}}(x_1)$ there exists an exit path $\gamma : \DeltaStrat{1} \to \strat{X}$ with
  $\gamma(0) = 0$ and $\gamma(1) = 1$.
  Every fibrant stratified space satisfies the frontier condition.
\end{para}

\begin{para}
  Let $\strat{L}$ be a compact and Hausdorff stratified space.
  The \defn{stratified cone} $\cone\strat{L}$ is the pushout of stratified spaces
  \[
    \begin{tikzcd}
      {\{ 0 \} \times \strat{L}} \ar[r] \ar[d, hook] \pushoutcorner &
      {\terminal} \ar[d] \\
      {\DeltaStrat{1} \times \strat{L}} \ar[r] &
      {\cone \strat{L}}
    \end{tikzcd}
  \]
  In particular the base poset of $\stratPos{\cone \strat{L}}$
  is the left cone
  $\stratPos{\strat{L}}^{\triangleleft}$, i.e. the union of $\stratPos{\strat{L}}$
  with an additional minimal element $\bot$.
  A stratified space $\strat{X}$ is \defn{conical} if for every point $x \in \strat{X}$
  with $p = \stratMap{\strat{X}}(x)$
  there exists an open neighbourhood $\strat{U}$ of $x$ in $\strat{X}$,
  an open neighbourhood $V$ of $x$ in its stratum $\strat{X}_p$ and a compact
  Hausdorff stratified space $\strat{L}$ so that $\strat{U}$ is isomorphic as a stratified
  space to $\cone \strat{L} \times V$.
  Every conical stratified space is fibrant.
  For every poset $P$ the stratified geometric realisation $\StratReal{P}$ is conical,
  and so in particular $\StratReal{P}$ is fibrant.
\end{para}

\begin{remark}
  The stratified cone $\cone \strat{L}$ can be defined for any stratified space $\strat{L}$
  via the teardrop topology (see for instance~\cite[Definition~A.5.3]{higher-algebra}).
  When $\strat{L}$ is compact and Hausdorff the general definition agrees with the pushout.
\end{remark}

We will make use of an alternative characterisation of triangulations for
stratified spaces. 

\begin{para}
  Let $T$ be a subset of the maps of $\Delta\ord{k}$ which includes all the degenerate
  maps so that $(\Delta\ord{k}, T)$ is a marked simplicial set. Denote by
  $\pi : \ord{k} \to P$ the quotient of the poset $\ord{k}$ which identifies
  $i, j \in \ord{k}$ when the map $i \to j$ is contained within $T$.
  We then let $\StratRealMark{\Delta\ord{k}, T}$ be the stratification of
  $\DeltaTop{k}$ over the poset $P$ with stratifying map
  \[
    \stratMap{\StratRealMark{\Delta\ord{k}, T}} :=
    \pi \circ 
    \stratMap{\DeltaStrat{k}}.
  \]
  This defines a nerve/realisation adjunction
  \[
    \StratRealMark{-} : \sSetM \rightleftarrows \Strat : \ExitMark
  \]
  A stratified space $\strat{X}$ is triangulable if and only if there exists a
  marked poset $P$ together with an isomorphism of stratified spaces
  $\StratRealMark{P} \cong \strat{X}$.
  For every marked poset $P$ the unit $\eta_S : S \to \ExitMark(\StratRealMark{S})$
  of the adjunction $\StratRealMark{-} \dashv \ExitMark$ is a marked equivalence.
\end{para}

\begin{remark}
  Presenting triangulable stratified spaces via marked simplicial sets is unusual.
  A more common approach defines a combinatorial analogue of $\Strat$ via the
  comma category of $\sSet$ over the inclusion $\Pos \hookrightarrow \sSet$.
  We haven chosen to use marked simplicial sets instead for additional flexibility.
  Concretely we make use of a map of marked simplicial sets $S \to \cat{C}^\natural$
  from a marked poset $S$ to a quasicategory with natural marking $\cat{C}^\natural$,
  which can not be expressed in $\sSet / \Pos$ in general.
\end{remark}

\begin{para}
  A map of stratified spaces $p : \strat{E} \to \strat{B}$ is a \defn{(stratified) trivial bundle}
  when there exists a stratified space $\strat{F}$ such that $p$ is isomorphic over $\strat{B}$
  to the projection map $\strat{F} \times \strat{B} \to \strat{B}$.
  A map of stratified spaces $p : \strat{E} \to \strat{B}$ is a
  \defn{stratified fibre bundle} if for every point $b \in \strat{B}$
  with $i = \stratMap{\strat{B}}(b)$
  there exists an open neighbourhood $U \subseteq \strat{B}_i$ within
  the stratum that contains $b$ such that $p$ restricts to a stratified trivial
  bundle $\strat{E} \times_{\strat{B}} U \to U$.
\end{para}

\begin{para}
  Let $p : \strat{X} \to \strat{Y}$ be a stratified fibre bundle such that
  the induced map of simplicial sets $\Exit(p)$ is an inner fibration.
  Then $\Exit(p)$ is an isofibration.
\end{para}




\subsection{PL Stratified Pseudomanifolds}\label{sec:b-geo-pseudo}

\begin{definition}\label{def:pl-pseudomanifold}
  A stratified space $\strat{X}$ is a \defn{PL stratified pseudomanifold}
  if $\unstrat(\strat{X})$ is piecewise linear and
  for every $x \in \strat{X}$ there exists a compact PL stratified pseudomanifold $\strat{L}$, called the \defn{link}, and a PL stratified embedding
  $u : \R^i \times \cone \strat{L} \hookrightarrow \strat{X}$
  such that $u(0, *) = x$.
\end{definition}

\begin{remark}
  The definition of a PL stratified pseudomanifold appears circular but is recursive:
  The link $\strat{L}$ at some point of a PL stratified pseudomanifold $\strat{X}$
  has fewer strata than $\strat{X}$.
  Our definition differs slightly from the variants typically found in the literature:
  \begin{enumerate}
    \item PL stratified pseudomanifolds are most often defined via a filtration
    instead of a stratification. With stratifications defined as above,
    this makes no difference.
    \item The literature concentrates on PL stratified $n$-pseudomanfolds $\strat{X}$ for
    some $n \geq 0$, which are restricted to contain only strata of dimension up to $n$
    so that the $n$-dimensional strata in $\strat{X}$ form a dense subspace.
    \item In some sources it is typical to assume that a stratified PL $n$-pseudomanifold
    has no strata of dimension $n - 1$. This is due to technical difficulties that
    arise in intersection homology if such strata were allowed.
    We do not use intersection homology and so this restriction does not make sense for our purposes.
  \end{enumerate}
\end{remark}

\begin{observation}
  Suppose that $\strat{X}$ is a PL stratified pseudomanifold.
  Then each stratum of $\strat{X}$ is a connected PL manifold.
  Since $\strat{X}$ is locally conical it is fibrant.
  Moreover $\strat{X}$ satisfies the frontier condition.
\end{observation}

\begin{proposition}\label{prop:coarsest-subdivision}
	Let $\strat{X}$ be a stratified space that admits a locally
  finite triangulation and let
  $A = (A_i \mid i \in I)$ be a locally
	finite family of closed PL subspaces of $\unstrat(\strat{X})$.
	Then there exists a coarsest refinement
  $\strat{Y} \to \strat{X}$
	such that $\strat{Y}$ is a PL stratified pseudomanifold
	and $A_i$ is a union of strata in $\strat{Y}$ for each $i \in I$.
\end{proposition}
\begin{proof}
  This follows from the theory of intrinsic dimension developed in~\cite{transversality-polyhedra}.
  Concretely we can construct $\strat{Y}$ as follows.
  We say that a triangulation of $\strat{X}$ is adapted to the family $A$ when
  each $A_i$ is a subcomplex of the triangulation.
  The intrinsic dimension $\dim(x)$ of a point $x \in \strat{X}$ is the largest $k \geq 0$ such that there exists an triangulation of $\strat{X}$ adapted to $A$
  for which $x$ is within the interior of a $k$-simplex.
  Then $\dim : \unstrat(\strat{X}) \to \Nat^\op$ defines a pre-stratification
  which induces the desired stratification $\strat{Y}$ by splitting the strata into path-connected components.
\end{proof}

\begin{observation}\label{obs:coarsest-subdivision-local}
  The coarsest stratification is determined locally
  since the intrinsic dimension is a local property.
  In the situation of Proposition~\ref{prop:coarsest-subdivision}
  suppose that $\strat{U} \subseteq \strat{X}$ is an open PL subspace.
  Then the coarsest PL stratified pseudomanifold that refines $\strat{U}$
  and is adapted to
  $A_i \cap \strat{U}$ for all $i \in I$
  agrees with $\strat{Y} \cap \strat{U}$.
\end{observation}

\begin{lemma}\label{lem:strat-frame-sober}
	Let $\strat{X}, \strat{Y}$ be stratified spaces
  with locally finite stratifications
  that satisfy the frontier condition,
	and let $f : \unstrat(\strat{X}) \to \unstrat(\strat{Y})$ be a map of the underlying topological spaces.
	Then $f$ respects the stratification if and only if
	$f^{-1}(\cl(\strat{Y}_p))$ is a union of strata
	for all $p \in \stratPos{\strat{Y}}$.
\end{lemma}
\begin{proof}
	Because $\strat{X}$ satisfies the frontier condition,
	there is an order-preserving isomorphism between
	the closed sets in $\stratPos{\strat{X}}$
	and the closed subsets of $X$ which are unions of strata.
	The same holds for $\strat{Y}$ as well.
	Because $\stratPos{\strat{X}}$, $\stratPos{\strat{Y}}$
	are locally finite, they are sober.
	But then the claim follows since a map between the posets
	is the same as a frame map.
\end{proof}

\begin{lemma}\label{lem:strat-coarsest-domain}
	Let $\strat{A}$, $\strat{X}$ be stratified spaces that
  both satisfy the frontier condition and admit locally finite
  triangulations.
	Suppose further that $f : \unstrat(\strat{A}) \to \unstrat(\strat{X})$ a PL map between the underlying topological spaces
	such that the fibres of $f$ intersect finitely many strata.
	Then there exists a coarsest refinement $\strat{B} \to \strat{A}$
	such that $f$ becomes a stratified map $\strat{B} \to \strat{X}$
	and $\strat{B}$ is a PL stratified pseudomanifold.
\end{lemma}
\begin{proof}
	By Lemma~\ref{lem:strat-frame-sober} a refinement
	$\strat{B} \to \strat{A}$ makes $f$ into a stratified map
	$\strat{B} \to \strat{X}$ if and only
	if for each $p \in \stratPos{X}$ the preimage
	$f^{-1}(\cl(\strat{X}_p))$ is a union of strata in $\strat{B}$.
	But then the claim follows by Proposition~\ref{prop:coarsest-subdivision}.
\end{proof}

\subsection{Open Mapping Cylinders}\label{sec:b-geo-ocyl}

The mapping cylinder of a map of topological spaces is a common construction 
in algebraic topology. For stratified spaces, there are multiple versions
of this construction, one of which is the open mapping cylinder~\cite{stratified-homotopy-hypothesis}.

\begin{construction}\label{con:b-geo-ocyl}
  Suppose that $\strat{X}_\bullet : \ord{k} \to \Strat$ is a sequence of
  open maps of stratified spaces $f_i : \strat{X}_{i - 1} \to \strat{X}_{i}$
  for $1 \leq i \leq k$.
  Then the \defn{$k$-fold open mapping cylinder} $\CylO(\strat{X}_\bullet)$
  is the stratified space obtained as the quotient of the disjoint union
  \[
    \coprod_{i \in \ord{k}} \strat{X}_i \times \DeltaStrat{k}_{\geq i}
  \]
  by the smallest equivalence relation $\sim$ such that
  $(x, b) \sim (f_i(x), b)$ whenever $1 \leq i \leq k$, $x \in \strat{X}_{i - 1}$
  and $b \in \DeltaStrat{k}_{\geq i}$.
  There is a natural projection map
  \[ \CylO(\strat{X}_\bullet) \longrightarrow \DeltaStrat{k} \]
  so that the $k$-fold open mapping cylinder defines a functor
  \[
    \CylO^k(-) : \Fun(\Delta\ord{k}, \Strat) \longrightarrow \Strat_{/ \DeltaStrat{k}}
  \]
\end{construction}

\begin{observation}\label{obs:b-geo-ocyl-pullback}
  For any order-preserving map $\ord{k} \to \ord{s}$, the $k$-fold and $s$-fold
  open mapping cylinder functors are compatible with the pullback of maps
  of stratified spaces:
  \[
    \begin{tikzcd}[column sep = large]
      {\Fun(\Delta\ord{s}, \Strat)} \ar[r, "\CylO^s(-)"] \ar[d] &
      {\Strat_{/ \Delta\ord{s}}} \ar[d] \\
      {\Fun(\Delta\ord{k}, \Strat)} \ar[r, "\CylO^k(-)"'] &
      {\Strat_{/ \Delta\ord{k}}}
    \end{tikzcd}
  \]
\end{observation}

\begin{observation}\label{obs:b-geo-ocyl-retract}
  Suppose that $\strat{X}_\bullet : \ord{k} \to \Strat$ is a sequence of
  open maps of stratified spaces. Then the inclusion map $\strat{X}_k \to \CylO(\strat{X}_\bullet)$
  has a retraction.
\end{observation}

\begin{lemma}\label{lem:b-geo-ocyl-isofibration}
  Suppose that we have a diagram of stratified spaces
  \[
    \begin{tikzcd}
      {\strat{X}_0} \ar[r, "\varphi_1"] \ar[d, "f_0"] &
      {\strat{X}_1} \ar[r] \ar[d, "f_1"] &
      {\cdots} \ar[r] &
      {\strat{X}_{k - 1}} \ar[r, "\varphi_k"] \ar[d, "f_{k - 1}"]  &
      {\strat{X}_k} \ar[d, "f_k"] \\
      {\strat{B}_0} \ar[r, "\psi_1"'] &
      {\strat{B}_1} \ar[r] &
      {\cdots} \ar[r] &
      {\strat{B}_{k - 1}} \ar[r, "\psi_k"'] &
      {\strat{B}_k}
    \end{tikzcd}
  \]
  such that each $\Exit(f_i)$ is an isofibration
  and the rows consist of open maps of stratified spaces.
  Let $\pi : \CylO^k(\strat{X}_\bullet) \to \CylO^k(\strat{B}_\bullet)$ be the induced map between 
  the $k$-fold open mapping cylinders of the rows.
  Then $\Exit(\pi)$ is an isofibration.
\end{lemma}
\begin{proof}
  Suppose that we have a lifting problem of the form
  \[
    \begin{tikzcd}
      {\HornStrat{i}{k}} \ar[r, "\sigma'"] \ar[d, hook] &
      {\CylO^k(\strat{X}_\bullet)} \ar[d, "\pi"] \\
      {\DeltaStrat{k}} \ar[r, "\tau"'] \ar[ur, dashed] &
      {\CylO^k(\strat{Y}_\bullet)}
    \end{tikzcd}
  \]
  for some $0 < i < k$.
  By induction on $k \geq 0$ we can reduce to the case that $\sigma'(k)$ 
  is contained within the image of $\strat{X}_k \times \DeltaStrat{k}_{\geq k} \to \CylO^k(\strat{X}_\bullet)$.
  Since $\Exit(f_k)$ is an isofibration, we then have a solution to the lifting problem
  \[
    \begin{tikzcd}
      {\HornStrat{i}{k}} \ar[r, "\sigma'"] \ar[d, hook] &
      {\CylO(\strat{X}_\bullet)} \ar[r] &
      {\strat{X}_k} \ar[d, "f_k"] \\
      {\DeltaStrat{k}} \ar[r, "\tau"'] \ar[urr, dashed, "\rho"{description}] &
      {\CylO(\strat{B}_\bullet)} \ar[r] &
      {\strat{B}_k}
    \end{tikzcd}
  \]
  We then define an extension $\sigma : \DeltaStrat{k} \to \CylO^k(\strat{X}_\bullet)$
  of $\sigma'$ by letting $\sigma(b) := [\rho(b), \tau(b)]$ for all
  $b \in \DeltaStrat{k} \setminus \HornStrat{i}{k}$.
  This defines a solution to our original lifting problem.
\end{proof}

\begin{lemma}\label{lem:b-geo-ocyl-cocartesian-map}
  Suppose that $\varphi : \strat{X} \to \strat{Y}$ is an open map of stratified spaces
  such that $\strat{Y}$ is fibrant.
  Let $\pi : \CylO(\varphi) \to \DeltaStrat{1}$ be the projection map
  from the open mapping cylinder and let $\gamma : \DeltaStrat{1} \to \CylO(\varphi)$ be a section of $\pi$
  such that $\Exit(\CylO(\varphi)) \to \Exit(\strat{Y})$ sends $\gamma$
  to an equivalence.
  Then $\gamma$ is $\Exit(\pi)$-cocartesian.
\end{lemma}
\begin{proof}
  Suppose that we have a lifting problem
  \[
    \begin{tikzcd}
      {\DeltaStrat{1}} \ar[d, hook, "{\langle 0, 1 \rangle}"'] \ar[dr, "\gamma"] \\
      {\HornStrat{0}{k}} \ar[r, "\sigma'"{description}] \ar[d, hook] &
      {\CylO(\varphi)} \ar[d] \\
      {\DeltaStrat{k}} \ar[r, "\tau"'] \ar[ur, dashed] &
      {\DeltaStrat{1}}
    \end{tikzcd}
  \]
  We then obtain an induced lifting problem
  \[
    \begin{tikzcd}
      {\HornStrat{i}{k}} \ar[r, "\sigma'"] \ar[d, hook] &
      {\CylO(\varphi)} \ar[r] &
      {\strat{Y}} \\
      {\DeltaStrat{k}} \ar[urr, dashed, "\rho"'] &
      {} &
      {}
    \end{tikzcd}
  \]
  which has a solution because $\strat{Y}$ is fibrant.
  We then define an extension $\sigma : \DeltaStrat{k} \to \CylO(\varphi)$
  of $\sigma'$ by letting $\sigma(b) := [\rho(b), \tau(b)]$ for all
  $b \in \DeltaStrat{k} \setminus \HornStrat{i}{k}$.
  This defines a solution to our original lifting problem.
\end{proof}

\begin{lemma}\label{lem:b-geo-ocyl-cocartesian-fibration}
  Suppose that $\varphi : \strat{X} \to \strat{Y}$ is an open map of fibrant stratified spaces. Let $\pi : \CylO(\varphi) \to \DeltaStrat{1}$ be the projection map
  from the open mapping cylinder. Then $\Exit(\pi)$ is a cocartesian fibration.
\end{lemma}
\begin{proof}
  By Lemma~\ref{lem:b-geo-ocyl-isofibration} the map $\Exit(\pi)$ is an isofibration.
  To show that $\Exit(\pi)$ is cocartesian, it suffices to show that the identity
  exit path $\id : \DeltaStrat{1} \to \DeltaStrat{1}$ has an $\Exit(\pi)$-cocartesian
  lift starting at any $x \in \strat{X}$.
  We let $\gamma : \DeltaStrat{1} \to \CylO(\varphi)$ be the map defined by
  $\gamma(t) = [x, t]$, then $\gamma$ is $\Exit(\pi)$-cocartesian by Lemma~\ref{lem:b-geo-ocyl-cocartesian-map}.
\end{proof}



%% file: chapter-fct.tex
\chapter{Framed Combinatorial Topology}\label{sec:fct}

\section{Framed Stratified Spaces and Bundles}\label{sec:fct-framing}

We study stratifications that are embedded into Euclidean space.
In particular, we are interested in the relative arrangement of the strata. 
For this purpose we equip a stratified space $\strat{X}$ with an $n$-framing, which is an embedding of
the underlying topological space $\unstrat(\strat{X})$ into $\R^n$, and allow only those
transformations that preserve the projections.
This allows us to detect and give algebraic significance to the changes in relative arrangement of strata, which will be useful for our theory of manifold diagrams.

\begin{definition}
  For every $0 \leq m \leq n$ there is a projection map
	$\R^{n} \cong \R^{n - m} \times \R^m \to \R^m$
	onto the last $m$ coordinates in the Euclidean standard basis.
  We will denote this map by $\pi^m$ assuming that $n$ will be clear from the context.
  For any $0 \leq m \leq n$ we define an equivalence relation $\sim_m$ on $\R^n$
  so that $x \sim_m y$ for $x, y \in \R^n$ whenever $\pi^m(x) = \pi^m(y)$.
\end{definition}

\begin{definition}
  An $n$-framed stratified space is a stratified space $\strat{X}$ together
  with a continuous embedding of the underlying topological space
  $\framing(\strat{X}) : \unstrat(\strat{X}) \hookrightarrow \R^n$.
  A map of $n$-framed stratified spaces $\varphi : \strat{X} \to \strat{Y}$
  is a map of stratified spaces such that for all $0 \leq m \leq n$ and
  elements $x_1, x_2 \in \strat{X}$ we have
  \[
    \framing(\strat{X})(x_1) \sim_m
    \framing(\strat{X})(x_2)
    \quad
    \Longrightarrow
    \quad
    \framing(\strat{Y})(\varphi(x_1)) \sim_m
    \framing(\strat{Y})(\varphi(x_2))
  \]
  The $n$-framed stratified spaces and their maps form a $1$-category
  $\FrStrat^n$.
\end{definition}

In $2$-dimensional diagrams we adopt the convention that, unless specified
otherwise, the first coordinate axis goes from left to right on the page
while the second coordinate axis goes from bottom to top.
This is consistent with a common convention on how to draw string diagrams.

\begin{example}
  Suppose we have a stratification $\strat{X}$ of the plane $\R^2$ with two $0$-strata:
  \[
    \begin{tikzpicture}[scale = 0.25, baseline=(current bounding box.center)]
      \fill[mesh-background] (0, 0) rectangle (6, 6);
      \node[mesh-vertex-blue] at (2, 2) {};
      \node[mesh-vertex-orange] at (4, 4) {};
    \end{tikzpicture}
  \]
  Then up to framed stratified isomorphism we may move around the $0$-strata,
  as long as the orange stratum remains above the blue one:
  \[
    \begin{tikzpicture}[scale = 0.25, baseline=(current bounding box.center)]
      \fill[mesh-background] (0, 0) rectangle (6, 6);
      \node[mesh-vertex-blue] at (2, 2) {};
      \node[mesh-vertex-orange] at (4, 4) {};
    \end{tikzpicture}
    \quad
    \cong
    \quad
    \begin{tikzpicture}[scale = 0.25, baseline=(current bounding box.center)]
      \fill[mesh-background] (0, 0) rectangle (6, 6);
      \node[mesh-vertex-blue] at (3, 2) {};
      \node[mesh-vertex-orange] at (3, 4) {};
    \end{tikzpicture}
    \quad
    \cong
    \quad
    \begin{tikzpicture}[scale = 0.25, baseline=(current bounding box.center)]
      \fill[mesh-background] (0, 0) rectangle (6, 6);
      \node[mesh-vertex-blue] at (4, 2) {};
      \node[mesh-vertex-orange] at (2, 4) {};
    \end{tikzpicture}
  \]  
  However we may not move the $0$-strata to the same height or past each other:
  \[
    \begin{tikzpicture}[scale = 0.25, baseline=(current bounding box.center)]
      \fill[mesh-background] (0, 0) rectangle (6, 6);
      \node[mesh-vertex-blue] at (2, 2) {};
      \node[mesh-vertex-orange] at (4, 4) {};
    \end{tikzpicture}
    \quad
    \not\cong
    \quad
    \begin{tikzpicture}[scale = 0.25, baseline=(current bounding box.center)]
      \fill[mesh-background] (0, 0) rectangle (6, 6);
      \node[mesh-vertex-blue] at (2, 3) {};
      \node[mesh-vertex-orange] at (4, 3) {};
    \end{tikzpicture}
    \quad
    \not\cong
    \quad
    \begin{tikzpicture}[scale = 0.25, baseline=(current bounding box.center)]
      \fill[mesh-background] (0, 0) rectangle (6, 6);
      \node[mesh-vertex-blue] at (2, 4) {};
      \node[mesh-vertex-orange] at (4, 2) {};
    \end{tikzpicture}
  \]
\end{example}

\begin{example}
  Say we have a stratification $\strat{X}$ of $\R^2$ consisting of the strata
  \[
    \{ (x_1, x_2) \mid x_1 = 0 \},
    \{ (x_1, x_2) \mid x_1 < 0 \},
    \{ (x_1, x_2) \mid x_1 > 0 \} \subset \R^2.
  \]
  Up to $2$-framed isomorphism we may perturb the vertical line horizontally
  \[
    \begin{tikzpicture}[scale = 0.25, baseline=(current bounding box.center)]
      \fill[mesh-background] (0, 0) rectangle (6, 6);
      \draw[mesh-stratum] (3, 0) -- (3, 6);
    \end{tikzpicture}
    \quad
    \cong
    \quad
    \begin{tikzpicture}[scale = 0.25, baseline=(current bounding box.center)]
      \fill[mesh-background] (0, 0) rectangle (6, 6);
      \draw[mesh-stratum] (3, 0) -- (2, 1) -- (3, 3) -- (5, 4) -- (3, 6);
    \end{tikzpicture}
    \quad
    \cong
    \quad
    \begin{tikzpicture}[scale = 0.25, baseline=(current bounding box.center)]
      \fill[mesh-background] (0, 0) rectangle (6, 6);
      \draw[mesh-stratum] (3, 0) .. controls +(0, 1) and +(0, -1) .. 
        (2, 3) .. controls +(0, 1) and +(0, -1) .. (5, 6);
    \end{tikzpicture}
  \]
  We can not make the line turn around or snake:
  \[
    \begin{tikzpicture}[scale = 0.25, baseline=(current bounding box.center)]
      \fill[mesh-background] (0, 0) rectangle (6, 6);
      \draw[mesh-stratum] (3, 0) -- (3, 6);
    \end{tikzpicture}
    \quad
    \not\cong
    \quad
    \begin{tikzpicture}[scale = 0.25, baseline=(current bounding box.center)]
      \fill[mesh-background] (0, 0) rectangle (6, 6);
      \draw[mesh-stratum] (1, 0) -- (1, 4) arc (180:0:1) -- (3, 2) arc (180:360:1) -- (5, 6);
    \end{tikzpicture}
  \]
\end{example}

\begin{definition}
	An \defn{$n$-framed stratified bundle} is a map
	$f : \strat{E} \to \strat{B}$
	of stratified spaces
	together with a continuous embedding
	$\framing(f) : \unstrat(\strat{E}) \hookrightarrow \R^n \times \unstrat(\strat{B})$ over $\unstrat(\strat{B})$.  
  A map of $n$-framed stratified bundles $f$ and $g$ is a diagram
  of stratified spaces
	\[
		\begin{tikzcd}
			\strat{E} \ar[d, "f", swap] \ar[r, "\varphi"] & \strat{D} \ar[d, "g"] \\
			\strat{B} \ar[r, "\psi", swap] & \strat{A}
		\end{tikzcd}
	\]
  such that for every point $b \in \strat{B}$ the map on fibres
  $
    \strat{E} \times_{\strat{B}} \{ b \} \to 
    \strat{D} \times_{\strat{A}} \{ \psi(b) \}
  $
  is a map of $n$-framed stratified spaces with the induced framing.
	The $n$-framed bundles and their maps form a $1$-category $\FrBun^n$.
\end{definition}

\begin{observation}
  We have a fully faithful inclusion functor $\FrStrat^n \hookrightarrow \FrBun^n$
  which sends an $n$-framed stratified space $\strat{X}$ to the $n$-framed
  stratified bundle $\strat{X} \to \DeltaStrat{0}$.
\end{observation}

\begin{construction}
  Suppose that $f : \strat{E} \to \strat{B}$ is an $n$-framed stratified bundle
  and $\psi : \strat{A} \to \strat{B}$ a map of stratified spaces.
  When
  \[
    \begin{tikzcd}
      {\strat{A} \times_{\strat{B}} \strat{E}} \ar[r, "\varphi"] \ar[d, "\psi^* f", swap] &
      {\strat{E}} \ar[d, "f"] \\
      {\strat{A}} \ar[r, "\psi"'] &
      {\strat{B}}
    \end{tikzcd}
  \]
  is the pullback in $\Strat$, we can equip $\psi^* f$ with a canonical framing
  $\framing(\psi^* f)$ induced by $\framing(f)$.
  This is a cartesian lift for the functor $\FrBun^n \to \Strat$ which sends
  an $n$-framed stratified bundle $\varphi : \strat{E} \to \strat{B}$ to the
  stratified base space $\strat{B}$.
\end{construction}


\begin{remark}
  Typically, constructions based on bundles and pullback run into coherence issues, with pullbacks only being defined up to isomorphism.
  Via the framing of an $n$-framed bundle
  $f : \strat{E} \to \strat{B}$ we can canonically treat
  the underlying space $\unstrat(\strat{E})$
  as a subspace of $\R^n \times \unstrat(\strat{B})$.
  We can also present the set of elements of the poset $\stratPos{\strat{E}}$
  of strata as a quotient of the points $\unstrat(\strat{E})$.
  This provides a canonical representative for every isomorphism class and thus a functorial strictification.
  We will use this trick implicitly for any constructions based on framed stratified bundles that we will define below.
\end{remark}

\begin{definition}
  An $n$-framed stratified bundle $f : \strat{E} \to \strat{B}$ is a \defn{submersion}
  when for every point $x \in \strat{E}$ there exists an open embedding 
  of $n$-framed stratified bundles
  \[
    \begin{tikzcd}
      {\strat{U} \times \strat{V}} \ar[r, hook, "u"] \ar[d] &
      {\strat{E}} \ar[d, "f"] \\
      {\strat{V}} \ar[r, hook, "v"'] &
      {\strat{B}}
    \end{tikzcd}
  \]
  such that $x$ is in the image of $u$.
\end{definition}
\begin{lemma}
  Framed stratified submersions are closed under pullback.
\end{lemma}

\begin{definition}
  Let $f : \strat{X} \to \strat{Y}$ be an $n$-framed stratified bundle
  and $g : \strat{Y} \to \strat{Z}$ an $m$-framed stratified bundle.
  The \defn{composite framing} on $g \circ f : \strat{X} \to \strat{Z}$ is
  the $(n + m)$-framing obtained as the composite of the individual framings
  \[
    \begin{tikzcd}[column sep = huge]
      \unstrat(\strat{X}) \ar[r, "\framing(f)", hook] &
      \R^n \times \unstrat(\strat{Y}) \ar[r, "\R^n \times \framing(g)", hook] &
      \R^n \times \R^m \times \unstrat(\strat{Z})
      =
      \R^{n + m} \times
      \unstrat(\strat{Z}) 
    \end{tikzcd}
  \]
\end{definition}

\begin{construction}
  Let $f : \strat{X} \to \strat{A}$ be an $n$-framed stratified bundle
  and $g : \strat{Y} \to \strat{B}$ an $m$-framed stratified bundle.
  Then the \defn{framed product} $f \frtimes g$ is the map of stratified spaces
  $\strat{X} \times \strat{Y} \to \strat{A} \times \strat{B}$
  equipped with the induced $(n + m)$-framing
  \[
    \begin{tikzcd}[column sep = huge]
      \unstrat(\strat{X} \times \strat{Y})
      \ar[r, hook, "\framing(f) \times \framing(g)"] &
      \R^n \times \unstrat(\strat{A}) \times \R^m \times \unstrat(\strat{B}) \cong
      \R^{n + m} \times \unstrat(\strat{A} \times \strat{B})
    \end{tikzcd}
  \]
  For each $n, m \geq 0$ the framed product defines a functor
  \[
    - \frtimes - : \FrBun^n \times \FrBun^m \longrightarrow \FrBun^{n + m}
  \]
\end{construction}

\begin{remark}
  Since the coordinate order in a framing is significant, the framed product
  is not symmetric up to framed isomorphism. We use the asymmetric symbol
  $\frtimes$ instead of $\times$ to indicate this in the notation.
\end{remark}

\begin{definition}\label{def:fct-framing-tame}
  An $n$-framed stratified bundle
  $f : \strat{E} \to \strat{B}$
  is \defn{tame} when $\strat{E}$ and $\R^n \times \strat{B}$ have 
  finite relative triangulations such that  
  the framing
  $\framing(f) : \unstrat(\strat{E}) \hookrightarrow \R^n \times \unstrat(\strat{B})$
  is a simplicial map.  
  An $n$-framed stratified bundle is \defn{essentially tame}
  when it is isomorphic in $\FrBun^n$ to a tame $n$-framed stratified bundle.
\end{definition}

\begin{example}\label{ex:fct-framing-non-tame}
  Our notion of tameness for bundles is stricter than that which would be imposed by
  o-minimality without additional conditions.
  Consider for example the stratified spaces $\strat{X}$ and $\strat{B}$,
  written as the union of their strata:
  \begin{align*}
    \unstrat(\strat{B}) :=&\ \{ (x, y) \mid 0 \leq x \leq y \leq 1 \} \subset \R^2 \\
    \unstrat(\strat{X}) :=&\ \{ ((1 - y) \tfrac{x}{y}, x, y) \mid 0 \leq x \leq y \leq 1, 0 < y \} \\
    \cup&\ \{ (t, 0, 0) \mid 0 < t < 1 \} \cup \{ (1, 0, 0) \} \cup \{ (0, 0, 0) \}
    \subset \R^3
  \end{align*}
  Then the projection onto the last two coordinates induces a $1$-framed stratified
  bundle $f : \strat{X} \to \strat{B}$ which is not essentially tame, although it is semi-algebraic.  
\end{example}

\begin{example}
  Consider the $3$-framed stratified spaces $\strat{X}, \strat{Y}, \strat{Z}$ which each consist of a single stratum
  given by the following subspaces of $\R^3$:
  \begin{align*}
    \strat{X} &:= \{ (0, 0, t) \mid t \in \R \} \subseteq \R^3, \\
    \strat{Y} &:= \{ (1, t \sin(\tfrac{1}{t}), t) \mid t \in \R \} \subseteq \R^3, \\
    \strat{Z} &:= \{ (1, 0, t) \mid t \in \R \} \subseteq \R^3.
  \end{align*}
  Then $\strat{X}$ and $\strat{Z}$ are tame but $\strat{Y}$ is not. However, as $3$-framed stratified
  spaces we have $\strat{Y} \cong \strat{Z}$ and so $\strat{Y}$ is essentially tame.
  The difference is important since the union $\strat{X} \cup \strat{Z}$ remains tame,
  while the union $\strat{X} \cup \strat{Y}$ is not even essentially tame.
\end{example}


\begin{definition}
  The \defn{frame order} on $\R^n$ is the smallest partial order $\frleq$
  such that for all points
  $x = (x_1, \ldots, x_n) \in \R^n$ and $y = (y_1, \ldots, y_n) \in \R^n$
  and $0 \leq i < n$ we have
  \[
    x \sim_i y \land x_{n - i} < y_{n - i}
    \quad
    \Rightarrow
    \quad
    x \frleq y
  \]
\end{definition}

\begin{observation}
  The frame order on $\R^n$ is the lexicographic order where elements of
  $\R^n$ are seen as sequences in $(\R, \leq)$ in reverse coordinate order.
\end{observation}

\begin{definition}
  Let $X \subseteq \R^n$ be a subset and $p \in \R^n$ a point.
  We say that $x \in X$ is the \defn{framed projection} of $p$ onto $X$
  when it is the unique point in $X$ that minimises the pointwise absolute value
  $|x - p|$
  in the frame order $\frleq$ on $\R^n$.
  In that case we write $\fproj(X, p) = x$.
\end{definition}



\section{Meshes}\label{sec:fct-mesh}

Meshes are framed stratified spaces in which all the information that is
detected by framed stratified isomorphisms is visible in the stratification.
Meshes organise into mesh bundles, and in the coordinate order prescribed by
the framing, an $n$-mesh bundle is built inductively as the framed composite
of $1$-mesh bundles that each add a single dimension.

We define two flavours of meshes:
Closed meshes are akin to commutative or pasting diagrams for higher categories,
and are glued together from compatible families of closed cells.
Open meshes are the Poincar\'e duals of closed meshes and model a form of higher-dimensional string diagrams.

We begin with closed meshes, which are technically simpler to manipulate than open ones. In~\S\ref{sec:fct-mesh-closed-1} we define closed $1$-mesh bundles and explore their properties. In particular, we see how closed $1$-mesh bundles are classified by a quasicategory $\BMeshClosed{1}$, equivalent to $\FinOrd^\op$.
Our results on closed $1$-mesh bundles then transfer to closed $n$-mesh bundles
in~\S\ref{sec:fct-mesh-closed-n}, which are classified by a quasicategory $\BMeshClosed{n}$.
We can equip closed $n$-mesh bundles with labels in any quasicategory $\cat{C}$,
leading to a quasicategory $\BMeshClosedL{n}{\cat{C}}$ in~\S\ref{sec:fct-mesh-closed-labelled}.
Via the packing construction of~\S\ref{sec:fct-mesh-closed-pack} we then see that $\BMeshClosedL{n + m}{\cat{C}} \simeq \BMeshClosedL{n}{\BMeshClosedL{m}{\cat{C}}}$.
We define open $n$-mesh bundles in~\S\ref{sec:fct-mesh-open-1} and transfer the relevant results from closed $n$-mesh bundles via a compactification construction. The result is a quasicategory $\BMeshOpenL{n}{\cat{C}}$ which classifies open $n$-mesh bundles with labels in $\cat{C}$.
In~\S\ref{sec:fct-mesh-grid} we explore grids, a special case of open and closed $n$-meshes, to show that the $n$-fold products $\FinOrd^{n, \op}$ and $\FinOrd^n$ are subcategories of $\BMeshClosed{n}$ and $\BMeshOpen{n}$, respectively.

\subsection{Closed $1$-Meshes}\label{sec:fct-mesh-closed-1}

\begin{definition}\label{def:fct-mesh-closed-1}
  A \defn{closed $1$-mesh} is a stratified space
  $\strat{X}$ such that the underlying space is a non-empty closed bounded interval $\unstrat(\strat{X}) \subset \R$
  and every stratum is either an open interval or an isolated point.
  The former strata are called \defn{regular},
  the latter \defn{singular}.
  A point $x \in \strat{X}$ is \defn{regular} or
  \defn{singular} if it is contained in a stratum that
  is regular or singular, respectively.
\end{definition}

We visualise a closed $1$-mesh by a line on which the singular strata are marked with a dot:
\[
  \begin{tikzpicture}[scale = 0.5, baseline=(current bounding box.center)]
    \node (sing) at (-4, -0.5) {singular};
    \node (reg) at (10, 0.5) {regular};
    \draw[mesh-stratum] (0, 0) -- (6, 0);
    \node[mesh-vertex] (s0) at (0, 0) {};
    \node[mesh-vertex] (s1) at (2, 0) {};
    \node[mesh-vertex] (s2) at (6, 0) {};
    \draw[->, gray] (sing) -- ($ (s0) + (0, -0.5) $) -- (s0.south);
    \draw[->, gray] (sing) -- ($ (s1) + (0, -0.5) $) -- (s1.south);
    \draw[->, gray] (sing) -- ($ (s2) + (0, -0.5) $) -- (s2.south);
    \draw[->, gray] (reg) -- ($ (1, 0.5) $) -- (1, 0);
    \draw[->, gray] (reg) -- ($ (4, 0.5) $) -- (4, 0);
  \end{tikzpicture}
\]
There is a closed $1$-mesh with any positive number of singular strata:
\[
  \begin{tikzpicture}[scale = 0.5, baseline=(current bounding box.center)]
    \node[mesh-vertex] (s0) at (0, 0) {};
  \end{tikzpicture}
  \quad
  \begin{tikzpicture}[scale = 0.5, baseline=(current bounding box.center)]
    \draw[mesh-stratum] (0, 0) -- (1, 0);
    \node[mesh-vertex] (s0) at (0, 0) {};
    \node[mesh-vertex] (s1) at (1, 0) {};
  \end{tikzpicture}
  \quad
  \begin{tikzpicture}[scale = 0.5, baseline=(current bounding box.center)]
    \draw[mesh-stratum] (0, 0) -- (2, 0);
    \node[mesh-vertex] (s0) at (0, 0) {};
    \node[mesh-vertex] (s1) at (1, 0) {};
    \node[mesh-vertex] (s2) at (2, 0) {};
  \end{tikzpicture}
  \quad
  \begin{tikzpicture}[scale = 0.5, baseline=(current bounding box.center)]
    \draw[mesh-stratum] (0, 0) -- (3, 0);
    \node[mesh-vertex] (s0) at (0, 0) {};
    \node[mesh-vertex] (s1) at (1, 0) {};
    \node[mesh-vertex] (s2) at (2, 0) {};
    \node[mesh-vertex] (s3) at (3, 0) {};
  \end{tikzpicture}
  \quad
  \text{etc.}
\]
Closed $1$-meshes organise into compatible families via closed $1$-mesh bundles.

\begin{definition}\label{def:fct-mesh-closed-1-bundle}
  A \defn{closed $1$-mesh bundle} is a $1$-framed
  stratified bundle $\xi : \strat{X} \to \strat{B}$
  that satisfies the following conditions:
  \begin{enumerate}
    \item $\xi$ is a stratified fibre bundle.
    \item For every $b \in \strat{B}$ the fibre $\xi^{-1}(b)$ is a closed $1$-mesh.
    \item The singular points in all fibres form a closed subset of $\strat{X}$.
    \item The framing map $\framing(\xi) : \unstrat(\strat{X}) \hookrightarrow \R^1 \times \unstrat(\strat{B})$ is a closed embedding.
  \end{enumerate}  
  We denote by $\sing(\strat{X})$ the subspace consisting of the singular points of $\strat{X}$.
\end{definition}

\begin{example}\label{ex:fct-mesh-closed-1-bordisms}
  The following are closed $1$-mesh bundles over $\DeltaStrat{1}$:
  \[
    \begin{tikzpicture}[yscale = 0.5, xscale = 0.25, baseline=(current bounding box.center)]
      \fill[mesh-background] (3, 0) rectangle (6, 2);
      \draw[mesh-stratum] (0, 0) -- (6, 0);
      \draw[mesh-stratum] (3, 0) -- (3, 2);
      \draw[mesh-stratum] (6, 0) -- (6, 2);
      \node[mesh-vertex] at (0, 0) {};
      \node[mesh-vertex] at (3, 0) {};
      \node[mesh-vertex] at (6, 0) {};
    \end{tikzpicture}
    \quad \xrightarrow{\xi_1} \quad
    \begin{tikzpicture}[scale = 0.5, baseline=(current bounding box.center)]
      \draw[mesh-stratum] (0, 0) -- (0, 2);
      \node[mesh-vertex] at (0, 0) {};
    \end{tikzpicture}
    \qquad\qquad
    \begin{tikzpicture}[scale = 0.5, xscale = 0.25, baseline=(current bounding box.center)]
      \fill[mesh-background] (0, 2) -- (6, 2) -- (3, 0) -- cycle;
      \draw[mesh-stratum] (3, 0) -- (0, 2);
      \draw[mesh-stratum] (3, 0) -- (6, 2);
      \node[mesh-vertex] at (3, 0) {};
    \end{tikzpicture}
    \quad \xrightarrow{\xi_2} \quad
    \begin{tikzpicture}[scale = 0.5, baseline=(current bounding box.center)]
      \draw[mesh-stratum] (0, 0) -- (0, 2);
      \node[mesh-vertex] at (0, 0) {};
    \end{tikzpicture}
    \qquad\qquad
    \begin{tikzpicture}[scale = 0.5, xscale = 0.25, baseline=(current bounding box.center)]
      \fill[mesh-background] (0, 0) rectangle (6, 2);
      \draw[mesh-stratum] (0, 0) -- (6, 0);
      \draw[mesh-stratum] (0, 0) -- (0, 2);
      \draw[mesh-stratum] (6, 0) -- (6, 2);
      \node[mesh-vertex] at (0, 0) {};
      \node[mesh-vertex] at (3, 0) {};
      \node[mesh-vertex] at (6, 0) {};
    \end{tikzpicture}
    \quad \xrightarrow{\xi_3} \quad
    \begin{tikzpicture}[scale = 0.5, baseline=(current bounding box.center)]
      \draw[mesh-stratum] (0, 0) -- (0, 2);
      \node[mesh-vertex] at (0, 0) {};
    \end{tikzpicture}
  \]
\end{example}

\begin{example}
  A closed $1$-mesh bundle over a stratification of the interval
  resembles a $2$-dimensional pasting diagram:
  \[
    \begin{tikzpicture}[scale = 0.5, baseline=(current bounding box.center)]
      \fill[mesh-background] (0, 0) rectangle (6, 2);
      \fill[mesh-background] (2, 2) rectangle (6, 6);
      \draw[mesh-stratum] (0, 0) -- (4, 0) -- (4, 2) -- (0, 2) -- cycle;
      \draw[mesh-stratum] (4, 0) -- (6, 0) -- (6, 2) -- (4, 2) -- cycle;
      \draw[mesh-stratum] (2, 4) -- (2, 2) -- (6, 2) -- (6, 4) -- cycle;
      \draw[mesh-stratum] (2, 4) -- (2, 6) -- (6, 6) -- (4, 4) -- cycle;
      \draw[mesh-stratum] (4, 4) -- (6, 4) -- (6, 6) -- cycle;
      \node[mesh-vertex] at (0, 0) {};
      \node[mesh-vertex] at (4, 0) {};
      \node[mesh-vertex] at (6, 0) {};
      \node[mesh-vertex] at (0, 2) {};
      \node[mesh-vertex] at (2, 2) {};
      \node[mesh-vertex] at (4, 2) {};
      \node[mesh-vertex] at (6, 2) {};
      \node[mesh-vertex] at (2, 4) {};
      \node[mesh-vertex] at (4, 4) {};
      \node[mesh-vertex] at (6, 4) {};
      \node[mesh-vertex] at (2, 6) {};
      \node[mesh-vertex] at (6, 6) {};
    \end{tikzpicture}
    \quad
    \longrightarrow
    \quad
    \begin{tikzpicture}[scale = 0.5, baseline=(current bounding box.center)]
      \draw[mesh-stratum] (0, 0) -- (0, 6);
      \node[mesh-vertex] at (0, 0) {};
      \node[mesh-vertex] at (0, 2) {};
      \node[mesh-vertex] at (0, 4) {};
      \node[mesh-vertex] at (0, 6) {};
    \end{tikzpicture}
  \]
\end{example}

\begin{example}
  The $1$-framed stratified bundle $f: \strat{E} \to \strat{B}$ from Example~\ref{ex:fct-framing-non-tame}
  is a stratified fibre bundle and each fibre is a closed $n$-mesh.
  However the subset of singular points is not closed in $\strat{E}$
  and therefore $f$ is \textbf{not} a closed $1$-mesh bundle.
  Meshes are in many respects similar to cell decompositions in o-minimal geometry, but impose more stringent requirements on how the cells
  fit together: the strata of $\strat{E}$ do form a semi-algebraic
  cell decomposition.
\end{example}

\begin{observation}
  Closed $1$-mesh bundles are closed under isomorphism of $1$-framed stratified bundles.
\end{observation}

\begin{lemma}\label{lem:fct-mesh-closed-1-pl}
  Let $\xi : \strat{X} \to \strat{B}$ be a $1$-framed stratified bundle
  that satisfies:
  \begin{enumerate}[label=(\arabic*)]
    \item $\xi$ is a stratified fibre bundle.
    \item For every $b \in \strat{B}$ the fibre $\xi^{-1}(b)$ is a closed $1$-mesh.
    \item $\xi$ is essentially tame.
    \item $\strat{X}$ satisfies the frontier condition.
    \item The framing map $\framing(\xi) : \unstrat(\strat{X}) \hookrightarrow \R^1 \times \unstrat(\strat{B})$ is a closed embedding.
  \end{enumerate}
  Then $\xi$ is a closed $1$-mesh bundle.
\end{lemma}
\begin{proof}
  Up to isomorphism of $1$-framed stratified bundles we can assume that
  $\xi$ is tame. Because $\strat{B}$ has a locally finite triangulation,
  we can consider $\strat{B}$ as a stratified subset of some $\R^m$.
  Moreover, we can assume that the underlying map $\unstrat(\xi)$ is the
  restriction of the projection map $\R \times \R^m \to \R^m$ and therefore
  linear.
  We pick a linear triangulation of $\strat{X}$.

  Suppose for the purpose of contradiction that $\sing(\strat{X})$ is not closed.
  Then the must be a regular stratum $r \in \stratPos{\strat{X}}$ and a singular
  stratum $s \in \stratPos{\strat{X}}$ such that $\strat{X}_r$ intersects
  the closure of $\strat{X}_s$.
  By the frontier condition this implies that $\strat{X}_r$ is contained
  within the boundary of $\strat{X}_s$.  
  Because regular strata of a closed $1$-mesh have more than one point,
  there must exist points
  $x_0, x_1 \in \strat{X}_r$ and $x_2 \in \strat{X}_s$ 
  such that $\xi(x_0) = \xi(x_1)$,
  $x_0 \neq x_1$ 
  and $\sigma(t_0, t_1, t_2) = t_0 x_1 + t_1 x_1 + t_2 x_2$  
  defines a stratified $2$-simplex $\sigma : \DeltaStrat{2} \to \strat{X}$
  in the chosen linear triangulation of $\strat{X}$.
  We let $\gamma_{02}, \gamma_{12} : \DeltaStrat{1} \to \strat{X}$
  denote the exit paths obtained from $\sigma$ via
  $\gamma_{02}(t) = \sigma(1 - t, 0, t)$ and $\gamma_{12}(t) = \sigma(0, 1 - t, t)$.
  By linearity of the map $\xi$ we have
  \[
    (\xi \circ \gamma_{02})(t) =
    (1 - t) \xi(x_0) + t \xi(x_2) =
    (1 - t) \xi(x_1) + t \xi(x_2) = 
    (\xi \circ \gamma_{12})(t).
  \]
  Since the stratum containing $\gamma_{02}(1) = \gamma_{12}(1)$ is singular,
  the assumptions (1) and (2) together imply that $\gamma_{02}(t) = \gamma_{12}(t)$
  for all $0 < t \leq 1$.
  But then by continuity we must have $x_0 = \gamma_{02}(0) = \gamma_{12}(0) = x_1$,
  contradicting our choice that $x_0 \neq x_1$.
\end{proof}

\begin{lemma}\label{lem:fct-mesh-closed-1-pullback}
  Closed $1$-mesh bundles are stable under base change.
\end{lemma}
\begin{proof}
  Let $\xi : \strat{X} \to \strat{B}$ be a closed $1$-mesh bundle and
  $\varphi : \strat{A} \to \strat{B}$ a map of stratified spaces.
  We then take the pullback of $\xi$ along $\varphi$ in the category of
  stratified spaces:
  \[
    \begin{tikzcd}
      {\strat{X} \times_{\strat{A}} \strat{B}} \ar[r, "\xi^* \varphi"] \ar[d, "\varphi^* \xi"'] \pullbackcorner &
      {\strat{X}} \ar[d, "\xi"] \\
      {\strat{A}} \ar[r, "\varphi"'] &
      {\strat{B}}
    \end{tikzcd}
  \]
  The map $\varphi^* \xi$ is equipped with the induced $1$-framing.
  Since stratified fibre bundles are closed under pullback,
  the map $\varphi^* \xi$ is again a stratified fibre bundle.
  For every point $a \in \strat{A}$, the fibre of $\varphi^* \xi$ over $a$ agrees
  with the fibre of $\xi$ over $\varphi(a)$ by the pullback pasting lemma:
  \[
    \begin{tikzcd}
      {\xi^{-1}(\varphi(a))} \ar[r] \ar[d] \pullbackcorner &
      {\strat{X} \times_{\strat{A}} \strat{B}} \ar[r, "\xi^* \varphi"] \ar[d, "\varphi^* \xi"'] \pullbackcorner &
      {\strat{X}} \ar[d, "\xi"] \\
      {\terminal} \ar[r, "a"'] &
      {\strat{A}} \ar[r, "\varphi"'] &
      {\strat{B}}
    \end{tikzcd}
  \]
  Therefore $(\varphi^* \xi)^{-1}(a) \cong \xi^*(\varphi(a))$ is a closed $1$-mesh
  and a point $x \in \strat{X} \times_{\strat{A}} \strat{B}$ is singular if and only
  if $(\xi^* \varphi)(x) \in \strat{X}$ is singular.
  Hence, the subspace of singular points in $\strat{X} \times_{\strat{A}} \strat{B}$
  is the preimage of the subspace of singular points in $\strat{X}$
  via the continuous map $\xi^*\varphi$ and therefore closed.
\end{proof}

When $\xi : \strat{X} \to \strat{B}$ is a closed $1$-mesh bundle and
$\gamma : \DeltaTop{1} \to \strat{B}$ is a stratum preserving path,
the fibres $\xi^{-1}(\gamma(t))$ for all $0 \leq t \leq 1$ have the same number of
singular strata whose positions in $\R$ vary continuously.
Since the line $\R$ allows for no space for the singular strata to change position,
the restriction of $\xi$ over any stratum of $\strat{B}$ is determined uniquely
by the number of singular strata.
When $\gamma : \DeltaStrat{1} \to \strat{B}$ is an exit path between different
strata of $\strat{B}$, the singular strata of the closed $1$-meshes
$\gamma^{-1}(0)$ and $\gamma^{-1}(1)$ are related as follows.

\begin{lemma}\label{lem:fct-mesh-closed-1-sing-covering}
  Let $\xi : \strat{X} \to \strat{B}$ be a closed $1$-mesh bundle.
  Then the composite map
  \[
    \begin{tikzcd}[column sep = large]
      {\Exit(\sing(\strat{X}))} \ar[r, hook] &
      {\Exit(\strat{X})} \ar[r, "\Exit(\xi)"] &
      {\Exit(\strat{B})}
    \end{tikzcd}
  \]
  is a right covering map.
\end{lemma}
\begin{proof}

  We need to construct lifts for right horns
  \[
    \begin{tikzcd}
      {\Lambda^i\ord{k}} \ar[r] \ar[d] &
      {\Exit(\sing(\strat{X}))} \ar[d] \\
      {\Delta\ord{k}} \ar[r] \ar[ur, dashed] &
      {\Exit(\strat{B})}
    \end{tikzcd}
  \]
  with $0 < i \leq k$ and show that they are unique.
  By adjunction this is equivalent to constructing unique lifts of the form
  \[
    \begin{tikzcd}
      {\HornStrat{i}{k}} \ar[r] \ar[d] &
      {\strat{X}} \ar[d, "\xi"] \\
      {\DeltaStrat{k}} \ar[r, swap] \ar[ur, dashed] &
      {\strat{B}}
    \end{tikzcd}
  \]
  where the image of $\tau$ lies within $\sing(\strat{X})$.  
  Via base change along $\DeltaStrat{k} \to \strat{B}$ we can reduce to the case that
  $\strat{B} = \DeltaStrat{k}$:
  \[
    \begin{tikzcd}
      {\HornStrat{i}{k}} \ar[r, "\tau"] \ar[d] &
      {\strat{X}} \ar[d, "\xi"] \\
      {\DeltaStrat{k}} \ar[r, "\id", swap] \ar[ur, dashed, "\hat{\tau}"{description}] &
      {\DeltaStrat{k}}
    \end{tikzcd}
  \]
  We let 
  $\tau_k : \HornStrat{i}{k}_k \to \strat{X}$ be the restriction to the $k$th
  stratum.
  Using the trivialisation of $\xi$ over the $k$th stratum $\DeltaStrat{k}_k$ there must be a unique
  lift in the diagram
  \[
    \begin{tikzcd}
      {\HornStrat{i}{k}_k} \ar[r, "\tau_k"] \ar[d] &
      {\strat{X}} \ar[d, "\xi"] \\
      {\DeltaStrat{k}_k} \ar[r, hook] \ar[ur, "\hat{\tau}_k"{description}, hook] &
      {\DeltaStrat{k}}
    \end{tikzcd}
  \]
  For every $b \in \DeltaStrat{k}$ there must be at least one point
  $x \in \cl(\im(\hat{\tau}_k))$
  such that $\xi(x) = b$.
  Suppose that $x_1, x_2 \in \cl(\im(\hat{\tau}_k))$ with $\xi(x_1) = \xi(x_2) = b$
  then every point $x \in \xi^{-1}(b)$ with $x_1 \leq x \leq x_2$ in the fibrewise ordering must also be contained in $\cl(\im(\hat{\tau}_k))$.
  If $x_1 \neq x_2$, a regular stratum would intersect the closure of a singular stratum, which would contradict the assumption that the singular strata form a closed subspace of $\strat{X}$.
  Therefore, $\hat{\tau}_k$ uniquely extends to a lift $\hat{\tau}$ as required.
\end{proof}

\begin{construction}\label{con:fct-mesh-closed-1-sing}
  Let $\xi : \strat{E} \to \strat{B}$ be a closed $1$-mesh bundle.
  Lemma~\ref{lem:fct-mesh-closed-1-sing-covering} implies that there is a pullback square of simplicial sets
  \[
    \begin{tikzcd}
      {\Exit(\sing(\strat{E}))} \ar[r] \ar[d] \pullbackcorner &
      {\FinSet_*^\op} \ar[d] \\
      {\Exit(\strat{B})} \ar[r] &
      {\FinSet^\op}
    \end{tikzcd}
  \]
  where $\FinSet_* \to \FinSet$ is the functor from pointed finite sets to finite sets
  which forgets the distinguished point.
  The functor $\Exit(\strat{B}) \to \FinSet^\op$ sends a point $b \in \strat{B}$ in the base space to the set of singular points in the fibre $\xi^{-1}(b)$.  
  The fibres of $\xi$ are each embedded into $\R$ and thus inherit a canonical order.
  Because singular strata can not intersect each other
  and each fibre has at least one singular stratum,
  we obtain a pullback square
  \[
    \begin{tikzcd}
      {\Exit(\sing(\strat{E}))} \ar[r] \ar[d] \pullbackcorner &
      {\FinOrd_*^\op} \ar[d] \\
      {\Exit(\strat{B})} \ar[r] &
      {\FinOrd^\op}
    \end{tikzcd}
  \]  
  We denote the functor $\Exit(\strat{B}) \to \FinOrd^\op$ by $\sing(\xi)$.
\end{construction}

\begin{definition}\label{def:fct-mesh-closed-1-cat}
  We denote by $\BMeshClosed{1}$ the simplicial set whose $k$-simplices
  consist of a closed $1$-mesh bundle $\xi : \strat{X} \to \DeltaStrat{k}$.
  An order-preserving map $\ord{k'} \to \ord{k}$ acts by pullback along the
  induced map $\DeltaStrat{k'} \to \DeltaStrat{k}$.
  There is a map of simplicial sets $\sing : \BMeshClosed{1} \to \FinOrd$
  which sends a $k$-simplex given by $\xi : \strat{X} \to \DeltaStrat{k}$
  to the composite
  \[
    \begin{tikzcd}
      \Delta\ord{k} \ar[r, hook] &
      \Exit(\DeltaStrat{k}) \ar[r, "\sing(\xi)"] &
      \FinOrd^\op 
    \end{tikzcd}
  \]
\end{definition}

\begin{example}
  A $1$-simplex of $\BMeshClosed{1}$ is a closed $1$-mesh bundle $\xi$ over the
  stratified $1$-simplex $\DeltaStrat{1}$.
  The map of simplicial sets $\sing : \BMeshClosed{1} \to \FinOrd^\op$ then sends
  $\xi$ to a $1$-simplex of $\FinOrd^\op$, i.e.\ an order preserving map between
  finite non-empty ordinals.
  For the closed $1$-mesh bundles $\xi_1$, $\xi_2$ and $\xi_3$ from Example~\ref{ex:fct-mesh-closed-1-bordisms} this order-preserving map turns out as follows:
  \begin{align*}
    \sing(\xi_1) &= \langle 1, 2 \rangle : \ord{1} \to \ord{2} \\
    \sing(\xi_2) &= \langle 0, 0 \rangle : \ord{1} \to \ord{0} \\
    \sing(\xi_3) &= \langle 0, 2 \rangle : \ord{1} \to \ord{2}
  \end{align*}
\end{example}

\begin{proposition}\label{prop:fct-mesh-closed-1-sing-trivial}
  The map of simplicial sets
  $\sing : \BMeshClosed{1} \to \FinOrd^\op$ is a trivial fibration.
  In particular $\BMeshClosed{1}$ is a quasicategory so that
  $\BMeshClosed{1} \simeq \FinOrd^\op$.
\end{proposition}
\begin{proof}
  Suppose we have a lifting problem
	\begin{equation}\label{eq:fct-mesh-closed-1-sing-trivial}
	  \begin{tikzcd}
	    {\partial \Delta\ord{k}} \ar[r] \ar[d] &
	    {\BMeshClosed{1}} \ar[d, "\sing"] \\
	    {\Delta\ord{k}} \ar[r] \ar[ur, dashed] &
	    {\FinOrd^\op}
	  \end{tikzcd}
	\end{equation}
	When $k = 0$ the map $\Delta\ord{0} \to \FinOrd^\op$ is a finite non-empty ordinal $\ord{m}$.
	Then we can construct a closed $1$-mesh bundle
	$\xi$ over $\DeltaStrat{0}$ with $\sing(\xi) = \ord{m}$ by arranging
	$m + 1$ singular strata in a line, providing the lift for (\ref{eq:fct-mesh-closed-1-sing-trivial}).
	When $k > 0$ we are given a $k$-simplex $\sigma : \Delta\ord{k} \to \FinOrd^\op$
  and the map $\partial\Delta\ord{k} \to \BMeshClosed{1}$ corresponds to
	a closed $1$-mesh bundle $\xi' : \strat{X}' \to \partial \DeltaStrat{k}$.
  Moreover, the diagram
  \[
    \begin{tikzcd}
      {\Exit(\partial \DeltaStrat{k})} \ar[dr, "\sing(\xi')", swap] &
      {\partial \Delta\ord{k}} \ar[l, hook'] \ar[r, hook] \ar[d] &
      {\Delta\ord{k}} \ar[dl, "\sigma"] \\
      {} &
      {\FinOrd^\op}
    \end{tikzcd}
  \]
  commutes.  
	To extend $\xi'$ to a closed $1$-mesh bundle $\xi : \strat{X} \to \DeltaStrat{k}$
	we specify the heights of the singular strata over the points of $\DeltaStrat{k}$
  compatible with those in $\xi'$.

  Let $x \in \strat{X}'$ be a singular element with $\xi'(x) = k \in \DeltaStrat{k}$ and let $i \in \sing(\xi')(k)$ be its index.  
  The forgetful map $\FinOrd_*^\op \to \FinOrd^\op$ is a right covering map,
  and so we get a unique lift  
	\[
	  \begin{tikzcd}
	    {\{ k \}} \ar[r, "i"] \ar[d, hook] &
	    {\FinOrd_*^\op} \ar[d] \\
	    {\Delta \ord{k}} \ar[r, "\sigma"'] \ar[ur, dashed, "\hat{\sigma}"{description}] &
	    {\FinOrd^\op}
	  \end{tikzcd}
	\]
  Denote by $\partial \hat{\sigma}$ the restriction of $\hat{\sigma}$ to the boundary $\partial\Delta\ord{k}$.
  Then the universal property of the pullback determines a unique map
	\[
	  \begin{tikzcd}[column sep = large]
	    {\partial\Delta\ord{k}} \ar[r, dashed] \ar[d] \ar[rr, bend left, "\partial\hat{\sigma}"'] &
			{\Exit(\sing(\strat{X}))} \ar[r] \ar[d] \pullbackcorner &
	    {\FinOrd_*^\op} \ar[d] \\
			{\partial\Delta\ord{k}} \ar[r, hook] \ar[rr, bend right, "\sing(\xi')"] &
	    {\Exit(\partial \DeltaStrat{k})} \ar[r] &
	    {\FinOrd^\op}
	  \end{tikzcd}
	\]
  and therefore a section $s' : \partial\DeltaStrat{k} \to \strat{X}'$ of $\xi'$ so that the image of $s'$ consists of singular points and $s'(k) = x$.
  By interpolating the heights in the framing over the interior of $\DeltaStrat{k}$, this provides the height of the $i$th singular stratum.
\end{proof}


\begin{remark}
  The simplicial set $\BMeshClosed{1}$ is not the nerve of a $1$-category,
  but by Proposition~\ref{prop:fct-mesh-closed-1-sing-trivial} it is equivalent
  as an $\infty$-category to the $1$-category $\FinOrd^\op$.
  This raises the question if using quasicategories and $\infty$-categories
  at this point is not needlessly extravagant;
  algebraically we could just as well quotient closed $1$-mesh bundles by isotopy
  to obtain an ordinary $1$-category.
  We avoid the quotient to emphasise that the map $\sing : \BMeshClosed{1} \to \FinOrd^\op$
  bridges between geometry and combinatorics, as a $k$-simplex of $\BMeshClosed{1}$ is a concrete geometrical object.
  Over the course of this chapter we will extend this bridge to combinatorially capture higher-dimensional geometry.
  For our theory of manifold diagrams in~\S\ref{sec:d} we will then use meshes
  to explore geometries that are not homotopically trivial.
  We found the framework of quasicategories to be the most convenient to formally
  capture these ideas.
\end{remark}

The simplices of the simplicial set $\BMeshClosed{1}$ are closed $1$-mesh bundles
over the stratified simplices.
In general, we care about closed $1$-mesh bundles over more general stratified base spaces.
Via the following construction, each closed $1$-mesh bundle over an arbitrary stratified space
$\strat{B}$ induces a map from $\Exit(\strat{B})$ to $\BMeshClosed{1}$.

\begin{construction}\label{con:fct-mesh-closed-1-classify-map}
  Suppose that $\strat{B}$ is a stratified space and
  $\xi : \strat{X} \to \strat{B}$ is a closed $1$-mesh bundle.
  Then for any stratified $k$-simplex $\sigma : \DeltaStrat{k} \to \strat{B}$
  we can take the pullback of $\xi$ along $\sigma$ to obtain a closed $1$-mesh bundle $\xi_\sigma$:
  \[
    \begin{tikzcd}
      {\strat{X}_\sigma} \ar[r] \ar[d, "\xi_\sigma", swap] \pullbackcorner &
      {\strat{X}} \ar[d, "\xi"] \\
      {\DeltaStrat{k}} \ar[r, "\sigma", swap] &
      {\strat{B}}
    \end{tikzcd}
  \]
  The closed $1$-mesh bundle $\xi_\sigma$ then defines a $k$-simplex of $\BMeshClosed{1}$.
  By sending any stratified simplex $\sigma$ of $\strat{B}$ to $\xi_\sigma$ 
  the closed $1$-mesh bundle $f$ induces a \defn{classifying map}
  \[ \classify{f} : \Exit(\strat{B}) \longrightarrow \BMeshClosed{1}. \]
\end{construction}

In general not every map $\varphi : \Exit(\strat{B}) \to \BMeshClosed{1}$
arises as the classifying map of a closed $1$-mesh bundle over the stratified
space $\strat{B}$, even when $\strat{B}$ is a stratified $k$-simplex $\DeltaStrat{k}$
with $k > 0$.
However, as long as $\strat{B}$ is triangulable, we can ensure that $\varphi$
is equivalent to a classifying map
in the space of functors $\CatInfty(\Exit(\strat{B}), \BMeshClosed{1})$.
We recall from~\S\ref{sec:b-geo-strat} that a stratified space $\strat{B}$ is triangulable
if there exists a simplicial complex $(V, S)$ and a refinement map
$\StratReal{(V, S)} \to \strat{B}$.
To make our arguments easier we will use the alternative characterisation:
a stratified space $\strat{B}$ is triangulable when there
exists a marked poset $S$ so that $\StratRealMark{S} \cong \strat{B}$.



\begin{lemma}\label{lem:fct-mesh-closed-1-classify-marked}
  Let $S$ be a marked poset and $\varphi : S \to (\BMeshClosed{1})^\natural$
  a map of marked simplicial sets.
  Then there exists a unique closed $1$-mesh bundle
  $\xi$ over $\StratRealMark{S}$ 
  such that the following diagram commutes:
  \[
    \begin{tikzcd}
      {S} \ar[r] \ar[d, hook] &
      {(\BMeshClosed{1})^\natural} \\
      {\ExitMark(\StratRealMark{S})} \ar[ur, dashed, "\classify{\xi}"'] &
      {}
    \end{tikzcd}
  \]
\end{lemma}
\begin{proof}
  We first demonstrate this in the case that $S = (\Delta\ord{k}, T)$.  
  A map $\varphi : S \to (\BMeshClosed{1})^\natural$ then induces a
  $k$-simplex of $\BMeshClosed{1}$ and therefore a closed $1$-mesh bundle
  $\xi' : \strat{X}' \to \DeltaStrat{k}$.
  The map $\varphi$ sends marked edges of $S$ to equivalences in $\BMeshClosed{1}$,
  which preserve the number of singular strata.
  By Lemma~\ref{lem:fct-mesh-closed-1-sing-covering} when the number of singular
  strata is preserved, the heights of those singular strata in the fibre vary
  continuously.
  Therefore, $\xi'$ is the refinement of a closed $1$-mesh bundle
  $\xi : \strat{X} \to \StratRealMark{S}$ so that
  the classifying map $\classify{\xi}$ restricts to $\varphi$.

  A marked poset $S$ is the colimit in $\sSetM$ of its marked flags, i.e.\
  the simplices $\sigma : (\Delta\ord{k}, T) \hookrightarrow S$ arising
  from strictly increasing sequences in $S$ with the induced marking.
  For such a marked flag, the composite map
  $\varphi \circ \sigma : (\Delta\ord{k}, T) \to (\BMeshClosed{1})^\natural$
  corresponds by the first part of the proof to a closed $1$-mesh bundle  
  $\xi_\sigma : \strat{X}_\sigma \to \StratRealMark{(\Delta\ord{k}, T)}$.
  The marked stratified realisation functor $\StratRealMark{-}$ is a left adjoint
  and therefore preserves colimits.
  We glue together the closed $1$-mesh bundles $\xi_\sigma$
  for each proper flag $\sigma$ of $S$ to obtain a $1$-framed stratified bundle
  $\xi : \strat{X} \to \StratRealMark{S}$.

  We claim that $\xi$ is a closed $1$-mesh bundle.
  For each $b \in \StratRealMark{S}$ the fibre $\xi^{-1}(b)$ is the fibre
  of the closed $1$-mesh bundle $\xi_\sigma^{-1}(b)$ for some marked flag $\sigma$.
  The colimit that defines $\strat{X}$ is equivalently a quotient of the disjoint union
  $
    \coprod_{\sigma} \strat{X}_\sigma
  $
  ranging over all marked flags. 
  The preimage via the quotient map
  $\coprod_{\sigma} \strat{X}_\sigma \to \strat{X}$
  of $\sing(\strat{X})$ is the disjoint union
  $\coprod_{\sigma} \sing(\strat{X}_\sigma)$
  and therefore closed.
  Therefore, $\sing(\strat{X})$ is closed in $\strat{X}$ with the quotient topology.
\end{proof}

\begin{proposition}\label{prop:fct-mesh-closed-1-classify}
  Let $\strat{B}$ be a triangulable stratified space and
  $I : \strat{A} \hookrightarrow \strat{B}$
  a closed triangulable stratified subspace.
  Suppose that $\varphi : \Exit(\strat{B}) \to \BMeshClosed{1}$ is a map
  and $\xi'$ a closed $1$-mesh bundle over $\strat{A}$ such that
  $\classify{\xi'} = \varphi \circ I$.
  Then there exists a closed $1$-mesh bundle $\xi : \strat{X} \to \strat{B}$
  together with an equivalence
  $\varphi \simeq \classify{\xi}$ 
  relative to $\classify{\xi'}$.
\end{proposition}
\begin{proof}
  It suffices to address the case where $I = \StratRealMark{i}$ for an
  inclusion $i : R \hookrightarrow S$ of marked posets.  
  By Lemma~\ref{lem:fct-mesh-closed-1-classify-marked}
  there exists a unique closed $1$-mesh bundle $\xi : \strat{X} \to \StratRealMark{S}$
  such that $\classifyMark{\xi}$ fits into the diagram
  \[
    \begin{tikzcd}
    	{\ExitMark(\StratRealMark{R})} &&& {\ExitMark(\StratRealMark{S})} \\
    	& R & S \\
    	{(\BMeshClosed{1})^\natural} &&& {(\BMeshClosed{1})^\natural}
    	\arrow["i"', hook, from=2-2, to=2-3]
    	\arrow["{\eta_R}"{description}, hook', from=2-2, to=1-1]
    	\arrow[from=2-2, to=3-1]
    	\arrow["{\eta_S}"{description}, hook, from=2-3, to=1-4]
    	\arrow["{\varphi \circ \eta_S}"{description}, from=2-3, to=3-4]
    	\arrow["{\ExitMark(I)}", hook, from=1-1, to=1-4]
    	\arrow["{\classifyMark{\xi'}}"', from=1-1, to=3-1]
    	\arrow["{\classifyMark{\xi}}", dashed, from=1-4, to=3-4]
    	\arrow[equal, from=3-1, to=3-4]
    \end{tikzcd}
  \]  
  Since both $\classify{\xi}$ and $\varphi$ are solutions to the lifting problem
  \[
    \begin{tikzcd}
      {\ExitMark(\StratRealMark{R}) \sqcup_R S} \ar[d, hook, "\simeq"'] \ar[r] &
      {(\BMeshClosed{1})^\natural} \\
      {\ExitMark(\StratRealMark{S})} \ar[ur, dashed]
    \end{tikzcd}
  \]
  there exists an equivalence $\classify{\xi} \simeq \varphi$ 
  relative to $\classify{\xi'}$.
\end{proof}

\begin{proposition}\label{prop:fct-mesh-closed-1-pl}
  Let $\strat{B}$ be a triangulable stratified space and
  $I : \strat{A} \hookrightarrow \strat{B}$
  a closed triangulable stratified subspace.
  Suppose that $\xi : \strat{X} \to \strat{B}$ is a closed $1$-mesh bundle
  such that $I^*\xi$ is piecewise linear.
  Then there exists a piecewise linear closed $1$-mesh bundle $\xi'$ over $\strat{B}$, together with an equivalence $\classify{\xi} \simeq \classify{\xi'}$ that restricts to the identity on $\classify{I^*\xi}$.
\end{proposition}
\begin{proof}
  We again restrict to the case that $I = \StratRealMark{i}$ for an inclusion
  $i : R \hookrightarrow S$ of marked posets.
  We then have a diagram of marked simplicial sets
  \[
    \begin{tikzcd}[column sep = large, row sep = large]
      {R} \ar[hook, r, "\eta_R"] \ar[d, "i"', hook] \pullbackcorner &
      {\ExitMark(\StratRealMark{R})}
      \ar[d, "\ExitMark(I)"{description}, hook] \ar[r, "\classify{I^*\xi }"] &
      {(\BMeshClosed{1})^\natural} \ar[d, equal] \\
      {S} \ar[r, "\eta_S"', hook] &
      {\ExitMark(\StratRealMark{S})} \ar[r, "\classify{\xi}"'] &
      {(\BMeshClosed{1})^\natural}
    \end{tikzcd}
  \]
  By the assumption that the restriction $I^* \xi$ is piecewise linear,
  the composite $\classify{I^* \xi} \circ \eta_R : R \to (\BMeshClosed{1})^\natural$
  lands within the simplicial subset of $\BMeshClosed{1}$ containing only piecewise linear
  closed $1$-mesh bundles.
  The map $\classify{\xi} \circ \eta_S$ is a solution to the lifting problem
  \[
    \begin{tikzcd}[column sep = huge]
      {R} \ar[r, "\classify{I^* \xi} \circ \eta_R"] \ar[d, hook, "i"'] &
      {(\BMeshClosed{1})^\natural} \ar[d, "\sing^\natural"] \\
      {S} \ar[r, "\sing \circ \classify{\xi} \circ \eta_S"'] \ar[ur, dashed] &
      {(\FinOrd^\op)^\natural}
    \end{tikzcd}
  \]
  We can construct another solution $\varphi : S \to (\BMeshClosed{1})^\natural$
  to this lifting problem:
  Using the technique used in the proof of Proposition~\ref{prop:fct-mesh-closed-1-sing-trivial},
  the map $R \to (\BMeshClosed{1})^\natural$ can be extended simplex by simplex
  to a map $S \to (\BMeshClosed{1})^\natural$ using linear interpolation of the
  heights of the singular strata.
  Since the space of lifts is contractible, we have an equivalence of lifts
  $\classify{\xi} \circ \eta_S \simeq \varphi$.
  Now using Lemma~\ref{lem:fct-mesh-closed-1-classify-marked}
  we obtain a closed $1$-mesh bundle $\xi'$ and an equivalence $\classify{\xi} \simeq \classify{\xi'}$
  that restricts to the identity on $\classify{I^* \xi}$.
  The closed $1$-mesh bundle $\xi'$ is piecewise linear through the way we constructed
  $\varphi$ via linear interpolation.
\end{proof}

\begin{cor}\label{cor:fct-mesh-closed-1-total-triangulable}
  Let $\xi : \strat{X} \to \strat{B}$ be a closed $1$-mesh bundle so that
  $\strat{B}$ is triangulable. Then $\strat{X}$ is triangulable as well.
\end{cor}

\begin{lemma}\label{lem:fct-mesh-closed-1-bundle-cat-fib}
  Let $\xi : \strat{X} \to \strat{B}$ be a closed $1$-mesh bundle.
  Then the induced map
  \[
    \Exit(\xi) : \Exit(\strat{X}) \longrightarrow \Exit(\strat{B})
  \]
  is a categorical fibration.
  In particular when $\strat{B}$ is a fibrant stratified space then so is $\strat{X}$.
\end{lemma}
\begin{proof}
  To show that the map $\Exit(\xi)$ is an inner fibration,
  we need to construct lifts for inner horns
  \[
    \begin{tikzcd}
      {\Lambda^i\ord{k}} \ar[r] \ar[d] &
      {\Exit(\strat{X})} \ar[d, "\Exit(\xi)"] \\
      {\Delta\ord{k}} \ar[r] \ar[ur, dashed] &
      {\Exit(\strat{B})}
    \end{tikzcd}
  \]
  with $0 < i < k$.
  By adjunction this is equivalent to solving lifting problems of the form:
  \[
    \begin{tikzcd}
      {\HornStrat{i}{k}} \ar[r] \ar[d] &
      {\strat{X}} \ar[d, "\xi"] \\
      {\DeltaStrat{k}} \ar[r, swap] \ar[ur, dashed] &
      {\strat{B}}
    \end{tikzcd}
  \]
  Then by taking pullbacks we can reduce to the case that $\strat{B} = \DeltaStrat{k}$:
  \begin{equation}\label{eq:fct-mesh-closed-1-bundle-cat-fib:reduced}
    \begin{tikzcd}
      {\HornStrat{i}{k}} \ar[r, "\tau"] \ar[d] &
      {\strat{X}} \ar[d, "\xi"] \\
      {\DeltaStrat{k}} \ar[r, "\id", swap] \ar[ur, dashed] &
      {\DeltaStrat{k}}
    \end{tikzcd}
  \end{equation}
  When $\tau(k)$ is singular, then $\tau$ factors through $\sing(\strat{X})$,
  and we get a unique lift by Lemma~\ref{lem:fct-mesh-closed-1-sing-covering}.
  Now suppose instead that $\tau(k)$ is regular.
  There must exist singular points $s_0, s_1 \in \strat{X}$ with $\xi(s_0) = \xi(s_1) = k$
  which are adjacent to the stratum of $\tau(k)$ from above and below.
  By Lemma~\ref{lem:fct-mesh-closed-1-sing-covering} we then get unique lifts
  \[
    \begin{tikzcd}
      {\{ k \}} \ar[r, "s_0"] \ar[d] &
      {\strat{X}} \ar[d, "\xi"] \\
      {\DeltaStrat{k}} \ar[r, "\id", swap] \ar[ur, "\sigma_0"{description}] &
      {\DeltaStrat{k}}
    \end{tikzcd}
    \qquad
    \begin{tikzcd}
      {\{ k \}} \ar[r, "s_1"] \ar[d] &
      {\strat{X}} \ar[d, "\xi"] \\
      {\DeltaStrat{k}} \ar[r, "\id", swap] \ar[ur, "\sigma_1"{description}] &
      {\DeltaStrat{k}}
    \end{tikzcd}
  \]
  We have $\sigma_0(x) \leq \tau(x) \leq \sigma_1(x)$ for all $x \in \HornStrat{i}{k}$
  in the fibrewise ordering.
  By cutting off the strata below $\sigma_0$ and above $\sigma_1$
  followed by rescaling,
  we can reduce to the case that $\unstrat(\strat{X}) = \intCC{0, 1} \times \DeltaTop{k}$
  and that $\strat{X}$ has a single regular stratum over $k \in \DeltaStrat{k}$.
  But then the closure of that stratum must be the entirety of $\strat{X}$.
  We can therefore safely provide the filler for~(\ref{eq:fct-mesh-closed-1-bundle-cat-fib:reduced}) by linear interpolation.

  To see that $\Exit(\xi)$ is a categorical fibration it now remains to show that
  it is an isofibration.
  This followed from $\xi$ being a stratified fibration by using local trivialisations over any stratum of $\strat{B}$.
\end{proof}

\begin{lemma}\label{lem:fct-mesh-closed-1-recognise-fibration}
  Let $\xi : \strat{X} \to \strat{B}$ be a $1$-framed stratified bundle
  that satisfies:
  \begin{enumerate}[label=(\arabic*)]
    \item $\xi$ is a stratified fibre bundle.
    \item For every $b \in \strat{B}$ the fibre $\xi^{-1}(b)$ is a closed $1$-mesh.
    \item $\Exit(\xi)$ is an inner fibration.
  \end{enumerate}
  Then $\xi$ is a closed $1$-mesh bundle.
\end{lemma}
\begin{proof}
  Suppose for the purpose of contradiction that $\sing(\strat{X})$ is not closed.
  Then there must exist an exit path $\gamma_{12} : \DeltaStrat{1} \to \strat{X}$
  such that $x_1 := \gamma_{12}(0)$ is regular and $x_2 := \gamma_{12}(1)$ is singular.
  Since $x_1$ is regular there must be another point $x_0 \in \strat{X}$
  in the same stratum as $x_1$ so that $\xi(x_1) = \xi(x_0)$ but $x_1 \neq x_0$.
  We pick a path $\gamma_{01} : \DeltaStrat{1} \to \strat{X}$ from $x_0$ and $x_1$
  within the fibre $\xi^{-1}(\xi(x_1))$.
  Then $\gamma_{01}$ and $\gamma_{12}$ together define a stratified horn
  $\gamma : \HornStrat{1}{2} \to \strat{X}$.
  Since $\xi \circ \gamma_{01}$ is constant, the horn $\xi \circ \gamma$ extends
  to a stratified $2$-simplex $\sigma : \DeltaStrat{2} \to \strat{B}$ such that
  \[
    \begin{tikzcd}
      {\DeltaStrat{2}} \ar[r, "\sigma"] \ar[d, "{\langle 0, 1 \rangle}_*"'] &
      {\strat{B}} \\
      {\DeltaStrat{1}} \ar[ur, "\xi \circ \gamma_{01}"']
    \end{tikzcd}
  \]
  By assumption (3), we then have a lift
  \[
    \begin{tikzcd}
      {\HornStrat{1}{2}} \ar[r, "\gamma"] \ar[d, hook] &
      {\strat{X}} \ar[d, "\xi"] \\
      {\DeltaStrat{2}} \ar[r, "\sigma"'] \ar[ur, dashed, "\hat{\gamma}"{description}] &
      {\strat{B}}
    \end{tikzcd}
  \]
  Since $\hat{\gamma}(2) = \gamma(2)$ is singular,
  we have that the paths
  $\hat{\gamma} \circ \langle 0, 2 \rangle_*$
  and 
  $\hat{\gamma} \circ \langle 1, 2 \rangle_*$
  agree on $\DeltaStrat{1} \setminus \{ 0 \}$.
  But then by continuity we must have
  \[
    x_0 = \hat{\gamma}(0) =
    (\hat{\gamma} \circ \langle 0, 2 \rangle_*)(0) =
    (\hat{\gamma} \circ \langle 1, 2 \rangle_*)(0) =
    \hat{\gamma}(1) = x_1
  \]
  This contradicts our choice of $x_0 \neq x_1$.
\end{proof}

\begin{lemma}\label{lem:fct-mesh-closed-1-bundle-flat}
  Let $\xi : \strat{X} \to \strat{B}$ be a closed $1$-mesh bundle.
  Then the induced map
  \[
    \Exit(\xi) : \Exit(\strat{X}) \longrightarrow \Exit(\strat{B})
  \]
  is a flat categorical fibration.
\end{lemma}
\begin{proof}
  The map is a categorical fibration by Lemma~\ref{lem:fct-mesh-closed-1-bundle-cat-fib}.
  To show that $\Exit(\xi)$ is also flat,
  we demonstrate that for every $2$-simplex
  $\sigma : \Delta\ord{2} \to \Exit(\strat{B})$
  the map 
  \begin{equation}\label{eq:fct-mesh-closed-1-bundle-flat-map}
    \begin{tikzcd}
      \Fun_{/ \Exit(\strat{B})}(\Delta\ord{2}, \Exit(\strat{E}))
      \ar[r, "{\langle 0, 2 \rangle}^*"] &
      \Fun_{/ \Exit(\strat{B})}(\Delta\ord{1}, \Exit(\strat{E}))
    \end{tikzcd}
  \end{equation}
  has weakly contractible fibres
  over any $\tau : \Delta\ord{1} \to \Exit(\strat{E})$ such that
  \[
    \begin{tikzcd}
      {\Delta\ord{1}} \ar[r, "\tau"] \ar[d, "{\langle 0, 2 \rangle}", swap] &
      {\Exit(\strat{E})} \ar[d] \\
      {\Delta\ord{2}} \ar[r, "\sigma", swap] &
      {\Exit(\strat{B})}
    \end{tikzcd}
  \]
  Without loss of generality we can simplify to the situation where
  $\strat{B}$ is the stratified $2$-simplex $\DeltaStrat{2}$.
  Then $\tau$ corresponds to a partial section $s : \DeltaStrat{1} \to \strat{X}$  
  which fits into the diagram
  \[
    \begin{tikzcd}
      {\DeltaStrat{1}} \ar[d, hook, "{\langle 0, 2 \rangle}_*", swap] \ar[r, "s"] &
      {\strat{X}} \ar[d, "\xi"] \\
      {\DeltaStrat{2}} \ar[r, "\id", swap] &
      {\DeltaStrat{2}}
    \end{tikzcd}
  \]
  When $F$ denotes the fibre of (\ref{eq:fct-mesh-closed-1-bundle-flat-map}) over $\tau$,
  a $k$-simplex of $F$
  is a stratified map $h : \DeltaStrat{k} \times \DeltaStrat{2} \to \strat{X}$ such that
  the following diagram commutes:
  \[
    \begin{tikzcd}[column sep = large]
      {\DeltaStrat{k} \times \DeltaStrat{1}} \ar[d, hook, "{\id \times \langle 0, 2 \rangle}_*", swap] \ar[r, "s \circ \pi_2"] &
      {\strat{X}} \ar[d, "\xi"] \\
      {\DeltaStrat{k} \times \DeltaStrat{2}} \ar[r, swap, "\pi_2"] \ar[ur, "{h}"{description}]&
      {\DeltaStrat{2}}
    \end{tikzcd}
  \]
  Here $\pi_2$ is the projection onto the second factor.
  
  Suppose first that $s(2) \in \strat{X}$ is a singular point.
  Then $s : \DeltaStrat{1} \to \strat{X}$ must factor through $\sing(\strat{X})$
  and so by Lemma~\ref{lem:fct-mesh-closed-1-sing-covering}
  the simplicial set $F$ must have a unique $k$-simplex for every $k \geq 0$.
  It is therefore contractible.
  
  Suppose now that $s(2) \in \strat{X}$ is a regular point.
  Let $x_0 \in \xi^{-1}(0)$ be the highest singular point such that $x_0 \leq s(2)$
  in the fibrewise ordering. Similarly, let $x_1 \in \xi^{-1}(2)$ be the lowest
  singular point with $s(2) \leq x_1$.
  By Lemma~\ref{lem:fct-mesh-closed-1-sing-covering} there are unique lifts
  \[
    \begin{tikzcd}
      {\DeltaStrat{0}} \ar[r, "x_0"] \ar[d, "{\langle 2 \rangle}", swap] &
      {\sing(\strat{X})} \ar[d] \\
      {\DeltaStrat{2}} \ar[r, "\id", swap] \ar[ur, dashed, "\tau_0"{description}]&
      {\DeltaStrat{2}}
    \end{tikzcd}
    \qquad
    \begin{tikzcd}
      {\DeltaStrat{0}} \ar[r, "x_1"] \ar[d, "{\langle 2 \rangle}", swap] &
      {\sing(\strat{X})} \ar[d] \\
      {\DeltaStrat{2}} \ar[r, "\id", swap] \ar[ur, dashed, "\tau_1"{description}]&
      {\DeltaStrat{2}}
    \end{tikzcd}
  \]
  By restricting $\strat{X}$ to the points between $\tau_0$ and $\tau_1$
  and rescaling, we can simplify to the case that $\unstrat(\strat{X}) =  \intCC{0, 1} \times \DeltaTop{2}$ and $\strat{X}$ has exactly one regular stratum over $2 \in \DeltaStrat{2}$.
  Let $\strat{E}$ be the subspace of $\strat{X}$ consisting of those points
  $x \in \strat{X}$ such that there exists a map $\gamma : \DeltaStrat{1} \to \strat{X}$
  which fits into the diagram
  \[
    \begin{tikzcd}
      {\DeltaStrat{0}} \ar[r, "x"] \ar[d, "{\langle 2 \rangle}", hook, swap] &
      {\strat{X}} \ar[d, "\xi"] \\
      {\DeltaStrat{1}} \ar[r, "{\langle 1, 2 \rangle}", hook, swap] \ar[ur, "\gamma"{description}]&
      {\DeltaStrat{2}}
    \end{tikzcd}
  \]
  Then there is a map $F \to \Exit(\strat{E})$ which sends a $k$-simplex
  $h : \DeltaStrat{k} \times \DeltaStrat{2} \to \strat{X}$ to the restriction
  $h(-, 1) : \DeltaStrat{k} \to \strat{E}$.
  We see that this map is an equivalence and $\strat{E}$ is a stratified interval.
  It follows that $F$ is weakly contractible.
\end{proof}

\subsection{Closed $n$-Meshes}\label{sec:fct-mesh-closed-n}

\begin{definition}\label{def:fct-mesh-closed-n}
  A \defn{closed $n$-mesh bundle} is an $n$-framed stratified bundle
  $\xi : \strat{X}_n \to \strat{X}_0$ such that $\xi$ is the composite
  of a sequence of closed $1$-mesh bundles  
  \begin{equation}\label{eq:fct-mesh-closed-n}
    \begin{tikzcd}
      {\strat{X}_n} \ar[r, "\xi_n"] &
      {\strat{X}_{n - 1}} \ar[r] &
      \cdots \ar[r] &
      {\strat{X}_1} \ar[r, "\xi_1"] &
      {\strat{X}_0}
    \end{tikzcd}
  \end{equation}
  and $\xi$ is equipped with the composite $n$-framing.
\end{definition}

\begin{observation}\label{obs:fct-mesh-closed-n-composite}
  When we have a closed $n$-mesh bundle $\xi : \strat{X}_n \to \strat{X}_0$
  arising as the composite of a sequence of closed $1$-mesh bundles~(\ref{eq:fct-mesh-closed-n}),
  we can uniquely recover that sequence from $\xi$.
  For any $0 \leq i \leq n$ the stratified space
  $\strat{X}_i$ is the quotient of $\strat{X}_n$ which identifies two points
  when their last $i$ coordinates in the framing agree. 
\end{observation}

\begin{construction}
  Let $\xi : \strat{X}_n \to \strat{X}_0$ be a closed $n$-mesh bundle
  that decomposes into a sequence of closed $1$-mesh bundles
  $\xi = \xi_n \circ \cdots \circ \xi_1$.
  Whenever $0 \leq m \leq n$, we define the \defn{$m$-truncation} of $\xi$ to
  be the closed $m$-mesh bundle $\xi_m \circ \cdots \circ \xi_1$.
\end{construction}

\begin{lemma}\label{lem:fct-mesh-closed-n-bundle-flat}
  Let $\xi : \strat{X} \to \strat{B}$ be a closed $n$-mesh bundle. Then the induced map
  \[
    \Exit(\xi) : \Exit(\strat{X}) \longrightarrow \Exit(\strat{B})
  \]
  is a flat categorical fibration.
  In particular when $\strat{B}$ is a fibrant stratified space,
  then so is $\strat{X}$.
\end{lemma}
\begin{proof}
  We write $\xi$ as the composite of closed $1$-mesh bundles
  $\xi = \xi_n \circ \cdots \circ \xi_1$.
  Then $\Exit(\xi)$ is the composite of the induced maps
  $\Exit(\xi_n) \circ \cdots \circ \Exit(\xi_1)$
  which are flat categorical fibrations by Lemma~\ref{lem:fct-mesh-closed-1-bundle-flat}.
  Therefore $\Exit(\xi)$ is again a flat categorical fibration.
\end{proof}

\begin{lemma}\label{lem:fct-mesh-closed-n-pullback}
  Closed $n$-mesh bundles are closed under pullback.
\end{lemma}
\begin{proof}
  Let $\xi : \strat{X} \to \strat{B}$ be a closed $n$-mesh bundle and let
  $f : \strat{A} \to \strat{B}$ be a map of stratified spaces.
  We can write $\xi$ as the composite of closed $1$-mesh bundles
  $\xi_i : \strat{X}_i \to \strat{X}_{i - 1}$ for $1 \leq i \leq n$.
  In particular $\strat{X}_n = \strat{X}$ and $\strat{X}_0 = \strat{B}$.
  Then the pullback of $\xi$ along $f$ decomposes into a sequence of
  pullbacks of closed $1$-mesh bundles
  \[
    \begin{tikzcd}
      {\strat{Y} = \strat{Y}_n} \ar[r, "\zeta_n"] \ar[d] \pullbackcorner &
      {\strat{Y}_{n - 1}} \ar[r] \ar[d] &
      {\cdots} \ar[r] &
      {\strat{Y}_1} \ar[r, "\zeta_1"] \ar[d] \pullbackcorner &
      {\strat{Y}_0 = \strat{A}} \ar[d, "f"] \\
      {\strat{X} = \strat{X}_n} \ar[r, "\xi_n"'] &
      {\strat{X}_{n - 1}} \ar[r] &
      {\cdots} \ar[r] &
      {\strat{X}_1} \ar[r, "\xi_1"'] &
      {\strat{X}_0 = \strat{B}}
    \end{tikzcd}
  \]
  By Lemma~\ref{lem:fct-mesh-closed-1-pullback} closed $1$-mesh bundles are
  closed under pullback.
  Therefore, the pullback of $\xi$ along $f$ is the closed $n$-mesh bundle
  $\zeta_1 \circ \cdots \circ \zeta_n$.
\end{proof}

\begin{definition}
  We let $\BMeshClosed{n}$ be the simplicial set whose $k$-simplices consist of a
  closed $n$-mesh bundle $\xi : \strat{X} \to \DeltaStrat{k}$.
  An order-preserving map $\ord{k'} \to \ord{k}$ acts by pullback along the
  induced map $\DeltaStrat{k'} \to \DeltaStrat{k}$.
  Truncation of closed $(n + 1)$-mesh bundles defines a map of simplicial sets
  $\tr_n : \BMeshClosed{n + 1} \to \BMeshClosed{n}$.
\end{definition}

\begin{observation}\label{obs:fct-mesh-closed-n-base-case}
  In the case that $n = 0$,
  a $k$-simplex of $\BMeshClosed{0}$ is a closed $0$-mesh bundle
  $\xi : \strat{X} \to \DeltaStrat{k}$.
  This means that $\xi$ is the composite of a length zero sequence of $1$-mesh bundles,
  and so $\xi$ is the identity on $\DeltaStrat{k}$.
  It therefore follows that $\BMeshClosed{0} \cong \terminal$ is the terminal object
  in $\sSet$.
\end{observation}

\begin{lemma}\label{lem:fct-mesh-closed-n-cat-truncate}
  The truncation map $\tr_n : \BMeshClosed{n + 1} \to \BMeshClosed{n}$ is an inner fibration.
  In particular $\BMeshClosed{n}$ is a quasicategory for each $n \geq 0$.
\end{lemma}
\begin{proof}
  Suppose we have a lifting problem
  \[
    \begin{tikzcd}
      {\Lambda^i\ord{k}} \ar[r] \ar[d] &
      {\BMeshClosed{n + 1}} \ar[d, "\tr_n"] \\
      {\Delta\ord{k}} \ar[r] \ar[ur, dashed] &
      {\BMeshClosed{n}}
    \end{tikzcd}
  \]
  for some $0 < i < k$.
  Unpacking the definitions, we then have a diagram
  \[
    \begin{tikzcd}
      {\strat{X}'} \ar[d, "\xi'", swap] \\
      {\strat{Y}'} \ar[r, hook] \ar[d, "\zeta'", swap] \pullbackcorner &
      {\strat{Y}} \ar[d, "\zeta"] \\
      {\HornStrat{i}{k}} \ar[r, hook] &
      {\DeltaStrat{k}}
    \end{tikzcd}
  \]
  where $\xi$ is a closed $1$-mesh bundle and $\zeta$, $\zeta'$ are closed $n$-mesh bundles.
  By Corollary~\ref{cor:fct-mesh-closed-1-total-triangulable} we have that $\strat{Y}$ is triangulable.
  The pullback square of stratified spaces induces a pullback square of simplicial sets
  \[
    \begin{tikzcd}
      {\Exit(\strat{Y}')} \ar[r, hook] \ar[d, "\Exit(\zeta')", swap] \pullbackcorner &
      {\Exit(\strat{Y})} \ar[d, "\Exit(\zeta)"] \\
      {\Exit(\HornStrat{i}{k})} \ar[r, hook] &
      {\Exit(\DeltaStrat{k})}
    \end{tikzcd}
  \]
  By Lemma~\ref{lem:fct-mesh-closed-n-bundle-flat} the map $\Exit(\zeta)$ is a flat categorical fibration.
  Since $\DeltaStrat{k}$ is fibrant and the stratified inner horn inclusion
  is a stratified weak equivalence, it follows by Lemma~\ref{lem:b-poly-flat-cat-pullback-equiv} that
  the inclusion $\strat{Y}' \hookrightarrow \strat{Y}$ is a stratified weak equivalence.
  Since $\BMeshClosed{1}$ is a quasicategory, we can therefore find an extension
  \[
    \begin{tikzcd}
      {\Exit(\strat{Y}')} \ar[r, "\classify{\xi'}"] \ar[d, hook, "\simeq"'] &
      {\BMeshClosed{1}} \\
      {\Exit(\strat{Y})} \ar[ur, dashed, "\varphi"'] &
      {}
    \end{tikzcd}
  \]
  Then by Proposition~\ref{prop:fct-mesh-closed-1-classify}
  there exists a closed $1$-mesh bundle $\xi : \strat{X} \to \strat{Y}$
  so that $\xi$ restricts to $\xi'$ along the inclusion $\strat{Y}' \to \strat{Y}$.
  The composite $\zeta \circ \xi$ is a closed $(n + 1)$-mesh bundle over
  $\DeltaStrat{k}$ that represents a solution to the lifting problem.
\end{proof}

\begin{construction}
  Suppose that $\xi : \strat{X}_n \to \strat{X}_0$ is a closed $n$-mesh bundle.
  Then for every stratified $k$-simplex $\sigma : \DeltaStrat{k} \to \strat{X}_0$
  we have a closed $n$-mesh bundle $\sigma^* \xi$ over $\DeltaStrat{k}$ by pullback of $\xi$ along $\sigma$.
  This defines the \defn{classifying map} $\classify{\xi} : \Exit(\strat{X}_0) \to \BMeshClosed{n}$.
\end{construction}

\begin{lemma}\label{lem:fct-mesh-closed-n-classify-marked}
  Let $S$ be a marked poset and $\varphi : S \to (\BMeshClosed{n})^\natural$
  a map of marked simplicial sets.
  Then there exists a unique closed $n$-mesh bundle
  $\xi$ over $\StratRealMark{S}$ 
  such that the following diagram commutes:
  \[
    \begin{tikzcd}
      {S} \ar[r] \ar[d, hook] &
      {(\BMeshClosed{n})^\natural} \\
      {\ExitMark(\StratRealMark{S})} \ar[ur, dashed, "\classify{\xi}"'] &
      {}
    \end{tikzcd}
  \]
\end{lemma}
\begin{proof}
  By induction assume that the claim holds for some $n \geq 0$.
  Let $S$ be a marked poset and $\varphi : S \to (\BMeshClosed{n + 1})^\natural$
  a map of marked simplicial sets.
  By truncation the map $\varphi$ induces a map $\psi : S \to (\BMeshClosed{n})^\natural$.
  Then by induction there exists a unique closed $n$-mesh bundle
  $\zeta : \strat{Y} \to \StratRealMark{S}$ such that $\classify{\zeta}$
  fits into the diagram
  \[
    \begin{tikzcd}
      {S} \ar[r] \ar[d, hook] &
      {(\BMeshClosed{n})^\natural} \\
      {\ExitMark(\StratRealMark{S})} \ar[ur, dashed, "\classify{\zeta}"'] &
      {}
    \end{tikzcd}
  \]
  For every marked flag $\sigma$ in $S$ the map $\varphi$ holds a closed $n$-mesh
  bundle $\zeta_\sigma$ and a closed $1$-mesh bundle $\xi_\sigma$  
  which together with $\zeta$ fit into the diagram
  \[
    \begin{tikzcd}
      {\strat{X}_\sigma} \ar[d, "\xi_\sigma"'] \\
      {\strat{Y}_\sigma} \ar[r, hook] \ar[d, "\zeta_\sigma"'] \pullbackcorner &
      {\strat{Y}} \ar[d, "\zeta"] \\
      {\StratRealMark{(\DeltaTop{k}, T)}} \ar[r, hook, "\sigma"'] &
      {\StratRealMark{S}}
    \end{tikzcd}
  \]
  The stratified space $\strat{Y}_\sigma$ is a closed triangulable subspace
  of $\strat{Y}$ and the family of these subspaces for all marked flags $\sigma$ of $S$
  covers $\strat{Y}$.
  Using Lemma~\ref{lem:fct-mesh-closed-1-classify-marked} we may therefore glue together the closed $1$-mesh bundles
  to obtain a closed $1$-mesh bundle $\xi : \strat{X} \to \strat{Y}$.
  Then for each $\sigma$ the composite $\zeta \circ \xi$ restricts to the
  closed $(n + 1)$-mesh bundle $\zeta_\sigma \circ \xi_\sigma$:
  \[
    \begin{tikzcd}
      {\strat{X}_\sigma} \ar[d, "\xi_\sigma"'] \pullbackcorner \ar[r, hook] &
      {\strat{X}} \ar[d, "\xi"] \\
      {\strat{Y}_\sigma} \ar[r, hook] \ar[d, "\zeta_\sigma"'] \pullbackcorner &
      {\strat{Y}} \ar[d, "\zeta"] \\
      {\StratRealMark{(\DeltaTop{k}, T)}} \ar[r, hook, "\sigma"'] &
      {\StratRealMark{S}}
    \end{tikzcd}
  \]
  It follows that $\zeta \circ \xi$ is the unique closed $(n + 1)$-mesh bundle
  induced by $\varphi$.
\end{proof}

\begin{proposition}\label{prop:fct-mesh-closed-n-classify}
  Let $\strat{B}$ be a triangulable stratified space and
  $I : \strat{A} \hookrightarrow \strat{B}$
  a closed triangulable stratified subspace.
  Suppose that $\varphi : \Exit(\strat{B}) \to \BMeshClosed{n}$ is a map
  and $\xi'$ a closed $n$-mesh bundle over $\strat{A}$ such that
  $\classify{\xi'} = \varphi \circ I$.
  Then there is a closed $n$-mesh bundle $\xi$ over $\strat{B}$ together
  with an equivalence   
  $\varphi \simeq \classify{\xi}$ 
  relative to $\classify{\xi'}$.
\end{proposition}
\begin{proof}
  Analogous to the proof of Proposition~\ref{prop:fct-mesh-closed-1-classify}
  by using Lemma~\ref{lem:fct-mesh-closed-1-classify-marked}.
\end{proof}

\subsection{Closed $n$-Meshes with Labels}\label{sec:fct-mesh-closed-labelled}

Suppose that $\xi : \strat{X} \to \strat{B}$ is a closed $n$-mesh bundle and $C$ a simplicial set.
We say that $\xi$ is labelled in $C$ when it is equipped with a map of simplicial sets
$L : \Exit(\strat{X}) \to C$ which we will call a \defn{labelling map}.
This notion will be particularly well-behaved when $C = \cat{C}$ is a quasicategory.

\begin{observation}\label{obs:basic-pullback-section}
	Let $\cat{C}$ be a $1$-category and $f : E \to B$ a map in $\cat{C}$
	together with a section, i.e.\ a map $s : B \to E$ such that $\id = f \circ s$.
	When the pullback of $f$ along a map $g : A \to B$ exists in $\cat{C}$,
	then the universal property induces a unique section of the pulled back map
	\[
		\begin{tikzcd}
			{A} \ar[drr, bend left, "s \circ g", swap] \ar[ddr, bend right, "\id"] \ar[dr, dashed] \\
			&
			{E \times_B A} \ar[r] \ar[d] \pullbackcorner &
			{E} \ar[d, "f"] \\
			&
			{A} \ar[r, "g"] &
			{B}
		\end{tikzcd}
	\]
\end{observation}

\begin{definition}\label{def:fct-mesh-closed-labelled-pointed}
  We denote by $\EMeshClosed{n}$ the simplicial set
  whose $k$-simplices consist of a closed $n$-mesh bundle $\xi : \strat{X} \to \DeltaStrat{k}$
  together with a section $s : \DeltaStrat{k} \to \strat{X}$ with $\id = f \circ s$.
  An order-preserving map $\ord{k'} \to \ord{k}$ acts by pullback
  along the induced map $\DeltaStrat{k'} \to \DeltaStrat{k}$
  which induces a new section as in Observation~\ref{obs:basic-pullback-section}.
  We write $\MeshClosed{n} : \EMeshClosed{n} \to \BMeshClosed{n}$ for the map of simplicial sets
  which forgets the section.
  The map $\MeshClosed{n}$ induces a polynomial functor
  $\BMeshClosedL{n}{-} : \sSet \to \sSet$.
\end{definition}

\begin{observation}\label{obs:fct-mesh-closed-labelled-base-case}
  In Observation~\ref{obs:fct-mesh-closed-labelled-base-case} we have seen
  that $\BMeshClosed{0} \cong \terminal$.
  A $k$-simplex of $\EMeshClosed{0}$ is a closed $0$-mesh bundle
  $\xi : \strat{X} \to \DeltaStrat{k}$ together with a section
  $s : \DeltaStrat{k} \to \strat{X}$.
  Because $\xi$ is the identity on $\DeltaStrat{k}$,
  the section must also be the identity.
  Therefore, $\EMeshClosed{0} \cong \terminal$ is the terminal object
  of $\sSet$ and $\EMeshClosed{0} \to \BMeshClosed{0}$ is the identity on
  $\terminal$.
  It follows that for any simplicial set $C$ there is a natural isomorphism
  of simplicial sets $\BMeshClosedL{0}{C} \cong C$.
\end{observation}

\begin{lemma}\label{lem:fct-mesh-closed-labelled-pointed-cat-fib}
  The map
  $\MeshClosed{n} : \EMeshClosed{n} \to \BMeshClosed{n}$
  is a categorical fibration.
  In particular the simplicial set $\EMeshClosed{n}$ is a quasicategory.
\end{lemma}
\begin{proof}
  We first show that the map is an inner fibration.
  Suppose we have a lifting problem
  \[
    \begin{tikzcd}
      {\Lambda^i \ord{k}} \ar[r] \ar[d] &
      {\EMeshClosed{n}} \ar[d, "\MeshClosed{n}"] \\
      {\Delta\ord{k}} \ar[r] \ar[ur, dashed] &
      {\BMeshClosed{n}}
    \end{tikzcd}
  \]
  for some $0 < i < k$.
  The given data corresponds to a closed $n$-mesh bundle $\xi$ together with a partial section
  $s$ which fit into the diagram
  \[
    \begin{tikzcd}
      {\HornStrat{i}{k}} \ar[r, "s'"] \ar[d, hook] &
      {\strat{M}} \ar[d, "f"] \\
      {\DeltaStrat{k}} \ar[r] &
      {\DeltaStrat{k}}
    \end{tikzcd}
  \]
  Using the $\StratReal{-} \dashv \Exit$ adjunction this corresponds to a diagram
  \[
    \begin{tikzcd}
      {\Lambda^i\ord{k}} \ar[r] \ar[d, hook] &
      {\Exit(\strat{X})} \ar[d, "\Exit(\xi)"] \\
      {\Delta\ord{k}} \ar[r] &
      {\Exit(\DeltaStrat{k})}
    \end{tikzcd}
  \]
  By Lemma~\ref{lem:fct-mesh-closed-n-bundle-flat} the map $\Exit(\xi)$ is in particular
  an inner fibration, so there is a lift
  \[
    \begin{tikzcd}
      {\Lambda^i\ord{k}} \ar[r] \ar[d, hook] &
      {\Exit(\strat{X})} \ar[d, "\Exit(\xi)"] \\
      {\Delta\ord{k}} \ar[r] \ar[ur, dashed] &
      {\Exit(\DeltaStrat{k})}
    \end{tikzcd}
  \]
  which classifies an extension of the partial section $s'$ to a full section
  of $\xi$.
  Then $\xi$ together with the full section is a solution to the lifting problem.
  We have therefore shown that $\MeshClosed{n} : \EMeshClosed{n} \to \BMeshClosed{n}$ is an inner fibration. It follows similarly that $\MeshClosed{n}$ is an isofibration.
\end{proof}

\begin{lemma}\label{lem:fct-mesh-closed-labelled-universal}
  Let $\xi : \strat{X} \to \strat{B}$ be a closed $n$-mesh bundle.
  Then there is a pullback square
  \begin{equation}\label{eq:fct-mesh-closed-labelled-universal}
    \begin{tikzcd}
      {\Exit(\strat{X})} \ar[r] \ar[d, "\Exit(\xi)", swap] \pullbackcorner &
      {\EMeshClosed{n}} \ar[d, "\MeshClosed{n}"] \\
      {\Exit(\strat{B})} \ar[r, "\classify{\xi}", swap] &
      {\BMeshClosed{n}}
    \end{tikzcd}
  \end{equation}
  of simplicial sets.
  In particular, when $\strat{B}$ is a fibrant stratified space,
  the pullback square is a homotopy pullback of quasicategories.
\end{lemma}
\begin{proof}
  We begin by describing the map $\Exit(\strat{X}) \to \EMeshClosed{n}$ explicitly.
  A $k$-simplex of $\Exit(\strat{X})$ is a stratified $k$-simplex $\sigma : \DeltaStrat{k} \to \strat{X}$.
  We need to send $\sigma$ to a $k$-simplex of $\EMeshClosed{n}$, representing a closed $n$-mesh bundle over $\DeltaStrat{k}$ together with a section.
  To make the square (\ref{eq:fct-mesh-closed-labelled-universal}) commute,
  this closed $n$-mesh bundle needs to be the pullback
  \[
    \begin{tikzcd}
      {\strat{X} \times_{\DeltaStrat{k}} \strat{B}} \ar[r] \ar[d, "(\xi \circ \sigma)^* \xi"'] \pullbackcorner &
      {\strat{X} \times_{\strat{B}} \strat{X}} \ar[r] \ar[d, "\xi^* \xi"] &
      {\strat{X}} \ar[d, "\xi"] \\
      {\DeltaStrat{k}} \ar[r, "\sigma"'] &
      {\strat{X}} \ar[r, "\xi"'] &
      {\strat{B}}
    \end{tikzcd}
  \]
  Via the universal property of the pullback the map $\sigma : \DeltaStrat{k} \to \strat{X}$ induces a section of $(\xi \circ \sigma)^* \xi$ and therefore a $k$-simplex of
  $\EMeshClosed{n}$ as needed.

  It remains to show that the square~(\ref{eq:fct-mesh-closed-labelled-universal}) is a
  pullback square of simplicial sets.
  A $k$-simplex of $\Exit(\strat{B}) \times_{\BMeshClosed{n}} \EMeshClosed{n}$
  consists of a diagram
  \begin{equation}\label{eq:fct-mesh-closed-labelled-universal:pullback}
    \begin{tikzcd}
      {} &
      {\DeltaStrat{k} \times_{\strat{B}} \strat{X}} \ar[r, "\xi^* \sigma"] \ar[d] \pullbackcorner &
      {\strat{X}} \ar[d, "\xi"] \\
      {\DeltaStrat{k}} \ar[r, equal] \ar[ur, "s"] & 
      {\DeltaStrat{k}} \ar[r, "\sigma"'] &
      {\strat{B}}
    \end{tikzcd}
  \end{equation}
  The induced map
  $\Exit(\strat{B}) \times_{\BMeshClosed{n}} \EMeshClosed{n} \to \Exit(\strat{X})$
  sends the diagram (\ref{eq:fct-mesh-closed-labelled-universal:pullback}) to the stratified $k$-simplex
  $\xi^* \sigma \circ s : \DeltaStrat{k} \to \strat{X}$.
  Conversely, by the universal property of the pullback of stratified spaces in (\ref{eq:fct-mesh-closed-labelled-universal:pullback}),
  the stratified $k$-simplex of $\strat{X}$ uniquely determines the section $s$.
  Therefore, the square (\ref{eq:fct-mesh-closed-labelled-universal}) is a pullback square of simplicial sets.
\end{proof}

\begin{construction}
  Let $\xi : \strat{X} \to \strat{B}$ be a closed $n$-mesh bundle together
  with a labelling map $\ell : \Exit(\strat{X}) \to C$ for some simplicial set $C$.
  For any stratified $k$-simplex $\sigma : \DeltaStrat{k} \to \strat{B}$
  we have an induced diagram of simplicial sets  
  \[
    \begin{tikzcd}[row sep = large]
      {E} \ar[r, hook] \ar[d] \pullbackcorner &
      {\Exit(\strat{X}_{\sigma})} \ar[d, "\Exit(\sigma^* \xi)"{description}] \ar[r] \pullbackcorner &
      {\Exit(\strat{X})} \ar[r] \ar[d, "\Exit(\xi)"{description}] \pullbackcorner &
      {\EMeshClosed{n}} \ar[d] \\
      {\Delta\ord{k}} \ar[r, hook] &
      {\Exit(\DeltaStrat{k})} \ar[r] &
      {\Exit(\strat{B})} \ar[r, "\classify{\xi}"'] &
      {\BMeshClosed{n}}
    \end{tikzcd}
  \]
  Then $\classify{\xi}$ together with the composite
  \[
    \begin{tikzcd}
      {E} \ar[r] &
      {\Exit(\strat{X}_\sigma)} \ar[r] &
      {\Exit(\strat{X})} \ar[r, "\ell"] &
      {C}
    \end{tikzcd}
  \]
  determine a $k$-simplex of $\BMeshClosedL{n}{C}$.
  This defines the \defn{classifying map} $\classify{\xi; \ell} : \Exit(\strat{B}) \to \BMeshClosedL{n}{C}$ and is compatible with  
  the unlabelled case by making the following diagram commute:
  \[
    \begin{tikzcd}
      {} &
      {\BMeshClosedL{n}{C}} \ar[d] \\
      {\Exit(\strat{B})} \ar[r, "\classify{\xi}"'] \ar[ur, "\classify{\xi; \ell}"] &
      {\BMeshClosed{n}}
    \end{tikzcd}
  \]
\end{construction}

\begin{lemma}\label{lem:fct-mesh-closed-labelled-pointed-flat}
  The forgetful map
  $\MeshClosed{n} : \EMeshClosed{n} \to \BMeshClosed{n}$
  is a flat categorical fibration.
\end{lemma}
\begin{proof}
  The map $\EMeshClosed{n} \to \BMeshClosed{n}$ is already a categorical fibration by
  Lemma~\ref{lem:fct-mesh-closed-labelled-pointed-cat-fib}.  
  The map is flat if for every $2$-simplex
  $\sigma : \Delta\ord{2} \to \BMeshClosed{n}$ the pulled back map
  \begin{equation}\label{eq:fct-mesh-labelled-pointed-flat:pullback}
    \begin{tikzcd}
      {\Delta\ord{2} \times_{\BMesh{n}} \EMesh{n}} \ar[r] \ar[d, "\varphi", swap] \pullbackcorner &
      {\EMesh{n}} \ar[d, "\Mesh{n}"] \\
      {\Delta\ord{2}} \ar[r, "\sigma", swap] &
      {\BMesh{n}}
    \end{tikzcd}
  \end{equation}
  is flat.
  The $2$-simplex $\sigma$ determines a closed $n$-mesh bundle
  $\xi : \strat{X} \to \DeltaStrat{2}$ by the definition of $\BMeshClosed{n}$.
  Using Lemma~\ref{lem:fct-mesh-closed-labelled-universal} the square
  (\ref{eq:fct-mesh-labelled-pointed-flat:pullback}) factors as the composite
  of the pullback squares
  \[
    \begin{tikzcd}
      {\Delta\ord{2} \times_{\BMeshClosed{n}} \EMeshClosed{n}} \ar[r] \ar[d, "\varphi", swap] \pullbackcorner &
      {\Exit(\strat{X})} \ar[r] \ar[d, "\Exit(\xi)", swap] \pullbackcorner &
      {\EMeshClosed{n}} \ar[d, "\Mesh{n}"] \\
      {\Delta\ord{2}} \ar[r] &
      {\Exit(\DeltaStrat{2})} \ar[r, "\classify{\xi}", swap] &
      {\BMeshClosed{n}}
    \end{tikzcd}
  \]
  By Lemma~\ref{lem:fct-mesh-closed-1-bundle-flat} the map $\Exit(\xi)$ is a flat categorical fibration.
  Because $\varphi$ is the pullback of $\Exit(\xi)$, it follows that it is a flat categorical fibration as well.
\end{proof}


\begin{proposition}\label{prop:fct-mesh-closed-labelled-cat}
  The functor $\BMeshClosedL{n}{-} : \sSet \to \sSet$ satisfies the following properties:
  \begin{enumerate}
    \item $\BMeshClosedL{n}{\cat{C}} \to \BMesh{n}$ is a categorical fibration when $\cat{C}$ is a quasicategory.
    \item $\BMeshClosedL{n}{-}$ preserves quasicategories.
    \item $\BMeshClosedL{n}{-}$ preserves categorical equivalences between quasicategories.
  \end{enumerate}
\end{proposition}
\begin{proof}
  By Lemma~\ref{lem:fct-mesh-closed-labelled-pointed-flat} the map
  $\EMeshClosed{n} \to \BMeshClosed{n}$ is a flat categorical fibration.
  Therefore, the claims follow by Proposition~\ref{prop:b-poly-flat-cat}.
\end{proof}

\begin{proposition}\label{prop:fct-mesh-closed-labelled-classify}
  Let $\strat{B}$ be a triangulable stratified space and
  $I : \strat{A} \hookrightarrow \strat{B}$
  a closed triangulable stratified subspace.
  Let $\cat{C}$ be a quasicategory.
  Suppose that $\varphi : \Exit(\strat{B}) \to \BMeshClosedL{n}{\cat{C}}$
  is a map and $\xi' : \strat{X}' \to \strat{A}$ a closed $n$-mesh bundle
  with a labelling map $\ell' : \Exit(\strat{X}') \to \cat{C}$
  such that $\classify{\xi'; \ell'} = \varphi \circ \Exit(I)$.
  Then there is a closed $n$-mesh bundle $\xi : \strat{X} \to \strat{B}$ 
  with a labelling map $\ell : \Exit(\strat{X}) \to \cat{C}$
  together with an equivalence   
  $\varphi \simeq \classify{\xi; \ell}$ 
  relative to $\classify{\xi'; \ell'}$.
\end{proposition}
\begin{proof}
  By restriction along the map that forgets the labelling,
  we obtain a map $\psi$ which together with the other given data fits into the diagram:
  \[
    \begin{tikzcd}
    	{\Exit(\strat{A})} &&& {\Exit(\strat{A})} \\
    	& {\BMeshClosedL{n}{\cat{C}}} & {\BMeshClosed{n}} \\
    	{\Exit(\strat{B})} &&& {\Exit(\strat{B})}
    	\arrow[from=2-2, to=2-3]
    	\arrow[equal, from=1-1, to=1-4]
    	\arrow[equal, from=3-1, to=3-4]
    	\arrow[hook, from=1-4, to=3-4]
    	\arrow[hook', from=1-1, to=3-1]
    	\arrow["\varphi"{description}, from=3-1, to=2-2]
    	\arrow["{\classify{\xi';\ell'}}"{description}, from=1-1, to=2-2]
    	\arrow["{\classify{\xi'}}"{description}, from=1-4, to=2-3]
    	\arrow["\psi"{description}, from=3-4, to=2-3]
    \end{tikzcd}
  \]
  The closed $n$-mesh bundle $\xi'$ together with the map $\psi$
  satisfy the conditions of Proposition~\ref{prop:fct-mesh-closed-n-classify},
  and so we obtain a closed $n$-mesh bundle $\xi : \strat{X} \to \strat{B}$
  together with an equivalence $h : \psi \simeq \classify{\xi}$ relative to
  $\classify{\xi'}$.
  Because $\cat{C}$ is a quasicategory,
  by Proposition~\ref{prop:fct-mesh-closed-labelled-cat} the forgetful map
  $\BMeshClosedL{n}{\cat{C}} \to \BMeshClosedL{n}{\cat{C}}$ is a categorical
  fibration.
  Therefore we can lift the equivalence $h$ to an equivalence
  $\varphi \simeq \classify{\xi; \ell}$ relative to $\classify{\xi'; \ell'}$
  where $\ell : \Exit(\strat{X}) \to \cat{C}$ is a labelling functor for
  $\xi$.
\end{proof}

\subsection{Packing for Closed $n$-Meshes}\label{sec:fct-mesh-closed-pack}

\begin{construction}
  We construct a \defn{packing map}
  $\BMeshClosed{n + 1} \to \BMeshClosedL{n}{\BMeshClosed{1}}$.
  A $k$-simplex of $\BMeshClosed{n + 1}$ is a closed $(n + 1)$-mesh bundle
  over $\DeltaStrat{k}$ which we write as the composite $\chi \circ \xi$
  of a closed $1$-mesh bundle $\xi : \strat{X}_{n + 1} \to \strat{X}_n$
  and a closed $n$-mesh bundle $\zeta : \strat{X}_n \to \DeltaStrat{k}$.  
  Then the classifying maps $\classify{\xi}$ and $\classify{\zeta}$
  of the closed mesh bundles together fit into a diagram
  \[
    \begin{tikzcd}[column sep = large, row sep = large]
      {\BMeshClosed{1}} \ar[r, equal] &
      {\BMeshClosed{1}}
      \\
      {\cat{E}} \ar[r, hook, "\simeq"] \ar[d] \pullbackcorner \ar[u] &
      {\Exit(\strat{X}_n)} \ar[r] \ar[d, "\Exit(\zeta)"{description}] \pullbackcorner \ar[u, "\classify{\xi}"'] &
      {\EMeshClosed{n}} \ar[d, "\MeshClosed{n}"] \\
      {\Delta\ord{k}} \ar[r, hook, "\simeq"'] &
      {\Exit(\DeltaStrat{k})} \ar[r, "\classify{\zeta}"'] &
      {\BMeshClosed{n}}
    \end{tikzcd}
  \]
  This diagram then determines a $k$-simplex of $\BMeshClosedL{n}{\BMeshClosed{1}}$.
\end{construction}

\begin{observation}
  Let $\xi : \strat{X} \to \strat{Y}$ be a closed $1$-mesh bundle and
  $\zeta : \strat{Y} \to \strat{B}$ a closed $n$-mesh bundle. Then the packing map
  sends the classifying map of the composite closed $(n + 1)$-mesh bundle $\classify{\zeta \circ \xi} : \Exit(\strat{B}) \to \BMeshClosed{n + 1}$
  to the classifying map $\classify{\zeta; \classify{\xi}}$.
\end{observation}

\begin{lemma}\label{lem:fct-mesh-closed-n-packing-unlabelled}
  The packing map
  $\BMeshClosed{n + 1} \to \BMeshClosedL{n}{\BMeshClosed{1}}$
  is a categorical equivalence.
\end{lemma}
\begin{proof}
  Suppose we have a lifting problem of the form
  \begin{equation}\label{eq:fct-mesh-closed-n-packing-unlabelled}
    \begin{tikzcd}
      {\partial\Delta\ord{k}} \ar[r] \ar[d] &
      {\BMeshClosed{n + 1}} \ar[d] \\
      {\Delta\ord{k}} \ar[r] \ar[ur, dashed] &
      {\BMeshClosedL{n}{\BMeshClosed{1}}}
    \end{tikzcd}
  \end{equation}
  Unpacking the definitions, we are given closed $n$-mesh bundles $\zeta$, $\zeta'$ and a closed $1$-mesh bundle $\xi'$ which together fit into the diagram
  \[
    \begin{tikzcd}
      {X_{n + 1}'} \ar[d, "\xi'"'] \\
      {X_n'} \ar[r] \ar[d, "\zeta'"'] \pullbackcorner &
      {X_n} \ar[d, "\zeta"] \\
      {\partial \DeltaStrat{k}} \ar[r] &
      {\DeltaStrat{k}}
    \end{tikzcd}
  \]
  We have pullback squares
  \[
    \begin{tikzcd}
      {\cat{E}'} \ar[r, hook] \ar[d] \pullbackcorner &
      {\Exit(\strat{X}_n')} \ar[d] \ar[r] \pullbackcorner &
      {\EMeshClosed{n}} \ar[d] \\
      {\partial\Delta\ord{k}} \ar[r, "\simeq"', hook] &
      {\Exit(\partial\DeltaStrat{k})} \ar[r, "\classify{\zeta'}"'] &
      {\BMeshClosed{n}}
    \end{tikzcd}
    \qquad
    \begin{tikzcd}
      {\cat{E}} \ar[r, hook] \ar[d] \pullbackcorner &
      {\Exit(\strat{X}_n)} \ar[d, "\simeq"] \ar[r] \ar[d] \pullbackcorner &
      {\EMeshClosed{n}} \ar[d] \\
      {\Delta\ord{k}} \ar[r, "\simeq"', hook] &
      {\Exit(\DeltaStrat{k})} \ar[r, "\classify{\zeta}"'] &
      {\BMeshClosed{n}}
    \end{tikzcd}
  \]
  By Lemma~\ref{lem:fct-mesh-closed-labelled-pointed-flat} the map $\EMeshClosed{n} \to \BMeshClosed{n}$
  is a flat categorical fibration.
  Since the inclusions $\partial \Delta\ord{k} \hookrightarrow \Exit(\partial\DeltaStrat{k})$
  and $\Delta\ord{k} \to \Exit(\DeltaStrat{k})$ are categorical equivalences,
  it then follows by Lemma~\ref{lem:b-poly-flat-cat-pullback-equiv} that the
  inclusions
  $\cat{E}' \hookrightarrow \Exit(\strat{X}_n')$ and
  $\cat{E} \hookrightarrow \Exit(\strat{X}_n)$ are categorical equivalences as well.
  We have a labelling functor $\ell : \cat{E} \to \BMeshClosed{1}$ which fits
  into the diagram
  \[
    \begin{tikzcd}
      {\cat{E}'} \ar[r, hook] \ar[d, hook] &
      {\cat{E}} \ar[d, "\ell"] \\
      {\Exit(\strat{X}')} \ar[r, "\classify{\xi'}"'] &
      {\BMeshClosed{1}}
    \end{tikzcd}
  \]
  Since $\BMeshClosed{1}$ is a quasicategory, we can then find a lift
  \[
    \begin{tikzcd}
      {\Exit(\strat{X}') \sqcup_{\cat{E}'} \cat{E}}  \ar[r] \ar[d, hook, "\simeq"'] &
      {\BMeshClosed{1}} \\
      {\Exit(\strat{X})} \ar[ur, dashed, "\varphi"'] &
      {}
    \end{tikzcd}
  \]
  By Corollary~\ref{cor:fct-mesh-closed-1-total-triangulable} we have that $\strat{X}_n$ is triangulable and
  $\strat{X}_n' \subseteq \strat{X}_n$ a closed subspace.
  Then by Proposition~\ref{prop:fct-mesh-closed-1-classify}
  we can find a closed $1$-mesh bundle
  $\xi : \strat{X}_{n + 1} \to \strat{X}_n$ extending $\xi'$ such that
  $\classify{\xi} \simeq \varphi$ while preserving the restriction
  $\classify{\xi'}$.
  Then the closed $(n + 1)$-mesh bundle $\zeta \circ \xi$ determines a lift
  $\Delta\ord{k} \to \BMeshClosed{n + 1}$ which makes the upper left
  triangle of the lifting problem~\ref{eq:fct-mesh-closed-n-packing-unlabelled} commute.
  The lower right triangle commutes up to homotopy given by the path $\classify{\xi} \simeq \varphi$.
\end{proof}

\begin{cor}
  The packing map $\pack_n : \BMeshClosed{n + 1} \to \BMeshClosedL{n}{\BMeshClosed{1}}$
  is a split monomorphism.
\end{cor}
\begin{proof}
  The map $\pack_n$ is a categorical equivalence and a cofibration,
  and $\BMeshClosed{n + 1}$ is a quasicategory. Therefore there exists a lift in the diagram
  \[
    \begin{tikzcd}
      {\BMeshClosed{n + 1}} \ar[r, "\id"] \ar[d, "\pack_n"'] &
      {\BMeshClosed{n + 1}}  \\
      {\BMeshClosedL{n}{\BMeshClosed{1}}} \ar[ur, dashed] &
      {}
    \end{tikzcd}
  \]
  and therefore a retraction for $\pack_n$.
\end{proof}

\begin{proposition}\label{prop:fct-mesh-closed-n-pack}
  For every quasicategory $\cat{C}$ there is a categorical equivalence
  \[
    \begin{tikzcd}
      {\BMeshClosedL{n + 1}{\cat{C}}} \ar[r, "\simeq", dashed] \ar[d] &
      {\BMeshClosedL{n}{\BMeshClosedL{1}{\cat{C}}}} \ar[d] \\
      {\BMeshClosed{n + 1}} \ar[r, "\pack_n"'] &
      {\BMeshClosedL{n}{\BMeshClosed{1}}}
    \end{tikzcd}
  \]
\end{proposition}
\begin{proof}
  Let us write $\EMeshClosedL{n}{\BMeshClosed{1}}$ for the simplicial set obtained by the pullback
  \[
    \begin{tikzcd}
      {\EMeshClosedL{n}{\BMeshClosed{1}}} \ar[r] \ar[d] \pullbackcorner &
      {\BMeshClosedL{n}{\BMeshClosed{1}}} \ar[d] \\
      {\EMeshClosed{n}} \ar[r] &
      {\BMeshClosed{n}}
    \end{tikzcd}
  \]
  The packing map fits into a pullback square
  \begin{equation}\label{eq:fct-mesh-n-pack:pullback}
    \begin{tikzcd}
      {\EMeshClosed{n + 1}} \ar[r, hook] \ar[d] \pullbackcorner &
      {\EMeshClosedL{n}{\BMeshClosed{1}} \times_{\BMeshClosed{1}} \EMeshClosed{1}} \ar[d] \\
      {\BMeshClosed{n + 1}} \ar[r, hook] &
      {\BMeshClosedL{n}{\BMeshClosed{1}}}
    \end{tikzcd}
  \end{equation}
  which can be verified by unrolling the definitions.
  Then via Construction~\ref{con:b-poly-functorial}
  the pullback square~(\ref{eq:fct-mesh-n-pack:pullback})
  induces a map between the induced polynomial functors.
  By Lemma~\ref{lem:b-poly-composite} the polynomial functor
  induced by the map on the right of~(\ref{eq:fct-mesh-n-pack:pullback}) is
  the composite $\BMeshL{n}{\BMeshL{1}{-}}$.
  Then by Lemma~\ref{lem:b-poly-functorial-equiv} the map is a categorical equivalence.
\end{proof}

\begin{cor}\label{cor:fct-mesh-closed-n-pack-iterate}
  Let $\cat{C}$ be a quasicategory. Then there is an equivalence
  \[
    \BMeshClosedL{n}{\cat{C}} \simeq
    \underbrace{(\BMeshClosed{1} \circ \cdots \circ \BMeshClosed{1})}_{\textup{$n$ times}}(\cat{C})
  \]
\end{cor}


\subsection{Open Meshes}\label{sec:fct-mesh-open-1}

Open meshes are the dual of closed meshes, both geometrically and algebraically.
Where closed meshes model stratifications that are generalised pasting diagrams,
open meshes represent generalised string diagrams.

\begin{definition}
  An \defn{open $1$-mesh} $\strat{M}$ is a finite stratification of $\R$ whose strata are
  either \defn{singular}, consisting of just a point,
  or \defn{regular}, consisting of an open interval.
  A point in $\strat{M}$ is singular or regular if it is contained in a stratum that is
  singular or regular, respectively.
\end{definition}

\begin{definition}\label{def:fct-mesh-open-1-bundle}
  An \defn{open $1$-mesh bundle} is a $1$-framed
  stratified bundle $f : \strat{M} \to \strat{B}$
  that satisfies the following conditions:
  \begin{enumerate}
    \item $f$ is a stratified fibre bundle.
    \item For every $b \in \strat{B}$ the fibre $f^{-1}(b)$ is an open $1$-mesh.
    \item The regular points in all fibres form an open subset of $\strat{M}$.
  \end{enumerate}  
\end{definition}

\begin{construction}
  We pick any continuous and order-preserving isomorphism between the open interval
  $\intOO{-1, -1}$ and $\R$,
  for instance $x \mapsto \tan(\tfrac{\pi}{2} x)$ or any piecewise linear approximation of that map.
  From this isomorphism we can derive an induced embedding $\R \hookrightarrow \intCC{-1, 1}$ whose
  image consists of the open interval $\intOO{-1, 1} \subseteq \intCC{-1, 1}$.
  Suppose that $f : \strat{M} \to \strat{B}$ is an open $1$-mesh bundle.
  Then there exists a closed $1$-mesh bundle $\xi : \strat{X} \to \strat{B}$ such that
  $\unstrat(\xi)$ is the projection $\intCC{-1, 1} \times \unstrat(\strat{B}) \to \unstrat(\strat{B})$,
  and we have a pullback square of stratified spaces
  \[
    \begin{tikzcd}
      {\strat{M}} \ar[r, hook] \ar[d] \pullbackcorner &
      {\strat{X}} \ar[d] \\
      {\R \times \unstrat(\strat{B})} \ar[r, hook] &
      {\intCC{-1, 1} \times \unstrat(\strat{B})}
    \end{tikzcd}
  \]
  We call $\xi$ the \defn{compactification} of $f$.
  Any different choice of continuous order-preserving isomorphism $\intOO{-1, 1} \cong \R$
  yields a compactification that is framed isomorphic.
\end{construction}

We use the compactification of open $1$-mesh bundles to transfer most of
the theory of closed meshes to open meshes.

\begin{lemma}\label{lem:fct-mesh-open-1-bundle-flat}
  Let $f : \strat{M} \to \strat{B}$ be an open $1$-mesh bundle.
  Then the induced map $\Exit(f)$ is a flat categorical fibration.
\end{lemma}
\begin{proof}
  When $\xi : \strat{X} \to \strat{B}$ is the compactification of $f$
  we have constructible embedding $i : \strat{M} \hookrightarrow \strat{X}$,
  and so $\Exit(i)$ must be a categorical fibration.
  To see that $\Exit(i)$ is also flat, we consider a diagram
  \[
    \begin{tikzcd}
      {\Delta\ord{1}} \ar[r] \ar[d, hook, "{\langle 0, 2 \rangle}"'] &
      {\Exit(\strat{M})} \ar[hook, d, "\Exit(i)"] \\
      {\Delta\ord{2}} \ar[r, "\sigma"'] \ar[ur, dashed] &
      {\Exit(\strat{X})}
    \end{tikzcd}
  \]  
  There are no exit paths in $\strat{X}$ which start in $\strat{M}$ but end in the
  boundary $\strat{X} \setminus \strat{M}$.
  Therefore $\sigma(0) \in \strat{X}$ being in $\strat{M}$ implies that also
  $\sigma(1)$ is in $\strat{M}$.
  Since $\strat{M}$ is a constructible subspace of $\strat{X}$,
  the stratified $2$-simplex $\sigma$ must then factor uniquely through $\strat{M}$.
  Hence, $\Exit(i)$ is flat.
  The induced map $\Exit(\xi)$ is a flat categorical fibration by Lemma~\ref{lem:fct-mesh-closed-1-bundle-flat}, and so
  the composite $\Exit(f) = \Exit(\xi) \circ \Exit(i)$ is a flat categorical fibration as well.
\end{proof}

\begin{definition}
  We denote by $\BMeshOpen{1}$ the simplicial set whose $k$-simplices are
  the open $1$-mesh bundles $f : \strat{M} \to \DeltaStrat{k}$.
  An order-preserving map $\ord{k'} \to \ord{k}$ acts by pullback along the induced
  map $\DeltaStrat{k'} \to \DeltaStrat{k}$.
  Compactification of open $1$-mesh bundles defines a map of simplicial sets
  $\BMeshOpen{1} \to \BMeshClosed{1}$.
\end{definition}

\begin{definition}\label{def:basic-strict-int}
  The category $\StrictInt$ of \defn{strict intervals} is the subcategory of
  $\FinOrd$ whose objects are finite ordinals $\ord{n}$ with $n > 0$ and whose maps
  are order-preserving maps $\alpha : \ord{n} \to \ord{m}$ that are endpoint preserving,
  i.e. satisfy $\alpha(0) = 0$ and $\alpha(n) = m$.
\end{definition}

\begin{construction}
  We let $\BMeshClosedCB{1}$ be the simplicial subset of $\BMeshClosed{1}$
  of \defn{continuously bounded} closed $1$-meshes, defined by the pullback
  of simplicial sets
  \[
    \begin{tikzcd}
      {\BMeshClosedCB{1}} \ar[r, hook] \ar[d] \pullbackcorner &
      {\BMeshClosed{1}} \ar[d, "\sing"] \\
      {\StrictInt^\op} \ar[r, hook] &
      {\FinOrd^\op}
    \end{tikzcd}
  \]
  The induced map $\BMeshClosedCB{1} \to \StrictInt^\op$ is again a trivial fibration and
  so $\BMeshClosedCB{1}$ is a quasicategory.
  The compactification map $\BMeshOpen{1} \to \BMeshClosed{1}$ factors through
  the inclusion $\BMeshClosedCB{1} \hookrightarrow \BMeshClosed{1}$.
\end{construction}

\begin{lemma}
  The compactification map $\BMeshOpen{1} \to \BMeshClosedCB{1}$ is a split monomorphism.
  In particular $\BMeshOpen{1}$ is a retract of $\BMeshClosedCB{1}$ and therefore quasicategory.
\end{lemma}
\begin{proof}
  We define the retraction $\BMeshClosedCB{1} \to \BMeshOpen{1}$ as follows:
  Given a continuously bounded closed $1$-mesh bundle $\xi : \strat{X} \to \DeltaStrat{k}$,
  we perform the following three steps:
  \begin{enumerate}
    \item Linearly rescale the framing to be fibrewise supported on the closed interval $\intCC{-1, 1}$.
    \item Remove the first and last singular stratum in each fibre.
    \item Rescale using the chosen continuous order-preserving isomorphism $\intOO{-1, 1} \cong \R$.
  \end{enumerate}
  The result is an open $1$-mesh bundle over $\DeltaStrat{k}$ and therefore a $k$-simplex of $\BMeshOpen{1}$.
  
  When we start with an open $1$-mesh bundle and compactify, the resulting closed $1$-mesh bundle is already supported on $\intCC{-1, 1}$ in each fibre and so the rescaling does not affect it.
  We then remove the strata that compactification has added, and undo the rescaling from $\R$ to $\intOO{-1, 1}$ performed by compactification. Therefore, we end up with the open $1$-mesh bundle that we started with.
\end{proof}

\begin{lemma}\label{lem:fct-mesh-open-1-reg-left}
  Let $f : \strat{M} \to \strat{B}$ be an open $1$-mesh bundle.
  Then the restriction
  \[
    \begin{tikzcd}[column sep = large]
      {\Exit(\reg(\strat{M}))} \ar[r, hook] &
      {\Exit(\strat{M})} \ar[r, "\Exit(f)"] &
      {\Exit(\strat{B})}
    \end{tikzcd}
  \]
  is a left fibration.
\end{lemma}
\begin{proof}
  By Lemma~\ref{lem:fct-mesh-open-1-bundle-flat}
  we have that $\Exit(f)$ is a categorical fibration.
  Because $\reg(\strat{M})$ is a constructible subspace of $\strat{M}$,
  the inclusion map induces a categorical fibration $\Exit(\reg(\strat{M})) \to \Exit(\strat{M})$.
  Therefore, the composite
  $\Exit(\reg(\strat{M})) \to \Exit(\strat{B})$
  is a categorical fibration as well.

  Via the framing we interpret $\unstrat(\strat{M})$ as an open subspace of $\R \times \unstrat(\strat{B})$ and $\unstrat(f)$ as the projection.
  Let $\gamma : \DeltaStrat{1} \to \strat{B}$ be an exit path and $(h, b) \in \strat{M}$ a regular point with $f(h, b) = b = \gamma(0)$.
  The subspace $\reg(\strat{M})$ of regular points is open in $\strat{M}$,
  and so it follows that there exists an $\eps > 0$ so that the map $\hat{\gamma} : \DeltaStratR{1} \cap \intCO{0, \eps} \to \strat{M}$ defined by $\hat{\gamma}(t) = (h, \gamma(t))$ is a well-defined stratified map.
  Then via the trivialisation of $f$ over the stratum of $\strat{B}$ containing $\gamma(1)$, we can extend $\hat{\gamma}$ to a lift $\hat{\gamma} : \DeltaStrat{1} \to \strat{M}$ of $\gamma$ starting at $(b, h)$.
  The space of such lifts is contractible since they necessarily all land within the same regular stratum.
  Therefore, it follows that $\Exit(\reg(\strat{M})) \to \Exit(\strat{B})$ is a left fibration.
\end{proof}

\begin{construction}
  Suppose that $f : \strat{M} \to \strat{B}$ is an open $1$-mesh bundle,
  then by Lemma~\ref{lem:fct-mesh-open-1-reg-left} there is a diagram
  \[
    \begin{tikzcd}
      {\Exit(\reg(\strat{M}))} \ar[r] \ar[d] &
      {\Space_*} \ar[d] \\
      {\Exit(\strat{B})} \ar[r] &
      {\Space}
    \end{tikzcd}
  \]
  where $\Space_*$ is the $\infty$-category of pointed spaces.
  The map $\Exit(\strat{B}) \to \Space$ sends a point $b \in \strat{B}$ to the
  subspace of regular points in the fibre $f^{-1}(b)$.
  Using the functor $\pi_0 : \Space \to \Set$ which sends a space to the set of its connected components, we then get a diagram
  \[
    \begin{tikzcd}
      {\Exit(\reg(\strat{M}))} \ar[r] \ar[d] &
      {\Set_*} \ar[d] \\
      {\Exit(\strat{B})} \ar[r] &
      {\Set}
    \end{tikzcd}
  \]
  The number of regular strata over any point $b \in \strat{B}$ is finite but never zero,
  and the regular strata in $f^{-1}(b)$ inherit an ordering from the $1$-framing of $f$.
  Therefore, we obtain an induced diagram
  \[
    \begin{tikzcd}
      {\Exit(\reg(\strat{M}))} \ar[r] \ar[d] &
      {\FinOrd_*} \ar[d] \\
      {\Exit(\strat{B})} \ar[r, "\reg(f)", swap] &
      {\FinOrd}
    \end{tikzcd}
  \]
\end{construction}

\begin{observation}
  For any finite non-empty ordinal $\ord{n}$ we can equip the set of order-preserving
  maps $\ord{n} \to \ord{1}$ with the pointwise order; the result is a finite total
  order with at least two elements.
  An order-preserving map $\ord{n} \to \ord{m}$ acts upon a map $\ord{m} \to \ord{1}$
  by precomposition, preserving the pointwise order as well as the maximum and minimum element.
  This defines a functor $\FinOrd \to \StrictInt^\op$ which moreover is an isomorphism.
\end{observation}

\begin{example}
  We can interpret the functor $\FinOrd \to \StrictInt^\op$ as sending an ordinal $\ord{n}$
  to the totally ordered set of gaps in between the elements of $\ord{n}$, including a gap at the beginning and end.
  For example, the functor sends $\ord{3}$ to the ordinal $\ord{4}$:
  \[
    \begin{tikzpicture}[scale = 0.5]
      \node at (1, 0) {0};
      \node at (3, 0) {1};
      \node at (5, 0) {2};
      \node at (7, 0) {3};

      \draw[BrickRed, thick] (0, -0.5) -- +(0, 1) node[anchor=south]{0};
      \draw[BrickRed, thick] (2, -0.5) -- +(0, 1) node[anchor=south]{1};
      \draw[BrickRed, thick] (4, -0.5) -- +(0, 1) node[anchor=south]{2};
      \draw[BrickRed, thick] (6, -0.5) -- +(0, 1) node[anchor=south]{3};
      \draw[BrickRed, thick] (8, -0.5) -- +(0, 1) node[anchor=south]{4};
    \end{tikzpicture}
  \]
\end{example}

\begin{lemma}\label{lem:fct-mesh-open-closed-reg-sing}
  The map $\reg : \BMeshOpen{1} \to \FinOrd$ fits into the diagram
  \[
    \begin{tikzcd}
      {\BMeshOpen{1}} \ar[r, hook] \ar[d, "\reg"'] &
      {\BMeshClosedCB{1}} \ar[r, hook] \ar[d] \pullbackcorner &
      {\BMeshClosed{1}} \ar[d, "\sing^\op"] \\
      {\FinOrd} \ar[r, "\cong"'] &
      {\StrictInt^\op} \ar[r, hook]  &
      {\FinOrd^\op}
    \end{tikzcd}
  \]
  In particular the map $\reg : \BMeshOpen{1} \to \FinOrd$ is a retract of the map
  $\BMeshClosedCB{1} \to \StrictInt^\op$ and therefore a trivial fibration.
\end{lemma}
\begin{proof}
  Suppose that $f : \strat{M} \to \DeltaStrat{k}$ is an open $1$-mesh bundle
  and $\xi : \strat{X} \to \DeltaStrat{k}$ the compactification of $f$.
  Let $i \in \ord{k}$ then $\sing(\xi)(i)$ is the set of singular strata of
  $\xi$ over $i \in \DeltaStrat{k}$, ordered via the framing.
  Since the regular strata of $\xi$ correspond exactly to the regular strata of
  $f$, we have that $\reg(f)(i)$ classifies the set of regular strata of $\xi$
  over $i \in \DeltaStrat{k}$ ordered via the framing.

  For every singular stratum $s \in \sing(\xi)(i)$ we have an order-preserving
  map $\alpha^i_s : \reg(f)(i) \to \ord{1}$ which sends a regular stratum to
  $0$ exactly when it is below the singular stratum $s$ in the fibrewise ordering.
  The assignment $s \mapsto \alpha^i_s$ is an order-preserving isomorphism between
  $\sing(\xi)(i)$ and the poset of order-preserving maps $\reg(f)(i) \to \ord{1}$.

  Any section of $\xi$ which only intersects singular strata cuts $\strat{X}$
  into two parts, above and below the section.
  Together with the previous observation, it follows that the duality isomorphism
  $\FinOrd \cong \StrictInt^\op$ sends $\reg(f)$ to $\sing(\xi)$.
\end{proof}

\begin{observation}
  We have a span of equivalences
  \[
    \begin{tikzcd}[column sep = large]
      {\BMeshOpen{1}} \ar[r, "\reg"] &
      {\FinOrd} &
      {(\BMeshClosed{1})^\op} \ar[l, "\sing^\op"']
    \end{tikzcd}
  \]
  which induces an equivalence between $\BMeshOpen{1}$
  and the opposite $(\BMeshClosed{1})^\op$.
  This equivalence implements a form of Poincar\'e duality between open and closed $1$-meshes.
  We will see that this equivalence extends to higher-dimensional meshes as well.
\end{observation}

At this point the rest of the theory of closed meshes transfers entirely analogously.
Instead of individually reproving each step, we summarise the results as follows.

\begin{definition}\label{def:fct-mesh-open-n}
  An open $n$-mesh bundle is a sequence of open $1$-mesh bundles
  \begin{equation}\label{eq:fct-mesh-open-n}
    \begin{tikzcd}
      {\strat{M}_n} \ar[r] &
      {\strat{M}_{n - 1}} \ar[r] &
      \cdots \ar[r] &
      {\strat{M}_1} \ar[r] &
      {\strat{M}_0}
    \end{tikzcd}
  \end{equation}
  We denote by $\BMeshOpen{n}$ the simplicial set whose $k$-simplices are open
  $n$-mesh bundles over $\DeltaStrat{k}$.
  The $k$-simplices of the simplicial set $\EMeshOpen{n}$ are open $n$-mesh bundles
  together with a section.
  There is a projection map $\EMeshOpen{n} \to \BMeshOpen{n}$ which forgets the section
  and a truncation map $\BMeshOpen{n + 1} \to \BMeshOpen{n}$ which cuts off the last $1$-mesh bundle.
  We write $\BMeshOpenL{n}{-}$ for the polynomial functor on simplicial sets induced by
  the map $\EMeshOpen{n} \to \BMeshOpen{n}$.
\end{definition}

\begin{observation}
  The truncation map $\BMeshOpen{n + 1} \to \BMeshOpen{n}$ is a categorical fibration,
  and therefore for all $n \geq 0$ the simplicial set $\BMeshOpen{n}$ is a quasicategory.
  The forgetful map $\EMeshOpen{n} \to \BMeshOpen{n}$ is a flat categorical fibration,
  and so the simplicial set $\EMeshOpen{n}$ is a quasicategory as well.
  For each open $n$-mesh bundle $f : \strat{M} \to \strat{B}$ there is a 
  classifying map $\classify{f} : \Exit(\strat{B}) \to \BMeshOpen{n}$
  which fits into the pullback square
  \[
    \begin{tikzcd}
      {\Exit(\strat{M})} \ar[r] \ar[d, "\Exit(f)"'] \pullbackcorner &
      {\EMeshOpen{n}} \ar[d] \\
      {\Exit(\strat{B})} \ar[r, "\classify{f}"'] &
      {\BMeshOpen{n}}
    \end{tikzcd}
  \]
  Then for every quasicategory $\cat{C}$ the map
  $\BMeshOpenL{n}{\cat{C}} \to \BMeshOpen{n}$ is a categorical fibration,
  and $\BMeshOpen{n}(-)$ preserves categorical equivalences between quasicategories.
  For any simplicial set $C$ there is a packing map
  $\BMeshOpenL{n + 1}{C} \to \BMeshOpenL{n}{\BMeshOpenL{1}{C}}$
  which is a categorical equivalence when $C = \cat{C}$ is a quasicategory.
\end{observation}

\subsection{Meshed Grids}\label{sec:fct-mesh-grid}

The functor $\sing : \BMeshClosed{1} \to \FinOrd^\op$ is a trivial fibration
and therefore has a section $\gridMeshClosed^1 : \FinOrd^\op \to \BMeshClosed{1}$
which sends an ordinal $\ord{i}$ to the closed $1$-mesh with $i + 1$ singular
strata. Similarly, for open meshes, the functor $\reg : \BMeshOpen{1} \to \FinOrd$
has an essentially unique section $\gridMeshOpen^1 : \FinOrd \to \BMeshOpen{1}$.
Using the stacked product of framed stratified bundles, defined below, we can
use these sections to obtain the grid mesh functors $\gridMeshClosed^n : \FinOrd^{n, \op} \to \BMeshClosed{n}$ and $\gridMeshOpen^n : \FinOrd^n \to \BMeshOpen{n}$.

\begin{construction}\label{con:fct-stacked-product}
  Suppose that $p_0 : \strat{E}_0 \to \strat{B}$ is an $n$-framed stratified bundle
  and $p_1 : \strat{E}_1 \to \strat{B}$ is an $m$-framed stratified bundle.
  Let $p_1^* p_0 : \strat{E}_0 \times_{\strat{B}} \strat{E}_1 \to \strat{E}_1$ be the $n$-framed
  stratified bundle defined by the pullback
  \[
    \begin{tikzcd}
      {\strat{E}_0 \times_{\strat{B}} \strat{E}_1} \ar[r] \ar[d, "p_1^* p_0"'] \pullbackcorner &
      {\strat{E}_0} \ar[d, "p_0"] \\
      {\strat{E}_1} \ar[r, "p_1"'] &
      {\strat{B}}
    \end{tikzcd}
  \]
  Then the \defn{stacked product} $p_0 \triangleleft p_1$ of $p_0$ and $p_1$ is the composite
  \[
    \begin{tikzcd}
      {\strat{E}_0 \times_{\strat{B}} \strat{E}_1} \ar[r, "p_1^* p_0"] &
      {\strat{E}_1} \ar[r, "p_1"] &
      {\strat{B}}
    \end{tikzcd}
  \]
  with the composite $(n + m)$-framing.
\end{construction}


\begin{lemma}
  Let $p_0 : \strat{E}_0 \to \strat{B}$ be a closed $n$-mesh bundle
  and $p_1 : \strat{E}_1 \to \strat{B}$ a closed $m$-mesh bundle.
  Then the stacked product $p_0 \triangleleft p_1$ is a closed $(n + m)$-mesh bundle.
  Moreover the analogous claim holds for open mesh bundles.
\end{lemma}
\begin{proof}
  By Lemma~\ref{lem:fct-mesh-closed-n-pullback} the pullback
  $p_1^* p_0$ with the induced $n$-framing is again a closed $n$-mesh bundle.
  Then the result follows since the framed composite of a closed $n$-mesh bundle
  followed by a closed $m$-mesh bundle is a closed $(n + m)$-mesh bundle.  
\end{proof}

\begin{construction}\label{con:fct-mesh-grid-closed}
  Suppose that $\gridMeshClosed^1 : \FinOrd^\op \to \BMeshClosed{1}$ is a section
  of the equivalence $\sing : \BMeshClosed{n} \to \FinOrd^\op$.
  Then via the stacked product of closed $n$-meshes, we can define
  the \defn{closed grid mesh} functor
  $\gridMeshClosed^n : \FinOrd^{n, \op} \to \BMeshClosed{n}$.
  A $k$-simplex $\sigma : \Delta\ord{k} \to \FinOrd^{n, \op}$ of the $n$-fold
  product of $\FinOrd^\op$ with itself
  consists of a
  family of $k$-simplices $\sigma_i : \Delta\ord{k} \to \FinOrd^\op$ for
  all $1 \leq i \leq k$.
  Then $\gridMeshClosed^n$ sends $\sigma$ to the closed $n$-mesh bundle
  over $\DeltaStrat{k}$
  obtained as the stacked product
  \[
    \gridMeshClosed^n(\sigma_1, \ldots, \sigma_n) :=
    \gridMeshClosed^n(\sigma_1) \triangleleft
    \cdots \triangleleft
    \gridMeshClosed^n(\sigma_n)
  \]
\end{construction}

\begin{example}
  The name \textit{grid mesh} is justified by the following examples:
  \begin{align*}
    \gridMeshClosed^2 \ord{0, 0} &\simeq \quad
    \begin{tikzpicture}[baseline=(current bounding box.center)]
      \node[mesh-vertex] at (0, 0) {};
    \end{tikzpicture}
    &
    \gridMeshClosed^2 \ord{1, 0} &\simeq \quad
    \begin{tikzpicture}[baseline=(current bounding box.center)]
      \draw[mesh-stratum] (0, 0) -- (1, 0);
      \node[mesh-vertex] at (0, 0) {};
      \node[mesh-vertex] at (1, 0) {};
    \end{tikzpicture}
    &
    \gridMeshClosed^2 \ord{2, 0} &\simeq \quad
    \begin{tikzpicture}[baseline=(current bounding box.center)]
      \draw[mesh-stratum] (0, 0) -- (2, 0);
      \node[mesh-vertex] at (0, 0) {};
      \node[mesh-vertex] at (2, 0) {};
      \node[mesh-vertex] at (1, 0) {};
    \end{tikzpicture} 
    \\[5mm]
    \gridMeshClosed^2 \ord{0, 1} &\simeq \quad
    \begin{tikzpicture}[baseline=(current bounding box.center)]
      \draw[mesh-stratum] (0, 0) -- (0, 1);
      \node[mesh-vertex] at (0, 0) {};
      \node[mesh-vertex] at (0, 1) {};
    \end{tikzpicture}
    &
    \gridMeshClosed^2 \ord{1, 1} &\simeq \quad
    \begin{tikzpicture}[baseline=(current bounding box.center)]
      \fill[mesh-background] (0, 0) rectangle (1, 1);
      \draw[mesh-stratum] (0, 0) rectangle +(1, 1);
      \node[mesh-vertex] at (0, 0) {};
      \node[mesh-vertex] at (0, 1) {};
      \node[mesh-vertex] at (1, 1) {};
      \node[mesh-vertex] at (1, 0) {};
    \end{tikzpicture}
    &
    \gridMeshClosed^2 \ord{2, 1} &\simeq \quad
    \begin{tikzpicture}[baseline=(current bounding box.center)]
      \fill[mesh-background] (0, 0) rectangle (2, 1);
      \draw[mesh-stratum] (0, 0) rectangle +(1, 1);
      \draw[mesh-stratum] (1, 0) rectangle +(1, 1);
      \node[mesh-vertex] at (0, 0) {};
      \node[mesh-vertex] at (0, 1) {};
      \node[mesh-vertex] at (1, 1) {};
      \node[mesh-vertex] at (2, 1) {};
      \node[mesh-vertex] at (2, 0) {};
      \node[mesh-vertex] at (1, 0) {};
    \end{tikzpicture} 
    \\[5mm]
    \gridMeshClosed^2 \ord{0, 2} &\simeq \quad
    \begin{tikzpicture}[baseline=(current bounding box.center)]
      \draw[mesh-stratum] (0, 0) -- (0, 2);
      \node[mesh-vertex] at (0, 0) {};
      \node[mesh-vertex] at (0, 1) {};
      \node[mesh-vertex] at (0, 2) {};
    \end{tikzpicture}
    &
    \gridMeshClosed^2 \ord{1, 2} &\simeq \quad
    \begin{tikzpicture}[baseline=(current bounding box.center)]
      \fill[mesh-background] (0, 0) rectangle (1, 2);
      \draw[mesh-stratum] (0, 0) rectangle +(1, 1);
      \draw[mesh-stratum] (0, 1) rectangle +(1, 1);
      \node[mesh-vertex] at (0, 0) {};
      \node[mesh-vertex] at (0, 1) {};
      \node[mesh-vertex] at (1, 1) {};
      \node[mesh-vertex] at (1, 0) {};
      \node[mesh-vertex] at (0, 2) {};
      \node[mesh-vertex] at (1, 2) {};
    \end{tikzpicture}
    &
    \gridMeshClosed^2 \ord{2, 2} &\simeq \quad
    \begin{tikzpicture}[baseline=(current bounding box.center)]
      \fill[mesh-background] (0, 0) rectangle (2, 2);
      \draw[mesh-stratum] (0, 0) rectangle +(1, 1);
      \draw[mesh-stratum] (1, 0) rectangle +(1, 1);
      \draw[mesh-stratum] (0, 1) rectangle +(1, 1);
      \draw[mesh-stratum] (1, 1) rectangle +(1, 1);
      \node[mesh-vertex] at (0, 0) {};
      \node[mesh-vertex] at (0, 1) {};
      \node[mesh-vertex] at (1, 1) {};
      \node[mesh-vertex] at (2, 1) {};
      \node[mesh-vertex] at (2, 0) {};
      \node[mesh-vertex] at (1, 0) {};
      \node[mesh-vertex] at (0, 2) {};
      \node[mesh-vertex] at (1, 2) {};
      \node[mesh-vertex] at (2, 2) {};
    \end{tikzpicture} 
  \end{align*}
\end{example}

\begin{construction}\label{con:fct-mesh-grid-open}
  Suppose that $\gridMeshOpen^1 : \FinOrd^\op \to \BMeshClosed{1}$ is a section
  of the equivalence $\reg : \BMeshOpen{n} \to \FinOrd$.
  Then via the stacked product of open $n$-meshes, we can define
  the \defn{open grid mesh} functor
  $\gridMeshOpen^n : \FinOrd^{n} \to \BMeshOpen{n}$.
  A $k$-simplex $\sigma : \Delta\ord{k} \to \FinOrd^{n}$ of the $n$-fold
  product of $\FinOrd$ with itself
  consists of a
  family of $k$-simplices $\sigma_i : \Delta\ord{k} \to \FinOrd$ for
  all $1 \leq i \leq k$.
  Then $\gridMeshOpen^n$ sends $\sigma$ to the open $n$-mesh bundle
  over $\DeltaStrat{k}$
  obtained as the stacked product
  \[
    \gridMeshOpen^n(\sigma_1, \ldots, \sigma_n) :=
    \gridMeshOpen^n(\sigma_1) \triangleleft
    \cdots \triangleleft
    \gridMeshOpen^n(\sigma_n)
  \]
\end{construction}
\begin{example}
  The open grid meshes look like this:
  \begin{align*}
    \gridMeshOpen^2 \ord{0, 0} &\simeq \quad
    \begin{tikzpicture}[scale = 0.5, baseline=(current bounding box.center)]
      \fill[mesh-background] (0, 0) rectangle (3, 3);
    \end{tikzpicture}
    &
    \gridMeshOpen^2 \ord{1, 0} &\simeq \quad
    \begin{tikzpicture}[scale = 0.5, baseline=(current bounding box.center)]
      \fill[mesh-background] (0, 0) rectangle (3, 3);
      \draw[mesh-stratum] (1.5, 0) -- +(0, 3);
    \end{tikzpicture}
    &
    \gridMeshOpen^2 \ord{2, 0} &\simeq \quad
    \begin{tikzpicture}[scale = 0.5, baseline=(current bounding box.center)]
      \fill[mesh-background] (0, 0) rectangle (3, 3);
      \draw[mesh-stratum] (1, 0) -- +(0, 3);
      \draw[mesh-stratum] (2, 0) -- +(0, 3);
    \end{tikzpicture}
    \\[5mm]
    \gridMeshOpen^2 \ord{0, 1} &\simeq \quad
    \begin{tikzpicture}[scale = 0.5, baseline=(current bounding box.center)]
      \fill[mesh-background] (0, 0) rectangle (3, 3);
      \draw[mesh-stratum] (0, 1.5) -- +(3, 0);
    \end{tikzpicture}
    &
    \gridMeshOpen^2 \ord{1, 1} &\simeq \quad
    \begin{tikzpicture}[scale = 0.5, baseline=(current bounding box.center)]
      \fill[mesh-background] (0, 0) rectangle (3, 3);
      \draw[mesh-stratum] (1.5, 0) -- +(0, 3);
      \draw[mesh-stratum] (0, 1.5) -- +(3, 0);
      \node[mesh-vertex] at (1.5, 1.5) {};
    \end{tikzpicture}
    &
    \gridMeshOpen^2 \ord{2, 1} &\simeq \quad
    \begin{tikzpicture}[scale = 0.5, baseline=(current bounding box.center)]
      \fill[mesh-background] (0, 0) rectangle (3, 3);
      \draw[mesh-stratum] (1, 0) -- +(0, 3);
      \draw[mesh-stratum] (2, 0) -- +(0, 3);
      \draw[mesh-stratum] (0, 1.5) -- +(3, 0);
      \node[mesh-vertex] at (1, 1.5) {};
      \node[mesh-vertex] at (2, 1.5) {};
    \end{tikzpicture}
    \\[5mm]
    \gridMeshOpen^2 \ord{0, 1} &\simeq \quad
    \begin{tikzpicture}[scale = 0.5, baseline=(current bounding box.center)]
      \fill[mesh-background] (0, 0) rectangle (3, 3);
      \draw[mesh-stratum] (0, 1) -- +(3, 0);
      \draw[mesh-stratum] (0, 2) -- +(3, 0);
    \end{tikzpicture}
    &
    \gridMeshOpen^2 \ord{1, 1} &\simeq \quad
    \begin{tikzpicture}[scale = 0.5, baseline=(current bounding box.center)]
      \fill[mesh-background] (0, 0) rectangle (3, 3);
      \draw[mesh-stratum] (1.5, 0) -- +(0, 3);
      \draw[mesh-stratum] (0, 1) -- +(3, 0);
      \draw[mesh-stratum] (0, 2) -- +(3, 0);
      \node[mesh-vertex] at (1.5, 1) {};
      \node[mesh-vertex] at (1.5, 2) {};
    \end{tikzpicture}
    &
    \gridMeshOpen^2 \ord{2, 1} &\simeq \quad
    \begin{tikzpicture}[scale = 0.5, baseline=(current bounding box.center)]
      \fill[mesh-background] (0, 0) rectangle (3, 3);
      \draw[mesh-stratum] (1, 0) -- +(0, 3);
      \draw[mesh-stratum] (2, 0) -- +(0, 3);
      \draw[mesh-stratum] (0, 1) -- +(3, 0);
      \draw[mesh-stratum] (0, 2) -- +(3, 0);
      \node[mesh-vertex] at (1, 1) {};
      \node[mesh-vertex] at (1, 2) {};
      \node[mesh-vertex] at (2, 1) {};
      \node[mesh-vertex] at (2, 2) {};
    \end{tikzpicture}
  \end{align*}
\end{example}

\begin{observation}
  For $n \leq 1$ the grid mesh functors $\gridMeshClosed^1$ and $\gridMeshOpen^1$
  are equivalences. For $n > 1$ the functors are fully faithful, exhibiting
  $\FinOrd^{n, \op}$ and $\FinOrd^n$ as full subcategories of $\BMeshClosed{n}$
  and $\BMeshOpen{n}$.
\end{observation}

\section{Trusses}\label{sec:fct-truss}

In~\S\ref{sec:fct-mesh-closed-1} we have seen that 
the $\infty$-category $\BMeshClosed{1}$ that classifies
closed $1$-mesh bundles is equivalent to the $1$-category $\FinOrd^\op$.
The geometric theory of closed $1$-mesh bundles is therefore captured,
up to equivalence,
by the combinatorics of finite non-empty ordinals and order-preserving maps.
Closed $n$-trusses are the combinatorial objects that extend this correspondence to closed $n$-meshes.
Dually, open $n$-trusses are the combinatorial objects associated with open $n$-meshes.

Open trusses are a generalisation of the zigzag construction of~\cite{zigzag-contraction, zigzag-normalisation},
which admits a computable representation and serves as the core of a graphical proof assistant for higher categories.
The data structures and algorithms for the zigzag construction straightforwardly
extend to trusses.
To keep the scope of this thesis manageable, we will however not explore the computational aspects further.

\subsection{Closed Trusses}\label{sec:fct-truss-closed}

We begin with closed trusses.
With the intention in mind that the $1$-category $\BTrussClosed{1}$ that classifies closed $1$-trusses should be equivalent to $\BMeshClosed{1}$
while being purely combinatorial in nature,
the equivalence $\sing : \BMeshClosed{1} \to \FinOrd^\op$
allows us to simply define $\BTrussClosed{1} := \FinOrd^\op$.
We can then define the counterpart $\ETrussClosed{1}$ of $\EMeshClosed{1}$ as follows.

\begin{definition}\label{def:truss-closed-1}
	The objects of the $1$-category $\ETrussClosed{1}$ are \textit{singular objects} of
	the form $\trussCS{i} \ord{n}$ for natural numbers $0 \leq i \leq n$ and \textit{regular objects} $\trussCR{i}
		\ord{n}$ for $0 \leq i < n$.
	The morphisms of $\ETrussClosed{1}$ are defined as
	\begin{align*}
		\ETrussClosed{1}(\trussCS{j} \ord{m}, \trussCS{i} \ord{n}) & = \{ \alpha : \ord{n} \to \ord{m} \mid \alpha(i) = j \}                       \\
		\ETrussClosed{1}(\trussCR{j} \ord{m}, \trussCR{i} \ord{n}) & = \{ \alpha : \ord{n} \to \ord{m} \mid \alpha(i) \leq j < \alpha(i + 1) \}    \\
		\ETrussClosed{1}(\trussCS{j} \ord{m}, \trussCR{i} \ord{n}) & = \{ \alpha : \ord{n} \to \ord{m} \mid \alpha(i) \leq j \leq \alpha(i + 1) \} \\
		\ETrussClosed{1}(\trussCR{j} \ord{m}, \trussCS{i} \ord{n}) & = \varnothing
	\end{align*}
	and compose as the opposite of the underlying order-preserving maps. There is a
	canonical forgetful functor $\TrussClosed{1} : \ETrussClosed{1} \to \FinOrd^\op$ called the
	\textit{universal closed $1$-truss bundle} which sends both
	$\trussCS{i}\ord{n}$ and $\trussCR{i}\ord{n}$ to $\ord{n}$.
	A \textit{closed $1$-truss bundle} is a map $f : E \to B$ of simplicial sets
	together with a pullback square
	\[
		\begin{tikzcd}
			E \arrow[d, "f", swap] \ar[r] \pullbackcorner & \ETrussClosed{1} \ar[d, "\TrussClosed{1}"] \\
			B \ar[r] & \FinOrd^\op
		\end{tikzcd}
	\]
  The map $\TrussClosed{1}$ induces a polynomial functor $\BTrussClosedL{1}{-} : \sSet \to \sSet$.
\end{definition}

\begin{example}\label{ex:fct-truss-closed-1-fibre-object}
  The fibre of $\TrussClosed{1} : \ETrussClosed{1} \to \FinOrd^\op$ over a finite ordinal $\ord{n}$
  is a sequence of $n$ consecutive cospans between $n$ regular and $n + 1$ singular elements
  \[
    \begin{tikzcd}
      \trussCS{0}\ord{n} \ar[r] &
      \trussCR{0}\ord{n} &
      \trussCS{1}\ord{n} \ar[l] \ar[r] &
      \cdots &
      \trussCS{n - 1}\ord{n} \ar[r] \ar[l] &
      \trussCR{n - 1}\ord{n} &
      \trussCS{n}\ord{n} \ar[l]
    \end{tikzcd}
  \]
  In particular the fibre over $\ord{0}$ consists of a single singular element $\trussCS{0}\ord{0}$.
  We observe that these fibres are equivalent to the exit path $\infty$-categories of closed $1$-meshes with the requisite number of singular strata.
\end{example}

\begin{example}\label{ex:fct-truss-closed-1-fibre-map}
  The fibre of $\TrussClosed{1} : \ETrussClosed{1} \to \FinOrd^\op$ over the opposite of an order
  preserving map $\alpha : \ord{n} \to \ord{m}$ consists of two sequences of cospans
  with $m + 1$ and $n + 1$ singular elements, respectively, connected together
  to form a planar diagram.
  The maps between the singular elements are determined directly by $\alpha$,
  while the remaining maps follow.
  \begin{enumerate}
    \item The opposite of the map $\ord{1} \to \ord{2}$ defined by $0 \mapsto 1$ and $1 \mapsto 2$
    restricts a sequence of two consecutive cospans to the second cospan.
    We interpret this as the restriction to a subdiagram.
    \[
      \begin{tikzcd}
        \trussCS{0}\ord{2} \ar[r] &
        \trussCR{0}\ord{2} &
        \trussCS{1}\ord{2} \ar[l] \ar[r] \ar[d] &
        \trussCR{1}\ord{2} \ar[d] &
        \trussCS{2}\ord{2} \ar[l] \ar[d] \\
        {} &
        {} &
        \trussCS{0}\ord{1} \ar[r] &
        \trussCR{0}\ord{1} &
        \trussCS{1}\ord{1} \ar[l]
      \end{tikzcd}
    \]

    \item The opposite of the unique map $\ord{1} \to \ord{0}$ splits a singular element by inserting a regular element between them.
    We interpret this as inserting an identity level into a diagram.
    \[
      \begin{tikzcd}
        {} &
        \trussCS{0}\ord{0} \ar[dl] \ar[dr] \\
        \trussCS{0}\ord{1} \ar[r] &
        \trussCR{0}\ord{1} &
        \trussCS{1}\ord{1} \ar[l]
      \end{tikzcd}
    \]

    \item The opposite of the map $\ord{1} \to \ord{2}$ defined by $0 \mapsto 0$ and $1 \mapsto 2$
    merges two regular levels.
    We interpret this as a composition operation.
    \[
      \begin{tikzcd}
        \trussCS{0}\ord{2} \ar[r] \ar[d] &
        \trussCR{0}\ord{2} \ar[dr] &
        \trussCS{1}\ord{2} \ar[l] \ar[r] &
        \trussCR{1}\ord{2} \ar[dl]&
        \trussCS{2}\ord{2} \ar[l] \ar[d] \\
        \trussCS{0}\ord{1} \ar[rr] &&
        \trussCR{0}\ord{1} &&
        \trussCS{1}\ord{1} \ar[ll]
      \end{tikzcd}
    \]
  \end{enumerate}
\end{example}

\begin{lemma}\label{lem:fct-truss-closed-1-categorical}
  The map $\TrussClosed{1} : \ETrussClosed{1} \to \FinOrd^\op$ is a categorical fibration.
\end{lemma}
\begin{proof}
  Being a functor between $1$-categories, the map is automatically an inner fibration
  by Lemma~\ref{lem:b-inner-fib-1-cat}.
  Because $\FinOrd^\op$ is skeletal, i.e.\ the isomorphisms in $\FinOrd^\op$ are exactly the
  identity maps, the map $\TrussClosed{1} : \ETrussClosed{1} \to \FinOrd^\op$ is trivially an isofibration.
\end{proof}

\begin{lemma}\label{lem:truss-closed-1-poset}
  Let $P$ be a poset, $\varphi : P \to \FinOrd^\op$ a functor and
  \[
    \begin{tikzcd}
      {\ETrussClosed{1} \times_{\FinOrd^\op} P} \ar[r] \ar[d] \pullbackcorner &
      {\ETrussClosed{1}} \ar[d] \\
      {P} \ar[r, "\varphi", swap] &
      {\FinOrd^\op}
    \end{tikzcd}
  \]
  the pullback of simplicial sets.
  Then $\ETrussClosed{1} \times_{\FinOrd^\op} P$ is a poset.  
\end{lemma}

The stratified geometric realisation functor $\StratReal{-} : \sSet \to \Strat$
gives us a method to geometrically realise a closed $1$-truss bundle
as a map of stratified spaces.
Equipped with the appropriate $1$-framing, the result is a closed $1$-mesh bundle.
Remember that when $P$ is a poset, the stratified geometric realisation $\StratReal{P}$ is glued together from a stratified $k$-simplex $\DeltaStrat{k}$ for every flag $\ord{k} \hookrightarrow P$.

\begin{construction}\label{con:fct-truss-closed-1-realise}
  We construct a geometric realisation functor $\FinOrd^\op \to \BMeshClosed{1}$
  that sends a $k$-simplex $\alpha : \ord{k} \to \FinOrd^\op$ to a closed $1$-mesh
  bundle $\xi : \strat{X} \to \DeltaStrat{k}$ with $\sing(\xi) = \alpha$.  
  We first obtain a closed $1$-truss bundle
  \[
    \begin{tikzcd}
      {\cat{E}} \ar[r, "\alpha'"] \ar[d, "F", swap] \pullbackcorner &
      {\ETrussClosed{1}} \ar[d] \\
      {\ord{k}} \ar[r, "\alpha", swap] &
      {\FinOrd^\op}
    \end{tikzcd}
  \]
  by pullback, where $\cat{E}$ is a poset by Lemma~\ref{lem:truss-closed-1-poset}.
  To every element $e \in \cat{E}$ we assign a coordinate
  $R(e) \in \DeltaTop{k} \times \R$ with $R(e) = (F(e), 2i)$ when
  $\alpha'(e) = \trussCS{i}\ord{n}$ and $R(e) = (F(e), 2i + 1)$ when
  $\alpha'(e) = \trussCR{i}\ord{n}$.
  This assignment extends to an embedding $\TopReal{\cat{E}} \to \DeltaTop{k} \times \R$ by linear interpolation, defining a framing for the stratified bundle
  $\StratReal{F} : \StratReal{\cat{E}} \to \DeltaStrat{k}$.
  Then $\xi := \StratReal{F}$ is a closed $1$-mesh bundle with $\sing(\xi) = \alpha$.
\end{construction}

\begin{remark}
  There are other possible coordinate assignments   
  for Construction~\ref{con:fct-truss-closed-1-realise}
  that result in the geometric realisation being a closed $1$-mesh bundle.  
  In particular any coordinate assignment which respects the frame order  
  (see~\cite{framed-combinatorial-topology}) will lead to an equivalent result.
\end{remark}

\begin{lemma}\label{lem:fct-truss-closed-1-realise}
  The geometric realisation functor $\FinOrd^\op \to \BMeshClosed{1}$ 
  from Construction~\ref{con:fct-truss-closed-1-realise}
  is a homotopy inverse to the equivalence
  $\sing : \BMeshClosed{1} \to \FinOrd^\op$.
\end{lemma}
\begin{proof}
  By Proposition~\ref{prop:fct-mesh-closed-1-sing-trivial} the functor
  $\sing$ is a trivial fibration and therefore
  and equivalence of quasicategories.
  The functor $\FinOrd^\op \to \BMeshClosed{1}$ is a section of the equivalence
  $\sing$ and therefore a homotopy inverse.
\end{proof}

\begin{construction}\label{con:fct-truss-closed-1-pointed-realise}
  We construct a map $\ETrussClosed{1} \to \EMeshClosed{1}$ which fits into the diagram
  \[
    \begin{tikzcd}
      {\ETrussClosed{1}} \ar[r] \ar[d, swap, "\TrussClosed{1}"] &
      {\EMeshClosed{1}} \ar[d, "\MeshClosed{1}"] \\
      {\FinOrd^\op} \ar[r] &
      {\BMeshClosed{1}}
    \end{tikzcd}
  \]
  where $\FinOrd^\op \to \BMeshClosed{1}$ is the geometric realisation functor
  from Construction~\ref{con:fct-truss-closed-1-realise}.
  Let $\alpha' : \ord{k} \to \ETrussClosed{1}$ be a $k$-simplex.
  Then by composition with $\ETrussClosed{1} \to \FinOrd^\op$, we get a $k$-simplex
  $\alpha : \ord{k} \to \FinOrd^\op$.
  By taking the pullback and by the universal property of that pullback,
  we then have a diagram
  \[
    \begin{tikzcd}
      {} &
      {\cat{E}} \ar[r] \ar[d, "F", swap] \pullbackcorner &
      {\ETrussClosed{1}} \ar[d] \\
      {\ord{k}} \ar[r, "\id", swap] \ar[ur, "S"] &
      {\ord{k}} \ar[r, "\alpha", swap] &
      {\FinOrd^\op}
    \end{tikzcd}
  \]
  The $k$-simplex $\alpha' : \ord{k} \to \ETrussClosed{1}$ therefore determines
  a $1$-truss bundle together with a section.  
  We can then take the stratified geometric realisation
  \[
    \begin{tikzcd}
      {} &
      {\StratReal{\cat{E}}} \ar[d, "\StratReal{F}"] \\
      {\DeltaStrat{k}} \ar[r, "\id", swap] \ar[ur, "\StratReal{S}"] &
      {\DeltaStrat{k}}
    \end{tikzcd}
  \]
  and equip the stratified bundle $\StratReal{F}$ with a $1$-framing as in
  Construction~\ref{con:fct-truss-closed-1-realise}.
  The closed $1$-mesh bundle $\StratReal{F}$ and its section $\StratReal{S}$
  together determine a $k$-simplex of $\EMeshClosed{1}$.
\end{construction}

\begin{lemma}\label{lem:fct-truss-closed-1-pointed-realise}
  The geometric realisation map
  $
    \ETrussClosed{1} \to \EMeshClosed{1}
  $
  from Construction~\ref{con:fct-truss-closed-1-pointed-realise}
  is a categorical equivalence.
\end{lemma}
\begin{proof}
  We construct a map $\EMeshClosed{1} \to \ETrussClosed{1}$ and demonstrate that it
  is a deformation retraction of $\ETrussClosed{1} \to \EMeshClosed{1}$.
  A $k$-simplex of $\EMeshClosed{1}$ is represented by
  a closed $1$-mesh bundle $\xi : \strat{X} \to \DeltaStrat{k}$
  together with a section $s : \DeltaStrat{k} \to \strat{X}$.
  The map $\sing(\xi) : \ord{k} \to \FinOrd^\op$ induces a closed $1$-truss bundle
  \[
    \begin{tikzcd}[column sep = large]
      {\cat{E}} \ar[r] \ar[d, "F", swap] \pullbackcorner &
      {\ETrussClosed{1}} \ar[d, "\TrussClosed{1}"] \\
      {\ord{k}} \ar[r, "\sing(\xi)", swap] &
      {\FinOrd^\op}
    \end{tikzcd}
  \]
  and by Construction~\ref{con:fct-truss-closed-1-realise} a closed $1$-mesh bundle $\StratReal{F} : \StratReal{\cat{E}} \to \DeltaStrat{k}$ with $\sing(\StratReal{F}) = \sing(\xi)$.
  We have by Proposition~\ref{prop:fct-mesh-closed-1-sing-trivial} that
  $\sing : \BMeshClosed{1} \to \FinOrd^\op$ is a trivial fibration;
  therefore there exists a closed $1$-mesh bundle
  $h : \strat{H} \to \DeltaTop{1} \times \DeltaStrat{k}$  
  which restricts to $\xi$ and $\StratReal{F}$ 
  \[
    \begin{tikzcd}[row sep = large]
      {\strat{X}} \ar[r] \ar[d, "\xi", swap] \pullbackcorner &
      {\strat{H}} \ar[d, "h"{description}] &
      {\StratReal{\cat{E}}} \ar[d, "\StratReal{F}"] \ar[l] \pullbackdl \\
      {\{ 0 \} \times \DeltaStrat{k}} \ar[r] &
      {\DeltaTop{1} \times \DeltaStrat{k}} &
      {\{ 1 \} \times \DeltaStrat{k}} \ar[l]
    \end{tikzcd}
  \]
  As a closed $1$-mesh bundle $h$ is a stratified fibre bundle,
  and therefore we must have isomorphisms of posets
  $\stratPos{\strat{X}} \cong \stratPos{H} \cong \cat{E}$.
  Consequently, we have a diagram  
  \[
    \begin{tikzcd}[column sep = large]
      {} &
      {\stratPos{\strat{X}}} \ar[r, "\cong"] \ar[d, "\stratPos{\xi}"] &
      {\cat{E}} \ar[d, "F"] \ar[r] \pullbackcorner &
      {\ETrussClosed{1}} \ar[d, "\TrussClosed{1}"] \\
      {\ord{k}} \ar[r, "\id", swap] \ar[ur, "\stratPos{s}"] &
      {\ord{k}} \ar[r, "\id", swap] &
      {\ord{k}} \ar[r, "\sing(\xi)", swap] &
      {\FinOrd^\op}
    \end{tikzcd}
  \]
  The composite $\ord{k} \to \ETrussClosed{1}$ then defines a $k$-simplex of $\ETrussClosed{1}$
  as desired.
  This defines the deformation retraction $\EMeshClosed{1} \to \ETrussClosed{1}$.
\end{proof}

\begin{lemma}
  The map $\TrussClosed{1} : \ETrussClosed{1} \to \FinOrd^\op$ is a flat categorical fibration.
\end{lemma}
\begin{proof}
  The map already is a categorical fibration by Lemma~\ref{lem:fct-truss-closed-1-categorical}.
  By Lemma~\ref{lem:fct-mesh-closed-labelled-pointed-flat} the map
  $\MeshClosed{1} : \EMeshClosed{1} \to \BMeshClosed{1}$ is a flat categorical fibration.
  The geometric realisation functors fit into the diagram
  \[
    \begin{tikzcd}
      {\ETrussClosed{1}} \ar[r, "\simeq"] \ar[d, swap, "\TrussClosed{1}"] &
      {\EMeshClosed{1}} \ar[d, "\MeshClosed{1}"] \\
      {\FinOrd^\op} \ar[r, "\simeq", swap] &
      {\BMeshClosed{1}}
    \end{tikzcd}
  \]
  in which by the horizontal maps are equivalences
  by Lemma~\ref{lem:fct-truss-closed-1-realise} and Lemma~\ref{lem:fct-truss-closed-1-pointed-realise}.
  Therefore the map $\TrussClosed{1} : \ETrussClosed{1} \to \FinOrd^\op$
  must be flat as well.
\end{proof}

\begin{proposition}\label{prop:fct-truss-closed-1-labelled-cat}
  The polynomial functor
  $
    \BTrussClosedL{1}{-} : \sSet \to \sSet
  $
  induced by the universal closed $1$-truss bundle
  $\TrussClosed{1} : \ETrussClosed{1} \to \FinOrd^\op$
  satisfies the following properties:
  \begin{enumerate}
    \item $\BTrussClosedL{1}{\cat{C}} \to \BTrussClosed{1}$ is a categorical fibration when $\cat{C}$ is a quasicategory.
    \item $\BTrussClosedL{1}{-}$ preserves quasicategories.
    \item $\BTrussClosedL{1}{-}$ preserves $1$-categories.
    \item $\BTrussClosedL{1}{-}$ preserves categorical equivalences between quasicategories.
  \end{enumerate}
\end{proposition}

\begin{construction}
  The geometric realisation maps from
  Construction~\ref{con:fct-truss-closed-1-realise} and
  Construction~\ref{con:fct-truss-closed-1-pointed-realise}
  arrange into a diagram
  \[
    \begin{tikzcd}
      {\ETrussClosed{1}} \ar[r, hook, "\simeq"] \ar[d, "\TrussClosed{1}"'] &
      {\EMeshClosed{1}} \ar[d, "\MeshClosed{1}"] \\
      {\BTrussClosed{1}} \ar[r, hook, "\simeq"'] &
      {\BMeshClosed{1}}
    \end{tikzcd}
  \]
  and are equivalences by
  Lemma~\ref{lem:fct-truss-closed-1-realise} and
  Lemma~\ref{lem:fct-truss-closed-1-pointed-realise}.
  Since the maps are also monomorphisms
  they are deformation retractions.
  Then for any simplicial set $C$
  we obtain by Construction~\ref{con:b-poly-retract}
  a geometric realisation functor for closed labelled $1$-truss bundles
  \[
    \begin{tikzcd}
      {\BTrussClosedL{1}{C}} \ar[r, hook] \ar[d] &
      {\BMeshClosedL{1}{C}} \ar[d] \\
      {\BTrussClosed{1}} \ar[r, hook] &
      {\BMeshClosed{1}}
    \end{tikzcd}
  \]
  When $\cat{C}$ is a quasicategory,
  the geometric realisation map $\BTrussClosedL{1}{\cat{C}} \to \BMeshClosedL{1}{\cat{C}}$
  is a categorical equivalence by Lemma~\ref{lem:b-poly-retract-equiv}.
\end{construction}

We recall from~\S\ref{sec:fct-mesh-closed-pack} that for any quasicategory
$\cat{C}$ there is an equivalence between the quasicategory $\BMeshClosedL{n}{\cat{C}}$,
which classifies closed $n$-meshes with labels in $\cat{C}$, and the $n$-fold application
$(\BMeshClosed{1} \circ \cdots \circ \BMeshClosed{1})(\cat{C})$
of the endofunctor $\BMeshClosedL{1}{-} : \sSet \to \sSet$ to $\cat{C}$.
In the case of closed trusses, we define $\BTrussOpenL{n}{-}$ from $\BTrussOpenL{1}{-}$
in this way.

\begin{construction}
  We write $\BTrussClosedL{n}{-} : \sSet \to \sSet$ for the $n$-fold composite
  \[
    \BTrussClosedL{n}{C} := \underbrace{(\BTrussClosed{1} \circ \cdots \circ \BTrussClosed{1})}_{\textup{$n$ times}}(C)
  \]
  We abbreviate $\BTrussClosed{n} := \BTrussClosedL{n}{\terminal}$ in the unlabelled case.
  Applying Proposition~\ref{prop:fct-truss-closed-1-labelled-cat} inductively,
  the simplicial set $\BTrussClosedL{n}{\cat{C}}$ is a quasicategory for every quasicategory
  of labels $\cat{C}$.
  Similarly, $\BTrussClosedL{n}{\cat{C}}$ is a $1$-category whenever $\cat{C}$ is.
\end{construction}


\begin{observation}
  Suppose that $A$, $C$ are simplicial sets.
  Using the characterisation of maps into polynomial functors from Lemma~\ref{lem:b-poly-characterise},
  we see that a map of simplicial sets $A \to \BTrussClosedL{n}{C}$ is determined
  by a diagram of simplicial sets
  \[
    \begin{tikzcd}
      {C} &
      {E_n} \ar[r, "F_n"] \ar[l, "\ell"'] &
      {E_{n - 1}} \ar[r] &
      {\cdots} \ar[r] &
      {E_1} \ar[r, "F_1"] & 
      {E_0}
    \end{tikzcd}
  \]
  together with a pullback square of simplicial sets
  \[
    \begin{tikzcd}
      {E_i} \ar[r] \ar[d, "F_i"] \pullbackcorner &
      {\ETrussOpen{1}} \ar[d] \\
      {E_{i - 1}} \ar[r] &
      {\FinOrd^\op}
    \end{tikzcd}
  \]
  for each $1 \leq i \leq n$ that exhibits $F_i$ as a closed $1$-truss bundle.
  We call this data a \defn{closed $n$-truss bundle} on $A$ with labels in $C$.
\end{observation}

\begin{construction}
  Let $\cat{C}$ be a quasicategory.
  The geometric realisation functor for closed $1$-trusses
  $\BTrussClosedL{1}{-} \to \BMeshClosedL{1}{-}$
  together with the packing equivalence of Corollary~\ref{cor:fct-mesh-closed-n-pack-iterate}
  induces a geometric realisation functor for closed $n$-trusses
  \[
    \BTrussClosedL{n}{\cat{C}}
    =
    \underbrace{(\BTrussClosed{1} \circ \cdots \circ \BTrussClosed{1})}_{\textup{$n$ times}}(\cat{C})
    \simeq
    \underbrace{(\BMeshClosed{1} \circ \cdots \circ \BMeshClosed{1})}_{\textup{$n$ times}}(\cat{C}) \simeq
    \BMeshClosedL{n}{\cat{C}}
  \]
  For the case $\cat{C} = \terminal$ we denote the geometric realisation functor
  by $\TrussClosedReal{-} : \BTrussClosed{n} \to \BMeshClosed{n}$ and its
  inverse, the truss nerve, by   
  $
    \TrussClosedNerve : \BMeshClosed{n} \to \BTrussClosed{n}
  $.
\end{construction}

\begin{construction}\label{con:fct-truss-grid-closed}
  Via the geometric realisation equivalence for closed trusses and meshes,
  the closed grid mesh functor from Construction~\ref{con:fct-mesh-grid-closed}
  induces an analogous functor for closed trusses:
  \[
    \begin{tikzcd}
      {\FinOrd^{n, \op}} \ar[r, equal] \ar[d, dashed, "\gridTrussClosed^n"'] &
      {\FinOrd^{n, \op}} \ar[d, "\gridMeshClosed^n"] \\
      {\BTrussClosed{n}} \ar[r, "\simeq"'] &
      {\BMeshClosed{n}}
    \end{tikzcd}
  \]
\end{construction}

\subsection{Open Trusses}

We have a duality equivalence between closed and open $1$-meshes
\[
  \begin{tikzcd}
    {\BMeshOpen{1}} \ar[r, "\simeq", dashed] \ar[d, "\reg"'] &
    {(\BMeshClosed{1})^\op} \ar[d, "\sing^\op"] \\
    {\FinOrd} \ar[r] &
    {\FinOrd}
  \end{tikzcd}
\]
using that $\reg$ and $\sing$ are equivalences.
In the case of trusses we want the duality equivalence to be an isomorphism of $1$-categories,
leading us to define open $1$-trusses as follows.


\begin{definition}\label{def:fct-truss-open-1}
	The objects of the $1$-category $\ETrussOpen{1}$ are \textit{singular objects} of
	the form $\trussOS{i} \ord{n}$ for natural numbers $0 \leq i < n$ and \textit{regular objects} $\trussOR{i}
		\ord{n}$ for $0 \leq i \leq n$.
	The morphisms of $\ETrussOpen{1}$ are defined as
	\begin{align*}
		\ETrussOpen{1}(\trussOR{i} \ord{n}, \trussOR{j} \ord{m}) & = \{ \alpha : \ord{n} \to \ord{m} \mid \alpha(i) = j \}                       \\
		\ETrussOpen{1}(\trussOS{i} \ord{n}, \trussOS{j} \ord{m}) & = \{ \alpha : \ord{n} \to \ord{m} \mid \alpha(i) \leq j < \alpha(i + 1) \}    \\
		\ETrussOpen{1}(\trussOS{i} \ord{n}, \trussOR{j} \ord{m}) & = \{ \alpha : \ord{n} \to \ord{m} \mid \alpha(i) \leq j \leq \alpha(i + 1) \} \\
		\ETrussOpen{1}(\trussOR{i} \ord{n}, \trussOS{j} \ord{m}) & = \varnothing
	\end{align*}
	and compose as the underlying order-preserving maps. There is a
	canonical forgetful functor $\ETrussOpen{1} \to \FinOrd$ called the
	\textit{universal open $1$-truss bundle} which sends both
	$\trussOS{i}\ord{n}$ and $\trussOR{i}\ord{n}$ to $\ord{n}$.
	An \textit{open $1$-truss bundle} is a map of simplicial sets
  $f : E \to B$
	together with a pullback square
	\[
		\begin{tikzcd}
			E \arrow[d, "f", swap] \ar[r] \pullbackcorner & \ETrussOpen{1} \ar[d] \\
			B \ar[r] & \FinOrd
		\end{tikzcd}
	\]
\end{definition}

\begin{observation}
  There is a unique isomorphism $(-)^\dagger : \ETrussOpen{1} \to (\ETrussClosed{1})^\op$
  over $\FinOrd$ which sends a regular object $\trussOR{i}\ord{n}$ to
  the singular object $\trussCS{i}\ord{n}$ and a singular object
  $\trussOS{i}\ord{n}$ to the regular object $\trussCR{i}\ord{n}$.
  This is the \defn{duality isomorphism} between open and closed $1$-trusses.
\end{observation}

\begin{construction}
  The \defn{compactification functor} for $1$-trusses
  is the inclusion $\ETrussOpen{1} \hookrightarrow \ETrussClosed{1}$ 
  that is defined on objects by
  $\trussOR{i}\ord{n} \mapsto \trussCR{i}\ord{n + 1}$ and $\trussOS{i}\ord{n} \mapsto \trussCS{i + 1}\ord{n + 1}$
  and fits into the diagram
  \[
    \begin{tikzcd}
      {\ETrussOpen{1}} \ar[r, hook] \ar[d] &
      {\ETrussClosed{1}} \ar[d] \\
      {\FinOrd} \ar[r, hook] &
      {\FinOrd^\op}
    \end{tikzcd}
  \]
  A fibre of $\ETrussOpen{1} \to \FinOrd$ over some ordinal is a zigzag of maps
  between regular and singular elements, ending with regular elements on both sides.
  The functor $\ETrussOpen{1} \to \ETrussClosed{1}$ includes these zigzags into the
  middle of the zigzag that is padded with singular elements on the ends.
\end{construction}

\begin{example}
  The compactification functor
  $\ETrussOpen{1} \hookrightarrow \ETrussClosed{1}$ embeds
  the fibre of $\ETrussOpen{1} \to \FinOrd$ over $\ord{1}$
  into the fibre of $\ETrussClosed{1} \to \FinOrd^\op$ over $\ord{2}$ as follows:
  \[
    \begin{tikzcd}
      {} &
      {\trussOR{0} \ord{1}} \ar[d, mapsto] &
      {\trussOS{0} \ord{1}} \ar[d, mapsto] \ar[l] \ar[r] &
      {\trussOR{1} \ord{1}} \ar[d, mapsto] &
      {} \\
      {\trussCS{0} \ord{2}} \ar[r] &
      {\trussCR{0} \ord{2}} &
      {\trussCS{1} \ord{2}} \ar[l] \ar[r] &
      {\trussCR{1} \ord{2}} &
      {\trussCS{2} \ord{2}} \ar[l]
    \end{tikzcd}
  \]
\end{example}



We know that the map $\reg : \BMeshOpen{1} \to \FinOrd$
is a trivial fibration and therefore admits a homotopy inverse
which geometrically realises a finite non-empty ordinal $\ord{i}$ as an
open $n$-truss with $i + 1$ regular strata.
Just as with closed $1$-trusses, we can find an explicit construction
for such a geometric realisation functor that extends also to a functor
$\ETrussOpen{1} \to \EMeshOpen{1}$ of pointed trusses and meshes.

However, in contrast to closed $1$-trusses, the geometric realisation
of an open $1$-truss bundle $f : E \to B$ is \textbf{not}
$\StratReal{f} : \StratReal{E} \to \StratReal{B}$ regardless of
the framing we might put on $\StratReal{f}$.
For any $b \in B$ the fibre $f^{-1}(b)$ is necessarily finite,
and therefore its geometric realisation as a topological space $\TopReal{f^{-1}(b)}$
is compact. The underlying space of an open $1$-mesh is $\R$,
which is not compact, and so $\StratReal{f}$ can not be an open $1$-mesh bundle.
Non-compact stratified spaces in general do not have finite triangulations, but they can have finite
relative triangulations, i.e.\ be presented as the complement
$\StratReal{B} \setminus \StratReal{B'}$ where $B' \subseteq B$ is a downwards closed subposet.
For the geometric realisation of open $1$-trusses, we make us of this idea as follows: We first compactify the open $1$-truss to obtain a closed $1$-truss,
which we can then geometrically realise to a closed $1$-mesh. We can then
remove the first and last singular strata in each fibre, resulting in
the open $1$-mesh we need.

\begin{construction}
  We can construct geometric realisation functors for open $1$-trusses via
  compactification and the geometric realisation functors of closed $1$-trusses
  from Construction~\ref{con:fct-truss-closed-1-realise} and
  Construction~\ref{con:fct-truss-closed-1-pointed-realise}.
  In particular we have a diagram
  \[
    \begin{tikzcd}[column sep = {3.5em, between origins}, row sep = {3.5em, between origins}]
    	& {\ETrussClosed{1}} && {\EMeshClosed{1}} && {\ETrussClosed{1}} \\
    	{\ETrussOpen{1}} && {\EMeshOpen{1}} && {\ETrussOpen{1}} \\
    	& {\FinOrd^\op} && {\BMeshClosed{1}} && {\FinOrd^\op} \\
    	{\FinOrd} && {\BMeshOpen{1}} && {\FinOrd}
    	\arrow[from=2-1, to=4-1]
    	\arrow[dashed, from=4-1, to=4-3]
    	\arrow[from=4-3, to=4-5]
    	\arrow[hook, from=2-1, to=1-2]
    	\arrow[hook, from=2-3, to=1-4]
    	\arrow[hook, from=2-5, to=1-6]
    	\arrow[hook, from=4-1, to=3-2]
    	\arrow[hook, from=4-3, to=3-4]
    	\arrow[hook, from=4-5, to=3-6]
    	\arrow[from=1-6, to=3-6]
    	\arrow[from=1-4, to=3-4]
    	\arrow[from=1-2, to=3-2]
    	\arrow[from=3-2, to=3-4]
    	\arrow[from=3-4, to=3-6]
    	\arrow[from=1-4, to=1-6]
    	\arrow[from=1-2, to=1-4]
    	\arrow[crossing over, from=2-5, to=4-5]
    	\arrow[crossing over, from=2-3, to=4-3]
    	\arrow[crossing over, dashed, from=2-1, to=2-3]
    	\arrow[crossing over, dashed, from=2-3, to=2-5]
    \end{tikzcd}
  \]
  in which all rows consist of equivalences and compose to the identity.
  Concretely, the functor $\FinOrd \to \BMeshOpen{1}$ is the section
  of $\reg : \BMeshOpen{1} \to \FinOrd$ which sends $\ord{i} \in \FinOrd$
  to the geometric realisation of $\ord{i + 1} \in \FinOrd^\op$
  as a closed $1$-truss. It then removes the first and last singular
  strata of the resulting closed $1$-mesh and rescales to $\R$.  
  The functor $\ETrussOpen{1} \to \EMeshOpen{1}$ proceeds similarly,
  keeping track of the section.
\end{construction}

\begin{observation}
  At this point we can again proceed analogously to our discussion
  of closed $n$-trusses from~\S\ref{sec:fct-truss-closed}.
  We quickly state the corresponding results.
  The universal open $1$-truss bundle
  $\TrussOpen{1} : \ETrussOpen{1} \to \FinOrd$
  is a flat categorical fibration since
  it is equivalent to the flat categorical fibration $\MeshOpen{1} : \EMeshOpen{1} \to \BMeshOpen{1}$.
  Therefore, the induced polynomial functor $\BTrussOpenL{1}{-} : \sSet \to \sSet$ preserves quasicategories, $1$-categories, and categorical equivalences between quasicategories.
  Moreover, via geometric realisation we obtain a natural equivalence
  $\BTrussOpenL{1}{\cat{C}} \simeq \BMeshOpenL{1}{\cat{C}}$
  for any quasicategory $\cat{C}$.
  We can then define $\BTrussOpenL{n}{C}$ for any $n \geq 0$ and simplicial set
  $C$ by iterating the functor $\BTrussOpenL{1}{-}$:
  \[
    \BTrussOpenL{n}{C} := \underbrace{(\BTrussOpen{1} \circ \cdots \circ \BTrussOpen{1})}_{\textup{$n$ times}}(C)
  \]
  In particular we then have a $1$-category $\BTrussOpen{n} := \BTrussOpenL{n}{*}$ which classifies unlabelled open $n$-truss bundles.
  Using the packing construction for open $n$-meshes, 
  we have a natural equivalence $\BTrussOpenL{n}{\cat{C}} \simeq \BMeshOpenL{n}{\cat{C}}$ for every quasicategory $\cat{C}$.  
\end{observation}

\begin{construction}
  The universal closed $1$-truss bundle $\TrussClosed{1}$ is (isomorphic to) the categorical opposite of the universal open $1$-truss bundle $\TrussOpen{1}$:
  \[
    \begin{tikzcd}
      {\ETrussOpen{1}} \ar[r, "\cong"] \ar[d, "\TrussOpen{1}"'] &
      {\ETrussClosed{1, \op}} \ar[d, "\TrussClosed{1, \op}"] \\
      {\FinOrd} \ar[r, "\cong"'] &
      {\FinOrd}
    \end{tikzcd}
  \]
  We therefore have a natural isomorphism between the induced polynomial functors
  \[ \BTrussOpenL{1}{C^\op} \cong \BTrussClosedL{1}{C}^\op \]
  for all simplicial sets of labels $C$. This induces a natural isomorphism
  \[
    \BTrussOpenL{n}{C} =
    \underbrace{(\BTrussOpen{1} \circ \cdots \circ \BTrussOpen{1})}_{\textup{$n$ times}}(C^\op) \cong
    \underbrace{(\BTrussClosed{1} \circ \cdots \circ \BTrussClosed{1})}_{\textup{$n$ times}}(C)^\op =
    \BTrussClosedL{n}{C}^\op
  \]
  for all $n \geq 0$ and simplicial sets $C$.
  Together with the geometric realisation equivalences for open and closed
  trusses, we obtain a duality equivalence between open and closed $n$-meshes
  with labels in any quasicategory $\cat{C}$:
  \[
    \begin{tikzcd}
      {\BTrussOpenL{n}{\cat{C}^\op}} \ar[r, "\cong"] \ar[d, "\simeq"'] &
      {\BTrussClosedL{n}{\cat{C}}^\op} \ar[d, "\simeq"] \\
      {\BMeshOpenL{n}{\cat{C}^\op}} \ar[r, dashed, "\simeq"'] &
      {\BMeshClosedL{n}{\cat{C}}^\op}
    \end{tikzcd}
  \]
\end{construction}

\begin{construction}\label{con:fct-truss-grid-open}
  Via the geometric realisation equivalence for open trusses and meshes,
  the open grid mesh functor from Construction~\ref{con:fct-mesh-grid-open}
  induces an analogous functor for open trusses:
  \[
    \begin{tikzcd}
      {\FinOrd^{n}} \ar[r, equal] \ar[d, hook, dashed, "\gridTrussOpen^n"'] &
      {\FinOrd^{n}} \ar[d, "\gridMeshOpen^n", hook] \\
      {\BTrussOpen{n}} \ar[r, "\simeq"'] &
      {\BMeshOpen{n}}
    \end{tikzcd}
  \]
  Moreover the grid functors are compatible with duality:
  \[
    \begin{tikzcd}[row sep = large]
      {\FinOrd^{n}} \ar[r, equal] \ar[d, hook, "\gridMeshOpen^n"'] &
      {\FinOrd^{n}} \ar[d, "\gridTrussOpen^n"{description}, hook] \ar[r, equal] &
      {\FinOrd^{n}} \ar[d, "{\gridTrussClosed^{n, \op}}"{description}, hook] \ar[r, equal] &
      {\FinOrd^{n}} \ar[d, "{\gridMeshClosed^{n, \op}}", hook] \\
      {\BMeshOpen{n}} \ar[r, "\simeq"'] &
      {\BTrussOpen{n}} \ar[r, "\simeq"'] &
      {\BTrussClosed{n, \op}} \ar[r, "\simeq"'] &
      {\BMeshClosed{n, \op}} 
    \end{tikzcd}
  \]
\end{construction}



\section{Refinements}\label{sec:fct-ref}

Open and closed n-meshes are both n-framed stratified spaces and therefore objects of $\FrStrat^n$.
Here we study the interaction of meshes with refinement maps between n-framed stratified spaces.
In~\S\ref{sec:fct-ref-closed} we take a framed refinement map $\varphi : \strat{X} \to \strat{Y}$
of closed $n$-meshes and view it as a map $\strat{X} \pto \strat{Y}$ in $\BMeshClosed{n}$, called a degeneracy bordism. By taking compactifications, we get an analogous theory of
degeneracy bordisms of open $n$-meshes in~\S\ref{sec:fct-ref-open}.
In~\S\ref{sec:fct-ref-coarsest} we explore how to find the coarsest mesh that refines a tame stratified space,
which we then use in~\S\ref{sec:fct-ref-nf} to define normal forms of meshes.

\subsection{Degeneracy Bordisms of Closed Meshes}\label{sec:fct-ref-closed}

\begin{example}
  Suppose we have a $1$-framed refinement map between closed $1$-meshes
  \[
    \begin{tikzpicture}[scale = 0.5, rotate = -90, baseline=(current bounding box.center)]
      \node[] at (0, -1) {$\strat{X}_0$};
      \node[] at (2, -1) {$\strat{X}_1$};
      \draw[mesh-stratum] (2, 0) -- (2, 6);
      \draw[mesh-stratum] (0, 0) -- (0, 6);
      \node[mesh-vertex] at (0, 0) {};
      \node[mesh-vertex] at (0, 1) {};
      \node[mesh-vertex] at (0, 2) {};
      \node[mesh-vertex] at (0, 3) {};
      \node[mesh-vertex] at (0, 4) {};
      \node[mesh-vertex] at (0, 5) {};
      \node[mesh-vertex] at (0, 6) {};
      \node[mesh-vertex] at (2, 0) {};
      \node[mesh-vertex] at (2, 2) {};
      \node[mesh-vertex] at (2, 4) {};
      \node[mesh-vertex] at (2, 6) {};
      \draw[->>] (0.5, 3) -- (1.5, 3) node [midway, left] {$\varphi$};
    \end{tikzpicture}
  \]
  Then the projection map $\CylO(\varphi) \to \DeltaStrat{1}$ from the open
  mapping cylinder of $\varphi$ (see~\S\ref{sec:b-geo-ocyl}) can be equipped with a $1$-framing so
  that it becomes the closed $1$-mesh bundle  
  \[
    \begin{tikzpicture}[scale = 0.5, rotate = -90, baseline=(current bounding box.center)]
      \fill[mesh-background] (0, 0) -- (0, 6) -- (2, 6) -- (2, 0) -- cycle;
      \draw[mesh-stratum] (0, 0) -- (0, 6);
      \node[mesh-vertex] at (0, 0) {};
      \node[mesh-vertex] at (0, 1) {};
      \node[mesh-vertex] at (0, 2) {};
      \node[mesh-vertex] at (0, 3) {};
      \node[mesh-vertex] at (0, 4) {};
      \node[mesh-vertex] at (0, 5) {};
      \node[mesh-vertex] at (0, 6) {};
      \draw[mesh-stratum] (0, 0) -- (2, 0);
      \draw[mesh-stratum] (0, 2) -- (2, 2);
      \draw[mesh-stratum] (0, 4) -- (2, 4);
      \draw[mesh-stratum] (0, 6) -- (2, 6);
    \end{tikzpicture}
    \quad
    \longrightarrow
    \quad
    \begin{tikzpicture}[scale = 0.5, baseline=(current bounding box.center)]
      \draw[mesh-stratum] (0, 0) -- (0, 2);
      \node[mesh-vertex] at (0, 2) {};
    \end{tikzpicture}
  \]
  This closed $1$-mesh bundle in turn defines a bordism of closed $1$-meshes
  $\degMeshClosed(\varphi) : \strat{X} \pto \strat{Y}$ which we call the
  degeneracy bordism induced by $\varphi$.
\end{example}

Via the degeneracy bordisms arising from open mapping cylinders,
we can represent $n$-framed refinement maps between closed $n$-meshes as maps
within the $\infty$-category $\BMeshClosed{n}$. 
To make this work in general, we first need to equip open mapping cylinders of $n$-framed stratified
bundles with an induced $n$-framing.

\begin{construction}\label{con:fct-ref-ocyl-framing}
  Suppose that we have a sequence of $n$-framed refinement maps
  \[
    \begin{tikzcd}
      {\strat{X}_0} \ar[r, "\varphi_1"] \ar[d, "f_0"] &
      {\strat{X}_1} \ar[r] \ar[d, "\xi_1"] &
      {\cdots} \ar[r] &
      {\strat{X}_{k - 1}} \ar[r, "\varphi_k"] \ar[d, "f_{k - 1}"]  &
      {\strat{X}_k} \ar[d, "f_k"] \\
      {\strat{B}_0} \ar[r, "\psi_1"'] &
      {\strat{B}_1} \ar[r] &
      {\cdots} \ar[r] &
      {\strat{B}_{k - 1}} \ar[r, "\psi_k"'] &
      {\strat{B}_k}
    \end{tikzcd}
  \]
  between $n$-framed stratified bundles and let
  $F : \CylO^k(\strat{X}_\bullet) \to \CylO^k(\strat{B}_\bullet)$
  be the induced map between the open mapping cylinders.
  Then we can equip $F$ with an $n$-framing, defined for $x \in \strat{X}_0$
  and $b \in \DeltaStrat{k}$ by linear interpolation
  \[
    \framing(F)(x, b) := \sum_{i \in \ord{k}} b_i (\framing(f_i) \circ \varphi_i \circ \cdots \circ \varphi_1)(x)
  \]
  This is well-defined because each $\varphi_i$ is a refinement map and therefore
  the underlying map of topological spaces $\unstrat(\varphi_i)$ is an isomorphism. 
\end{construction}

\begin{proposition}\label{prop:fct-ref-closed-ocyl-mesh}
  Suppose that we have a sequence of $n$-framed refinement maps
  \[
    \begin{tikzcd}
      {\strat{X}_0} \ar[r, "\varphi_1"] \ar[d, "\xi_0"] &
      {\strat{X}_1} \ar[r] \ar[d, "\xi_1"] &
      {\cdots} \ar[r] &
      {\strat{X}_{k - 1}} \ar[r, "\varphi_k"] \ar[d, "\xi_{k - 1}"]  &
      {\strat{X}_k} \ar[d, "\xi_k"] \\
      {\strat{B}_0} \ar[r, "\psi_1"'] &
      {\strat{B}_1} \ar[r] &
      {\cdots} \ar[r] &
      {\strat{B}_{k - 1}} \ar[r, "\psi_k"'] &
      {\strat{B}_k}
    \end{tikzcd}
  \]
  between closed $n$-mesh bundles.
  Then the map $\Xi : \CylO^k(\strat{X}_\bullet) \to \CylO^k(\strat{B}_\bullet)$
  between the $k$-fold open mapping cylinders, equipped with the $n$-framing
  from Construction~\ref{con:fct-ref-ocyl-framing}, is a closed $n$-mesh bundle.
\end{proposition}
\begin{proof}
  When $n = 0$ the claim is trivial, so suppose that $n > 0$.
  By writing the closed $n$-mesh bundles as composites
  of closed $1$-mesh bundles followed by closed $(n - 1)$-mesh bundles, 
  we obtain a diagram
  \[
    \begin{tikzcd}
      {\strat{X}_0} \ar[r, "\varphi_1"] \ar[d, "\chi_0"] &
      {\strat{X}_1} \ar[r] \ar[d, "\chi_1"] &
      {\cdots} \ar[r] &
      {\strat{X}_{k - 1}} \ar[r, "\varphi_k"] \ar[d, "\chi_{k - 1}"]  &
      {\strat{X}_k} \ar[d, "\chi_k"] \\
      {\strat{Y}_0} \ar[r, "\rho_1"] \ar[d, "\zeta_0"] &
      {\strat{Y}_1} \ar[r] \ar[d, "\zeta_1"] &
      {\cdots} \ar[r] &
      {\strat{Y}_{k - 1}} \ar[r, "\rho_k"] \ar[d, "\zeta_{k - 1}"]  &
      {\strat{Y}_k} \ar[d, "\zeta_k"] \\
      {\strat{B}_0} \ar[r, "\psi_1"'] &
      {\strat{B}_1} \ar[r] &
      {\cdots} \ar[r] &
      {\strat{B}_{k - 1}} \ar[r, "\psi_k"'] &
      {\strat{B}_k}
    \end{tikzcd}
  \]
  By induction, the map between open $k$-fold mapping cylinders
  $G : \CylO^k(\strat{Y}_\bullet) \to \CylO^k(\strat{B}_\bullet)$
  induced by the lower half of this diagram
  is a closed $(n - 1)$-mesh bundle when equipped with the framing
  from Construction~\ref{con:fct-ref-ocyl-framing}.
  Let $F : \CylO^k(\strat{X}_\bullet) \to \CylO^k(\strat{Y}_\bullet)$
  be the map between open $k$-fold mapping cylinders induced by the
  upper half of the diagram and equip $F$ with the $1$-framing 
  from Construction~\ref{con:fct-ref-ocyl-framing}.
  Then the $1$-framing map $\framing(F)$ is closed, and by Lemma~\ref{lem:b-geo-ocyl-isofibration}
  the induced map $\Exit(F)$ is an isofibration.
  Therefore, by Lemma~\ref{lem:fct-mesh-closed-1-recognise-fibration}
  the $1$-framed stratified bundle $F$ is a closed $1$-mesh bundle.
  Therefore, the composite $G \circ F : \CylO^k(\strat{X}_\bullet) \to \CylO^k(\strat{B}_\bullet)$
  is a closed $n$-mesh bundle.
\end{proof}

\begin{definition}
  Let $\varphi: \strat{X} \to \strat{Y}$ be an $n$-framed refinement map 
  between closed $n$-meshes. Then the \defn{degeneracy bordism induced by
  $\varphi$} is the bordism of closed $n$-meshes
  $\degMeshClosed(\varphi) : \strat{X} \pto \strat{Y}$
  represented by the closed $n$-mesh bundle $\CylO(\varphi) \to \DeltaStrat{1}$.
\end{definition}

\begin{definition}
  Let $f : \strat{M} \pto \strat{N}$ be a bordism of closed $n$-meshes
  represented by a closed $n$-mesh bundle $\Xi : \strat{E} \to \DeltaStrat{1}$.
  Then $f$ is a \defn{degeneracy bordism} if $\stratPos{\Xi} : \stratPos{\strat{E}} \to \ord{1}$
  is the cocartesian fibration associated to a surjective map of posets
  $\stratPos{\strat{X}} \to \stratPos{\strat{Y}}$.
\end{definition}

\begin{proposition}\label{prop:fct-ref-closed-ocyl-bordism}
  Let $f : \strat{X} \pto \strat{Y}$ be a bordism of closed $n$-meshes.
  Then $f$ is a degeneracy bordism if and only if it is equivalent to
  $\degMeshClosed(\varphi)$ for some
  $n$-framed refinement map
  $\varphi : \strat{X} \to \strat{Y}$.
  In that case we have that $\stratPos{\varphi} : \stratPos{\strat{M}} \to \stratPos{\strat{N}}$ is the surjective map of posets associated to the degeneracy bordism $f$.
\end{proposition}
\begin{proof}
  Let $\Xi : \strat{E} \to \DeltaStrat{1}$ be the closed $1$-mesh bundle that
  represents $f$. Suppose first that $f$ is a degeneracy bordism and let
  $s : \stratPos{\strat{X}} \to \stratPos{\strat{Y}}$ be the surjective map
  of posets which unstraightens to the cocartesian fibration
  $\stratPos{\Xi} : \stratPos{\strat{E}} \to \ord{1}$.
  Via the equivalence $\BMeshClosed{n} \simeq \BTrussClosed{n}$
  between closed $n$-meshes and $n$-trusses, we may assume that $\strat{X}$, $\strat{Y}$ and $\Xi$ arise
  as the geometric realisation of closed $n$-truss bundles
  \[
    \begin{tikzcd}
      {\stratPos{\strat{Y}}} \ar[r] \ar[d] \pullbackcorner &
      {\stratPos{\strat{E}}} \ar[d] &
      {\stratPos{\strat{X}}} \ar[l] \ar[d] \pullbackdl \\
      {\ord{0}} \ar[r, hook, "\langle 0 \rangle"'] &
      {\ord{1}} &
      {\ord{0}} \ar[l, hook', "\langle 1 \rangle"]
    \end{tikzcd}
  \]  
  as described in Construction~\ref{con:fct-truss-closed-1-realise}.
  The surjective map of posets
  $s : \stratPos{\strat{X}} \to \stratPos{\strat{Y}}$
  induces a refinement map 
  $\varphi : \StratReal{\stratPos{\strat{X}}} \to \StratReal{\stratPos{Y}}$.
  Unpacking the details of Construction~\ref{con:fct-truss-closed-1-realise}, we can see that
  $\varphi$ preserves the $n$-framing and that $\Xi$ is the open mapping cylinder
  of $\varphi$.

  For the implication in the other direction,
  let $\varphi : \strat{X} \to \strat{Y}$ be an $n$-framed refinement map
  such that $\Xi : \CylO(\varphi) \to \DeltaStrat{1}$ represents the bordism
  $f : \strat{X} \pto \strat{Y}$.
  Then Lemma~\ref{lem:b-geo-ocyl-cocartesian-fibration} implies that $f$ is a degeneracy bordism as required.
\end{proof}

\begin{lemma}\label{lem:fct-ref-closed-composite}
  Degeneracy bordisms of closed $n$-meshes are closed under composition in the
  $\infty$-category $\BMeshClosed{n}$.
\end{lemma}
\begin{proof}
  We use that by Proposition~\ref{prop:fct-ref-closed-ocyl-bordism} each degeneracy bordism of closed $n$-meshes
  is induced by $n$-framed refinement maps.
  Let $\varphi_{01} : \strat{X}_0 \to \strat{X}_1$ and $\varphi_{12} : \strat{X}_1 \to \strat{X}_2$
  be $n$-framed refinement maps between closed $n$-mesh bundles.
  Then projection map $\CylO(\strat{X}_\bullet) \to \DeltaStrat{2}$ of the $2$-fold open mapping cylinder 
  is a closed $n$-mesh bundle by Proposition~\ref{prop:fct-ref-closed-ocyl-mesh} that represents the composition
  of the induced degeneracy bordisms $\degMeshClosed(\varphi_{01})$ followed by
  $\degMeshClosed(\varphi_{12})$. In particular the composite is
  $\degMeshClosed(\varphi_{12} \circ \varphi_{01})$ and therefore itself a degeneracy bordism
  by Proposition~\ref{prop:fct-ref-closed-ocyl-bordism}.
\end{proof}

We can broaden the concept of a degeneracy map of closed $n$-meshes to incorporate
labels within any quasicategory $\cat{C}$. The labelling of a degeneracy bordism
of labelled closed $n$-meshes is determined completely (up to contractible choice) by the labelling of
the codomain.

\begin{definition}
  Suppose that $\cat{C}$ is a quasicategory and $\pi : \BMeshClosedL{n}{\cat{C}} \to \BMeshClosed{n}$ the functor which forgets the labels.
  Then a bordism $f$ of closed $n$-meshes with labels in $\cat{C}$ is a \defn{degeneracy bordism}
  if $\pi(f)$ is a degeneracy bordism and $f$ is $\pi$-cartesian.
\end{definition}

\begin{lemma}\label{lem:fct-ref-closed-deg-map-lift}
  Let $\cat{C}$ be a quasicategory and $\pi : \BMeshClosedL{n}{\cat{C}} \to \BMeshClosed{n}$
  the functor which forgets the labels.
  Then $\pi$ admits cartesian lifts of degeneracy maps.
\end{lemma}
\begin{proof}
  Let $\varphi : \strat{X} \to \strat{Y}$ be an $n$-framed refinement map
  between closed $n$-mesh bundles and let $\ell_{\strat{Y}} : \Exit(\strat{Y}) \to \cat{C}$
  be a labelling map.
  Denote by $\pi : \CylO(\strat{X}) \to \DeltaStrat{1}$ the projection from the
  open mapping cylinder that defines the degeneracy bordism $\degMeshClosed(\varphi)$.
  A $\pi$-cartesian lift of $\degMeshClosed(\varphi)$ that ends in $(\strat{Y}, \ell_{\strat{Y}})$
  then corresponds to a right Kan extension $\ell : \Exit(\CylO(\strat{X})) \to \cat{C}$
  of $\ell_{\strat{Y}}$ along the inclusion map
  $\Exit(\strat{Y}) \hookrightarrow \Exit(\CylO(\strat{X}))$.
  As long as the requisite limits exist in $\cat{C}$, we can therefore
  compute $\ell$ as a pointwise Kan extension defined on
  $e \in \CylO(\strat{X})$ by
  \begin{equation}\label{eq:fct-ref-closed-deg-map-lift:limit}
    \ell(e) \simeq
    \lim(
    \begin{tikzcd}[cramped]
      (\Exit(\strat{Y}))_{e/} \ar[r] &
      {\Exit(\strat{Y})} \ar[r, "\ell_{\strat{Y}}"] &
      \cat{C}
    \end{tikzcd}
    )
  \end{equation}

  When $e$ is contained within $\strat{X} \subseteq \CylO(\varphi)$, then by Lemma~\ref{lem:b-geo-ocyl-cocartesian-map} the identity exit path
  $\id : \DeltaStrat{1} \to \DeltaStrat{1}$ has an $\Exit(\pi)$-cocartesian
  lift to an exit path $\gamma : \DeltaStrat{1} \to \CylO(\varphi)$ starting
  at $e$ and ending at $\varphi(e)$.
  Since the lift is $\Exit(\pi)$-cocartesian,
  this path is an initial object in the slice $\Exit(\strat{Y})_{e/}$
  and so the limit~(\ref{eq:fct-ref-closed-deg-map-lift:limit}) exists and evaluates to $\ell(e) \simeq \ell_{\strat{Y}}(\varphi(e))$.

  When $e$ is not contained within $\strat{X}$, then via the trivialisation
  of $\pi$ over the stratum $\intOC{0, 1} \subseteq \DeltaStrat{1}$ there exists
  a stratum-preserving path $\gamma : \DeltaTop{1} \to \CylO(\varphi)$ from $e$ to
  an element $y \in \strat{Y} \subseteq \CylO(\varphi)$.
  Therefore, the limit~(\ref{eq:fct-ref-closed-deg-map-lift:limit}) exists and evaluates on $e$ to $\ell(e) \simeq \ell_{\strat{Y}}(y)$.
\end{proof}

\subsection{Degeneracy Bordisms of Open Meshes}\label{sec:fct-ref-open}

\begin{example}
  Suppose we have a refinement map of open $1$-meshes
  \[
    \begin{tikzpicture}[scale = 0.5, rotate = -90, baseline=(current bounding box.center)]
      \node[] at (0, -1) {$M_0$};
      \node[] at (2, -1) {$M_1$};
      \draw[mesh-stratum] (2, 0) -- (2, 6);
      \draw[mesh-stratum] (0, 0) -- (0, 6);
      \node[mesh-vertex] at (0, 1) {};
      \node[mesh-vertex] at (0, 2) {};
      \node[mesh-vertex] at (0, 3) {};
      \node[mesh-vertex] at (0, 4) {};
      \node[mesh-vertex] at (0, 5) {};
      \node[mesh-vertex] at (2, 2) {};
      \node[mesh-vertex] at (2, 4) {};
      \draw[->>] (0.5, 3) -- (1.5, 3) node [midway, left] {$\varphi$};
    \end{tikzpicture}
  \]
  Then the projection map $\CylO(\varphi) \to \DeltaStrat{1}$ from the open mapping
  cylinder of $\varphi$ can be equipped with a $1$-framing so that it becomes
  the open $1$-mesh bundle
  \[
    \begin{tikzpicture}[scale = 0.5, rotate = -90, baseline=(current bounding box.center)]
      \fill[mesh-background] (0, 0) -- (0, 6) -- (2, 6) -- (2, 0) -- cycle;
      \draw[mesh-stratum] (0, 0) -- (0, 6);
      \node[mesh-vertex] at (0, 1) {};
      \node[mesh-vertex] at (0, 2) {};
      \node[mesh-vertex] at (0, 3) {};
      \node[mesh-vertex] at (0, 4) {};
      \node[mesh-vertex] at (0, 5) {};
      \draw[mesh-stratum] (0, 2) -- (2, 2);
      \draw[mesh-stratum] (0, 4) -- (2, 4);
    \end{tikzpicture}
    \quad
    \longrightarrow
    \quad
    \begin{tikzpicture}[scale = 0.5, baseline=(current bounding box.center)]
      \draw[mesh-stratum] (0, 0) -- (0, 2);
      \node[mesh-vertex] at (0, 2) {};
    \end{tikzpicture}
  \]
\end{example}

\begin{proposition}
  Suppose that we have a sequence of $n$-framed refinement maps
  \[
    \begin{tikzcd}
      {\strat{M}_0} \ar[r, "\varphi_1"] \ar[d, "f_0"] &
      {\strat{M}_1} \ar[r] \ar[d, "f_1"] &
      {\cdots} \ar[r] &
      {\strat{M}_{k - 1}} \ar[r, "\varphi_k"] \ar[d, "f_{k - 1}"]  &
      {\strat{M}_k} \ar[d, "f_k"] \\
      {\strat{B}_0} \ar[r, "\psi_1"'] &
      {\strat{B}_1} \ar[r] &
      {\cdots} \ar[r] &
      {\strat{B}_{k - 1}} \ar[r, "\psi_k"'] &
      {\strat{B}_k}
    \end{tikzcd}
  \]
  between open $n$-mesh bundles.
  Then the map $F : \CylO^k(\strat{X}_\bullet) \to \CylO^k(\strat{B}_\bullet)$
  between the $k$-fold open mapping cylinders, equipped with the $n$-framing
  from Construction~\ref{con:fct-ref-ocyl-framing}, is an open $n$-mesh bundle.
\end{proposition}
\begin{proof}
  The sequence of $n$-framed refinement maps from the claim gives rise to a sequence
  of $n$-framed refinement maps between the compactifications of the open $n$-mesh
  bundles:
  \[
    \begin{tikzcd}
      {\bar{\strat{M}}_0} \ar[r, "\varphi_1"] \ar[d, "\bar{f}_0"] &
      {\bar{\strat{M}}_1} \ar[r] \ar[d, "\bar{f}_1"] &
      {\cdots} \ar[r] &
      {\bar{\strat{M}}_{k - 1}} \ar[r, "\varphi_k"] \ar[d, "\bar{f}_{k - 1}"]  &
      {\bar{\strat{M}}_k} \ar[d, "\bar{f}_k"] \\
      {\strat{B}_0} \ar[r, "\psi_1"'] &
      {\strat{B}_1} \ar[r] &
      {\cdots} \ar[r] &
      {\strat{B}_{k - 1}} \ar[r, "\psi_k"'] &
      {\strat{B}_k}
    \end{tikzcd}
  \]
  We can then apply Proposition~\ref{prop:fct-ref-closed-ocyl-mesh} to see that the induced $n$-framed
  stratified bundle $\bar{F} : \CylO^k(\bar{\strat{M}}_\bullet) \to \CylO^k(\strat{B}_\bullet)$
  is a closed $n$-mesh bundle.
  The restriction of $\bar{F}$ to the constructible $n$-framed stratified subbundle $\CylO^k(\strat{M}_\bullet) \to \CylO^k(\strat{B}_\bullet)$
  then is an open $n$-mesh bundle, which is not identical with but isomorphic
  as an $n$-framed stratified bundle to
  the map $F$ equipped with the framing from Construction~\ref{con:fct-ref-ocyl-framing}.
\end{proof}

\begin{definition}
  Let $\varphi: \strat{M} \to \strat{N}$ be an $n$-framed refinement map 
  between open $n$-meshes. Then the \defn{degeneracy bordism induced by
  $\varphi$} is the bordism of open $n$-meshes
  $\degMeshOpen(\varphi) : \strat{M} \pto \strat{N}$
  represented by the open $n$-mesh bundle $\CylO(\varphi) \to \DeltaStrat{1}$.
\end{definition}

\begin{definition}
  Let $f : \strat{M} \pto \strat{N}$ be a bordism of open $n$-meshes
  represented by an open $n$-mesh bundle $F : \strat{E} \to \DeltaStrat{1}$.
  Then $f$ is a \defn{degeneracy bordism} if $\stratPos{F} : \stratPos{\strat{E}} \to \ord{1}$
  is the cocartesian fibration associated to a surjective map of posets
  $\stratPos{\strat{M}} \to \stratPos{\strat{N}}$.
\end{definition}

\begin{observation}\label{obs:fct-ref-open-compactify}
  Let $f : \strat{M} \pto \strat{N}$ be a bordism of open $n$-meshes.
  Then $f$ is a degeneracy bordism if and only if its compactification
  is a degeneracy bordism of closed $n$-mesh bundles.
\end{observation}

\begin{proposition}
  Let $f : \strat{M} \pto \strat{N}$ be a bordism of open $n$-meshes.
  Then $f$ is a degeneracy bordism if and only if it is equivalent to
  $\degMeshOpen(\varphi)$ for some
  $n$-framed refinement map
  $\varphi : \strat{M} \to \strat{N}$.
  In that case we have that $\stratPos{\varphi} : \stratPos{\strat{M}} \to \stratPos{\strat{N}}$ is the surjective map of posets associated to the degeneracy bordism $f$.
\end{proposition}
\begin{proof}
  Follows from Proposition~\ref{prop:fct-ref-closed-ocyl-bordism}
  by using Observation~\ref{obs:fct-ref-open-compactify}.
\end{proof}

\begin{lemma}
  Degeneracy bordisms of open $n$-meshes are closed under composition in the
  $\infty$-category $\BMeshOpen{n}$.
\end{lemma}
\begin{proof}
  Follows from Lemma~\ref{lem:fct-ref-closed-composite}
\end{proof}

\begin{definition}
  Suppose that $\cat{C}$ is a quasicategory and $\pi : \BMeshOpenL{n}{\cat{C}} \to \BMeshOpen{n}$ the functor which forgets the labels.
  Then a bordism $f$ of open $n$-meshes with labels in $\cat{C}$ is a \defn{degeneracy bordism}
  when $\pi(f)$ is a degeneracy bordism and $f$ is $\pi$-cartesian.
\end{definition}

\begin{lemma}\label{lem:fct-ref-open-deg-map-lift}
  Let $\cat{C}$ be a quasicategory and $\pi : \BMeshOpenL{n}{\cat{C}} \to \BMeshOpen{n}$
  the functor which forgets the labels.
  Then $\pi$ admits cartesian lifts of degeneracy maps.
\end{lemma}
\begin{proof}
  Analogous to the proof of Lemma~\ref{lem:fct-ref-closed-deg-map-lift}.
\end{proof}

\subsection{Coarsest Refining Meshes}\label{sec:fct-ref-coarsest}

When $\varphi : \strat{E} \to \strat{K}$ is an $n$-framed stratified bundle and $f$ is an open or closed $n$-mesh bundle that refines $\varphi$,
the singularities in $f$ detect the moments at which the arrangement of the strata in $\strat{E}$ changes qualitatively with respect to the framing.
It is in this sense that we may see framed combinatorial topology as a form of Morse theory.
Not every $n$-framed stratified bundle can be refined by an $n$-mesh bundle.
Obstructions to meshability include an infinite number of strata and various wild phenomena such as infinitely oscillating curves.
This is less a defect of meshes but a feature,
since meshability provides a criterion for tameness.

\begin{definition}
  We say that an $n$-framed stratified bundle $\varphi$ is \defn{closed meshable} when there exists a closed $n$-mesh bundle
  $\xi$ together with a refinement map $\xi \to \varphi$ of $n$-framed stratified bundles.
  Similarly, an $n$-framed stratified bundle $\varphi$ is \defn{open meshable} when it is refined by an open $n$-mesh bundle.
\end{definition}

\begin{lemma}\label{lem:coarsest-mesh-refinement-1-cube}
	Let $\varphi : \strat{E} \to \strat{K}$ be a tame $1$-framed stratified bundle
  such that $\unstrat(\varphi)$ is the projection $\intCC{-1, 1} \times \unstrat(\strat{K}) \to \unstrat(\strat{K})$.
	Then there exists a coarsest closed $1$-mesh bundle $\xi : \strat{X}_1 \to \strat{X}_0$
	that refines $\varphi$ as a $1$-framed stratified bundle so that $\strat{M}_0$ is a PL stratified pseudomanifold.
\end{lemma}
\begin{proof}
	We first apply Proposition~\ref{prop:coarsest-subdivision} to obtain the coarsest
	refinement $\strat{X}_0 \to \strat{K}$ so that $\strat{X}_0$ is a PL stratified pseudomanifold
	and for each stratum $p \in \stratPos{\strat{E}}$ the closed image $\cl(f(\strat{E}_p))$
	is a constructible subspace of $\strat{X}_0$.
	Then we apply Lemma~\ref{lem:strat-coarsest-domain}
  to obtain the coarsest refinement
	$\strat{X}_1 \to \strat{E}$ such that $\strat{X}_1$ is a PL stratified pseudomanifold
  and $\varphi$ induces a stratified map $\xi : \strat{X}_1 \to \strat{X}_0$.

	The stratified map $\xi : \strat{X}_1 \to \strat{X}_0$ inherits the $1$-framing from $\varphi$.
  The strata in the fibres of $\xi$ consist of
	isolated points and open intervals; half open or closed intervals in the interior
	are excluded since $\strat{X}_1$ is a PL stratified pseudomanifold by construction.  
	The number and type of strata in the fibres is constant over every stratum
	of $\strat{X}_0$ because the closure of the image $\cl(\varphi(\strat{E}_p))$ is a constructible subspace for each stratum $p \in \stratPos{\strat{E}}$.
	Non-intersecting strata embedded in $\R$ can not change their positions relative to each other and so $\xi$ must be a stratified fibre bundle.
  Since $\strat{X}_1$ is a PL stratified pseudomanifold, it satisfies the frontier condition,
  and so by Lemma~\ref{lem:fct-mesh-closed-1-pl} the $1$-framed bundle
	$\xi : \strat{X}_1 \to \strat{X}_0$ must be a closed $1$-mesh bundle.
	Any other mesh bundle that refines $\varphi$ must also satisfy the criteria
	that we enforced and so must refine $\xi$.
\end{proof}

\begin{lemma}\label{lem:coarsest-mesh-refinement-n-cube}
	Let $\varphi : \strat{E} \to \strat{B}$ be a tame $n$-framed stratified bundle
  such that $\unstrat(\varphi)$ is the projection $\intCC{-1, 1}^n \times \unstrat(\strat{K}) \to \unstrat(\strat{B})$.  
	Then there exists a coarsest closed $n$-mesh bundle
  $\xi : \strat{X}_n \to \strat{X}_0$
  so that $\xi$ refines $\varphi$ as an $n$-framed stratified bundle
  and $\strat{X}_0$ is a PL stratified pseudomanifold.  
\end{lemma}
\begin{proof}
	The claim is trivial for $n = 0$ and covered by Lemma~\ref{lem:coarsest-mesh-refinement-1-cube}
	for $n = 1$. Suppose for induction that $n \geq 2$.
	Using Proposition~\ref{prop:coarsest-subdivision} and
	Lemma~\ref{lem:strat-coarsest-domain}
	we can find the coarsest PL stratified pseudomanifold $\strat{D}$
  with $\unstrat(\strat{D}) = \intCC{-1, 1}^{n - 1} \times \unstrat(\strat{K})$  
	so that
	$\pi^{\geq n} : \unstrat(\strat{D}) \to \unstrat(\strat{K})$
	induces a stratified map $\strat{D} \to \strat{K}$.
	Then we find the coarsest PL stratified pseudomanifold $\strat{E}'$ with a refinement
	$\strat{E}' \to \strat{E}$ so that the projection $\pi^{\geq n - 1} : \stratForget{E}
		\to D$ becomes a stratified map $\strat{E}' \to \strat{D}$.
	At this point we have constructed the diagram of stratified spaces
	\begin{equation}\label{eq:coarsest-d-mesh-split}
		\begin{tikzcd}
			{\strat{E}'} \ar[r] \ar[d] &
			{\strat{D}} \ar[r] &
			{\strat{K}} \ar[d, "\id"] \\
			{\strat{E}} \ar[rr, "\varphi"'] &
			{} &
			{\strat{K}}
		\end{tikzcd}
	\end{equation}
	which consists of the coarsest refinements with PL stratified pseudomanifolds that make $\varphi : \strat{E} \to \strat{K}$
	factor through the projection.
	The map $\strat{E}' \to \strat{D}$ inherits a $1$-framing
  and satisfies the conditions of Proposition~\ref{lem:coarsest-mesh-refinement-1-cube}.
  We may therefore refine the $1$-framed bundle
	$\strat{E}' \to \strat{D}$ with the coarsest refining closed $1$-mesh bundle
	$\strat{X}_n \to \strat{X}_{n - 1}$ over a PL stratified pseudomanifold
  $\strat{X}_{n - 1}$.
	This results in the diagram
	\begin{equation}\label{eq:coarsest-d-mesh-left}
		\begin{tikzcd}
			{\strat{X}_n} \ar[r] \ar[d] &
			{\strat{X}_{n - 1}} \ar[r] \ar[d] &
			{\strat{K}} \ar[d, "\id"] \\
			{\strat{E}'} \ar[r] &
			{\strat{D}} \ar[r] &
			{\strat{K}}
		\end{tikzcd}
	\end{equation}
	The map $\strat{X}_{n - 1} \to \strat{K}$ is an $n$-framed bundle
  that satisfies the conditions of this proposition,
	and so by induction we can obtain the coarsest refining closed $(n - 1)$-mesh bundle
	$\strat{X}_{n - 1}' \to \strat{X}_0'$.
	By Lemma~\ref{lem:fct-mesh-closed-1-pullback} closed $1$-mesh bundles are closed under pullback, which we use to fill
	in the left square in the diagram
	\begin{equation}\label{eq:coarsest-d-mesh-right}
		\begin{tikzcd}
			{\strat{X}'_n} \ar[r] \ar[d] \pullbackcorner &
			{\strat{X}'_{n - 1}} \ar[r] \ar[d] &
			{\strat{X}'_0} \ar[d] \\
			{\strat{X}_n} \ar[r] &
			{\strat{X}_{n - 1}} \ar[r] &
			{\strat{K}}
		\end{tikzcd}
	\end{equation}
	This finishes the construction of a closed $n$-mesh bundle
  $\strat{X}_n' \to \strat{X}_0$
	that refines the given $n$-framed bundle $\varphi$.
	Our construction guarantees that this closed $n$-mesh bundle is also the
	coarsest one over a base space that is a PL stratified pseudomanifold, which can be checked by observing how any other such closed $n$-mesh bundle
  is consecutively refined by the maps in the diagrams
	(\ref{eq:coarsest-d-mesh-split}), (\ref{eq:coarsest-d-mesh-left}) and (\ref{eq:coarsest-d-mesh-right}).
\end{proof}


\begin{proposition}\label{prop:coarsest-mesh-refinement-open-pl}
  Let $\varphi : \strat{E} \to \strat{K}$ be an essentially tame $n$-framed stratified bundle such that $\unstrat(\varphi)$ is the projection $\R^n \times \unstrat(\strat{K}) \to \unstrat(\strat{K})$.  
	Then there exists a coarsest open $n$-mesh bundle
  $f : \strat{M}_n \to \strat{M}_0$
  so that $f$ refines $\varphi$ as an $n$-framed stratified bundle
  and $\strat{M}_0$ is a PL stratified pseudomanifold.  
\end{proposition}
\begin{proof}
  We pick any order-preserving and piecewise linear isomorphism $\intOO{-1, 1} \to \R$.
  Then the $n$-framed stratified bundle $\varphi$ is isomorphic
  to a tame $n$-framed stratified bundle $\varphi' : \strat{E}' \to \strat{K}$
  where $\unstrat(\strat{E}') = \intOO{-1, 1}^n \times \unstrat(\strat{K})$.
  Then we can construct the coarsest stratification $\strat{E}''$ of
  $\intCC{-1, 1}^n \times \unstrat(\strat{K})$ so that $\strat{E}'$ is a constructible
  subspace of $\strat{E}''$.
  The projection map $\varphi'' : \strat{E}'' \to \strat{K}$ with the induced
  $n$-framing then satisfies the conditions of Lemma~\ref{lem:coarsest-mesh-refinement-n-cube}.
  Hence there exists a coarsest closed $n$-mesh bundle $\xi : \strat{X}_n \to \strat{X}_0$
  that refines $\varphi''$ so that $\strat{X}_0$ is a PL stratified pseudomanifold.
  We write $\strat{M}_0 := \strat{X}_0$ and let $\strat{M}_n$ be the intersection
  $\strat{X}_n \cap (\intOO{-1, 1}^n \times \unstrat(\strat{K}))$.
  Rescaling the framing then induces the open $n$-mesh bundle $f: \strat{M}_n \to \strat{M}_0$ as desired.
\end{proof}

\begin{observation}\label{obs:coarsest-mesh-refinement-local}
  By Observation~\ref{obs:coarsest-subdivision-local} the coarsest subdivisions
  obtained via Proposition~\ref{prop:coarsest-subdivision} are determined locally.
  Going through the constructions in Lemma~\ref{lem:coarsest-mesh-refinement-1-cube}
  and Lemma~\ref{lem:coarsest-mesh-refinement-n-cube} we see that the coarsest
  open $n$-mesh obtained in Proposition~\ref{prop:coarsest-mesh-refinement-open-pl}
  is also determined locally.
  Concretely, suppose that $\varphi : \strat{E} \to \strat{K}$ is an $n$-framed stratified
  bundle that satisfies the conditions of Proposition~\ref{prop:coarsest-mesh-refinement-open-pl} and let
  \[
    \begin{tikzcd}
      {\strat{E}'} \ar[r, hook] \ar[d, "\varphi'", swap] \pullbackcorner &
      {\strat{E}} \ar[d, "\varphi"] \\
      {\strat{K}'} \ar[r, hook] &
      {\strat{K}}
    \end{tikzcd}
  \]
  be the restriction of $\varphi$ over an open subspace $\strat{K}' \subseteq \strat{K}$.
  Let $f$, $f'$ are the coarsest open $n$-mesh bundles over PL stratified pseudomanfolds
  which refine $\varphi$ and $\varphi'$, respectively,
  as produced by Proposition~\ref{prop:coarsest-mesh-refinement-open-pl}. 
  Then $f$ restricts to $f'$:
  \[
    \begin{tikzcd}
      {\strat{M}'} \ar[r, hook] \ar[d, "f'", swap] \pullbackcorner &
      {\strat{M}} \ar[d, "f"] \\
      {\strat{B}'} \ar[r, hook] &
      {\strat{B}}
    \end{tikzcd}
  \]
\end{observation}

\begin{remark}
  Coarsest refining meshes have been discussed before in~\cite{framed-combinatorial-topology}
  using a different approach.
  While we have constructed a coarsest refining mesh for PL framed stratified spaces
  directly, Dorn and Douglas show that there exists some refining mesh
  that is not necessarily the coarsest.
  They then demonstrate that every meshable framed stratified space admits 
  a coarsest mesh refinement by calculating the join of all refining meshes.
  The techniques of~\cite{framed-combinatorial-topology} likely generalise
  to bundles as well, but this is not discussed in detail.
\end{remark}

\subsection{Normal Forms}\label{sec:fct-ref-nf}

For an unlabelled open $n$-mesh $\strat{M}$ there always exists a refinement
map $\strat{M} \to \R^n$ which forgets all strata of $\strat{M}$, and therefore
a degeneracy bordism $\strat{M} \pto \R^n$. However, when $\strat{M}$ is equipped
with labels in some quasicategory $\cat{C}$, not every refinement map out of
$\strat{M}$ is compatible with the labelling.

\begin{example}
  Consider the following $1$-mesh $\strat{M}$ with labels in $\ord{1}$:
  \[
    \begin{tikzpicture}
      \draw[mesh-stratum] (0, 0) -- (4, 0);
      \node[mesh-vertex] at (1, 0) {};
      \node[mesh-vertex] at (2, 0) {};
      \node[mesh-vertex] at (3, 0) {};
      \node[] at (0.5, 0.5) {1}; 
      \node[] at (1, 0.5) {0}; 
      \node[] at (1.5, 0.5) {1}; 
      \node[] at (2, 0.5) {1}; 
      \node[] at (2.5, 0.5) {1}; 
      \node[] at (3, 0.5) {0};
      \node[] at (3.5, 0.5) {1}; 
    \end{tikzpicture}
  \]
  Because the singular strata of $\strat{M}$ on the left and right are labelled differently from their surrounding regular strata, we can not coarsen $\strat{M}$ in any way which would forget these two singular strata
  while still preserving the labelling.
  This is not an issue for the middle singular stratum, and so we
  can coarsen $\strat{M}$ by merging that singular stratum with its
  surrounding regular strata as follows:
  \[
    \begin{tikzpicture}
      \draw[mesh-stratum] (0, 0) -- (4, 0);
      \node[mesh-vertex] at (1, 0) {};
      \node[mesh-vertex] at (3, 0) {};
      \node[] at (0.5, 0.5) {1}; 
      \node[] at (1, 0.5) {0}; 
      \node[] at (2, 0.5) {1}; 
      \node[] at (3, 0.5) {0};
      \node[] at (3.5, 0.5) {1}; 
    \end{tikzpicture}
  \]
  This open $1$-mesh now can not be coarsened further, and therefore
  is the normal form $\nf(\strat{M})$ of the open labelled $1$-mesh $\strat{M}$.
\end{example}

\begin{proposition}\label{prop:fct-ref-nf-open}
  Let $\cat{C}$ be a quasicategory and $\strat{M} \in \BMeshOpenL{n}{\cat{C}}$ an open $n$-mesh with labels
  in $\cat{C}$. Then the $\infty$-subcategory of the slice $\BMeshOpenL{n}{\cat{C}}_{\strat{M}/}$ consisting
  of degeneracy maps out of $\strat{M}$ has a terminal object
  $\strat{M} \pto \nf(\strat{M})$.
\end{proposition}
\begin{proof}
  Up to equivalence in $\BMeshOpenL{n}{\cat{C}}$ we can assume that $\strat{M}$ is tame.
  We then let $\strat{M}$ be the coarsest $n$-framed stratified space with an $n$-framed refinement
  $\varphi : \strat{M} \to \strat{E}$ such that the labelling map $\ell_{\strat{M}} : \Exit(\strat{M}) \to \cat{C}$
  factors through $\Exit(\varphi)$.
  While $\strat{E}$ itself is not necessarily an open $n$-mesh, 
  by Proposition~\ref{prop:coarsest-mesh-refinement-open-pl} there exists a coarsest open $n$-mesh $\strat{N}$
  with a refinement map $\strat{N} \to \strat{E}$.
  Because $\strat{M}$ is an open $n$-mesh which also refines $\psi$,
  there then must be a refinement map $\psi : \strat{M} \to \strat{N}$.
  We then define the labelling map $\ell_{\strat{N}} : \Exit(\strat{N}) \to \cat{C}$
  via the induced labelling of $\strat{E}$:
  \[
    \begin{tikzcd}
      {\Exit(\strat{M})} \ar[r] \ar[dr, "\ell_{\strat{M}}"'] &
      {\Exit(\strat{N})} \ar[r] \ar[d, dashed] &
      {\Exit(\strat{E})} \ar[dl] \\
      {} &
      {\cat{C}}
    \end{tikzcd}
  \]  
  The refinement map $\psi : \strat{M} \to \strat{N}$ then induces
  a degeneracy bordism $\strat{M} \pto \strat{N}$ of $n$-meshes with
  labels in $\cat{C}$. By construction of $\strat{N}$ as the
  coarsest open $n$-mesh that refines $\strat{E}$ this
  degeneracy bordism is the claimed terminal object.
\end{proof}

\begin{remark}
  When $\cat{C}$ is a Kan complex and $\strat{M}$ is an open $n$-mesh with
  labels in $\cat{C}$, then the normal form of $\strat{M}$ is the open $n$-mesh $\R^n$
  with the induced labels. 
\end{remark}

\begin{remark}
  When $\strat{M} \pto \strat{N}$ is a bordism of open $n$-meshes
  with labels in some quasicategory $\cat{C}$, we do \textbf{not} in general
  have an induced bordism $\nf(\strat{M}) \pto \nf(\strat{N})$ between
  the normal forms.
\end{remark}

\begin{remark}
  Via the equivalence $\BTrussOpenL{n}{\cat{C}} \simeq \BMeshOpenL{n}{\cat{C}}$
  the concept of degeneracy maps transfers from open meshes to trusses.
  The degeneracy maps of open trusses then agree (up to taking opposites)
  with the degeneracy maps in the zigzag category that were discussed in~\cite{zigzag-normalisation}.
  We refer to there for an alternative combinatorial proof of~\ref{prop:fct-ref-nf-open} as well as an algorithm that efficiently computes the normal form.
\end{remark}

\section{Embeddings}\label{sec:fct-embed}

Similar to framed refinement maps, also constructible embeddings between open or closed $n$-meshes induce bordisms in the respective classifying $\infty$-categories $\BMeshOpen{n}$ and $\BMeshClosed{n}$.
We begin in~\S\ref{sec:fct-embed-closed} by demonstrating how constructible embeddings between closed $n$-meshes correspond to inert bordisms in $\BMeshClosed{n}$, using a reversed mapping cylinder construction.
Then in~\S\ref{sec:fct-embed-open} we use the duality between open and closed $n$-meshes to discuss constructible embeddings and inert bordisms for open $n$-meshes.
In~\S\ref{sec:fct-embed-active} we show that inert bordisms are part of an orthogonal factorisation system on $\BMeshOpen{n}$ and $\BMeshClosed{n}$.
Open and closed meshes are covered by atoms and cells, respectively,
which we discuss in~\S\ref{sec:fct-embed-atoms-cells}.
Finally, in~\S\ref{sec:fct-embed-stype} we characterise atoms and cells
by their singularity type.


\subsection{Inert Bordisms of Closed Meshes}\label{sec:fct-embed-closed}

\begin{example}
  Suppose we have a closed $1$-mesh $\strat{X}$ and a constructible
  subspace $\strat{Y}$ which is again a closed $1$-mesh, as below:  
  \[
    \begin{tikzpicture}[scale = 0.5, baseline=(current bounding box.center)]
      \node[mesh-vertex] at (0, 0) {};
      \node[mesh-vertex] at (1, 0) {};
      \node[mesh-vertex] at (2, 0) {};
      \node[mesh-vertex] at (3, 0) {};
      \node[mesh-vertex] at (4, 0) {};
      \node[mesh-vertex] at (5, 0) {};
      \node[mesh-vertex] at (1, 2) {};
      \node[mesh-vertex] at (2, 2) {};
      \node[mesh-vertex] at (3, 2) {};
      \draw[mesh-stratum] (0, 0) -- (5, 0);
      \draw[mesh-stratum] (1, 2) -- (3, 2);
      \draw[->] (2, 1.5) -- (2, 0.5);
    \end{tikzpicture}
  \]
  We can then represent the embedding of $\strat{Y}$ into $\strat{X}$
  in $\BMeshClosed{1}$ as a bordism of closed $1$-meshes
  $\strat{X} \pto \strat{Y}$ by constructing the reversed mapping cylinder:
  \[
    \begin{tikzpicture}[scale = 0.5, baseline=(current bounding box.center)]
      \fill[mesh-background] (1, 0) rectangle (3, 2);
      \node[mesh-vertex] at (0, 0) {};
      \node[mesh-vertex] at (1, 0) {};
      \node[mesh-vertex] at (2, 0) {};
      \node[mesh-vertex] at (3, 0) {};
      \node[mesh-vertex] at (4, 0) {};
      \node[mesh-vertex] at (5, 0) {};
      \draw[mesh-stratum] (1, 0) -- +(0, 2);
      \draw[mesh-stratum] (2, 0) -- +(0, 2);
      \draw[mesh-stratum] (3, 0) -- +(0, 2);
      \draw[mesh-stratum] (0, 0) -- (5, 0);
    \end{tikzpicture}
    \quad
    \longrightarrow
    \quad
    \begin{tikzpicture}[scale = 0.5, baseline=(current bounding box.center)]
      \node[mesh-vertex] at (0, 0) {};
      \draw[mesh-stratum] (0, 0) -- +(0, 2);
    \end{tikzpicture}
  \]
\end{example}

\begin{construction}\label{con:fct-emb-closed}
  Let $\strat{X}$ be a closed $n$-mesh and $\strat{Y} \subseteq \strat{X}$
  a closed constructible $n$-submesh.
  The \defn{reversed mapping cylinder} of the inclusion $\strat{Y} \subseteq \strat{X}$
  is the constructible subbundle
  $\strat{E} \to \DeltaStrat{1}$
  of $\strat{X} \times \DeltaStrat{1} \to \DeltaStrat{1}$
  where
  \[ \strat{E} = (\strat{X} \times \{ 0 \}) \cup (\strat{Y} \times \intOC{0, 1}) \subseteq \strat{X} \times \DeltaStrat{1}. \]
\end{construction}

To show that the reversed mapping cylinder of any constructible embedding
between closed $n$-meshes is a closed $n$-mesh bundle itself, we use an auxiliary
result that lets us restrict closed $n$-mesh bundles to constructible subbundles.

\begin{lemma}\label{lem:fct-mesh-closed-subbundle}
  Let $\xi : \strat{X} \to \strat{B}$ be a closed $n$-mesh bundle and
  $\xi' : \strat{X}' \to \strat{B}'$ a closed constructible subbundle
  such that each fibre of $\xi'$ is a closed $n$-mesh.
  Then $\xi'$ is a closed $n$-mesh bundle itself.
\end{lemma}
\begin{proof}
  We begin by writing $\xi : \strat{X} \to \strat{B}$ as the composite of
  a closed $1$-mesh bundle $\chi : \strat{X} \to \strat{Y}$ followed by a
  closed $(n - 1)$-mesh bundle $\zeta : \strat{Y} \to \strat{B}$.
  We can then also factor the closed constructible subbundle $\xi'$ to obtain
  the diagram
  \[
    \begin{tikzcd}
      {\strat{X}'} \ar[r, hook] \ar[d, "\chi'"] \ar[dd, bend right, "\xi'"'] &
      {\strat{X}} \ar[d, "\chi"'] \ar[dd, bend left, "\xi"] \\
      {\strat{Y}'} \ar[r, hook] \ar[d, "\zeta'"] &
      {\strat{Y}} \ar[d, "\zeta"'] \\
      {\strat{B}'} \ar[r, hook] &
      {\strat{B}}
    \end{tikzcd}
  \]
  The two squares are closed constructible embeddings of framed stratified bundles.
  For every $b \in \strat{B}$ the fibre of $\xi'(b)$ is a closed $n$-mesh.
  By the pasting lemma for pullbacks, the fibre of $\xi' = \zeta' \circ \chi'$ over $b$ can be
  computed in two stages:
  \[
    \begin{tikzcd}
      {\strat{X}''} \ar[r, hook] \ar[d] \pullbackcorner &
      {\strat{X}'} \ar[d, "\chi'"] \\
      {\strat{Y}''} \ar[r, hook] \ar[d] \pullbackcorner &
      {\strat{Y}'} \ar[d, "\zeta'"] \\
      {\{ b \}} \ar[r, hook] &
      {\strat{B}}
    \end{tikzcd}
  \]
  Hence the fibre of $\zeta'$ over $b$ is the closed $(n - 1)$-mesh $\strat{Y}''$.  
  Therefore, by induction on $n \geq 0$ the $(n - 1)$-framed stratified bundle $\zeta'$
  is a closed $(n - 1)$-mesh bundle.

  It remains to show that $\chi'$ is a closed $1$-mesh bundle.
  Since the closed $1$-mesh bundle $\chi$ is a stratified fibre bundle
  and $\chi'$ is a constructible subbundle, it follows that $\chi'$ is
  also a stratified fibre bundle.
  Let $y \in \strat{Y}'$ be any point, then the fibre of $\chi'$ over
  $y$ is the closed $1$-mesh $\strat{X}'''$ constructed by taking the pullbacks
  \[
    \begin{tikzcd}
      {\strat{X}'''} \ar[r, hook] \ar[d] \pullbackcorner &
      {\strat{X}''} \ar[r, hook] \ar[d] \pullbackcorner &
      {\strat{X}'} \ar[d, "\chi'"] \\
      {\{ y \}} \ar[r, hook] &
      {\strat{Y}''} \ar[r, hook] \ar[d] \pullbackcorner &
      {\strat{Y}'} \ar[d, "\zeta'"] \\
      {} &
      {\{ \chi'(y) \}} \ar[r, hook] &
      {\strat{B}}
    \end{tikzcd}
  \]
  The subspace of singular points $\sing(\strat{X}') \subseteq \strat{X}'$ is
  the preimage of the closed subset of singular points $\sing(\strat{X}) \subseteq \strat{X}'$ via the embedding $\strat{X}' \hookrightarrow \strat{X}$.
  Therefore, $\sing(\strat{X}')$ is closed in $\strat{X}'$.
  Moreover the framing map of $\chi'$ is the composite of the closed embeddings
  \[
    \begin{tikzcd}
      \unstrat(\strat{X}') \ar[r, hook] &
      \unstrat(\strat{X}) \ar[r, hook, "\framing(\chi)"] &
      \R \times \unstrat(\strat{Y}')
    \end{tikzcd}
  \]
  and is therefore closed itself.
\end{proof}

\begin{definition}
  Let $f : \strat{X} \pto \strat{Y}$ be a bordism of closed $n$-meshes
  represented by a closed $n$-mesh bundle $\xi : \strat{E} \to \DeltaStrat{1}$.
  Then $f$ is \defn{inert} if $\stratPos{\xi} : \stratPos{\strat{E}} \to \ord{1}$
  is the cartesian fibration associated to an inclusion of posets
  $\stratPos{\strat{Y}} \hookrightarrow \stratPos{\strat{X}}$.
\end{definition}

\begin{proposition}\label{prop:fct-emb-closed-rcyl-bordism}
  Let $f : \strat{X} \pto \strat{Y}$ be a bordism of closed $n$-meshes.
  Then $f$ is inert if and only if it is equivalent to the bordism
  induced by the reversed mapping cylinder of a constructible embedding $\strat{X} \hookrightarrow \strat{Y}$.
\end{proposition}
\begin{proof}
  Analogous to Proposition~\ref{prop:fct-ref-closed-ocyl-bordism}.
\end{proof}

\begin{lemma}\label{lem:fct-emb-closed-composite}
  Inert bordisms of closed $n$-meshes are closed under composition in $\BMeshClosed{n}$.
\end{lemma}
\begin{proof}
  Suppose that we have inert bordisms $f_{01} : \strat{X}_1 \pto \strat{X}_0$ and $f_{12} : \strat{X}_2 \pto \strat{X}_1$ of closed $n$-meshes.
  By Proposition~\ref{prop:fct-emb-closed-rcyl-bordism} we may assume that $f_{01}$ and $f_{12}$ arise
  as the reversed mapping cylinders of constructible embeddings
  $\varphi_{01} : \strat{X}_0 \hookrightarrow \strat{X}_1$
  and $\varphi_{12} : \strat{X}_1 \hookrightarrow \strat{X}_2$.
  We then let $\varphi_{02} : \strat{X}_0 \hookrightarrow \strat{X}_2$
  be the composite $\varphi_{02} = \varphi_{12} \circ \varphi_{01}$.
  Consider the closed $n$-mesh bundle
  $\Xi : \strat{E} \to \DeltaStrat{2}$
  where $\strat{E}$ is the constructible subspace
  of $\strat{X}_2 \times \DeltaStrat{2}$ where
  $\stratPos{\strat{E}}$ is the subposet of $\stratPos{\strat{X}_2} \times \ord{2}$ defined by
  \begin{align*}
    \stratPos{\strat{E}} := &\
    \{
      (p, 0) \mid p \in \stratPos{\strat{X}_2}
    \} \\
    \cup &\
    \{ 
      (\varphi_{01}(p), 1) \mid p \in \stratPos{\strat{X}_1}
    \} \\
    \cup &\
    \{ 
      (\varphi_{02}(p), 2) \mid p \in \stratPos{\strat{X}_0}
    \}
  \end{align*}
  By Lemma~\ref{lem:fct-mesh-closed-subbundle} we have that $\Xi$ is a closed $n$-mesh bundle.
  Moreover, $\Xi$ restricts to the reversed mapping cylinders of the
  inclusions $\varphi_{01}$, $\varphi_{02}$ and $\varphi_{12}$.
  Therefore, $\Xi$ witnesses the composition of the inert bordisms
  $f_{01}$ and $f_{12}$.
\end{proof}


\begin{definition}
  Let $\cat{C}$ be a quasicategory.
  A bordism $f$ of closed $n$-meshes with labels in $\cat{C}$ is \defn{inert}
  if the functor $\pi : \BMeshClosedL{n}{\cat{C}} \to \BMeshClosed{n}$ which forgets the labels
  sends $f$ to an inert bordism $\pi(f)$ and $f$ is $\pi$-cocartesian.
\end{definition}

When $\strat{Y} \hookrightarrow \strat{X}$ is a constructible embedding between
closed $n$-meshes and we have a label functor $\ell : \Exit(\strat{X}) \to \cat{C}$
in some quasicategory $\cat{C}$, we can restrict $\ell$ to a label functor for
$\strat{Y}$. Expressed in the language of inert bordisms, this restriction is
realised by a cocartesian lift of the inert bordism $\strat{X} \pto \strat{Y}$
associated to the constructible embedding.

\begin{lemma}\label{lem:fct-emb-closed-inert-map-lift}
  Let $\cat{C}$ be a quasicategory.
  Then the functor $\BMeshClosedL{n}{\cat{C}} \to \BMeshClosed{n}$
  which forgets the labels has cocartesian lifts of inert maps.
\end{lemma}
\begin{proof}
  Analogous to Lemma~\ref{lem:fct-ref-open-deg-map-lift}.
\end{proof}

\begin{lemma}
  Let $\cat{C}$ be a quasicategory.
  Then the inert bordisms of closed $n$-meshes with labels in $\cat{C}$
  are closed under composition in $\BMeshClosedL{n}{\cat{C}}$.
\end{lemma}
\begin{proof}
  Inert bordisms of closed unlabelled $n$-meshes are closed under composition
  by Lemma~\ref{lem:fct-emb-closed-composite}. The claim then follows since cocartesian maps are also
  closed under composition.
\end{proof}

\subsection{Inert Bordisms of Open Meshes}\label{sec:fct-embed-open}

Open $n$-meshes also admit a notion of inert bordisms that represents
constructible embeddings. The inert bordisms of open meshes are dual to those
of closed meshes.

\begin{example}\label{ex:fct-emb-open}
  We consider the embedding of open $1$-meshes
  \[
    \begin{tikzpicture}[rotate = -90, scale = 0.5, baseline=(current bounding box.center)]
      \node at (0, -1) {$M_0$};
      \node at (2, -1) {$M_1$};
      \draw[mesh-stratum-orange] (0, 0) -- (0, 6);
      \draw[mesh-stratum] (2, 0) -- (2, 6);
      \draw[mesh-stratum-orange] (2, 1) -- (2, 4);
      \node[mesh-vertex] at (2, 1) {};
      \node[mesh-vertex-orange] at (2, 2) {};
      \node[mesh-vertex-orange] at (2, 3) {};
      \node[mesh-vertex] at (2, 4) {};
      \node[mesh-vertex] at (2, 5) {};
      \node[mesh-vertex-orange] at (0, 2) {};
      \node[mesh-vertex-orange] at (0, 3) {};
      \draw[right hook->] (0.5, 2.5) -- (1.5, 2.5) node [midway, left] {$\varphi$};
    \end{tikzpicture}
  \]
  Then the associated inert bordism $\strat{M}_0 \pto \strat{M}_1$ of open $1$-meshes looks like this:
  \[
    \begin{tikzpicture}[rotate = -90, scale = 0.5, baseline=(current bounding box.center)]
      \fill[mesh-background] (0, 0) -- (0, 6) -- (2, 6) -- (2, 0) -- cycle;
      \draw[mesh-stratum] (0, 0) -- (0, 6);
      \node[mesh-vertex] at (0, 2) {};
      \node[mesh-vertex] at (0, 3) {};
      \draw[mesh-stratum] (0, 0) -- (2, 1);
      \draw[mesh-stratum] (0, 2) -- (2, 2);
      \draw[mesh-stratum] (0, 3) -- (2, 3);
      \draw[mesh-stratum] (0, 6) -- (2, 4);
      \draw[mesh-stratum] (0, 6) -- (2, 5);
    \end{tikzpicture}
    \quad
    \longrightarrow
    \quad
    \begin{tikzpicture}[scale = 0.5, baseline=(current bounding box.center)]
      \draw[mesh-stratum] (0, 0) -- (0, 2);
      \node[mesh-vertex] at (0, 2) {};
    \end{tikzpicture}
  \]
  Here we use that the framing on an open $n$-mesh bundle is open, so that
  we can gradually move the strata of $\strat{M}_1$ which are not in the image
  of $\varphi$ to infinity on either side.
\end{example}

\begin{definition}
  Let $f : \strat{M} \pto \strat{N}$ be a bordism of open $n$-meshes
  represented by an open $n$-mesh bundle $F : \strat{E} \to \DeltaStrat{1}$.
  Then $f$ is \defn{inert} if $\stratPos{F} : \stratPos{\strat{E}} \to \ord{1}$
  is the cocartesian fibration associated to an inclusion of posets
  $\stratPos{\strat{M}} \hookrightarrow \stratPos{\strat{N}}$.
\end{definition}

\begin{lemma}\label{lem:fct-emb-open-dual}
  A bordism of open $n$-meshes is inert if and only if it is sent to an inert
  bordism of closed $n$-meshes by the duality equivalence $\BMeshOpen{n} \simeq \BMeshClosed{n, \op}$.
\end{lemma}
\begin{proof}
  Suppose that $F : \strat{E} \to \DeltaStrat{1}$ is an open $n$-mesh bundle.
  Let $\Xi : \strat{X} \to \StratReal{\Delta\ord{1}^\op}$ be the closed $n$-mesh bundle
  that is dual to $F$ via the equivalence $\BMeshOpen{n} \simeq \BMeshClosed{n, \op}$.
  We then have $\stratPos{F} = \stratPos{\Xi}$.
  Therefore, $F$ represents an inert bordism of open $n$-meshes precisely when
  $\Xi$ is an inert bordism of closed $n$-meshes.
\end{proof}

\begin{cor}
  Inert bordisms of open $n$-meshes are closed under composition in $\BMeshOpen{n}$.
\end{cor}
\begin{proof}
  Via Lemma~\ref{lem:fct-emb-open-dual} this follows by duality from Lemma~\ref{lem:fct-emb-closed-composite}.
\end{proof}

We can use also use duality to construct open mapping cylinders for every constructible
inclusion between open $n$-meshes, which then represent the appropriate inert bordisms.
Using duality for this purpose is convenient since the resulting open $n$-mesh bundle
does not have the $n$-framing from Construction~\ref{con:fct-ref-ocyl-framing}.
Rather, as we have seen in Example~\ref{ex:fct-emb-open} above, the $n$-framing
for an inert bordism of open $n$-meshes sends the parts that are not in the image
of the embedding off to infinity on the sides.

\begin{construction}
  Suppose that $\strat{M}$ is an open $n$-mesh and $\strat{N} \subseteq \strat{M}$
  a constructible subspace that is also an open $n$-mesh.
  We can then construct an inert map $\strat{N} \pto \strat{M}$ as follows.
  Via duality between open and closed $n$-meshes, the dual $\strat{N}^\dagger$
  is a constructible subspace of $\strat{M}^\dagger$.
  Then by Construction~\ref{con:fct-emb-closed} we obtain an inert bordism of closed $n$-meshes
  $\strat{M}^\dagger \pto \strat{N}^\dagger$.
  Applying duality again then yields the inert bordism of open $n$-meshes
  $\strat{N} \pto \strat{M}$.
\end{construction}

\subsection{Active Bordisms and Factorisations}\label{sec:fct-embed-active}

\begin{definition}
  Let $f : \strat{X} \pto \strat{Y}$ be a bordism of closed $n$-meshes
  represented by a closed $n$-mesh bundle $\xi : \strat{E} \to \DeltaStrat{1}$.
  Then $f$ is \defn{active} if for every $p \in \stratPos{\strat{X}}$
  there exists a $q \in \stratPos{\strat{Y}}$ such that $p \leq q$ in
  $\stratPos{\strat{E}}$.
  A bordism of closed $n$-meshes with labels in a quasicategory $\cat{C}$
  is \defn{active} if it is active after forgetting the labels.
\end{definition}

\begin{lemma}\label{lem:fct-emb-active-compose-closed}
  Let $\cat{C}$ be a quasicategory.
  Active bordisms of closed $n$-meshes with labels in $\cat{C}$
  are closed under composition in
  the $\infty$-category $\BMeshClosedL{n}{\cat{C}}$.
\end{lemma}
\begin{proof}
  Suppose that $\xi : \strat{E} \to \DeltaStrat{2}$ is a closed $n$-mesh bundle
  which represents the composition of two active bordisms:
  \[
    \begin{tikzcd}[column sep = small]
      {} &
      {\strat{X}_1} \ar[dr, mid vert, "f_{12}"] \\
      {\strat{X}_0} \ar[rr, mid vert, "f_{02}"] \ar[ur, mid vert, "f_{01}"] &
      {} &
      {\strat{X}_2}
    \end{tikzcd}
  \]
  Let $p_0 \in \stratPos{\strat{X}_0}$ be an element.
  Because $f_{01}$ is active, there exists an element $p_1 \in \stratPos{\strat{X}_1}$
  such that $p_0 \leq p_1$ in $\stratPos{\strat{E}}$.
  Further because $f_{12}$ is active, there is an element $p_2 \in \stratPos{\strat{X}_2}$
  such that $p_1 \leq p_2$ in $\stratPos{\strat{E}}$.
  Therefore, $p_0 \leq p_2$ and $f_{02}$ is active.
\end{proof}


\begin{lemma}\label{lem:fct-emb-closed-smallest-subbundle}
  Let $\xi : \strat{X} \to \strat{B}$ be a closed $n$-mesh bundle
  and $Q \subseteq \stratPos{\strat{X}}$ a subposet.
  Then there exists a smallest closed $n$-mesh bundle $\xi' : \strat{X}' \to \strat{B}'$ with a
  closed constructible embedding into $\xi$ such that $Q \subseteq \stratPos{\strat{X}'} \subseteq \stratPos{\strat{X}}$.
\end{lemma}
\begin{proof}
  Without loss of generality we can assume that $Q$ is downwards closed.
  When $n = 0$ then $\xi$ is the identity map of $\strat{X} = \strat{B}$
  and $\xi'$ is the identity map of the constructible subspace $\strat{B}' \subseteq \strat{B}$
  with $\stratPos{\strat{B}'} = Q \subseteq \stratPos{\strat{B}}$.
  Now suppose that we have $n > 0$.
  We can then write $\xi$ as the composite of a closed $1$-mesh bundle
  $\chi : \strat{X} \to \strat{Y}$ followed by a closed $(n - 1)$-mesh bundle
  $\zeta : \strat{Y} \to \strat{X}$.  
  By induction on $n$ there exists a smallest closed $(n - 1)$-mesh bundle
  $\zeta' : \strat{Y}' \to \strat{B}'$ together with a closed constructible embedding
  \[
    \begin{tikzcd}
      {\strat{Y}'} \ar[r, hook] \ar[d, "\zeta'"] &
      {\strat{Y}} \ar[d, "\zeta"'] \\
      {\strat{B}'} \ar[r, hook] &
      {\strat{B}}
    \end{tikzcd}
  \]
  such that $\stratPos{\strat{Y}'}$ contains the subposet
  $Q' = \stratPos{\chi}(Q) \subseteq \stratPos{\strat{Y}}$.
  Let $\strat{X}'' \subseteq \strat{X}$ be the closed constructible subspace
  such that $\stratPos{\strat{X}''} = Q \subseteq \stratPos{\strat{X}}$.
  Then the closed $1$-mesh bundle $\chi$ restricts to the $1$-framed stratified
  subbundle $\chi''$ in the diagram
  \[
    \begin{tikzcd}
      {\strat{X}''} \ar[d, "\chi''"] \ar[r, hook] &
      {\strat{X}} \ar[d, "\chi"'] \ar[dd, bend left, "\xi"] \\
      {\strat{Y}'} \ar[r, hook] \ar[d, "\zeta'"] &
      {\strat{Y}} \ar[d, "\zeta"'] \\
      {\strat{B}'} \ar[r, hook] &
      {\strat{B}}
    \end{tikzcd}
  \]
  The $1$-framed stratified bundle $\chi''$ is not a closed $1$-mesh bundle
  in general: The fibres of $\chi''$ are closed constructible subspaces
  of closed $1$-meshes but may fail to be closed $1$-meshes themselves
  by being disconnected. We therefore let $\chi' : \strat{X}' \to \strat{Y}'$
  be the fibrewise convex closure of $\chi'' : \strat{X}' \to \strat{Y}'$.  
  Then $\chi'$ is a closed $1$-mesh bundle we can let $\xi' = \zeta' \circ \chi'$.
\end{proof}

\begin{lemma}\label{lem:fct-emb-active-factorisation-closed}
  The inert and active bordisms of closed $n$-meshes form an orthogonal
  factorisation system on $\BMeshClosed{n}$.
\end{lemma}
\begin{proof}
  Let $f : \strat{X} \pto \strat{Y}$ be a bordism of closed $n$-meshes.
  We let $\xi : \strat{E} \to \DeltaStrat{2}$ be the closed $n$-mesh bundle
  that corresponds to the commutative triangle of bordisms
  \[
    \begin{tikzcd}[column sep = small]
      {} &
      {\strat{X}} \ar[dr, mid vert, "f"] \\
      {\strat{X}} \ar[rr, mid vert, "f"'] \ar[ur, mid vert, "\id"] &
      {} &
      {\strat{Y}}
    \end{tikzcd}
  \]
  We then use Lemma~\ref{lem:fct-emb-closed-smallest-subbundle} and
  let $\xi' : \strat{E}' \to \DeltaStrat{2}$ be the smallest
  closed $n$-mesh bundle contained within $\xi$ such that $\stratPos{\strat{E}'} \subseteq \stratPos{\strat{E}}$
  contains the subposet
  \begin{equation}\label{eq:fct-emb-active-factorisation-closed:subposet}
    \{ p \mid \exists q \in \stratPos{\strat{Y}}.\ \stratPos{\xi}(q) = 2 \land p \leq q \}
    \cup
    \{ p \mid \stratPos{\xi}(p) = 0 \} \subseteq \stratPos{\strat{E}}.
  \end{equation}
  The closed $n$-mesh bundle $\xi'$ then represents a diagram of closed $n$-mesh bordisms
  \[
    \begin{tikzcd}[column sep = small]
      {} &
      {\strat{Z}} \ar[dr, mid vert, "f_1"] \\
      {\strat{X}} \ar[rr, mid vert, "f"'] \ar[ur, mid vert, "f_0"] &
      {} &
      {\strat{Y}}
    \end{tikzcd}
  \]
  where $f_0$ is inert and $f_1$ is active by construction.

  Now suppose that we have another factorisation of $f$ into an inert map
  $f_0' : \strat{X} \pto \strat{Z}'$ followed by an active map
  $f_1' : \strat{Z}' \pto \strat{Y}$.
  Let $\xi'' : \strat{E}'' \to \DeltaStrat{2}$ be the closed $2$-mesh bundle
  which represents this factorisation.
  Then there is a natural inert map between the commutative triangles
  of closed $n$-mesh bordisms
  \[
    \begin{tikzcd}
    	{\strat{X}} &&& {\strat{X}} \\
    	& {\strat{Z}'} & {\strat{X}} \\
    	{\strat{Y}} &&& {\strat{Y}}
    	\arrow["f"', mid vert, from=1-1, to=3-1]
    	\arrow["{f_0'}", mid vert, from=1-1, to=2-2]
    	\arrow["{f_1'}", mid vert, from=2-2, to=3-1]
    	\arrow[equal, from=1-4, to=2-3]
    	\arrow["f", mid vert, from=1-4, to=3-4]
    	\arrow["f"', mid vert, from=2-3, to=3-4]
    	\arrow[from=3-4, to=3-1, equal]
    	\arrow[from=1-4, to=1-1, equal]
    	\arrow["{f_0'}", mid vert, from=2-3, to=2-2]
    \end{tikzcd}
  \]
  and therefore $\xi''$ is a constructible subbundle of $\xi$.
  Because $f_1'$ is required to be active, we see that the subposet
  $\stratPos{\strat{E}''} \subseteq \stratPos{\strat{E}}$ must agree with
  the subposet~(\ref{eq:fct-emb-active-factorisation-closed:subposet}).
  Therefore, $\xi''$ and $\xi'$ must agree by construction of $\xi'$.  
\end{proof}

\begin{proposition}\label{prop:fct-emb-active-factorisation-closed}
  Let $\cat{C}$ be a quasicategory.
  Then the inert and active bordisms of closed $n$-meshes with labels in $\cat{C}$
  form an orthogonal factorisation system on the $\infty$-category $\BMeshClosedL{n}{\cat{C}}$.
\end{proposition}
\begin{proof}
  Let $f : \strat{X} \pto \strat{Y}$ be a bordism of closed $n$-meshes
  with labels in $\cat{C}$ and let $\pi : \BMeshClosedL{n}{\cat{C}} \to \BMeshClosed{n}$
  denote the forgetful functor.
  By Proposition~\ref{lem:fct-emb-active-factorisation-closed} the unlabelled bordism
  $\pi(f)$ factors essentially uniquely into an inert bordism
  $f_0 : \pi(\strat{X}) \pto \strat{Z}$ followed by an active bordism
  $f_1 : \strat{Z} \pto \pi(\strat{X})$.
  By Lemma~\ref{lem:fct-emb-closed-inert-map-lift} the forgetful functor
  $\pi$ has cocartesian lifts of inert bordisms and so we can find
  an essentially unique inert bordism $\hat{f}_0 : \strat{X} \pto \hat{\strat{Z}}$
  with $\pi(\hat{f}_0) = f_0$. By the universal property of cocartesian
  maps, we then also have an essentially unique bordism $\hat{f}_1 : \hat{\strat{Z}} \pto \strat{Y}$ such that $\pi(\hat{f}_1) = f_1$ and $f$ is the composite
  of $\hat{f}_0$ followed by $\hat{f}_1$.
  This is the essentially unique inert/active factorisation of $f$.
\end{proof}

The inert bordisms of open $n$-meshes are precisely the duals of the inert bordisms
of closed $n$-meshes. We can therefore use the duality equivalence $\BMeshOpenL{n}{\cat{C}} \simeq \BMeshClosedL{n}{\cat{C}^\op}^\op$ to obtain an active/inert factorisation system on $\BMeshOpenL{n}{\cat{C}}$ for each quasicategory $\cat{C}$.

\begin{definition}
  Let $f : \strat{M} \pto \strat{N}$ be a bordism of open $n$-meshes
  represented by an open $n$-mesh bundle $F : \strat{E} \to \DeltaStrat{1}$.
  Then $f$ is \defn{active} if for every $q \in \stratPos{\strat{N}}$
  there exists a $p \in \stratPos{\strat{M}}$ such that $p \leq q$ in
  $\stratPos{\strat{E}}$.
  A bordism of open $n$-meshes with labels in a quasicategory $\cat{C}$
  is \defn{active} if it is active after forgetting the labels.
\end{definition}

\begin{lemma}\label{lem:fct-emb-active-dual}
  Let $\cat{C}$ be a quasicategory.
  A bordism of open $n$-meshes with labels in $\cat{C}$ is active if and only if it is sent to an active
  bordism of closed $n$-meshes by the duality equivalence $\BMeshOpenL{n}{\cat{C}} \simeq \BMeshClosedL{n}{\cat{C}^\op}^\op$.
\end{lemma}
\begin{proof}
  Analogous to Lemma~\ref{lem:fct-emb-open-dual}.
\end{proof}

\begin{cor}
  Let $\cat{C}$ be a quasicategory.
  Active bordisms of open $n$-meshes with labels in $\cat{C}$ are closed under composition in $\BMeshOpenL{n}{\cat{C}}$.
\end{cor}
\begin{proof}
  Using Lemma~\ref{lem:fct-emb-active-dual} and Lemma~\ref{lem:fct-emb-active-compose-closed}.
\end{proof}

\begin{cor}\label{cor:fct-emb-active-factorisation-open}
  Let $\cat{C}$ be a quasicategory.
  The active and inert bordisms of open $n$-meshes with labels in $\cat{C}$ form an orthogonal factorisation system
  on the $\infty$-category $\BMeshOpenL{n}{\cat{C}}$.
\end{cor}
\begin{proof}
  Using Lemma~\ref{lem:fct-emb-active-dual} and Proposition~\ref{prop:fct-emb-active-factorisation-closed}.
\end{proof}

\subsection{Atoms and Cells}\label{sec:fct-embed-atoms-cells}

\begin{definition}
  An open $n$-mesh $\strat{M}$ is an \defn{atom} when $\stratPos{\strat{M}}$
  has a minimal element. A closed $n$-mesh $\strat{X}$ is a \defn{cell} when
  $\stratPos{\strat{X}}$ has a maximal element.
\end{definition}

\begin{example}
  An atom and its dual cell for $n = 2$:
  \[
    \begin{tikzpicture}[scale = 0.5, baseline=(current bounding box.center)]
      \fill[mesh-background] (0, 0) -- (0, 4) -- (4, 4) -- (4, 0) -- cycle;
      \node[mesh-vertex] at (2, 2) {};
      \draw[mesh-stratum] (0, 2) -- (4, 2);
      \draw[mesh-stratum] (2, 2) -- (2, 4);
      \draw[mesh-stratum] (2, 2) -- (1, 0);
      \draw[mesh-stratum] (2, 2) -- (3, 0);
    \end{tikzpicture}
    \qquad
    \overset{\dagger}{\mapsto}
    \qquad
    \begin{tikzpicture}[scale = 0.5, baseline=(current bounding box.center)]
      \fill[mesh-background] (0, 0) -- (0, 4) -- (4, 4) -- (4, 0) -- cycle;
      \draw[mesh-stratum] (0, 0) -- (0, 4) -- (4, 4) -- (4, 0) -- cycle;
      \node[mesh-vertex] at (0, 0) {};
      \node[mesh-vertex] at (0, 4) {};
      \node[mesh-vertex] at (4, 4) {};
      \node[mesh-vertex] at (4, 0) {};
      \node[mesh-vertex] at (2, 0) {};
    \end{tikzpicture}
  \]
\end{example}

\begin{example}  
  Cells do not need to be square shaped:
  \[
    \begin{tikzpicture}[scale = 0.5, baseline=(current bounding box.center)]
      \fill[mesh-background] (0, 0) -- (0, 4) -- (4, 4) -- (4, 0) -- cycle;
      \node[mesh-vertex] at (2, 2) {};
      \draw[mesh-stratum] (0, 2) -- (4, 2);
      \draw[mesh-stratum] (2, 2) -- (2, 4);
    \end{tikzpicture}
    \qquad
    \overset{\dagger}{\mapsto}
    \qquad
    \begin{tikzpicture}[scale = 0.5, baseline=(current bounding box.center)]
      \fill[mesh-background] (2, 0) -- (0, 4) -- (4, 4) -- cycle;
      \draw[mesh-stratum] (2, 0) -- (0, 4) -- (4, 4) -- cycle;
      \node[mesh-vertex] at (0, 4) {};
      \node[mesh-vertex] at (4, 4) {};
      \node[mesh-vertex] at (2, 0) {};
    \end{tikzpicture}
  \]
\end{example}

\begin{example}
  Via the inclusion $\gridMeshOpen : \FinOrd^n \hookrightarrow \BMeshOpen{n}$
  from Construction~\ref{con:fct-mesh-grid-open}, every object $\ord{k_1, \ldots, k_n} \in \FinOrd^\op$ induces an open $n$-mesh.
  This open $n$-mesh is an atom exactly when $0 \leq k_i \leq 1$ for all
  $1 \leq i \leq n$.
  For $n = 2$ these atoms look as follows:
  \begin{align*}
    \gridMeshOpen\ord{1, 1} \quad&\simeq\quad
    \begin{tikzpicture}[scale = 0.5, baseline=(current bounding box.center)]
      \fill[mesh-background] (0, 0) -- (0, 4) -- (4, 4) -- (4, 0) -- cycle;
      \node[mesh-vertex] at (2, 2) {};
      \draw[mesh-stratum] (0, 2) -- (4, 2);
      \draw[mesh-stratum] (2, 2) -- (2, 4);
      \draw[mesh-stratum] (2, 2) -- (2, 0);
    \end{tikzpicture}
    &
    \gridMeshOpen\ord{0, 1} \quad&\simeq\quad
    \begin{tikzpicture}[scale = 0.5, baseline=(current bounding box.center)]
      \fill[mesh-background] (0, 0) -- (0, 4) -- (4, 4) -- (4, 0) -- cycle;
      \draw[mesh-stratum] (0, 2) -- (4, 2);
    \end{tikzpicture}
    \\
    \gridMeshOpen\ord{1, 0} \quad&\simeq\quad
    \begin{tikzpicture}[scale = 0.5, baseline=(current bounding box.center)]
      \fill[mesh-background] (0, 0) -- (0, 4) -- (4, 4) -- (4, 0) -- cycle;
      \draw[mesh-stratum] (2, 0) -- (2, 4);
    \end{tikzpicture}
    &
    \gridMeshOpen\ord{0, 0} \quad&\simeq\quad
    \begin{tikzpicture}[scale = 0.5, baseline=(current bounding box.center)]
      \fill[mesh-background] (0, 0) -- (0, 4) -- (4, 4) -- (4, 0) -- cycle;
    \end{tikzpicture}
  \end{align*}
  Via Construction~\ref{con:fct-mesh-grid-closed} we also have an inclusion
  $\gridMeshClosed : \FinOrd^{n, \op} \hookrightarrow \BMeshClosed{n}$.
  For any $\ord{k_1, \ldots, k_n} \in \FinOrd^n$ the closed $n$-mesh
  $\gridMeshClosed\ord{k_1, \ldots, k_n}$ is a cell if and only if 
  $0 \leq k_i \leq 1$ for all $1 \leq i \leq n$.
  For $n = 2$ we have the following cells:
  \begin{align*}
    \gridMeshClosed\ord{1, 1} \quad&\simeq\quad
    \begin{tikzpicture}[scale = 0.5, baseline=(current bounding box.center)]
      \fill[mesh-background] (0, 0) -- (0, 4) -- (4, 4) -- (4, 0) -- cycle;
      \draw[mesh-stratum] (0, 0) -- (0, 4) -- (4, 4) -- (4, 0) -- cycle;
      \node[mesh-vertex] at (0, 0) {};
      \node[mesh-vertex] at (4, 0) {};
      \node[mesh-vertex] at (0, 4) {};
      \node[mesh-vertex] at (4, 4) {};
    \end{tikzpicture}
    &
    \gridMeshClosed\ord{0, 1} \quad&\simeq\quad
    \begin{tikzpicture}[scale = 0.5, baseline=(current bounding box.center)]
      \path [use as bounding box] (0, 0) rectangle (4, 4);
      \draw[mesh-stratum] (2, 0.5) -- (2, 3.5);
      \node[mesh-vertex] at (2, 0.5) {};
      \node[mesh-vertex] at (2, 3.5) {};
    \end{tikzpicture}
    \\
    \gridMeshClosed\ord{1, 0} \quad&\simeq\quad
    \begin{tikzpicture}[scale = 0.5, baseline=(current bounding box.center)]
      \path [use as bounding box] (0, 0) rectangle (4, 4);
      \draw[mesh-stratum] (0.5, 2) -- (3.5, 2);
      \node[mesh-vertex] at (0.5, 2) {};
      \node[mesh-vertex] at (3.5, 2) {};
    \end{tikzpicture}
    &
    \gridMeshClosed\ord{0, 0} \quad&\simeq\quad
    \begin{tikzpicture}[scale = 0.5, baseline=(current bounding box.center)]
      \path [use as bounding box] (0, 0) rectangle (4, 4);
      \node[mesh-vertex] at (2, 2) {};
    \end{tikzpicture}
  \end{align*}
\end{example}

\begin{lemma}
  The duality equivalence $\BMeshOpen{n} \simeq (\BMeshClosed{n})^\op$ 
  between open and closed $n$-meshes restricts
  to an equivalence between the full subcategories of atoms and cells.
\end{lemma}
\begin{proof}
  When $\strat{M}$ is an open $n$-mesh and $\strat{X}$ a closed $n$-mesh
  such that $\strat{M}^\dagger \simeq \strat{X}$, we have that
  $\stratPos{\strat{M}} \cong \stratPos{\strat{X}}^\op$.
  In particular $\stratPos{\strat{M}}$ has a minimal element if and only if
  $\stratPos{\strat{X}}$ has a maximal element.
\end{proof}

\begin{remark}
  When $\strat{M}$ is an open $n$-mesh, the map $\Exit(\strat{M}) \to \stratPos{\strat{M}}$
  is an equivalence. Therefore, $\strat{M}$ is an atom if and only if $\Exit(\strat{M})$
  has an initial object.
  Analogously, a closed $n$-mesh $\strat{X}$ is a cell if and only if $\Exit(\strat{X})$
  has a terminal object.
\end{remark}

Whenever $\strat{X}$ is a closed $n$-mesh, every element of the poset
$\stratPos{\strat{X}}$ corresponds to a cell that is embedded in $\strat{X}$.
Keeping in mind that the inert bordisms of closed $n$-meshes are oriented
in the opposite direction to embeddings, the assignment from $\stratPos{\strat{X}}$
to the corresponding cell of $\strat{X}$ induces a functor into the slice category
$\stratPos{\strat{X}} \to \BMeshClosed{n}_{\strat{X}/}$,
consisting of inert bordisms out of $\strat{X}$ and in between the cells of $\strat{X}$.

\begin{construction}
  Let $\strat{X}$ be a closed $n$-mesh.
  For any element $p \in \stratPos{\strat{X}}$ we can find the smallest
  closed constructible $n$-submesh $\cell{\strat{X}}(p)$ of $\strat{X}$ such
  that $\stratPos{\cell{\strat{X}}(p)} \subseteq \stratPos{\strat{X}}$ 
  contains $p$.
  By varying $p \in \stratPos{\strat{X}}^\op$ and sending the initial object
  of the left cone $\bot \in \stratPos{\strat{X}}^{\op, \triangleleft}$
  to $\strat{X}$ itself we obtain the \defn{cellular covering functor}
  \[ \cell{\strat{X}}(-) : \stratPos{\strat{X}}^{\op, \triangleleft} \longrightarrow \BMeshClosed{n}. \]
  Explicitly this functor is the classifying map of the closed $n$-mesh bundle
  $\strat{E} \to \StratReal{\stratPos{\strat{X}}^{\op, \triangleleft}}$
  where $\strat{E}$ is the constructible subspace
  of $\strat{X} \times \StratReal{\stratPos{\strat{X}}^{\op, \triangleleft}}$
  given by the subposet
  \[
    \stratPos{\strat{E}} :=
    \{
      (p, q) \mid
      p \leq q
    \}
    \cup
    \{
      (p, \bot)
    \}
    \subseteq
    \stratPos{\strat{X}} \times
    \stratPos{\strat{X}}^{\op, \triangleleft}.
  \]
\end{construction}

\begin{lemma}\label{lem:fct-cell-covering-limit}
  Let $\strat{X}$ be a closed $n$-mesh. Then cellular covering functor
  \[
    \cell{\strat{X}}(-) : \stratPos{\strat{M}}^{\op, \triangleleft} \longrightarrow \BMeshClosed{n}
  \]
  is a limit diagram in $\BMeshClosed{n}$.
\end{lemma}
\begin{proof}
  Suppose that we have a closed $n$-mesh $\strat{Y}$ and a compatible family
  of maps $\strat{Y} \pto \cell{\strat{X}}(x)$ for all $x \in \stratPos{\strat{X}}$,
  represented by closed $n$-mesh bundles $f_x : \strat{F}_x \to \DeltaStrat{1}$.
  By induction on $\stratPos{\strat{X}}$ we can adjust these closed $n$-mesh bundles
  up to equivalence so that $f_x$ embeds into $f_y$ for all elements $x \leq y$ in $\stratPos{\strat{X}}$.
  Then we obtain a closed $n$-mesh bundle $f : \strat{F} \to \DeltaStrat{1}$
  representing a map $\strat{Y} \pto \strat{X}$ which is compatible with the maps
  $\strat{Y} \pto \cell{\strat{X}}(x)$.
  Since $f$ is unique up to equivalence this is the universal map into the limit.
\end{proof}

\begin{construction}
  Let $\strat{M}$ be an open $n$-mesh.
  For any element $p \in \stratPos{\strat{M}}$ we can find the smallest
  open constructible $n$-submesh $\atom{\strat{M}}(p)$ of $\strat{M}$ such
  that $\stratPos{\atom{\strat{M}}(p)} \subseteq \stratPos{\strat{M}}$ 
  contains $p$.
  By varying $p \in \stratPos{\strat{M}}$ and sending the terminal object
  of the right cone $\bot \in \stratPos{\strat{M}}^{\triangleright}$
  to $\strat{M}$ itself we obtain the \defn{atomic covering functor}
  \[ \atom{\strat{M}}(-) : \stratPos{\strat{M}}^{\triangleright} \longrightarrow \BMeshOpen{n}. \]
  Explicitly this functor is the classifying map of the open $n$-mesh bundle
  $\strat{E} \to \StratReal{\stratPos{\strat{M}}^{\triangleright}}$
  where $\strat{E}$ is the constructible subspace
  of $\strat{X} \times \StratReal{\stratPos{\strat{M}}^{\triangleright}}$
  given by the subposet
  \[
    \stratPos{\strat{E}} :=
    \{
      (p, q) \mid
      p \geq q
    \}
    \cup
    \{
      (p, \top)
    \}
    \subseteq
    \stratPos{\strat{M}} \times
    \stratPos{\strat{M}}^{\triangleright}.
  \]
\end{construction}

\begin{observation}
  When $\strat{M}$ is an open $n$-mesh and $\strat{X} = \strat{M}^\dagger$ its dual
  closed $n$-mesh, then the atomic covering diagram of $\strat{X}$ and the cellular
  covering diagram of $\strat{M}$ are dual to each other as well:
  \[
    \begin{tikzcd}
      {\stratPos{\strat{M}}^{\triangleright, \op}} \ar[d, "\atom{\strat{M}}"'] \ar[r, "\cong"] &
      {\stratPos{\strat{X}}^{\op, \triangleleft}} \ar[d, "\cell{\strat{X}}"] \\
      {\BMeshOpen{n, \op}} \ar[r, "-^\dagger"'] &
      {\BMeshClosed{n}}
    \end{tikzcd}
  \]
\end{observation}

\begin{lemma}\label{lem:fct-atom-covering-colimit}
  Let $\strat{M}$ be an open $n$-mesh. Then the atomic covering functor
  \[
    \atom{\strat{M}}(-) : \stratPos{\strat{M}}^{\op, \triangleright} \longrightarrow \BMeshOpen{n}
  \]
  is a colimit diagram in $\BMeshOpen{n}$.
\end{lemma}
\begin{proof}
  Using Lemma~\ref{lem:fct-cell-covering-limit} and duality.
\end{proof}

\begin{lemma}\label{lem:fct-degeneracy-preserve-atom}
  Let $f : \strat{A} \pto \strat{M}$ be a degeneracy map of open $n$-meshes
  so that $\strat{A}$ is an atom. Then $\strat{M}$ is again an atom.
  Moreover when $\varphi : \strat{A} \to \strat{M}$ is a refinement map so
  that $f$ is the open mappying cylinder of $\varphi$, then the induced map
  $\stratPos{\varphi}$
  sends the minimal element of $\stratPos{\strat{A}}$ to minimal element of
  $\stratPos{\strat{M}}$.
\end{lemma}
\begin{proof}
  Since $\varphi$ is a refinement map
  the induced map $\stratPos{\varphi}$ must be surjective.
  The claim then follows since surjective order-preserving maps must preserve the minimal element.  
\end{proof}

\subsection{Singularity Types}\label{sec:fct-embed-stype}

\begin{definition}
  Let $\strat{X}$ be an $n$-framed stratified space that is either an atom or a cell.
  We then have a sequence of projection maps
  \[
    \begin{tikzcd}
      \strat{X} = \strat{X}_n \ar[r, "f_n"] &
      \strat{X}_{n - 1} \ar[r] &
      \cdots \ar[r] &
      \strat{X}_i \ar[r, "f_1"] & 
      \strat{X}_0 = \DeltaStrat{0}
    \end{tikzcd}
  \]
  which are each $1$-framed stratified bundles.
  The \defn{singularity type} of $\strat{X}$ is the object
  \[ \stype(\strat{X}) := \ord{k_1, \ldots, k_n} \in \FinOrd^n \]
  such that $k_i = 0$ if 
  $f_i : \strat{X}_i \to \strat{X}_{i - 1}$
  is an $n$-framed stratified trivial bundle
  and $k_i = 1$ otherwise.
  The \defn{singular depth} of $\strat{X}$ is the largest $\sdepth(\strat{X}) := d \in \ord{k}$
  such that $f_i : \strat{X}_i \to \strat{X}_{i - 1}$ is an $n$-framed stratified
  trivial bundle for all $1 \leq i \leq d$.
\end{definition}

\begin{observation}\label{obs:fct-embed-stype-depth}
  Suppose that $\strat{A}$ is an atom.
  Then the singular depth of $\strat{A}$ is equivalently the largest $d \geq 0$
  such that there exists an $(n - d)$-framed stratified space $\strat{B}$
  and an isomorphism $\strat{B} \frtimes \R^d \cong \strat{A}$ of $n$-framed
  stratified spaces.
\end{observation}

\begin{remark}
  Suppose that $\strat{M}$ is an open $n$-mesh whose strata are smooth
  submanifolds of $\R^n$ via the $n$-framing.   
  Let $p \in \stratPos{\strat{M}}$ be a stratum and $\tau : \strat{M}_p \to \textup{Gr}_k(\R^n)$
  the continuous map into the Grassmannian which classifies the tangent bundle of
  $\strat{M}_p$ as embedded into $\R^n$.
  Then the singularity type of the atom $\atom{\strat{M}}(p)$ corresponds to the
  Schubert cell of $\textup{Gr}_k(\R^n)$ through which $\tau$ factors.
\end{remark}

\begin{construction}
  The \defn{singular shape} of an atom $\strat{A}$ is the atom
  \[ \shape{\strat{A}} := \gridMeshOpen(\stype(\strat{A})). \]
  Analogously the \defn{singular shape} of a cell $\strat{X}$ is the cell
  \[ \shape{\strat{X}} := \gridMeshClosed(\stype(\strat{X})). \]
\end{construction}

\begin{construction}
  Let $\strat{A}$ be an $n$-mesh atom such that $0 \in \Exit(\strat{A})$.
  We construct an active map of $n$-mesh atoms $\shape{\strat{A}} \pto \strat{A}$
  such that $\shape{\strat{A}}$ is the geometric realisation of the open $n$-truss $\stype(\strat{A})$.
  We first equip the set $\{ -1, 0, +1 \}$ with the smallest partial order $\preceq$
  such that $0 \preceq -1$ and $0 \preceq +1$.
  For every $1 \leq i \leq n$ we define a poset $P_i$ by
  \begin{alignat*}{2}
    P_i & := (\{ -1, 0, +1 \}, \preceq) && \qquad\text{(if $\stype(A)_i = \ord{1}$)} \\
    P_i & := \{ 0 \} && \qquad\text{(if $\stype(A)_i = \ord{0}$)}
  \end{alignat*}
  Then we realise $\shape{\strat{A}}$ as the stratification of $\R^n$
  with poset $\stratPos{\shape{\strat{A}}} = P_1 \times \cdots \times P_n$
  so that the stratifying map is defined in components by
  \begin{alignat*}{2}
    \stratMap{\shape{\strat{A}}}(x)_i & := \sgn(x_i) && \qquad\text{(if $\stype(A)_i = \ord{1}$)} \\
    \stratMap{\shape{\strat{A}}}(x)_i & := 0 && \qquad\text{(if $\stype(A)_i = \ord{0}$)}
  \end{alignat*}
  for every $1 \leq i \leq n$.
  To construct the map $\shape{\strat{A}} \pto \strat{A}$,
  let $\strat{M}$
  be the unique stratification of $\R^n \times \DeltaTop{1}$
  so that $\stratPos{\strat{M}}$ contains $\stratPos{\strat{A}} \sqcup \stratPos{\shape{\strat{A}}}$ and
  \begin{alignat*}{2}
    \stratMap{\strat{M}}(x, t) & := \stratMap{\strat{A}}(\tfrac{1}{t} x) &&
    \qquad \text{(when $0 < t \leq 1$)} \\
    \stratMap{\strat{M}}(x, t) & := \stratMap{\shape{\strat{A}}}(x) &&
  \end{alignat*}
  Then the projection map $\strat{M} \to \DeltaStrat{1}$ with the canonical
  $n$-framing defines an open $n$-mesh bundle that induces the map
  $\shape{\strat{A}} \pto \strat{A}$.
\end{construction}

\begin{example}
  \[
    \shape{\strat{A}_0} \quad = \quad
    \begin{tikzpicture}[scale = 0.5, baseline=(current bounding box.center)]
      \fill[mesh-background] (0, 0) -- (0, 4) -- (4, 4) -- (4, 0) -- cycle;
      \node[mesh-vertex] at (2, 2) {};
      \draw[mesh-stratum] (0, 2) -- (4, 2);
      \draw[mesh-stratum] (2, 2) -- (2, 4);
      \draw[mesh-stratum] (2, 2) -- (2, 0);
    \end{tikzpicture}
    \quad
    \pto
    \quad
    \begin{tikzpicture}[scale = 0.5, baseline=(current bounding box.center)]
      \fill[mesh-background] (0, 0) -- (0, 4) -- (4, 4) -- (4, 0) -- cycle;
      \node[mesh-vertex] at (2, 2) {};
      \draw[mesh-stratum] (0, 2) -- (4, 2);
      \draw[mesh-stratum] (2, 2) -- (1, 0);
      \draw[mesh-stratum] (2, 2) -- (3, 0);
    \end{tikzpicture}
    \quad = \quad \strat{A}_0
  \]

  \[
    \shape{\strat{A}_1} \quad = \quad
    \begin{tikzpicture}[scale = 0.5, baseline=(current bounding box.center)]
      \fill[mesh-background] (0, 0) -- (0, 4) -- (4, 4) -- (4, 0) -- cycle;
      \draw[mesh-stratum] (2, 0) -- (2, 4);
    \end{tikzpicture}
    \quad
    \pto
    \quad
    \begin{tikzpicture}[scale = 0.5, baseline=(current bounding box.center)]
      \fill[mesh-background] (0, 0) -- (0, 4) -- (4, 4) -- (4, 0) -- cycle;
      \draw[mesh-stratum] (2, 0) -- (1, 1) -- (3, 3) -- (2, 4);
    \end{tikzpicture}
    \quad = \quad \strat{A}_1
  \]
\end{example}

\begin{construction}\label{con:fct-embed-stype-active-closed}
  Let $\strat{C}$ be an $n$-mesh cell such that $0 \in \Exit(\strat{C})$
  is terminal and $\unstrat(\strat{C}) \subset \intOO{-1, 1}^n$.
  We can then construct a bordism of closed $n$-meshes
  $\strat{C} \pto \shape{\strat{C}}$ as follows.
  We first equip the stratified projection map $\xi' : \shape{\strat{C}} \times \intOC{0, 1} \to \intOC{0, 1}$
  with the $n$-framing induced by linear interpolation between the framings of $\shape{\strat{C}}$ and
  $\strat{C}$ via the framed projection:
  \[
    \framing(\xi')(x, t) := t \framing(\shape{\strat{C}})(x) + (1 - t) \framing(\strat{C})(\fproj(\strat{C}, x))
  \]
  Then $\xi'$ extends to an $n$-framed stratified bundle $\xi : \strat{X} \to \DeltaStrat{1}$
  which represents the closed $n$-mesh bordism $\strat{C} \pto \shape{\strat{C}}$.  
\end{construction}

\begin{example}
  \[
    \strat{C}_0 \quad = \quad
    \quad
    \begin{tikzpicture}[scale = 0.5, baseline=(current bounding box.center)]
      \fill[mesh-background] (0, 0) -- (2, 0) -- (4, 0) -- (2, 4) -- cycle;
      \node[mesh-vertex] at (0, 0) {};
      \node[mesh-vertex] at (4, 0) {};
      \node[mesh-vertex] at (2, 4) {};
      \node[mesh-vertex] at (2, 0) {};
      \draw[mesh-stratum] (0, 0) -- (2, 0) -- (4, 0) -- (2, 4) -- cycle;
    \end{tikzpicture}
    \pto
    \quad
    \begin{tikzpicture}[scale = 0.5, baseline=(current bounding box.center)]
      \fill[mesh-background] (0, 0) -- (0, 4) -- (4, 4) -- (4, 0) -- cycle;
      \node[mesh-vertex] at (0, 0) {};
      \node[mesh-vertex] at (0, 4) {};
      \node[mesh-vertex] at (4, 0) {};
      \node[mesh-vertex] at (4, 4) {};
      \draw[mesh-stratum] (0, 0) rectangle (4, 4);
    \end{tikzpicture}
    \quad = \quad \shape{\strat{C}_0}
  \]

  \[
    \strat{C}_1 \quad = \quad
    \begin{tikzpicture}[scale = 0.5, baseline=(current bounding box.center)]
      \path [use as bounding box] (0, 0) rectangle (4, 4);
      \draw[mesh-stratum] (2, 0) -- (1, 1) -- (3, 3) -- (2, 4);
      \node[mesh-vertex] at (2, 0) {};
      \node[mesh-vertex] at (2, 4) {};
    \end{tikzpicture}
    \quad
    \pto
    \quad
    \begin{tikzpicture}[scale = 0.5, baseline=(current bounding box.center)]
      \path [use as bounding box] (0, 0) rectangle (4, 4);
      \draw[mesh-stratum] (2, 0) -- (2, 4);
      \node[mesh-vertex] at (2, 0) {};
      \node[mesh-vertex] at (2, 4) {};
    \end{tikzpicture}
    \quad = \quad \shape{\strat{C}_1}
  \]
\end{example}

\section{Transverse Configurations}\label{sec:fct-conf}

\begin{definition}
  Let $\strat{B}$ be a stratified space and $(\xi, f)$ a pair consisting of a
  closed $1$-mesh bundle $\xi : \strat{X} \to \strat{B}$ together with an
  open $1$-mesh bundle $f : \strat{M} \to \strat{B}$.
  The pair $(\xi, f)$ is a \defn{transverse configuration of $1$-meshes}
  when every point $e \in \strat{X} \cap \strat{M}$
  is singular in at most one of $\strat{X}$ and $\strat{M}$.  
  In this case we write $\xi \pitchfork f$.
\end{definition}

\begin{example}
  The following is a transverse configuration of an open $1$-mesh $\strat{M}$ and a closed $1$-mesh $\strat{X}$.
  By convention we will always render the open mesh in black and the closed mesh in red.
  \[
    \begin{tikzpicture}
      \node at (-0.5, 0) {$\strat{M}$};
      \node[text=BrickRed] at (8.5, -0.25) {$\strat{X}$};
      \draw[mesh-stratum] (0, 0) -- (9, 0);
      \draw[mesh-stratum-dual] (2, -0.25) -- (8, -0.25);
      \node[mesh-vertex] at (1, 0) {};
      \node[mesh-vertex] at (3, 0) {};
      \node[mesh-vertex] at (4, 0) {};
      \node[mesh-vertex] at (7, 0) {};
      \node[mesh-vertex-dual] at (2, -0.25) {};
      \node[mesh-vertex-dual] at (5, -0.25) {};
      \node[mesh-vertex-dual] at (6, -0.25) {};
      \node[mesh-vertex-dual] at (8, -0.25) {};
    \end{tikzpicture}
  \]
  We can interpret the singular strata of the closed mesh $\strat{X}$ as parentheses
  which collect the the singular strata of the open mesh $\strat{M}$ into groups:
  \[
    \begin{tikzpicture}
      \draw[mesh-stratum] (0, 0) -- (9, 0);
      \node[mesh-vertex] at (1, 0) {};
      \node[mesh-vertex] at (3, 0) {};
      \node[mesh-vertex] at (4, 0) {};
      \node[mesh-vertex] at (7, 0) {};
      \node[text=BrickRed, font=\bfseries] at (2, 0) {(};
      \node[text=BrickRed, font=\bfseries] at (5, 0) {)(};
      \node[text=BrickRed, font=\bfseries] at (6, 0) {)(};
      \node[text=BrickRed, font=\bfseries] at (8, 0) {)};
    \end{tikzpicture}
  \]
  We have that $\sing(\strat{X}) = \ord{3}$ is the finite ordinal with one element
  for each singular stratum of $\strat{X}$. Dually we have $\reg(\strat{M}) = \ord{4}$
  corresponding to the regular strata of $\strat{M}$.
  We number the singular and regular strata of $\strat{X}$ and $\strat{M}$, respectively:
  \[
    \begin{tikzpicture}
      \node at (-0.5, 0) {$\strat{M}$};
      \node[text=BrickRed] at (8.5, -0.25) {$\strat{X}$};
      \draw[mesh-stratum] (0, 0) -- (9, 0);
      \draw[mesh-stratum-dual] (2, -0.25) -- (8, -0.25);
      \node[mesh-vertex] at (1, 0) {};
      \node[mesh-vertex] at (3, 0) {};
      \node[mesh-vertex] at (4, 0) {};
      \node[mesh-vertex] at (7, 0) {};
      \node[label=above:$0$] at (0.5, 0) {};
      \node[label=above:$1$] at (2, 0) {};
      \node[label=above:$2$] at (3.5, 0) {};
      \node[label=above:$3$] at (5.5, 0) {};
      \node[label=above:$4$] at (8, 0) {};
      \node[mesh-vertex-dual, label={[text=BrickRed]below:$0$}] at (2, -0.25) {};
      \node[mesh-vertex-dual, label={[text=BrickRed]below:$1$}] at (5, -0.25) {};
      \node[mesh-vertex-dual, label={[text=BrickRed]below:$2$}] at (6, -0.25) {};
      \node[mesh-vertex-dual, label={[text=BrickRed]below:$3$}] at (8, -0.25) {};
    \end{tikzpicture}
  \]
  A singular stratum of $\strat{X}$ contains a single point $x \in \strat{X}$.
  Transversality then guarantees that $x$ is contained in a regular stratum of $\strat{M}$.
  The transverse configuration $(\strat{X}, \strat{M})$ induces an order-preserving
  map $\alpha : \sing(\strat{X}) \to \reg(\strat{M})$ which sends the index of a singular stratum
  of $\strat{X}$ to the index of the regular stratum of $\strat{M}$ that contains its point.
  In our example, this is the map $\alpha = \langle 1, 3, 3, 4 \rangle$.
\end{example}

\begin{definition}
  Let $\strat{B}$ be a stratified space and $(\xi, f)$ a pair consisting of a
  closed $n$-mesh bundle $\xi : \strat{X}_n \to \strat{B}$ together with an
  open $n$-mesh bundle $f : \strat{M}_n \to \strat{B}$.
  We write $\xi$ and $f$ as a composite of $1$-mesh bundles
  \[
    \begin{tikzcd}[row sep = tiny]
      {\xi : \strat{X}_n} \ar[r] &
      {\strat{X}_{n - 1}} \ar[r] &
      {\cdots} \ar[r] &
      {\strat{X}_1} \ar[r] &
      {\strat{X}_0 = \strat{B}} \\
      {f : \strat{M}_n} \ar[r] &
      {\strat{M}_{n - 1}} \ar[r] &
      {\cdots} \ar[r] &
      {\strat{M}_1} \ar[r] &
      {\strat{M}_0 = \strat{B}}
    \end{tikzcd}
  \]
  The pair $(\xi, f)$ is a \defn{transverse configuration of $n$-meshes}
  when every point
  $e \in \strat{X}_i \cap \strat{M}_i$ 
  is singular in at most one of $\strat{X}_i$ and $\strat{M}_i$
  for every $1 \leq i \leq n$.
  In this case we write $\xi \pitchfork f$.
\end{definition}

\begin{example}
  This is an example of a transverse configuration of a closed $2$-mesh $\strat{X}$,
  by convention drawn in red, and an open $2$-mesh $\strat{M}$, drawn in black:
  \[
    \begin{tikzpicture}[scale = 0.6, baseline=(current bounding box.center)]
      \fill[mesh-background] (0, 0) rectangle (8, 7);
      \node[mesh-vertex-dual] at (1, 1) {};
      \node[mesh-vertex-dual] at (3, 1) {};
      \node[mesh-vertex-dual] at (6, 1) {};
      \node[mesh-vertex-dual] at (1, 3) {};
      \node[mesh-vertex-dual] at (3, 3) {};
      \node[mesh-vertex-dual] at (6, 3) {};
      \node[mesh-vertex-dual] at (1, 5) {};
      \node[mesh-vertex-dual] at (6, 5) {};
      \draw[mesh-stratum-dual] (1, 1) rectangle ++(2, 2);
      \draw[mesh-stratum-dual] (3, 1) rectangle ++(3, 2);
      \draw[mesh-stratum-dual] (1, 3) rectangle ++(5, 2);
      \draw[mesh-stratum] (1.5, 0) -- (1.5, 1) .. controls +(0, 0.2) and +(-0.5, 0) .. (2, 2);
      \draw[mesh-stratum] (2.5, 0) -- (2.5, 1) .. controls +(0, 0.2) and +(0.5, 0) .. (2, 2);
      \draw[mesh-stratum] (2, 2) -- (2, 3) .. controls +(0, 0.2) and +(-1, 0) .. (3, 4);
      \draw[mesh-stratum] (4, 2) -- (4, 3) .. controls +(0, 0.2) and +(1, 0) .. (3, 4);
      \draw[mesh-stratum] (5, 4) -- (5, 5) .. controls +(0, 0.2) and +(-1, 0) .. (6, 6);
      \draw[mesh-stratum] (7, 0) -- (7, 5) .. controls +(0, 0.2) and +(1, 0) .. (6, 6);
      \draw[mesh-stratum] (6, 6) -- (6, 7);
      \draw[mesh-stratum] (0, 2) -- (8, 2);
      \draw[mesh-stratum] (0, 4) -- (8, 4);
      \draw[mesh-stratum] (0, 6) -- (8, 6);
      \node[mesh-vertex] at (4, 2) {};
      \node[mesh-vertex] at (3, 4) {};
      \node[mesh-vertex] at (5, 4) {};
      \node[mesh-vertex] at (5, 4) {};
      \node[mesh-vertex] at (2, 2) {};
      \node[mesh-vertex] at (6, 6) {};
      \node[mesh-vertex] at (7, 2) {};
      \node[mesh-vertex] at (7, 4) {};
    \end{tikzpicture}
  \]
  The closed mesh $\strat{X}$ can be thought of as generalised higher-dimensional
  parentheses which are placed around the elements of the open mesh $\strat{M}$.
\end{example}

\begin{example}
  We illustrate a few examples of pairs of $2$-meshes which do \textbf{not} form
  transverse configurations.
  \[
    \begin{tikzpicture}[scale = 0.5, baseline=(current bounding box.center)]
      \fill[mesh-background] (0, 0) rectangle (6, 6);
      \draw[mesh-stratum-dual] (1, 1) -- (3, 5) -- (5, 1) -- cycle;
      \draw[mesh-stratum] (0, 1) -- (6, 1);
      \node[mesh-vertex] at (3, 1) {};
      \node[mesh-vertex-dual] at (1, 1) {};
      \node[mesh-vertex-dual] at (3, 5) {};
      \node[mesh-vertex-dual] at (5, 1) {};
    \end{tikzpicture}
    \qquad
    \begin{tikzpicture}[scale = 0.5, baseline=(current bounding box.center)]
      \fill[mesh-background] (0, 0) rectangle (6, 6);
      \draw[mesh-stratum-dual] (1, 1) -- (3, 5) -- (5, 1) -- cycle;
      \draw[mesh-stratum] (0, 3) -- (6, 3);
      \node[mesh-vertex] at (2, 3) {};
      \node[mesh-vertex-dual] at (1, 1) {};
      \node[mesh-vertex-dual] at (3, 5) {};
      \node[mesh-vertex-dual] at (5, 1) {};
    \end{tikzpicture}
    \qquad
    \begin{tikzpicture}[scale = 0.5, baseline=(current bounding box.center)]
      \fill[mesh-background] (0, 0) rectangle (6, 6);
      \draw[mesh-stratum-dual] (1, 1) -- (3, 5) -- (5, 1) -- cycle;
      \draw[mesh-stratum] (4, 0) -- (4, 6);
      \node[mesh-vertex-dual] at (1, 1) {};
      \node[mesh-vertex-dual] at (3, 5) {};
      \node[mesh-vertex-dual] at (5, 1) {};
    \end{tikzpicture}
  \]
\end{example}

In this section we show that every transverse configuration of $n$-meshes
$\strat{X} \pitchfork \strat{M}$ determines a bordism of open $n$-trusses
between the truss nerves
$\TrussClosedNerve(\strat{X})^\dagger \pto \TrussOpenNerve(\strat{M})$
and, conversely, that every open $n$-truss bordism arises in this way.
Moreover, the bordism of open $n$-trusses arising from a transverse configuration
$\strat{X} \pitchfork \strat{M}$ is natural in both $\strat{X}$ and $\strat{M}$.
This is realised by constructing a quasicategory $\BConf{n}$ which classifies
transverse configurations of $n$-meshes, and then demonstrating that it is
equivalent to the twisted arrow category $\Tw(\BTrussOpen{n})$:
\[
  \begin{tikzcd}
    {\BConf{n}} \ar[r, "\simeq"] \ar[d] &
    {\Tw(\BTrussOpen{n})} \ar[d] \\
    {\BMeshClosed{n} \times \BMeshOpen{n}} \ar[r, "\simeq"'] &
    {\BTrussClosed{n} \times \BTrussOpen{n}}
  \end{tikzcd}
\]

\subsection{Classifying Category for Transverse Configurations}

\begin{lemma}\label{lem:fct-conf-pullback}
  Let $(\xi, f)$ be a transverse configuration of $n$-mesh bundles over
  a stratified space $\strat{B}$.
  When $\varphi : \strat{A} \to \strat{B}$ is a map of stratified spaces, then
  the pulled back $n$-mesh bundles $(\varphi^* \xi, \varphi^* f)$
  form a transverse configuration of $n$-mesh bundles over $\strat{A}$.
\end{lemma}

\begin{definition}
  We let $\BConf{n}$ denote the simplicial set whose $k$-simplices
  consist of a transverse configuration $(\xi, f)$ of a closed
  $n$-mesh bundle $\xi$ and an open $n$-mesh bundle $f$ over
  the stratified $k$-simplex $\DeltaStrat{k}$.
  An order-preserving map $\ord{k'} \to \ord{k}$ acts via pullback
  of $\xi$ and $f$ along the induced map $\DeltaStrat{k'} \to \DeltaStrat{k}$.
\end{definition}

\begin{observation}\label{obs:fct-conf-base-case}
  A $k$-simplex of $\BConf{0}$ is a transverse configuration of
  a closed $0$-mesh bundle $\xi : \strat{X} \to \DeltaStrat{k}$ and an open
  $0$-mesh bundle $f : \strat{M} \to \DeltaStrat{k}$. 
  Both $\xi$ and $f$ are the identity map on $\DeltaStrat{k}$.
  Therefore $\BConf{0} \cong \terminal$.
\end{observation}

\begin{para}
  A $k$-simplex of $\BConf{n}$ contains both a closed an an open $n$-mesh
  bundle over the stratified $k$-simplex $\DeltaStrat{k}$.
  These mesh bundles individually represent a $k$-simplex of
  $\BMeshClosed{n}$ and $\BMeshOpen{n}$ respectively.
  We therefore have projection maps
  \[
    \begin{tikzcd}[row sep = tiny]
      {\BMeshClosed{n}} &
      {\BConf{n}} \ar[l] \ar[r] &
      {\BMeshOpen{n}} \\
      {\xi} &
      {(\xi, f)} \ar[l, mapsto] \ar[r, mapsto] &
      {f}
    \end{tikzcd}
  \]
\end{para}

\begin{lemma}\label{lem:fct-conf-cocartesian}
  The projection map $\BConf{n} \to \BMeshClosed{n}$
  is a cocartesian fibration
  so that a map is cocartesian
  when it is sent to an equivalence
  by the other projection
  $\BConf{n} \to \BMeshOpen{n}$.
  In particular $\BConf{n}$ is a quasicategory.
  Moreover the dual claim holds for 
  $\BConf{n} \to \BMeshOpen{n}$.
\end{lemma}
\begin{proof}
  We prove the claim for the projection onto the closed $n$-mesh bundle;
  the other part follows similarly.
  Suppose that we have a lifting problem
  \begin{equation}\label{eq:fct-conf-cococartesian:lift}
    \begin{tikzcd}
      {\Lambda^i\ord{k}} \ar[r, "\tau"] \ar[d] &
      {\BConf{n}} \ar[d] \\
      {\Delta\ord{k}} \ar[r] \ar[ur, dashed] &
      {\BMeshClosed{n}}
    \end{tikzcd}
  \end{equation}
  Unpacking the definitions, we are given transverse configuration
  $(\xi', f')$ of $n$-mesh bundles over $\HornStrat{i}{k}$ as well as a
  closed $n$-mesh bundle $\xi$ over $\DeltaStrat{k}$ so that
  $\xi$ restricts to $\xi'$.
  The open $n$-mesh bundle $f'$ classifies a horn in $\BMeshOpen{n}$:
  \begin{equation}\label{eq:fct-conf-cocartesian:lift-open}
    \begin{tikzcd}
      {\Lambda^i\ord{k}} \ar[r] \ar[d] &
      {\BMeshOpen{n}} \\
      {\Delta\ord{k}} \ar[ur, dashed] &
      {}
    \end{tikzcd}
  \end{equation}
  We can find a lift for~(\ref{eq:fct-conf-cocartesian:lift-open})
  in the following cases since $\BMeshOpen{n}$ is a quasicategory:
  \begin{enumerate}
    \item $0 < i < k$.
    \item $k = 1$ and $i = 0$. We can pick the filler to be an equivalence.
    \item $k \geq 1$ and the first edge of $\Lambda^i\ord{k} \to \BMeshOpen{n}$ is an equivalence.
  \end{enumerate}
  A solution to the lifting problem~(\ref{eq:fct-conf-cocartesian:lift-open})
  classifies an open $n$-mesh bundle $f$ over $\DeltaStrat{k}$ which extends $f'$.
  The pair $(\xi, f)$ is not a transverse configuration in general;
  however there must exist a small open neighbourhood
  $\HornStrat{i}{k} \subset \strat{U} \subset \DeltaStrat{k}$
  such that $\xi$ and $f$ are transverse after being restricted to $\strat{U}$.
  Using the trivialisation of $\xi$ over the interior,
  we can then adjust $f$ to an open $n$-mesh bundle $\hat{f}$
  over $\DeltaStrat{k}$ which agrees with $f$ over $\strat{U}$ and 
  is transverse to $\xi$.
  The pair $(\xi, \hat{f})$ is a solution to the lifting problem.
\end{proof}

\begin{remark}
  As a simplicial set $\BConf{n}$ is a subobject of the product
  $\BMeshClosed{n} \times \BMeshOpen{n}$.
  However the inclusion map
  $
    \BConf{n} \hookrightarrow \BMeshClosed{n} \times \BMeshOpen{n}
  $
  is \textbf{not} an inner fibration so we can not see $\BConf{n}$ as a subcategory of the product.
\end{remark}

\begin{construction}
  Suppose that $(\xi, f)$ is a transverse configuration of
  $n$-mesh bundles over a stratified space $\strat{B}$.
  Individually, we have classifying maps
  \[
    \classify{\xi} : \Exit(\strat{B}) \to \BMeshClosed{n}
    \qquad
    \classify{f} : \Exit(\strat{B}) \to \BMeshOpen{n}
  \]
  Since by Lemma~\ref{lem:fct-conf-pullback} transversality is preserved by pullbacks, the map
  \[
    \begin{tikzcd}[column sep = large]
      {\Exit(\strat{B})} \ar[r] &
      {\Exit(\strat{B}) \times \Exit(\strat{B})} \ar[r, "\classify{\xi} \times \classify{f}"] &
      {\BMeshClosed{n} \times \BMeshOpen{n}}
    \end{tikzcd}
  \]
  factors through $\BConf{n}$. This defines the \defn{classifying map}
  \[
    \classify{\xi, f} : \Exit(\strat{B})
    \longrightarrow
    \BConf{n}.
  \]
\end{construction}

\begin{proposition}\label{prop:fct-conf-classify}
  Let $\strat{B}$ be a triangulable stratified space and $I : \strat{A} \hookrightarrow \strat{B}$ a closed triangulable subspace.
  Suppose that $\varphi : \Exit(\strat{B}) \to \BConf{n}$ is a map and
  $\xi' \pitchfork f'$ a transverse configuration of $n$-mesh bundles over
  $\strat{A}$ such that $\classify{\xi', f'} = \varphi \circ \Exit(I)$.
  Then there exists a transverse configuration $\xi \pitchfork f$ of $n$-mesh bundles 
  over $\strat{B}$ together with an equivalence $\varphi \simeq \classify{\xi, f}$
  relative to $\classify{\xi', f'}$.
\end{proposition}
\begin{proof}
  It suffices to address the case where $I = \StratRealMark{i}$ for an inclusion
  of marked simplicial sets $i : R \hookrightarrow S$.
  We define maps $\bar{\varphi}$ and $\dot{\varphi}$ from $\varphi$ via
  \[
    \begin{tikzcd}
      {} &
      {\Exit(\strat{B})} \ar[dr, "\dot{\varphi}"] \ar[dl, "\bar{\varphi}"'] \ar[d, "\varphi"{description}] &
      {} \\
      {\BMeshClosed{n}} &
      {\BConf{n}} \ar[l] \ar[r] &
      {\BMeshOpen{n}}
    \end{tikzcd}
  \]
  Then Lemma~\ref{lem:fct-mesh-closed-n-classify-marked} produces a unique
  closed $n$-mesh bundle $\xi$ over $\StratRealMark{S}$ which makes the following diagram commute:
  \[
    \begin{tikzcd}
    	{\ExitMark(\StratRealMark{R})} &&& {\ExitMark(\StratRealMark{S})} \\
    	& R & S \\
    	{(\BMeshClosed{n})^\natural} &&& {(\BMeshClosed{n})^\natural}
    	\arrow["i"', hook, from=2-2, to=2-3]
    	\arrow["{\eta_R}"{description}, hook', from=2-2, to=1-1]
    	\arrow[from=2-2, to=3-1]
    	\arrow["{\eta_S}"{description}, hook, from=2-3, to=1-4]
    	\arrow["{\bar{\varphi} \circ \eta_S}"{description}, from=2-3, to=3-4]
    	\arrow["{\ExitMark(I)}", hook, from=1-1, to=1-4]
    	\arrow["{\classifyMark{\xi'}}"', from=1-1, to=3-1]
    	\arrow["{\classifyMark{\xi}}", dashed, from=1-4, to=3-4]
    	\arrow[equal, from=3-1, to=3-4]
    \end{tikzcd}
  \]
  Similarly via the map $\dot{\varphi}$ there is a unique open $n$-mesh bundle
  $f$ over $\StratRealMark{B}$ which fits the analogous diagram.
  For any marked flag $\sigma$ of $B$ the bundles $\xi$ and $f$ restrict to
  \[
    \begin{tikzcd}
      {\strat{X}_\sigma} \ar[r, hook] \ar[d, "\xi_\sigma"'] \pullbackcorner &
      {\strat{X}} \ar[d, "\xi"] \\
      {\StratRealMark{(\Delta\ord{k}, T)}} \ar[r, "\sigma"', hook] &
      {\StratRealMark{S}}
    \end{tikzcd}
    \qquad
    \begin{tikzcd}
      {\strat{M}_\sigma} \ar[r, hook] \ar[d, "f_\sigma"'] \pullbackcorner &
      {\strat{M}} \ar[d, "f"] \\
      {\StratRealMark{(\Delta\ord{k}, T)}} \ar[r, "f"', hook] &
      {\StratRealMark{S}}
    \end{tikzcd}
  \]
  where $\xi_\sigma$ and $f_\sigma$ are the $n$-mesh bundles represented by
  $\bar{\varphi} \circ \sigma$ and $\dot{\varphi} \circ \sigma$.
  Since $\bar{\varphi}$ and $\dot{\varphi}$ together arise from $\varphi$,
  which classifies a transverse configuration, we have $\xi_\sigma \pitchfork f_\sigma$.
  Because this holds for all marked flags of $S$, we have $\xi \pitchfork f$.
  
  Now both $\classify{\xi, f}$ and $\varphi$ are solutions to the lifting problem  
  \[
    \begin{tikzcd}
      {\ExitMark(\StratRealMark{R}) \sqcup_R S} \ar[d, hook, "\simeq"'] \ar[r] &
      {(\BMeshClosed{1})^\natural} \\
      {\ExitMark(\StratRealMark{S})} \ar[ur, dashed]
    \end{tikzcd}
  \]
  Therefore there exists an equivalence $\classify{\xi, f} \simeq \varphi$
  relative to $\classify{\xi', f'}$.  
\end{proof}

\subsection{Attaching Labels to Transverse Configurations}

\begin{construction}
  Suppose that $\xi \pitchfork f$ is a transverse configuration of a closed $n$-mesh bundle
  $\xi : \strat{X} \to \strat{B}$
  and an open $n$-mesh bundle
  $f : \strat{M} \to \strat{B}$.
  Via the maps induced by the framing, we define the intersection $\strat{X} \cap \strat{M}$
  as the pullback of stratified spaces  
  \[
    \begin{tikzcd}
      {\strat{X} \cap \strat{M}} \ar[r] \ar[d] \pullbackcorner &
      {\strat{M}} \ar[d] \\
      {\strat{X}} \ar[r] &
      {\R^n \times \unstrat(\strat{B})}
    \end{tikzcd}
  \]
  The underlying unstratified space $\unstrat(\strat{X} \cap \strat{M})$ is
  isomorphic to $\strat{X}$, but the stratification of $\strat{X} \cap \strat{M}$
  is finer than that of $\strat{X}$.
  We let $\xi \cap f : \strat{X} \cap \strat{M}$ denote the induced $n$-framed stratified bundle which fits into the diagrams
  \[
    \begin{tikzcd}
      {\strat{X} \cap \strat{M}} \ar[r] \ar[d, "\xi \cap f"'] &
      {\strat{X}} \ar[d, "\xi"] \\
      {\strat{B}} \ar[r, equal] &
      {\strat{B}}
    \end{tikzcd}
    \qquad
    \begin{tikzcd}
      {\strat{X} \cap \strat{M}} \ar[r] \ar[d, "\xi \cap f"'] &
      {\strat{M}} \ar[d, "f"] \\
      {\strat{B}} \ar[r, equal] &
      {\strat{B}}
    \end{tikzcd}
  \]
  The $n$-framed stratified bundle $\xi \cap f : \strat{X} \cap \strat{M} \to \strat{B}$
  then is a closed $n$-mesh bundle refining $\xi$.
  A section $s : \strat{B} \to \strat{X} \cap \strat{M}$ of $\xi \cap f$
  simultaneously determines
  a section of $\xi$ and $f$; conversely any pair of sections of $\xi$ and $f$
  that agree in framing coordinates arises from a section of the closed $n$-mesh bundle $\xi \cap f$.
\end{construction}

\begin{definition}\label{def:fct-conf-label}
  A labelling of a transverse configuration $\xi \pitchfork f$ of
  a closed $n$-mesh
  $\xi : \strat{X} \to \strat{B}$ and an open $n$-mesh
  $f : \strat{M} \to \strat{B}$
  with a simplicial set of labels $X$
  consists of maps $(L_0, L, L_1)$ which together fit into the diagram
  \[
    \begin{tikzcd}
      {\Exit(\strat{X})} \ar[d, "L_0"'] &
      {\Exit(\strat{X} \cap \strat{M})} \ar[r] \ar[d, "L"{description}] \ar[l] &
      {\Exit(\strat{M})} \ar[d, "L_1"] \\
      {X^\op} &
      {\Tw(X)} \ar[r] \ar[l] &
      {X}
    \end{tikzcd}
  \]
\end{definition}

\begin{observation}
  In the situation of Definition~\ref{def:fct-conf-label},
  the map $L_0$ determines a labelling of the closed $n$-mesh $\xi$
  in $X^\op$ while the map $L_1$ is a labelling of the open $n$-mesh in $X$.
  The map $L$ assigns to every point $x \in \strat{X} \cap \strat{M}$
  that is common to both meshes a map $L_0(x) \to L_1(x)$ in $X$.
  This assignment is functorial in $x$ in that an exit path in $\strat{X} \cap \strat{M}$
  acts contravariantly on the domain and covariantly on the codomain.
  This twist in the direction of the labelling is consistent with
  the duality equivalence
  $\BMeshClosedL{n}{\cat{C}^\op} \simeq (\BMeshOpenL{n}{\cat{C}})^\op$.
\end{observation}

In \S\ref{sec:fct-mesh-closed-labelled} we defined the simplicial set $\BMeshClosedL{n}{X}$
of closed $n$-mesh bundles with labels in some simplicial set $X$ as a polynomial functor.
This polynomial functor was induced by a forgetful map
$\MeshClosed{n} : \EMeshClosed{n} \to \BMeshClosed{n}$
from the simplicial set $\EMeshClosed{n}$ that classifies closed $n$-mesh bundles that are additionally equipped with a section.
While we do not expect that we can define a simplicial set $\BConfL{n}{X}$ classifying transverse configurations of $n$-meshes with labels in $X$ as a polynomial functor directly, we rely on a similar idea.
 We define three different polynomial functors, based on transverse configurations
$\xi \pitchfork f$ equipped with sections of $\xi$, $\xi \cap f$ and $f$, respectively.

\begin{definition}\label{def:fct-conf-cat-pointed}
  We let $\EConf{n}$ denote the simplicial set whose $k$-simplices
  consist of a transverse configuration $\xi \pitchfork f$ of a closed
  $n$-mesh bundle $\xi$ and an open $n$-mesh bundle $f$ over
  the stratified $k$-simplex $\DeltaStrat{k}$
  together with a section $s$
  of $\xi \cap f : \strat{X} \cap \strat{M} \to \DeltaStrat{k}$.
  An order-preserving map $\ord{k'} \to \ord{k}$ acts via pullback
  of $\xi$ and $f$ along the induced map $\DeltaStrat{k'} \to \DeltaStrat{k}$
  and adjusts the section accordingly.
\end{definition}

\begin{para}
  We get a family of projection maps which together fit into the diagram
  \[
    \begin{tikzcd}
      {\EMeshClosed{n}} \ar[d, "\MeshClosed{n}"'] &
      {\EConf{n}} \ar[l] \ar[r] \ar[d, "\EConfMap{n}"{description}] &
      {\EMeshOpen{n}} \ar[d, "\MeshOpen{n}"] \\
      {\BMeshClosed{n}} &
      {\BConf{n}} \ar[l] \ar[r] &
      {\BMeshOpen{n}}
    \end{tikzcd}
  \]
  The left and right squares then factor through the pullbacks:
  \begin{equation}\label{eq:fct-conf-label-l-lr-r}
    \begin{tikzcd}[row sep = large]
      {\EMeshClosed{n}} \ar[d, "\MeshClosed{n}"'] &
      {\LConf{n}} \ar[l] \ar[d, "\LConfMap{n}"{description}] \pullbackdl &
      {\EConf{n}} \ar[l] \ar[r] \ar[d, "\EConfMap{n}"{description}] &
      {\RConf{n}} \ar[r] \ar[d, "\RConfMap{n}"{description}] \pullbackcorner &
      {\EMeshOpen{n}} \ar[d, "\MeshOpen{n}"] \\
      {\BMeshClosed{n}} &
      {\BConf{n}} \ar[l] &
      {\BConf{n}} \ar[l, equal] \ar[r, equal] &
      {\BConf{n}} \ar[r] &
      {\BMeshOpen{n}}
    \end{tikzcd}
  \end{equation}
  Unwrapping the definitions, a $k$-simplex of the simplicial set
  $\LConf{n}$ is a transverse configuration
  of $k$-mesh bundles over $\DeltaStrat{k}$ together with a section of the closed $n$-mesh bundle. The map $\LConfMap{n} : \LConf{n} \to \BConf{n}$ forgets the section.
  Analogously, a $k$-simplex of $\RConf{n}$ is a transverse configuration with a section
  of the open $n$-mesh bundle with $\RConfMap{n} : \RConf{n} \to \BConf{n}$ forgetting the section.
\end{para}

\begin{observation}\label{obs:fct-conf-intersect-closed}
  When $\xi \pitchfork f$ is a transverse configuration of $n$-mesh bundles
  over the same stratified space $\strat{B}$,
  the intersection $\xi \cap f$ is a closed $n$-mesh bundle over $\strat{B}$
  and so by Lemma~\ref{lem:fct-mesh-closed-n-bundle-flat}
  the forgetful map $\EConfMap{n} : \EConf{n} \to \BConf{n}$ is a categorical fibration.
  Since the maps $\MeshClosed{n} : \EMeshClosed{n} \to \BMeshClosed{n}$ and
  $\MeshOpen{n} : \EMeshOpen{n} \to \BMeshOpen{n}$ are categorical fibrations,
  so are the maps $\LConfMap{n} : \LConf{n} \to \BConf{n}$ and 
  $\RConfMap{n} : \RConf{n} \to \BConf{n}$ defined by the pullbacks.
\end{observation}

\begin{lemma}\label{lem:fct-conf-universal}
  Let $\xi \pitchfork f$ be a transverse configuration of a closed $n$-mesh bundle
  $\xi : \strat{X} \to \strat{B}$ and an open $n$-mesh bundle
  $f : \strat{M} \to \strat{B}$. Then there are pullback squares of simplicial sets
  \[
    \begin{tikzcd}
      {\Exit(\strat{X})} \ar[r] \ar[d, "\xi_*"'] \pullbackcorner &
      {\LConf{n}} \ar[d] \\
      {\Exit(\strat{B})} \ar[r, "\classify{\xi, f}", swap] &
      {\BConf{n}}
    \end{tikzcd}
    \qquad
    \begin{tikzcd}
      {\Exit(\strat{X} \cap \strat{M})} \ar[r] \ar[d, "(\xi \cap f)_*"'] \pullbackcorner &
      {\EConf{n}} \ar[d] \\
      {\Exit(\strat{B})} \ar[r, "\classify{\xi, f}", swap] &
      {\BConf{n}}
    \end{tikzcd}
    \qquad
    \begin{tikzcd}
      {\Exit(\strat{M})} \ar[r] \ar[d, "f_*"'] \pullbackcorner &
      {\RConf{n}} \ar[d] \\
      {\Exit(\strat{B})} \ar[r, "\classify{\xi, f}", swap] &
      {\BConf{n}}
    \end{tikzcd}
  \]
\end{lemma}
\begin{proof}
  Analogous to the proof of Lemma~\ref{lem:fct-mesh-closed-labelled-universal}.
\end{proof}

\begin{construction}
  The maps $\LConfMap{n}$, $\EConfMap{n}$ and $\RConfMap{n}$ define polynomial
  functors
  \begin{align*}
    \LConfL{n}{-} &: \sSet \to \sSet \\
    \EConfL{n}{-} &: \sSet \to \sSet \\
    \RConfL{n}{-} &: \sSet \to \sSet
  \end{align*}
  Moreover Construction~\ref{eq:b-poly-functorial-base} defines natural transformations
  \[
    \begin{tikzcd}
      {\LConfL{n}{-}} \ar[d] & 
      {\EConfL{n}{-}} \ar[r] \ar[d] \ar[l] &
      {\RConfL{n}{-}} \ar[d] \\
      {\BConf{n}} &
      {\BConf{n}} \ar[r, equal] \ar[l, equal] &
      {\BConf{n}}
    \end{tikzcd}
  \]  
\end{construction}

\begin{lemma}\label{lem:fct-conf-pointed-lifts}
  Let $\xi \pitchfork f$ be a transverse configuration of $n$-mesh bundles
  $\xi : \strat{X} \to \strat{B}$ and $f : \strat{M} \to \strat{B}$
  and let $C$ be a simplicial set. We then have:
  \begin{enumerate}
    \item There is a bijection between the set of maps of simplicial sets
    $\Exit(\strat{X}) \to C$
    and the set of lifts in the diagram
    \[
      \begin{tikzcd}
        {} &
        {\LConfL{n}{C}} \ar[d] \\
        {\Exit(\strat{B})} \ar[r, "\classify{\xi, f}"'] \ar[ur, dashed] &
        {\BConf{n}}
      \end{tikzcd}
    \]
    \item There is a bijection between the set of maps of simplicial sets
    $\Exit(\strat{X} \cap \strat{M}) \to C$
    and the set of lifts in the diagram
    \[
      \begin{tikzcd}
        {} &
        {\EConfL{n}{C}} \ar[d] \\
        {\Exit(\strat{B})} \ar[r, "\classify{\xi, f}"'] \ar[ur, dashed] &
        {\BConf{n}}
      \end{tikzcd}
    \]
    \item There is a bijection between the set of maps of simplicial sets
    $\Exit(\strat{M}) \to C$
    and the set of lifts in the diagram
    \[
      \begin{tikzcd}
        {} &
        {\RConfL{n}{C}} \ar[d] \\
        {\Exit(\strat{B})} \ar[r, "\classify{\xi, f}"'] \ar[ur, dashed] &
        {\BConf{n}}
      \end{tikzcd}
    \]
  \end{enumerate}
\end{lemma}
\begin{proof}
  This follows from Lemma~\ref{lem:fct-conf-universal}.
\end{proof}

\begin{definition}\label{def:fct-conf-label-cat}
  When $C$ is a simplicial set, we write $\BConfL{n}{C}$
  for the limit of the following diagram of simplicial sets:
  \begin{equation}\label{eq:fct-conf-labelled}
    \begin{tikzcd}[column sep = small]
      {\LConfL{n}{C^\op}} \ar[r] &
      {\EConfL{n}{C^\op}} &
      {\EConfL{n}{\Tw(C)}} \ar[l] \ar[r] &
      {\EConfL{n}{C}} &
      {\RConfL{n}{C}} \ar[l]
    \end{tikzcd}
  \end{equation}
  The two maps on the inside are induced via functoriality by the forgetful
  maps $\Tw(C) \to C^\op$ and $\Tw(C) \to C$ that send a twisted arrow to its
  domain and codomain. The two outer maps are the components at $C^\op$ and $C$
  of the natural transformations
  $\EConfL{n}{-} \to \LConfL{n}{-}$ and
  $\EConfL{n}{-} \to \RConfL{n}{-}$, respectively.
  By naturality of the limit this defines a functor
  $\BConfL{n}{-} : \sSet \to \sSet$ together with a natural transformation
  $\BConfL{n}{-} \to \BConf{n}$.
\end{definition}

\begin{observation}\label{obs:fct-conf-labelled-base-case}
  Similar to Observation~\ref{obs:fct-mesh-closed-labelled-base-case} we have
  \[ \LConf{0} \cong \EConf{0} \cong \RConf{0} \cong \terminal. \]
  Therefore for any simplicial set $C$ we have natural isomorphisms of
  simplicial sets
  \[
    \LConfL{0}{C} \cong \EConfL{0}{C} \cong \RConfL{0}{C} \cong C
  \]
  Then the diagram~(\ref{eq:fct-conf-labelled}) specialises to
  \[
    \begin{tikzcd}[column sep = small, row sep = large]
      {\LConfL{0}{C^\op}} \ar[r] \ar[d, "\cong"'] &
      {\EConfL{0}{C^\op}} \ar[d, "\cong"{description}] &
      {\EConfL{0}{\Tw(C)}} \ar[l] \ar[r] \ar[d, "\cong"{description}]  &
      {\EConfL{0}{C}} \ar[d, "\cong"{description}] &
      {\RConfL{0}{C}} \ar[l] \ar[d, "\cong"] \\
      {C^\op} \ar[r, equal] &
      {C^\op} &
      {\Tw(C)} \ar[l] \ar[r] & 
      {C} &
      {C} \ar[l, equal]
    \end{tikzcd}
  \]
  Hence we have a natural isomorphism $\BConfL{0}{C} \cong \Tw(C)$.
\end{observation}

\begin{observation}
  Suppose that $\xi \pitchfork f$ is a transverse configuration of $n$-mesh bundles
  $\xi : \strat{X} \to \strat{B}$ and $f : \strat{M} \to \strat{B}$.
  Let $C$ be any simplicial set.
  Via the characterisation of lifts from Lemma~\ref{lem:fct-conf-pointed-lifts},
  we see that a lift in the diagram
  \[
    \begin{tikzcd}
      {} &
      {\BConfL{n}{C}} \ar[d] \\
      {\Exit(\strat{B})} \ar[r, "\classify{\xi, f}"'] \ar[ur, dashed] &
      {\BConf{n}}
    \end{tikzcd}
  \]
  corresponds exactly to a diagram
  \[
    \begin{tikzcd}
      {\Exit(\strat{X})} \ar[d, "L_0"'] &
      {\Exit(\strat{X} \cap \strat{M})} \ar[r] \ar[d, "L"{description}] \ar[l]
      \ar[dl, "L_0'"{description}] \ar[dr, "L_1'"{description}] &
      {\Exit(\strat{M})} \ar[d, "L_1"] \\
      {C^\op} &
      {\Tw(C)} \ar[r] \ar[l] &
      {C}
    \end{tikzcd}
  \]
  The five maps $L_0$, $L_0'$, $L$, $L_1'$ and $L_1$ are each induced by one
  of the five simplicial sets in the limit diagram (\ref{eq:fct-conf-labelled})
  which defines $\BConfL{n}{C}$.
  Therefore the simplicial set $\BConfL{n}{C}$ classifies transverse configurations
  of $n$-meshes with labels in $C$, as we outlined earlier in Definition~\ref{def:fct-conf-label}.
  Moreover there is a diagram of forgetful maps
  \[
    \begin{tikzcd}
      {\BMeshClosedL{n}{C^\op}} \ar[d] &
      {\BConfL{n}{C}} \ar[r] \ar[d] \ar[l] &
      {\BMeshOpenL{n}{C}} \ar[d] \\
      {\BMeshClosed{n}} &
      {\BConf{n}} \ar[r] \ar[l] &
      {\BMeshOpen{n}}
    \end{tikzcd}
  \]
\end{observation}

\begin{lemma}\label{lem:fct-conf-flat}
  The projection maps
  \begin{align*}
    \LConfMap{n} &: \LConf{n} \to \BConf{n} \\
    \EConfMap{n} &: \EConf{n} \to \BConf{n} \\
    \RConfMap{n} &: \RConf{n} \to \BConf{n}
  \end{align*}
  are flat categorical fibrations.
\end{lemma}
\begin{proof}
  The map $\LConfMap{n}$ is a pullback of $\MeshClosed{n} : \EMeshClosed{n} \to \BMeshClosed{n}$ which is a flat
  categorical fibration by Lemma~\ref{lem:fct-mesh-closed-labelled-pointed-flat}.
  Similarly the map $\RConfMap{n}$ is a pullback of the flat categorical fibration
  $\MeshOpen{n} : \EMeshOpen{n} \to \BMeshOpen{n}$.
  Since the intersection $\xi \cap f$ is a closed $n$-mesh bundle,
  by Lemma~\ref{lem:fct-mesh-closed-n-bundle-flat} the induced map
  $\Exit(\xi \cap f)$ is a flat categorical fibration.
  Then we can proceed analogously as in the proof of Lemma~\ref{lem:fct-mesh-closed-labelled-pointed-flat} to see that $\EConfMap{n}$ must be flat as well.
\end{proof}

\begin{proposition}
  Let $\cat{C}$ be a quasicategory.
  Then the projection $\BConfL{n}{\cat{C}} \to \BConf{n}$ is a categorical
  fibration.
  In particular $\BConfL{n}{\cat{C}}$ is a quasicategory.
\end{proposition}
\begin{proof}
  By Lemma~\ref{lem:fct-conf-flat} the maps
  $\LConfMap{n}$, $\EConfMap{n}$ and $\RConfMap{n}$ are flat categorical
  fibrations.
  The projections $\Tw(\cat{C}) \to \cat{C}^\op$ and
  $\Tw(\cat{C}) \to \cat{C}$ are categorical fibrations and $\Tw(\cat{C})$
  a quasicategory.
  When we equip the diagram~(\ref{eq:fct-conf-labelled}) with the forgetful
  maps into $\BConf{n}$, it follows by Proposition~\ref{prop:b-poly-flat-cat} that
  the resulting diagram in the slice category $\sSet_{/ \BConf{n}}$
  is fibrant in the slice model structure. Moreover the maps  
  \[
    \begin{tikzcd}
      {\EConfL{n}{\cat{C}^\op}} \ar[dr] &
      {\EConfL{n}{\Tw(\cat{C})}} \ar[l] \ar[r] \ar[d] &
      {\EConfL{n}{\cat{C}}} \ar[dl] \\
      {} &
      {\BConf{n}} &
      {}
    \end{tikzcd}
  \]
  are fibrations in the slice model category.
  Therefore the limit $\BConfL{n}{\cat{C}} \to \BConf{n}$ is fibrant as well
  and thus a categorical fibration.
\end{proof}

\subsection{Characterising Configurations of $1$-Meshes}

Let us for now restrict our attention to transverse configurations of $1$-meshes,
and build up to the result that for any $1$-category $\cat{C}$ there is a natural
equivalence $\BConfL{1}{\cat{C}} \simeq \Tw(\BTrussOpenL{1}{\cat{C}})$.
We start with the unlabelled case and construct a map $\BConf{1} \to \Tw(\FinOrd)$.

\begin{observation}\label{obs:fct-conf-1-object}
  Let $\xi \pitchfork f$ be a transverse configuration of a closed $1$-mesh bundle
  $\xi : \strat{X} \to \DeltaStrat{0}$ and an open $1$-mesh bundle
  $f : \strat{M} \to \DeltaStrat{0}$.
  We have a finite ordinal $\sing(\xi)$ whose elements correspond
  to the singular strata of $\strat{X}$, and a finite ordinal
  $\reg(f)$ whose elements are the regular strata of $\strat{M}$.
  Then we can define a map
  \[
    \sing(\xi) \longrightarrow \reg(f)
  \]
  which sends a singular stratum to the regular stratum that it is contained in.
  This is well-defined by the transversality condition and is order-preserving.
  We can therefore assign to each object of $\BConf{1}$ an object of
  the twisted arrow category $\Tw(\FinOrd)$.
\end{observation}

\begin{lemma}
  Observation~\ref{obs:fct-conf-1-object} extends uniquely to a functor
  in the diagram
  \[
    \begin{tikzcd}
      {\BMeshClosed{1}} \ar[d, "\sing"'] &
      {\BConf{1}} \ar[d, dashed] \ar[l] \ar[r] &
      {\BMeshOpen{1}} \ar[d, "\reg"] \\
      {\FinOrd^\op} &
      {\Tw(\FinOrd)} \ar[l] \ar[r] &
      {\FinOrd}
    \end{tikzcd}
  \]
\end{lemma}
\begin{proof}
  Since $\Tw(\FinOrd)$ is a $1$-category, it suffices to construct a functor
  of $1$-categories
  \[ \Ho(\BConf{1}) \longrightarrow \Tw(\FinOrd). \]
  Observation~\ref{obs:fct-conf-1-object} specifies the action on objects.  
  A map of $\Ho(\BConf{1})$ is represented by 
  a transverse configuration $\xi \pitchfork f$ of a closed $1$-mesh bundle
  $\xi : \strat{X} \to \DeltaStrat{1}$ and an open $1$-mesh bundle
  $f : \strat{M} \to \DeltaStrat{1}$.
  Let us denote respectively by $\xi_0$, $\xi_1$ and by $f_0$, $f_1$ the
  restrictions of the $1$-mesh bundles $\xi$ and $f$ to the endpoints of the stratified interval $\DeltaStrat{1}$.
  We then obtain a twisted square of finite ordinals
  \begin{equation}\label{eq:fct-conf-1-map:twisted}
    \begin{tikzcd}[column sep = large]
      {\sing(\xi_0)} \ar[d] &
      {\sing(\xi_1)} \ar[d] \ar[l, "\sing(\xi)"'] \\
      {\reg(f_0)} \ar[r, "\reg(f)"'] &
      {\reg(f_1)}
    \end{tikzcd}
  \end{equation}

  We argue that this twisted square commutes by following the action
  of the maps on an element $i_1 \in \sing(\xi_1)$.
  The element $i_1$ corresponds to a singular point $x \in \strat{X}$.  
  Since by Lemma~\ref{lem:fct-mesh-closed-1-sing-covering} the
  map $\Exit(\xi)$ is a right covering map, there exists a unique filler for
  \[
    \begin{tikzcd}
      {\DeltaStrat{0}} \ar[r, "x"] \ar[hook, d, "\langle 1 \rangle_*"'] &
      {\strat{X}} \ar[d, "\xi"] \\
      {\DeltaStrat{1}} \ar[r, equal] \ar[ur, dashed, "\gamma"{description}] &
      {\DeltaStrat{1}}
    \end{tikzcd}
  \]
  where $\gamma$ may only intersect singular strata of $\strat{X}$.
  When we let $i_0 \in \sing(\xi_0)$ be the index of the singular point $\gamma(0)$,
  the exit path $\gamma$ witnesses that $\sing(\xi)(i_1) = i_0$.
  By transversality the exit path $\gamma$ that traverses only singular strata of
  $\strat{X}$ can intersect only regular strata of $\strat{M}$.
  Let $j_0 \in \reg(f_0)$ and $j_1 \in \reg(f_1)$ be the indices of the regular
  strata of $f_0$ and $f_1$ that contain the points $\gamma(0)$ and $\gamma(1)$.
  When we interpret $\gamma$ as an exit path in the regular subspace of $\strat{M}$,
  by construction of $\reg$ we have that $\reg(f)(i_0) = i_1$.
  We have thus shown that on the element $i_1 \in \sing(\xi_1)$ the twisted
  square of maps (\ref{eq:fct-conf-1-map:twisted}) acts by
  \[
    \begin{tikzcd}
      {i_0} \ar[d, mapsto] &
      {i_1} \ar[d, mapsto] \ar[l, mapsto] \\
      {j_0} \ar[r, mapsto] &
      {j_1}
    \end{tikzcd}
  \] 
  It follows that (\ref{eq:fct-conf-1-map:twisted}) commutes.
  Moreover $\sing$ and $\reg$ are invariant under equivalences of the $1$-mesh bundles.
  The maps $\sing(\xi_0) \to \reg(\xi_0)$ and $\sing(\xi_1) \to \reg(\xi_1)$
  respect equivalences of transverse configurations, since changing these maps
  would require moving a singular stratum of the closed mesh through a singular stratum of the open one, violating transversality on the way.
  Therefore the construction of the twisted square~(\ref{eq:fct-conf-1-map:twisted})
  respects the quotient defining the set of maps in $\Ho(\BConf{1})$.
  Functoriality of $\Ho(\BConf{1}) \to \Tw(\FinOrd)$ requires that the
  twisted squares of the form~(\ref{eq:fct-conf-1-map:twisted}) composes
  horizontally, which follows directly from functoriality of $\sing$ and $\reg$
  individually.
\end{proof}

\begin{proposition}\label{prop:fct-conf-bconf-to-tw}
  The map $\BConf{1} \to \Tw(\FinOrd)$ is a trivial fibration.
\end{proposition}
\begin{proof}
  Suppose we are given a lifting problem  
  \[
    \begin{tikzcd}
      {\partial \Delta \ord{k}} \ar[r] \ar[d] &
      {\BConf{1}} \ar[d] \\
      {\Delta \ord{k}} \ar[r] \ar[ur, dashed] &
      {\Tw(\FinOrd)}
    \end{tikzcd}
  \]
  We construct a lift by a case distinction on $k \geq 0$.
  
  \begin{enumerate}
    \item 
    When $k = 0$ the $0$-simplex $\Delta\ord{0} \to \Tw(\FinOrd)$ classifies an order-preserving map
    $v : \ord{a} \to \ord{b}$.
    Let $f : \strat{M} \to \DeltaStrat{0}$ be an open mesh bundle so that $\reg(f) = \ord{b}$.
    Then we can construct a closed mesh bundle $\xi : \strat{X} \to \DeltaStrat{0}$
    with $\sing(\xi) = \ord{a}$ such that the $i$th singular stratum is contained in the
    $v(i)$th regular stratum of $f$.
    The pair $\xi \pitchfork f$ is a transverse configuration and classifies a lift $\Delta\ord{0} \to \BConf{1}$.

    \item     
    For $k = 1$ we are given a twisted square of finite non-empty ordinals
    \[
      \begin{tikzcd}
        {\ord{a_0}} \ar[d, "v_0"'] &
        {\ord{a_1}} \ar[d, "v_1"] \ar[l, "\alpha"'] \\
        {\ord{b_0}} \ar[r, "\beta"'] &
        {\ord{b_1}}
      \end{tikzcd}
    \]
    and a transverse configuration $\xi' \pitchfork f'$ of a closed $1$-mesh bundle
    $\xi' : \strat{X}' \to \partial\DeltaStrat{1}$ and an open $1$-mesh bundle
    $f' : \strat{M}' \to \partial\DeltaStrat{1}$.
    We write $\xi_0$, $\xi_1$ and $f_0$, $f_1$ for the restrictions of $\xi'$ and $f'$
    to the two points of $\partial\DeltaStrat{1} = \{ 0 \} \sqcup \{ 1 \}$.
    We then have that $v_0$ and $v_1$ are the maps $\sing(\xi_0) \to \reg(f_0)$
    and $\sing(\xi_1) \to \reg(f_1)$ induced by $\xi_0 \pitchfork f_0$ and
    $\xi_1 \pitchfork f_1$.
    We use that $\reg : \BMeshOpen{1} \to \FinOrd$ is a trivial fibration and extend $f'$ to an open $1$-mesh bundle $f : \strat{M} \to \DeltaStrat{k}$ with $\reg(f) = \beta$.

    We construct an extension of $\xi'$ to a closed $1$-mesh bundle $\xi : \strat{X} \to \DeltaStrat{1}$ by specifying a sequence $h_i : \DeltaTop{1} \to \R$ determining the heights of the singular strata of $\xi$ for each $i \in \ord{a_1} = \sing(\xi_1)$.    
    Let $i \in \ord{a_1}$ be the index of such a singular stratum and suppose by induction that we have already determined $h_j$ for all $0 \leq j < i$.    
    Let $x_0, x_1 \in \strat{X}'$ be the singular points of $\xi'$ over $0, 1 \in \DeltaStrat{1}$ corresponding to the indices $\alpha(i) \in \ord{a_0}$ and $i \in \ord{a_1}$.    
    By transversality of $\xi'$ and $f'$ we have that $x_0$ and $x_1$ must lie in regular strata of $f'$, with indices $v_0(\alpha(i))$ and $v_1(i)$.
    Because $\reg(f) = \beta$ and $\beta(v_0(\alpha(i))) = v_1(i)$,
    there must exist a section $s : \DeltaStrat{1} \to \strat{M}$ of $f$
    with $s(0) = x_0$ and $s(1) = x_1$, traversing only regular strata of $\strat{M}$.
    We can nudge this section on the interior of $\DeltaStrat{1}$ to be higher in the framing than the singular strata we have constructed so far,
    and denote the result by $h_i$.

    The resulting closed $1$-mesh bundle $\xi$ must be transverse to $f$ by construction and the transverse configuration $\xi \pitchfork f$ 
    determines a solution of the lifting problem.
        
    \item When $k \geq 2$ we are given a map $v : \ord{k}^\op \join \ord{k} \to \FinOrd$
    together with a transverse configuration $\xi' \pitchfork f'$ of
    a closed $1$-mesh bundle $\xi' : \strat{X}' \to \partial\DeltaStrat{k}$
    and an open $1$-mesh bundle $f' : \strat{M}' \to \partial\DeltaStrat{k}$.
    We let $\alpha : \ord{k}^\op \to \FinOrd$ and $\beta : \ord{k} \to \FinOrd$
    denote the restrictions of $v$ to the two components of the join
    $\ord{k}^\op \join \ord{k}$.

    We use that the map $\reg : \BMeshOpen{1} \to \FinOrd$ is a trivial fibration and extend $f'$ to an open $1$-mesh bundle $f : \strat{M} \to \DeltaStrat{k}$ with $\reg(f) = \beta$.
    Analogously, we use that by Proposition~\ref{prop:fct-mesh-closed-1-sing-trivial} the map $\sing : \BMeshClosed{1} \to \FinOrd^\op$ is a trivial fibration to extend $\xi'$ to a closed $1$-mesh bundle $\xi'' : \strat{X}'' \to \DeltaStrat{k}$ with $\sing(\xi'') = \beta$.

    The $1$-mesh bundles $\xi''$ and $f$ are not transverse in general, but
    there is an open collar neighbourhood $\strat{U}$ of the boundary $\partial\DeltaStrat{k}$ so that $\xi''$ and $f$ are transverse over $\strat{U}$.
    Since we assumed that $k \geq 2$ we have that $\strat{U}$ is connected.
    Therefore the singular strata of $\xi''$ over $\strat{U} \setminus \partial\DeltaStrat{k}$ must lie within the same regular stratum of $f$.    
    Using the trivialisation of $f$ over the $k$th stratum of $\DeltaStrat{k}$ we then adjust $\xi''$ to an equivalent closed $1$-mesh bundle $\xi : \strat{X} \to \DeltaStrat{k}$ that agrees with $\xi''$ over $\strat{U}$ and is transverse to $f$.
    Then $\xi \pitchfork f$ represents a solution to the lifting problem.\qedhere
  \end{enumerate}
\end{proof}

\begin{lemma}\label{lem:fct-conf-econf-to-tw}
  There is a unique functor which fits into the diagram
  \begin{equation}\label{eq:fct-conf-econf-to-tw}
    \begin{tikzcd}[column sep={5em,between origins}, row sep={3em,between origins}]
    	& \ETrussClosed{1} && {\Tw(\ETrussOpen{1})} && \ETrussOpen{1} \\
    	\EMeshClosed{1} && \EConf{1} && \EMeshOpen{1} \\
    	& \BTrussClosed{1} && {\Tw(\BTrussOpen{1})} && \BTrussOpen{1} \\
    	\BMeshClosed{1} && \BConf{1} && \BMeshOpen{1}
    	\arrow["\simeq"{description}, from=4-1, to=3-2]
    	\arrow["\simeq"{description}, from=4-3, to=3-4]
    	\arrow["\simeq"{description}, from=4-5, to=3-6]
    	\arrow[from=3-4, to=3-6]
    	\arrow[from=4-3, to=4-5]
    	\arrow[from=2-1, to=4-1]
    	\arrow[from=1-4, to=3-4]
    	\arrow[from=1-6, to=3-6]
    	\arrow[from=1-2, to=3-2]
    	\arrow["\simeq"{description}, from=2-1, to=1-2]
    	\arrow[dashed, from=2-3, to=1-4]
    	\arrow["\simeq"{description}, from=2-5, to=1-6]
    	\arrow[from=1-4, to=1-6]
    	\arrow[from=4-3, to=4-1]
    	\arrow[from=3-4, to=3-2]
    	\arrow[from=1-4, to=1-2]
    	\arrow[crossing over, from=2-3, to=2-5]
    	\arrow[crossing over, from=2-5, to=4-5]
    	\arrow[crossing over, from=2-3, to=4-3]
    	\arrow[crossing over, from=2-3, to=2-1]
    \end{tikzcd}
  \end{equation}
\end{lemma}
\begin{proof}
  A $k$-simplex of $\EConf{1}$ consists of a transverse configuration
  $(\xi, f)$ of a closed $1$-mesh bundle $\xi : \strat{X} \to \DeltaStrat{k}$ and
  an open $1$-mesh bundle $f : \strat{M} \to \DeltaStrat{k}$
  together with a common section $s : \DeltaStrat{k} \to \strat{X} \cap \strat{M}$.
  The functor $\BConf{1} \to \Tw(\BTrussOpen{1})$ sends the transverse configuration
  $(\xi, f)$ to an open $1$-truss bundle
  \[
    \begin{tikzcd}
      {E} \ar[r, "\Phi"] \ar[d, "F"'] \pullbackcorner &
      {\ETrussOpen{1}} \ar[d] \\
      {\ord{k}^\op \join \ord{k}} \ar[r, "\varphi", swap] &
      {\FinOrd}
    \end{tikzcd}
  \]  
  We then need to construct a section $S$ of the open $1$-truss bundle $F$.
  The maps $\EMeshClosed{1} \to \ETrussClosed{1}$ and $\EMeshOpen{1} \to \ETrussOpen{1}$
  already determine the restrictions $S_0$, $S_1$ of $S$ to the two components
  $\ord{k}^\op$ and $\ord{k}$ of the join $\ord{k}^\op \join \ord{k}$:
  We therefore look for an $S$ which fits into the diagram
  \[
    \begin{tikzcd}
      {\ord{k}^\op \sqcup \ord{k}} \ar[r, "S_0 \sqcup S_1"] \ar[d, hook] &
      {E} \ar[d, "F"] \\
      {\ord{k}^\op \join \ord{k}} \ar[r, equal] \ar[ur, dashed, "S"{description}]&
      {\ord{k}^\op \join \ord{k}}
    \end{tikzcd}
  \]
  Since the join $\ord{k}^\op \join \ord{k} \cong \ord{2k + 1}$ is a poset,
  $E$ is a poset as well.
  The inclusion of $\ord{k}^\op \sqcup \ord{k}$ into
  $\ord{k}^\op \join \ord{k}$
  is a bijection on objects, and so $S$ is already uniquely determines as a map of sets.
  It remains to show that it is a well-defined map of posets.

  Suppose that $t \in \ord{k}$ and
  let us denote by $\varphi_t$ the order-preserving map $\varphi(t^\circ) \to \varphi(t)$.
  We verify by case distinction that there is a unique map
  $\Phi(S_0(t)) \to \Phi(S_1(t))$ in $\ETrussOpen{1}$ over $\varphi_t$,
  which implies that $S(t^\circ) \leq S(t)$.
  For reference, we refer back to the Definition~\ref{def:fct-truss-open-1} which describes the maps
  in $\ETrussOpen{1}$.
  \begin{enumerate}
    \item If $\Phi(S_0(t)) = \trussCS{i}\ord{p} = \trussOR{i}\ord{p}$ and $\Phi(S_1(t)) = \trussOR{j}\ord{q}$
    then the $i$th singular stratum of $\strat{X}$ over $t \in \DeltaStrat{k}$
    is contained in the $j$th regular stratum of $\strat{M}$ and so $\varphi_t(i) = j$.
    Therefore we have $\varphi_t(i) = j$.

    \item If $\Phi(S_0(t)) = \trussCR{i}\ord{p} = \trussOS{i}\ord{p}$ and
    $\Phi(S_1(t)) = \trussOR{j}\ord{q}$ then $s(t) \in \strat{X} \cap \strat{M}$
    is contained in the $i$th regular stratum of $\strat{X}$ and the $j$th regular
    stratum of $\strat{M}$ over $t \in \DeltaStrat{k}$.
    Because $\xi$ is a closed $1$-mesh bundle, the $i$th regular stratum in
    $\strat{X}$ over $t$ must be surrounded by the $i$th and $(i + 1)$st singular
    strata. By transversality these singular strata must lie within regular strata
    of $\strat{M}$ over $t$.
    When $j_0$ and $j_1$ are the indices of those regular strata, we must therefore have
    $\varphi_t(i) = j_0 \leq j \leq j_1 = \varphi_t(i + 1)$.

    \item If $\Phi(S_0(t)) = \trussCR{i}\ord{p} = \trussOS{i}\ord{p}$ and
    $\Phi(S_1(t)) = \trussOS{j}\ord{q}$ then the $i$th regular stratum of
    $\strat{X}$ contains the $j$th singular stratum of $\strat{M}$ over $t$.
    The regular stratum of $\strat{X}$ is again surrounded by the $i$th and
    $(i + 1)$st singular strata of $\strat{X}$ over $t$ of because $\xi$ is closed.
    By transversality these singular strata must lie within some regular strata
    of $\strat{M}$ with indices $j_0$ and $j_1$.
    Then we must have $\varphi_t(i) = j_0 \leq j < j_1 = \varphi_t(i + 1)$.

    \item Finally the case that $\Phi(S_0(t)) = \trussCS{i}\ord{p}$ and
    $\Phi(S_1(t)) = \trussOS{j}\ord{q}$ is impossible since the transversality
    condition prevents the point $s(t)$ to be singular in both $\strat{X}$ and $\strat{M}$.
    \qedhere
  \end{enumerate}
\end{proof}

\begin{lemma}\label{lem:fct-conf-econf-to-tw-trivial}
  The map $\EConf{1} \to \Tw(\ETrussOpen{1})$ is a trivial fibration.
\end{lemma}



    

\begin{proof}
  Since $\Tw(\BTrussOpen{1})$ is a $1$-category and $\EConf{1}$ is a quasicategory by Observation~\ref{obs:fct-conf-intersect-closed},
  the map must be an inner fibration already by Lemma~\ref{lem:b-inner-fib-1-cat}.
  It therefore suffices to show that the map is a categorical equivalence.

  To see that the map is essentially surjective,
  let $\hat{\alpha}$ be a map in $\ETrussOpen{n}$ which is sent to
  the order-preserving map $\alpha : \ord{n} \to \ord{m}$ by $\ETrussOpen{1} \to \BTrussOpen{1}$.
  By Proposition~\ref{prop:fct-conf-bconf-to-tw} we can construct a transverse configuration $(\xi, f)$ of
  a closed $1$-mesh bundle $\xi : \strat{X} \to \DeltaStrat{0}$ and an open
  $1$-mesh bundle $f : \strat{M} \to \DeltaStrat{0}$ which is sent to $\alpha$
  by $\BConf{1} \to \Tw(\FinOrd)$.
  It remains to show that there exists a point $x \in \strat{X} \cap \strat{M}$
  such that $\EConf{1} \to \Tw(\ETrussOpen{1})$ sends $(\xi, f)$ together with $x$
  to $\hat{\alpha}$.
  We verify this with a case distinction.  
  \begin{enumerate}
    \item Suppose that $\hat{\alpha} : \trussOR{i} \ord{n} \to \trussOR{j} \ord{n}$.
    We then have $\alpha(i) = j$ and so the $i$th singular stratum of $\strat{X}$
    must intersect the $j$th regular stratum of $\strat{M}$.
    Then the point $x \in \strat{X} \cap \strat{M}$ corresponding to the $i$th singular
    stratum of $\strat{X}$ is sent to $\hat{\alpha}$.
    
    \item Suppose that $\hat{\alpha} : \trussOS{i} \ord{n} \to \trussOR{j} \ord{n}$. Then
    we have $\alpha(i) \leq j \leq \alpha(i + 1)$.
    This entails that the $i$th and $(i + 1)$st singular strata of $\strat{X}$ must intersect
    regular strata of $\strat{M}$ that are above and below, respectively, of the $j$th
    regular stratum of $\strat{M}$.
    In particular the $i$th regular stratum of $\strat{X}$ intersects the $j$th regular stratum of $\strat{M}$ non-trivially, and any point in this intersection is sent to $\hat{\alpha}$.

    \item Suppose that $\hat{\alpha} : \trussOS{i} \ord{n} \to \trussOS{j} \ord{n}$.
    Then $\alpha(i) \leq j < \alpha(i + 1)$.
    The $j$th singular stratum of $\strat{M}$ is surrounded by the $j$th and $(j + 1)$st regular strata.    
    The $i$th singular stratum of $\strat{X}$ intersects a regular stratum of $\strat{M}$ below the $j$th regular stratum, and the $(i + 1)$st singular stratum of $\strat{X}$ intersects a regular stratum that is above the $(j + 1)$st regular stratum.
    Therefore the $i$th regular stratum of $\strat{X}$ must cross the $j$th singular stratum of $\strat{M}$.
    This intersection point is then sent to $\hat{\alpha}$.

    \item The case $\hat{\alpha} : \trussOR{i} \ord{n} \to \trussOS{j} \ord{n}$ is impossible since there are no maps from regular to singular elements in $\ETrussOpen{1}$.
  \end{enumerate}

  For faithfulness it suffices to observe that an equivalence of
  transverse configurations can not change which strata intersect each other
  since this would require passing singular strata of the closed and open mesh across each other.
  Finally fullness can again be verified via a case distinction similar to one ones above.
\end{proof}

\begin{lemma}\label{lem:fct-conf-tw-truss}
  Let $\cat{C}$ be a $1$-category.
  Then there exists an equivalence
  \[
    \begin{tikzcd}    
      {\BMeshClosedL{1}{\cat{C}}} \ar[d] &
      {\BConfL{1}{\cat{C}}} \ar[l] \ar[r] \ar[d, dashed] &
      {\BMeshOpenL{1}{\cat{C}}} \ar[d] \\
      {\BTrussClosedL{1}{\cat{C}}} &
      {\Tw(\BTrussOpenL{1}{\cat{C}})} \ar[l] \ar[r] &
      {\BTrussOpenL{1}{\cat{C}}}
    \end{tikzcd}
  \]
  which is natural in $\cat{C}$.
\end{lemma}
\begin{proof}
  Let us begin by constructing the analogue of the diagram~(\ref{eq:fct-conf-label-l-lr-r})
  for trusses:
  \[
    \begin{tikzcd}[row sep = large]
      {\ETrussClosed{1}} \ar[d, "\TrussClosed{1}"'] &
      {\Tw(\ETrussOpen{1}) \times_{\BTrussClosed{1}} \ETrussClosed{1}} \ar[l] \ar[d, "\pi_0^*\Tw(\TrussOpen{1})"{description}] \pullbackdl &
      {\Tw(\ETrussOpen{1})} \ar[l] \ar[r] \ar[d, "\Tw(\TrussOpen{1})"{description}] &
      {\Tw(\ETrussOpen{1}) \times_{\BTrussOpen{1}} \ETrussOpen{1}} \ar[r] \ar[d, "\pi_1^* \Tw(\TrussOpen{1})"{description}] \pullbackcorner &
      {\ETrussOpen{1}} \ar[d, "\TrussOpen{1}"] \\
      {\BTrussClosed{1}} &
      {\Tw(\BTrussOpen{1})} \ar[l, "\pi_0"] &
      {\Tw(\BTrussOpen{1})} \ar[l, equal] \ar[r, equal] &
      {\Tw(\BTrussOpen{1})} \ar[r, "\pi_1", swap] &
      {\BTrussOpen{1}}
    \end{tikzcd}
  \]
  The diagonal maps in the diagram (\ref{eq:fct-conf-econf-to-tw}) are trivial fibrations.
  Therefore, by Construction~\ref{con:b-poly-retract} we obtain maps between the induced polynomial functors
  \[
    \begin{tikzcd}[column sep = small]
      {\poly{\LConfMap{1}}{\cat{C}^\op}} \ar[r] \ar[d] &
      {\poly{\EConfMap{1}}{\cat{C}^\op}} \ar[d] &
      {\poly{\EConfMap{1}}{\Tw(\cat{C})}} \ar[l] \ar[r] \ar[d] &
      {\poly{\EConfMap{1}}{\cat{C}}} \ar[d] &
      {\poly{\RConfMap{1}}{\cat{C}}} \ar[l] \ar[d] \\      
      {\poly{\pi_0^* \Tw(\TrussOpen{1})}{\cat{C}^\op}} \ar[r] &
      {\poly{\Tw(\TrussOpen{1})}{\cat{C}^\op}} &
      {\poly{\Tw(\TrussOpen{1})}{\Tw(\cat{C})}} \ar[l] \ar[r] &
      {\poly{\Tw(\TrussOpen{1})}{\cat{C}}} &
      {\poly{\pi_1^* \Tw(\TrussOpen{1})}{\cat{C}}} \ar[l] &
    \end{tikzcd}
  \]
  Let us denote the limit of the lower row by $\hat{\textup{T}}(\cat{C})$.
  We obtain an induced map between the limits $\BConfL{1}{\cat{C}} \to \hat{\textup{T}}(\cat{C})$.
  Since $\cat{C}$ is a $1$-category, the limits are homotopy limits and
  so the induced map $\BConfL{1}{\cat{C}} \to \hat{\textup{T}}(\cat{C})$ is an equivalence as well.
  
  We now show that there is an isomorphism of simplicial sets
  \[
    \begin{tikzcd}[column sep = small]
      {\hat{\textup{T}}(\cat{C})} \ar[rr, "\cong"] \ar[dr] &
      {} &
      {\Tw(\BTrussOpenL{1}{\cat{C}})} \ar[dl] \\
      {} &
      {\Tw(\BTrussOpen{1})}
    \end{tikzcd}
  \]  
  We first make some auxiliary constructions derived from a $k$-simplex
  $\tau : \ord{k} \to \Tw(\BTrussOpen{1})$.
  There is a diagram of simplicial sets
  \begin{equation}\label{eq:fct-conf-truss-tw}
    \begin{tikzcd}[column sep={4em,between origins}, row sep={3.5em,between origins}]
    	& {E_0^\op} &&
      {\hat{E}} &&
      {E_1} \\
    	{\ETrussClosed{1}} && {\Tw(\ETrussOpen{1})} && {\ETrussOpen{1}} \\
    	& {\ord{k}} && {\ord{k}} && {\ord{k}} \\
    	{\BTrussClosed{1}} && {\Tw(\BTrussOpen{1})} && {\BTrussOpen{1}}
    	\arrow[from=1-2, to=2-1]
    	\arrow[from=1-4, to=2-3]
    	\arrow[from=1-6, to=2-5]
    	\arrow[from=1-4, to=1-2]
    	\arrow[from=1-4, to=1-6]
    	\arrow[from=2-1, to=4-1]
    	\arrow[from=1-6, to=3-6]
    	\arrow[from=1-4, to=3-4]
    	\arrow[from=1-2, to=3-2]
    	\arrow[from=3-6, to=4-5]
    	\arrow[from=3-4, to=4-3, "\tau"{description}]
    	\arrow[from=3-2, to=4-1]
    	\arrow[from=4-3, to=4-1]
    	\arrow[from=4-3, to=4-5]
    	\arrow[equals, from=3-4, to=3-2]
    	\arrow[equals, from=3-4, to=3-6]
    	\arrow[crossing over, from=2-3, to=2-1]
    	\arrow[crossing over, from=2-3, to=2-5]
    	\arrow[crossing over, from=2-3, to=4-3]
    	\arrow[crossing over, from=2-5, to=4-5]
    \end{tikzcd}
  \end{equation}
  in which the diagonal squares are pullback squares and
  $E_0$, $\hat{E}$ and $E_1$ are posets.
  By the definition of $\Tw$ the $k$-simplex $\tau$ corresponds uniquely to a map
  $\sigma : \ord{k}^\op \join \ord{k} \to \BTrussOpen{1}$.
  This map then determines a diagram
  \begin{equation}\label{eq:fct-conf-tw-truss:2}
    \begin{tikzcd}[column sep={4em,between origins}, row sep={3.5em,between origins}]
    	& {E_0} && E && {E_1} \\
    	{\ETrussOpen{1}} && {\ETrussOpen{1}} && {\ETrussOpen{1}} \\
    	& {\ord{k}^\op} && {\ord{k}^\op * \ord{k}} && {\ord{k}} \\
    	{\BTrussOpen{1}} && {\BTrussOpen{1}} && {\BTrussOpen{1}}
    	\arrow[hook, from=3-2, to=3-4]
    	\arrow[hook', from=3-6, to=3-4]
    	\arrow[equal, from=4-1, to=4-3]
    	\arrow[equal, from=4-5, to=4-3]
    	\arrow[from=1-2, to=3-2]
    	\arrow[from=2-1, to=4-1]
    	\arrow[from=1-4, to=3-4, "F"{pos=0.3}]
    	\arrow[from=1-6, to=3-6]
    	\arrow[hook, from=1-2, to=1-4]
    	\arrow[hook', from=1-6, to=1-4]
    	\arrow[from=1-2, to=2-1]
    	\arrow[from=3-2, to=4-1]
    	\arrow["\sigma"{description}, from=3-4, to=4-3]
    	\arrow[from=3-6, to=4-5]
    	\arrow[from=1-6, to=2-5]
    	\arrow[from=1-4, to=2-3]
    	\arrow[crossing over, from=2-5, to=2-3]
    	\arrow[crossing over, from=2-5, to=4-5]
    	\arrow[crossing over, from=2-3, to=4-3]
    	\arrow[crossing over, from=2-1, to=2-3]
    \end{tikzcd}
  \end{equation}
  in which the diagonal squares are pullback squares and $E$ is a poset.
  The posets $E_0$ and $E_1$ agree with their earlier construction.
  The twisted arrow construction $\Tw(-) : \sSet \to \sSet$ is a right
  adjoint and therefore preserves limits.
  We apply $\Tw$ to the middle pullback square in~(\ref{eq:fct-conf-tw-truss:2}) and the further
  take the pullback along the unit $\eta$ of the adjunction
  to obtain the diagram
  \begin{equation}\label{eq:fct-conf-tw-truss:decompose}
    \begin{tikzcd}
      {\hat{E}} \ar[r, hook, "I"] \ar[d] \pullbackcorner &
      {\Tw(E)} \ar[d] \ar[r] \pullbackcorner &
      {\Tw(\ETrussOpen{1})} \ar[d] \\
      {\ord{k}} \ar[r, hook, "\eta"'] &
      {\Tw(\ord{k}^\op \join \ord{k})} \ar[r, "\Tw(\sigma)"'] &
      {\Tw(\BTrussOpen{1})}
    \end{tikzcd}
  \end{equation}
  The two squares in this diagram compose to the middle square of the diagram~\ref{eq:fct-conf-truss-tw}.

  A $k$-simplex of $\Tw(\BTrussOpenL{1}{X})$ over $\tau$ now is determined
  by a diagram
  \begin{equation}\label{eq:fct-conf-tw-truss:label-1}
    \begin{tikzcd}
      {E_0} \ar[r, hook] \ar[dr, "L_0"'] &
      {E} \ar[d, "R"{description}] &
      {E_1} \ar[l, hook'] \ar[dl, "L_1"] \\
      {} &
      {\cat{C}}
    \end{tikzcd}
  \end{equation}
  while a $k$-simplex of $\hat{\textup{T}}(X)$ over $\tau$ corresponds to a diagram
  \begin{equation}\label{eq:fct-conf-tw-truss:label-2}
    \begin{tikzcd}
      {E_0^\op} \ar[d, "L_0^\op"'] &
      {\hat{E}} \ar[d, "L"{description}] \ar[l] \ar[r] &
      {E_1} \ar[d, "L_1"] \\
      {\cat{C}^\op} &
      {\Tw(\cat{C})} \ar[l] \ar[r] &
      {\cat{C}}
    \end{tikzcd}
  \end{equation}
  The maps in~(\ref{eq:fct-conf-tw-truss:label-1}) determine the maps in~(\ref{eq:fct-conf-tw-truss:label-2}) by letting $L = \Tw(R) \circ I$.
  This defines the map $\hat{\textup{T}}(\cat{C}) \to \Tw(\BTrussOpenL{1}{\cat{C}}$.
  For the other direction, we first use that $\cat{C}$ is a $1$-category
  and $F$ a flat categorical fibration to reduce to the case that $k = 1$.
  We then construct a functor of $1$-categories $R : E \to \cat{C}$ as follows.
  On objects we let $R(e_0) = L_0(e_0)$ when $e_0 \in E_0 \subseteq E$ and
  $R(e_1) = L_1(e_1)$ for $e_1 \in E_1 \subseteq E$.
  To define the action of $R$ on the maps of $E$,
  suppose that $a, b \in E$ with $a \leq b$. We then proceed by a case distinction:
  \begin{enumerate}
    \item When $a, b \in E_0$ we let $R(a \leq b) := L_0(a \leq b)$.
    \item When $a, b \in E_1$ we let $R(a \leq b) := L_1(a \leq b)$.    

    \item Suppose that $F(a) = i^\circ$ and $F(b) = i$ for some $i \in \ord{1}$.
    Then the interval $[a, b]$ is an element of the poset $\hat{E}$ and so $L$ sends
    $[a, b]$ to an object of the twisted arrow category $\Tw(\cat{C})$.
    By commutativity of~(\ref{eq:fct-conf-tw-truss:label-2}) this object of
    $\Tw(\cat{C})$ is a map $L_0(a) \to L_1(b)$ in $\cat{C}$. We then let
    $R(a \leq b) := L([a, b])$ be this map.

    \item Suppose that $F(a) = 0^\circ$ and $F(b) = 1$.
    Since $F$ is flat we can find an element $e \in E$ such that $a \leq e \leq b$
    and $F(e) = 0$. We then let
    \[
      R(a \leq b) := L_1(e \leq b) \circ L([a, e]).
    \]
    The element $e$ is not unique and so we need to verify that the definition
    of $R(a \leq b)$ does not depend on the chosen factorisation.
    Flatness of $F$ not only guarantees existence of the factorisation but also
    weak contractibility of the category of factorisations.
    In particular any two factorisations are connected by a zigzag of maps.
    Considering one step of the zigzag at a time, suppose that we have two factorisations
    \[
      \begin{tikzcd}[row sep = small]
        {} &
        {e_0} \ar[dd] \ar[dr] &
        {} \\
        {a} \ar[ur] \ar[dr] & 
        {} &
        {b} \\
        {} &
        {e_1} \ar[ur]
      \end{tikzcd}
    \]
    where $F(e_0) = F(e_1) = 0$.
    We then have a diagram
    \[
      \begin{tikzcd}[row sep = small]
        {} &
        {L_1(e_0)} \ar[dd, "\varphi"{description}] \ar[dr, "L_1(e_0 \leq b)"] &
        {} \\
        {L_0(a)} \ar[ur, "{L([a, e_0])}"] \ar[dr, "{L([a, e_1])}"'] & 
        {} &
        {L_1(b)} \\
        {} &
        {L_1(e_1)} \ar[ur, "L_1(e_1 \leq b)"']
      \end{tikzcd}
    \]
    which is rendered commutative by $\varphi = L([a, e_0] \leq [a, e_1]) = L_1(e_0 \leq e_1)$.
    \item The remaining case is $F(a) = 1^\circ$ and $F(b) = 0$.
    Analogously to the previous case we pick a factorisation
    $a \leq e \leq b$ with $F(e) = 0^\circ$ and let
    \[
      R(a \leq b) := L([e, b]) \circ L_0(a \leq e).
    \]
  \end{enumerate}
  Functoriality of $R$ follows from its construction. We see that
  the diagrams~(\ref{eq:fct-conf-tw-truss:label-1}) and (\ref{eq:fct-conf-tw-truss:label-2})
  uniquely determine each other.
\end{proof}

\begin{remark}
  We have shown Lemma~\ref{lem:fct-conf-tw-truss} for the case that $\cat{C}$
  is a $1$-category to simplify the proof. We expect that it also holds
  in the more general case when $\cat{C}$ is a quasicategory, but this would
  require a more sophisticated proof. Because $\BTrussOpen{n}$ is a $1$-category
  for all $n \geq 0$ the current form of Lemma~\ref{lem:fct-conf-tw-truss} is
  sufficient for our purpose of demonstrating an equivalence
  $\BConf{n} \simeq \Tw(\BTrussOpen{n})$ via the packing construction below.
\end{remark}

\subsection{Packing for Configurations of $n$-Meshes}

The goal in this section is to construct an equivalence $\BConf{n} \simeq \Tw(\BTrussOpen{n})$.
We do so inductively using the ability to label transverse configurations,
transferring one dimension at a time into the open truss by a packing
construction similar to that for closed and open meshes.

\begin{lemma}\label{lem:fct-conf-pack}
  Let $\cat{C}$ be a $1$-category.
  Then there exists an equivalence
  \[
    \begin{tikzcd}    
      {\BMeshClosedL{n + 1}{\cat{C}^\op}} \ar[d, "\simeq"'] &
      {\BConfL{n + 1}{\cat{C}}} \ar[l] \ar[r] \ar[d, dashed] &
      {\BMeshOpenL{n + 1}{\cat{C}}} \ar[d, "\simeq"] \\
      {\BMeshClosedL{n}{\BTrussClosedL{1}{\cat{C}^\op}}} &
      {\BConfL{n}{\BTrussOpenL{1}{\cat{C}}}} \ar[l] \ar[r] &
      {\BMeshOpenL{n}{\BTrussOpenL{1}{\cat{C}}}}
    \end{tikzcd}
  \]
  which is natural in $\cat{C}$.
\end{lemma}
\begin{proof}
  A $k$-simplex of $\BConfL{n + 1}{\cat{C}}$ is a labelled transverse configuration
  of $(n + 1)$-mesh bundles over $\DeltaStrat{k}$. We present the closed $(n + 1)$-mesh bundle as the composite of a closed $1$-mesh bundle $\xi : \strat{X} \to \strat{Y}$
  followed by a closed $n$-mesh bundle $\zeta : \strat{Y} \to \DeltaStrat{k}$.
  Similarly the open $(n + 1)$-mesh bundle is the composite of an open $1$-mesh bundle
  $f : \strat{M} \to \strat{N}$ and an open $n$-mesh bundle $g : \strat{N} \to \DeltaStrat{k}$.
  The labelling of the transverse configuration $(\zeta \circ \xi) \pitchfork (g \circ f)$
  is a diagram
  \begin{equation}\label{eq:fct-conf-pack:labels}
    \begin{tikzcd}[row sep = large]
      {\Exit(\strat{X})} \ar[d, "L_0"'] &
      {\Exit(\strat{X} \cap \strat{M})} \ar[r] \ar[d, "L"{description}] \ar[l] &
      {\Exit(\strat{M})} \ar[d, "L_1"] \\
      {\cat{C}^\op} &
      {\Tw(\cat{C})} \ar[r] \ar[l] &
      {\cat{C}}
    \end{tikzcd}
  \end{equation}

  By taking the pullbacks of $1$-framed stratified bundles
  \[
    \begin{tikzcd}
      {\strat{X}'} \ar[r, "p"] \ar[d, "\xi'"'] \pullbackcorner &
      {\strat{X}} \ar[d, "\xi"] \\
      {\strat{Y} \cap \strat{N}} \ar[r] &
      {\strat{Y}}
    \end{tikzcd}
    \qquad
    \qquad
    \begin{tikzcd}
      {\strat{M}'} \ar[r, "q"] \ar[d, "f'"'] \pullbackcorner &
      {\strat{M}} \ar[d, "f"] \\
      {\strat{Y} \cap \strat{N}} \ar[r] &
      {\strat{N}}
    \end{tikzcd}
  \]
  we obtain a transverse configuration $\xi' \pitchfork f'$ 
  of $1$-mesh bundles over $\strat{Y} \cap \strat{N}$.
  The label maps from (\ref{eq:fct-conf-pack:labels}) restrict to labels
  \begin{equation}\label{eq:fct-conf-pack:labels-restrict}
    \begin{tikzcd}[row sep = large]
      {\Exit(\strat{X}')} \ar[d, "L_0 \circ p_*"'] &
      {\Exit(\strat{X}' \cap \strat{M}')} \ar[r] \ar[d, "L \circ (p \cap q)_*"{description}] \ar[l] &
      {\Exit(\strat{M}')} \ar[d, "L_1 \circ q_*"] \\
      {\cat{C}^\op} &
      {\Tw(\cat{C})} \ar[r] \ar[l] &
      {\cat{C}}
    \end{tikzcd}
  \end{equation}
  Then $\xi' \pitchfork f'$ together with~(\ref{eq:fct-conf-pack:labels-restrict})
  determine a map $R : \Exit(\strat{Y} \cap \strat{N}) \to \BConfL{1}{\cat{C}}$
  which fits into the diagram
  \begin{equation}\label{eq:fct-conf-pack:packed}
    \begin{tikzcd}[column sep = huge]
      {\Exit(\strat{Y})} \ar[d, "\classify{\xi; L_0}"']
      &
      {\Exit(\strat{Y} \cap \strat{N})} \ar[l, "p_*"'] \ar[r, "q_*"] \ar[d, "R"{description}]
      \ar[dl, "{\classify{\xi'; L_0 \circ p_*}}"{description}]
      \ar[dr, "{\classify{f'; L_1 \circ q_*}}"{description}] 
      &
      {\Exit(\strat{N})} \ar[d, "\classify{f; L_1}"] 
      \\
      {\BMeshClosedL{1}{C^\op}} \ar[d] &
      {\BConfL{1}{\cat{C}}} \ar[l] \ar[r] \ar[d, "\simeq"{description}] &
      {\BMeshOpenL{1}{\cat{C}}} \ar[d] \\
      {\BTrussClosedL{1}{\cat{C}}} &
      {\Tw(\BTrussOpenL{1}{\cat{C}})} \ar[l] \ar[r] &
      {\BTrussOpenL{1}{\cat{C}}}
    \end{tikzcd}
  \end{equation}
  This diagram is a labelling of the transverse configuration $\zeta \pitchfork g$
  in $\BTrussOpenL{1}{\cat{C}}$, and therefore determines a $k$-simplex
  in $\BConfL{n}{\BTrussOpenL{1}{\cat{C}}}$.

  This concludes our construction of the functor
  $\BConfL{n + 1}{\cat{C}} \to \BConfL{n}{\BTrussOpenL{1}{\cat{C}}}$.
  To see that it is an equivalence, we use that the map
  $\BConfL{1}{\cat{C}} \to \Tw(\ETrussOpenL{1}{\cat{C}})$ from Lemma~\ref{lem:fct-conf-tw-truss}
  is an equivalence.
  We can then reconstruct the $1$-mesh bundles $\xi : \strat{X} \to \strat{Y}$
  and $f : \strat{M} \to \strat{N}$ from the data in the diagram~(\ref{eq:fct-conf-pack:packed}) via Proposition~\ref{prop:fct-conf-classify}.
\end{proof}

\begin{thm}
  Let $\cat{C}$ be a $1$-category.
  Then there exists an equivalence
  \[
    \begin{tikzcd}
      {\BMeshClosedL{n}{\cat{C}^\op}} \ar[d, "\simeq"'] &
      {\BConfL{n}{\cat{C}}} \ar[d, dashed] \ar[l] \ar[r] &
      {\BMeshOpenL{n}{\cat{C}}} \ar[d, "\simeq"] \\
      {\BTrussClosedL{n}{\cat{C}^\op}} &
      {\Tw(\BTrussOpenL{n}{\cat{C}})} \ar[l] \ar[r] &
      {\BTrussOpenL{n}{\cat{C}}}
    \end{tikzcd}
  \]
\end{thm}
\begin{proof}
  When $\cat{C}$ is a $1$-category then so are $\BTrussOpenL{n}{\cat{C}}$
  and $\BTrussClosedL{n}{\cat{C}}$.
  By applying the packing equivalence of Lemma~\ref{lem:fct-conf-pack}
  in sequence $n$ times, we obtain an equivalence of the spans
  \[
    \begin{tikzcd}
      {\BMeshClosedL{n}{\cat{C}^\op}} \ar[d, "\simeq"'] &
      {\BConfL{n}{\cat{C}}} \ar[d, "\simeq"{description}] \ar[l] \ar[r] &
      {\BMeshOpenL{n}{\cat{C}}} \ar[d, "\simeq"] \\
      {\BMeshClosedL{0}{\BTrussClosedL{n}{\cat{C}^\op}}} &
      {\BConfL{0}{\BTrussOpenL{n}{\cat{C}}}} \ar[l] \ar[r] &
      {\BMeshOpenL{0}{\BTrussOpenL{n}{\cat{C}}}}
    \end{tikzcd}
  \]
  The claim then follows by applying the isomorphisms
  \[
    \begin{tikzcd}
      {\BMeshClosedL{0}{\BTrussClosedL{n}{\cat{C}^\op}}} \ar[d, "\cong"'] &
      {\BConfL{0}{\BTrussOpenL{n}{\cat{C}}}} \ar[d, "\cong"{description}] \ar[l] \ar[r] &
      {\BMeshOpenL{0}{\BTrussOpenL{n}{\cat{C}}}} \ar[d, "\cong"] \\ 
      {\BTrussClosedL{n}{\cat{C}^\op}} &
      {\Tw(\BTrussOpenL{n}{\cat{C}})} \ar[l] \ar[r] &
      {\BTrussOpenL{n}{\cat{C}}}
    \end{tikzcd}
  \]
  from Observation~\ref{obs:fct-mesh-closed-labelled-base-case} and
  Observation~\ref{obs:fct-conf-labelled-base-case}.
\end{proof}

\begin{cor}
  There exists an equivalence of $\infty$-categories
  \[
    \begin{tikzcd}
      {\BConf{n}} \ar[r, dashed, "\simeq"] \ar[d] &
      {\Tw(\BTrussOpen{n})} \ar[d] \\
      {\BMeshClosed{n} \times \BMeshOpen{n}} \ar[r, "\simeq"'] &
      {\BTrussClosed{n} \times \BTrussOpen{n}}
    \end{tikzcd}
  \]
\end{cor}

\subsection{Orthogonality}\label{sec:fct-conf-ortho}

We wish to extend our concept of transversality to pairs consisting of a closed $n$-mesh
$\strat{X}$ and any stratification $\strat{E}$ of $\R^n$ that does not need to be
an open $n$-mesh itself.
In the cases that $\strat{E}$ is meshable we would want that transversality of
$\strat{X}$ and $\strat{E}$ implies transversality of $\strat{X}$ with the coarsest
refining mesh of $\strat{E}$.
However, as illustrated below, there is no local
criterion that we can impose on the intersection of the strata in $\strat{X}$ and $\strat{E}$ which would guarantee this property.

\begin{example}
  Consider the $3$-dimensional diagram, presented in projection as follows:
  \[
    \begin{tikzpicture}[scale = 0.5]
      \fill[mesh-background] (0, 0) rectangle (6, 5);
      \draw[mesh-stratum] (1, 0) -- (1, 1) .. controls +(0, 1) and +(0, -1) .. (2, 3) -- (2, 5);
      \draw[mesh-stratum-over] (2, 1) .. controls +(0, 1) and +(0, -1) .. (1, 3);
      \draw[mesh-stratum] (2, 0) -- (2, 1) .. controls +(0, 1) and +(0, -1) .. (1, 3) -- (1, 5);
      \draw[mesh-stratum] (4, 0) -- (4, 5);
      \node[mesh-vertex] at (4, 3) {};
      \draw[mesh-stratum-dual] (3, 2) -- (5, 2) -- (5, 4) -- (3, 4) -- cycle;
      \node[mesh-vertex-dual] at (3, 2) {};
      \node[mesh-vertex-dual] at (5, 2) {};
      \node[mesh-vertex-dual] at (5, 4) {};
      \node[mesh-vertex-dual] at (3, 4) {};
    \end{tikzpicture}
  \]
  We have an essentially tame $3$-framed stratification $\strat{E}$, drawn in black, and
  a closed $3$-mesh $\strat{X}$, drawn in red.
  When we take the coarsest
  refining mesh $\strat{M}$ of $\strat{E}$,
  the result is not a transverse configuration of $\strat{X}$ and $\strat{M}$.
  \[  
    \begin{tikzpicture}[scale = 0.5]
      \fill[mesh-background] (0, 0) rectangle (6, 5);
      \draw[mesh-stratum] (1, 0) -- (1, 1) .. controls +(0, 1) and +(0, -1) .. (2, 3) -- (2, 5);
      \draw[mesh-stratum] (2, 0) -- (2, 1) .. controls +(0, 1) and +(0, -1) .. (1, 3) -- (1, 5);
      \draw[mesh-stratum] (0, 2) -- (6, 2);
      \draw[mesh-stratum] (0, 3) -- (6, 3);
      \draw[mesh-stratum] (4, 0) -- (4, 5);
      \node[mesh-vertex] at (4, 3) {};
      \node[mesh-vertex] at (1.5, 2) {};
      \node[mesh-vertex] at (1, 3) {};
      \node[mesh-vertex] at (2, 3) {};
      \draw[mesh-stratum-dual] (3, 2) -- (5, 2) -- (5, 4) -- (3, 4) -- cycle;
      \node[mesh-vertex-dual] at (3, 2) {};
      \node[mesh-vertex-dual] at (5, 2) {};
      \node[mesh-vertex-dual] at (5, 4) {};
      \node[mesh-vertex-dual] at (3, 4) {};
    \end{tikzpicture}
  \]
  The issue arises from the braid in $\strat{E}$ that induces a
  stratum in $\strat{M}$ intersecting $\strat{X}$ non-tranversally.
  We can not place a local condition on $\strat{E}$ and $\strat{X}$
  which excludes this situation since the braid is not within
  any arbitrarily small neigbourhood of $\strat{X}$.
\end{example}

We resolve this issue by imposing a stronger compatibility condition on the pair
$(\strat{X}, \strat{E})$ which is preserved by replacing $\strat{E}$ with
its coarsest refining mesh. We call this condition orthogonality and define
it inductively as follows.

\begin{definition}
  A closed $(n + 1)$-mesh $\strat{X}$ and a stratification
  $\strat{E}$ of $\R^{n + 1}$ are \defn{$\varepsilon$-orthogonal}
  for some $\varepsilon > 0$ when they satisfy:
  \begin{enumerate}
    \item For each singular height $h$ of $\strat{X}$, $x \in \R^n$ and
    $t \in \intOO{-\varepsilon, \varepsilon}$ we have
    \[
      \stratMap{\strat{M}}(x_1, \ldots, x_n, h)
      =
      \stratMap{\strat{M}}(x_1, \ldots, x_n, h + t)
    \]
    \item For each height $h \in \R$ the slices
    $\strat{X} \cap (\R^n \times \{ h \})$
    and
    $\strat{E} \cap (\R^n \times \{ h \})$
    are $\varepsilon$-orthogonal.
  \end{enumerate}
  In the base case $n = 0$ the unique closed $0$-mesh is $\varepsilon$-orthogonal to the
  unique stratification of $\R^0$ for any $\varepsilon > 0$.
  When $\strat{X}$ and $\strat{M}$ are $\varepsilon$-orthogonal we write $\strat{X} \perp_{\varepsilon} \strat{M}$
  and call the pair $(\strat{X}, \strat{M})$ an $\varepsilon$-orthogonal configuration.
\end{definition}

\begin{definition}
  Let $\xi : \strat{X} \to \strat{B}$ be a closed $n$-mesh bundle and
  $\varphi : \strat{E} \to \strat{B}$ an $n$-framed stratified bundle so that
  $\unstrat(\strat{E}) = \unstrat(\strat{B}) \times \R^n$.
  We say that $\xi$ and $\varphi$ are \defn{$\varepsilon$-orthogonal} 
  for some $\varepsilon > 0$ when for every
  $b \in \strat{B}$ the fibres $\xi^{-1}(b)$ and $\varphi^{-1}(b)$
  are $\varepsilon$-orthogonal.
  In this case we write $\xi \perp_\varepsilon \varphi$ and call the pair $(\xi, \varphi)$
  an $\varepsilon$-orthogonal configuration.
\end{definition}

\begin{definition}
  Let $\xi : \strat{X} \to \strat{B}$ be a closed $n$-mesh bundle and
  $\varphi : \strat{E} \to \strat{B}$ an $n$-framed stratified bundle so that
  $\unstrat(\strat{E}) = \unstrat(\strat{B}) \times \R^n$.
  We say $\xi$ and $\varphi$ are \defn{orthogonal} when there exists a
  $\varepsilon > 0$ such that $\xi$ and $\varphi$ are $\varepsilon$-orthogonal.
  We then write $\xi \perp \varphi$ and call the pair $(\xi, \varphi)$ an
  orthogonal configuration.
\end{definition}

\begin{observation}\label{obs:fct-conf-ortho-coarsest}
  Suppose that $\strat{X}$ is a closed $n$-mesh and $p : \strat{E} \to \strat{K}$ an open meshable $n$-framed stratified bundle such that $\strat{X} \perp p^{-1}(b)$ for every $b \in \strat{K}$.
  Let $f : \strat{M} \to \strat{B}$ be the coarsest open $n$-mesh bundle
  over a PL stratified pseudomanifold which refines $p$.
  Then we also have $\strat{X} \perp f^{-1}(b)$ for all $b \in \strat{B}$,
  since else we could coarsen $f$ further.
\end{observation}

\begin{observation}
  Let $(\xi, f)$ be an orthogonal configuration of a closed
  $n$-mesh bundle $\xi : \strat{X} \to \strat{B}$ and an open
  $n$-mesh bundle $f : \strat{M} \to \strat{B}$. Then $(\xi, f)$ is a transverse
  configuration.
\end{observation}

The orthogonal configurations of closed and open $n$-mesh bundles therefore
form a simplicial subset $\BConfOrtho{n}$ of the quasicategory $\BConf{n}$.
We show that $\BConfOrtho{n}$ is itself a quasicategory and equivalent
to $\BConf{n}$ itself.
In particular the equivalence $\BConf{n} \simeq \Tw(\BTrussOpen{n})$
restricts to an equivalence $\BConfOrtho{n} \simeq \Tw(\BTrussOpen{n})$.

\begin{definition}
  We denote by $\BConfOrtho{n}$ the simplicial subset of $\BConf{n}$ whose
  $k$-simplices consist of a pair $(\xi, f)$ of a closed $n$-mesh bundle
  $\xi : \strat{X} \to \DeltaStrat{k}$ and an open $n$-mesh bundle
  $f : \strat{M} \to \DeltaStrat{k}$ such that $\xi \perp f$.
\end{definition}

\begin{lemma}\label{lem:fct-conf-ortho-cocartesian}
  The projection $\BConfOrtho{n} \to \BMeshClosed{n}$ is a cocartesian
  fibration. Moreover a map in $\BConfOrtho{n}$ is cocartesian for the
  projection map $\BConfOrtho{n} \to \BMeshClosed{n}$ if and only if it
  is sent to an equivalence by the other projection $\BConfOrtho{n} \to \BMeshOpen{n}$.
\end{lemma}
\begin{proof}
  Analogous to Lemma~\ref{lem:fct-conf-cocartesian}.
\end{proof}

\begin{lemma}
  The inclusion map $\BConfOrtho{n} \hookrightarrow \BConf{n}$
  is a categorical equivalence.
\end{lemma}
\begin{proof}
  Both projection maps $\BConfOrtho{n} \to \BMeshClosed{n}$
  and $\BConf{n} \to \BMeshClosed{n}$ are cocartesian fibrations.
  It therefore suffices to show that for every closed $n$-mesh $\strat{X}$,
  the induced map on fibres
  \[
    \BConfOrtho{n} \times_{\BMeshClosed{n}} \{ \strat{X} \}
    \longrightarrow
    \BConf{n} \times_{\BMeshClosed{n}} \{ \strat{X} \}
  \]
  is a categorical equivalence.
  This can be demonstrated by stretching out transverse configurations
  within a neighbourhood of the singular strata of $\strat{X}$
  so that they become orthogonal.
\end{proof}

\begin{remark}
  We also have a projection map $\BConfOrtho{n} \to \BMeshOpen{n}$
  which sends an orthogonal configuration $\xi \perp f$ to the
  open $n$-mesh bundle $f$. This map is not an inner fibration,
  in contrast to the corresponding
  projection map $\BConf{n} \to \BMeshOpen{n}$.
\end{remark}

%% file: chapter-diagrams.tex
\chapter{Diagrammatic Higher Categories}\label{sec:d}

\section{Diagrammatic Spaces}\label{sec:d-diag-space}

The $n$-fold Segal space model for $(\infty, n)$-categories uses the $n$-fold product
$\FinOrd^n$ of the simplex category as its category of shapes.
A diagram in an $n$-fold Segal space is glued together from little cubes
with shape $\ord{k_1, \ldots, k_n}$ where $0 \leq k_i \leq 1$
that represent a morphism of dimension $k_1 + \cdots + k_n$.
For instance a $2$-morphism $\alpha$ in a $2$-fold Segal space $\cat{C}$ is a point
in the space of maps $\ord{1, 1} \to \cat{C}$.
The source and target $1$-morphisms of $\alpha$ are obtained by restricting
$\ord{1, 1} \to \cat{C}$ along the two inclusions $\ord{1, 0} \hookrightarrow \ord{1, 1}$.
The collection of shapes offered by $\FinOrd^n$ is therefore quite limited,
and does not allow us to ask directly for $n$-morphisms whose source or target is a composite
or even includes an isotopy.
In this section we show how we can use open $n$-trusses to generalise the shapes
with which $(\infty, n)$-categories can be probed.

\begin{observation}
  By Construction~\ref{con:fct-truss-grid-open} we have a fully faithful functor $\grid : \FinOrd^n \hookrightarrow \BTrussOpen{n}$
  from the $n$-fold product of $\FinOrd$ with itself.
  For brevity of notation we omit the dimension $n$ from the name of this functor.
  We then get an adjunction of $\infty$-categories of space-valued presheaves
  \[
    \begin{tikzcd}[column sep = large]
      {\PSh(\BTrussOpen{n})}
      \ar[r, shift left = 1.5, "(\grid)^*"{name=0}, anchor=center]
      &
      {\PSh(\FinOrd^n)}
      \ar[l, shift left = 1.5, "(\grid)_{*}"{name=1}, anchor=center]
  	  \arrow["{\scriptscriptstyle\dashv}"{anchor=center, rotate=-90}, draw=none, from=0, to=1]
    \end{tikzcd}
  \]
  given by restriction $(\grid)^*$ and right Kan extension $(\grid)_*$.
  Because $\grid$ is fully faithful, so is the right adjoint $(\grid)_*$.
  The $n$-simplicial spaces therefore embed into presheaves on $\BTrussOpen{n}$.
\end{observation}

\begin{example}
  Suppose that $\cat{C}$ is a $2$-uple Segal space and that
  $A$ is the $2$-dimensional atom which realises to the open mesh
  \[
    \begin{tikzpicture}[scale = 0.5, baseline=(current bounding box.center)]
      \fill[mesh-background] (0, 0) rectangle (4, 4);
      \node[mesh-vertex] at (2, 2) {};
      \draw[mesh-stratum] (1.5, 0) -- (1.5, 1) .. controls +(0, 0.2) and +(-0.5, 0) .. (2, 2);
      \draw[mesh-stratum] (2.5, 0) -- (2.5, 1) .. controls +(0, 0.2) and +(0.5, 0) .. (2, 2);
      \draw[mesh-stratum] (2, 2) -- (2, 4);
      \draw[mesh-stratum] (0, 2) -- (4, 2);
    \end{tikzpicture}
  \]
  Then a point in $(\grid)_* \cat{C}(A)$ is a compatible family of cells in $\cat{C}$,
  consisting of a map $\ord{k_1, k_2} \to \cat{C}$ for every map $\ord{k_1, k_2} \to A$.
  We illustrate this geometrically by using transverse configurations,
  rendering $A$ in black and the closed mesh associated to some $\ord{k_1, k_2} \in \FinOrd^2$
  in red.
  We only list the elements which are non-degenerate.
  \begin{itemize}
    \item Objects $\ord{0, 0} \to \cat{C}$ for each
      \[
        \begin{tikzpicture}[scale = 0.5, baseline=(current bounding box.center)]
          \fill[mesh-background] (0, 0) rectangle (4, 4);
          \node[mesh-vertex] at (2, 2) {};
          \draw[mesh-stratum] (1.5, 0) -- (1.5, 1) .. controls +(0, 0.2) and +(-0.5, 0) .. (2, 2);
          \draw[mesh-stratum] (2.5, 0) -- (2.5, 1) .. controls +(0, 0.2) and +(0.5, 0) .. (2, 2);
          \draw[mesh-stratum] (2, 2) -- (2, 4);
          \draw[mesh-stratum] (0, 2) -- (4, 2);
          \node[mesh-vertex-dual] at (1, 1) {};
        \end{tikzpicture}
        \qquad
        \begin{tikzpicture}[scale = 0.5, baseline=(current bounding box.center)]
          \fill[mesh-background] (0, 0) rectangle (4, 4);
          \node[mesh-vertex] at (2, 2) {};
          \draw[mesh-stratum] (1.5, 0) -- (1.5, 1) .. controls +(0, 0.2) and +(-0.5, 0) .. (2, 2);
          \draw[mesh-stratum] (2.5, 0) -- (2.5, 1) .. controls +(0, 0.2) and +(0.5, 0) .. (2, 2);
          \draw[mesh-stratum] (2, 2) -- (2, 4);
          \draw[mesh-stratum] (0, 2) -- (4, 2);
          \node[mesh-vertex-dual] at (3, 1) {};
        \end{tikzpicture}
        \qquad
        \begin{tikzpicture}[scale = 0.5, baseline=(current bounding box.center)]
          \fill[mesh-background] (0, 0) rectangle (4, 4);
          \node[mesh-vertex] at (2, 2) {};
          \draw[mesh-stratum] (1.5, 0) -- (1.5, 1) .. controls +(0, 0.2) and +(-0.5, 0) .. (2, 2);
          \draw[mesh-stratum] (2.5, 0) -- (2.5, 1) .. controls +(0, 0.2) and +(0.5, 0) .. (2, 2);
          \draw[mesh-stratum] (2, 2) -- (2, 4);
          \draw[mesh-stratum] (0, 2) -- (4, 2);
          \node[mesh-vertex-dual] at (1, 3) {};
        \end{tikzpicture}
        \qquad
        \begin{tikzpicture}[scale = 0.5, baseline=(current bounding box.center)]
          \fill[mesh-background] (0, 0) rectangle (4, 4);
          \node[mesh-vertex] at (2, 2) {};
          \draw[mesh-stratum] (1.5, 0) -- (1.5, 1) .. controls +(0, 0.2) and +(-0.5, 0) .. (2, 2);
          \draw[mesh-stratum] (2.5, 0) -- (2.5, 1) .. controls +(0, 0.2) and +(0.5, 0) .. (2, 2);
          \draw[mesh-stratum] (2, 2) -- (2, 4);
          \draw[mesh-stratum] (0, 2) -- (4, 2);
          \node[mesh-vertex-dual] at (3, 3) {};
        \end{tikzpicture}
        \qquad
        \begin{tikzpicture}[scale = 0.5, baseline=(current bounding box.center)]
          \fill[mesh-background] (0, 0) rectangle (4, 4);
          \node[mesh-vertex] at (2, 2) {};
          \draw[mesh-stratum] (1.5, 0) -- (1.5, 1) .. controls +(0, 0.2) and +(-0.5, 0) .. (2, 2);
          \draw[mesh-stratum] (2.5, 0) -- (2.5, 1) .. controls +(0, 0.2) and +(0.5, 0) .. (2, 2);
          \draw[mesh-stratum] (2, 2) -- (2, 4);
          \draw[mesh-stratum] (0, 2) -- (4, 2);
          \node[mesh-vertex-dual] at (2, 1) {};
        \end{tikzpicture}
      \]
    \item Horizontal $1$-morphisms $\ord{1, 0} \to \cat{C}$ between the objects for each
      \[
        \begin{tikzpicture}[scale = 0.5, baseline=(current bounding box.center)]
          \fill[mesh-background] (0, 0) rectangle (4, 4);
          \node[mesh-vertex] at (2, 2) {};
          \draw[mesh-stratum] (1.5, 0) -- (1.5, 1) .. controls +(0, 0.2) and +(-0.5, 0) .. (2, 2);
          \draw[mesh-stratum] (2.5, 0) -- (2.5, 1) .. controls +(0, 0.2) and +(0.5, 0) .. (2, 2);
          \draw[mesh-stratum] (2, 2) -- (2, 4);
          \draw[mesh-stratum] (0, 2) -- (4, 2);
          \draw[mesh-stratum-dual] (1, 1) -- (2, 1);
          \node[mesh-vertex-dual] at (1, 1) {};
          \node[mesh-vertex-dual] at (2, 1) {};
        \end{tikzpicture}
        \qquad
        \begin{tikzpicture}[scale = 0.5, baseline=(current bounding box.center)]
          \fill[mesh-background] (0, 0) rectangle (4, 4);
          \draw[mesh-stratum-dual] (2, 1) -- (3, 1);
          \node[mesh-vertex-dual] at (2, 1) {};
          \node[mesh-vertex-dual] at (3, 1) {};
          \node[mesh-vertex] at (2, 2) {};
          \draw[mesh-stratum] (1.5, 0) -- (1.5, 1) .. controls +(0, 0.2) and +(-0.5, 0) .. (2, 2);
          \draw[mesh-stratum] (2.5, 0) -- (2.5, 1) .. controls +(0, 0.2) and +(0.5, 0) .. (2, 2);
          \draw[mesh-stratum] (2, 2) -- (2, 4);
          \draw[mesh-stratum] (0, 2) -- (4, 2);
        \end{tikzpicture}
        \qquad
        \begin{tikzpicture}[scale = 0.5, baseline=(current bounding box.center)]
          \fill[mesh-background] (0, 0) rectangle (4, 4);
          \draw[mesh-stratum-dual] (1, 3) -- (3, 3);
          \node[mesh-vertex-dual] at (1, 3) {};
          \node[mesh-vertex-dual] at (3, 3) {};
          \node[mesh-vertex] at (2, 2) {};
          \draw[mesh-stratum] (0, 2) -- (4, 2);
          \draw[mesh-stratum] (1.5, 0) -- (1.5, 1) .. controls +(0, 0.2) and +(-0.5, 0) .. (2, 2);
          \draw[mesh-stratum] (2.5, 0) -- (2.5, 1) .. controls +(0, 0.2) and +(0.5, 0) .. (2, 2);
          \draw[mesh-stratum] (2, 2) -- (2, 4);
        \end{tikzpicture}
        \qquad
        \begin{tikzpicture}[scale = 0.5, baseline=(current bounding box.center)]
          \fill[mesh-background] (0, 0) rectangle (4, 4);
          \draw[mesh-stratum-dual] (1, 1) -- (3, 1);
          \node[mesh-vertex-dual] at (1, 1) {};
          \node[mesh-vertex-dual] at (3, 1) {};
          \node[mesh-vertex] at (2, 2) {};
          \draw[mesh-stratum] (0, 2) -- (4, 2);
          \draw[mesh-stratum] (1.5, 0) -- (1.5, 1) .. controls +(0, 0.2) and +(-0.5, 0) .. (2, 2);
          \draw[mesh-stratum] (2.5, 0) -- (2.5, 1) .. controls +(0, 0.2) and +(0.5, 0) .. (2, 2);
          \draw[mesh-stratum] (2, 2) -- (2, 4);
        \end{tikzpicture}
      \]
    \item Vertical $1$-morphisms $\ord{0, 1} \to \cat{C}$ between the objects for each
      \[
        \begin{tikzpicture}[scale = 0.5, baseline=(current bounding box.center)]
          \fill[mesh-background] (0, 0) rectangle (4, 4);
          \draw[mesh-stratum-dual] (1, 1) -- (1, 3);
          \node[mesh-vertex-dual] at (1, 1) {};
          \node[mesh-vertex-dual] at (1, 3) {};
          \node[mesh-vertex] at (2, 2) {};
          \draw[mesh-stratum] (0, 2) -- (4, 2);
          \draw[mesh-stratum] (1.5, 0) -- (1.5, 1) .. controls +(0, 0.2) and +(-0.5, 0) .. (2, 2);
          \draw[mesh-stratum] (2.5, 0) -- (2.5, 1) .. controls +(0, 0.2) and +(0.5, 0) .. (2, 2);
          \draw[mesh-stratum] (2, 2) -- (2, 4);
        \end{tikzpicture}
        \qquad
        \begin{tikzpicture}[scale = 0.5, baseline=(current bounding box.center)]
          \fill[mesh-background] (0, 0) rectangle (4, 4);
          \draw[mesh-stratum-dual] (3, 1) -- (3, 3);
          \node[mesh-vertex-dual] at (3, 1) {};
          \node[mesh-vertex-dual] at (3, 3) {};
          \node[mesh-vertex] at (2, 2) {};
          \draw[mesh-stratum] (0, 2) -- (4, 2);
          \draw[mesh-stratum] (1.5, 0) -- (1.5, 1) .. controls +(0, 0.2) and +(-0.5, 0) .. (2, 2);
          \draw[mesh-stratum] (2.5, 0) -- (2.5, 1) .. controls +(0, 0.2) and +(0.5, 0) .. (2, 2);
          \draw[mesh-stratum] (2, 2) -- (2, 4);
        \end{tikzpicture}
      \]
    \item A $2$-morphism $\ord{1, 1} \to \cat{C}$ for
      \[
        \begin{tikzpicture}[scale = 0.5, baseline=(current bounding box.center)]
          \fill[mesh-background] (0, 0) rectangle (4, 4);
          \draw[mesh-stratum-dual] (1, 1) rectangle (3, 3);
          \node[mesh-vertex-dual] at (1, 1) {};
          \node[mesh-vertex-dual] at (1, 3) {};
          \node[mesh-vertex-dual] at (3, 1) {};
          \node[mesh-vertex-dual] at (3, 3) {};
          \node[mesh-vertex] at (2, 2) {};
          \draw[mesh-stratum] (0, 2) -- (4, 2);
          \draw[mesh-stratum] (1.5, 0) -- (1.5, 1) .. controls +(0, 0.2) and +(-0.5, 0) .. (2, 2);
          \draw[mesh-stratum] (2.5, 0) -- (2.5, 1) .. controls +(0, 0.2) and +(0.5, 0) .. (2, 2);
          \draw[mesh-stratum] (2, 2) -- (2, 4);
        \end{tikzpicture}
      \]
    \item A witness of composition of horizontal $1$-morphisms $\ord{2, 0} \to \cat{C}$
      which composes to the chosen $1$-morphism $\ord{1, 0} \to \cat{C}$ from above
      \[
        \begin{tikzpicture}[scale = 0.5, baseline=(current bounding box.center)]
          \fill[mesh-background] (0, 0) rectangle (4, 4);
          \draw[mesh-stratum-dual] (1, 1) -- (3, 1);
          \node[mesh-vertex-dual] at (1, 1) {};
          \node[mesh-vertex-dual] at (2, 1) {};
          \node[mesh-vertex-dual] at (3, 1) {};
          \node[mesh-vertex] at (2, 2) {};
          \draw[mesh-stratum] (1.5, 0) -- (1.5, 1) .. controls +(0, 0.2) and +(-0.5, 0) .. (2, 2);
          \draw[mesh-stratum] (2.5, 0) -- (2.5, 1) .. controls +(0, 0.2) and +(0.5, 0) .. (2, 2);
          \draw[mesh-stratum] (2, 2) -- (2, 4);
          \draw[mesh-stratum] (0, 2) -- (4, 2);
        \end{tikzpicture}
        \quad
        \longrightarrow
        \quad
        \begin{tikzpicture}[scale = 0.5, baseline=(current bounding box.center)]
          \fill[mesh-background] (0, 0) rectangle (4, 4);
          \draw[mesh-stratum-dual] (1, 1) -- (3, 1);
          \node[mesh-vertex-dual] at (1, 1) {};
          \node[mesh-vertex-dual] at (3, 1) {};
          \node[mesh-vertex] at (2, 2) {};
          \draw[mesh-stratum] (1.5, 0) -- (1.5, 1) .. controls +(0, 0.2) and +(-0.5, 0) .. (2, 2);
          \draw[mesh-stratum] (2.5, 0) -- (2.5, 1) .. controls +(0, 0.2) and +(0.5, 0) .. (2, 2);
          \draw[mesh-stratum] (2, 2) -- (2, 4);
          \draw[mesh-stratum] (0, 2) -- (4, 2);
        \end{tikzpicture}
      \]
  \end{itemize}
  A point in $(\grid)_* \cat{C}(A)$ therefore is a $2$-morphism of $\cat{C}$
  whose domain factors as the composite of two chosen $1$-morphisms.
  In this way the presheaf $(\grid)_* \cat{C}$ lets us produce compatible families
  of morphisms in $\cat{C}$ according to each truss $T \in \BTrussOpen{2}$ interpreted
  as a string diagram.
  

  

\end{example}

\begin{definition}
  A presheaf $\psh{F} \in \PSh(\BTrussOpen{n})$ satisfies the
  \defn{Segal condition} when for every open $n$-truss $T$ the atomic covering
  diagram $U : \Exit(T)^{\op, \triangleright} \to \BTrussOpen{n}$ is sent to a
  limit diagram $U \circ \psh{F}$.
  We denote by $\Seg(\BTrussOpen{n})$ the reflective subcategory of $\PSh(\BTrussOpen{n})$
  consisting of the presheaves that satisfy the Segal condition.
\end{definition}

The Segal condition on $\PSh(\BTrussOpen{n})$ restricts to the Segal condition
on $\PSh(\FinOrd^n)$ for $n$-uple Segal spaces,
wherefore the restriction functor $\grid^*$ induces a functor
$\Seg(\BTrussOpen{n}) \to \Seg(\FinOrd^n)$.
To see that the functor $(\grid)_*$ also preserves the Segal condition,
we make the following observation.

\begin{definition}
  Let $E \in \BTrussOpen{n}$ be an open $n$-truss. We then write
  \[
    \Grid(E) := \BTrussOpen{n}(\grid(-), E).
  \]
  for the restricted representable functor.
\end{definition}

\begin{lemma}\label{lem:grid-freely-generated}
  Let $E \in \BTrussOpen{n}$ be an open $n$-truss.
  Then $\Grid(E)$ is the free $n$-uple Segal space with one generator
  for each atom $A \hookrightarrow E$.
\end{lemma}
\begin{proof}
  The claim is trivial for $n = 0$.
  When $n > 0$ we use the isomorphism $\BTrussOpen{n} \cong \BTrussOpenL{1}{\BTrussOpen{n - 1}}$ to write $E$ as an open $1$-truss labelled in $\BTrussOpen{n - 1}$
  and therefore a sequence of spans
  \begin{equation}\label{eq:grid-freely-generated:spans}
    \begin{tikzcd}
      {R_0} &
      {S_0} \ar[l] \ar[r] &
      {R_1} &
      {S_1} \ar[l] \ar[r] &
      {R_2} &
      \cdots \ar[l] \ar[r] &
      {R_k}
    \end{tikzcd}
  \end{equation}
  By induction we have that $\Grid(R_i)$ is the free $(n - 1)$-uple Segal space
  generated by the atoms in $R_i$ for all $0 \leq i \leq k$,
  and analogously $\Grid(S_i)$ is the free $(n - 1)$-uple Segal space
  generated by the atoms in $S_i$ for all $0 \leq i < k$.
  Via the maps of~(\ref{eq:grid-freely-generated:spans}) we then get a graph
  in the $\infty$-category of $(n - 1)$-fold Segal spaces:
  \[
    \begin{tikzcd}
      {G_1 := \coprod_{0 \leq i < k} \Grid(S_i)} \ar[r, shift left = 1] \ar[r, shift right = 1] &
      {\coprod_{0 \leq i \leq k} \Grid(R_i) =: G_0}
    \end{tikzcd}
  \]
  Then $\Grid(E)$ is the free Segal object in $\Seg(\FinOrd^{n - 1})$ generated
  by this graph, and the generators of $\Grid(E)$ are the generators of the
  free $(n - 1)$-fold Segal spaces of vertices $G_0$ and edges $G_1$.  
  When $A \hookrightarrow E$ is an atom with $\stype(A)_n = 0$,
  then $A \hookrightarrow E$ factors through a regular slice $R_i \hookrightarrow E$ for some
  $0 \leq i \leq k$. Therefore, the atom induces a generator in $G_0$.
  When $A \hookrightarrow E$ is an atom with $\stype(A)_n = 1$,
  then there is some $0 \leq i < k$ and a diagram of open $(n - 1)$-trusses
  \[
    \begin{tikzcd}
      {A_0} \ar[d] &
      {A_1} \ar[d] \ar[l] \ar[r] &
      {A_2} \ar[d] \\
      {R_i} &
      {S_i} \ar[l] \ar[r] &
      {R_{i + 1}}
    \end{tikzcd}
  \]
  where the vertical maps are inert and $A_1$ is an atom.
  Therefore, $A \hookrightarrow E$ corresponds exactly to the atom
  $A_1 \hookrightarrow S_i$ and therefore to a generator in $G_1$.
\end{proof}

\begin{cor}\label{cor:grid-segal-colimit}
  Let $E \in \BTrussOpen{n}$ be an open $n$-truss.
  When
  $U : \Exit(E)^{\op, \triangleright} \to \BTrussOpen{n}$
  is the atomic covering diagram of $E$,
  the induced diagram
  $\Grid(U(-))$
  is a colimit diagram in the $\infty$-category $\Seg(\FinOrd^n)$
  of $n$-uple Segal spaces.
\end{cor}
\begin{proof}
  This follows from Lemma~\ref{lem:grid-freely-generated} since
  $\Grid(E)$ is freely generated as an $n$-uple Segal space with one generator
  for every atom of $E$.
\end{proof}

\begin{example}
  Suppose we have an open $2$-truss $E \in \BTrussOpen{2}$ such that
  \[
    \TrussOpenReal{E}\quad\simeq\quad
    \begin{tikzpicture}[scale = 0.5, baseline=(current bounding box.center)]
      \fill[mesh-background] (0, -1) rectangle (9, 7);
      \draw[mesh-stratum] (1, -1) -- (1, 1) .. controls +(0, 0.2) and +(-1, 0) .. (2, 2);
      \draw[mesh-stratum] (3, -1) -- (3, 1) .. controls +(0, 0.2) and +(1, 0) .. (2, 2);
      \draw[mesh-stratum] (2, 2) -- (2, 5) .. controls +(0, 0.2) and +(-1, 0) .. (3, 6);
      \draw[mesh-stratum] (5, 0) -- (5, 4);
      \draw[mesh-stratum] (5, 4) .. controls +(-1, 0) and +(0, -0.2) .. (4, 5) .. controls +(0, 0.2) and +(1, 0) .. (3, 6);
      \draw[mesh-stratum] (5, 4) .. controls +(1, 0) and +(0, -0.2) .. (6, 5) .. controls +(0, 0.2) and +(-1, 0) .. (7, 6);
      \draw[mesh-stratum] (8, -1) -- (8, 5) .. controls +(0, 0.2) and +(1, 0) .. (7, 6);
      \draw[mesh-stratum] (7, 6) -- (7, 7);
      \node[mesh-vertex] at (5, 0) {};
      \node[mesh-vertex] at (5, 4) {};
      \node[mesh-vertex] at (2, 4) {};
      \node[mesh-vertex] at (8, 4) {};
      \node[mesh-vertex] at (2, 2) {};
      \node[mesh-vertex] at (5, 2) {};
      \node[mesh-vertex] at (8, 2) {};
      \node[mesh-vertex] at (7, 6) {};
      \node[mesh-vertex] at (3, 6) {};
      \node[mesh-vertex] at (1, 0) {};
      \node[mesh-vertex] at (3, 0) {};
      \node[mesh-vertex] at (8, 0) {};

      \draw[mesh-stratum] (0, 0) -- (9, 0);
      \draw[mesh-stratum] (0, 2) -- (9, 2);
      \draw[mesh-stratum] (0, 4) -- (9, 4);
      \draw[mesh-stratum] (0, 6) -- (9, 6);
    \end{tikzpicture}
  \]
  Then there is a $\ord{1, 1}$-cell in the $2$-uple Segal space $\Grid(E)$,
  corresponding to the unique active map $\ord{1, 1} \pto E$.
  This $\ord{1, 1}$-cell is the composite of the four $\ord{1, 1}$-cells: 
  \begin{align*}
    &\quad
    \begin{tikzpicture}[scale = 0.5, baseline=(current bounding box.center)]
      \fill[mesh-background] (0, 0) rectangle (9, 2);
      \draw[mesh-stratum] (0, 1) -- (9, 1);
      \draw[mesh-stratum] (2, 0) .. controls +(0, 0.2) and +(-1, 0) .. (3, 1);
      \draw[mesh-stratum] (4, 0) .. controls +(0, 0.2) and +(1, 0) .. (3, 1);
      \draw[mesh-stratum] (6, 0) .. controls +(0, 0.2) and +(-1, 0) .. (7, 1);
      \draw[mesh-stratum] (8, 0) .. controls +(0, 0.2) and +(1, 0) .. (7, 1);
      \draw[mesh-stratum] (7, 1) -- (7, 2);
      \node[mesh-vertex] at (3, 1) {};
      \node[mesh-vertex] at (7, 1) {};
    \end{tikzpicture}
    \\
    \circ_1 &\quad
    \begin{tikzpicture}[scale = 0.5, baseline=(current bounding box.center)]
      \fill[mesh-background] (0, 0) rectangle (9, 2);
      \draw[mesh-stratum] (0, 1) -- (9, 1);
      \draw[mesh-stratum] (5, 1) .. controls +(-1, 0) and +(0, -0.2) .. (4, 2);
      \draw[mesh-stratum] (5, 1) .. controls +(1, 0) and +(0, -0.2) .. (6, 2);
      \draw[mesh-stratum] (2, 0) -- (2, 2);
      \draw[mesh-stratum] (5, 0) -- (5, 1);
      \draw[mesh-stratum] (8, 0) -- (8, 2);
      \node[mesh-vertex] at (2, 1) {};
      \node[mesh-vertex] at (5, 1) {};
      \node[mesh-vertex] at (8, 1) {};
    \end{tikzpicture}
    \\
    \circ_1 &\quad
    \begin{tikzpicture}[scale = 0.5, baseline=(current bounding box.center)]
      \fill[mesh-background] (0, 0) rectangle (9, 2);
      \draw[mesh-stratum] (0, 1) -- (9, 1);
      \draw[mesh-stratum] (1, 0) .. controls +(0, 0.2) and +(-1, 0) .. (2, 1);
      \draw[mesh-stratum] (3, 0) .. controls +(0, 0.2) and +(1, 0) .. (2, 1);
      \draw[mesh-stratum] (2, 1) -- (2, 2);
      \draw[mesh-stratum] (5, 0) -- (5, 2);
      \draw[mesh-stratum] (8, 0) -- (8, 2);
      \node[mesh-vertex] at (2, 1) {};
      \node[mesh-vertex] at (5, 1) {};
      \node[mesh-vertex] at (8, 1) {};
    \end{tikzpicture}
    \\
    \circ_1 &\quad
    \begin{tikzpicture}[scale = 0.5, baseline=(current bounding box.center)]
      \fill[mesh-background] (0, 0) rectangle (9, 2);
      \draw[mesh-stratum] (0, 1) -- (9, 1);
      \draw[mesh-stratum] (1, 0) -- (1, 2);
      \draw[mesh-stratum] (3, 0) -- (3, 2);
      \draw[mesh-stratum] (5, 1) -- (5, 2);
      \draw[mesh-stratum] (8, 0) -- (8, 2);
      \node[mesh-vertex] at (1, 1) {};
      \node[mesh-vertex] at (3, 1) {};
      \node[mesh-vertex] at (5, 1) {};
      \node[mesh-vertex] at (8, 1) {};
    \end{tikzpicture}
  \end{align*}
  Each of these layers is itself the composite of its atoms:  
  \begin{align*}
    &\quad
    \left(
    \begin{tikzpicture}[scale = 0.5, baseline=(current bounding box.center)]
      \fill[mesh-background] (0, 0) rectangle (4, 2);
      \node[mesh-vertex] at (2, 1) {};
      \draw[mesh-stratum] (0, 1) -- (4, 1);
      \draw[mesh-stratum] (1, 0) .. controls +(0, 0.2) and +(-1, 0) .. (2, 1);
      \draw[mesh-stratum] (3, 0) .. controls +(0, 0.2) and +(1, 0) .. (2, 1);
    \end{tikzpicture}
    \quad\circ_2\quad
    \begin{tikzpicture}[scale = 0.5, baseline=(current bounding box.center)]
      \fill[mesh-background] (0, 0) rectangle (4, 2);
      \node[mesh-vertex] at (2, 1) {};
      \draw[mesh-stratum] (0, 1) -- (4, 1);
      \draw[mesh-stratum] (1, 0) .. controls +(0, 0.2) and +(-1, 0) .. (2, 1);
      \draw[mesh-stratum] (3, 0) .. controls +(0, 0.2) and +(1, 0) .. (2, 1);
      \draw[mesh-stratum] (2, 1) -- (2, 2);
    \end{tikzpicture}
    \right)
    \\
    \circ_1 &\quad
    \left(
    \begin{tikzpicture}[scale = 0.5, baseline=(current bounding box.center)]
      \fill[mesh-background] (0, 0) rectangle (2, 2);
      \node[mesh-vertex] at (1, 1) {};
      \draw[mesh-stratum] (0, 1) -- (2, 1);
      \draw[mesh-stratum] (1, 0) -- (1, 2);
    \end{tikzpicture}
    \quad\circ_2\quad
    \begin{tikzpicture}[scale = 0.5, baseline=(current bounding box.center)]
      \fill[mesh-background] (0, 0) rectangle (4, 2);
      \node[mesh-vertex] at (2, 1) {};
      \draw[mesh-stratum] (0, 1) -- (4, 1);
      \draw[mesh-stratum] (2, 1) .. controls +(-1, 0) and +(0, -0.2) .. (1, 2);
      \draw[mesh-stratum] (2, 1) .. controls +(1, 0) and +(0, -0.2) .. (3, 2);
      \draw[mesh-stratum] (2, 0) -- (2, 1);
    \end{tikzpicture}
    \quad\circ_2\quad
    \begin{tikzpicture}[scale = 0.5, baseline=(current bounding box.center)]
      \fill[mesh-background] (0, 0) rectangle (2, 2);
      \node[mesh-vertex] at (1, 1) {};
      \draw[mesh-stratum] (0, 1) -- (2, 1);
      \draw[mesh-stratum] (1, 0) -- (1, 2);
    \end{tikzpicture}
    \right)
    \\
    \circ_1 &\quad
    \left(
    \begin{tikzpicture}[scale = 0.5, baseline=(current bounding box.center)]
      \fill[mesh-background] (0, 0) rectangle (4, 2);
      \node[mesh-vertex] at (2, 1) {};
      \draw[mesh-stratum] (0, 1) -- (4, 1);
      \draw[mesh-stratum] (1, 0) .. controls +(0, 0.2) and +(-1, 0) .. (2, 1);
      \draw[mesh-stratum] (3, 0) .. controls +(0, 0.2) and +(1, 0) .. (2, 1);
      \draw[mesh-stratum] (2, 1) -- (2, 2);
    \end{tikzpicture}
    \quad\circ_2\quad
    \begin{tikzpicture}[scale = 0.5, baseline=(current bounding box.center)]
      \fill[mesh-background] (0, 0) rectangle (2, 2);
      \node[mesh-vertex] at (1, 1) {};
      \draw[mesh-stratum] (0, 1) -- (2, 1);
      \draw[mesh-stratum] (1, 0) -- (1, 2);
    \end{tikzpicture}
    \quad\circ_2\quad
    \begin{tikzpicture}[scale = 0.5, baseline=(current bounding box.center)]
      \fill[mesh-background] (0, 0) rectangle (2, 2);
      \node[mesh-vertex] at (1, 1) {};
      \draw[mesh-stratum] (0, 1) -- (2, 1);
      \draw[mesh-stratum] (1, 0) -- (1, 2);
    \end{tikzpicture}
    \right)
    \\
    \circ_1 &\quad
    \left(
    \begin{tikzpicture}[scale = 0.5, baseline=(current bounding box.center)]
      \fill[mesh-background] (0, 0) rectangle (2, 2);
      \node[mesh-vertex] at (1, 1) {};
      \draw[mesh-stratum] (0, 1) -- (2, 1);
      \draw[mesh-stratum] (1, 0) -- (1, 2);
    \end{tikzpicture}
    \quad\circ_2\quad
    \begin{tikzpicture}[scale = 0.5, baseline=(current bounding box.center)]
      \fill[mesh-background] (0, 0) rectangle (2, 2);
      \node[mesh-vertex] at (1, 1) {};
      \draw[mesh-stratum] (0, 1) -- (2, 1);
      \draw[mesh-stratum] (1, 0) -- (1, 2);
    \end{tikzpicture}
    \quad\circ_2\quad
    \begin{tikzpicture}[scale = 0.5, baseline=(current bounding box.center)]
      \fill[mesh-background] (0, 0) rectangle (2, 2);
      \node[mesh-vertex] at (1, 1) {};
      \draw[mesh-stratum] (0, 1) -- (2, 1);
      \draw[mesh-stratum] (1, 1) -- (1, 2);
    \end{tikzpicture}
    \quad\circ_2\quad
    \begin{tikzpicture}[scale = 0.5, baseline=(current bounding box.center)]
      \fill[mesh-background] (0, 0) rectangle (2, 2);
      \node[mesh-vertex] at (1, 1) {};
      \draw[mesh-stratum] (0, 1) -- (2, 1);
      \draw[mesh-stratum] (1, 0) -- (1, 2);
    \end{tikzpicture}
    \right)
  \end{align*}
  These atoms freely generate the $2$-uple Segal space $\Grid(E)$.
\end{example}

\begin{proposition}\label{prop:diagrammatic-segal-to-segal}
  The adjunction of $\infty$-categories of space-valued presheaves
  \[
    \begin{tikzcd}[column sep = large]
      {\PSh(\BTrussOpen{n})}
      \ar[r, shift left = 1.5, "(\grid)^*"{name=0}, anchor=center]
      &
      {\PSh(\FinOrd^n)}
      \ar[l, shift left = 1.5, "(\grid)_{*}"{name=1}, anchor=center, hook]
  	  \arrow["{\scriptscriptstyle\dashv}"{anchor=center, rotate=-90}, draw=none, from=0, to=1]
    \end{tikzcd}
  \]
  induced by the inclusion functor $\grid : \FinOrd^n \to \BTrussOpen{n}$
  restricts to an adjunction
  \[
    \begin{tikzcd}[column sep = large]
      {\Seg(\BTrussOpen{n})}
      \ar[r, shift left = 1.5, ""{name=0}, anchor=center]
      &
      {\Seg(\FinOrd^n)}
      \ar[l, shift left = 1.5, ""{name=1}, anchor=center, hook]
  	  \arrow["{\scriptscriptstyle\dashv}"{anchor=center, rotate=-90}, draw=none, from=0, to=1]
    \end{tikzcd}
  \]
\end{proposition}
\begin{proof}
  The left adjoint $(\grid)^*$ trivially preserves Segal objects,
  so it remains to show that $(\grid)_*$ does so as well.
  Let $\psh{F} \in \PSh(\FinOrd^n)$ be an $n$-uple Segal space.
  Then the value of the presheaf $(\grid)_* \psh{F} \in \PSh(\BTrussOpen{n})$
  on some $E \in \BTrussOpen{n}$ is the space of maps $\Grid(E) \to \psh{F}$
  in $\PSh(\FinOrd^n)$.
  Let $U : \Exit(E)^\triangleright \to \BTrussOpen{n}$ be the atomic covering
  diagram of $E$, then by Corollary~\ref{cor:grid-segal-colimit} we have that
  $\Grid(U(-))$ is a colimit diagram in $\Seg(\FinOrd^n)$.
  It therefore follows that 
  \[
    (\grid)_* \psh{F}(U(-)) \simeq
    \PSh(\FinOrd^n)(\Grid(U(-)), \psh{F}) \simeq
    \Seg(\FinOrd^n)(\Grid(U(-)), \psh{F})
  \]
  is a limit diagram. Hence, $(\grid)_* \psh{F}$ satisfies the Segal condition.
\end{proof}

\begin{example}
  While the induced functor $\Seg(\FinOrd^n) \to \Seg(\BTrussOpen{n})$ is fully faithful,
  it is not essentially surjective.
  Let us illustrate this by giving an example of a Segal object $\psh{F} \in \Seg(\BTrussOpen{2})$ that is not in the essential image.
  Suppose that $A$ is the atom
  \[
    A \quad = \quad
    \begin{tikzpicture}[scale = 0.5, baseline=(current bounding box.center)]
      \fill[mesh-background] (0, 0) rectangle (4, 4);
      \node[mesh-vertex] at (2, 2) {};
      \draw[mesh-stratum] (1, 0) -- (1, 1) .. controls +(0, 0.2) and +(-1, 0) .. (2, 2);
      \draw[mesh-stratum] (3, 0) -- (3, 1) .. controls +(0, 0.2) and +(1, 0) .. (2, 2);
      \draw[mesh-stratum] (2, 2) -- (2, 4);
      \draw[mesh-stratum] (0, 2) -- (4, 2);
    \end{tikzpicture}
  \]
  We let $\psh{C} \in \Seg(\FinOrd^2)$ be the $2$-fold Segal space that
  is generated by one $0$-cell $x$, three $1$-cells $f, g, h : x \to x$ and a $2$-cell
  $m : g \circ f \to h$.  
  We let $\psh{F} \hookrightarrow (\grid)_* \psh{C}$ be the smallest subobject
  in $\Seg(\BTrussOpen{2})$
  which contains the generators $\grid\ord{1, 0} \to (\grid)_*\psh{F}$ corresponding to $f$, $g$, $h$ and the generator $\grid\ord{1, 1} \to (\grid)_*\psh{F}$ corresponding to $m$.
  Then $\psh{F}(A)$ contains no point representing $m$ whose source restricts to the formal composite of $f$ and $g$:
  \[
    \begin{tikzpicture}[scale = 0.5, baseline=(current bounding box.center)]
      \fill[mesh-background] (0, 0) rectangle (4, 4);
      \draw[mesh-stratum] (1, 0) -- +(0, 4);
      \draw[mesh-stratum] (3, 0) -- +(0, 4);
      \node at (1, -0.5) {$f$};
      \node at (3, -0.5) {$g$};
      \node at (1, 4.5) {$f$};
      \node at (3, 4.5) {$g$};
    \end{tikzpicture}
    \in \psh{F}(\ord{2, 0})
    \qquad
    \begin{tikzpicture}[scale = 0.5, baseline=(current bounding box.center)]
      \fill[mesh-background] (0, 0) rectangle (4, 4);
      \node[mesh-vertex] at (2, 2) {};
      \draw[mesh-stratum] (2, 0) -- (2, 2);
      \draw[mesh-stratum] (2, 2) -- (2, 4);
      \draw[mesh-stratum] (0, 2) -- (4, 2);
      \node at (2, -0.5) {$g \circ f$};
      \node at (2, 4.5) {$h$};
      \node at (2.5, 2.5) {$m$};
    \end{tikzpicture}
    \in \psh{F}(\ord{1, 1})
    \qquad
    \begin{tikzpicture}[scale = 0.5, baseline=(current bounding box.center)]
      \fill[mesh-background] (0, 0) rectangle (4, 4);
      \node[mesh-vertex] at (2, 2) {};
      \draw[mesh-stratum] (1, 0) -- (1, 1) .. controls +(0, 0.2) and +(-1, 0) .. (2, 2);
      \draw[mesh-stratum] (3, 0) -- (3, 1) .. controls +(0, 0.2) and +(1, 0) .. (2, 2);
      \draw[mesh-stratum] (2, 2) -- (2, 4);
      \draw[mesh-stratum] (0, 2) -- (4, 2);
      \node at (1, -0.5) {$f$};
      \node at (3, -0.5) {$g$};
      \node at (2, 4.5) {$h$};
      \node at (2.5, 2.5) {$m$};
    \end{tikzpicture}
    \not\in \psh{F}(A)
  \]
  As a result $\psh{F}$ is not in the essential image of $(\grid)_*$.
\end{example}

\begin{construction}
  For an atom $A \in \BTrussOpen{n}$ we write $\partial A$ for the presheaf
  on $\BTrussOpen{n}$ that is the colimit of the representable presheaves
  $\BTrussOpen{n}(-, B)$ for all embeddings $B \hookrightarrow A$ except for the identity.
  When $\psh{F}$ is a presheaf on $\BTrussOpen{n}$, we write $\psh{F}(\partial A)$
  for the space of maps of presheaves $\partial A \to \psh{F}$, i.e.\
  the limit of $\psh{F}(B)$ for all non-trivial embeddings $B \hookrightarrow A$.
\end{construction}

We also recall the singular shape construction from~\S\ref{sec:fct-embed-stype} which assigns
to every mesh atom $\strat{A}$ the mesh atom $\shape{\strat{A}} := \gridMeshOpen^n(\stype(\strat{A}))$ together with an active bordism $\shape{\strat{A}} \pto \strat{A}$.
Via the equivalence between open meshes and open trusses, the singular shape
construction transfers to truss atoms. We note that for every truss atom $A$,
the singular shape $\shape{A}$ is within the image of $\gridTrussOpen^n : \FinOrd^n \hookrightarrow \BTrussOpen{n}$.

\begin{definition}
  A presheaf $\psh{F} \in \PSh(\BTrussOpen{n})$ is \defn{regular}
  when for every atom $A \in \BTrussOpen{n}$ the induced diagram of spaces
  \[
    \begin{tikzcd}
      {\psh{F}(A)} \ar[r] \ar[d] &
      {\psh{F}(\shape{A})} \ar[d] \\
      {\psh{F}(\partial A)} \ar[r] &
      {\psh{F}(\partial \shape{A})}
    \end{tikzcd}
  \]
  is a pullback square. The presheaf $\psh{F}$ is an
  \defn{$n$-diagrammatic space} when it is regular and satisfies the Segal condition.
  We denote by $\Seg^\regular(\BTrussOpen{n})$ the full subcategory of
  $\PSh(\BTrussOpen{n})$ consisting of the $n$-diagrammatic spaces.
\end{definition}

\begin{lemma}\label{lem:grid-regular-colimit}
  Let $0 \leq n \leq \infty$ and $A \in \BTrussOpen{n}$ an atom. Then
  \[
    \begin{tikzcd}
      {\Grid(\partial \shape{A})} \ar[r] \ar[d] &
      {\Grid(\partial A)} \ar[d] \\
      {\Grid(\shape{A})} \ar[r] &
      {\Grid(A)}
    \end{tikzcd}
  \]
  is a pushout square in the $\infty$-category $\Seg(\FinOrd^n)$
  of $n$-uple Segal spaces.
\end{lemma}
\begin{proof}
  By Lemma~\ref{lem:grid-freely-generated}
  the $n$-uple Segal space $\Grid(A)$ is freely generated
  with one generator for each atom embedded in $A$, including the
  identity embedding $\id : A \to A$.
  The boundary $\Grid(\partial A)$ is also freely generated
  as an $n$-uple Segal space with the same generators except $A$ itself.
  The pushout adds the missing generator.
\end{proof}

\begin{thm}
  The adjunction of $\infty$-categories of space-valued presheaves
  \[
    \begin{tikzcd}[column sep = large]
      {\PSh(\BTrussOpen{n})}
      \ar[r, shift left = 1.5, "(\grid)^*"{name=0}, anchor=center]
      &
      {\PSh(\FinOrd^n)}
      \ar[l, shift left = 1.5, "(\grid)_{*}"{name=1}, anchor=center]
  	  \arrow["{\scriptscriptstyle\dashv}"{anchor=center, rotate=-90}, draw=none, from=0, to=1]
    \end{tikzcd}
  \]
  induced by the inclusion functor $\grid : \FinOrd^n \hookrightarrow \BTrussOpen{n}$
  restricts to an adjoint equivalence
  \[
    \begin{tikzcd}[column sep = large]
      {\Seg^\regular(\BTruss{n})}
      \ar[r, shift left = 1.5, ""{name=0}, anchor=center]
      &
      {\Seg(\FinOrd^n)}
      \ar[l, shift left = 1.5, ""{name=1}, anchor=center]
  	  \arrow["{\scriptscriptstyle\simeq}"{anchor=center}, draw=none, from=0, to=1]
    \end{tikzcd}
  \]
  between the $\infty$-categories of $n$-diagrammatic spaces and $n$-uple Segal spaces.
\end{thm}
\begin{proof}
  The restriction functor $(\grid)^*$ preserves the Segal condition.
  For the other direction, suppose that $\psh{F} \in \Seg(\FinOrd^n)$ is an $n$-uple
  Segal space. By Proposition~\ref{prop:diagrammatic-segal-to-segal} we already know
  that $(\grid)_* \psh{F}$ also satisfies the Segal condition.
  Regularity of $(\grid)_* \psh{F}$ then follows from Lemma~\ref{lem:grid-regular-colimit}.
   
  Therefore, the adjunction $(\grid)^* \dashv (\grid)_*$ restricts to an adjunction
  $\Seg^\regular(\BTrussOpen{n}) \rightleftarrows \Seg(\FinOrd^n)$;
  it remains to show that this adjunction is an equivalence.
  Because $\grid$ is fully faithful, the counit $(\grid)^* (\grid)_{*} \to \id$
  is a natural equivalence and therefore $(\grid)^* (\grid)_{*} \psh{F} \to \psh{F}$
  is an equivalence for each $\psh{F} \in \Seg(\FinOrd^n)$.
  Now let $\psh{G} \in \Seg^\regular(\BTrussOpen{n})$.
  We show that for every open $n$-truss $E \in \BTrussOpen{n}$ the value of
  the unit $\psh{G}(E) \to (\grid)_{*} (\grid)^* \psh{G}(E)$ is an equivalence,
  using induction on the structure of $E$.
  The base case of $E = \ord{0}^n$ is clear.  
  If $E$ is an atom, then we get an induced diagram
  \[
  	\begin{tikzcd}[row sep = small, column sep = small]
  		& {(\grid)_{*} (\grid)^* \psh{G}(E)} && {(\grid)_{*} (\grid)^* \psh{G}(\shape{E})} \\
  		{\psh{G}(E)} && {\psh{G}(\shape{E})} \\
  		& {(\grid)_{*} \grid^* \psh{G}(\partial E)} && {(\grid)_{*} (\grid)^* \psh{G}(\partial \shape{E})} \\
  		{\psh{G}(\partial E)} && {\psh{G}(\partial \shape{E})}
  		\arrow[from=2-1, to=4-1]
  		\arrow[from=4-1, to=4-3]
  		\arrow[from=1-4, to=3-4]
  		\arrow[from=1-2, to=3-2]
  		\arrow[from=3-2, to=3-4]
  		\arrow[from=1-2, to=1-4]
  		\arrow[from=2-1, to=1-2]
  		\arrow[from=4-1, to=3-2, "\simeq"]
  		\arrow[from=4-3, to=3-4, "\simeq"]
  		\arrow[from=2-3, to=1-4, "\simeq"]
  		\arrow[crossing over, from=2-1, to=2-3]
  		\arrow[crossing over, from=2-3, to=4-3]
  	\end{tikzcd}
  \]
  The front square is a pullback square because $\psh{G}$ is regular by assumption,
  and the back square is a pullback square because $(\grid)_{*}$ sends any
  presheaf on $\FinOrd^n$ to a regular presheaf on $\BTrussOpen{n}$ as shown above.
  The diagonal maps on the right are equivalences since $\shape{E}$ is within the
  image of $\grid$.
  The diagonal map on the bottom right is an equivalence by induction.
  Therefore, it follows that the diagonal map on the top left must be an equivalence.

  When $E$ is not an atom, then there is a covering diagram
  $U : \Entr(E)^\triangleright \to \BTrussOpen{n}$ which covers $E$ with atoms.
  By induction on the number of atoms in the open truss we have that $\psh{G}(U(i)) \to (\grid)_{*}(\grid)^* \psh{G}(U(i))$
  is an equivalence for each $i \in \Entr(E)$.
  We have that $\psh{G}$ is Segal by assumption
  and $(\grid)_{*} (\grid)^* \psh{G}$ is Segal because both functors preserve
  Segal objects, as shown above.
  Therefore, it follows that $\psh{G}(E) \to (\grid)_{*} (\grid)^* \psh{G}(E)$
  must be an equivalence.
  This concludes the proof that the unit $\psh{G} \to (\grid)_{*} (\grid)^* \psh{G}$
  is an equivalence.
\end{proof}


\begin{para}
  Suppose that $\psh{F}$ is an $n$-diagrammatic space.
  We then say that $\psh{F}$ is globular when the $n$-uple Segal space $(\grid)^* \psh{F}$
  is globular, i.e.\ an $n$-fold Segal space.
  Similarly, we say that $\psh{F}$ is complete when the $n$-uple Segal space
  $(\grid)^* \psh{F}$ is complete.
  Then the $\infty$-category $\CatN{n}$ of $(\infty, n)$-categories is the reflective
  subcategory of $\Seg^\regular(\BTrussOpen{n})$ consisting of the globular and complete
  objects.
\end{para}

\begin{para}
  We have so far defined $n$-diagrammatic spaces as presheaves on $\BTrussOpen{n}$
  satisfying the regularity and the sheaf condition.
  The sequence of equivalences
  \[
    \BMeshOpen{n} \simeq
    \BTrussOpen{n} \simeq
    \BTrussClosed{n, \op} \simeq
    \BMeshClosed{n, \op}
  \]
  induces a sequence of equivalences between the $\infty$-categories of presheaves
  \[
    \PSh(\BMeshOpen{n}) \simeq
    \PSh(\BTrussOpen{n}) \simeq
    \PSh(\BTrussClosed{n, \op}) \simeq
    \PSh(\BMeshClosed{n, \op}).
  \]
  We say that a presheaf in any of these $\infty$-categories is an $n$-diagrammatic
  space if it becomes an $n$-diagrammatic space as a presheaf on $\BTrussOpen{n}$
  by applying the equivalences.  
  In particular in \S\ref{sec:d-mfld}, we will construct an $n$-diagrammatic space of manifold diagrams
  as a presheaf on $\BMeshClosed{n, \op}$, i.e. as a functor $\BMeshClosed{n} \to \Space$.
\end{para}

\section{Manifold Diagrams}\label{sec:d-mfld}

Open $n$-meshes capture all the shapes that we would expect from an $n$-manifold diagram
calculus and, via diagrammatic spaces, we can probe any $(\infty, n)$-category with
open $n$-meshes. However, the space of automorphisms of any open $n$-mesh is contractible.
This conflicts with our goal of an $n$-dimensional graphical calculus in which the weak interchange
law of higher categories is captured by the space of isotopies.
Below we describe a notion of $n$-manifold diagrams that generalises string diagrams and
surface diagrams and admits enough isotopies.




In \S\ref{sec:d-mfld-local} we characterise the local properties of $n$-manifold
diagrams so that their strata correspond to shapes of cells in $(\infty, n)$-categories.
In particular the strata are manifolds of varying dimension and satisfy a generalisation
of the progressivity condition of string diagrams.
In~\S\ref{sec:d-mfld-global} we define $n$-manifold diagrams supported on any closed $n$-mesh $\strat{X}$. When $\strat{X}$ is a closed $n$-cube, this specialises to a globular $n$-manifold diagram. Other choices of $\strat{X}$ lead to a natural definition of a pasting diagram of $n$-manifold diagrams.
Moving beyond individual diagrams, we define isotopies of $n$-manifold diagrams
in~\S\ref{sec:d-mfld-isotopy}. We show how $n$-manifold diagrams on a closed $n$-mesh
$\strat{X}$ together with their isotopies define a space $\MfldDiag^n(\strat{X})$.
In~\S\ref{sec:d-mfld-functorial} we see that these spaces organise into a functor
$\MfldDiag^n(-) : \BMeshClosed{n} \to \Space$ which we show to be an $n$-diagrammatic
space in~\S\ref{sec:d-mfld-diag-space}.
Finally, we explore an alternative but equivalent presentation of $n$-manifold diagrams in~\S\ref{sec:d-mfld-extended}
that will be useful when combinatorialising manifold diagrams in~\S\ref{sec:d-comb}.

\subsection{The Local Geometry of Manifold Diagrams}\label{sec:d-mfld-local}

As the first step towards characterising $n$-manifold diagrams
we study the local conditions that such a diagram should satisfy.
To give justice to the name, every stratum of an $n$-manifold diagram should be a $k$-manifold
for some $k \in \ord{n}$.
We will work with $n$-framed stratified spaces $\strat{M}$
that are equipped with a labelling map $\lbl{\strat{M}} : \Exit(\strat{M}) \to \ord{n}$.
A priori this map is arbitrary, but once we have placed the appropriate local conditions
on $\strat{M}$ the labelling map will indicate the dimension of the strata
of $\strat{M}$.

A $k$-dimensional stratum of an $n$-manifold diagram should correspond to an
$(n - k)$-dimensional cell of an $(\infty, n)$-category.
Via our discussion of $n$-diagrammatic spaces in \S\ref{sec:d-diag-space} we have identified
atoms as appropriate shapes for such cells.
We therefore require an $n$-manifold diagram to be covered by neighbourhoods
$\strat{U}$ that each admit a refinement map $r : \strat{A} \to \strat{U}$ from
some atom $\strat{A}$.

String diagrams satisfy the progressivity condition which guarantees that
the edges progress upwards, in the direction from the source to the target boundary,
without turning around or having sections that are horizontal.
Equivalently, the height function is injective when restricted to any edge of
the string diagram.
An $n$-manifold diagram should satisfy a generalisation of the progressivity
condition for strata of all dimensions.
We achieve this by placing a condition on the atoms that refine the local
neighbourhoods of the diagram.
Recall from Observation~\ref{obs:fct-embed-stype-depth} that the singular depth of an atom
$\strat{A} \in \BMeshOpenL{n}{\ord{n}}$ is the largest $d \geq 0$ such that
there is an atom $\strat{B} \in \BMeshOpenL{n - d}{\ord{n}}$ and an isomorphism
$\strat{A} \cong \strat{B} \times \R^d$
of labelled $n$-framed stratified spaces.

\begin{definition}
  An atom $\strat{A} \in \BMeshOpenL{n}{\ord{n}}$ is \defn{progressive} when
  its singular depth $\sdepth(\strat{A})$ agrees with its label $\lbl{\strat{A}}(\bot_{\strat{A}})$ and every proper subatom of $\strat{A}$ is progressive.
\end{definition}

\begin{definition}
  An \defn{$n$-framed basic} is an $n$-framed stratified space $\strat{U}$
  with a conservative labelling $\lbl{\strat{U}} : \Exit(\strat{U}) \to \ord{n}$
  such that there exists a progressive atom $\strat{A} \in \BMeshOpenL{n}{\ord{n}}$
  together with a refinement map $r : \strat{A} \to \strat{U}$ of labelled
  $n$-framed stratified spaces.
\end{definition}

\begin{example}
  Consider string diagrams for $2$-categories.
  The objects and $1$-cells are presented by $2$-dimensional and $1$-dimensional
  strata, respectively, that locally and up to isomorphism are of the following
  shape:
  \[
    \begin{tikzpicture}[scale = 0.5, baseline=(current bounding box.center)]
      \fill[mesh-background] (0, 0) rectangle (4, 4);
      \draw[mesh-stratum] (2, 0) -- (2, 4);
    \end{tikzpicture}
    \qquad
    \begin{tikzpicture}[scale = 0.5, baseline=(current bounding box.center)]
      \fill[mesh-background] (0, 0) rectangle (4, 4);
    \end{tikzpicture}
  \]
  Both of these are already mesh atoms.
  The $2$-cells are rendered in the string diagram as $0$-dimensional strata,
  whose domain and codomain is signified by the $1$-strata that approach
  the $0$-stratum from above and below.
  For any number of $1$-cells in the domain and codomain, we can find an atom
  that refines the local neighbourhood around the $0$-stratum:
  \[
    \begin{tikzpicture}[scale = 0.5, baseline=(current bounding box.center)]
      \fill[mesh-background] (0, 0) rectangle (4, 4);
      \node[mesh-vertex] at (2, 2) {};
      \draw[mesh-stratum] (1, 0) .. controls +(0, 0.4) and +(-1, 0) .. (2, 2);
      \draw[mesh-stratum] (3, 0) .. controls +(0, 0.4) and +(1, 0) .. (2, 2);
      \draw[mesh-stratum] (2, 2) -- (2, 4);
      \draw[mesh-stratum] (0, 2) -- +(4, 0);
    \end{tikzpicture}
    \quad
    \longrightarrow
    \quad
    \begin{tikzpicture}[scale = 0.5, baseline=(current bounding box.center)]
      \fill[mesh-background] (0, 0) rectangle (4, 4);
      \node[mesh-vertex] at (2, 2) {};
      \draw[mesh-stratum] (1, 0) .. controls +(0, 0.4) and +(-1, 0) .. (2, 2);
      \draw[mesh-stratum] (3, 0) .. controls +(0, 0.4) and +(1, 0) .. (2, 2);
      \draw[mesh-stratum] (2, 2) -- (2, 4);
    \end{tikzpicture}
  \]
\end{example}

\begin{example}
  Suppose that we have a map $f : X \otimes Y \to Z$ in a braided monoidal
  category $\cat{C}$. Colouring $X$, $Y$, $Z$ in blue, orange and black, respectively,
  we would render the map $f$ in a string diagram for $\cat{C}$ like this:
  \[
    \begin{tikzpicture}[scale = 0.5, baseline=(current bounding box.center)]
      \fill[mesh-background] (0, 0) rectangle (6, 6);
      \draw[mesh-stratum-blue] (2, 0) -- (3, 3);
      \draw[mesh-stratum-orange] (4, 0) -- (3, 3);
      \draw[mesh-stratum] (3, 6) -- (3, 3);
      \node[mesh-vertex] at (3, 3) {};
    \end{tikzpicture}
  \]
  Because $\cat{C}$ is braided monoidal the string diagrams for $\cat{C}$
  are projections of diagrams in $\R^3$. In particular the string diagram 
  for $f$ would be the projection of the following $3$-framed basic to the last
  two coordinate directions:
  \[
    \begin{tikzpicture}[scale = 0.3, baseline=(current bounding box.center)]
      \begin{scope}[shift={(0, 0, 7)}]
        \draw[->] (0,0,0) -- (1,0,0) node[anchor=west]{\tiny{3}};
        \draw[->] (0,0,0) -- (0,1,0) node[anchor=south]{\tiny{2}};
        \draw[->] (0,0,0) -- (0,0,1) node[anchor=north east]{\tiny{1}};
      \end{scope}
      
      \meshBoundingBox{6}{6}{6}
      
      \begin{scope}[canvas is zy plane at x = 0]
        \coordinate (s0-0-a) at (3, 2);
        \coordinate (s0-0-b) at (3, 4);
      \end{scope}

      \begin{scope}[canvas is zy plane at x = 3]
        \coordinate (s1-0-center) at (3, 3);
      \end{scope}

      \begin{scope}[canvas is zy plane at x = 6]
        \coordinate (s2-0-a) at (3, 3);
      \end{scope}

      \draw[mesh3-stratum-blue] (s0-0-a) -- (s1-0-center);
      \draw[mesh3-stratum-orange] (s0-0-b) -- (s1-0-center);
      \draw[mesh3-stratum] (s2-0-a) -- (s1-0-center);
      \node[mesh3-vertex] at (s1-0-center) {};
    \end{tikzpicture}
  \]
\end{example}

\begin{example}
  The following is a framed basic that represents a $3$-dimensional cell
  for a (directed) associativity law of a binary operation:
  \[
    \begin{tikzpicture}[scale = 0.5, baseline=(current bounding box.center)]
      \begin{scope}[shift={(0, 0, 7)}]
        \draw[->] (0,0,0) -- (1,0,0) node[anchor=west]{\tiny{3}};
        \draw[->] (0,0,0) -- (0,1,0) node[anchor=south]{\tiny{2}};
        \draw[->] (0,0,0) -- (0,0,1) node[anchor=north east]{\tiny{1}};
      \end{scope}
      
      \meshBoundingBox{6}{6}{6}
      
      \begin{scope}[canvas is zy plane at x = 0]
        \coordinate (s0-0-a) at (2, 2);
        \coordinate (s0-0-b) at (3.5, 4);

        \coordinate (s0-1-src-0) at (1, 0);
        \coordinate (s0-1-src-1) at (3, 0);
        \coordinate (s0-1-src-2) at (5, 0);
        \coordinate (s0-1-tgt-0) at (3.5, 6);
      \end{scope}

      \begin{scope}[canvas is zy plane at x = 3]
        \coordinate (s1-0-center) at (3, 3);

        \coordinate (s1-1-src-0) at (1, 0);
        \coordinate (s1-1-src-1) at (3, 0);
        \coordinate (s1-1-src-2) at (5, 0);
        \coordinate (s1-1-tgt-0) at (3, 6);
      \end{scope}

      \begin{scope}[canvas is zy plane at x = 6]
        \coordinate (s2-0-a) at (4, 2);
        \coordinate (s2-0-b) at (2.5, 4);

        \coordinate (s2-1-src-0) at (1, 0);
        \coordinate (s2-1-src-1) at (3, 0);
        \coordinate (s2-1-src-2) at (5, 0);
        \coordinate (s2-1-tgt-0) at (2.5, 6);
      \end{scope}

      \draw[mesh3-surface]
        (s0-1-src-0)
        -- (s0-0-a)
        -- (s1-0-center)
        -- (s2-0-b)
        -- (s2-1-src-0)
        -- (s1-1-src-0)
        -- cycle;

      \draw[mesh3-surface]
        (s0-1-src-1)
        -- (s0-0-a)
        -- (s1-0-center)
        -- (s2-0-a)
        -- (s2-1-src-1)
        -- (s1-1-src-1)
        -- cycle;

      \draw[mesh3-surface]
        (s0-1-src-2)
        -- (s0-0-b)
        -- (s1-0-center)
        -- (s2-0-a)
        -- (s2-1-src-2)
        -- (s1-1-src-2)
        -- cycle;

      \draw[mesh3-surface]
        (s0-1-tgt-0)
        -- (s0-0-b)
        -- (s1-0-center)
        -- (s2-0-b)
        -- (s2-1-tgt-0)
        -- (s1-1-tgt-0)
        -- cycle;

      \draw[mesh3-surface]
        (s0-0-a)
        -- (s0-0-b)
        -- (s1-0-center);

      \draw[mesh3-surface]
        (s2-0-a)
        -- (s2-0-b)
        -- (s1-0-center);
      
      \draw[mesh3-stratum] (s0-0-a) -- (s1-0-center);
      \draw[mesh3-stratum] (s0-0-b) -- (s1-0-center);
      \draw[mesh3-stratum] (s2-0-a) -- (s1-0-center);
      \draw[mesh3-stratum] (s2-0-b) -- (s1-0-center);

      \node[mesh3-vertex] at (s1-0-center) {};
    \end{tikzpicture}
  \]  
\end{example}

This definition of a framed basic implies that the strata are manifolds
whose dimension agrees with the label in $\ord{n}$. Moreover, each stratum
is progressive in a sense that generalises the progressivity condition of
string diagrams.
Because both of these properties are local conditions, they are then 
automatically satisfied by any $n$-manifold diagram that is covered with framed basics.

\begin{observation}
  Let $\strat{U}$ be an $n$-framed basic and $i \in \stratPos{\strat{U}}$ a stratum 
  labelled with $k \in \ord{n}$.  
  Then the projection $\R^n \to \R^k$ to the last $k$ dimensions
  restricts to an open immersion $\strat{U}_i \to \R^k$.
  In particular $\strat{U}_i$ is a topological $k$-manifold.
\end{observation}




\begin{example}\label{ex:d-mfld-local-unstable}
  There are framed basics that do not interact well with projection.  
  Consider for example the following $3$-framed basic in which the
  $1$-strata are arranged to lie within the same $(3, 1)$-plane:
  \[
    \begin{tikzpicture}[scale = 0.3, baseline=(current bounding box.center)]
      \begin{scope}[shift={(0, 0, 7)}]
        \draw[->] (0,0,0) -- (1,0,0) node[anchor=west]{\tiny{3}};
        \draw[->] (0,0,0) -- (0,1,0) node[anchor=south]{\tiny{2}};
        \draw[->] (0,0,0) -- (0,0,1) node[anchor=north east]{\tiny{1}};
      \end{scope}
      
      \meshBoundingBox{6}{6}{6}
      
      \begin{scope}[canvas is zy plane at x = 0]
        \coordinate (s0-0-a) at (2, 3);
        \coordinate (s0-0-b) at (4, 3);
      \end{scope}

      \begin{scope}[canvas is zy plane at x = 3]
        \coordinate (s1-0-center) at (3, 3);
      \end{scope}

      \begin{scope}[canvas is zy plane at x = 6]
        \coordinate (s2-0-a) at (3, 3);
      \end{scope}

      \draw[mesh3-stratum-blue] (s0-0-a) -- (s1-0-center);
      \draw[mesh3-stratum-orange] (s0-0-b) -- (s1-0-center);
      \draw[mesh3-stratum] (s2-0-a) -- (s1-0-center);
      \node[mesh3-vertex] at (s1-0-center) {};
    \end{tikzpicture}
  \]
  Taking the projection to the last two coordinate directions
  produces a diagram in which the orange $1$-stratum completely obscures the blue one:
  \[
    \begin{tikzpicture}[scale = 0.5, baseline=(current bounding box.center)]
      \fill[mesh-background] (0, 0) rectangle (6, 6);
      \draw[mesh-stratum-orange] (3, 0) -- (3, 3);
      \draw[mesh-stratum] (3, 6) -- (3, 3);
      \node[mesh-vertex] at (3, 3) {};
    \end{tikzpicture}
  \]
  In the string diagram calculus for braided monoidal categories, such arrangements
  are prohibited in order to avoid confusion.
  Similarly, we want to avoid framed basics like this from appearing within 
  manifold diagrams.
\end{example}

We restrict manifold diagrams to only contain framed basics whose framed stratified isomorphism type is stable under small perturbations.
This is not the case for Example~\ref{ex:d-mfld-local-unstable} since the two overlapping $1$-strata may be slightly nudged so that they can be distinguished in the projection.

\begin{definition}
  A \defn{perturbation} of an $n$-framed stratified space
  $\strat{M}$ is an $n$-framed stratified bundle
  $f : \strat{E} \to \DeltaTop{1}$ that satisfies the following conditions:
  \begin{enumerate}
    \item $\strat{M}$ is the fibre of $f$ over $0 \in \DeltaTop{1}$.
    \item $f$ is a stratified trivial bundle after forgetting the framing.
    \item $f$ is an $n$-framed stratified submersion.
  \end{enumerate}
  A compact $n$-framed stratified space $\strat{M}$ is \defn{stable}
  when for every perturbation $f : \strat{E} \to \DeltaTop{1}$ of
  $\strat{M}$ there exists an $\eps > 0$ such that $f$ restricts to 
  an $n$-framed stratified trivial bundle over the subinterval
  $\intCO{0, \eps} \subset \DeltaTop{1}$: 
  \[
    \begin{tikzcd}
      {\strat{M} \times \intCO{0, \eps}} \ar[r, hook] \ar[d] \pullbackcorner &
      {\strat{E}} \ar[d] \\
      {\intCO{0, \eps}} \ar[r, hook] &
      {\DeltaTop{1}}
    \end{tikzcd}
  \]
  An $n$-framed stratified space $\strat{M}$ is \defn{stable} if
  every compact subspace $\strat{K} \subseteq \strat{M}$ is stable.
\end{definition}




\subsection{The Global Geometry of Manifold Diagrams}\label{sec:d-mfld-global}

We have characterised the local properties that an $n$-framed stratified
space $\strat{M}$ should satisfy in order to be an $n$-manifold diagram.
We now turn to the global conditions.
An $n$-manifold diagram contains information that is not directly detected by just the stratification, but by the relative arrangement of the strata in the ambient space.
This arrangement determines which coherence structure of an $(\infty, n)$-category $\cat{C}$
should be used when interpreting the diagram in $\cat{C}$.
To keep this information finite, we require an $n$-manifold diagram to be essentially tame
as a $n$-framed stratified space according to Definition~\ref{def:fct-framing-tame}.

\begin{remark}
  Framed basics are refined by atoms and are therefore always essentially tame.
  Therefore, when an $n$-framed stratified space $\strat{M}$ is covered by framed basics,
  every point in $\strat{M}$ has an open neighbourhood that is essentially tame.
  However, this is not sufficient for $\strat{M}$ to be essentially tame itself.
  For instance the $3$-framed stratification containing two $1$-strata
  given by the subsets
  \[
    \{ (-1, 0, t) \mid t \in \R \},
    \{ (1, t\sin(t^{-1}), t) \mid t \in \R \} \subset \R^3
  \]
  is covered by framed basics but is not essentially tame because the two $1$-strata 
  cross over each other an infinite number of times.
  Such a diagram would be impossible to interpret in an arbitrary
  $(\infty, 3)$-category without assuming additional structure.
\end{remark}

String diagrams are often defined as stratifications of a closed $2$-cube $\intCC{0, 1}^2$.
The surface diagrams for Gray categories of~\cite{gray-category-diagrams} similarly
are defined as stratifications of the closed $3$-cube $\intCC{0, 1}^3$.
A natural continuation would be to define $n$-manifold diagrams as stratifications
of the closed $n$-cube $\intCC{0, 1}^n$.
That choice, however, would prevent us from defining a semi-strict composition
operation on $n$-manifold diagrams and therefore nullify one of their major
technical advantages.

Manifold diagrams are composed by juxtaposition or, in other words, by placing diagrams 
next to each other along shared boundaries.
The composition of two $1$-manifold diagrams on $\intCC{0, 1}$ should be a $1$-manifold diagram on the longer interval $\intCC{0, 2}$.
This appears insubstantial at first, since $\intCC{0, 1}$ and $\intCC{0, 2}$ are
isomorphic, but there is no way in which we can rescale the composite diagram to fit into
the unit interval
$\intCC{0, 1}$ and simultaneously retain associativity of the composition.
Similarly, strict unitality of diagram composition can only be achieved when we allow diagrams that have a length of zero in the direction of composition.

We therefore do not pick any particular subspace of $\R^n$ on which to define
$n$-manifold diagrams but allow their underlying space to vary.
In particular, we have a concept of $n$-manifold diagrams supported on any
closed $n$-mesh $\strat{X}$.
Notably this includes the closed subspaces
$
  \intCC{a_1, b_1} \times \cdots \times \intCC{a_n, b_n} \subset \R^n
$
for all choices of $a_i \leq b_i$, enabling semistrict composition.
Other choices of closed $n$-meshes will turn out to be meaningful as well.

For technical convenience we let an $n$-manifold diagram $\strat{M}$
that is supported on a closed $n$-mesh $\strat{X}$ overshoot the boundary
of $\strat{X}$ so that the underlying unstratified space $\unstrat(\strat{M})$ is an open subspace of $\R^n$ that contains
$\unstrat(\strat{X})$.
With that convention every point of an $n$-manifold diagram can have an open
neighbourhood that is a framed basic.
We identify two $n$-manifold diagrams on $\strat{X}$ when they agree on
their intersection with $\strat{X}$.

The globularity condition of a $2$-dimensional string diagram controls the interaction of
the diagram's strata with the boundary. In particular, it ensures that the
vertices of the diagram must lie within the interior and the edges may only
intersect the top and bottom boundaries.
For a $3$-dimensional string diagram, the globularity condition also guarantees
that all braids must occur within the diagram's interior.
For an $n$-manifold diagram $\strat{M}$ on a closed $n$-mesh $\strat{X}$ we ensure globularity
by requiring that $\strat{M}$ is orthogonal to $\strat{X}$ in the sense of
\S\ref{sec:fct-conf-ortho}.


\begin{definition}
  An \defn{$n$-manifold diagram} is an essentially tame $n$-framed stratified space
  $\strat{M}$ with labels in $\ord{n}$ that is covered with stable framed basics.
  When $\strat{X}$ is a closed $n$-mesh, an \defn{$n$-manifold diagram on $\strat{X}$}
  is an $n$-manifold diagram $\strat{M}$ that is orthogonal to $\strat{X}$.
  When $\strat{M}_0$, $\strat{M}_1$ are two $n$-manifold diagrams on $\strat{X}$
  we write $\strat{M}_0 \sim_\strat{X} \strat{M}_1$ whenever
  $
    \lbl{\strat{M}_0}(x) = \lbl{\strat{M}_1}(x)
  $
  for all $x \in \unstrat(\strat{X})$.
  We denote the equivalence classes for this relation by $[-]_{\strat{X}}$.
\end{definition}

\begin{observation}
  For the closed $n$-mesh $\StratIntLR^n$ we can characterise the orthogonality
  condition in a more direct way. 
  Suppose that $\strat{M}$ is a stratified space such that
  $\intCC{-1, 1}^n \subseteq \unstrat(\strat{M}) \subseteq \R^n$.
  Then $\strat{M}$ is orthogonal to $\StratIntLR^n$ exactly if there
  is an $\eps > 0$
  such that for every $x \in \strat{M}$ with $\| x \|_\infty \in \intOO{1 - \eps, 1 + \eps}$
  the stratification of $\strat{M}$ satisfies
  $\stratMap{\strat{M}}(x) = \stratMap{\strat{M}}(\| x \|_\infty^{-1} x)$.
\end{observation}

\begin{example}
  The $2$-manifold diagrams on the $2$-cube $\StratIntLR^2$ are string diagrams:
  \[
    \begin{tikzpicture}[scale = 0.5, baseline=(current bounding box.center)]
      \fill[mesh-background] (0, 0) rectangle (7, 6);
      \node[mesh-vertex-dual] at (1, 1) {};
      \node[mesh-vertex-dual] at (6, 1) {};
      \node[mesh-vertex-dual] at (1, 5) {};
      \node[mesh-vertex-dual] at (6, 5) {};
      \draw[mesh-stratum-dual] (1, 1) rectangle ++(5, 4);
      \draw[mesh-stratum] (1.5, 0) -- (1.5, 1) .. controls +(0, 0.2) and +(-0.5, 0) .. (2, 2);
      \draw[mesh-stratum] (2.5, 0) -- (2.5, 1) .. controls +(0, 0.2) and +(0.5, 0) .. (2, 2);
      \draw[mesh-stratum] (2, 2) -- (2, 3) .. controls +(0, 0.2) and +(-1, 0) .. (3, 4);
      \draw[mesh-stratum] (4, 2) -- (4, 3) .. controls +(0, 0.2) and +(1, 0) .. (3, 4);
      \draw[mesh-stratum] (5, 2) -- (5, 6);
      \node[mesh-vertex] at (4, 2) {};
      \node[mesh-vertex] at (3, 4) {};
      \node[mesh-vertex] at (5, 2) {};
      \node[mesh-vertex] at (2, 2) {};
    \end{tikzpicture}
  \]
\end{example}

\begin{example}
  We can pick a more exotic closed $2$-mesh than the cube:
  \[
    \begin{tikzpicture}[scale = 0.5, baseline=(current bounding box.center)]
      \fill[mesh-background] (0, 0) rectangle (6, 6);

      \draw[mesh-stratum-dual] (1, 1) rectangle (5, 5);
      \node[mesh-vertex-dual] at (1, 1) {};
      \node[mesh-vertex-dual] at (3, 1) {};
      \node[mesh-vertex-dual] at (5, 1) {};
      \node[mesh-vertex-dual] at (1, 5) {};
      \node[mesh-vertex-dual] at (5, 5) {};

      \draw[mesh-stratum] (1.5, 0) -- (1.5, 1) .. controls +(0, 0.2) and +(-0.5, 0) .. (2, 2);
      \draw[mesh-stratum] (2.5, 0) -- (2.5, 1) .. controls +(0, 0.2) and +(0.5, 0) .. (2, 2);
      \draw[mesh-stratum] (4, 0) -- (4, 3) .. controls +(0, 0.2) and +(1, 0) .. (3, 4);
      \draw[mesh-stratum] (2, 2) -- (2, 3) .. controls +(0, 0.2) and +(-1, 0) .. (3, 4);
      \draw[mesh-stratum] (3, 4) .. controls +(-0.5, 0) and +(0, -0.2) .. (2.5, 5) -- (2.5, 6);
      \draw[mesh-stratum] (3, 4) .. controls +(0.5, 0) and +(0, -0.2) .. (3.5, 5) -- (3.5, 6);

      \node[mesh-vertex] at (2, 2) {};
      \node[mesh-vertex] at (3, 4) {};
    \end{tikzpicture}
  \]
  This is a $2$-manifold diagram whose source boundary is a composite of two $1$-manifold diagrams.
  In the notation of~\cite{modular-categories-internal-diagrams}
  this would correspond to the internal string diagram:
  \[
    \begin{tikzpicture}[scale = 0.5, baseline=(current bounding box.center)]
      \fill[mesh-background] (1, 0) .. controls +(0, 2) and +(0, -2) .. (2, 6)
        -- (4, 6) .. controls +(0, -2) and +(0, 2) .. (5, 0)
        -- (3.5, 0) -- (3.5, 2.5) arc (0:180:0.25) -- (3, 0) -- cycle;

      \draw[mesh-stratum-dual] (1, 0) .. controls +(0, 2) and +(0, -2) .. (2, 6);
      \draw[mesh-stratum-dual] (5, 0) .. controls +(0, 2) and +(0, -2) .. (4, 6);
      \draw[mesh-stratum-dual] (3, 0) -- (3, 2.5) arc (180:0:0.25) -- (3.5, 0);

      \draw[mesh-stratum] (1.5, 0) -- (1.5, 1) .. controls +(0, 0.2) and +(-0.5, 0) .. (2, 2);
      \draw[mesh-stratum] (2.5, 0) -- (2.5, 1) .. controls +(0, 0.2) and +(0.5, 0) .. (2, 2);
      \draw[mesh-stratum] (4, 0) -- (4, 3) .. controls +(0, 0.2) and +(1, 0) .. (3, 4);
      \draw[mesh-stratum] (2, 2) -- (2, 3) .. controls +(0, 0.2) and +(-1, 0) .. (3, 4);
      \draw[mesh-stratum] (3, 4) .. controls +(-0.5, 0) and +(0, -0.2) .. (2.5, 5) -- (2.5, 6);
      \draw[mesh-stratum] (3, 4) .. controls +(0.5, 0) and +(0, -0.2) .. (3.5, 5) -- (3.5, 6);

      \node[mesh-vertex] at (2, 2) {};
      \node[mesh-vertex] at (3, 4) {};
    \end{tikzpicture}
  \]
\end{example}

\begin{example}
  When the closed $2$-mesh $\strat{X}$ is not a cell, we can interpret a
  $2$-manifold diagram on $\strat{X}$ as a pasting diagram of the individual $2$-manifold
  diagrams on the cells of $\strat{X}$:
  \[
    \begin{tikzpicture}[scale = 0.5, baseline=(current bounding box.center)]
      \fill[mesh-background] (0, 0) rectangle (7, 6);
      \node[mesh-vertex-dual] at (1, 1) {};
      \node[mesh-vertex-dual] at (3, 1) {};
      \node[mesh-vertex-dual] at (6, 1) {};
      \node[mesh-vertex-dual] at (1, 3) {};
      \node[mesh-vertex-dual] at (3, 3) {};
      \node[mesh-vertex-dual] at (6, 3) {};
      \node[mesh-vertex-dual] at (1, 5) {};
      \node[mesh-vertex-dual] at (6, 5) {};
      \draw[mesh-stratum-dual] (1, 1) rectangle ++(2, 2);
      \draw[mesh-stratum-dual] (3, 1) rectangle ++(3, 2);
      \draw[mesh-stratum-dual] (1, 3) rectangle ++(5, 2);
      \draw[mesh-stratum] (1.5, 0) -- (1.5, 1) .. controls +(0, 0.2) and +(-0.5, 0) .. (2, 2);
      \draw[mesh-stratum] (2.5, 0) -- (2.5, 1) .. controls +(0, 0.2) and +(0.5, 0) .. (2, 2);
      \draw[mesh-stratum] (2, 2) -- (2, 3) .. controls +(0, 0.2) and +(-1, 0) .. (3, 4);
      \draw[mesh-stratum] (4, 2) -- (4, 3) .. controls +(0, 0.2) and +(1, 0) .. (3, 4);
      \draw[mesh-stratum] (5, 4) -- (5, 6);
      \node[mesh-vertex] at (4, 2) {};
      \node[mesh-vertex] at (3, 4) {};
      \node[mesh-vertex] at (5, 4) {};
      \node[mesh-vertex] at (5, 4) {};
      \node[mesh-vertex] at (2, 2) {};
    \end{tikzpicture}
  \]
\end{example}

\subsection{Isotopies of Manifold Diagrams}\label{sec:d-mfld-isotopy}

The interpretation of a string diagram in a monoidal category is invariant under isotopy: however we might deform a string diagram, it represents the same morphism.
This is the content of the coherence theorems of~\cite{geometry-tensor-calculus} and is implicitly relied on wherever string diagrams are used.
Isotopy invariance becomes algebraically significant for string diagrams in $\R^3$ and $\R^4$ for braided and symmetric monoidal categories,
in which the braiding operation together with the coherence axioms are encoded into geometric arrangements.
In higher dimensions, the isotopies of $n$-manifold diagrams should capture all the coherence information that is part of having a weak interchange law.

Isotopies of braided and symmetric monoidal categories are more restricted than general isotopies of the diagrams interpreted as stratified spaces.
The order of edges entering a vertex matters,
but a stratified isotopy can permute these edges freely.
If this was valid in any symmetric monoidal category, every monoid would have to be automatically commutative.
For braided and symmetric monoidal categories the solution is relatively straightforward:
We require an isotopy to preserve the order of edges as seen in a projection.

\begin{example}
  Trimble's notes on surface diagrams~\cite{trimble-surface} described an example of a deformation of surface diagrams that should be excluded from the valid deformations.
  Taking artistic liberties in the presentation, this illegal deformation looks as follows:
  \[
    \begin{tikzpicture}[scale = 0.5, baseline=(current bounding box.center)]
      \fill[mesh-background] (0, 0) rectangle (6, 6);
      \draw[mesh-stratum]
        (2, 0)
        -- (2, 1)
        .. controls +(0, 1) and +(0, -1) .. (4, 3)
        .. controls +(0, 0.2) and +(1, 0) .. (3, 4);
      \draw[mesh-stratum-over]
        (4, 0)
        -- (4, 1)
        .. controls +(0, 1) and +(0, -1) .. (2, 3)
        .. controls +(0, 0.2) and +(-1, 0) .. (3, 4);
      \draw[mesh-stratum]
        (4, 0)
        -- (4, 1)
        .. controls +(0, 1) and +(0, -1) .. (2, 3)
        .. controls +(0, 0.2) and +(-1, 0) .. (3, 4);
      \node[mesh-vertex] at (3, 4) {};
    \end{tikzpicture}
    \quad
    \sim
    \quad
    \begin{tikzpicture}[scale = 0.5, baseline=(current bounding box.center)]
      \fill[mesh-background] (0, 0) rectangle (6, 6);
      \draw[mesh-stratum]
        (2, 0)
        -- (2, 2)
        .. controls +(0, 1) and +(0, -1) .. (4, 3.5)
        .. controls +(0, 0.2) and +(1, 0) .. (3, 4);
      \draw[mesh-stratum-over]
        (4, 0)
        -- (4, 2)
        .. controls +(0, 1) and +(0, -1) .. (2, 3.5)
        .. controls +(0, 0.2) and +(-1, 0) .. (3, 4);
      \draw[mesh-stratum]
        (4, 0)
        -- (4, 2)
        .. controls +(0, 1) and +(0, -1) .. (2, 3.5)
        .. controls +(0, 0.2) and +(-1, 0) .. (3, 4);
      \node[mesh-vertex] at (3, 4) {};
    \end{tikzpicture}
    \quad
    \sim
    \quad
    \begin{tikzpicture}[scale = 0.5, baseline=(current bounding box.center)]
      \fill[mesh-background] (0, 0) rectangle (6, 6);
      \draw[mesh-stratum]
        (2, 0)
        -- (2, 3)
        .. controls +(0, 0.2) and +(-1, 0) .. (3, 4);
      \draw[mesh-stratum]
        (4, 0)
        -- (4, 3)
        .. controls +(0, 0.2) and +(1, 0) .. (3, 4);
      \node[mesh-vertex] at (3, 4) {};
    \end{tikzpicture}
  \]
  Starting with two braided $1$-strata that end in a $0$-stratum,
  we continuously push the height at which the braiding occurs upwards.
  In the limit we reach the diagram in which the two $1$-strata are not braided
  any longer.
  The notes do not mention how to avoid this situation in general, but do hint at a solution involving coincidences where two strata agree in some coordinates.
\end{example}

Manifold diagrams can have strata of any dimension and arrangements that are more complicated than just their order. 
Requiring isotopies of manifold diagrams to preserve the framing would prevent this problem but would also be too strict, as it would prohibit even braids.
Instead, we demand that an isotopy preserves the framing locally by being
a framed stratified submersion.

\begin{definition}
  An \defn{$n$-manifold diagram bundle} is an essentially tame $n$-framed stratified submersion
  $p : \strat{E} \to B$ with labels in $\ord{n}$ where $B$ is an unstratified space and every fibre of $p$ is an $n$-manifold diagram.
  An \defn{isotopy of $n$-manifold diagrams} is an $n$-manifold diagram bundle
  over the interval $\DeltaTop{1}$. 
\end{definition}

By combining the requirements that each fibre must be an $n$-manifold diagram
and that the bundle is a framed stratified submersion, we have the following
more direct characterisation of $n$-manifold diagram bundles after unpacking the definitions.

\begin{observation}
  Suppose that $p : \strat{E} \to B$ is an essentially tame $n$-framed stratified bundle
  with labels in $\ord{n}$
  such that $B$ is an unstratified space.
  Then $p$ is an $n$-manifold diagram bundle if and only if for every point
  $x \in \strat{E}$ there exists a framed basic $\strat{U}$, an open
  neighbourhood $V \subseteq B$ of $f(x)$ and an open embedding of
  labelled $n$-framed stratified spaces
  \[
    \begin{tikzcd}
      {\strat{U} \times \strat{V}} \ar[r, hook, "I"] \ar[d] &
      {\strat{E}} \ar[d, "p"] \\
      {V} \ar[r, hook] &
      {B}
    \end{tikzcd}
  \]
  such that $I(\bot_{\strat{U}}, f(x)) = x$.
\end{observation}




\begin{definition}
  Let $\strat{X}$ be a closed $n$-mesh.
  An \defn{$n$-manifold diagram bundle on $\strat{X}$} is an $n$-manifold
  diagram bundle $p : \strat{E} \to B$ that is fibrewise orthogonal to $\strat{X}$.  
  When $p_0: \strat{E}_0 \to B$ and $p_1 : \strat{E}_1 \to B$ are two
  $n$-manifold diagram bundles on $\strat{X}$ we write $p_0 \sim_\strat{X} p_1$ whenever
  $
    \lbl{\strat{E}_0}(x, b) = \lbl{\strat{E}_1}(x, b)
  $
  for all $x \in \unstrat(\strat{X})$ and $b \in B$.  
  We write $[-]_{\strat{X}}$ for the equivalence classes.
\end{definition}

\begin{definition}\label{def:d-mfld-isotopy-space}
  Let $\strat{X}$ be a closed $n$-mesh.
  The \defn{space of $n$-manifold diagrams on $\strat{X}$}
  is the simplicial set $\MfldDiag^n(\strat{X})$ whose $k$-simplices
  consist of an equivalence class $[p]_{\strat{X}}$ of an $n$-manifold
  diagram bundle $p : \strat{E} \to \DeltaTop{k}$ on $\strat{X}$.
\end{definition}

\begin{lemma}
  Let $\strat{X}$ be a closed $n$-mesh.
  Then $\MfldDiag^n(\strat{X})$ is a Kan complex.
\end{lemma}
\begin{proof}
  Suppose we have a lifting problem
  \[
    \begin{tikzcd}
      {\Lambda^i\ord{k}} \ar[r] \ar[d, hook] &
      {\MfldDiag^n} \\
      {\Delta\ord{k}} \ar[ur, dashed] &
      {}
    \end{tikzcd}
  \]
  By unpacking the definitions and picking a representative,
  we have an $\eps > 0$ and an $n$-manifold diagram
  bundle $f_0 : \strat{E}_0 \to \HornTop{i}{k}$ that is fibrewise $\eps$-orthogonal to
  $\strat{X}$ and defined on
  \[
    \unstrat(\strat{E}_0) = \{ (e, b) \mid \exists x \in \strat{X}.\ \| e - x \|_\infty < \eps \} \subset \R^n \times \HornTop{i}{k}.
  \]
  We then pick any piecewise linear retraction $R : \DeltaTop{k} \to \HornTop{i}{k}$
  and let $f_1 : \strat{E}_1 \to \DeltaTop{k}$ be the pullback of $f_0$ along
  $R$.
  Then $f_1$ represents a solution to the lifting problem.
\end{proof}

\begin{example}
  Paths in the space $\MfldDiag^2(\StratIntLR^2)$ are isotopies of string diagrams:
  \[
    \begin{tikzpicture}[scale = 0.5, baseline=(current bounding box.center)]
      \fill[mesh-background] (0, 0) rectangle (7, 6);
      \node[mesh-vertex-dual] at (1, 1) {};
      \node[mesh-vertex-dual] at (6, 1) {};
      \node[mesh-vertex-dual] at (1, 5) {};
      \node[mesh-vertex-dual] at (6, 5) {};
      \draw[mesh-stratum-dual] (1, 1) rectangle ++(5, 4);
      \draw[mesh-stratum] (1.5, 0) -- (1.5, 1) .. controls +(0, 0.2) and +(-0.5, 0) .. (2, 2);
      \draw[mesh-stratum] (2.5, 0) -- (2.5, 1) .. controls +(0, 0.2) and +(0.5, 0) .. (2, 2);
      \draw[mesh-stratum] (2, 2) -- (2, 3) .. controls +(0, 0.2) and +(-1, 0) .. (3, 4);
      \draw[mesh-stratum] (4, 2) -- (4, 3) .. controls +(0, 0.2) and +(1, 0) .. (3, 4);
      \draw[mesh-stratum] (5, 2) -- (5, 6);
      \node[mesh-vertex] at (4, 2) {};
      \node[mesh-vertex] at (3, 4) {};
      \node[mesh-vertex] at (5, 2) {};
      \node[mesh-vertex] at (2, 2) {};
    \end{tikzpicture}
    \quad
    \sim
    \quad
    \begin{tikzpicture}[scale = 0.5, baseline=(current bounding box.center)]
      \fill[mesh-background] (0, 0) rectangle (7, 6);
      \node[mesh-vertex-dual] at (1, 1) {};
      \node[mesh-vertex-dual] at (6, 1) {};
      \node[mesh-vertex-dual] at (1, 5) {};
      \node[mesh-vertex-dual] at (6, 5) {};
      \draw[mesh-stratum-dual] (1, 1) rectangle ++(5, 4);
      \draw[mesh-stratum] (1.5, 0) -- (1.5, 1) .. controls +(0, 0.2) and +(-0.5, 0) .. (2, 2);
      \draw[mesh-stratum] (2.5, 0) -- (2.5, 1) .. controls +(0, 0.2) and +(0.5, 0) .. (2, 2);
      \draw[mesh-stratum] (2, 2) -- (2, 3) .. controls +(0, 0.2) and +(-1, 0) .. (3, 4);
      \draw[mesh-stratum] (4, 2) -- (4, 3) .. controls +(0, 0.2) and +(1, 0) .. (3, 4);
      \draw[mesh-stratum] (5, 4) -- (5, 6);
      \node[mesh-vertex] at (4, 2) {};
      \node[mesh-vertex] at (3, 4) {};
      \node[mesh-vertex] at (5, 4) {};
      \node[mesh-vertex] at (2, 2) {};
    \end{tikzpicture}
    \quad
    \sim
    \quad
    \begin{tikzpicture}[scale = 0.5, baseline=(current bounding box.center)]
      \fill[mesh-background] (0, 0) rectangle (7, 6);
      \node[mesh-vertex-dual] at (1, 1) {};
      \node[mesh-vertex-dual] at (6, 1) {};
      \node[mesh-vertex-dual] at (1, 5) {};
      \node[mesh-vertex-dual] at (6, 5) {};
      \draw[mesh-stratum-dual] (1, 1) rectangle ++(5, 4);
      \draw[mesh-stratum] (2.5, 0) -- (2.5, 1) .. controls +(0, 0.2) and +(-0.5, 0) .. (3, 2);
      \draw[mesh-stratum] (3.5, 0) -- (3.5, 1) .. controls +(0, 0.2) and +(0.5, 0) .. (3, 2);
      \draw[mesh-stratum] (3, 2) -- (3, 3) .. controls +(0, 0.2) and +(-1, 0) .. (4, 4);
      \draw[mesh-stratum] (5, 2) -- (5, 3) .. controls +(0, 0.2) and +(1, 0) .. (4, 4);
      \draw[mesh-stratum] (2, 2) -- (2, 6);
      \node[mesh-vertex] at (5, 2) {};
      \node[mesh-vertex] at (4, 4) {};
      \node[mesh-vertex] at (3, 2) {};
      \node[mesh-vertex] at (2, 2) {};
    \end{tikzpicture}
  \]
  For any $m \geq 0$ the space $\MfldDiag^2(\StratIntLR^2)$ contains as
  a connected component the unordered configuration space of $m$ points in the plane.  
\end{example}

\begin{definition}
  Let $\strat{X}$ be a closed $n$-mesh and $\eps > 0$.
  We let $\MfldDiag^{n, \eps}(\strat{X})$ be the simplicial set
  whose $k$-simplices consist of an $n$-manifold diagram bundle
  $p : \strat{E} \to \DeltaTop{k}$
  that is fibrewise $\eps$-orthogonal to $\strat{X}$ and defined on the open subset
  \[
    \unstrat(\strat{E}) = \{ (e, b) \mid \exists x \in \strat{X}.\ \| e - x \|_\infty < \eps \} \subset \R^n \times \DeltaTop{k}.
  \]
  Whenever $\eps_0 > \eps_1 > 0$ we have an induced monomorphism of simplicial sets
  \[
    \begin{tikzcd}
      \MfldDiag^{n, \eps_0}(\strat{X}) \ar[r, hook] &
      \MfldDiag^{n, \eps_1}(\strat{X}).
    \end{tikzcd}
  \]
\end{definition}

\begin{observation}\label{obs:d-mfld-space-eps}  
  For every closed $n$-mesh $\strat{X}$ the simplicial set
  $\MfldDiag^n(\strat{X})$
  is the filtered colimit of
  $\MfldDiag^{n, \eps}(\strat{X})$
  for all $\eps > 0$.
  In particular for every $\eps > 0$ there is a map
  $\MfldDiag^{n, \eps}(\strat{X}) \to \MfldDiag^n(\strat{X})$
  which sends an $n$-manifold diagram bundle $p$ to its equivalence class
  $[p]_{\strat{X}}$.
  Moreover, for any sufficiently small $\eps > 0$ the map
  $\MfldDiag^{n, \eps}(\strat{X}) \to \MfldDiag^n(\strat{X})$
  is the inclusion of a deformation retract.
\end{observation}

\subsection{Functoriality in the Closed Mesh}\label{sec:d-mfld-functorial}

For any equivalence of closed $n$-meshes $\strat{X} \simeq \strat{Y}$ 
we obtain an equivalence between the spaces of $n$-manifold diagrams
$\MfldDiag^n(\strat{X}) \simeq \MfldDiag^n(\strat{Y})$.
More generally, the assignment of $\MfldDiag^n(\strat{X})$ to any closed $n$-mesh $\strat{X}$
extends to a functor $\MfldDiag^n(-) : \BMeshClosed{n} \to \Space$.
We construct this functor as a left fibration of simplicial sets.

\begin{definition}
  Suppose that $\xi : \strat{X} \to \strat{B}$ is a closed $n$-mesh bundle.
  Let $p_0 : \strat{E}_0 \to \unstrat(\strat{B})$ and
  $p_1 : \strat{E}_1 \to \unstrat(\strat{B})$ are $n$-manifold diagram bundles
  that are orthogonal to $\xi$.
  We write $p_0 \sim_\xi p_1$ whenever  
  $
    \lbl{\strat{E}_0}(x, b) = \lbl{\strat{E}_1}(x, b)
  $
  for all $(x, b) \in \unstrat(\strat{X})$.
  We denote the equivalence classes by $[-]_\xi$.
\end{definition}

\begin{definition}
  We let $\MfldDiag^n$ denote the simplicial set whose $k$-simplices consist of
  a pair $(\xi, [p]_\xi)$ where
  $\xi : \strat{X} \to \DeltaStrat{k}$ is a closed $n$-mesh bundle and  
  $p : \strat{E} \to \DeltaTop{k}$ an $n$-manifold diagram bundle orthogonal to $\xi$.
  There is a canonical projection map $\MfldDiag^n \to \BMeshClosed{n}$ which sends
  the pair $(\xi, [p]_\xi)$ to the closed $n$-mesh bundle $\xi$.
\end{definition}

\begin{proposition}\label{prop:d-mfld-functorial-left}
  The projection map $\MfldDiag^n \to \BMeshClosed{n}$ is a left fibration.
\end{proposition}
\begin{proof}
  Suppose we have a lifting problem
  \[
    \begin{tikzcd}
      {\Lambda^i\ord{k}} \ar[r] \ar[d] &
      {\MfldDiag^n} \ar[d] \\
      {\Delta\ord{k}} \ar[r] \ar[ur, dashed] &
      {\BMeshClosed{n}}
    \end{tikzcd}
  \]
  for some $0 \leq i < k$.
  After unpacking the definitions,
  we have an $n$-manifold diagram bundle $f' : \strat{M}' \to \HornTop{i}{k}$
  and a closed $n$-mesh bundle $\xi : \strat{X} \to \DeltaStrat{k}$ such that
  $f'$ is orthogonal to the restriction of $\xi$ over $\HornStrat{i}{k}$.
  We pick a retraction $R : \DeltaTop{k} \to \HornTop{i}{k}$ of the inclusion
  $\HornTop{i}{k} \hookrightarrow \DeltaTop{k}$ and let $f''$ be the pullback
  of $f'$ along $R$:
  \[
    \begin{tikzcd}
      {\strat{M}''} \ar[r] \ar[d, "f''"'] \pullbackcorner &
      {\strat{M}'} \ar[d, "f'"] \\
      {\DeltaTop{k}} \ar[r] &
      {\HornTop{i}{k}}
    \end{tikzcd}
  \]
  The $n$-manifold diagram bundle $f''$ will in general not be orthogonal to $\xi$.
  However, there exists an open neighbourhood $\strat{U}$ of $\HornStrat{i}{k}$
  in $\DeltaStrat{k}$ such that $f''$ and $\xi$ are orthogonal over $\strat{U}$.
  We can then use the trivialisation of $\xi$ over the $k$th stratum of
  $\DeltaStrat{k}$ to modify $f''$ into an $n$-manifold diagram bundle
  $f : \strat{M} \to \DeltaTop{k}$ such that $f \perp \xi$.
  Then $(\xi, [f]_\xi)$ is a solution to the lifting problem.
\end{proof}

\begin{observation}
  As a left fibration, the map
  $\MfldDiag^n \to \BMeshClosed{n}$
  induces a functor
  \[ \MfldDiag^n(-) : \BMeshClosed{n} \to \Space \]
  and therefore a space $\MfldDiag^n(\strat{X})$ for any closed $n$-mesh $\strat{X}$.
  The space $\MfldDiag^n(\strat{X})$ can be computed explicitly as a Kan complex by taking the
  pullback of simplicial sets  
  \[
    \begin{tikzcd}
      {\MfldDiag^n(\strat{X})} \ar[r] \ar[d] \pullbackcorner &
      {\MfldDiag^n} \ar[d] \\
      {\Delta\ord{0}} \ar[r, "\strat{X}"'] &
      {\BMeshClosed{n}}
    \end{tikzcd}
  \]
  We therefore see that the space $\MfldDiag^n(\strat{X})$ is modelled by the Kan
  complex of $n$-manifold diagrams supported on $\strat{X}$ which we defined earlier in
  Definition~\ref{def:d-mfld-isotopy-space}.
\end{observation}

\begin{example}
  Suppose that we have an embedding of closed $2$-meshes
  \[
    \strat{X}
    \quad
    =
    \quad
    \begin{tikzpicture}[scale = 0.5, baseline=(current bounding box.center)]
      \fill[mesh-background] (1, 1) rectangle ++(2, 2);
      \node[mesh-vertex-dual] at (1, 1) {};
      \node[mesh-vertex-dual] at (3, 1) {};
      \node[mesh-vertex-dual] at (1, 3) {};
      \node[mesh-vertex-dual] at (3, 3) {};
      \draw[mesh-stratum-dual] (1, 1) rectangle ++(2, 2);
    \end{tikzpicture}
    \quad
    \overset{e}{\hookrightarrow}
    \quad
    \begin{tikzpicture}[scale = 0.5, baseline=(current bounding box.center)]
      \fill[mesh-background] (1, 1) rectangle ++(5, 4);
      \node[mesh-vertex-dual] at (1, 1) {};
      \node[mesh-vertex-dual] at (3, 1) {};
      \node[mesh-vertex-dual] at (6, 1) {};
      \node[mesh-vertex-dual] at (1, 3) {};
      \node[mesh-vertex-dual] at (3, 3) {};
      \node[mesh-vertex-dual] at (6, 3) {};
      \node[mesh-vertex-dual] at (1, 5) {};
      \node[mesh-vertex-dual] at (6, 5) {};
      \draw[mesh-stratum-dual] (1, 1) rectangle ++(2, 2);
      \draw[mesh-stratum-dual] (3, 1) rectangle ++(3, 2);
      \draw[mesh-stratum-dual] (1, 3) rectangle ++(5, 2);
    \end{tikzpicture}
    \quad
    =
    \quad
    \strat{Y}
  \]
  The image of $e$ is the lower left corner in $\strat{Y}$.
  The functor $\MfldDiag^2(\strat{Y}) \to \MfldDiag^2(\strat{X})$,
  induced by the inert map $\strat{Y} \pto \strat{X}$ associated to $e$,
  restricts a string diagram on $\strat{Y}$ to its subdiagram on $\strat{X}$.
  On an example string diagram this looks like:
  \[
    \begin{tikzpicture}[scale = 0.5, baseline=(current bounding box.center)]
      \fill[mesh-background] (0, 0) rectangle (7, 6);
      \node[mesh-vertex-dual] at (1, 1) {};
      \node[mesh-vertex-dual] at (3, 1) {};
      \node[mesh-vertex-dual] at (6, 1) {};
      \node[mesh-vertex-dual] at (1, 3) {};
      \node[mesh-vertex-dual] at (3, 3) {};
      \node[mesh-vertex-dual] at (6, 3) {};
      \node[mesh-vertex-dual] at (1, 5) {};
      \node[mesh-vertex-dual] at (6, 5) {};
      \draw[mesh-stratum-dual] (1, 1) rectangle ++(2, 2);
      \draw[mesh-stratum-dual] (3, 1) rectangle ++(3, 2);
      \draw[mesh-stratum-dual] (1, 3) rectangle ++(5, 2);
      \draw[mesh-stratum] (1.5, 0) -- (1.5, 1) .. controls +(0, 0.2) and +(-0.5, 0) .. (2, 2);
      \draw[mesh-stratum] (2.5, 0) -- (2.5, 1) .. controls +(0, 0.2) and +(0.5, 0) .. (2, 2);
      \draw[mesh-stratum] (2, 2) -- (2, 3) .. controls +(0, 0.2) and +(-1, 0) .. (3, 4);
      \draw[mesh-stratum] (4, 2) -- (4, 3) .. controls +(0, 0.2) and +(1, 0) .. (3, 4);
      \draw[mesh-stratum] (5, 4) -- (5, 6);
      \node[mesh-vertex] at (4, 2) {};
      \node[mesh-vertex] at (3, 4) {};
      \node[mesh-vertex] at (5, 4) {};
      \node[mesh-vertex] at (5, 4) {};
      \node[mesh-vertex] at (2, 2) {};
    \end{tikzpicture}
    \quad
    \mapsto
    \quad
    \begin{tikzpicture}[scale = 0.5, baseline=(current bounding box.center)]
      \fill[mesh-background] (0, 0) rectangle (4, 4);
      \node[mesh-vertex-dual] at (1, 1) {};
      \node[mesh-vertex-dual] at (3, 1) {};
      \node[mesh-vertex-dual] at (1, 3) {};
      \node[mesh-vertex-dual] at (3, 3) {};
      \draw[mesh-stratum-dual] (1, 1) rectangle ++(2, 2);
      \draw[mesh-stratum] (1.5, 0) -- (1.5, 1) .. controls +(0, 0.2) and +(-0.5, 0) .. (2, 2);
      \draw[mesh-stratum] (2.5, 0) -- (2.5, 1) .. controls +(0, 0.2) and +(0.5, 0) .. (2, 2);
      \draw[mesh-stratum] (2, 2) -- (2, 4);
      \node[mesh-vertex] at (2, 2) {};
    \end{tikzpicture}
  \]
\end{example}

\begin{example}
  Suppose we have a bordism of closed $2$-meshes $\strat{Y} \pto \strat{X}$ in $\BMeshClosed{2}$ given by
  \[
    \begin{tikzpicture}[scale = 0.5, baseline=(current bounding box.center)]
      \fill[mesh-background] (1, 1) rectangle ++(5, 4);
      \node[mesh-vertex-dual] at (1, 1) {};
      \node[mesh-vertex-dual] at (3, 1) {};
      \node[mesh-vertex-dual] at (6, 1) {};
      \node[mesh-vertex-dual] at (1, 3) {};
      \node[mesh-vertex-dual] at (3, 3) {};
      \node[mesh-vertex-dual] at (6, 3) {};
      \node[mesh-vertex-dual] at (1, 5) {};
      \node[mesh-vertex-dual] at (6, 5) {};
      \draw[mesh-stratum-dual] (1, 1) rectangle ++(2, 2);
      \draw[mesh-stratum-dual] (3, 1) rectangle ++(3, 2);
      \draw[mesh-stratum-dual] (1, 3) rectangle ++(5, 2);
    \end{tikzpicture}
    \quad
    \pto
    \quad
    \begin{tikzpicture}[scale = 0.5, baseline=(current bounding box.center)]
      \fill[mesh-background] (1, 1) rectangle ++(5, 4);
      \node[mesh-vertex-dual] at (1, 1) {};
      \node[mesh-vertex-dual] at (6, 1) {};
      \node[mesh-vertex-dual] at (1, 5) {};
      \node[mesh-vertex-dual] at (6, 5) {};
      \draw[mesh-stratum-dual] (1, 1) rectangle ++(5, 4);
    \end{tikzpicture}
  \]
  Then we have an induced map of the spaces
  $\MfldDiag^2(\strat{Y}) \to \MfldDiag^2(\strat{X})$
  which on an example string diagram acts as follows.
  \[
    \begin{tikzpicture}[scale = 0.5, baseline=(current bounding box.center)]
      \fill[mesh-background] (0, 0) rectangle (7, 6);
      \node[mesh-vertex-dual] at (1, 1) {};
      \node[mesh-vertex-dual] at (3, 1) {};
      \node[mesh-vertex-dual] at (6, 1) {};
      \node[mesh-vertex-dual] at (1, 3) {};
      \node[mesh-vertex-dual] at (3, 3) {};
      \node[mesh-vertex-dual] at (6, 3) {};
      \node[mesh-vertex-dual] at (1, 5) {};
      \node[mesh-vertex-dual] at (6, 5) {};
      \draw[mesh-stratum-dual] (1, 1) rectangle ++(2, 2);
      \draw[mesh-stratum-dual] (3, 1) rectangle ++(3, 2);
      \draw[mesh-stratum-dual] (1, 3) rectangle ++(5, 2);
      \draw[mesh-stratum] (1.5, 0) -- (1.5, 1) .. controls +(0, 0.2) and +(-0.5, 0) .. (2, 2);
      \draw[mesh-stratum] (2.5, 0) -- (2.5, 1) .. controls +(0, 0.2) and +(0.5, 0) .. (2, 2);
      \draw[mesh-stratum] (2, 2) -- (2, 3) .. controls +(0, 0.2) and +(-1, 0) .. (3, 4);
      \draw[mesh-stratum] (4, 2) -- (4, 3) .. controls +(0, 0.2) and +(1, 0) .. (3, 4);
      \draw[mesh-stratum] (5, 4) -- (5, 6);
      \node[mesh-vertex] at (4, 2) {};
      \node[mesh-vertex] at (3, 4) {};
      \node[mesh-vertex] at (5, 4) {};
      \node[mesh-vertex] at (5, 4) {};
      \node[mesh-vertex] at (2, 2) {};
    \end{tikzpicture}
    \quad
    \mapsto
    \quad
    \begin{tikzpicture}[scale = 0.5, baseline=(current bounding box.center)]
      \fill[mesh-background] (0, 0) rectangle (7, 6);
      \node[mesh-vertex-dual] at (1, 1) {};
      \node[mesh-vertex-dual] at (6, 1) {};
      \node[mesh-vertex-dual] at (1, 5) {};
      \node[mesh-vertex-dual] at (6, 5) {};
      \draw[mesh-stratum-dual] (1, 1) rectangle ++(5, 4);
      \draw[mesh-stratum] (1.5, 0) -- (1.5, 1) .. controls +(0, 0.2) and +(-0.5, 0) .. (2, 2);
      \draw[mesh-stratum] (2.5, 0) -- (2.5, 1) .. controls +(0, 0.2) and +(0.5, 0) .. (2, 2);
      \draw[mesh-stratum] (2, 2) -- (2, 3) .. controls +(0, 0.2) and +(-1, 0) .. (3, 4);
      \draw[mesh-stratum] (4, 2) -- (4, 3) .. controls +(0, 0.2) and +(1, 0) .. (3, 4);
      \draw[mesh-stratum] (5, 4) -- (5, 6);
      \node[mesh-vertex] at (4, 2) {};
      \node[mesh-vertex] at (3, 4) {};
      \node[mesh-vertex] at (5, 4) {};
      \node[mesh-vertex] at (5, 4) {};
      \node[mesh-vertex] at (2, 2) {};
    \end{tikzpicture}
  \]
  We see this as composition: On the left we have a pasting scheme of the
  parts of the string diagram for each subcell of $\strat{Y}$.  
  On the right have removed the interior divisions of the closed $2$-mesh,
  leaving a single cell $\strat{X}$ which contains the composite string diagram.
\end{example}

\begin{example}
  By partially composing two diagrams containing identities, we can produce
  a formal composition operation:
  \[
    \begin{tikzpicture}[scale = 0.5, baseline=(current bounding box.center)]
      \fill[mesh-background] (0, 0) rectangle (6, 6);
      \draw[mesh-stratum-dual] (1, 1) rectangle +(2, 4);
      \draw[mesh-stratum-dual] (3, 1) rectangle +(2, 4);
      \node[mesh-vertex-dual] at (1, 1) {};
      \node[mesh-vertex-dual] at (1, 5) {};
      \node[mesh-vertex-dual] at (5, 1) {};
      \node[mesh-vertex-dual] at (5, 5) {};
      \node[mesh-vertex-dual] at (3, 5) {};
      \node[mesh-vertex-dual] at (3, 1) {};
      \draw[mesh-stratum] (2, 0) -- +(0, 6);
      \draw[mesh-stratum] (4, 0) -- +(0, 6);
    \end{tikzpicture}
    \quad
    \mapsto
    \quad
    \begin{tikzpicture}[scale = 0.5, baseline=(current bounding box.center)]
      \fill[mesh-background] (0, 0) rectangle (6, 6);
      \draw[mesh-stratum-dual] (1, 1) rectangle +(4, 4);
      \node[mesh-vertex-dual] at (1, 1) {};
      \node[mesh-vertex-dual] at (1, 5) {};
      \node[mesh-vertex-dual] at (5, 1) {};
      \node[mesh-vertex-dual] at (5, 5) {};
      \node[mesh-vertex-dual] at (3, 1) {};
      \draw[mesh-stratum] (2, 0) -- +(0, 6);
      \draw[mesh-stratum] (4, 0) -- +(0, 6);
    \end{tikzpicture}
  \]
\end{example}


\subsection{The Diagrammatic Space of Manifold Diagrams}\label{sec:d-mfld-diag-space}

So far we have constructed a functor $\MfldDiag^n(-) : \BMeshClosed{n} \to \Space$
which sends a closed $n$-mesh $\strat{X}$ to the space $\MfldDiag^n(\strat{X})$ of $n$-manifold diagrams on $\strat{X}$.
Via functoriality in $\strat{X}$ we have seen how to restrict and compose manifold diagrams.
In this section we prove the following result.

\begin{thm}\label{thm:d-mfld-diag-space}
  The functor $\MfldDiag^n(-) : \BMeshClosed{n} \to \Space$ is an $n$-diagrammatic space.
\end{thm}
\begin{proof}
  The functor satisfies the Segal condition by
  Lemma~\ref{lem:d-mfld-diag-space-segal}
  and the regularity condition by Lemma~\ref{lem:d-mfld-diag-space-regular},
  which we state and prove below.
\end{proof}

\begin{cor}
  The functor $\MfldDiag^n(-)$ restricts to an $n$-uple Segal space
  \[
    \begin{tikzcd}[column sep = large]
      {\FinOrd^{n, \op}} \ar[r, hook, "\gridMeshClosed^n"] &
      {\BMeshClosed{n}} \ar[r, "\MfldDiag^n(-)"] &
      \Space.
    \end{tikzcd}
  \]
\end{cor}

The Segal condition of an $n$-diagrammatic space requires us to demonstrate that for any closed $n$-mesh $\strat{X}$, the space
$\MfldDiag^n(\strat{X})$ of $n$-manifold diagrams on $\strat{X}$ is the homotopy limit of the spaces of $n$-manifold
diagrams on every cell of $\strat{X}$ individually.
This guarantees that by gluing together manifold diagrams along common boundaries,
we have a composition operation that interacts well with isotopies of diagrams.
For any $n$-cell $\strat{X}$ we write $\MfldDiag^n(\partial \strat{X})$ for
the limit of the simplicial sets $\MfldDiag^n(\strat{C})$ for each cell
$\strat{C} \hookrightarrow \strat{X}$, except for $\strat{X}$
itself.
To help with calculating the homotopy limit as an ordinary limit of simplicial sets, we show then show that the restriction to the boundary
$\MfldDiag^n(\strat{X}) \to \MfldDiag^n(\partial \strat{X})$
is a Kan fibration.
We gradually build up to this result, starting with the $n$-cube $\strat{X} = \StratIntLR^n$ for which we can
write down the stratifications of lifts directly.
The regularity condition for $\MfldDiag^n(-)$ follows from intermediate results.

\begin{lemma}\label{lem:d-mfld-diag-boundary-kan-cube-regular}
  The map $\MfldDiag^n(\StratIntLR^n) \to \MfldDiag^n(\partial\StratIntLR^n)$ is a Kan fibration.
\end{lemma}
\begin{proof}
  Suppose that we have a lifting problem
  \[
    \begin{tikzcd}
      {\Lambda^i \ord{k}} \ar[r] \ar[d, hook] &
      {\MfldDiag^n(\StratIntLR^n)} \ar[d] \\
      {\Delta \ord{k}} \ar[r] \ar[ur, dashed] &
      {\MfldDiag^n(\partial\StratIntLR^n)}
    \end{tikzcd}
  \]
  Unpacking the given data and picking representatives,
  we have an $\eps > 0$ and an $n$-manifold diagram bundle
  $f_0 : \strat{E}_0 \to \DeltaTop{k}$ that is fibrewise orthogonal to 
  $\StratIntLR^n$ and defined on
  \begin{align*}
    \unstrat(\strat{E}_0) =&\
    \{ (e, b) \mid 1 - \eps < \| e \|_\infty < 1 + \eps \} \\
    \cup&\ 
    \{ (e, b) \mid \| e \|_\infty < 1 + \eps \land b \in \HornTop{i}{k} \}.
  \end{align*}
  
  We define $\tau : \R^n \to \intCC{0, 1}$ be the piecewise linear function:
  \begin{alignat*}{2}
    \tau(e) &:= 0 &&\qquad\text{(when $\| e \|_\infty < 1 - 2\eps$)} \\
    \tau(e) &:= \tfrac{1}{\eps}(\| e \|_\infty - (1 - 2\eps)) &&\qquad\text{(when $1 - 2\eps \leq \| e \|_\infty \leq 1 - \eps$)} \\
    \tau(e) &:= 1 &&\qquad\text{(when $1 - \eps < \| e \|_\infty$)}
  \end{alignat*}

  Let $R : \DeltaTop{k} \to \HornTop{i}{k}$ be a piecewise linear retraction of the inclusion $\HornTop{i}{k} \hookrightarrow \DeltaTop{k}$.
  We then define a retraction
  $P : \intCC{-1 - \eps, 1 + \eps}^n \times \DeltaTop{k} \to \unstrat(\strat{E}_0)$
  as follows:
  \[
    P(e, b) := (e, \tau(e) b + (1 - \tau(e)) R(b))
  \]
  We let $\strat{E}_1$ be the stratification of $\intCC{-1 - \eps, 1 + \eps}^n \times \DeltaTop{k}$
  defined by $\stratMap{\strat{E}_1} = \stratMap{\strat{E}_0} \circ P$
  and let $f_1 : \strat{E}_1 \to \DeltaTop{k}$ be the projection.
  Then $f_1$ is an $n$-manifold diagram bundle which extends $f_0$ and
  is orthogonal to $\StratIntLR^n$.
  Therefore, the equivalence class $[f_1]_{\StratIntLR^n}$ is a solution to the lifting problem.
\end{proof}

\begin{lemma}\label{lem:d-mfld-diag-boundary-kan-cube-irregular}
  Let $s_1, \ldots, s_n \in \ord{1}$ and let
  \[
    \strat{X} := \{ (x_1, \ldots, x_n) \mid \forall 1 \leq i \leq n.\ s_i = 0 \Rightarrow x_i = 0 \} \subseteq \StratIntLR^n
  \]
  be the $n$-framed stratified subspace.
  Then $\MfldDiag^n(\strat{X}) \to \MfldDiag^n(\partial \strat{X})$
  is a Kan fibration.
\end{lemma}
\begin{proof}
  Suppose that we have a lifting problem
  \[
    \begin{tikzcd}
      {\Lambda^i \ord{k}} \ar[r] \ar[d, hook] &
      {\MfldDiag^n(\strat{X})} \ar[d] \\
      {\Delta \ord{k}} \ar[r] \ar[ur, dashed] &
      {\MfldDiag^n(\partial\strat{X})}
    \end{tikzcd}
  \]
  Unpacking the definitions, we have some $0 < \eps < 1$ and an $n$-manifold diagram bundle
  $f_0 : \strat{E}_0 \to \DeltaTop{k}$ that is $\eps$-orthogonal to $\strat{X}$ and defined on the subspace
  \begin{align*}
    \unstrat(\strat{E}_0) =&\
    \{ (e_1, \ldots, e_n, b) \mid \forall 1 \leq i \leq n.\ s_i - \eps < |e_i| < s_i + \eps \} \\
    \cup &\
    \{ (e_1, \ldots, e_n, b) \mid \forall 1 \leq i \leq n.\ |e_i| < s_i + \eps \land b \in \HornTop{i}{k} \}.
  \end{align*}

  We let $\strat{D}_0$ be a stratification with underlying space
  \begin{align*}
    \unstrat(\strat{D}_0) =&\
    \{ (e, b) \mid 1 - \eps < \| e \|_\infty < 1 + \eps \} \\
    \cup &\
    \{ (e, b) \mid \| e \|_\infty < 1 + \eps \land b \in \HornTop{i}{k} \}.
  \end{align*}
  The stratifying map of $\strat{D}_0$ is derived from $\strat{E}_0$ by stretching
  $\strat{E}_0$ into the coordinate directions where $s_i = 0$ so that the result
  is defined around the $n$-cube:
  \[
    \stratMap{\strat{D}_0}(e_1, \ldots, e_n, b) =
    \stratMap{\strat{E}_0}(s_1 e_1, \ldots, s_n e_n, b).
  \]
  Then the projection $g_0 : \strat{D}_0 \to \DeltaTop{k}$ is an
  $n$-manifold diagram bundle  
  and fibrewise orthogonal to $\StratIntLR^n$.
  It classifies a lifting problem  
  \[
    \begin{tikzcd}
      {\Lambda^i \ord{k}} \ar[r] \ar[d, hook] &
      {\MfldDiag^n(\StratIntLR^n)} \ar[d] \\
      {\Delta \ord{k}} \ar[r] \ar[ur, dashed] &
      {\MfldDiag^n(\partial\StratIntLR^n)}
    \end{tikzcd}
  \]
  which has a solution by Lemma~\ref{lem:d-mfld-diag-boundary-kan-cube-regular}.
  Such a solution is given by an $n$-manifold diagram bundle  
  $g_1 : \strat{D}_1 \to \DeltaTop{k}$ that is orthogonal to $\StratIntLR^n$
  and extends $g_0$.

  We then let $\strat{E}_1$ be the extension of $\strat{E}_0$ so that
  \[
    \unstrat(\strat{E}_1) = \{ (e_1, \ldots, e_n, b) \mid \forall 1 \leq i \leq n.\ | e_i | < s_i + \eps \}
  \]
  and the stratification is defined by
  \[
    \stratMap{\strat{E}_1}(e_1, \ldots, e_n, b) =
    \stratMap{\strat{D}_1}(s_1 e_1 + (1 - s_1), \ldots, s_n e_n + (1 - s_n), b).
  \]
  The term $1 - s_i$ in each coordinate ensures that when $s_i = 0$,
  we sample the stratification of $\strat{D}_1$ along the boundary where
  it is required to be orthogonal to $\StratIntLR^n$.
  The projection $f_1 : \strat{E}_1 \to \DeltaTop{k}$ then is an
  $n$-manifold diagram bundle  
  that is orthogonal to $\strat{X}$ and extends $f_0$.
  Therefore, $[f_1]_{\strat{X}}$ is a solution to the lifting problem.
\end{proof}

\begin{lemma}\label{lem:d-mfld-diag-boundary-kan-shape}
  Let $\strat{X}$ be an $n$-cell. Then $\MfldDiag^n(\shape{\strat{X}}) \to \MfldDiag^n(\partial \shape{\strat{X}})$ is a Kan fibration.
\end{lemma}
\begin{proof}
  This follows directly from Lemma~\ref{lem:d-mfld-diag-boundary-kan-cube-irregular} since $\shape{\strat{X}}$ is of the required shape.
\end{proof}

We wish to use the result of Lemma~\ref{lem:d-mfld-diag-boundary-kan-shape} to imply that also
$\MfldDiag^n(\strat{X}) \to \MfldDiag^n(\partial \strat{X})$ is a Kan fibration.
For that we recall the bordism of closed $n$-meshes $\strat{X} \pto \shape{\strat{X}}$
from Construction~\ref{con:fct-embed-stype-active-closed}. For any $t \in \DeltaStrat{1}$ we denote by $\shapeBord{\strat{X}}{t}$
the fibre of the active bordism $\strat{X} \pto \shape{\strat{X}}$ over $t$.

\begin{example}
  The active bordism $\strat{X} \pto \shape{\strat{X}}$ takes the cubically
  shaped $\shape{\strat{X}}$ and makes it approach the shape of $\strat{X}$
  as follows:
  \[
    \begin{gathered}
      \begin{tikzpicture}[scale = 0.5, baseline=(current bounding box.center)]
        \useasboundingbox (0, 0) rectangle (6, 6);
        \fill[mesh-background] (3, 2) -- (2, 4) -- (4, 4) -- cycle;
        \draw[mesh-stratum-dual] (3, 2) -- (2, 4) -- (4, 4) -- cycle;
        \node[mesh-vertex-dual] at (3, 2) {};
        \node[mesh-vertex-dual] at (2, 4) {};
        \node[mesh-vertex-dual] at (4, 4) {};
      \end{tikzpicture}
      \\
      \strat{X} = \shapeBord{\strat{X}}{0}
    \end{gathered}
    \qquad
    \begin{gathered}
      \begin{tikzpicture}[scale = 0.5, baseline=(current bounding box.center)]
        \useasboundingbox (0, 0) rectangle (6, 6);
        \fill[mesh-background] (2.5, 1.5) -- (3.5, 1.5) -- (4.5, 4.5) -- (1.5, 4.5) -- cycle;
        \draw[mesh-stratum-dual] (2.5, 1.5) -- (3.5, 1.5) -- (4.5, 4.5) -- (1.5, 4.5) -- cycle;
        \node[mesh-vertex-dual] at (2.5, 1.5) {};
        \node[mesh-vertex-dual] at (3.5, 1.5) {};
        \node[mesh-vertex-dual] at (4.5, 4.5) {};
        \node[mesh-vertex-dual] at (1.5, 4.5) {};
      \end{tikzpicture}
      \\
      \shapeBord{\strat{X}}{0.5}
    \end{gathered}
    \qquad
    \begin{gathered}
      \begin{tikzpicture}[scale = 0.5, baseline=(current bounding box.center)]
        \useasboundingbox (0, 0) rectangle (6, 6);
        \fill[mesh-background] (1, 1) rectangle (5, 5);
        \draw[mesh-stratum-dual] (1, 1) rectangle (5, 5);
        \node[mesh-vertex-dual] at (1, 1) {};
        \node[mesh-vertex-dual] at (1, 5) {};
        \node[mesh-vertex-dual] at (5, 1) {};
        \node[mesh-vertex-dual] at (5, 5) {};
      \end{tikzpicture}
      \\
      \shape{\strat{X}} = \shapeBord{\strat{X}}{1}
    \end{gathered}
  \]
\end{example}

\begin{lemma}\label{lem:d-mfld-diag-regular-approx}
  Let $\strat{X}$ be an $n$-cell and $\eps > 0$.
  Then there exists a $0 < t \leq 1$ so that
  \[
    \begin{tikzcd}
      {\MfldDiag^{n}(\strat{X})} \ar[r] \ar[d] &
      {\MfldDiag^{n}(\shapeBord{\strat{X}}{t})} \ar[d] \\
      {\MfldDiag^{n}(\partial \strat{X})} \ar[r] &
      {\MfldDiag^{n}(\partial \shapeBord{\strat{X}}{t})}
    \end{tikzcd}
  \]
  restricts along the natural inclusion $\MfldDiag^{n, \eps}(-) \hookrightarrow \MfldDiag^{n}(-)$
  to a pullback square of simplicial sets  
  \[
    \begin{tikzcd}
      {\MfldDiag^{n, \eps}(\strat{X})} \ar[r] \ar[d] \pullbackcorner &
      {\MfldDiag^{n, \eps}(\shapeBord{\strat{X}}{t})} \ar[d] \\
      {\MfldDiag^{n, \eps}(\partial \strat{X})} \ar[r] &
      {\MfldDiag^{n, \eps}(\partial \shapeBord{\strat{X}}{t})}
    \end{tikzcd}
  \]
\end{lemma}
\begin{proof}
  The bordism of closed $n$-meshes $\strat{X} \to \shape{\strat{X}}$ is represented
  by the closed $n$-mesh bundle $\shapeBord{\strat{X}}{\bullet} \to \DeltaStrat{1}$.
  The fibre $\shapeBord{\strat{X}}{t}$ over any $0 < t \leq 1$ is framed isomorphic to $\shape{\strat{X}}$,
  but as $t \to 0$ it converges to the shape of $\strat{X}$.
  In particular there must be a $0 < t \leq 1$ such that
  \[
    \unstrat(\partial \strat{X}) \subseteq
    \{
      e \mid \exists x \in \partial\shapeBord{\strat{X}}{t}.\ \| e - x \|_\infty < \eps
    \}.
  \]  
  Then the square must be a pullback square due to $\eps$-orthogonality.
\end{proof}

\begin{lemma}\label{lem:d-mfld-diag-boundary-kan}
  Let $\strat{X}$ be an $n$-cell.
  Then $\MfldDiag^n(\strat{X}) \to \MfldDiag^n(\partial \strat{X})$ is a Kan fibration.
\end{lemma}
\begin{proof}
  By Lemma~\ref{lem:d-mfld-diag-regular-approx} for any $\eps > 0$ there exists a $t > 0$ such that
  we have a diagram
  \[
    \begin{tikzcd}
      {\MfldDiag^{n, \eps}(\strat{X})} \ar[r] \ar[d] \pullbackcorner &
      {\MfldDiag^{n, \eps}(\shapeBord{\strat{X}}{t})} \ar[d] \ar[r] &
      {\MfldDiag^n(\shapeBord{\strat{X}}{t})} \ar[d] \ar[r, "\cong"] &
      {\MfldDiag^n(\shape{\strat{X}})} \ar[d] \\
      {\MfldDiag^{n, \eps}(\partial \strat{X})} \ar[r] &
      {\MfldDiag^{n, \eps}(\partial \shapeBord{\strat{X}}{t})} \ar[r] &
      {\MfldDiag^n(\partial \shapeBord{\strat{X}}{t})} \ar[r, "\cong"'] &
      {\MfldDiag^n(\partial \shape{\strat{X}})}
    \end{tikzcd}
  \]
  where the first square is a pullback square, the second a retract and the third an isomorphism of maps.
  By Lemma~\ref{lem:d-mfld-diag-boundary-kan-shape} the map on the right $\MfldDiag^n(\shape{\strat{X}}) \to \MfldDiag^n(\partial \shape{\strat{X}})$
  is a Kan fibration, and therefore so is the map on the left
  $\MfldDiag^{n, \eps}(\strat{X}) \to \MfldDiag^{n, \eps}(\partial \strat{X})$.
  Since Kan fibrations are closed under filtered colimits, it follows that
  $\MfldDiag^n(\strat{X}) \to \MfldDiag^n(\partial \strat{X})$
  is a Kan fibration.
\end{proof}

\begin{lemma}\label{lem:d-mfld-diag-space-segal}
  Let $\strat{X}$ be a closed $n$-mesh and let
  $U : \stratPos{\strat{X}}^{\op, \triangleleft} \to \BMeshClosed{n}$
  be its cellular covering diagram.
  Then the diagram $\MfldDiag^n(U(-)) : \stratPos{\strat{X}}^{\op, \triangleleft} \to \sSet$
  induces a limit diagram of spaces.
  In particular the functor $\MfldDiag^n(-) : \BMeshClosed{n} \to \Space$ satisfies the Segal condition.
\end{lemma}
\begin{proof}
  Every closed $n$-mesh is essentially tame, and so up
  to equivalence we can ensure that $\strat{X}$ and each of its
  cells is tame.
  For every tame closed $n$-mesh $\strat{Y}$ we have a deformation retract
  $\MfldDiagTame^n(\strat{Y})$ of $\MfldDiag^n(\strat{Y})$ consisting
  only of those $n$-manifold diagrams on $\strat{Y}$ that are tame as
  $n$-framed stratified spaces.
  It therefore suffices to show that $\MfldDiagTame^n(U(-))$ induces a limit
  diagram of spaces.
  By gluing together the $n$-manifold diagrams on the cells of $\strat{X}$,
  the diagram $\MfldDiagTame^n(U(-))$ is a limit diagram of simplicial sets.
  Kan fibrations are closed under retract, and so 
  by Lemma~\ref{lem:d-mfld-diag-boundary-kan} the diagram is Reedy fibrant.
  Therefore, the limit diagram of simplicial sets $\MfldDiagTame^n(U(-))$  
  represents a limit diagram of spaces.  
\end{proof}

\begin{lemma}\label{lem:d-mfld-diag-space-regular}
  Let $\strat{X}$ be an $n$-cell. Then
  \[
    \begin{tikzcd}
      {\MfldDiag^{n}(\strat{X})} \ar[r] \ar[d] \pullbackcorner &
      {\MfldDiag^{n}(\shape{\strat{X}})} \ar[d] \\
      {\MfldDiag^{n}(\partial \strat{X})} \ar[r] &
      {\MfldDiag^{n}(\partial \shape{\strat{X}})}
    \end{tikzcd}
  \]
  is a pullback square of spaces.
  In particular the functor $\MfldDiag^n(-) : \BMeshClosed{n} \to \Space$ satisfies the regularity condition.
\end{lemma}
\begin{proof}
  This follows from Lemma~\ref{lem:d-mfld-diag-regular-approx}
  and Observation~\ref{obs:d-mfld-space-eps}.
\end{proof}

\subsection{Extended Manifold Diagrams}\label{sec:d-mfld-extended}

An $n$-manifold diagram on a closed $n$-mesh $\strat{X}$ is defined within an arbitrary small open neighbourhood of $\strat{X}$.
Extended $n$-manifold diagrams $\strat{M}$ on $\strat{X}$ are an alternative presentation
for which we require $\strat{M}$ to be defined everywhere on $\R^n$.
Extended $n$-manifold diagrams organise into an $n$-diagrammatic space
$\MfldDiagExt^n(-) : \BMeshClosed{n} \to \Space$ that is equivalent to
the $n$-diagrammatic space of $n$-manifold diagrams $\MfldDiag^n(-)$.

\begin{definition}
  We write $\MfldDiagExt^n$ for the simplicial set of \defn{extended $n$-manifold diagrams}
  whose $k$-simplices consist
  of a pair $(\xi, p)$ where $\xi : \strat{X} \to \DeltaStrat{k}$ is a closed
  $n$-mesh bundle and $p : \strat{E} \to \DeltaTop{k}$ is an $n$-manifold diagram
  bundle orthogonal to $\xi$ such that  
  $\unstrat(\strat{E}) = \R^n \times \DeltaTop{k}$.
  There is a projection map $\MfldDiagExt^n \to \BMeshClosed{n}$ which sends
  a pair $(\xi, p)$ to the closed $n$-mesh bundle $\xi$.
\end{definition}

\begin{proposition}
  The map $\MfldDiagExt^n \to \BMeshClosed{n}$ is a left fibration.
\end{proposition}
\begin{proof}
  The proof is analogous to that of Proposition~\ref{prop:d-mfld-functorial-left},
  except that this time we do not take equivalence classes.
\end{proof}

\begin{observation}
  Since $\MfldDiagExt^n \to \BMeshClosed{n}$ is a left fibration it induces a
  functor of $\infty$-categories
  \[ \MfldDiagExt^n(-) : \BMeshClosed{n} \longrightarrow \Space.\]
  The space $\MfldDiagExt^n(\strat{X})$ for any closed $n$-mesh $\strat{X}$ is presented
  as a Kan complex via the pullback of simplicial sets
  \[
    \begin{tikzcd}
      {\MfldDiagExt^n(\strat{X})} \ar[r] \ar[d] \pullbackcorner &
      {\MfldDiagExt^n} \ar[d] \\
      {\Delta\ord{0}} \ar[r, "\strat{X}"'] &
      {\BMeshClosed{n}}
    \end{tikzcd}
  \]
  In particular a $k$-simplex of $\MfldDiagExt^n(\strat{X})$ is an
  $n$-manifold diagram bundle $p : \strat{E} \to \DeltaTop{k}$ with  
  $\unstrat(\strat{E}) = \R^n \times \DeltaTop{k}$
  that is fibrewise orthogonal to $\strat{X}$.
\end{observation}

\begin{example}\label{ex:d-mfld-diag-off-to-infinity}
  An extended $2$-manifold diagram on the $2$-cube $\StratIntLR^2$ resembles a string diagram
  but can have non-trivial geometry outside the bounds of $\StratIntLR^2$:
  \[
    \begin{tikzpicture}[scale = 0.5, baseline=(current bounding box.center)]
      \fill[mesh-background] (0, 0) rectangle (7, 6);
      \node[mesh-vertex-dual] at (1, 1) {};
      \node[mesh-vertex-dual] at (3, 1) {};
      \node[mesh-vertex-dual] at (1, 3) {};
      \node[mesh-vertex-dual] at (3, 3) {};
      \draw[mesh-stratum-dual] (1, 1) rectangle ++(2, 2);
      \draw[mesh-stratum] (1.5, 0) -- (1.5, 1) .. controls +(0, 0.2) and +(-0.5, 0) .. (2, 2);
      \draw[mesh-stratum] (2.5, 0) -- (2.5, 1) .. controls +(0, 0.2) and +(0.5, 0) .. (2, 2);
      \draw[mesh-stratum] (2, 2) -- (2, 3) .. controls +(0, 0.2) and +(-1, 0) .. (3, 4);
      \draw[mesh-stratum] (4, 2) -- (4, 3) .. controls +(0, 0.2) and +(1, 0) .. (3, 4);
      \draw[mesh-stratum] (5, 4) -- (5, 6);
      \node[mesh-vertex] at (4, 2) {};
      \node[mesh-vertex] at (3, 4) {};
      \node[mesh-vertex] at (5, 4) {};
      \node[mesh-vertex] at (5, 4) {};
      \node[mesh-vertex] at (2, 2) {};
    \end{tikzpicture}
  \]
  Up to isotopy the diagram is fully determined by its intersection with $\StratIntLR^2$.  
  To see this we first move the parts of the string diagram that are above
  $\StratIntLR^2$ to infinity and off the top of the diagram, using that $\R^2$ is open:
  \[
    \begin{tikzpicture}[scale = 0.5, baseline=(current bounding box.center)]
      \fill[mesh-background] (0, 0) rectangle (7, 6);
      \node[mesh-vertex-dual] at (1, 1) {};
      \node[mesh-vertex-dual] at (3, 1) {};
      \node[mesh-vertex-dual] at (1, 3) {};
      \node[mesh-vertex-dual] at (3, 3) {};
      \draw[mesh-stratum-dual] (1, 1) rectangle ++(2, 2);
      \draw[mesh-stratum] (1.5, 0) -- (1.5, 1) .. controls +(0, 0.2) and +(-0.5, 0) .. (2, 2);
      \draw[mesh-stratum] (2.5, 0) -- (2.5, 1) .. controls +(0, 0.2) and +(0.5, 0) .. (2, 2);
      \draw[mesh-stratum] (2, 2) -- (2, 3) .. controls +(0, 0.2) and +(-1, 0) .. (3, 4);
      \draw[mesh-stratum] (4, 2) -- (4, 3) .. controls +(0, 0.2) and +(1, 0) .. (3, 4);
      \draw[mesh-stratum] (5, 4) -- (5, 6);
      \node[mesh-vertex] at (4, 2) {};
      \node[mesh-vertex] at (3, 4) {};
      \node[mesh-vertex] at (5, 4) {};
      \node[mesh-vertex] at (5, 4) {};
      \node[mesh-vertex] at (2, 2) {};
    \end{tikzpicture}
    \quad
    \sim
    \quad
    \begin{tikzpicture}[scale = 0.5, baseline=(current bounding box.center)]
      \fill[mesh-background] (0, 0) rectangle (7, 6);
      \node[mesh-vertex-dual] at (1, 1) {};
      \node[mesh-vertex-dual] at (3, 1) {};
      \node[mesh-vertex-dual] at (1, 3) {};
      \node[mesh-vertex-dual] at (3, 3) {};
      \draw[mesh-stratum-dual] (1, 1) rectangle ++(2, 2);
      \draw[mesh-stratum] (1.5, 0) -- (1.5, 1) .. controls +(0, 0.2) and +(-0.5, 0) .. (2, 2);
      \draw[mesh-stratum] (2.5, 0) -- (2.5, 1) .. controls +(0, 0.2) and +(0.5, 0) .. (2, 2);
      \draw[mesh-stratum] (2, 2) -- (2, 4) .. controls +(0, 0.2) and +(-1, 0) .. (3, 5);
      \draw[mesh-stratum] (4, 2) -- (4, 4) .. controls +(0, 0.2) and +(1, 0) .. (3, 5);
      \draw[mesh-stratum] (5, 5) -- (5, 6);
      \node[mesh-vertex] at (4, 2) {};
      \node[mesh-vertex] at (3, 5) {};
      \node[mesh-vertex] at (5, 5) {};
      \node[mesh-vertex] at (5, 5) {};
      \node[mesh-vertex] at (2, 2) {};
    \end{tikzpicture}
    \quad
    \sim
    \quad
    \begin{tikzpicture}[scale = 0.5, baseline=(current bounding box.center)]
      \fill[mesh-background] (0, 0) rectangle (7, 6);
      \node[mesh-vertex-dual] at (1, 1) {};
      \node[mesh-vertex-dual] at (3, 1) {};
      \node[mesh-vertex-dual] at (1, 3) {};
      \node[mesh-vertex-dual] at (3, 3) {};
      \draw[mesh-stratum-dual] (1, 1) rectangle ++(2, 2);
      \draw[mesh-stratum] (1.5, 0) -- (1.5, 1) .. controls +(0, 0.2) and +(-0.5, 0) .. (2, 2);
      \draw[mesh-stratum] (2.5, 0) -- (2.5, 1) .. controls +(0, 0.2) and +(0.5, 0) .. (2, 2);
      \draw[mesh-stratum] (2, 2) -- (2, 6);
      \draw[mesh-stratum] (4, 2) -- (4, 6);
      \node[mesh-vertex] at (4, 2) {};
      \node[mesh-vertex] at (2, 2) {};
    \end{tikzpicture}
  \]
  We then move everything that is to the right of $\StratIntLR^2$ off the boundary
  on the right-hand side, so that we end up with just the diagram contained
  within $\StratIntLR^2$:
  \[
    \begin{tikzpicture}[scale = 0.5, baseline=(current bounding box.center)]
      \fill[mesh-background] (0, 0) rectangle (7, 6);
      \node[mesh-vertex-dual] at (1, 1) {};
      \node[mesh-vertex-dual] at (3, 1) {};
      \node[mesh-vertex-dual] at (1, 3) {};
      \node[mesh-vertex-dual] at (3, 3) {};
      \draw[mesh-stratum-dual] (1, 1) rectangle ++(2, 2);
      \draw[mesh-stratum] (1.5, 0) -- (1.5, 1) .. controls +(0, 0.2) and +(-0.5, 0) .. (2, 2);
      \draw[mesh-stratum] (2.5, 0) -- (2.5, 1) .. controls +(0, 0.2) and +(0.5, 0) .. (2, 2);
      \draw[mesh-stratum] (2, 2) -- (2, 6);
      \draw[mesh-stratum] (4, 2) -- (4, 6);
      \node[mesh-vertex] at (4, 2) {};
      \node[mesh-vertex] at (2, 2) {};
    \end{tikzpicture}
    \quad
    \sim
    \quad
    \begin{tikzpicture}[scale = 0.5, baseline=(current bounding box.center)]
      \fill[mesh-background] (0, 0) rectangle (7, 6);
      \node[mesh-vertex-dual] at (1, 1) {};
      \node[mesh-vertex-dual] at (3, 1) {};
      \node[mesh-vertex-dual] at (1, 3) {};
      \node[mesh-vertex-dual] at (3, 3) {};
      \draw[mesh-stratum-dual] (1, 1) rectangle ++(2, 2);
      \draw[mesh-stratum] (1.5, 0) -- (1.5, 1) .. controls +(0, 0.2) and +(-0.5, 0) .. (2, 2);
      \draw[mesh-stratum] (2.5, 0) -- (2.5, 1) .. controls +(0, 0.2) and +(0.5, 0) .. (2, 2);
      \draw[mesh-stratum] (2, 2) -- (2, 6);
      \draw[mesh-stratum] (5.5, 2) -- (5.5, 6);
      \node[mesh-vertex] at (5.5, 2) {};
      \node[mesh-vertex] at (2, 2) {};
    \end{tikzpicture}
    \quad
    \sim
    \quad
    \begin{tikzpicture}[scale = 0.5, baseline=(current bounding box.center)]
      \fill[mesh-background] (0, 0) rectangle (7, 6);
      \node[mesh-vertex-dual] at (1, 1) {};
      \node[mesh-vertex-dual] at (3, 1) {};
      \node[mesh-vertex-dual] at (1, 3) {};
      \node[mesh-vertex-dual] at (3, 3) {};
      \draw[mesh-stratum-dual] (1, 1) rectangle ++(2, 2);
      \draw[mesh-stratum] (1.5, 0) -- (1.5, 1) .. controls +(0, 0.2) and +(-0.5, 0) .. (2, 2);
      \draw[mesh-stratum] (2.5, 0) -- (2.5, 1) .. controls +(0, 0.2) and +(0.5, 0) .. (2, 2);
      \draw[mesh-stratum] (2, 2) -- (2, 6);
      \node[mesh-vertex] at (2, 2) {};
    \end{tikzpicture}
  \]
\end{example}

\begin{example}\label{ex:d-mfld-diag-restrict-2}
  Suppose that we have an embedding of closed $2$-meshes
  \[
    \strat{X}
    \quad
    =
    \quad
    \begin{tikzpicture}[scale = 0.5, baseline=(current bounding box.center)]
      \fill[mesh-background] (1, 1) rectangle ++(2, 2);
      \node[mesh-vertex-dual] at (1, 1) {};
      \node[mesh-vertex-dual] at (3, 1) {};
      \node[mesh-vertex-dual] at (1, 3) {};
      \node[mesh-vertex-dual] at (3, 3) {};
      \draw[mesh-stratum-dual] (1, 1) rectangle ++(2, 2);
    \end{tikzpicture}
    \quad
    \overset{e}{\hookrightarrow}
    \quad
    \begin{tikzpicture}[scale = 0.5, baseline=(current bounding box.center)]
      \fill[mesh-background] (1, 1) rectangle ++(5, 4);
      \node[mesh-vertex-dual] at (1, 1) {};
      \node[mesh-vertex-dual] at (3, 1) {};
      \node[mesh-vertex-dual] at (6, 1) {};
      \node[mesh-vertex-dual] at (1, 3) {};
      \node[mesh-vertex-dual] at (3, 3) {};
      \node[mesh-vertex-dual] at (6, 3) {};
      \node[mesh-vertex-dual] at (1, 5) {};
      \node[mesh-vertex-dual] at (6, 5) {};
      \draw[mesh-stratum-dual] (1, 1) rectangle ++(2, 2);
      \draw[mesh-stratum-dual] (3, 1) rectangle ++(3, 2);
      \draw[mesh-stratum-dual] (1, 3) rectangle ++(5, 2);
    \end{tikzpicture}
    \quad
    =
    \quad
    \strat{Y}
  \]
  The image of $e$ is the lower left corner in $\strat{Y}$.
  We can illustrate the action of $e : \strat{X} \hookrightarrow \strat{Y}$
  on an extended $n$-manifold diagram on $\strat{X}$ as follows:
  \[
    \begin{tikzpicture}[scale = 0.5, baseline=(current bounding box.center)]
      \fill[mesh-background] (0, 0) rectangle (7, 6);
      \node[mesh-vertex-dual] at (1, 1) {};
      \node[mesh-vertex-dual] at (3, 1) {};
      \node[mesh-vertex-dual] at (6, 1) {};
      \node[mesh-vertex-dual] at (1, 3) {};
      \node[mesh-vertex-dual] at (3, 3) {};
      \node[mesh-vertex-dual] at (6, 3) {};
      \node[mesh-vertex-dual] at (1, 5) {};
      \node[mesh-vertex-dual] at (6, 5) {};
      \draw[mesh-stratum-dual] (1, 1) rectangle ++(2, 2);
      \draw[mesh-stratum-dual] (3, 1) rectangle ++(3, 2);
      \draw[mesh-stratum-dual] (1, 3) rectangle ++(5, 2);
      \draw[mesh-stratum] (1.5, 0) -- (1.5, 1) .. controls +(0, 0.2) and +(-0.5, 0) .. (2, 2);
      \draw[mesh-stratum] (2.5, 0) -- (2.5, 1) .. controls +(0, 0.2) and +(0.5, 0) .. (2, 2);
      \draw[mesh-stratum] (2, 2) -- (2, 3) .. controls +(0, 0.2) and +(-1, 0) .. (3, 4);
      \draw[mesh-stratum] (4, 2) -- (4, 3) .. controls +(0, 0.2) and +(1, 0) .. (3, 4);
      \draw[mesh-stratum] (5, 4) -- (5, 6);
      \node[mesh-vertex] at (4, 2) {};
      \node[mesh-vertex] at (3, 4) {};
      \node[mesh-vertex] at (5, 4) {};
      \node[mesh-vertex] at (5, 4) {};
      \node[mesh-vertex] at (2, 2) {};
    \end{tikzpicture}
    \quad
    \mapsto
    \quad
    \begin{tikzpicture}[scale = 0.5, baseline=(current bounding box.center)]
      \fill[mesh-background] (0, 0) rectangle (7, 6);
      \node[mesh-vertex-dual] at (1, 1) {};
      \node[mesh-vertex-dual] at (3, 1) {};
      \node[mesh-vertex-dual] at (1, 3) {};
      \node[mesh-vertex-dual] at (3, 3) {};
      \draw[mesh-stratum-dual] (1, 1) rectangle ++(2, 2);
      \draw[mesh-stratum] (1.5, 0) -- (1.5, 1) .. controls +(0, 0.2) and +(-0.5, 0) .. (2, 2);
      \draw[mesh-stratum] (2.5, 0) -- (2.5, 1) .. controls +(0, 0.2) and +(0.5, 0) .. (2, 2);
      \draw[mesh-stratum] (2, 2) -- (2, 3) .. controls +(0, 0.2) and +(-1, 0) .. (3, 4);
      \draw[mesh-stratum] (4, 2) -- (4, 3) .. controls +(0, 0.2) and +(1, 0) .. (3, 4);
      \draw[mesh-stratum] (5, 4) -- (5, 6);
      \node[mesh-vertex] at (4, 2) {};
      \node[mesh-vertex] at (3, 4) {};
      \node[mesh-vertex] at (5, 4) {};
      \node[mesh-vertex] at (5, 4) {};
      \node[mesh-vertex] at (2, 2) {};
    \end{tikzpicture}
  \]
  The closed $n$-mesh changes while the diagram itself remains untouched.
  However, we can still interpret this as a restriction of the string diagram
  to the subcell.
  While the rest of the string diagram is still there, it is now no longer
  contained within the closed $2$-mesh, and we are therefore free to move it aside,
  as we did in Example~\ref{ex:d-mfld-diag-off-to-infinity}.
\end{example}

\begin{example}
  Even when the closed $n$-mesh $\strat{X}$ has more exotic shapes, we can
  move aside the parts of the diagram that are not contained within $\strat{X}$.
  \[
    \begin{tikzpicture}[scale = 0.5, baseline=(current bounding box.center)]
      \fill[mesh-background] (0, 0) rectangle (6, 6);
      \draw[mesh-stratum-dual] (5, 1) -- (1, 1) -- (1, 5) -- (5, 5);
      \node[mesh-vertex-dual] at (5, 1) {};
      \node[mesh-vertex-dual] at (1, 1) {};
      \node[mesh-vertex-dual] at (1, 5) {};
      \node[mesh-vertex-dual] at (5, 5) {};
      \draw[mesh-stratum] (2, 0) -- +(0, 2) .. controls +(0, 0.2) and +(-1, 0) .. (3, 3);
      \draw[mesh-stratum] (4, 0) -- +(0, 2) .. controls +(0, 0.2) and +(1, 0) .. (3, 3);
      \draw[mesh-stratum] (3, 3) -- (3, 6);
      \node[mesh-vertex] at (3, 3) {};
    \end{tikzpicture}
    \quad
    \sim
    \quad
    \begin{tikzpicture}[scale = 0.5, baseline=(current bounding box.center)]
      \fill[mesh-background] (0, 0) rectangle (6, 6);
      \draw[mesh-stratum-dual] (5, 1) -- (1, 1) -- (1, 5) -- (5, 5);
      \node[mesh-vertex-dual] at (5, 1) {};
      \node[mesh-vertex-dual] at (1, 1) {};
      \node[mesh-vertex-dual] at (1, 5) {};
      \node[mesh-vertex-dual] at (5, 5) {};
      \draw[mesh-stratum] (2, 0) -- +(0, 2) .. controls +(0, 1) and +(-1, 0) .. (6, 3);
      \draw[mesh-stratum] (4, 0) -- +(0, 2) .. controls +(0, 1) and +(-1, 0) .. (6, 3);
      \draw[mesh-stratum] (3, 6) -- (3, 4) .. controls +(0, -1) and +(-1, 0) .. (6, 3);
    \end{tikzpicture}
  \]
  
\end{example}

\begin{observation}
  There is a map of simplicial sets $\MfldDiagExt^n \to \MfldDiag^n$
  which restricts extended $n$-manifold diagrams to an open neighbourhood around
  the closed $n$-mesh.
  This map commutes with the projections onto $\BMeshClosed{n}$ and therefore
  induces a natural transformation $\MfldDiagExt^n(-) \to \MfldDiag^n(-)$ of
  functors $\BMeshClosed{n} \to \Space$.
\end{observation}


\begin{proposition}\label{prop:d-mfld-ext-equiv}
  The natural transformation $\MfldDiagExt^n(-) \to \MfldDiag^n(-)$ is an equivalence
  of functors $\BMeshClosed{n} \to \Space$.
  In particular $\MfldDiagExt^n(-)$ is an $n$-diagrammatic space.
\end{proposition}
\begin{proof}
  This follows directly from Lemma~\ref{lem:d-mfld-ext-equiv} below.
\end{proof}

We first demonstrate the equivalence between extended and non-extended $n$-manifold
diagrams on the closed $n$-cube $\StratIntLR^n$.

\begin{lemma}\label{lem:d-mfld-ext-equiv-cube}
  The map $\MfldDiagExt^n(\StratIntLR^n) \to \MfldDiag^n(\StratIntLR^n)$ is a trivial fibration.
\end{lemma}
\begin{proof}
  Suppose we have a lifting problem
  \[
    \begin{tikzcd}
      {\partial\Delta\ord{k}} \ar[r] \ar[d] &
      {\MfldDiagExt^n(\StratIntLR^n)} \ar[d] \\
      {\Delta\ord{k}} \ar[r] \ar[ur, dashed] &
      {\MfldDiag^n(\StratIntLR^n)}
    \end{tikzcd}
  \]
  Unpacking the given data and choosing representatives, we have an $\eps > 0$
  and an $n$-manifold diagram bundle $f_0 : \strat{E}_0 \to \DeltaTop{k}$ that
  is fibrewise $\eps$-orthogonal to $\StratIntLR^n$ and defined on
  \[
    \unstrat(\strat{E}_0) = \{ (e, b) \mid \| e \|_\infty < 1 + \eps \} \cup (\R^n \times \partial \DeltaTop{k}) \subseteq \R^n \times \DeltaTop{k}.
  \]
  Let $N : \partial\DeltaTop{k} \times \intCO{0, 1} \hookrightarrow \DeltaTop{k}$
  be a collar neighbourhood of the boundary $\partial\DeltaTop{k}$.
  We then define a continuous function $P : \R^n \times \DeltaTop{k} \to \R^n$ as follows:
  \begin{alignat*}{2}
    P(e, b) &:= (1 - t) e + t \| e \|_\infty^{-1} e &&\qquad\text{(when $\| e \|_\infty > 1$ and $N(b', t) = b$)} \\
    P(e, b) &:= \| e \|_\infty^{-1} e &&\qquad\text{(when $\| e \|_\infty > 1$ and $b \not \in \im(N)$)} \\
    P(e, b) &:= e &&\qquad\text{(when $\| e \|_\infty \leq 1$)}
  \end{alignat*}
  In particular we have $P(e, b) = e$ whenever $b \in \partial\DeltaTop{k}$.
  We let $\strat{E}_1$ be the stratification of $\R^n \times \DeltaTop{k}$
  such that
  $\stratMap{\strat{E}_1}(e, b) = \stratMap{\strat{E}_0}(P(e, b), b)$.
  Then the projection $f_1 : \strat{E}_1 \to \DeltaTop{k}$ is an $n$-manifold diagram bundle
  that extends $f_0$ and is fibrewise orthogonal to $\StratIntLR^n$.
\end{proof}

\begin{construction}
  Let $\xi : \strat{X} \to \strat{B}$ be a closed $n$-mesh bundle.
  The \defn{$m$-fold cylinder} of $\xi$ is the closed $(n + m)$-mesh bundle
  $\CylMeshClosed^m(\xi) : \StratIntLR^m \times \strat{X} \to \strat{B}$
  that arises as the stacked product (see Definition~\ref{con:fct-stacked-product})
  of $\xi$ followed by the closed $m$-mesh bundle $\StratIntLR^m \to \DeltaStrat{0}$.
\end{construction}

\begin{lemma}\label{lem:d-mfld-ext-equiv}
  The map $\core(\MfldDiagExt^{n}) \to \core(\MfldDiag^{n})$ is a trivial fibration.
\end{lemma}
\begin{proof}
  We need to show that every lifting problem of the form
  \[
    \begin{tikzcd}
      {\partial\Delta\ord{k}} \ar[r] \ar[d] &
      {\core(\MfldDiagExt^{n})} \ar[d] \\
      {\Delta\ord{k}} \ar[r] \ar[ur, dashed] &
      {\core(\MfldDiag^{n})}
    \end{tikzcd}
  \]
  has a solution.
  We do this by induction on $n, m \geq 0$ such that $0 \leq n \leq m$
  for the cases where the closed $n$-mesh bundle induced by the $k$-simplex
  $\Delta\ord{k} \to \MfldDiag^{n, \simeq}$ is framed isomorphic to
  the $m$-fold cylinder $\CylMeshClosed^m(\xi)$ of a closed $(n - m)$-mesh bundle over $\DeltaTop{k}$.
  
  The case $n = 0$ is trivial and the case that $n = m$ follows from Lemma~\ref{lem:d-mfld-ext-equiv-cube}.
  Therefore, for now we assume that $0 \leq m < n$.
  Unpacking the definitions and making convenient choices of representatives
  the lifting problem provides the following data:
  We are given a closed $(n - m - 1)$-mesh bundle
  $\xi : \strat{X} \to \strat{Y}$
  and a closed $1$-mesh bundle
  $\zeta : \strat{Y} \to \DeltaTop{k}$.
  Without loss of generality we can assume that the closed $n$-mesh bundle
  is equal and not just framed isomorphic to the $m$-fold cylinder
  $\CylMeshClosed^m(\zeta \circ \xi) : \StratIntLR^m \times \strat{X} \to \DeltaTop{k}$.
  
  Let $s + 1$ be the number of singular strata of $\strat{Y}$ and for every $i \in \ord{s}$ let $h_i : \DeltaTop{k} \to \R$ be the height function of the $i$th singular stratum.
  We further have an $\eps > 0$ and an $n$-manifold diagram bundle  
  $f_0 : \strat{E}_0 \to \DeltaTop{k}$
  which is $2\eps$-orthogonal to the closed $n$-mesh bundle $\CylMeshClosed^m(\zeta \circ \xi) : \StratIntLR^m \times \strat{X} \to \DeltaTop{k}$ and defined on  
  \[
    \unstrat(\strat{E}_0) = \{
      (e, b) \in \R^n \times \DeltaTop{k} \mid
      \exists x \in \StratIntLR^m \times \strat{X}.\ 
      \| e - x \|_\infty < \eps
    \}
    \cup
    (\R^n \times \partial\DeltaTop{k})
  \]
  Shrinking $\strat{E}_0$ if necessary we also ensure that $4\eps < h_{i + 1}(b) - h_i(b)$ for all $i \in \ord{s}$ with $i < s$ and $b \in \DeltaTop{k}$.
  We then gradually extend $f_0$ as follows.

  \begin{enumerate}
    \item 
    We extend $f_0$ in a neighbourhood around the singular slices of $\strat{Y}$.
    Let $i \in \ord{s}$ be the index of a singular slice and
    $\sigma_i : \strat{S}_i \to \DeltaTop{k}$ be the closed $(n - m - 1)$-mesh bundle obtained via the pullback
    \[
      \begin{tikzcd}
        {\strat{S}_i} \ar[r, hook] \ar[d] \pullbackcorner &
        {\strat{X}} \ar[d, "\xi"] \\
        {\DeltaTop{k}} \ar[r, "h_i"', hook] &
        {\strat{Y}}
      \end{tikzcd}
    \]
    Moreover we let $\strat{D}_i$ be the stratified subspace of $\strat{E}_0$
    given by
    \[
      \strat{D}_i :=
      \strat{E}_0 \cap
      \{ (e_1, \ldots, e_n, b) \in \R^n \times \DeltaTop{k} \mid e_{n - m} = h_i(b) \}
    \]
    We interpret $\unstrat(\strat{D}_i)$ as a subspace of $\R^{n - 1} \times \DeltaTop{k}$,
    such that the projection map $g_i : \strat{D}_i \to \DeltaTop{k}$ is an
    $(n - 1)$-manifold diagram bundle that is fibrewise $2\eps$-orthogonal to
    the closed $(n - 1)$-mesh bundle $\CylMeshClosed^m(\sigma_i)$.
    Then $g_i$ induces a lifting problem
    \[
      \begin{tikzcd}
        {\partial\Delta\ord{k}} \ar[r] \ar[d] &
        {\core(\MfldDiagExt^{n - 1})} \ar[d] \\
        {\Delta\ord{k}} \ar[r] \ar[ur, dashed] &
        {\core(\MfldDiag^{n - 1})}
      \end{tikzcd}
    \]
    which by induction on $n$ has a solution.
    By assembling these solutions for all singular slices $i \in \ord{s}$
    and using orthogonality of $f_0$ and $\CylMeshClosed^m(\zeta \circ \xi)$
    we can extend $f_0$ to an $n$-manifold diagram bundle
    $f_1 : \strat{E}_1 \to \DeltaTop{k}$ that is defined within a neighbourhood
    of the singular slices of $\strat{X}$:    
    \[
      \unstrat(\strat{E}_1) = \unstrat(\strat{E}_0)
      \cup
      \{
        (e_1, \ldots, e_n, b) \mid
        \exists i \in \ord{s}.\ 
        | e_{n - m} - h_i(b) | < \eps_0
      \}
    \]

    \item
    We now extend $f_1$ between the singular slices, but leaving a gap to the extension around the singular slices.
    Let $i \in \ord{s}$ be the index of a singular slice of $\strat{Y}$ such that $i < s$. We define a stratified space $\strat{B}_i$ whose underlying space is
    \[ \unstrat(\strat{B}_i) := \{ (t, b) \in \strat{Y} \mid h_i(b) + 2\eps \leq  t \leq h_{i + 1}(b) - 2\eps \} \subseteq \unstrat(\strat{Y}) \]
    and whose three strata are the subspaces where one or none of the inequalities is tight.
    There is a canonical map $\strat{B}_i \to \strat{Y}$ which is the inclusion on the underlying topological spaces.
    We construct a closed $(n - m - 1)$-mesh bundle $\rho_i$ via the pullback
    along this map:
    \[
      \begin{tikzcd}
        {\strat{R}_i} \ar[r] \ar[d, "\rho_i"'] \pullbackcorner &
        {\strat{X}} \ar[d, "\xi"] \\
        {\strat{B}_i} \ar[r] &
        {\strat{Y}}
      \end{tikzcd}
    \]

    We note that the projection map $\beta_i : \strat{B}_i \to \DeltaTop{k}$ is a closed $1$-mesh bundle and isomorphic as such to the projection $\StratIntLR \times \DeltaTop{k} \to \DeltaTop{k}$.
    We therefore have that the $m$-fold cylinder $\CylMeshClosed^m(\beta_i \circ \rho_i)$ is isomorphic as an $n$-framed stratified bundle to the $(m + 1)$-fold cylinder $\CylMeshClosed^{m + 1}(\rho_i)$.

    We let $\strat{D}_i$ be the stratified subspace of $\strat{E}_1$ such that
    \[
      \unstrat(\strat{D}_i) = \{
        (e_1, \ldots, e_n, b) \mid
        h_i(b) + \eps < e_{n - m} < h_{i + 1}(b) - \eps
      \}
    \]
    and let $g_i : \strat{D}_i \to \DeltaTop{k}$ denote the projection map.
    Then $g_i$ is $\eps$-orthogonal to the $(m + 1)$-fold cylinder
    $\CylMeshClosed^{m + 1}(\rho_i)$,
    defining a lifting problem
    \[
      \begin{tikzcd}
        {\partial\Delta\ord{k}} \ar[r] \ar[d] &
        {\core(\MfldDiagExt^{n})} \ar[d] \\
        {\Delta\ord{k}} \ar[r] \ar[ur, dashed] &
        {\core(\MfldDiag^{n})}
      \end{tikzcd}
    \]
    By induction on $m$ this lifting problem has a solution.
    These solutions for all $i \in \ord{s}$ with $i < s$
    then assemble into an extension
    $f_2 : \strat{E}_2 \to \DeltaTop{k}$ of $f_1$ defined on
    \[
      \unstrat(\strat{E}_2) = \unstrat(\strat{E}_1)
      \cup
      \{
        (e_1, \ldots, e_n, b) \mid
        \exists i \in \ord{s}.\ 
        i < s \land
        h_i + \eps < e_{n - m} < h_{i + 1} - \eps
      \}
    \]

    \item 
    Since $\eps_0 < \eps$ there is an $(\eps - \eps_0)$-sized gap
    between the extensions constructed around the singular slices
    in step 1
    and the extensions constructed between the singular slices in step 2.
    Since $f_0$ is $2\eps$-orthogonal to $\CylMeshClosed^m(\zeta \circ \xi)$
    we can fill these gaps by solving lifting problems of the form
    \[
      \begin{tikzcd}
        {\Delta\ord{1} \times \partial\Delta\ord{k}} \ar[r] \ar[d] &
        {\core(\MfldDiagExt^{n - 1})} \ar[d] \\
        {\Delta\ord{1} \times \Delta\ord{k}} \ar[r] \ar[ur, dashed] &
        {\core(\MfldDiag^{n - 1})}
      \end{tikzcd}
    \]
    It is always possible to find a solution since by induction on $n$ the map
    $\core(\MfldDiagExt^{n - 1}) \to \core(\MfldDiag^{n - 1})$ is a trivial fibration.
    We therefore have an extension $f_2 : \strat{E}_2 \to \DeltaTop{k}$
    of $f_1$ defined on the space
    \[
      \unstrat(\strat{E}_2) = \unstrat(\strat{E}_1)
      \cup
      \{
        (e_1, \ldots, e_n, b) \mid
        h_0 - \eps_0 < 
        e_{n - m} <
        h_s + \eps_0
      \}
    \]

    \item
    Finally, it suffices to extend the diagram above the largest singular height
    $h_m$ and below the smallest singular height $h_0$.
    This can be done analogously as in the proof of Lemma~\ref{lem:d-mfld-ext-equiv}.\qedhere
  \end{enumerate}
\end{proof}

\section{Combinatorial Manifold Diagrams}\label{sec:d-comb}

Meshes are designed to bridge between geometric and combinatorial descriptions of framed stratified spaces.
Any extended $n$-manifold diagram is essentially tame, and therefore can be refined by an open $n$-mesh with labels
in $\ord{n}$.
In this section we show how to use this in order to derive a combinatorial description
of the $n$-diagrammatic space of $n$-manifold diagrams $\MfldDiag^n(-)$.
In particular, we define functors $\MeshDiagExt^n(-)$ and $\TrussDiagExt^n(-)$ based on
open $n$-meshes and trusses that present $\MfldDiag^n(-)$ by taking classifying spaces:
\[
  \begin{tikzcd}[row sep = large]
    {\BTrussClosed{n}} \ar[r, "\simeq"] \ar[d, "\TrussDiagExt^n(-)"{description}] & 
    {\BMeshClosed{n}} \ar[r, equal] \ar[d, "\MeshDiagExt^n(-)"{description}] &
    {\BMeshClosed{n}} \ar[r, equal] \ar[d, "\MfldDiagExt^n(-)"{description}] &
    {\BMeshClosed{n}} \ar[d, "\MfldDiag^n(-)"{description}] \\
    {\Cat} \ar[r, hook] &
    {\CatInfty} \ar[r, "\CatToSpace"'] &
    {\Space} \ar[r, equal] &
    {\Space}
  \end{tikzcd}
\]

In \S\ref{sec:d-comb-meshes} we characterise those open $n$-mesh bundles
that refine $n$-manifold diagram bundles and show that they form an $\infty$-subcategory
$\BMeshDiag{n}$ of $\BMeshOpen{n}{\ord{n}}$.
Via orthogonal configurations between these diagrammatic $n$-mesh bundles and
closed $n$-mesh bundles, we then introduce the functor $\MeshDiagExt^n(-)$
in \S\ref{sec:d-comb-meshed-diag} classifying extended meshed $n$-manifold diagrams.
For every closed $n$-mesh $\strat{X}$ the space $\MfldDiagExt^n(\strat{X})$
is the classifying space $\CatToSpace(\MeshDiagExt^n(\strat{X}))$.
Finally, in \S\ref{sec:d-comb-truss} we use the equivalence $\BConfOrtho{n} \simeq \Tw(\BTrussOpen{n})$
between orthogonal configurations of $n$-meshes and the twisted arrow category
from \S\ref{sec:fct-conf-ortho} to define the functor $\TrussDiagExt^n(-)$.

\subsection{Diagrammatic Meshes}\label{sec:d-comb-meshes}

\begin{definition}
  Suppose that $\strat{M}$ is an open $n$-mesh with labels in $\ord{n}$.
  We say that $\strat{M}$ is a \defn{diagrammatic $n$-mesh} when there exists
  an extended $n$-manifold diagram $\strat{E}$ and a refinement map of labelled $n$-framed stratified
  spaces $r : \strat{M} \to \strat{E}$.
\end{definition}

\begin{definition}
  Suppose that $f : \strat{M} \to \strat{B}$ is an open $n$-mesh bundle
  with labels in $\ord{n}$.  
  We say that $f$ is a \defn{diagrammatic $n$-mesh bundle} when there
  exists an $n$-manifold diagram bundle $p : \strat{E} \to \unstrat(\strat{B})$ 
  with a refinement map of labelled $n$-framed stratified bundles
  \[
    \begin{tikzcd}
      {\strat{M}} \ar[r] \ar[d, "f"'] &
      {\strat{E}} \ar[d, "p"] \\
      {\strat{B}} \ar[r] &
      {\unstrat(\strat{B})}
    \end{tikzcd}
  \]
  The diagrammatic $n$-mesh bundles form a simplicial subset
  $\BMeshDiag{n}$ of $\BMeshOpenL{n}{\ord{n}}$.
\end{definition}

\begin{definition}
  Let $\cat{C}$ be a quasicategory and
  $f : \strat{S} \pto \strat{T}$
  a bordism of open $n$-meshes with labels in $\cat{C}$.
  We say that $f$ is \defn{submersive} when, for every submesh  
  $i : \strat{S}' \hookrightarrow \strat{S}$
  whose normal form $\nf(\strat{S}')$ is an atom,
  the active part of the active/inert factorisation of $f \circ i$ normalises to an equivalence:
  \[
    \begin{tikzcd}
      {\nf(\strat{S}')} \ar[d, dashed, "\simeq"'] &
      {\strat{S}'} \ar[r, mid vert, hook, "i"] \ar[d, mid vert] \ar[l, mid vert, two heads] &
      {\strat{S}} \ar[d, mid vert, "f"] \\
      {\nf(\strat{T}')} &
      {\strat{T}'} \ar[r, mid vert, hook] \ar[l, mid vert, two heads] &
      {\strat{T}}
    \end{tikzcd}
  \]
\end{definition}

\begin{proposition}\label{prop:d-comb-meshes-iff-submersive}
  Let $f : \strat{S} \pto \strat{T}$ be a bordism of
  open $n$-meshes with labels in $\ord{n}$ such that both 
  $\strat{S}$, $\strat{T}$ are diagrammatic $n$-meshes.
  Then the bordism $f$ is in $\BMeshDiag{n}$ if and only if it is submersive.
\end{proposition}
\begin{proof}
  The backwards implication is clear.
  Suppose therefore that $f$ is in $\BMeshDiag{n}$ and let
  $i : \strat{S}' \hookrightarrow \strat{S}$
  such that the normal form $\nf(\strat{S}')$ is an atom.
  We need to show that there exists an equivalence
  \[
    \begin{tikzcd}
      {\nf(\strat{S}')} \ar[d, dashed, "\simeq"'] &
      {\strat{S}'} \ar[r, mid vert, hook, "i"] \ar[d, mid vert, "f'"] \ar[l, mid vert, two heads] &
      {\strat{S}} \ar[d, mid vert, "f"] \\
      {\nf(\strat{T}')} &
      {\strat{T}'} \ar[r, mid vert, hook] \ar[l, mid vert, two heads] &
      {\strat{T}}
    \end{tikzcd}
  \]
  where the right square is the active/inert factorisation of $f \circ i$.
  We rephrase the problem in terms of bundles as follows.
  
  Let $F : \strat{M} \to \DeltaStrat{1}$ be the open $n$-mesh bundle that
  represents the bordism $f$ and analogously let $F' : \strat{M}' \to \DeltaStrat{1}$
  be the open $n$-mesh bundle representing the active bordism $f'$.
  By the definition of $\BMeshDiag{n}$ there exists an $n$-diagrammatic bundle
  $p : \strat{E} \to \DeltaTop{1}$ determined by the dimension labelling
  so that $f$ refines $p$.
  The $n$-diagrammatic bundle $p$ restricts to an $n$-diagrammatic bundle
  $p' : \strat{E}' \to \DeltaTop{1}$ that refines $f'$.
  The submesh $\strat{S}'$ refines a stable framed basic $\strat{U}$.
  All of this data organises into the following diagram of $n$-framed stratified bundles:
  \[
    \begin{tikzcd}[column sep = {4em, between origins}, row sep = {3em, between origins}]
    	& {\strat{U}} && {\strat{E}'} && {\strat{E}} \\
    	{\strat{S}'} && {\strat{M}'} && {\strat{M}} \\
    	& {\DeltaTop{0}} && {\DeltaTop{1}} && {\DeltaTop{1}} \\
    	{\DeltaStrat{0}} && {\DeltaStrat{1}} && {\DeltaStrat{1}}
    	\arrow[from=2-5, to=1-6]
    	\arrow[hook, from=1-4, to=1-6]
    	\arrow[hook, from=1-2, to=1-4]
    	\arrow[from=2-1, to=1-2]
    	\arrow[from=2-3, to=1-4]
    	\arrow[from=2-1, to=4-1]
    	\arrow[from=1-2, to=3-2]
    	\arrow["{p'}"{pos=0.2}, from=1-4, to=3-4]
    	\arrow["p"{pos=0.2}, from=1-6, to=3-6]
    	\arrow[equal, from=3-4, to=3-6]
    	\arrow[equal, from=4-3, to=4-5]
    	\arrow[hook, from=4-1, to=4-3]
    	\arrow[hook, from=3-2, to=3-4]
    	\arrow[from=4-1, to=3-2]
    	\arrow[from=4-3, to=3-4]
    	\arrow[from=4-5, to=3-6]
    	\arrow[crossing over, hook, from=2-3, to=2-5]
    	\arrow[crossing over, hook, from=2-1, to=2-3]
    	\arrow["{F'}"{pos=0.2}, from=2-3, to=4-3, crossing over]
    	\arrow["F"{pos=0.2}, from=2-5, to=4-5, crossing over]
    \end{tikzcd}
  \]

  To construct the equivalence between the normal forms, it suffices to show
  that $F'$ refines a trivial $n$-diagrammatic $n$-mesh bundle.
  Because the framed basic $\strat{U}$ is stable, there exists an $\eps > 0$
  such that the $n$-diagrammatic bundle $p'$ restricts to a trivial $n$-framed
  stratified bundle over $\intCO{0, \eps}$.
  When we restrict the open $n$-mesh bundle $F'$ accordingly, we obtain the diagram
  \[
    \begin{tikzcd}[column sep = {4em, between origins}, row sep = {3em, between origins}]
    	& {\strat{U}} && {\strat{U} \times \intCO{0, \eps}} && {\strat{E}'} \\
    	{\strat{S}'} && {\strat{M}''} && {\strat{M}'} \\
    	& {\DeltaTop{0}} && {\intCO{0, \eps}} && {\DeltaStrat{1}} \\
    	{\DeltaStrat{0}} && {\DeltaStrat{1} \cap \intCO{0, \eps}} && {\DeltaStrat{1}}
    	\arrow[from=2-5, to=1-6]
    	\arrow[hook, from=1-4, to=1-6]
    	\arrow[hook, from=1-2, to=1-4]
    	\arrow[from=2-1, to=1-2]
    	\arrow[from=2-3, to=1-4]
    	\arrow[from=2-1, to=4-1]
    	\arrow[from=1-2, to=3-2]
    	\arrow["{p''}"{pos=0.2}, from=1-4, to=3-4]
    	\arrow["p'"{pos=0.2}, from=1-6, to=3-6]
    	\arrow[hook, from=3-4, to=3-6]
    	\arrow[hook, from=4-3, to=4-5]
    	\arrow[hook, from=4-1, to=4-3]
    	\arrow[hook, from=3-2, to=3-4]
    	\arrow[from=4-1, to=3-2]
    	\arrow[from=4-3, to=3-4]
    	\arrow[from=4-5, to=3-6]
    	\arrow[crossing over, hook, from=2-3, to=2-5]
    	\arrow[crossing over, hook, from=2-1, to=2-3]
    	\arrow[crossing over, "{F'}"{pos=0.2}, from=2-5, to=4-5]
    	\arrow[crossing over, "{F''}"{pos=0.2}, from=2-3, to=4-3]
    \end{tikzcd}
  \]
  In this diagram the top, back and bottom squares of the cube are pullback
  squares. It follows that the front squares are pullback squares as well,
  and therefore $F''$ is an open $n$-mesh bundle.
  Because $F''$ refines the trivial $n$-diagrammatic bundle $p''$,
  it refines a trivial diagrammatic $n$-mesh bundle when equipped with the induced labelling.
  The inclusion $\DeltaStrat{1} \cap \intCO{0, \eps} \hookrightarrow \DeltaStrat{1}$ is a
  stratified homotopy equivalence. Therefore, $F'$ also refines a trivial diagrammatic $n$-mesh bundle.
\end{proof}

\begin{remark}
  The characterisation of maps in $\BMeshDiag{n}$ as submersive mesh bordisms
  depends essentially on the framed basics to be stable.  
\end{remark}

\begin{lemma}\label{lem:d-comb-meshes-submersive-composite}
  Let $\cat{C}$ be a quasicategory.
  Then the submersive bordisms of open $n$-meshes with labels in $\cat{C}$
  are closed under composition in $\BMeshOpenL{n}{\cat{C}}$.
\end{lemma}
\begin{proof}
  Suppose we have two composable mesh bordisms 
  \[
    \begin{tikzcd}
      {\strat{R}} \ar[r, "f", mid vert] &
      {\strat{S}} \ar[r, "g", mid vert] &
      {\strat{T}}
    \end{tikzcd}
  \]
  in $\BMeshOpenL{n}{\cat{C}}$
  such that both $f$ and $g$ are submersive.
  Let $\strat{R}'$ be a submesh of $\strat{R}$ that normalises
  to an atom.
  We then have a diagram
  \[
    \begin{tikzcd}
      {\nf(\strat{R}')} \ar[d, dashed, "\simeq"'] &
      {\strat{R}'} \ar[r, mid vert, hook] \ar[d, mid vert, "f'"] \ar[l, mid vert, two heads] &
      {\strat{R}} \ar[d, mid vert, "f"] \\
      {\nf(\strat{S}')} \ar[d, dashed, "\simeq"'] &
      {\strat{S}'} \ar[r, mid vert, hook] \ar[d, mid vert, "g'"] \ar[l, mid vert, two heads] &
      {\strat{S}} \ar[d, mid vert, "g"] \\
      {\nf(\strat{T}')} &
      {\strat{T}'} \ar[r, mid vert, hook] \ar[l, mid vert, two heads] &
      {\strat{T}}
    \end{tikzcd}
  \]
  where $f'$ and $g'$ are produced by the active/inert factorisation.
  The first equivalence between normal forms exists since $f$ is submersive
  and $\strat{R}'$ normalises to an atom.
  But then $\strat{S}'$ also normalises to an atom and so
  the second equivalence between normal forms exists because $g$ is submersive.
  By functoriality of the active/inert factorisation it then follows
  that the composite $g \circ f$ is submersive.
\end{proof}

\begin{cor}\label{cor:d-comb-meshes-subcategory}
  The simplicial set of diagrammatic $n$-meshes $\BMeshDiag{n}$ is a quasicategory.
  Moreover, the inclusion map 
  $\BMeshDiag{n} \hookrightarrow \BMeshOpenL{n}{\ord{n}}$ is a categorical fibration
  that exhibits $\BMeshDiag{n}$ as an $\infty$-subcategory of $\BMeshOpenL{n}{\ord{n}}$.
\end{cor}
\begin{proof}
  Combining Proposition~\ref{prop:d-comb-meshes-iff-submersive} with Lemma~\ref{lem:d-comb-meshes-submersive-composite}.
\end{proof}

\begin{observation}
  Let $f : \strat{S} \to \strat{T}$ be a bordism of diagrammatic $n$-meshes.
  Suppose that $f$ fits into a diagram of bordisms of open $n$-meshes with labels
  in $\ord{n}$ of the form
  \[
    \begin{tikzcd}
      {} &
      {\strat{R}} \ar[dr, "f_1", mid vert, hook] \\
      {\strat{S}} \ar[rr, mid vert, "f"'] \ar[ur, "f_0", mid vert] &
      {} &
      {\strat{T}}
    \end{tikzcd}
  \]
  where $f_1$ is an inert bordism.
  Then $f_0$ and $f_1$ are both bordisms of diagrammatic $n$-meshes.
  This closure property allows us to lift the active/inert factorisation from $\BMeshOpenL{n}{\ord{n}}$
  to the $\infty$-category $\BMeshDiag{n}$ of diagrammatic $n$-meshes.
\end{observation}
  


\subsection{Meshed Manifold Diagrams}\label{sec:d-comb-meshed-diag}

We can find an alternative construction of the $n$-diagrammatic space of
manifold diagrams by using $n$-diagrammatic meshes.

\begin{definition}
  We write $\meshFC : \BMeshDiag{n} \to \BMeshOpen{n}$ for the composite
  of the inclusion map $\BMeshDiag{n} \hookrightarrow \BMeshOpenL{n}{\ord{n}}$
  followed by the forgetful map $\BMeshOpenL{n}{\ord{n}} \to \BMeshOpen{n}$.
  Consider the diagram of simplicial sets
  \begin{equation}\label{eq:d-mesh-diag}
    \begin{tikzcd}
      {\MeshDiagExt^n} \ar[r] \ar[d] \pullbackcorner &
      {\BConfOrtho{n}} \ar[d] \ar[r] &
      {\BMeshClosed{n}} \\
      {\BMeshDiag{n}} \ar[r, "\meshFC"'] &
      {\BMeshOpen{n}}
    \end{tikzcd}
  \end{equation}
  in which the square is a pullback square.
  Then $\MeshDiagExt^n$ is the simplicial set of \defn{extended meshed $n$-manifold diagrams}.
  The top row of the diagram composes into a map $\MeshDiagExt^n \to \BMeshClosed{n}$.  
\end{definition}

\begin{observation}
  Concretely a $k$-simplex of $\MeshDiagExt^n$ consists of a closed $n$-mesh
  bundle $\xi : \strat{X} \to \DeltaStrat{k}$ together with a diagrammatic $n$-mesh bundle
  $f : \strat{M} \to \DeltaStrat{k}$ such that $\xi$ and $f$ are orthogonal.
  The map $\MeshDiagExt^n$ sends the pair $(\xi, f)$ to $\xi$.
\end{observation}

\begin{lemma}\label{lem:d-comb-meshed-diag-cocartesian}
  The map $\MeshDiagExt^n \to \BMeshClosed{n}$ is a cocartesian fibration.
  Moreover, a map is cocartesian for $\MeshDiagExt^n \to \BMeshClosed{n}$ if and only if it is
  sent to an equivalence by $\MeshDiagExt^n \to \BMeshDiag{n}$.
\end{lemma}
\begin{proof}
  The inclusion map $\BMeshDiag{n} \hookrightarrow \BMeshOpenL{n}{\Nat^\op}$ is a categorical fibration by Corollary~\ref{cor:d-comb-meshes-subcategory} and the forgetful map $\BMeshOpenL{n}{\Nat^\op} \to \BMeshOpen{n}$ is a categorical
  fibration. Therefore, the composite map $\BMeshDiag{n} \to \BMeshOpen{n}$
  is again a categorical fibration.
  The map $\BConfOrtho{n} \to \BMeshClosed{n}$ is a cocartesian fibration by Lemma~\ref{lem:fct-conf-ortho-cocartesian} and a map is cocartesian if and only if $\BConfOrtho{n} \to \BMeshOpen{n}$
  sends it to an equivalence.
  The claim then follows Lemma~\ref{lem:b-cat-cocartesian-top-row}.
\end{proof}

\begin{observation}
  The cocartesian fibration $\MeshDiagExt^n \to \BMeshClosed{n}$ induces a functor
  \[
    \MeshDiagExt^n(-) : \BMeshClosed{n} \to \CatInfty
  \]
  sending each closed $n$-mesh $\strat{X}$ to an $\infty$-category
  $\MeshDiagExt^n(\strat{X})$.
  We can compute this $\infty$-category concretely as a quasicategory by taking
  the pullback of simplicial sets
  \[
    \begin{tikzcd}
      {\MeshDiagExt^n(\strat{X})} \ar[r, hook] \ar[d] \pullbackcorner &
      {\MeshDiagExt^n} \ar[d] \\
      {\Delta\ord{0}} \ar[r, "\strat{X}"', hook] &
      {\BMeshClosed{n}}
    \end{tikzcd}
  \]
  More concretely, a $k$-simplex of $\MeshDiagExt^n(\strat{X})$ consists of an
  diagrammatic $n$-mesh bundle $f : \strat{M} \to \DeltaStrat{k}$ that
  is fibrewise orthogonal to $\strat{X}$.
\end{observation}

\begin{observation}
  There is a canonical map of simplicial sets $\MeshDiagExt^n \to \MfldDiagExt^n$ which sends
  a pair $(\xi, f)$ to the pair $(\xi, p)$ where $p$ is the
  $n$-manifold diagram bundle that is refined by the diagrammatic $n$-mesh bundle $f$.
  Because this map preserves the closed $n$-mesh bundle $\xi$, it commutes with the
  projections to $\BMeshClosed{n}$:
  \[
    \begin{tikzcd}[column sep = small]
      {\MeshDiagExt^n} \ar[rr] \ar[dr] &
      {} &
      {\MfldDiagExt^n} \ar[dl] \\
      {} &
      {\BMeshClosed{n}}
    \end{tikzcd}
  \]
  We therefore have a natural transformation
  $\MeshDiagExt^n(-) \to \MfldDiagExt^n(-)$.
  The component of this natural transformation on a closed $n$-mesh $\strat{X}$
  sends the $\infty$-category $\MeshDiagExt^n(\strat{X})$
  of extended diagrammatic $n$-meshes on $\strat{X}$ to the space
  $\MfldDiagExt^n(\strat{X})$
  of extended $n$-manifold diagrams on $\strat{X}$.
\end{observation}



\begin{lemma}\label{lem:mesh-diag-to-mfld-diag}
  The map of cocartesian fibrations
  \[
    \begin{tikzcd}[column sep = small]
      {\MeshDiagExt^n} \ar[rr] \ar[dr] &
      {} &
      {\MfldDiagExt^n} \ar[dl] \\
      {} &
      {\BMeshClosed{n}}
    \end{tikzcd}
  \]
  exhibits $\MfldDiagExt^n \to \BMeshClosed{n}$ as the free left fibration
  induced by $\MeshDiagExt^n \to \BMeshClosed{n}$.
\end{lemma}
\begin{proof}
  We check that for each closed $n$-mesh $\strat{X}$
  the induced map on fibres
  $\MeshDiagExt^n(\strat{X}) \to \MfldDiagExt^n(\strat{X})$
  is a Kan equivalence.
  The map 
  $\MeshDiagExt^n(\strat{X}) \to \MfldDiagExt^n(\strat{X})$
  canonically extends to a map on the Kan fibrant replacement
  \[
    \begin{tikzcd}
      {\MeshDiagExt^n(\strat{X})} \ar[r] \ar[d, hook, "\simeq"'] &
      {\MfldDiagExt^n(\strat{X})} \\
      {\Ex^\infty(\MeshDiagExt^n(\strat{X}))} \ar[ur] &
      {}
    \end{tikzcd}
  \]
  and it suffices to show that $\Ex^\infty(\MeshDiagExt^n(\strat{X})) \to \MfldDiagExt^n(\strat{X})$
  is a Kan equivalence.
  For this purpose we solve lifting problems of the form
  \begin{equation}\label{eq:mesh-diag-to-mfld-diag:lift}
    \begin{tikzcd}
      {\partial \Delta\ord{k}} \ar[r] \ar[d] &
      {\Ex^\infty(\MeshDiagExt^n(\strat{X}))} \ar[d] \\
      {\Delta\ord{k}} \ar[r] \ar[ur, dashed, ""{name=0, anchor=center, inner sep=0}] &
      {\MfldDiagExt^n(\strat{X})}
      \arrow["\simeq"{description, pos = 0.7}, draw=none, from=0, to=2-2]
    \end{tikzcd}
  \end{equation}
  where the lower right square needs to commute up to equivalence.

  When $k = 0$ we have an $n$-manifold diagram $\strat{E}$ which is orthogonal to $\strat{X}$.
  Then by Observation~\ref{obs:fct-conf-ortho-coarsest} the coarsest refining mesh $\strat{M}$ of $\strat{E}$ remains orthogonal to $\strat{X}$.
  Therefore, the pair $(\strat{X}, \strat{M})$ is a $0$-simplex of $\Ex^\infty(\MeshDiag^n(\strat{X}))$ which solves our lifting problem.

  When $k > 0$ we have an $n$-manifold diagram bundle $p : \strat{E} \to \DeltaTop{k}$
  that is fibrewise orthogonal to $\strat{X}$.
  We also have an $i \geq 0$ and an $n$-manifold diagram bundle $f' : \strat{M}' \to \partial\StratReal{\sd^i \ord{k}}$ that is fibrewise orthogonal to $\strat{X}$.
  We write $p' : \strat{E}' \to \partial\DeltaTop{k}$ for the restriction of $p$
  over the boundary. Then $p'$ is refined by $f'$.

  Since the lower right corner of the lifting problem~(\ref{eq:mesh-diag-to-mfld-diag:lift})
  only needs to commute up to equivalence, we are free to replace $p$ up to equivalence
  of $k$-simplices in $\MfldDiagExt^n(\strat{X})$ relative to the boundary.
  We use this affordance to assume without loss of generality that there exists
  an open collar neighbourhood $\partial\DeltaTop{k} \times (\DeltaTop{1} \setminus \{ 1 \}) \hookrightarrow \DeltaTop{k}$ such that
  \[
    \begin{tikzcd}
      {\strat{E}'} \ar[d, "p'"'] &
      {\strat{E}' \times (\DeltaTop{1} \setminus \{ 1 \})} \ar[r, hook] \ar[l] \ar[d] \pullbackcorner \pullbackdl &
      {\strat{E}} \ar[d, "p"] \\
      {\partial \DeltaTop{k}} &
      {\partial \DeltaTop{k} \times (\DeltaTop{1} \setminus \{ 1 \})} \ar[r, hook] \ar[l] &
      {\DeltaTop{k}}
    \end{tikzcd}
  \]
  

  By Proposition~\ref{prop:coarsest-mesh-refinement-open-pl} there is a coarsest open $n$-mesh bundle
  $g_0 : \strat{N}_0 \to \strat{B}_0$
  which refines $p$. By our setup the restriction of $p$ to the collar neighbourhood
  is refined by the open $n$-mesh bundle obtained from thickening $f'$ to the collar:  
  \[
    \begin{tikzcd}
      {\strat{M}'} \ar[d, "f'"'] &
      {\strat{M}' \times (\DeltaStrat{1} \setminus \{ 1 \})} \ar[d] \ar[r, dashed] \ar[l] \pullbackdl &
      {\strat{E}' \times (\DeltaTop{1} \setminus \{ 1 \})} \ar[r, hook] \ar[d] \pullbackcorner &
      {\strat{E}} \ar[d, "p"] \\
      {\partial \StratReal{\sd^i \ord{k}}} &
      {\partial \StratReal{\sd^i \ord{k}} \times (\DeltaStrat{1} \setminus \{ 1 \})} \ar[r, dashed] \ar[l] &
      {\partial \DeltaTop{k} \times (\DeltaTop{1} \setminus \{ 1 \})} \ar[r, hook] &
      {\DeltaTop{k}}
    \end{tikzcd}
  \]
  Since by Observation~\ref{obs:coarsest-mesh-refinement-local} the coarsest open $n$-mesh bundles that refine tame stratifications are determined locally, the restriction of the open $n$-mesh bundle $g_0 : \strat{N}_0 \to \strat{B}_0$ to the collar must be coarser than the thickening of $f'$:
  \[
    \begin{tikzcd}
      {\strat{M}'} \ar[d, "f'"'] &
      {\strat{M}' \times (\DeltaStrat{1} \setminus \{ 1 \})} \ar[d] \ar[r, dashed] \ar[l] \pullbackdl &
      {\strat{N}'_0 \times (\DeltaStrat{1} \setminus \{ 1 \})} \ar[r, hook] \ar[d] \pullbackcorner &
      {\strat{N}_0} \ar[d, "g_0"] \\
      {\partial \StratReal{\sd^i \ord{k}}} &
      {\partial \StratReal{\sd^i \ord{k}} \times (\DeltaStrat{1} \setminus \{ 1 \})} \ar[r, dashed] \ar[l] &
      {\partial \strat{B}_0 \times (\DeltaStrat{1} \setminus \{ 1 \})} \ar[r, hook] &
      {\strat{B}_0}
    \end{tikzcd}
  \]
  We can therefore further refine $g_0 : \strat{N}_0 \to \strat{B}_0$ to an open $n$-mesh bundle
  $g_1 : \strat{N}_1 \to \strat{B}_1$
  which restricts to $f' : \strat{M}' \to \partial \StratReal{\sd^i \ord{k}}$
  on the boundary $\partial \strat{B}_1 = \partial \StratReal{\sd^i \ord{k}}$.
  Then there exists a $j \geq i$ and a refinement map $s : \StratReal{\sd^j \ord{k}} \to \strat{B}_1$ that is the identity on the boundary.
  We then take the pullbacks
  \[
    \begin{tikzcd}
      {\strat{M}} \ar[r] \ar[d, "f"'] \pullbackcorner &
      {\strat{N}_1} \ar[d, "g_1"] \\
      {\StratReal{\sd^j \ord{k}}} \ar[r, "s"'] &
      {\strat{B}_1}
    \end{tikzcd}
    \qquad
    \begin{tikzcd}
      {\strat{D}} \ar[r] \ar[d, "q"'] \pullbackcorner &
      {\strat{E}} \ar[d, "p"] \\
      {\DeltaTop{k}} \ar[r, "\unstrat(s)"'] &
      {\DeltaTop{k}}
    \end{tikzcd}
  \]
  The $n$-manifold diagram bundles $p$ and $q$ are equivalent relative to the boundary.
  Therefore, the open $n$-mesh bundle $f : \strat{M} \to \StratReal{\sd^j \ord{k}}$
  is a solution to the lifting problem.
\end{proof}

\begin{example}
  We illustrate the technique used in the proof of Lemma~\ref{lem:mesh-diag-to-mfld-diag}
  on an example. We begin with the $2$-manifold diagram bundle $p : \strat{E} \to \DeltaTop{1}$ over the unstratified
  interval that is fibrewise orthogonal to $\StratIntLR^2$, presenting the isotopy which performs two braids:
  \[
    \begin{tikzpicture}[scale = 0.5, baseline=(current bounding box.center)]
      \fill[mesh-background] (0, 0) rectangle (6, 6);

      \draw[mesh-stratum-dual] (1, 0) -- (1, 6);
      \draw[mesh-stratum-dual] (5, 0) -- (5, 6);

      \draw[mesh-stratum] (2, 0) -- (2, 1) .. controls +(0, 1) and +(0, -1) .. (4, 3);
      \draw[mesh-stratum-over] (4, 0) -- (4, 1) .. controls +(0, 1) and +(0, -1) .. (2, 3);
      \draw[mesh-stratum] (4, 0) -- (4, 1) .. controls +(0, 1) and +(0, -1) .. (2, 3);

      \draw[mesh-stratum] (2, 3) .. controls +(0, 1) and +(0, -1) .. (4, 5) -- (4, 6);
      \draw[mesh-stratum-over] (4, 3) .. controls +(0, 1) and +(0, -1) .. (2, 5) -- (2, 6);
      \draw[mesh-stratum] (4, 3) .. controls +(0, 1) and +(0, -1) .. (2, 5) -- (2, 6);
    \end{tikzpicture}
    \quad
    \longrightarrow
    \quad
    \begin{tikzpicture}[scale = 0.5, baseline=(current bounding box.center)]
      \draw[mesh-stratum] (3, 0) -- (3, 6);
    \end{tikzpicture}
  \]
  We can then find a diagrammatic $2$-mesh bundle $f : \strat{M} \to \strat{B}$ that refines $p$:
  \[
    \begin{tikzpicture}[scale = 0.5, baseline=(current bounding box.center)]
      \fill[mesh-background] (0, 0) rectangle (6, 6);

      \draw[mesh-stratum-dual] (1, 0) -- (1, 6);
      \draw[mesh-stratum-dual] (5, 0) -- (5, 6);

      \draw[mesh-stratum] (2, 0) -- (2, 1) .. controls +(0, 1) and +(0, -1) .. (4, 3);
      \node[mesh-vertex-over] at (3, 2) {};
      \node[mesh-vertex] at (3, 2) {};
      \draw[mesh-stratum] (0, 2) -- +(6, 0);
      \draw[mesh-stratum] (4, 0) -- (4, 1) .. controls +(0, 1) and +(0, -1) .. (2, 3);

      \draw[mesh-stratum] (2, 3) .. controls +(0, 1) and +(0, -1) .. (4, 5) -- (4, 6);
      \node[mesh-vertex-over] at (3, 4) {};
      \node[mesh-vertex] at (3, 4) {};
      \draw[mesh-stratum] (0, 4) -- +(6, 0);
      \draw[mesh-stratum] (4, 3) .. controls +(0, 1) and +(0, -1) .. (2, 5) -- (2, 6);

      \draw[mesh-stratum] (0, 0) -- +(6, 0);
      \node[mesh-vertex] at (2, 0) {};
      \node[mesh-vertex] at (4, 0) {};

      \draw[mesh-stratum] (0, 6) -- +(6, 0);
      \node[mesh-vertex] at (2, 6) {};
      \node[mesh-vertex] at (4, 6) {};
    \end{tikzpicture}
    \quad
    \longrightarrow
    \quad
    \begin{tikzpicture}[scale = 0.5, baseline=(current bounding box.center)]
      \draw[mesh-stratum] (3, 0) -- (3, 6);
      \node[mesh-vertex] at (3, 0) {};
      \node[mesh-vertex] at (3, 2) {};
      \node[mesh-vertex] at (3, 4) {};
      \node[mesh-vertex] at (3, 6) {};
    \end{tikzpicture}
  \]
  Because $\strat{B}$ is a refinement of $\DeltaTop{1}$, there exists a
  $j \geq 0$ and a refinement $\StratReal{\sd^j \ord{1}} \to \strat{B}$.
  In our case $j = 2$ is sufficient.
  Pulling back $f$ along this refinement then yields a diagrammatic $2$-mesh
  bundle over $\StratReal{\sd^2\ord{1}}$, and therefore a $1$-simplex
  of $\Ex^\infty(\MfldDiagExt^2(\StratIntLR^2))$:
  \[
    \begin{tikzpicture}[scale = 0.5, baseline=(current bounding box.center)]
      \fill[mesh-background] (0, 0) rectangle (6, 6);

      \draw[mesh-stratum-dual] (1, 0) -- (1, 6);
      \draw[mesh-stratum-dual] (5, 0) -- (5, 6);

      \draw[mesh-stratum] (2, 0) -- (2, 0.5) .. controls +(0, 1) and +(0, -1) .. (4, 2.5) -- (4, 3.5);
      \node[mesh-vertex-over] at (3, 1.5) {};
      \node[mesh-vertex] at (3, 1.5) {};
      \draw[mesh-stratum] (0, 1.5) -- +(6, 0);
      \draw[mesh-stratum] (4, 0) -- (4, 0.5) .. controls +(0, 1) and +(0, -1) .. (2, 2.5) -- (2, 3.5);

      \draw[mesh-stratum] (2, 3.5) .. controls +(0, 1) and +(0, -1) .. (4, 5.5) -- (4, 6);
      \node[mesh-vertex-over] at (3, 4.5) {};
      \node[mesh-vertex] at (3, 4.5) {};
      \draw[mesh-stratum] (0, 4.5) -- +(6, 0);
      \draw[mesh-stratum] (4, 3.5) .. controls +(0, 1) and +(0, -1) .. (2, 5.5) -- (2, 6);

      \draw[mesh-stratum] (0, 0) -- +(6, 0);
      \node[mesh-vertex] at (2, 0) {};
      \node[mesh-vertex] at (4, 0) {};

      \draw[mesh-stratum] (0, 3) -- +(6, 0);
      \node[mesh-vertex] at (2, 3) {};
      \node[mesh-vertex] at (4, 3) {};

      \draw[mesh-stratum] (0, 6) -- +(6, 0);
      \node[mesh-vertex] at (2, 6) {};
      \node[mesh-vertex] at (4, 6) {};
    \end{tikzpicture}
    \quad
    \longrightarrow
    \quad
    \begin{tikzpicture}[scale = 0.5, baseline=(current bounding box.center)]
      \draw[mesh-stratum] (3, 0) -- (3, 6);
      \node[mesh-vertex] at (3, 0) {};
      \node[mesh-vertex] at (3, 1.5) {};
      \node[mesh-vertex] at (3, 3) {};
      \node[mesh-vertex] at (3, 4.5) {};
      \node[mesh-vertex] at (3, 6) {};
    \end{tikzpicture}
  \]
\end{example}

\subsection{Combinatorial Manifold Diagrams}\label{sec:d-comb-truss}

In the previous two sections we have seen how to express the $n$-diagrammatic space
of manifold diagrams purely in terms of diagrammatic $n$-meshes.
Now we use the equivalence between open meshes and open trusses to derive
a combinatorial description.
In particular we construct a functor $\TrussDiagExt^n(-) : \BTrussClosed{n} \to \Cat$
which fits into the diagram of $\infty$-categories
\[
  \begin{tikzcd}[row sep = large]
    {\BTrussClosed{n}} \ar[r, "\simeq"] \ar[d, "\TrussDiagExt^n(-)"{description}] & 
    {\BMeshClosed{n}} \ar[r, equal] \ar[d, "\MeshDiagExt^n(-)"{description}] &
    {\BMeshClosed{n}} \ar[r, equal] \ar[d, "\MfldDiagExt^n(-)"{description}] &
    {\BMeshClosed{n}} \ar[d, "\MfldDiag^n(-)"{description}] \\
    {\Cat} \ar[r, hook] &
    {\CatInfty} \ar[r, "\CatToSpace"'] &
    {\Space} \ar[r, equal] &
    {\Space}
  \end{tikzcd}
\]

We first need to identify the open $n$-trusses with labels in $\ord{n}$
that are the analogue of diagrammatic $n$-meshes. For that we can use the
geometric realisation functor.



\begin{definition}\label{def:d-comb-truss}
  We define the $1$-category $\BTrussDiag{n}$ of \defn{diagrammatic $n$-trusses}
  as the subcategory of $\BTrussOpenL{n}{\ord{n}}$ constructed via the
  pullback of simplicial sets
  \[
    \begin{tikzcd}
      {\BTrussDiag{n}} \ar[r, "\simeq"] \ar[d, hook] \pullbackcorner &
      {\BMeshDiag{n}} \ar[d, hook] \\
      {\BTrussOpenL{n}{\ord{n}}} \ar[r, "\simeq"'] &
      {\BMeshOpenL{n}{\ord{n}}}
    \end{tikzcd}
  \]  
  where $\BTrussOpenL{n}{\ord{n}} \to \BMeshOpenL{n}{\ord{n}}$ is the geometric realisation functor.
\end{definition}

\begin{observation}
  We use Proposition~\ref{prop:b-poly-flat-cat} to see that the inclusion map
  $\BMeshDiag{n} \hookrightarrow \BMeshOpenL{n}{\ord{n}}$ is a categorical fibration.  
  Therefore, the pullback square of simplicial sets in Definition~\ref{def:d-comb-truss} is
  also a homotopy pullback square and so the induced map
  $\BTrussDiag{n} \to \BMeshDiag{n}$ is a categorical equivalence.
\end{observation}

\begin{definition}\label{def:d-truss-diag}
  We write $\trussFC : \BTrussDiag{n} \to \BTrussOpen{n}$ for the composite
  of the inclusion map $\BTrussDiag{n} \hookrightarrow \BTrussOpenL{n}{\ord{n}}$
  followed by the forgetful map $\BTrussOpenL{n}{\ord{n}} \to \BTrussOpen{n}$.
  Consider the diagram of simplicial sets
  \begin{equation}\label{eq:d-truss-diag}
    \begin{tikzcd}
      {\TrussDiagExt^n} \ar[r] \ar[d] \pullbackcorner &
      {\Tw(\BTrussOpen{n})} \ar[d] \ar[r] &
      {\BTrussClosed{n}} \\
      {\BTrussDiag{n}} \ar[r, "\trussFC"'] &
      {\BTrussOpen{n}}
    \end{tikzcd}
  \end{equation}
  where the square is a pullback square.
  Then $\TrussDiagExt^n$ is the simplicial set of
  \defn{extended combinatorial $n$-manifold diagrams}.
  The top row of the diagram composes into a projection map
  $\TrussDiagExt^n \to \BTrussClosed{n}$.
\end{definition}

\begin{observation}
  The simplicial set $\TrussDiagExt^n$ is a $1$-category whose objects consist
  of a diagrammatic $n$-truss $S \in \BTrussDiag{n}$, a closed $n$-truss
  $X$ and a bordism of unlabelled open $n$-trusses $X^\dagger \pto S$.
  A map in the $1$-category $\TrussDiagExt^n$
  consists of a bordism $S_0 \pto S_1$ of diagrammatic $n$-trusses
  and a bordism $X_0 \pto X_1$ of closed $n$-trusses
  which together fit into a diagram of unlabelled open $n$-trusses
  \[
    \begin{tikzcd}
      {X_0^\dagger} \ar[d, mid vert] &
      {X_1^\dagger} \ar[d, mid vert] \ar[l, mid vert] \\
      {S_0} \ar[r, mid vert] &
      {S_1}
    \end{tikzcd}
  \]
\end{observation}

\begin{observation}\label{obs:d-comb-truss-equiv}
  Using the equivalence $\Tw(\BTrussOpen{n}) \simeq \BConfOrtho{n}$ from \S\ref{sec:fct-conf-ortho}
  as well as the geometric realisation functors for open and closed trusses,
  the diagrams~(\ref{eq:d-truss-diag}) and~(\ref{eq:d-mesh-diag}) which define $\TrussDiagExt^n$
  and $\MeshDiagExt^n$ together fit within the diagram
  \[
    \begin{tikzcd}[column sep = {3.5em, between origins}, row sep = {3.5em, between origins}]
    	& {\MeshDiagExt^n} && {\BConfOrtho{n}} && {\BMeshClosed{n}} \\
    	{\TrussDiagExt^n} && {\Tw(\BTrussOpen{n})} && {\BTrussClosed{n}} \\
    	& {\BMeshDiag{n}} && {\BMeshOpen{n}} \\
    	{\BTrussDiag{n}} && {\BTrussOpen{n}}
    	\arrow[from=1-2, to=1-4]
    	\arrow["\simeq"{description}, dashed, from=2-1, to=1-2]
    	\arrow["\simeq"{description}, from=2-3, to=1-4]
    	\arrow[from=2-1, to=4-1]
    	\arrow[from=1-2, to=3-2]
    	\arrow["\simeq"{description}, from=2-5, to=1-6]
    	\arrow[from=1-4, to=3-4]
    	\arrow["\simeq"{description}, from=4-1, to=3-2]
    	\arrow["\simeq"{description}, from=4-3, to=3-4]
    	\arrow[from=3-2, to=3-4]
    	\arrow[from=4-1, to=4-3]
    	\arrow[from=1-4, to=1-6]
    	\arrow[crossing over, from=2-3, to=4-3]
    	\arrow[crossing over, from=2-1, to=2-3]
    	\arrow[crossing over, from=2-3, to=2-5]
    \end{tikzcd}
  \]
  The projection map $\Tw(\BTrussOpen{n}) \to \BTrussOpen{n}$ 
  from the twisted arrow category
  is a categorical fibration and so the pullback
  square in the front of the cube is a homotopy pullback square.
  Therefore, the induced map between the pullbacks
  $\MeshDiagExt^n \to \TrussDiagExt^n$
  is a categorical equivalence.
\end{observation}

\begin{observation}\label{obs:d-comb-truss-1-cat}
  The projection maps onto the closed mesh and truss form a diagram
  \[
    \begin{tikzcd}
      {\MeshDiagExt^n} \ar[r, "\simeq"] \ar[d] &
      {\TrussDiagExt^n} \ar[d] \\
      {\BMeshClosed{n}} \ar[r, "\simeq"'] &
      {\BTrussClosed{n}}
    \end{tikzcd}
  \]
  and so it follows that the categorical fibration $\TrussDiagExt^n \to \BTrussClosed{n}$
  is also cocartesian.
  Because both $\TrussDiagExt^n$ and $\BTrussClosed{n}$ are $1$-categories,
  the map therefore represents a functor
  \[ 
    \TrussDiagExt^n(-) : \BTrussClosed{n} \longrightarrow \Cat.
  \]
  For any closed $n$-truss $X$ we can compute the $1$-category
  $\TrussDiagExt^n(X)$ by taking the pullback
  of simplicial sets
  \[
    \begin{tikzcd}
      {\TrussDiagExt^n(X)} \ar[r, hook] \ar[d] \pullbackcorner &
      {\TrussDiagExt^n} \ar[d] \\
      {\Delta\ord{0}} \ar[r, hook, "X"'] &
      {\BTrussClosed{n}}
    \end{tikzcd}
  \]
  An object in $\TrussDiagExt^n(X)$ consists of a diagrammatic $n$-truss
  $S$ together with a bordism of unlabelled open $n$-trusses $X^\dagger \pto S$.
  A map in $\TrussDiagExt^n(X)$ is a bordism of diagrammatic $n$-trusses
  $S \pto T$ that fits into a diagram of unlabelled open $n$-trusses
  \[
    \begin{tikzcd}[column sep = small]
      {} &
      {X^\dagger} \ar[dl, mid vert] \ar[dr, mid vert] \\
      {S} \ar[rr, mid vert] &
      {} &
      {T} 
    \end{tikzcd}
  \]
\end{observation}

\subsection{Active Bordisms and Non-Extended Diagrams}

An $n$-manifold diagram on a closed $n$-mesh $\strat{X}$
is completely determined
by its intersection with $\strat{X}$. To facilitate the equivalence
with meshed manifold diagrams we
introduced extended manifold diagrams
that spill over the closed $n$-mesh to fill all of $\R^n$.
Using active bordisms we can find non-extended
variants of $\MeshDiagExt^n(-)$ and $\TrussDiagExt^n(-)$.





\begin{construction}
  Let $X$ be a closed $n$-truss.
  We let $\TrussDiag^n(X)$ denote the full subcategory of $\TrussDiagExt^n(X)$
  whose objects consist of a diagrammatic $n$-truss $S$ together
  with an active bordism $T^\dagger \pto S$ of open unlabelled $n$-trusses.
\end{construction}

\begin{lemma}\label{lem:d-truss-diag-canonicalise}
  Let $X$ be a closed $n$-truss.
  The inclusion $i : \TrussDiag^n(X) \hookrightarrow \TrussDiagExt^n(X)$
  admits a retraction $r$ together with a natural transformation
  $\eps : i \circ r \to \id$.
\end{lemma}
\begin{proof}
  An object of $\TrussDiagExt^n(X)$ consists of a diagrammatic $n$-truss
  $S$ together with a bordism of open unlabelled $n$-trusses $f : X^\dagger \pto S$.
  The bordism $f$ then factors into an active bordism $f_0$ followed by an inert bordism $f_1$ as follows:
  \[
    \begin{tikzcd}
      {} &
      {X^\dagger} \ar[dl, "f_0"', mid vert, hook] \ar[dr, "f", mid vert] \\
      {S'} \ar[rr, "f_1"', mid vert] &
      {} &
      {S}
    \end{tikzcd}
  \]
  Since $f_1 : S' \pto S$ is inert, there is a canonical labelling of $S'$
  such that $f_1$ induces a map of diagrammatic $n$-trusses
  $S' \pto S$.
  We define $r(S, f) := (S', f_0)$ and let $f_1$ be the component
  of the natural transformation $\eps : i \circ r \to \id$ on $(S, f)$. 

  A map in $\TrussDiagExt^n(X)$ consists of a bordism of diagrammatic $n$-trusses $\alpha : S \pto R$ which fits into a diagram of unlabelled open $n$-trusses
  \[
    \begin{tikzcd}[column sep = small]
      {} &
      {X^\dagger} \ar[dl, mid vert, "f"'] \ar[dr, mid vert, "g"] \\
      {S} \ar[rr, mid vert, "\alpha"'] &
      {} &
      {R} 
    \end{tikzcd}
  \]
  When we factor $f$ and $g$ into inert bordisms followed by active bordisms,
  functoriality of the factorisation yields a diagram of unlabelled open $n$-trusses
  \[
    \begin{tikzcd}
    	{S'} &&& {R'} \\
    	& {X^\dagger} & {X^\dagger} \\
    	S &&& R
    	\arrow["{f_1}"', from=1-1, to=3-1, mid vert, hook]
    	\arrow["{f_0}"', from=2-2, to=1-1, mid vert, hook]
    	\arrow["f", from=2-2, to=3-1, mid vert]
    	\arrow["{g_0}", from=2-3, to=1-4, mid vert]
    	\arrow["g"', from=2-3, to=3-4, mid vert]
    	\arrow["{g_1}", from=1-4, to=3-4, mid vert]
    	\arrow[from=1-1, to=1-4, "\alpha'", mid vert]
    	\arrow[from=3-1, to=3-4, "\alpha"', mid vert]
    	\arrow[equal, from=2-2, to=2-3]
    \end{tikzcd}
  \]
  The labels of the bordism of diagrammatic $n$-trusses $\alpha$ then
  restrict to labels for $\alpha'$, making $\alpha'$ itself into a bordism
  of diagrammatic $n$-trusses.
  This is the value of $r$ on the map in $\TrussDiagExt^n(X)$,
  and the inert maps $f_1, g_1$ together form the component of $\eta$ on this map.
  Finally, functoriality of $r$ follows from functoriality of the active/inert factorisation.
\end{proof}

\begin{construction}
  Suppose that $\strat{X}$ is a closed $n$-mesh.
  By Observation~\ref{obs:d-comb-truss-equiv} we have an equivalence
  $\MeshDiagExt^n(\strat{X}) \simeq \TrussDiagExt^n(\TrussClosedNerve(\strat{X}))$.  
  We can then take the pullback of simplicial sets along this equivalence
  to define the meshed version $\MeshDiag^n(\strat{X})$ of $\TrussDiag^n(\strat{X})$:
  \[
    \begin{tikzcd}
      {\MeshDiag^n(\strat{X})} \ar[r, hook, "\simeq"] \ar[d, "\simeq"'] \pullbackcorner &
      {\MeshDiagExt^n(\strat{X})} \ar[d, "\simeq"] \\
      {\TrussDiag^n(\TrussClosedNerve(\strat{X}))} \ar[r, hook, "\simeq"'] &
      {\TrussDiagExt^n(\TrussClosedNerve(\strat{X}))}
    \end{tikzcd}
  \]
  Because the inclusion map $\TrussDiag^n(\TrussClosedNerve(\strat{X})) \hookrightarrow \TrussDiagExt^n(\TrussClosedNerve(\strat{X}))$
  is both a categorical fibration and categorical equivalence,
  it follows that $\MeshDiag^n(\strat{X})$ is a quasicategory
  and equivalent to $\MeshDiagExt^n(\strat{X})$.
\end{construction}

\begin{cor}\label{cor:d-mesh-diag-canonicalise}
  Let $\strat{X}$ be a closed $n$-mesh.
  The inclusion $i : \MeshDiag^n(X) \hookrightarrow \MeshDiagExt^n(X)$
  admits a retraction $r$ together with a natural transformation
  $\eps : i \circ r \to \id$.
\end{cor}

\section{Manifold Diagrams with Labels}\label{sec:d-label}

The $n$-diagrammatic space $\MfldDiag^n(-)$ is freely generated by one generator
for each framed isomorphism class of framed basics.
In this section we explore how to freely generate $n$-diagrammatic spaces,
and therefore $(\infty, n)$-categories,
by attaching labels to manifold diagrams.

\subsection{Diagrammatic Computads}

For manifold diagrams to represent pasting diagrams for higher categories,
we need a way to attach labels to a manifold diagram.
Seeing a manifold diagram $\strat{M}$ as a stratified space, we could ask
for a functor $\Exit(\strat{M}) \to \cat{C}$ into some $\infty$-category
$\cat{C}$, but that approach would not enforce compatibility of the labelling
with the local shape of the diagram;
nothing would stop us, e.g., from assigning a label meant to represent a $1$-morphism
to a surface in a $2$-manifold diagram.
We therefore take a different approach, by making use of our characterisation
of manifold diagrams via diagrammatic meshes.

A label schema should assign to every manifold diagram a collection of ways
in that the diagram can be labelled.
To embed our theory into the wider world of homotopy coherent mathematics
and with a view towards the comparison with $n$-diagrammatic spaces,
this collection of admissible labellings should be a space.
An isotopy of manifold diagrams should act on its labelling;
this is achieved by space-valued presheaves $\Sigma : \BMeshDiag{n, \op} \to \Space$.
The space of labellings should be independent of the mesh that we used to
refine the manifold diagram and invariant under isotopy of manifold diagrams;
the presheaf $\Sigma$ therefore should send active maps to equivalences.
Finally, a labelling should be local in that it is completely determined by assigning labels to the individual parts that make up a manifold diagram.
This requirement translates into the presheaf $\Sigma$ being a sheaf.

\begin{definition}\label{def:d-computad}
  A presheaf $\Sigma : \BMeshDiag{n, \op} \to \Space$ is a
  \defn{diagrammatic $n$-computad}
  when it satisfies the following conditions:
  \begin{enumerate}
    \item \textbf{Isotopy}: Every active bordism of diagrammatic $n$-meshes
    $\strat{M} \pto \strat{N}$
    induces an equivalence of spaces    
    $\Sigma(\strat{N}) \to \Sigma(\strat{M})$.
    \item \textbf{Sheaf:} For every diagrammatic $n$-mesh $\strat{M}$,
    the atomic covering diagram
    \[ U : \stratPos{\strat{M}}^{\op, \triangleright} \longrightarrow \BMeshDiag{n} \]
    induces a limit diagram of spaces
    \[ \Sigma(U(-)) : \stratPos{\strat{M}}^\triangleleft \longrightarrow \Space. \]
  \end{enumerate}
  We denote by $\Cptd{n}$ the reflective $\infty$-subcategory of $\PSh(\BMeshDiag{n})$
  whose objects are the diagrammatic $n$-computads.
\end{definition}

\begin{observation}
  Via the equivalence $\BMeshDiag{n} \simeq \BTrussDiag{n}$ between diagrammatic
  $n$-meshes and diagrammatic $n$-trusses we can equivalently define the
  $\infty$-category of diagrammatic $n$-computads $\Cptd{n}$ as a reflective
  $\infty$-subcategory of $\PSh(\BTrussDiag{n})$:
  \[
    \begin{tikzcd}[column sep = small]
      {\PSh(\BMeshDiag{n})} \ar[rr, "\simeq"] \ar[dr, hook] &
      {} &
      {\PSh(\BTrussDiag{n})} \ar[dl, hook'] \\
      {} &
      {\Cptd{n}}
    \end{tikzcd}
  \]    
  Using $\BTrussDiag{n}$ over $\BMeshDiag{n}$ is preferable in applications
  where we want to produce an $n$-diagrammatic computad from combinatorial data,
  such as a \textsc{homotopy.io} signature.
  Our aim for now is to show how $n$-diagrammatic computads can be used to
  label manifold diagrams, for which meshes are more immediately useful.
\end{observation}

\begin{example}
  Given the generating data for plain, braided or symmetric monoidal categories,
  we can construct a diagrammatic computad in dimension $2$, $3$ or $4$, respectively,
  so that the labelled manifold diagrams are labelled string diagrams.
  More generally, each signature in the proof assistant \textsc{homotopy.io} with generators
  up to dimension $n$ canonically defines a diagrammatic $n$-computad.
\end{example}

\begin{observation}
  Let us denote by $\Act$ the wide subcategory of $\BMeshDiag{n}$ consisting
  of the active bordisms. The universal property of localisation of $\infty$-categories
  guarantees that we have an equivalence between the $\infty$-categories
  \[
    \PSh(\BMeshDiag{n})[\Act^{-1}] \simeq
    \PSh(\BMeshDiag{n}[\Act^{-1}])
  \]
  Here $\PSh(\BMeshDiag{n})[\Act^{-1}]$ is the reflective localisation of
  $\PSh(\BMeshDiag{n})$ at the maps that are the image of an active bordism
  under the Yoneda embedding. This is precisely the reflective subcategory
  of $\PSh(\BMeshDiag{n})$ whose objects satisfy the isotopy condition.
\end{observation}

In Definition~\ref{def:d-computad} we have named the property that allows
us to glue along atomic covers of diagrammatic $n$-meshes the sheaf condition.
However, the covers do not form a topology on $\BMeshDiag{n}$ directly:
The reflective subcategory $\PSh(\BMeshDiag{n})$ satisfying
only the sheaf condition is not an $\infty$-category of sheaves.
This is corrected by the isotopy condition: The atomic covering diagrams
form a topology on the localisation $\BMeshDiag{n}[\Act^{-1}]$.

\begin{lemma}\label{lem:d-cptd-covering-forward}
  Let $f : \strat{M} \pto \strat{N}$ be a bordism of diagrammatic
  $n$-meshes and $\strat{A}$ an atom of $\strat{M}$. Then there exists an 
  atom $\strat{B}$ of $\strat{N}$ such that $f$ fits into a diagram of
  diagrammatic $n$-meshes of the form  
  \[
    \begin{tikzcd}
      {\nf(\strat{A})} \ar[r, "\simeq", mid vert] &
      {\nf(\strat{S})} &
      {\nf(\strat{B})} \ar[l, "\simeq"', mid vert] \\
      {\strat{A}} \ar[r, mid vert] \ar[d, hook, mid vert] \ar[u, two heads, mid vert] &
      {\strat{S}} \ar[d, hook, mid vert] \ar[u, two heads, mid vert] & 
      {\strat{B}} \ar[d, hook, mid vert] \ar[l, hook'] \ar[u, two heads, mid vert] \\
      {\strat{M}} \ar[r, mid vert, "f"'] &
      {\strat{N}} &
      {\strat{N}} \ar[l, equal]
    \end{tikzcd}
  \]
  and the middle row consists of active bordisms that normalise to equivalences.
\end{lemma}
\begin{proof}
  We construct the diagram piece by piece as follows.
  We begin by using the active/inert factorisation to obtain the square
  \[
    \begin{tikzcd}
      {\strat{A}} \ar[r, mid vert] \ar[d, hook, mid vert] &
      {\strat{S}} \ar[d, hook, mid vert] \\
      {\strat{M}} \ar[r, mid vert, "f"'] &
      {\strat{N}}
    \end{tikzcd}
  \]
  By Proposition~\ref{prop:d-comb-meshes-iff-submersive} the bordisms of diagrammatic $n$-meshes are submersive.
  Because the bordism $\strat{A} \pto \strat{S}$ is active by construction and $\strat{A}$
  is already an atom, it follows that $\strat{A} \pto \strat{S}$ normalises
  to an equivalence
  \[
    \begin{tikzcd}
      {\nf(\strat{A})} \ar[r, "\simeq", mid vert] &
      {\nf(\strat{S})} \\
      {\strat{A}} \ar[r, mid vert] \ar[d, hook, mid vert] \ar[u, two heads, mid vert] &
      {\strat{S}} \ar[d, hook, mid vert] \ar[u, two heads, mid vert] \\
      {\strat{M}} \ar[r, mid vert, "f"'] &
      {\strat{N}}
    \end{tikzcd}
  \]
  We see that $\nf(\strat{S}) \simeq \nf(\strat{A})$ is an atom and therefore
  $\strat{S}$ refines a framed basic.
  It follows that there exists at least one an atom $\strat{B}$ of $\strat{S}$
  such that $\strat{B}$ refines the same framed basic as $\strat{S}$.
  We therefore have an equivalence $\nf(\strat{S}) \simeq \nf(\strat{B})$
  which completes the diagram in the claim.
\end{proof}


\begin{lemma}\label{lem:d-cptd-covering-backward}
  Let $f : \strat{M} \pto \strat{N}$ be an active bordism of diagrammatic
  $n$-meshes and $\strat{B}$ an atom in $\strat{N}$.
  Then there exists an atom $\strat{A}$ of $\strat{M}$ such that $f$ fits into
  a diagram of diagrammatic $n$-meshes of the form
  \[
    \begin{tikzcd}
      {\nf(\strat{A})} \ar[r, "\simeq", mid vert] &
      {\nf(\strat{S})} &
      {\nf(\strat{B})} \ar[l, "\simeq"', mid vert] \\
      {\strat{A}} \ar[r, mid vert] \ar[d, hook, mid vert] \ar[u, two heads, mid vert] &
      {\strat{S}} \ar[d, hook, mid vert] \ar[u, two heads, mid vert] & 
      {\strat{B}} \ar[d, hook, mid vert] \ar[l, hook'] \ar[u, two heads, mid vert] \\
      {\strat{M}} \ar[r, mid vert, "f"'] &
      {\strat{N}} &
      {\strat{N}} \ar[l, equal]
    \end{tikzcd}
  \]
  and the middle row consists of active bordisms that normalise to equivalences.
\end{lemma}
  

\begin{proposition}
  The localisation $\BMeshDiag{n}[\Act^{-1}]$ at the active bordisms
  of diagrammatic $n$-meshes has a topology that is generated by the atomic
  covering diagrams
  \[
    \stratPos{\strat{M}}^{\op} \longrightarrow (\BMeshDiag{n})_{/ \strat{M}}
  \]
  for each diagrammatic $n$-mesh $\strat{M} \in \BMeshDiag{n}$.  
  In particular $\Cptd{n}$ is an $\infty$-topos of sheaves.
\end{proposition}
\begin{proof}
  By~\cite[where]{htt} a Grothendieck topology on an $\infty$-category $\cat{C}$
  is exactly a Grothendieck topology on the homotopy $1$-category $\Ho(\cat{C})$.
  It therefore suffices to show that the covering diagrams induce a topology
  on $\Ho(\BMeshDiag{n}[\Act^{-1}])$.
  In particular, we show that they define a coverage:
  For every map $\varphi : [\strat{M}] \to [\strat{N}]$ in $\Ho(\BMeshDiag{n}[\Act^{-1}])$
  and every atom $\strat{B}$ in $\strat{N}$ there is an atom $\strat{A}$
  of $\strat{M}$ so that $\varphi$ restricts to a map between the atoms
  \[
    \begin{tikzcd}
      {[\strat{A}]} \ar[r, dashed] \ar[d, hook] &
      {[\strat{B}]} \ar[d, hook] \\
      {[\strat{M}]} \ar[r, "\varphi"'] &
      {[\strat{N}]}
    \end{tikzcd}
  \]

  Any map $\varphi : [\strat{M}] \pto [\strat{N}]$ can be written as a zigzag of maps
  in $\BMeshDiag{n}$ starting at $\strat{M}$ and ending at $\strat{N}$ in
  which the backwards maps are active bordisms.
  We can therefore reduce to the case that $\varphi$ is induced by an individual
  map in $\BMeshDiag{n}$.
  When $\varphi = [f]$ for some bordism of diagrammatic $n$-meshes $f : \strat{M} \pto \strat{N}$,
  the coverage property then follows by Lemma~\ref{lem:d-cptd-covering-forward}.
  In the case that $\varphi = [g]^{-1}$ for some active bordism of diagrammatic
  $n$-meshes $g : \strat{N} \pto \strat{M}$ we can similarly apply Lemma~\ref{lem:d-cptd-covering-backward}.

  Therefore, the atomic covers of diagramamtic $n$-meshes form a topology
  on the homotopy $1$-category $\Ho(\BMeshDiag{n}[\Act^{-1}])$ and consequently
  also a topology on the localisation $\BMeshDiag{n}[\Act^{-1}]$.  
  Then $\Cptd{n}$ is the topological localisation of the $\infty$-category
  of presheaves $\PSh(\BMeshDiag{n}[\Act^{-1}])$ and therefore an $\infty$-topos
  of sheaves.
\end{proof}

\subsection{The Diagrammatic Space of Labelled Manifold Diagrams}

Given any diagrammatic $n$-computad $\Sigma$ we define an $n$-diagrammatic space
$\MfldDiagExt^n(\Sigma)$ consisting of extended manifold diagrams with labels in $\Sigma$.
For the purposes of this construction we represent the
presheaf $\Sigma$ as a right fibration $\Sigma : \cat{C} \to \BMeshDiag{n}$.


\begin{construction}
  Let $\Sigma : \cat{C} \to \BMeshDiag{n}$ be a right fibration
  representing a space-valued presheaf on $\BMeshDiag{n}$.
  Consider the diagram of simplicial sets
  \[
    \begin{tikzcd}
      {\MeshDiagExt^n(\Sigma)} \ar[r] \ar[d] \pullbackcorner &
      {\MeshDiagExt^n} \ar[d] \ar[r] &
      {\BMeshClosed{n}} \\
      {\cat{C}} \ar[r, "\Sigma"'] &
      {\BMeshDiag{n}}
    \end{tikzcd}
  \]
  where the square is a pullback square.
  Then $\MeshDiagExt^n(\Sigma)$ is the simplicial set of \defn{extended meshed $n$-manifold
  diagrams with labels in $\Sigma$}.
  The composite of the top row of the diagram is projection map $\MeshDiagExt^n(\Sigma) \to \BMeshClosed{n}$.
\end{construction}

\begin{lemma}
  Let $\Sigma : \cat{C} \to \BMeshClosed{n}$ be a right fibration.
  Then the projection map $\MeshDiagExt^n(\Sigma) \to \BMeshClosed{n}$
  is a cocartesian fibration.
\end{lemma}
\begin{proof}
  Since $\Sigma : \cat{C} \to \BMeshClosed{n}$ is a right fibration, it is in particular
  a categorical fibration. The map $\MeshDiagExt^n \to \BMeshClosed{n}$ is a cocartesian
  fibration by Lemma~\ref{lem:d-comb-meshed-diag-cocartesian} and a map is
  cocartesian if and only if $\MeshDiagExt^n \to \BMeshDiag{n}$ sends it to
  an equivalence.  
  The claim then follows Lemma~\ref{lem:b-cat-cocartesian-top-row}.
\end{proof}

\begin{observation}
  Suppose that $\Sigma : \cat{C} \to \BMeshClosed{n}$ is a right fibration.
  Then the cocartesian fibration $\MeshDiagExt^n(\Sigma) \to \BMeshClosed{n}$
  induces a functor of $\infty$-categories
  \[
    \MeshDiagExt^n(\Sigma)(-) : \BMeshClosed{n} \longrightarrow \CatInfty.
  \]
  Given any closed $n$-mesh $\strat{X}$ the $\infty$-category
  $\MeshDiagExt^n(\Sigma)(\strat{X})$ is represented by the quasicategory
  that arises from taking the pullback of simplicial sets
  \[
    \begin{tikzcd}
      {\MeshDiagExt^n(\Sigma)(\strat{X})} \ar[r, hook] \ar[d] &
      {\MeshDiagExt^n(\Sigma)} \ar[d] \\
      {\Delta\ord{0}} \ar[r, hook, "\strat{X}"'] &
      {\BMeshClosed{n}}
    \end{tikzcd}
  \]
\end{observation}

In~\S\ref{sec:d-comb-meshed-diag} we have seen that we have an alternative presentation of the $n$-diagrammatic space of
extended $n$-manifold diagrams $\MfldDiagExt^n(-)$, as defined originally in~\S\ref{sec:d-mfld-extended},
using $n$-diagrammatic meshes: For every closed $n$-mesh $\strat{X}$ there is an equivalence
between the space $\MfldDiagExt^n(\strat{X})$ and the classifying space $\CatToSpace(\MeshDiagExt^n(\strat{X}))$.
By analogy, we define the space $\MfldDiagExt^n(\Sigma)(\strat{X})$ of extended $n$-manifold diagrams
with labels in $\Sigma$ by taking the classifying space of the meshed version.

\begin{definition}
  Let $\Sigma : \BMeshDiag{n, \op} \to \Space$ be a presheaf.
  For every closed $n$-mesh we define the space of \defn{extended $n$-manifold diagrams
  on $\strat{X}$ with labels in $\strat{X}$} as the classifying space
  \[
    \MfldDiagExt^n(\Sigma)(\strat{X}) := \CatToSpace(\MeshDiagExt^n(\Sigma)(\strat{X}))
  \]
  This definition extends to a functor 
  \[
    \MfldDiagExt^n(-)(-) : \PSh(\BMeshDiag{n}) \longrightarrow \Fun(\BMeshClosed{n}, \Space).
  \]
\end{definition}




\begin{thm}\label{thm:mfld-diag-labelled-cat}
  Let $\Sigma$ be a $n$-diagrammatic computad.
  Then the functor
  \[ \MfldDiagExt^n(\Sigma) : \BMeshClosed{n} \longrightarrow \Space \]
  is an $n$-diagrammatic space.
\end{thm}
\begin{proof}
  This follows from Lemma~\ref{lem:d-cptd-segal} and Lemma~\ref{lem:d-cptd-regular} below.
\end{proof}

In order to prove this we build upon the previous result from Theorem~\ref{thm:d-mfld-diag-space} that unlabelled $n$-manifold diagrams form
an $n$-diagrammatic space $\MfldDiag^n$:
We verify the Segal condition for $\MfldDiag^n(\Sigma)$ 
by first gluing together the unlabelled $n$-manifold diagrams and then using the sheaf condition on $\Sigma$ to glue together the labels as well.
The isotopy condition on $\Sigma$ ensures that labels in $\Sigma$ can be transported along isotopies of $n$-manifold diagrams.
We spend the remainder of this section with the technical details.

\begin{observation}
  Suppose that $\Sigma$ is an $n$-diagrammatic computad that is presented by a
  right fibration $\Sigma : \cat{C} \to \BMeshDiag{n}$.
  For every closed $n$-mesh $\strat{X}$  
  we have a forgetful map
  $\MeshDiagExt^n(\Sigma)(\strat{X}) \to \MeshDiagExt^n(\strat{X})$
  that throws away the labels.
  This map does not in general induce a Kan fibration between the Kan fibrant replacements
  \[
    \Ex^\infty(\MeshDiagExt^n(\Sigma)(\strat{X}))
    \longrightarrow
    \Ex^\infty(\MeshDiagExt^n(\strat{X})).
  \]
  In particular, we can not lift all isotopies of extended $n$-manifold diagrams
  that have new cells emerge from infinity, outside the confines of the closed $n$-mesh $\strat{X}$.
\end{observation}

We solve this issue by introducing a labelled $n$-manifold diagrams that are not extended,
so that we avoid any non-trivial geometry outside the closed $n$-mesh.


\begin{construction}
  Let $\Sigma : \cat{C} \to \BMeshDiag{n}$ be a right fibration.
  For any closed $n$-mesh we define the simplicial set  
  $\MeshDiag^n(\Sigma)(\strat{X})$ 
  of \defn{meshed $n$-manifold diagrams on $\strat{X}$ with labels in $\Sigma$}
  by taking the pullback of simplicial sets
  \[
    \begin{tikzcd}
      {\MeshDiag^n(\Sigma)(\strat{X})} \ar[r, hook] \ar[d] \pullbackcorner &
      {\MeshDiagExt^n(\Sigma)(\strat{X})} \ar[d] \\
      {\MeshDiag^n(\strat{X})} \ar[r, hook] &
      {\MeshDiagExt^n(\strat{X})}
    \end{tikzcd}
  \]
  Then the space $\MfldDiag^n(\Sigma)(\strat{X})$
  of \defn{$n$-manifold diagrams on $\strat{X}$ with labels in $\Sigma$}
  is the Kan fibrant replacement $\Ex^\infty(\MeshDiag^n(\Sigma)(\strat{X}))$.
\end{construction}

\begin{lemma}\label{lem:d-cptd-canonical-equiv}
  Let $\Sigma : \cat{C} \to \BMeshDiag{n}$ be a right fibration that
  represents an $n$-diagrammatic computad.
  Then for every closed $n$-mesh $\strat{X}$ the induced inclusion map
  \begin{equation}\label{eq:d-cptd-canonical-equiv}
    \begin{tikzcd}
    {\Ex^\infty(\MeshDiag^n(\Sigma)(\strat{X}))} \ar[r, hook] &
    {\Ex^\infty(\MeshDiagExt^n(\Sigma)(\strat{X}))}
    \end{tikzcd}
  \end{equation}
  is an equivalence of Kan complexes.
\end{lemma}
\begin{proof}
  Let $i : \MeshDiag^n(\strat{X}) \hookrightarrow \MeshDiagExt^n(\strat{X})$
  denote the inclusion map. By Corollary~\ref{cor:d-mesh-diag-canonicalise}
  there exists a retraction $r$ of $i$ together with a natural transformation
  $\eps : i \circ r \to \id$.
  By taking the pullbacks along $i$ and $r$
  \[
    \begin{tikzcd}
      {\MeshDiagExt^n(\Sigma)(\strat{X})} \ar[r, "r'"] \ar[d] \pullbackcorner &
      {\MeshDiag^n(\Sigma)(\strat{X})} \ar[d] \ar[hook, r, "i'"] \pullbackcorner &
      {\MeshDiagExt^n(\Sigma)(\strat{X})} \ar[d] \\
      {\MeshDiagExt^n(\strat{X})} \ar[r, "r"'] &
      {\MeshDiag^n(\strat{X})} \ar[hook, r, "i"'] &
      {\MeshDiagExt^n(\strat{X})}
    \end{tikzcd}
  \]
  we then have an inclusion map $i'$ with a retraction $r'$.
  Morever, by taking the pullback
  \[
    \begin{tikzcd}
      {\Delta\ord{1} \times \MeshDiagExt^n(\Sigma)(\strat{X})} \ar[r, "\eps'"] \ar[d] \pullbackcorner &
      {\MeshDiagExt^n(\Sigma)(\strat{X})} \ar[d] \\
      {\Delta\ord{1} \times \MeshDiagExt^n(\strat{X})} \ar[r, "\eps"'] &
      {\MeshDiagExt^n(\strat{X})}
    \end{tikzcd}
  \]
  we have a natural transformation $\eps' : i' \circ r' \to \id$.
  When we then apply $\Ex^\infty$ to $i', r', \eps'$ we see that
  $\Ex^\infty(\MeshDiag^n(\Sigma)(\strat{X}))$ is a deformation
  retract of $\Ex^\infty(\MeshDiagExt^n(\Sigma)(\strat{X}))$.
\end{proof}

\begin{lemma}\label{lem:d-cptd-forget-label-kan}
  Let $\Sigma : \cat{C} \to \BMeshDiag{n}$ be a right fibration that
  represents an $n$-diagrammatic computad.
  Then for every closed $n$-mesh $\strat{X}$ the forgetful map
  \[
    \Ex^\infty(\MeshDiag^n(\Sigma)(\strat{X}))
    \longrightarrow
    \Ex^\infty(\MeshDiag^n(\strat{X}))
  \]
  is a Kan fibration.
\end{lemma}
\begin{proof}
  Suppose we have a lifting problem
  \[
    \begin{tikzcd}
      {\Lambda^i\ord{k}} \ar[r] \ar[d] &
      {\Ex^\infty(\MeshDiag^n(\Sigma)(\strat{X}))} \ar[d] \\
      {\Delta\ord{k}} \ar[r] \ar[ur, dashed] &
      {\Ex^\infty(\MeshDiag^n(\strat{X}))}
    \end{tikzcd}
  \]
  Then there exists a $t \geq 0$ such that the lifting problem factors
  through the $t$h stage $\Ex^t$ of the filtered colimit that defines the functor $\Ex^\infty$:
  \[
    \begin{tikzcd}
      {\Lambda^i\ord{k}} \ar[r] \ar[d] &
      {\Ex^t(\MeshDiag^n(\Sigma)(\strat{X}))} \ar[d] \\
      {\Delta\ord{k}} \ar[r] \ar[ur, dashed] &
      {\Ex^t(\MeshDiag^n(\strat{X}))}
    \end{tikzcd}
  \]
  By adjunction the lifting problem is then equivalent to
  \[
    \begin{tikzcd}
      {\Sd^t(\Lambda^i\ord{k})} \ar[r] \ar[d] &
      {\MeshDiag^n(\Sigma)(\strat{X})} \ar[d] \\
      {\Sd^t(\Delta\ord{k})} \ar[r] \ar[ur, dashed] &
      {\MeshDiag^n(\strat{X})}
    \end{tikzcd}
  \]
  To solve this lifting problem we must find a filler for the diagram
  \begin{equation}\label{eq:d-cptd-forget-label-kan:lift-1}
    \begin{tikzcd}
      {\Sd^t(\Lambda^i\ord{k})} \ar[rr] \ar[d] &
      {} &
      {\cat{C}} \ar[d] \\
      {\Sd^t(\Delta\ord{k})} \ar[r] \ar[urr, dashed] &
      {\MeshDiag^n(\strat{X})} \ar[r] &
      {\BMeshDiag{n}}
    \end{tikzcd}
  \end{equation}
  which extends the labelling of the meshed $n$-manifold diagram.
    
  Since the right fibration $\Sigma : \cat{C} \to \BMeshDiag{n}$
  represents an $n$-diagrammatic computad,
  it is fibrant in the $\Act$-local contravariant model structure
  on $\sSet_{/ \BMeshDiag{n}}$ representing presheaves on the localisation
  $\PSh(\BMeshDiag{n}[\Act^{-1}])$.
  Then by the dual of~\cite[Lemma 2.5.]{localizing-hypercover} we have
  that the map $\Sigma' : \cat{C}' \to \Act$ obtained by the pullback
  \[
    \begin{tikzcd}
      {\cat{C}'} \ar[r, hook] \ar[d, "\Sigma'"'] \pullbackcorner &
      {\cat{C}} \ar[d, "\Sigma"] \\
      {\Act} \ar[r, hook] &
      {\BMeshDiag{n}}
    \end{tikzcd}
  \]
  is a Kan fibration.
  The map $\MeshDiag^n(\strat{X}) \to \BMeshDiag{n}$ factors through the inclusion
  $\Act \hookrightarrow \BMeshDiag{n}$ and so the lifting problem (\ref{eq:d-cptd-forget-label-kan:lift-1}) factors through the pullback
  \begin{equation}\label{eq:d-cptd-forget-label-kan:lift-2}
    \begin{tikzcd}
      {\Sd^t(\Lambda^i\ord{k})} \ar[rr] \ar[d, hook] &
      {} &
      {\cat{C}'} \ar[d, "\Sigma'"{description}] \ar[r, hook] \pullbackcorner &
      {\cat{C}} \ar[d, "\Sigma"] \\
      {\Sd^t(\Delta\ord{k})} \ar[r] \ar[urr, dashed] &
      {\MeshDiag^n(\strat{X})} \ar[r] &
      {\Act} \ar[r, hook] & 
      {\BMeshDiag{n}}
    \end{tikzcd}
  \end{equation}
  The inclusion $\Sd^i(\Lambda^i\ord{k}) \hookrightarrow \Sd^i(\Delta\ord{k})$
  is a Kan equivalence, and so we have a lift in (\ref{eq:d-cptd-forget-label-kan:lift-2}).  
\end{proof}

\begin{lemma}\label{lem:lift-homotopy-limit-kan-fibration}
  Let $f : \cat{I}^\triangleleft \to \sSet^{\ord{1}}$ be a diagram
  which satisfies the following conditions.
  \begin{enumerate}
    \item For each $i \in \cat{I}$ the map $f(i) : E(i) \to B(i)$ is a Kan fibration between Kan complexes.
    \item The diagram $B : \cat{I}^\triangleleft \to \sSet$ is a homotopy limit diagram.
    \item For any vertex $x \in B(\bot)_0$ let
    $F_x : \cat{I}^\triangleleft \to \sSet$
    be the diagram which sends $i \in \cat{I}^\triangleleft$ to the fibre
    \[
      \begin{tikzcd}
        {F_x(i)} \ar[r] \ar[d] \pullbackcorner &
        {E(i)} \ar[d, "X(i)"] \\
        {\{ x \}} \ar[r] &
        {B(i)}
      \end{tikzcd}
    \]
    Then $F_x$ is a homotopy limit diagram.
  \end{enumerate}
  Then $f$ is a homotopy limit diagram.
\end{lemma}
\begin{proof}
  Let $L : \sSet \to \Space$ the localisation functor which inverts the Kan equivalences
  and $\cod : \Space^{\ord{1}} \to \Space$ the codomain fibration.
  Since we assumed $B : \cat{I}^\triangleleft \to \sSet$ to be a homotopy limit diagram, the composite $L \circ B$ is a limit diagram in $\Space$.
  Therefore, by~\cite[\href{https://kerodon.net/tag/02KR}{Corollary 02KR}]{kerodon}
  the diagram $L^{\ord{1}} \circ f$ is a limit diagram if and only if
  it is a $\cod$-limit diagram.
  For each $i \in \cat{I}^\triangleleft$ we take the pullback of simplicial sets
  \begin{equation}\label{eq:left-homotopy-limit-Kan-fibration:pullback}
    \begin{tikzcd}
      {E'(i)} \ar[r] \ar[d, "f'(i)", swap] \pullbackcorner &
      {E(i)} \ar[d, "f(i)"] \\
      {B(\bot)} \ar[r] &
      {B(i)}
    \end{tikzcd}
  \end{equation}
  By naturality we obtain a functor $f' : \cat{I}^\triangleleft \to \sSet^{\ord{1}}$
  as well as a natural transformation $f' \to f$.
  This descends to a natural transformation between the localisations
  $L^{\ord{1}} \circ f' \to L^{\ord{1}} \circ f$.
  Because every map $f(i) : E(i) \to B(i)$ is a Kan fibration
  and all involved objects are Kan complexes,
  the square (\ref{eq:left-homotopy-limit-Kan-fibration:pullback}) is a homotopy pullback.
  A map in the arrow category $\Space^{\ord{1}}$ is $\cod$-cartesian if
  it is a pullback square in $\Space$;
  therefore the components of the induced natural transformation $L^{\ord{1}} \circ f' \to L^{\ord{1}} \circ f$ are $\cod$-cartesian maps.
  Moreover the component on the cone point $\bot \in \cat{I}^\triangleleft$ is an equivalence.
  By the dual of~\cite[\href{https://kerodon.net/tag/031R}{Proposition 031R}]{kerodon}
  it follows that $L^{\ord{1}} \circ f$ is a $\cod$-limit diagram if
  $L^{\ord{1}} \circ f'$ is a $\cod$-limit diagram.
  
  Since $\Space$ is complete, the codomain functor $\cod : \Space^{\ord{1}} \to \Space$
  is a cartesian fibration.
  The fibre of $\cod$ over a space $X \in \Space$ is the slice $\infty$-category
  $\Space_{/ X}$ which is also complete.
  Moreover, for any map of spaces $X \to Y$ the contravariant transport functor
  $\Space_{/ Y} \to \Space_{/ X}$ is given by pullback and therefore preserves
  limits.
  Then by 
  \cite[\href{https://kerodon.net/tag/031S}{Proposition 031S}]{kerodon}
  it follows that $L^{\ord{1}} \circ f'$ is a $\cod$-limit diagram
  if and only if $L^{\ord{1}} \circ f'$ is a limit diagram in the slice
  $\infty$-category $\Space_{/ L(B(\bot))}$.
  For each $i \in \cat{I}^\triangleleft$ and vertex $x \in B(\bot)_0$
  the simplicial set $F_x(i)$ is a Kan complex representing the homotopy
  fibre of $f'(i) : E'(i) \to B(\bot)$ over $x$
  because $f'(i)$ is a Kan fibration (being a pullback of the Kan fibration
  $f(i)$).
  Limits in $\Space_{/ L(B(\bot))}$ are computed fibrewise,
  and so $L^{\ord{1}} \circ f'$ is a limit diagram in the slice
  by the assumption that $F_x : \cat{I}^\triangleleft \to \sSet$ is a
  homotopy limit diagram for all $x \in B(\bot)_0$.
\end{proof}

\begin{lemma}\label{lem:d-cptd-segal}
  Let $\Sigma : \cat{C} \to \BMeshDiag{n}$ be a right fibration
  representing an $n$-diagrammatic computad.
  Let $\strat{X}$ be a closed $n$-mesh and
  $U : \stratPos{\strat{X}}^{\triangleleft} \to \BMeshClosed{n}$
  its cell covering diagram.
  Then the following diagrams $\stratPos{\strat{X}}^{\triangleleft} \to \sSet$
  are homotopy limit diagrams:
  \begin{enumerate}[label=(\arabic*)]
    \item $\MfldDiagExt^n(U(-))$
    \item $\Ex^\infty(\MeshDiagExt^n(U(-)))$
    \item $\Ex^\infty(\MeshDiag^n(U(-)))$
    \item $\Ex^\infty(\MeshDiag^n(\Sigma)(U(-)))$
    \item $\Ex^\infty(\MeshDiagExt^n(\Sigma)(U(-)))$
  \end{enumerate}
  In particular $\MfldDiagExt^n(\Sigma)(-)$ satisfies the Segal condition of an $n$-diagrammatic space.
\end{lemma}
\begin{proof}
  By Theorem~\ref{thm:d-mfld-diag-space} the functor $\MfldDiagExt^n(-)$ defines an $n$-diagrammatic space.
  Therefore, (1) is a homotopy limit by the Segal condition.
  By Lemma~\ref{lem:mesh-diag-to-mfld-diag} the map $\MfldDiagExt^n \to \BMeshClosed{n}$ is the free left fibration generated by the cocartesian fibration $\MeshDiagExt^n \to \BMeshClosed{n}$.
  This implies that (1) and (2) are equivalent as diagrams of spaces.
  Now by Lemma~\ref{lem:d-cptd-canonical-equiv} the diagrams (2) and (3) are equivalent as well.
  Therefore, (3) is a homotopy limit diagram.
  
  Let $\strat{M}$ be a diagrammatic $n$-mesh that represents a point in
  $\Ex^\infty(\MeshDiag^n(\strat{X}))$.
  For every $i \in \stratPos{\strat{X}}^\triangleleft$ we denote by $\strat{M}_i$
  the image of $\strat{M}$ via the restriction map
  \[
    \Ex^\infty(\MeshDiag^n(\strat{X})) \longrightarrow \Ex^\infty(\MeshDiag^n(U(i)))
  \]
  This defines a covering diagram
  $V : \stratPos{\strat{X}}^{\op, \triangleright} \to \BMeshDiag{n}$
  with $V(i) = \strat{M}_i$ and the induced inert maps in between.
  For any $i \in \stratPos{\strat{X}}^{\triangleright}$
  we have a pullback square of simplicial sets
  \[
    \begin{tikzcd}
      {\Sigma(V(i))} \ar[r] \ar[d] \pullbackcorner &
      {\Ex^\infty(\MeshDiag^n(\Sigma)(U(i)))} \ar[d] \\
      {\Delta\ord{0}} \ar[r, "\strat{M}_i"'] &
      {\Ex^\infty(\MeshDiag^n(U(i)))}
    \end{tikzcd}
  \]
  By Lemma~\ref{lem:d-cptd-forget-label-kan} the maps on the right of this diagram are Kan fibrations, and therefore the pullback is a homotopy pullback as well.
  The diagram $\Sigma(V(-)) : \stratPos{\strat{X}}^\triangleleft \to \sSet$
  is a homotopy limit diagram because $\Sigma$ is an $n$-diagrammatic computad
  and $V(-)$ is a covering diagram.
  Then Lemma~\ref{lem:lift-homotopy-limit-kan-fibration} shows that (4) is a homotopy limit diagram.
  Finally, by Lemma~\ref{lem:d-cptd-canonical-equiv} the diagrams (4) and (5) are equivalent and so (5) is also a homotopy limit diagram.
\end{proof}

\begin{lemma}\label{lem:d-cptd-regular}
  Let $\Sigma : \cat{C} \to \BMeshDiag{n}$ be a left fibration representing a label system.
  Then for any closed $n$-mesh $\strat{X}$ the induced square
  \[
    \begin{tikzcd}
      {\MeshDiagExt(\Sigma)(\strat{X})} \ar[r] \ar[d] &
      {\MeshDiagExt(\Sigma)(\shape{\strat{X}})} \ar[d] \\
      {\MeshDiagExt(\Sigma)(\partial\strat{X})} \ar[r] &
      {\MeshDiagExt(\Sigma)(\partial\shape{\strat{X}})}
    \end{tikzcd}
  \]
  is a pullback square of spaces.
  In particular the map $\MfldDiagExt^n(\Sigma)(-) : \BMeshClosed{n} \to \Space$
  satisfies the regularity condition of an $n$-diagrammatic space.
\end{lemma}
\begin{proof}
  This can be shown analogously to Lemma~\ref{lem:d-cptd-segal} using that $\MfldDiagExt^n(-)$ is an $n$-diagrammatic space
  by Proposition~\ref{prop:d-mfld-ext-equiv}.
\end{proof}

\subsection{\texorpdfstring{Free $(\infty, n)$-Categories via Manifold Diagrams}{Free Higher Categories via Manifold Diagrams}}

\begin{construction}
  For every $n$-diagrammatic computad $\Sigma$ we have constructed
  an $n$-diagrammatic space $\MfldDiag^n(\Sigma)$ of $n$-manifold diagrams
  labelled in $\Sigma$.
  By the equivalence $\Seg^\regular(\BTrussOpen{n}) \simeq \Seg(\FinOrd^n)$
  the $n$-diagrammatic space $\MfldDiag^n(\Sigma)$ induces an $n$-uple Segal
  space, and after applying the localisation functor $\Seg(\FinOrd^n) \to \CatN{n}$
  an $(\infty, n)$-category $\DiagFree(\Sigma)$.  
  This process extends to a functor
  $\DiagFree : \Cptd{n} \to \CatN{n}$.
\end{construction}

For an $n$-diagrammatic computad $\Sigma$ we call $\DiagFree(\Sigma)$
the \defn{free $(\infty, n)$-category} generated by $\Sigma$.
To justify this terminology, we use this section to demonstrate that
$\DiagFree$ has a right adjoint $\DiagForget$.



\begin{lemma}
  The following diagram of $\infty$-categories commutes:
  \[
    \begin{tikzcd}[column sep = huge]
      {\PSh(\BTrussDiag{n})} \ar[d, "\simeq"'] \ar[r, "\trussFC_!"] &
      {\PSh(\BTrussOpen{n})} \ar[d, "\simeq"]  \\
      {\PSh(\BMeshDiag{n})} \ar[r, "\MfldDiag^n(\cdot)", swap] &
      {\PSh(\BMeshClosed{n, \op})}
    \end{tikzcd}
  \]
\end{lemma}
\begin{proof}
  Let $\Sigma : \cat{C} \to \BTrussDiag{n}$ be a right fibration which represents
  a space-valued presheaf on $\BTrussDiag{n}$.
  The equivalence $\PSh(\BTrussDiag{n}) \to \PSh(\BMeshDiag{n})$ sends $\Sigma$ to
  the right fibration $\Sigma' : \cat{C}' \to \BMeshDiag{n}$ constructed
  as the pullback of simplicial sets  
  \[
    \begin{tikzcd}
      {\cat{C}'} \ar[r, "\simeq"] \ar[d, "\Sigma'"'] \pullbackcorner &
      {\cat{C}} \ar[d, "\Sigma"] \\
      {\BMeshDiag{n}} \ar[r, "\simeq"'] &
      {\BTrussDiag{n}}
    \end{tikzcd}
  \]
  We can then form a commutative cube of simplicial sets
  \[
  	\begin{tikzcd}[
      column sep={4.5em,between origins},
      row sep={4em,between origins},
    ]
  		& {\Tw(\BTrussOpen{n}) \times_{\BTrussOpen{n}} \cat{C}} &&
      {\Tw(\BTrussOpen{n})} \\
  		{\MeshDiagExt^n(\Sigma')} && {\BConfOrtho{n}} \\
  		& {\cat{C}} && {\BTrussOpen{n}} \\
  		{\cat{C}'} && {\BMeshOpen{n}}
  		\arrow[from=2-1, to=4-1]
  		\arrow[from=4-1, to=4-3, "\meshFC \circ \Sigma'"]
  		\arrow[from=1-4, to=3-4]
  		\arrow[from=1-2, to=3-2]
  		\arrow[from=3-2, to=3-4, "\trussFC \circ \Sigma"{pos = 0.2}]
  		\arrow[from=1-2, to=1-4]
  		\arrow[from=2-1, to=1-2, dashed]
  		\arrow[from=4-1, to=3-2]
  		\arrow[from=4-3, to=3-4]
  		\arrow[from=2-3, to=1-4]
  		\arrow[crossing over, from=2-1, to=2-3]
  		\arrow[crossing over, from=2-3, to=4-3]
  	\end{tikzcd}
  \]
  in which the front and back square are pullback squares of simplicial sets.
  The dashed diagonal map is induced by the universal property of the pullback.
  Since both $\BConf{n} \to \BMeshOpen{n}$ and $\Tw(\BTrussOpen{n}) \to \BTrussOpen{n}$
  are categorical fibrations, the front and back square are homotopy pullback squares.
  The solid diagonal maps are categorical equivalences, it follows that also the
  dashed map is a categorical equivalence.
  This induced equivalence then is the top horizontal map in the diagram
  \[
    \begin{tikzcd}
      {\MeshDiagExt^n(\Sigma')} \ar[r, "\simeq"] \ar[d] &
      {\Tw(\BTrussOpen{n}) \times_{\BTrussOpen{n}} \cat{C}} \ar[d] \\
      {\BMeshClosed{n}} \ar[r, "\simeq"'] &
      {\BTrussOpen{n, \op}}
    \end{tikzcd}
  \]
  where the vertical maps are cocartesian fibrations.
  The free left fibration associated to $\MeshDiagExt^n(\Sigma') \to \BMeshClosed{n}$ represents the functor $\MfldDiag^n(\Sigma')$.
  Meanwhile, the free left fibration induced by  
  $\Tw(\BTrussOpen{n}) \times_{\BTrussOpen{n}} \cat{C} \to \BTrussOpen{n, \op}$  
  represents the presheaf $\trussFC_!(\Sigma)$.
  It follows that the square in the claim commutes.
\end{proof}

\begin{proposition}\label{prop:d-free-right-adjoint}
  The functor $\DiagFree : \Cptd{n} \to \CatN{n}$ has a right adjoint.
\end{proposition}
\begin{proof}
  We write $\Act$ for the wide subcategory of active maps in $\BTrussDiag{n}$ and denote the
  localisation map by $L : \BTrussDiag{n} \to \BTrussDiag{n}[\Act^{-1}]$.
  We then consider the chain of adjunctions
  \begin{equation}\label{eq:d-free-forget-long-chain}
    \begin{tikzcd}
    	{\PSh(\BTrussDiag{n})} & {\PSh(\BTrussDiag{n}[\Act^{-1}])} & {\PSh(\BTrussDiag{n})} & {\PSh(\BTrussOpen{n})} & {\CatN{n}}
    	\arrow[""{name=0, anchor=center, inner sep=0}, "{L_!}", shift left=2, from=1-1, to=1-2]
    	\arrow[""{name=1, anchor=center, inner sep=0}, "{L^*}", shift left=2, hook, from=1-2, to=1-3]
    	\arrow[""{name=2, anchor=center, inner sep=0}, "{L_*}", shift left=2, from=1-3, to=1-2]
    	\arrow[""{name=3, anchor=center, inner sep=0}, "{L^*}", shift left=2, hook, from=1-2, to=1-1]
    	\arrow[""{name=4, anchor=center, inner sep=0}, "{\trussFC_!}", shift left=2, from=1-3, to=1-4]
    	\arrow[""{name=5, anchor=center, inner sep=0}, "{\trussFC^*}", shift left=2, from=1-4, to=1-3]
    	\arrow[""{name=6, anchor=center, inner sep=0}, "{}", shift left=2, from=1-4, to=1-5]
    	\arrow[""{name=7, anchor=center, inner sep=0}, hook, "{}", shift left=2, from=1-5, to=1-4]
    	\arrow["\dashv"{anchor=center, rotate=-90}, draw=none, from=0, to=3]
    	\arrow["\dashv"{anchor=center, rotate=-90}, draw=none, from=1, to=2]
    	\arrow["\dashv"{anchor=center, rotate=-90}, draw=none, from=4, to=5]
    	\arrow["\dashv"{anchor=center, rotate=-90}, draw=none, from=6, to=7]
    \end{tikzcd}
  \end{equation}

  Suppose that $\Sigma$ is an $n$-diagrammatic computad,
  then $\Sigma$ sends active maps to equivalences by the isotopy axiom.
  Therefore, the adjunction unit $\Sigma \to (L^* \circ L_!)(\Sigma)$ is an equivalence.
  In particular the top row of~(\ref{eq:d-free-forget-long-chain}) restricts to
  the functor $\DiagFree : \Cptd{n} \to \CatN{n}$.

  It remains to show that the composite $L^* \circ L_* \circ \trussFC^*$  
  sends $n$-diagrammatic spaces to $n$-diagrammatic computads;
  then the restriction defines the right adjoint $\DiagForget$ of $\DiagFree$.
  Let $\psh{F}$ be a presheaf on $\BTrussOpen{n}$ that is an $n$-diagrammatic space.
  Then the Segal condition directly implies the sheaf condition for $\trussFC^* \psh{F}$.
  However, $\trussFC^* \psh{F}$ does not in general satisfy the isotopy condition
  and therefore fails to be a diagrammatic computad.

  We rectify this by applying the coreflector $L^* \circ L_*$
  which sends $\trussFC^* \psh{F}$ to its maximal subobject
  $\Sigma \hookrightarrow \trussFC^* \psh{F}$ in $\PSh(\BTrussDiag{n})$
  so that $\Sigma$ satisfies the isotopy condition.
  We verify that $\Sigma$ still satisfies the sheaf condition and therefore
  is a diagrammatic computad.

  Let $\Sigma'$ be the smallest subobject of $\trussFC^* \psh{F}$ that contains
  $\Sigma$ so that $\Sigma'$ satisfies the sheaf condition.
  When we show that $\Sigma'$ also satisfies the isotopy condition, we must have
  $\Sigma \simeq \Sigma'$ because $\Sigma$ was chosen to be the maximal
  such subobject of $\trussFC^* \psh{F}$.
  We proceed to show that $\Sigma'$ sends active bordisms to equivalences
  in increasing generality.
  \begin{enumerate}
    \item When $A \in \BTrussDiag{n}$ is an atom, the inclusion
    $\Sigma(A) \hookrightarrow \Sigma'(A)$
    must already be an equivalence. Whenever $\alpha : A \pto B$ is an active bordism between atoms, the induced map $\Sigma(\alpha) : \Sigma(B) \to \Sigma(A)$ is an equivalence
    since $\Sigma$ satisfies the isotopy condition.
    It then follows that $\Sigma'(\alpha)$ must also be an equivalence since it fits into the diagram
    \[
      \begin{tikzcd}
        {\Sigma(B)} \ar[r, "\Sigma(\alpha)"] \ar[d, "\simeq"'] &
        {\Sigma(A)} \ar[d, "\simeq"] \\
        {\Sigma'(B)} \ar[r, "\Sigma'(\alpha)"'] &
        {\Sigma'(A)}
      \end{tikzcd}
    \]
    \item Let $\alpha : S \pto A$ be a degeneracy bordism in $\BTrussDiag{n}$
    where $A$ is an atom. We cover $S$ with a family of atoms $\{ B_i \}_i$
    and take the active/inert factorisation
    \[
      \begin{tikzcd}
        {B_i} \ar[r, hook, mid vert] \ar[d, "\alpha_i"', mid vert] &
        {S} \ar[d, "\alpha", mid vert] \\
        {A_i} \ar[r, hook, mid vert] &
        {A}
      \end{tikzcd}
    \]
    The bordism $\alpha_i : B_i \pto A_i$ is again a degeneracy
    bordism and so the diagrammatic $n$-truss $A_i$ is an atom.
    Therefore, $\alpha_i$ is an active map between atoms and so by the previous
    step we know that the induced map $\Sigma'(\alpha_i) : \Sigma'(A_i) \to \Sigma'(B_i)$
    is an equivalence.
    Since $\Sigma'$ satisfies the sheaf condition by construction, the map $\Sigma'(\alpha)$
    decomposes as the limit
    \[
      \begin{tikzcd}[column sep = huge]
        {\Sigma'(A)} \ar[r, "\Sigma'(\alpha)"] \ar[d, "\simeq"'] &
        {\Sigma'(S)} \ar[d, "\simeq"] \\
        {\lim_i \Sigma'(A_i)} \ar[r, "\lim_i \Sigma'(\alpha_i)"'] &
        {\lim_i \Sigma'(B_i)}
      \end{tikzcd}
    \]
    A limit of equivalences is again an equivalence and so $\Sigma'(\alpha)$
    is an equivalence as well.

    \item Now let $\alpha : A \pto S$ be any active map in $\BTrussDiag{n}$
    so that $A$ is an atom. By the definition of $\BTrussDiag{n}$ we then have
    an equivalence between the normal forms    
    \[
      \begin{tikzcd}
        {A} \ar[r, "\alpha", mid vert] \ar[d, mid vert] &
        {S} \ar[d, mid vert] \\
        {\nf(A)} \ar[r, "\equiv"', mid vert, dashed] &
        {\nf(S)}
      \end{tikzcd}
    \]
    In particular $\nf(S)$ is also an atom and the vertical maps
    in the diagram are degeneracy bordisms into an atom.
    Applying the functor $\Sigma'$ we obtain a diagram of spaces
    \[
      \begin{tikzcd}
        {\Sigma'(\nf(S))} \ar[r, "\simeq"] \ar[d, "\simeq"'] &
        {\Sigma'(\nf(A))} \ar[d, "\simeq"] \\
        {\Sigma'(S)} \ar[r] &
        {\Sigma'(A)}
      \end{tikzcd}
    \]
    where the vertical maps are equivalences by the previous step.
    Therefore, $\Sigma'(\alpha)$ is an equivalence.

    \item Finally, let $\alpha : R \pto S$ be any active bordism in $\BTrussDiag{n}$.
    Then the atoms $\{ A_i \pto R \}_i$ which cover $R$ induce a covering
    $\{ S_i \pto S \}_i$ via the active/inert factorisation:
    \[
      \begin{tikzcd}
        {A_i} \ar[r, hook, mid vert] \ar[d, "\alpha_i"', mid vert] &
        {R} \ar[d, "\alpha", mid vert] \\
        {S_i} \ar[r, hook, mid vert] &
        {S}
      \end{tikzcd}
    \]
    Because $\Sigma'$ satisfies the sheaf condition, the map
    $\Sigma'(\alpha)$ is equivalent to its decomposition for the
    chosen covers 
    \[
      \begin{tikzcd}[column sep = large]
        {\Sigma'(S)} \ar[r, "\Sigma'(\alpha)"] \ar[d, "\simeq"'] &
        {\Sigma'(R)} \ar[d, "\simeq"] \\
        {\lim_{i} \Sigma'(S_i)} \ar[r, "\lim_i \Sigma'(\alpha_i)"'] &
        {\lim_{i} \Sigma'(A_i)}
      \end{tikzcd}
    \]
    By the previous step each $\Sigma'(\alpha_i)$ is an equivalence.
    It then follows that $\Sigma'(\alpha)$ must be an equivalence as well.
  \end{enumerate}
  The presheaf $\Sigma$ on $\BTrussDiag{n}$ resulted from applying the composite
  functor $L^* \circ L_* \circ \trussFC^*$ to the presheaf $\psh{F}$
  representing an $n$-diagrammatic space.
  Therefore, the bottom row of~(\ref{eq:d-free-forget-long-chain}) restricts
  to a functor $\DiagForget : \CatN{n} \to \Cptd{n}$ that is right adjoint
  to $\DiagFree$.
\end{proof}



\begin{observation}
  Let us unpack what the functor $\DiagForget : \CatN{n} \to \Cptd{n}$ does
  on some $(\infty, n)$-category $\cat{C}$ presented as an $n$-diagrammatic space.
  Suppose that $A \in \BTrussOpen{n}$ is an (unlabelled) atom and
  $a : A \to \cat{C}$ a cell of shape $A$ in $\cat{C}$.
  The \defn{algebraic codimension} of the cell $a$ is the maximum
  $0 \leq d \leq n$ such that $a$ factors through a degeneracy bordism 
  \[
    \begin{tikzcd}
      {A} \ar[d, mid vert] \ar[r] & {\cat{C}} \\
      {B} \ar[ur, dashed]
    \end{tikzcd}
  \]
  where $B$ is an atom of singular depth $d$.
  When $A$ is equipped with labels $\ell$ in $\ord{n}$ so that $(A, \ell)$
  is a diagrammatic $n$-truss, then the maps $(A, \ell) \to \DiagForget(\cat{C})$
  correspond to maps $A \to \cat{C}$ with the restriction that each for each
  subatom $B \pto A$ the algebraic codimension of the induced subcell in $\cat{C}$
  is bounded above by the label.
  This restriction is enforced by the coreflector $L^* \circ L_*$ applied
  in Proposition~\ref{prop:d-free-right-adjoint}.
  In particular, every cell in $\cat{C}$ has multiple versions as a generator
  in the diagrammatic computad $\unstrat(\cat{C})$.
  For instance, an object $x$ of $\cat{C}$ is represented as a $0$-cell,
  as the identity $1$-cell $\id : x \to x$, as an identity $2$-cell
  $\id : \id_x \to \id_x$, etc.
\end{observation}